\documentclass[12pt,a4paper]{amsart}
\usepackage{texdraw,multicol,rotating,color}
\usepackage{amsmath}
\usepackage{amssymb}
\usepackage{geometry}
\usepackage[colorlinks=true, pdfstartview=FitV, linkcolor=blue, citecolor=blue, urlcolor=blue, breaklinks=true]{hyperref}
\newtheorem{theorem}{Theorem}[section]
\newtheorem{proposition}[theorem]{Proposition}
\newtheorem{definition}[theorem]{Definition}
\newtheorem{example}[theorem]{\bf Example}
\newtheorem{remark}[theorem]{\bf Remark}
\numberwithin{equation}{section}
\everytexdraw{
	\drawdim in \arrowheadsize l:0.065 w:0.03 \arrowheadtype t:F
	\textref h:C v:C
}
\DeclareMathOperator{\wt}{wt}
\def\wt{\mbox{\sl wt}\,}

\title[A new Young wall realization of $B(\lambda)$ and $B(\infty)$]{A new Young wall realization of $B(\lambda)$ and $B(\infty)$}

\author[Zhaobing Fan]{Zhaobing Fan}
\address{Harbin Engineering University, Harbin, China}
\email{fanzhaobing@hrbeu.edu.cn}

\author[Shaolong Han]{Shaolong Han}
\address{Harbin Engineering University, Harbin, China}
\email{algebra@hrbeu.edu.cn}

\author[Seok-Jin Kang]{Seok-Jin Kang}
\address{Korea Research Institute of Arts and Mathematics, Asan-si, Chungcheongnam-do, 31551, Korea}
\email{soccerkang@hotmail.com}

\author[Young Rock Kim]{Young Rock Kim${}^{*}$}
\address{Graduate School of Education, Hankuk University of Foreign Studies, Seoul, 02450,  Korea}
\email{rocky777@hufs.ac.kr}
\thanks{${}^{*}$ Corresponding author.}

\keywords{quantum affine algebra, crystal basis, highest weight crystal, Young wall}

\subjclass[2010] {17B37, 17B67, 16G20}

\begin{document}
\title[A new Young wall realization of $B(\lambda)$ and $B(\infty)$]{A new Young wall realization of $B(\lambda)$ and $B(\infty)$}

\begin{abstract}	
Using new combinatorics of Young walls, we give a new construction of the arbitrary level highest weight crystal $B(\lambda)$ for the quantum affine algebras of types $A^{(2)}_{2n}$, $D^{(2)}_{n+1}$, $A^{(2)}_{2n-1}$, $D^{(1)}_n$, $B^{(1)}_n$ and $C^{(1)}_n$. We show that the crystal consisting of reduced Young walls is isomorphic to the crystal $B(\lambda)$. Moreover, we provide a new realization of the crystal $B(\infty)$ in terms of reduced virtual Young walls and reduced extended Young walls.
\end{abstract}
	\maketitle
	\tableofcontents
	
\section{Introduction}

\vskip 3mm 

Arguably, one of the most remarkable achievements in representation theory in 1990's would be the canonical basis theory and crystal basis theory for quantum groups, which were developed 
by Lusztig and Kashiwara, respectively \cite{Lu90,Lu91,Kash90,Kash91}. They have made
very deep and important impacts on a wide variety of areas in mathematics. 

\vskip 2mm 

Among others, when we narrow down to the case of quantum affine algebras, 
it was discovered that their representation theory has a close connection with mathematical physics
such as soliton equations, conformal field theory and vertex models.
For example, the one point function for the 6-vertex model can be expressed as the
quotient of the string function by the character of $V(\Lambda_0)$ over $ U_{q}(\widehat{\mathfrak{sl}}_2)$. 
Thus it is a very important problem to give an explicit realization of integrable highest weight
modules over quantum affine algebras. 

\vskip 2mm 

Motivated by this problem, Kang, Kashiwara, Misra, Miwa, Nakashima and Nakayashiki developed the 
theory of {\it perfect crystals} \cite{KMN1,KMN2}. 
Using perfect crystals, one can give a {\it path realization} of the crystal $B(\lambda)$ of the integrable 
highest weight module $V(\lambda)$. In this way, the theory of vertex models can be explained 
in the language of crystal basis theory of quantum affine algebras. 

\vskip 2mm 

The aim of this paper is to provide a new combinatorial realization of the crystal $B(\lambda)$ and $B(\infty)$ for the quantum affine algebras of types $A_{2n}^{(2)}$, $D_{n+1}^{(2)}$, $A_{2n-1}^{(2)}$, $D_{n}^{(1)}$, $B_{n}^{(1)}$ and $C_{n}^{(1)}$. Here, $\lambda$ is a dominant integral weight of arbitrary level. 

\vskip 2mm 

The basic idea behind this realization is {\it  combinatorics of Young walls}; i.e., playing a game with 
various blocks of different colors and different shapes. 
The combinatorics of Young walls and its connection with crystal bases were introduced by Kang \cite{Kang03}. 
More precisely, he first introduced the notion of Young walls and reduced Young walls, 
and explained their combinatorics.  
The action of Kashiwara operators was defined in terms of combinatorics of Young walls,
which provides a crystal structure on the set of (reduced) Young walls. 
Then he  proved that the crystal consisting of reduced Young walls is isomorphic to the level-$1$ highest weight 
crystal $B(\lambda)$. 

\vskip 2mm 

Kang's construction was extended to the case of higher level Young walls by Kang and Lee \cite{KL06,KL09}. 
As a natural approach, to define higher level Young walls, they used multiple layers of level-$1$ Young walls. 
They defined the notion of Young walls and reduced Young walls, and proved that 
the crystal consisting of reduced Young walls is isomorphic to the crystal $B(\lambda)$ 
with a dominant integral weight of higher level. 
However, in this approach, it is quite complicated to describe the combinatorics.
In particular, the definition of Kashiwara operators is not easy to follow. 
Thus it is desirable to find another combinatorial construction which has 
much simpler description of  Young walls, reduced 
Young walls and Kashiwara operators. We think it is worth to mention that 
Lee gave a Young wall realization of the crystal $B(\infty)$ as the direct limit of the crystals $B(\lambda)$
using marginally large Young walls \cite{Lee2014}. Also, in \cite{Sav06}, Savage gave a combinatorial realization of 
higher level crystals for $U_q(A_{n}^{(1)})$ in terms of Young pyramids. 

\vskip 2mm 

In this paper, we take a different approach from the one given in \cite{KL06,KL09}. 
As usual, we will denote the size of a block by width $\times$ depth $\times$ height. 
Let us begin with a unit cube, i.e., a block of size $1 \times 1 \times 1$. 
Given a positive integer $l \ge 1$, we cut this unit cube vertically to obtain the blocks
of size $\frac{1}{l} \times 1 \times 1$, $\frac{1}{l} \times \frac{1}{2} \times 1$, $\frac{1}{2l} \times 1 \times 1$ and $\frac{1}{2l} \times \frac{1}{2} \times 1$. 
Note that the {\it vertical cutting} includes two vertical directions. 

\vskip 2mm 

We give a coloring to these blocks depending on the type of affine Cartan datum. Then we provide explicit 
stacking patterns for blocks and introduce the notion of $\delta$-slices, Young columns and 
reduced Young columns. Each reduced Young column corresponds to an element in a given perfect 
crystal.   We now define and describe new combinatorics of all these materials, 
which is much simpler and easier to deal with than the one given in \cite{KL06,KL09}.  
Based on this new combinatorics, the set of  all (reduced) Young columns is endowed with a crystal
structure. In particular, the description of Kashiwara operators, defined by adding and removing colored
blocks,  is quite explicit and easy to follow. 

\vskip 2mm 

Now we go ahead and provide a combinatorial construction of the level-$l$ highest weight crystal $B(\lambda)$.
A level-$l$ Young wall is an infinite sequence of level-$l$ Young columns extending to the left satisfying certain conditions  given in Section \ref{Level-$l$ Young Wall}. We introduce the notion of $\delta$-slices, {\it ground-state} walls of weight $\lambda$, {\it reduced} Young walls and define a crystal structure on the set of (reduced) Young walls.
As one of our main results, we prove that the crystal  consisting of reduced Young walls on $\lambda$  is isomorphic to the crystal $B(\lambda)$. 

\vskip 2mm

Finally, we take $l= 0$ and allow the {\it negative integers} to {\it count} the number of slices 
in the reduced Young columns.  
Under this expanded concept, we now follow a similar process of defining Young columns and Young walls for higher level case. We introduce the notion of virtual Young columns, extended Young columns, virtual Young walls and extended Young walls. Using the expression of the crystal $B(\infty)$ in \cite{KKM},
we obtain an explicit construction of $B(\infty)$ in terms of virtual Young walls and extended Young walls.

\vskip 2mm 

The combinatorial theory of level-$1$ Young walls has many applications in representation theory. In \cite{BK},  Brundan and Kleshchev demonstrated that 
the irreducible representations of the Hecke-Clifford superalgebra $\mathcal H_{N}(\zeta)$
with $\zeta$ a primitive $(2n+1)$-th root of unity are parametrized by the set of reduced  Young walls of type $A_{2n}^{(2)}$ with $N$ blocks. 
In \cite{KK08}, Kang and Kwon constructed a Fock space representation of all quantum 
affine algebras of classical type using combinatorics of Young walls. 
In addition, they discovered a generalized version of Lascoux-Leclerc-Thibon algorithm of computing the lower 
global basis elements corresponding to reduced Young walls. 
It is an interesting problem to understand the coefficients in the algorithm as in the Lascoux-Leclerc-Thibon-Ariki theory.  

\vskip 2mm

The Young wall realization is closely related to the representation theory of Khovanov-Lauda-Rouquier algebras. For instance, in \cite{OP14}, Oh and Park obtained the graded dimension formulas for the  cyclotomic quotient of Khovanov-Lauda-Rouquier algebra at $\Lambda_0$ of types $A_{2n}^{(2)}$ and $D_{n+1}^{(2)}$ in terms of Young walls. In \cite{BKOP}, Benkart, Kang, Oh and Park gave a combinatorial construction of irreducible modules over Khovanov-Lauda-Rouquier algebras of finite classcal type. In the forthcoming work, we will give a Young wall construction of finite dimensional irreducible modules over Khovanov-Lauda-Rouquier algebras associated with quantum affine algebras.

\vskip 2mm 

As is the case with combinatorics of level-$1$ Young walls, we expect that this new combinatorics of 
higher level Young walls will give rise to  a lot more interesting developments in the near future.

\vskip 2mm 

This paper is organized as follows.	In Section \ref{Quantum affine algebra}, we 
review some of the basic theory of quantum affine algebras, perfect crystals and path realization. 
In Section \ref{Level-$l$ Young Column}, we introduce the new notion of blocks, higher level slices and describe 
their stacking patterns, which yields the definition of higher level (reduced) Young columns. 
We define the action of Kashiwara operators carefully on the set of (reduced) Young columns
and show that the crystal consisting of reduced Young columns is isomorphic to the given 
perfect crystal. 
In Section \ref{Level-$l$ Young Wall}, we define the notions of ground-state walls, level-$l$ Young walls and level-$l$ reduced Young walls. We define the Kashiwara operators on level-$l$ Young walls using the tensor product rule. Thus the set of level-$l$ reduced Young walls is endowed with a crystal structure. 
We prove that the crystal consisting of reduced Young walls is isomorphic to the crystal $B(\lambda)$.
We illustrate the top parts of the level-$2$ highest weight crystals $B(\lambda)$
for types $B_3^{(1)}$, $A_2^{(2)}$, $D_3^{(2)}$ and $C_2^{(1)}$. 
The final section presents the new Young wall realization of $B(\infty)$ in terms of reduced virtual Young walls and reduced extended Young walls. As an example, we use reduced virtual Young walls to describe the top part of the crystal $B(\infty)$ of type $D_3^{(2)}$.
 
\vskip 3mm

{\bf Acknowledgements.}

\vskip 2mm 

Z. Fan was partially supported by the NSF of China grant 11671108, the NSF of Heilongjiang Province grant JQ2020A001, and the Fundamental Research Funds for the central universities. S.-J.  Kang was supported by  the NSF of China grant 11671108. Young Rock Kim was supported by the National Research Foundation of Korea (NRF) grant funded by the Korea government (MSIT) (No. 2021R1A2C1011467) and was supported by Hankuk University of Foreign Studies Research fund.

\vskip 5mm 

\section{Quantum affine algebras and perfect crystals}\label{Quantum affine algebra}
	Let $I=\{0,1,\cdots,n\}$ be a finite index set and let $(A,
	P, \Pi, P^\vee, \Pi^\vee)$ be an affine Cartan datum consisting of: 
	\begin{enumerate}
		\item a symmetrizable Cartan matrix $A = (a_{ij})_{i, j \in I}$ of affine type,
		\item a free abelian group $P$ of rank $n+2$,
		\item a $\mathbf Q$-linearly independent set $\Pi = \{\alpha_i \in P ~|~ i \in I\}$,
		\item $P^\vee = \text{Hom}_\mathbf Z (P, \mathbf Z)$,
		\item $\Pi^\vee = \{h_i \in P^\vee ~|~ i \in I\}$
	\end{enumerate}
such that $\langle h_i, \alpha_j \rangle = a_{ij}$ $(i, j \in I)$.
\vskip 2mm
We can write $P^\vee=\oplus_{i=0}^{n}\mathbf Z h_i\oplus\mathbf Z  d$	satisfying $\alpha_i(d) = \delta_{i,0}$ $(i \in I)$. Let $\mathfrak h = \mathbf Q \otimes_\mathbf Z P^\vee$. The
\textit{fundamental weights} $\Lambda_i \in \mathfrak h^\ast$ $(i \in I)$ are
given by
\begin{equation*}
	\begin{array}{ll}
		\Lambda_i (h_j) = \delta_{ij}, & \Lambda_i(d) = 0 \ \ \text{for $i, j \in I$}.
	\end{array}
\end{equation*}

There is a non-degenerate symmetric bilinear form $(\ ,\ )$ on $\mathfrak h^\ast$ such that
\begin{equation*}
	\langle h_i, \lambda \rangle = \frac{2(\alpha_i, \lambda)}{(\alpha_i, \alpha_i)} \ \ \text{for $\lambda \in P$, $i \in I$}.
\end{equation*}

Let $q$ be an indeterminate and let
\begin{equation*}
	\begin{array}{ll}
		q_i = q^{(\alpha_i, \alpha_i)/2}, & [n]_i = \displaystyle \frac{q_i^n - q_i^{-n}}{q_i - q_i^{-1}}.
	\end{array}
\end{equation*}

\begin{definition}\label{quantum affine algebra}
	The \textit{quantum affine algebra} $U_q(\mathfrak{g})$ associated with $(A, P, \Pi, $ $P^\vee, \Pi^\vee)$
	is the $\mathbf Q(q)$-algebra with 1 generated by $e_i$, $f_i$ $(i \in I)$, $q^h$ $(h \in P^\vee)$ subject to the defining relations
	\begin{equation*}
		\begin{aligned}
			& q^0 = 1, \ q^h q^{h'} = q^{h+h'} \ (h, h' \in P^\vee), \\
			& q^h e_j q^{-h} = q^{\langle h, \alpha_j \rangle} e_j, \\
			& q^h f_j q^{-h} = q^{-\langle h, \alpha_j \rangle} f_j \ \text{$(h \in P^\vee$, $j \in I)$}, \\
			& e_i f_j - f_j e_i = \delta_{ij} \frac{K_i - K_i^{-1}}{q_i - q_i^{-1}}, \\
			& \sum_{k=0}^{1 - a_{ij}} (-1)^k e_i^{(1 - a_{ij} - k)} e_j e_i^{(k)} = 0 \ (i \neq j), \\
			& \sum_{k=0}^{1 - a_{ij}} (-1)^k f_i^{(1 - a_{ij} - k)} f_j f_i^{(k)} = 0 \ (i \neq j),
		\end{aligned}
	\end{equation*}
	where $K_i = q^{\frac{(\alpha_i, \alpha_i)}{2} h_i}$, $e_i^{(k)} = e_i^k / [k]_i !$ and $f_i^{(k)} = f_i^k / [k]_i !$.
\end{definition}

The subalgebra $U_{q}'(\mathfrak g)$ generated by $e_{i}$, $f_{i}$, $K_{i}^{\pm 1}$ $(i \in I)$ is also called the {\it quantum affine algebra}. 

\vskip 2mm

Let $c = c_0 h_0 + \cdots + c_n h_n$ be the \textit{canonical central element} and let $\delta = d_0 \alpha_0 + \cdots + d_n \alpha_n$ be the \textit{null root}. Here the coefficients $c_{i}$ and $d_{i}\ (i\in I)$ are non-negative integers given in \cite[Section 6.2]{Kac90}. Then we have 

\begin{equation*}
		P= \bigoplus_{i \in I} \mathbf Z \Lambda_i \oplus \mathbf Z \left( \frac{1}{d_0} \delta \right).
\end{equation*}

Let $\text{cl}: P \to P / \mathbf Z(\frac{1}{d_0} \delta)$ be the canonical projection and let
\begin{equation*}
	P_\text{cl} := \text{cl}(P) = \bigoplus_{i \in I} \mathbf Z \text{cl}(\Lambda_i).
\end{equation*}

We define
\begin{equation*}
	P^+ := \{\lambda \in P ~|~ \langle h_i, \lambda \rangle \geq 0 \ \ \text{for all $i \in I$}\}
\end{equation*}
and set $P_\text{cl}^+ := \text{cl} (P^+)$, the set of {\it dominant integral weights}. 
The integer $\langle c, \lambda \rangle$ is called the \textit{level} of $\lambda \in P$. 

\vskip 2mm

The difference between $U_{q}(\mathfrak g)$ and $U_{q}'(\mathfrak g)$ are as follows. 
The non-trivial 
integrable irreducible $U_{q}(\mathfrak g)$-modules should be the  highest weight modules $V(\lambda)$ with dominant integral weights. They are of infinite dimension but  the weight spaces have finite dimension. On the other hand, $U_{q}'(\mathfrak g)$ may have finite dimensional irreducible modules but the weight spaces of $V(\lambda)$ should be of infinite dimension.

\vskip 2mm	
   In the following figures, we list the Dynkin diagrams, the canonical central
	elements and the null roots for affine Cartan data of types $A_{2n}^{(2)}$, $D_{n+1}^{(2)}$, $A_{2n-1}^{(2)}$, $D_{n}^{(1)}$, $B_n^{(1)}$ and $C_n^{(1)}$:
	
	\vskip 2mm

(1)  $A_{2n}^{(2)}$ ($n\geq 2$)

\vskip 2mm
\begin{center}
	\begin{texdraw}
		\small
		\drawdim em
		\setunitscale 0.3
		\move(0 0)
		\lcir r:1
		\move(1 0)\lvec(9 0)
		\move(10 0)\lcir r:1
		\move(11 0)\lvec(19 0)
		\htext(23.6 0){$\cdots$}
		\move(28 0)\lvec(36 0)
		\move(37 0)\lcir r:1
		\move(38 0.4)\lvec(46 0.4)
		\move(38 -0.4)\lvec(46 -0.4)
		\move(39.2 1.2)\lvec(38 0)\lvec(39.2 -1.2)
		\move(47 0)\lcir r:1
		\move(-1 0.4)\lvec(-9 0.4)
		\move(-1 -0.4)\lvec(-9 -0.4)
		\move(-10 0)\lcir r:1
		\move(-7.8 1.2)\lvec(-9 0)\lvec(-7.8 -1.2)
		\htext(-10 -3){$0$}
		\htext(0 -3){$1$}
		\htext(10 -3){$2$}
		\htext(37 -3){$n-1$}
		\htext(47 -3){$n$}
		\move(-12 -4.6)\move(49 3)
	\end{texdraw}
\end{center}

\vskip 2mm
\begin{equation*}
	\begin{aligned}\mbox{}
		& c = h_0+2h_1+\cdots+2h_{n-1}+2h_n,\\
		& \delta = 2\alpha_0 + 2\alpha_1 + \cdots + 2\alpha_{n-1}+\alpha_{n}.
	\end{aligned}
\end{equation*}

\vskip 5mm
(2) $D_{n+1}^{(2)}$ ($n\geq2$)

\vskip 2mm
\begin{center}
	\begin{texdraw}
		\small
		\drawdim em
		\setunitscale 0.3
		\move(0 0)
		\lcir r:1
		\move(1 0)\lvec(9 0)
		\move(10 0)\lcir r:1
		\move(11 0)\lvec(19 0)
		\htext(23.6 0){$\cdots$}
		\move(28 0)\lvec(36 0)
		\move(37 0)\lcir r:1
		\move(38 0.4)\lvec(46 0.4)
		\move(38 -0.4)\lvec(46 -0.4)
		\move(44.8 1.2)\lvec(46 0)\lvec(44.8 -1.2)
		\move(47 0)\lcir r:1
		\move(-1 0.4)\lvec(-9 0.4)
		\move(-1 -0.4)\lvec(-9 -0.4)
		\move(-10 0)\lcir r:1
		\move(-7.8 1.2)\lvec(-9 0)\lvec(-7.8 -1.2)
		\htext(-10 -3){$0$}
		\htext(0 -3){$1$}
		\htext(10 -3){$2$}
		\htext(37 -3){$n-1$}
		\htext(47 -3){$n$}
		\move(-12 -4.6)\move(49 3)
	\end{texdraw}
\end{center}

\vskip 2mm
\begin{equation*}
	\begin{aligned}\mbox{}
		& c = h_0+2h_1+\cdots+2h_{n-1}+h_n,\\
		& \delta=\alpha_0+\alpha_1+ \cdots+\alpha_{n-1}+\alpha_n.
	\end{aligned}
\end{equation*}

\vskip 5mm

(3) $A_{2n-1}^{(2)}$ ($n\geq 3$)
\vskip 2mm
\begin{center}
	\begin{texdraw}
		\small
		\drawdim em
		\setunitscale 0.3
		\move(0 0)
		\lcir r:1
		\move(1 0)\lvec(9 0)
		\move(10 0)\lcir r:1
		\move(11 0)\lvec(19 0)
		\htext(23.6 0){$\cdots$}
		\move(28 0)\lvec(36 0)
		\move(37 0)\lcir r:1
		\move(38 0.4)\lvec(46 0.4)
		\move(38 -0.4)\lvec(46 -0.4)
		\move(39.2 1.2)\lvec(38 0)\lvec(39.2 -1.2)
		\move(47 0)\lcir r:1
		\move(10 1)\lvec(10 9)
		\move(10 10)\lcir r:1
		\htext(0 -3){$1$}
		\htext(10 -3){$2$}
		\htext(37 -3){$n-1$}
		\htext(47 -3){$n$}
		\htext(12.5 11){$0$}
		\move(-2 -4.6)\move(49 12.6)
	\end{texdraw}
\end{center}

\vskip 2mm
\begin{equation*}
	\begin{aligned}\mbox{}
		& c = h_0 + h_1 + 2 h_2 +\cdots + 2h_{n-1}+2 h_n, \\
		& \delta = \alpha_0 + \alpha_1 
		+ 2\alpha_2 + \cdots + 2\alpha_{n-1} + \alpha_n.
	\end{aligned}
\end{equation*}

\vskip 5mm
(4) $D_n^{(1)}$ ($n\geq 4$)

\vskip 2mm
\begin{center}
	\begin{texdraw}
		\small
		\drawdim em
		\setunitscale 0.3
		\move(0 0)
		\lcir r:1
		\move(1 0)\lvec(9 0)
		\move(10 0)\lcir r:1
		\move(11 0)\lvec(19 0)
		\htext(23.6 0){$\cdots$}
		\move(28 0)\lvec(36 0)
		\move(37 0)\lcir r:1
		\move(38 0)\lvec(46 0)
		\move(47 0)\lcir r:1
		\move(10 1)\lvec(10 9)\move(10 10)\lcir r:1
		\move(37 1)\lvec(37 9)\move(37 10)\lcir r:1
		\htext(0 -3){$1$}
		\htext(10 -3){$2$}
		\htext(37 -3){$n-2$}
		\htext(47 -3){$n$}
		\htext(12.5 11){$0$}
		\htext(42.5 11){$n-1$}
		\move(-2 -4.6)\move(49 12.6)
	\end{texdraw}
\end{center}

\vskip 2mm
\begin{equation*}
	\begin{aligned}\mbox{}
		& c = h_0+h_1+2h_2+\cdots+2h_{n-2}+h_{n-1}+h_n,\\
		& \delta = \alpha_0+\alpha_1+2\alpha_2+\cdots+2\alpha_{n-2}+\alpha_{n-1}
		+\alpha_n.
	\end{aligned}
\end{equation*}

\vskip 5mm
(5) $B_n^{(1)}$ ($n\geq3$)

\vskip 2mm
\begin{center}
	\begin{texdraw}
		\small
		\drawdim em
		\setunitscale 0.3
		\move(0 0)
		\lcir r:1
		\move(1 0)\lvec(9 0)
		\move(10 0)\lcir r:1
		\move(11 0)\lvec(19 0)
		\htext(23.6 0){$\cdots$}
		\move(28 0)\lvec(36 0)
		\move(37 0)\lcir r:1
		\move(38 0.4)\lvec(46 0.4)
		\move(38 -0.4)\lvec(46 -0.4)
		\move(44.8 1.2)\lvec(46 0)\lvec(44.8 -1.2)
		\move(47 0)\lcir r:1
		\move(10 1)\lvec(10 9)
		\move(10 10)\lcir r:1
		\htext(0 -3){$1$}
		\htext(10 -3){$2$}
		\htext(37 -3){$n-1$}
		\htext(47 -3){$n$}
		\htext(12.5 11){$0$}
		\move(-2 -4.6)\move(49 12.6)
	\end{texdraw}
\end{center}

\vskip 2mm
\begin{equation*}
	\begin{aligned}\mbox{}
		& c = h_0 + h_1 + 2h_2 + \cdots + 2h_{n-1} + h_n,\\
		& \delta = \alpha_0+\alpha_1+2\alpha_2+\cdots+2\alpha_{n-1}+2\alpha_n.
	\end{aligned}
\end{equation*}

\vskip 5mm
(6) $C_{n}^{(1)}$ ($n\geq2$)

\vskip 2mm
\begin{center}
	\begin{texdraw}
		\small
		\drawdim em
		\setunitscale 0.3
		\move(0 0)
		\lcir r:1
		\move(1 0)\lvec(9 0)
		\move(10 0)\lcir r:1
		\move(11 0)\lvec(19 0)
		\htext(23.6 0){$\cdots$}
		\move(28 0)\lvec(36 0)
		\move(37 0)\lcir r:1
		\move(38 0.4)\lvec(46 0.4)
		\move(38 -0.4)\lvec(46 -0.4)
		
		\move(39.2 1.2)\lvec(38 0)\lvec(39.2 -1.2)
		\move(47 0)\lcir r:1
		\move(-1 0.4)\lvec(-9 0.4)
		\move(-1 -0.4)\lvec(-9 -0.4)
		\move(-10 0)\lcir r:1
		\move(-2.1 1.2)\lvec(-1 0)\lvec(-2.1 -1.2)
		\htext(-10 -3){$0$}
		\htext(0 -3){$1$}
		\htext(10 -3){$2$}
		\htext(37 -3){$n-1$}
		\htext(47 -3){$n$}
		\move(-12 -4.6)\move(49 3)
	\end{texdraw}
\end{center}

\vskip 2mm
\begin{equation*}
	\begin{aligned}\mbox{}
		& c = h_0+h_1+\cdots+h_{n-1}+h_n,\\
		& \delta=\alpha_0+2\alpha_1+ \cdots+2\alpha_{n-1}+\alpha_n.
	\end{aligned}
\end{equation*}

\begin{definition}
	A $U_q(\mathfrak g)$-crystal (resp. a $U_{q}'(\mathfrak g)$-crystal) is a set $B$ together with the maps $\tilde{e}_i$, $\tilde{f}_i: B \to B ~ \cup ~ \{0\}$, $\varepsilon_i$, $\varphi_i: B \to \mathbf Z ~ \cup ~ \{-\infty\}$, $\wt: B \to P$ (resp. $\overline{\wt}: B \to P_\text{cl}$) satisfying the following conditions:
	\begin{enumerate}
		\item for each $i \in I$, we have $\varphi_i(b) = \varepsilon_i(b) + \langle h_i, \wt(b) \rangle$ (resp. $\varphi_i(b) = \varepsilon_i(b) + \langle h_i, \overline{\wt}(b) \rangle$),
		\item $\wt(\tilde{e}_i b) = \wt (b) + \alpha_i$ (resp.  $\overline{\wt}(\tilde{e}_i b) = \overline{\wt}(b) + \alpha_i$) if $\tilde{e}_i b \in B$,
		\item $\wt(\tilde{f}_i b) = \wt (b) - \alpha_i$ (resp.  $\overline{\wt}(\tilde{f}_i b) = \overline{\wt}(b) - \alpha_i$) if $\tilde{f}_i b \in B$,
		\item $\varepsilon_i (\tilde{e}_i b) = \varepsilon_i (b) - 1$, $\varphi_i (\tilde{e}_i b) = \varphi_i (b) + 1$ if $\tilde{e}_i b \in B$,
		\item $\varepsilon_i (\tilde{f}_i b) = \varepsilon_i (b) + 1$, $\varphi_i (\tilde{f}_i b) = \varphi_i (b) - 1$ if $\tilde{f}_i b \in B$,
		\item $\tilde{f}_i b = b^{'}$ if and only if $b = \tilde{e}_i b^{'}$ for all $b, b^{'} \in B$, $i \in I$,
		\item if $\varphi_i (b) = -\infty$, then $\tilde{e}_i b = \tilde{f}_i b = 0$.
	\end{enumerate}
\end{definition}

Let $B$ be a $U_q'(\mathfrak g)$-crystal. For an element $b \in B$, we define
\begin{equation*}
		\varepsilon (b)   = \sum_{i \in I} \varepsilon_i (b) \text{cl} (\Lambda_i), \quad
		\varphi (b)   = \sum_{i \in I} \varphi_i (b) \text{cl} (\Lambda_i).
\end{equation*}

For simplicity, we often write $\Lambda_i$ for $\text{cl}(\Lambda_i)$ $(i \in I)$ and $\wt (b)$ for $\overline{\wt}(b)$ if there is no danger of confusion.

\vskip 2mm

Let $B_1$, $B_2$ be $U_q(\mathfrak g)$-crystals (or $U_{q}'(\mathfrak g)$-crystals). The \textit{tensor product} $B_1 \otimes B_2 = B_1 \times B_2$ is defined by

\begin{equation} \label{eq:tensor}
	\begin{aligned}
		\wt (b_1 \otimes b_2)           & = \wt (b_1) + \wt (b_2)\ ({\rm resp}.\ \overline{\wt}(b_1\otimes b_2)= \overline{\wt}(b_1)+\overline{\wt}(b_2)), \\
		\varepsilon_i (b_1 \otimes b_2)       & = \max (\varepsilon_i (b_1), \varepsilon_i (b_2) - \langle h_i, \wt (b_1) \rangle), \\
		\varphi_i (b_1 \otimes b_2)       & = \max (\varphi_i (b_2), \varphi_i (b_1) + \langle h_i, \wt (b_2) \rangle), \\
		\tilde{e}_i (b_1 \otimes b_2)   & = \begin{cases}
			\tilde{e}_i b_1 \otimes b_2 & \text{if $\varphi_i (b_1) \geq \varepsilon_i (b_2)$}, \\
			b_1 \otimes \tilde{e}_i b_2 & \text{if $\varphi_i (b_1) < \varepsilon_i (b_2)$},
		\end{cases} \\
		\tilde{f}_i (b_1 \otimes b_2)   & = \begin{cases}
			\tilde{f}_i b_1 \otimes b_2 & \text{if $\varphi_i (b_1) > \varepsilon_i (b_2)$}, \\
			b_1 \otimes \tilde{f}_i b_2 & \text{if $\varphi_i (b_1) \leq \varepsilon_i (b_2)$}.
		\end{cases}
	\end{aligned}
\end{equation}

\begin{definition}
	Let $l$ be a positive integer. A finite $U_q'(\mathfrak g)$-crystal\ $B$ is called a \textit{perfect crystal of level $l$} if $B$ satisfies the following conditions:
	\begin{enumerate}
		\item there exists a finite dimensional $U_q'(\mathfrak g)$-module with crystal $B$,
		\item $B \otimes B$ is connected,
		\item there is $\lambda \in P_\text{cl}$ such that $\wt(B) \subset \lambda - \sum_{i \neq 0} \mathbf Z_{\geq 0} ~ \alpha_i$, $\# (B_\lambda) = 1$,
		\item for any $b \in B$, we have $\langle c, \varepsilon_i(b) \rangle \geq l$,
		\item for each $\lambda \in P_\text{cl}^+$ with $\langle c, \lambda \rangle = l$, there exist uniquely determined elements $b^\lambda$ and $b_\lambda$ such that $\varepsilon(b^\lambda) = \varphi (b_\lambda) = \lambda$.
	\end{enumerate}
\end{definition}

The elements $b^\lambda$ and $b_\lambda$ are called the \textit{minimal vectors}.

\vskip 2mm

Let $V(\lambda)$ be the irreducible highest weight $U_q'(\mathfrak g)$-module with a dominant integral weight $\lambda \in P^+_{\text{cl}}$ and let $B(\lambda)$ be the crystal of $V(\lambda)$. Then by \cite{KMN1,KMN2}, we have a $U_q'(\mathfrak g)$-crystal isomorphism
\begin{equation} \label{eqn: Psi}
	\begin{array}{rccc}
		\Psi_\lambda:   & B(\lambda)  & \stackrel{\sim}{\longrightarrow}  & B(\varepsilon(b_\lambda)) \otimes B \\
		& u_\lambda     & \longmapsto                       & u_{\varepsilon(b_\lambda)} \otimes b_\lambda,
	\end{array}
\end{equation}
where $u_\lambda$ denotes the highest weight vector of $B(\lambda)$. Applying the isomorphism (\ref{eqn: Psi}) repeatedly, we obtain a sequence of $U_q(\mathfrak g)$-crystal isomorphisms
\begin{equation*}
	\begin{array}{ccccccl}
		B(\lambda) & \stackrel{\sim}{\longrightarrow} & B(\lambda_1) \otimes B & \stackrel{\sim}{\longrightarrow} & B(\lambda_2) \otimes B \otimes B & \stackrel{\sim}{\longrightarrow} & \cdots \\
		u_\lambda & \longmapsto & u_{\lambda_1} \otimes b_0 & \longmapsto & u_{\lambda_2} \otimes b_1 \otimes b_0 & \longmapsto & \cdots,
	\end{array}
\end{equation*}
where
\begin{equation*}
	\begin{array}{ll}
		\lambda_0 = \lambda,        & b_0 = b_{\lambda_0} = b_\lambda, \\
		\lambda_{k+1} = \varepsilon(b_k), & b_{k+1} = b_{\lambda_{k+1}} \ \ \text{for $k \geq 0$}.
	\end{array}
\end{equation*}

\begin{definition}
The infinite sequence $\mathbf{b}_\lambda =(b_k)_{k \geq 0} = \cdots \otimes b_k \otimes \cdots \otimes b_1 \otimes b_0 \in B^{\otimes \infty}$ is called the \textit{ground-state path of weight $\lambda$}. The infinite sequence $\mathbf{p} = (p_k)_{k \geq 0} \in B^{\otimes \infty}$ is called a \textit{$\lambda$-path} if $p_k = b_k$ for $k \gg 0$.	
	\end{definition}

Let $\mathcal{P}(\lambda)$ denote the set of all $\lambda$-paths. By the tensor
product rule \eqref{eq:tensor}, $\mathcal{P}(\lambda)$ has a $U_q'(\mathfrak g)$-crystal
structure. It is proved in \cite{KMN1,KMN2} that $\mathcal{P}(\lambda)$ provides
a path realization of $B(\lambda)$ as a $U_q'(\mathfrak g)$-crystal.

\begin{proposition} [\cite{KMN1,KMN2}]
	For each $\lambda \in P^+_{\text{cl}}$, there exists a $U_q'(\mathfrak g)$-crystal isomorphism
	\begin{equation*}
		\begin{array}{ll}
			\Phi_\lambda: B(\lambda) \longrightarrow \mathcal{P}(\lambda) \ \ \  \text{given by} & u_\lambda \longmapsto \mathbf{b}_\lambda.
		\end{array}
	\end{equation*}
\end{proposition}

For the quantum affine algebras of types $A_{2n}^{(2)}$, $D_{n+1}^{(2)}$, $A_{2n-1}^{(2)}$, $D_{n}^{(1)}$, $B_n^{(1)}$ and $C_{n}^{(1)}$, it was
shown in \cite{KKM,KMN2} that there exists a
\emph{coherent family of perfect crystals} $\{ \mathcal B^{(l)} |\ l \in
\mathbf Z_{>0} \}$. We give an explicit description of these perfect crystals as follows:

\begin{enumerate}
\item $A_{2n}^{(2)}$ $(n\geq 1)$\\[1.5mm]
$\mathcal B^{(l)} =
\Big\{ (x_1,\dots,x_n | \bar{x}_n,\dots, \bar{x}_1) \,\Big\vert\,
x_i, \bar{x}_i \in \mathbf Z_{\geq0},\
\textstyle\sum_{i=1}^{n} (x_i + \bar{x}_i) \leq \textnormal{$l$} \Big\}$,
\item $D_{n+1}^{(2)}$ $(n\geq 2)$\\[1.5mm]
$\mathcal B^{(l)} =
\left\{
(x_1,\dots,x_n |x_0| \bar{x}_n,\dots, \bar{x}_1) \,\Bigg\vert\,
\begin{aligned}
	& x_0 = 0 \textup{ or } 1,\ x_i, \bar{x}_i \in \mathbf Z_{\geq0},\\
	& x_0 + \textstyle\sum_{i=1}^{n} (x_i + \bar{x}_i) \leq \textnormal{$l$}
\end{aligned}
\right\}$,
	\item $C_n^{(1)}$ $(n \geq 2)$\\[1.5mm]
$\mathcal B^{(l)} =
\Big\{ (x_1,\dots,x_n | \bar{x}_n,\dots, \bar{x}_1) \,\Big\vert\,
x_i, \bar{x}_i \in \mathbf Z_{\geq0},\
\textnormal{$2l$} \geq \textstyle\sum_{i=1}^{n} (x_i +\bar{x}_i) \in 2\mathbf Z \Big\}$,
\item $A_{2n-1}^{(2)}$ $(n\geq 3)$\\[1.5mm]
	$\mathcal B^{(l)} =
	\Big\{ (x_1,\dots,x_n | \bar{x}_n,\dots, \bar{x}_1) \,\Big\vert\,
	x_i, \bar{x}_i \in \mathbf Z_{\geq0},\
	\textstyle\sum_{i=1}^{n} (x_i + \bar{x}_i) = \textnormal{$l$} \Big\}$,
	\item $D_{n}^{(1)}$ $(n\geq 4)$\\[1.5mm]
	$\mathcal B^{(l)} =
	\left\{
	(x_1,\dots,x_n| \bar{x}_n,\dots, \bar{x}_1) \,\Bigg\vert\,
	\begin{aligned}
		& x_n = 0 \textup{ or } \bar{x}_n=0,\ x_i, \bar{x}_i \in \mathbf Z_{\geq0},\\
		& \textstyle\sum_{i=1}^{n} (x_i + \bar{x}_i) = \textnormal{$l$}
	\end{aligned}
	\right\}$,
	\item $B_n^{(1)}$ $(n\ge 3)$\\[1.5mm]
$\mathcal B^{(l)} =
\left\{
(x_1,\dots,x_n |x_0| \bar{x}_n,\dots, \bar{x}_1) \, \Bigg\vert\,
\begin{aligned}
	& x_0=0 \textup{ or } 1,\ x_i, \bar{x}_i \in \mathbf Z_{\geq0}, \\
	& x_0 + \textstyle\sum_{i=1}^{n} (x_i + \bar{x}_i) = \textnormal{$l$}
\end{aligned}
\right\}$.

\end{enumerate}

\vskip 2mm

 For the convenience of readers, we give an explicit description of the maps $\wt$, $\tilde{e}_i$, $\tilde{f}_i$,
$\varepsilon_i$ and $\varphi_i$ $(i\in I)$ for the perfect crystal ${\mathcal B}^{(l)}$ of type $B^{(1)}_{n}$. Let  $b = (x_1,\dots,x_n|x_0 |\bar{x}_n,\dots, \bar{x}_1) \in
\mathcal B^{(l)}$ such that $x_0=0$ or $1$, $x_i, \bar{x}_i \in
\mathbf Z_{\ge 0}$, and $x_0 + \sum_{i=1}^{n}(x_i+\bar{x}_i) = l$.
\vskip 2mm
For $i=0$ and $i=n$, the Kashiwara operators $\tilde{f}_i$ are given by
\begin{equation}\label{Kashiwara operators1} 
	\begin{aligned}
\tilde{f}_0 b &=
\begin{cases}
	(x_1,x_2+1,x_3,\dots,x_n|x_0|\bar{x}_n,\dots,\bar{x}_2,\bar{x}_1-1)
	&
	\textnormal{if $x_2 \geq \bar{x}_2$,}\\
	(x_1+1,x_2,\dots,x_n|x_0|\bar{x}_n,\dots,\bar{x}_3,\bar{x}_2-1,\bar{x}_1)
	&
	\textnormal{if $x_2 < \bar{x}_2$,}
\end{cases}\\	
		\tilde{f}_n b & =
\begin{cases}
	(x_1,\dots,x_{n-1},x_n - 1|x_0+1|\bar{x}_n,\dots,\bar{x}_1) &
	\textnormal{if $x_0=0$,}\\
	(x_1,\dots,x_n|x_0-1|\bar{x}_n + 1,\bar{x}_{n-1},\dots,\bar{x}_1)
	& \textnormal{if $x_0=1$.}
\end{cases}
	\end{aligned}
\end{equation}

For $i=1,\dots,n-1$,  the Kashiwara operators $\tilde{f}_i$ are given by
\begin{equation}\label{Kashiwara operators2} 
	\begin{aligned}
		\tilde{f}_i b & =
		\begin{cases}
			(x_1,\dots,x_i-1,x_{i+1}+1,\dots,x_n|x_0|
			\bar{x}_n,\dots,\bar{x}_1) &
			\textnormal{if $x_{i+1} \geq \bar{x}_{i+1}$,}\\
			(x_1,\dots,x_n|x_0| \bar{x}_n,\dots,,\bar{x}_{i+1}-1,
			\bar{x}_{i}+1,\dots,\bar{x}_1) & \textnormal{if $x_{i+1} <
				\bar{x}_{i+1}$.}
		\end{cases}
	\end{aligned}
\end{equation}

The action of Kashiwara operators $\tilde{e}_i$ is given by $\tilde{e}_i b=b'$ if and only if $\tilde{f}_i b'=b$.

The maps $\varepsilon_i$, $\varphi_i$
 and $\wt$  are given by
\begin{align*}
	&\varepsilon_0 (b) =
	x_1+(x_2-\bar{x}_2)_+,\ \varepsilon_n (b) =
	2\bar{x}_n+x_0,\\
	&\varphi_0 (b) =
	\bar{x}_1+(\bar{x}_2-x_2)_+,\ \varphi_n (b) =
	2x_n+x_0,\\
	&\varepsilon_i (b) =
	\bar{x}_i+({x}_{i+1}-\bar{x}_{i+1})_+ \quad (i=1,\dots,n-1),\\
    &\varphi_i (b) =
	x_i + (\bar{x}_{i+1} - x_{i+1})_+ \quad (i=1,\dots,n-1),\\
	&\wt(b) =
	\sum_{i=0}^n (\varphi_i(b) - \varepsilon_i(b))\Lambda_i,
\end{align*}
where $(a)_+ = \max(0,a)$ for $a\in\mathbf Z$.

\vskip 2mm 

For other types, we can find precise definitions of maps $\tilde{e}_i$ $\tilde{f}_i$ , $\varepsilon_i$, $\varphi_i$ and $\wt$ in \cite{KKM,KMOTU}.

\vskip 5mm

\section{Higher level  Young columns}\label{Level-$l$ Young Column}

\vskip 2mm 

In this section, we define the level-$l$ Young columns and level-$l$ reduced Young columns. We also 
define the crystal structure on the set ${\mathcal C}^{(l)}$ of level-$l$ reduced Young columns.
Then we will prove that ${\mathcal C}^{(l)}$ has the same crystal structure as ${\mathcal B}^{(l)}$.

 \vskip 3mm 
 
 \subsection{Level-$l$ stacking patterns and Young columns}\label{stacking patterns and Young Columns}
 \hfill  
 
 \vskip 2mm
  
 The colored blocks we will work with are one of the following sizes which are obtained by splitting the unit cube: (we use width $\times$ depth $\times$ height)
 
 \vskip 5mm 
 
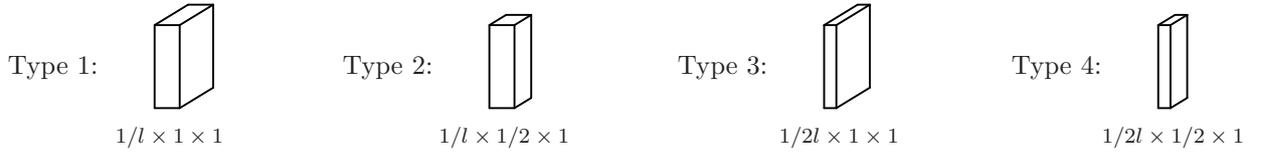
 \begin{center}
 	\begin{texdraw}
 		\fontsize{3}{3}\selectfont
 		\drawdim em
 		\setunitscale 2.3
	\move(0 0)
\bsegment
 	
\move(1.4 0)\lvec(2 0)

\move(1.4 2)\lvec(2 2)

\move(2 0)\lvec(2 2)

\move(2 2)\lvec(2.8 2.5)

\move(2 0)\lvec(2.8 0.5)

\move(2.8 2.5)\lvec(2.8 0.5)

\move(2.2 2.5)\lvec(2.8 2.5)

\move(1.4 0)\lvec(1.4 2)
\move(1.4 2)\lvec(2.2 2.5)

\htext(-1 1){\fontsize{10}{10}\selectfont{Type 1:}}

\htext(1.7 -0.7){\fontsize{8}{8}\selectfont{  $1/l  \times 1 \times 1$}}

\esegment

\move(8 0)
\bsegment

\move(1.4 0)\lvec(2 0)
\move(1.4 2)\lvec(2 2)
\move(2 0)\lvec(2 2)
\move(2 2)\lvec(2.4 2.25)
\move(2 0)\lvec(2.4 0.25)
\move(2.4 0.25)\lvec(2.4 2.25)
\move(2.4 2.25)\lvec(1.8 2.25)
\move(1.4 0)\lvec(1.4 2)
\move(1.4 2)\lvec(1.8 2.25)

\htext(-1 1){\fontsize{10}{10}\selectfont{Type 2:}}

\htext(1.7 -0.7){\fontsize{8}{8}\selectfont{  $1/l  \times 1/2 \times 1$}}

\esegment 	
	
 		\move(16 0)
 		\bsegment
 		
 			\move(1.4 0)\lvec(1.7 0)
 		
 		\move(1.4 2)\lvec(1.7 2)
 		
 		\move(1.7 0)\lvec(1.7 2)
 		
 		\move(1.7 2)\lvec(2.5 2.5)
 		
 		\move(1.7 0)\lvec(2.5 0.5)
 		
 		\move(2.5 2.5)\lvec(2.5 0.5)
 		
 		\move(2.2 2.5)\lvec(2.5 2.5)
 		
 		\move(1.4 0)\lvec(1.4 2)
 		\move(1.4 2)\lvec(2.2 2.5)
 		
 		\htext(-1 1){\fontsize{10}{10}\selectfont{Type 3:}}	
 		
 		\htext(1.7 -0.7){\fontsize{8}{8}\selectfont{  $1/2l  \times 1 \times 1$}}
 
 		\esegment
 		
	\move(24 0)
\bsegment

  	\move(1.4 0)\lvec(1.7 0)
 \move(1.4 2)\lvec(1.7 2)
 \move(1.7 0)\lvec(1.7 2)
 \move(1.7 2)\lvec(2.1 2.25)
 \move(1.7 0)\lvec(2.1 0.25)
 \move(2.1 0.25)\lvec(2.1 2.25)
 \move(2.1 2.25)\lvec(1.8 2.25)
 \move(1.4 0)\lvec(1.4 2)
 \move(1.4 2)\lvec(1.8 2.25)
 
 \htext(-1 1){\fontsize{10}{10}\selectfont{Type 4:}}
 
 \htext(1.7 -0.7){\fontsize{8}{8}\selectfont{  $1/2l  \times 1/2 \times 1$}}

\esegment 		
 	\end{texdraw}
 \end{center}  
 
 \vskip 1.5mm
 
 Here, $l$ stands for the level. 
 
 \vskip 2mm
 
 For each affine Cartan datum, we use different coloring and shapes of blocks given below. 
 
 \vskip 2mm 
 
 (1) $ A_{2n}^{(2)}\ (1\le i\le n)$
 
 \vskip 3mm
 
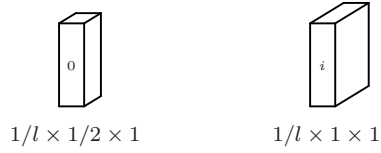
  \begin{center}
 	\begin{texdraw}
 		\fontsize{3}{3}\selectfont
 		\drawdim em
 		\setunitscale 2.3
 		\move(0 0)
 		\bsegment
 		
 		\move(1.4 0)\lvec(2 0)
 		\move(1.4 2)\lvec(2 2)
 		\move(2 0)\lvec(2 2)
 		\move(2 2)\lvec(2.4 2.25)
 		\move(2 0)\lvec(2.4 0.25)
 		\move(2.4 0.25)\lvec(2.4 2.25)
 		\move(2.4 2.25)\lvec(1.8 2.25)
 		\move(1.4 0)\lvec(1.4 2)
 		\move(1.4 2)\lvec(1.8 2.25)
 			\htext(1.7 1){$0$}

 		\htext(1.7 -0.7){\fontsize{8}{8}\selectfont{  $1/l  \times 1/2 \times 1$}}
 		
 		\esegment
 		
 		\move(6 0)
 		\bsegment
 		\move(1.4 0)\lvec(2 0)
 		
 		\move(1.4 2)\lvec(2 2)
 		
 		\move(2 0)\lvec(2 2)
 		
 		\move(2 2)\lvec(2.8 2.5)
 		
 		\move(2 0)\lvec(2.8 0.5)
 		
 		\move(2.8 2.5)\lvec(2.8 0.5)
 		
 		\move(2.2 2.5)\lvec(2.8 2.5)
 		
 		\move(1.4 0)\lvec(1.4 2)
 		\move(1.4 2)\lvec(2.2 2.5)

 			\htext(1.7 1){$i$}
 		\htext(1.7 -0.7){\fontsize{8}{8}\selectfont{  $1/l  \times 1 \times 1$}}
 		\esegment 	
 		
 	\end{texdraw}
 \end{center}

 \vskip 5mm 
 
 (2) $D_{n+1}^{(2)}\ (1\le i\le n-1)$
 
 \vskip 3mm 
 
 \begin{center}
 	\begin{texdraw}
 		\fontsize{3}{3}\selectfont
 		\drawdim em
 		\setunitscale 2.3
 		\move(0 0)
 		\bsegment
 		
 		\move(1.4 0)\lvec(2 0)
 		\move(1.4 2)\lvec(2 2)
 		\move(2 0)\lvec(2 2)
 		\move(2 2)\lvec(2.4 2.25)
 		\move(2 0)\lvec(2.4 0.25)
 		\move(2.4 0.25)\lvec(2.4 2.25)
 		\move(2.4 2.25)\lvec(1.8 2.25)
 		\move(1.4 0)\lvec(1.4 2)
 		\move(1.4 2)\lvec(1.8 2.25)
 		\htext(1.7 1){$0$}

 		\htext(1.7 -0.7){\fontsize{8}{8}\selectfont{  $1/l  \times 1/2 \times 1$}}
 		
 		\esegment
 		
 		\move(6 0)
 		\bsegment
 		\move(1.4 0)\lvec(2 0)
 	\move(1.4 2)\lvec(2 2)
 	\move(2 0)\lvec(2 2)
 	\move(2 2)\lvec(2.4 2.25)
 	\move(2 0)\lvec(2.4 0.25)
 	\move(2.4 0.25)\lvec(2.4 2.25)
 	\move(2.4 2.25)\lvec(1.8 2.25)
 	\move(1.4 0)\lvec(1.4 2)
 	\move(1.4 2)\lvec(1.8 2.25)
 	\htext(1.7 1){$n$}

 	\htext(1.7 -0.7){\fontsize{8}{8}\selectfont{  $1/l  \times 1/2 \times 1$}}
 		\esegment 	
 	
 		\move(12 0)
 	\bsegment
 	\move(1.4 0)\lvec(2 0)
 	
 	\move(1.4 2)\lvec(2 2)
 	
 	\move(2 0)\lvec(2 2)
 	
 	\move(2 2)\lvec(2.8 2.5)
 	
 	\move(2 0)\lvec(2.8 0.5)
 	
 	\move(2.8 2.5)\lvec(2.8 0.5)
 	
 	\move(2.2 2.5)\lvec(2.8 2.5)
 	
 	\move(1.4 0)\lvec(1.4 2)
 	\move(1.4 2)\lvec(2.2 2.5)
 	
 	\htext(1.7 1){$i$}
 	\htext(1.7 -0.7){\fontsize{8}{8}\selectfont{  $1/l  \times 1 \times 1$}}
 	\esegment 	
 		
 	\end{texdraw}
 \end{center}

 \vskip 5mm

 (3) $ A_{2n-1}^{(2)}\ (2\le i\le n)$
 
 \vskip 3mm 
 
  \begin{center}
 	\begin{texdraw}
 		\fontsize{3}{3}\selectfont
 		\drawdim em
 		\setunitscale 2.3
 		\move(0 0)
 		\bsegment
 		
 		\move(1.4 0)\lvec(2 0)
 		\move(1.4 2)\lvec(2 2)
 		\move(2 0)\lvec(2 2)
 		\move(2 2)\lvec(2.4 2.25)
 		\move(2 0)\lvec(2.4 0.25)
 		\move(2.4 0.25)\lvec(2.4 2.25)
 		\move(2.4 2.25)\lvec(1.8 2.25)
 		\move(1.4 0)\lvec(1.4 2)
 		\move(1.4 2)\lvec(1.8 2.25)
 		\htext(1.7 1){$0$}

 		\htext(1.7 -0.7){\fontsize{8}{8}\selectfont{  $1/l  \times 1/2 \times 1$}}
 		
 		\esegment
 		
 		\move(6 0)
 		\bsegment
 		\move(1.4 0)\lvec(2 0)
 		\move(1.4 2)\lvec(2 2)
 		\move(2 0)\lvec(2 2)
 		\move(2 2)\lvec(2.4 2.25)
 		\move(2 0)\lvec(2.4 0.25)
 		\move(2.4 0.25)\lvec(2.4 2.25)
 		\move(2.4 2.25)\lvec(1.8 2.25)
 		\move(1.4 0)\lvec(1.4 2)
 		\move(1.4 2)\lvec(1.8 2.25)
 		\htext(1.7 1){$1$}

 		\htext(1.7 -0.7){\fontsize{8}{8}\selectfont{  $1/l  \times 1/2 \times 1$}}
 		\esegment 	
 		
 		\move(12 0)
 		\bsegment
 		\move(1.4 0)\lvec(2 0)
 		
 		\move(1.4 2)\lvec(2 2)
 		
 		\move(2 0)\lvec(2 2)
 		
 		\move(2 2)\lvec(2.8 2.5)
 		
 		\move(2 0)\lvec(2.8 0.5)
 		
 		\move(2.8 2.5)\lvec(2.8 0.5)
 		
 		\move(2.2 2.5)\lvec(2.8 2.5)
 		
 		\move(1.4 0)\lvec(1.4 2)
 		\move(1.4 2)\lvec(2.2 2.5)
 		
 		\htext(1.7 1){$i$}
 		\htext(1.7 -0.7){\fontsize{8}{8}\selectfont{  $1/l  \times 1 \times 1$}}
 		\esegment 	
 		
 	\end{texdraw}
 \end{center}
 
 \vskip 5mm 
 
 (4) $D_{n}^{(1)}\ (2\le i\le n-2)$
 
 \vskip 3mm 
 
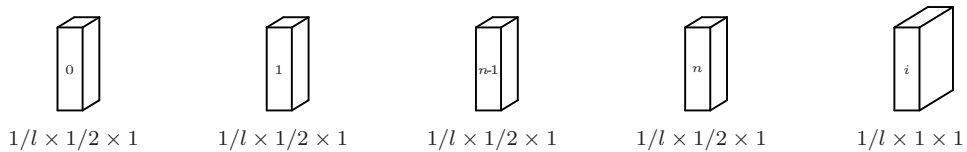
  \begin{center}
 	\begin{texdraw}
 		\fontsize{3}{3}\selectfont
 		\drawdim em
 		\setunitscale 2.3
 		\move(0 0)
 		\bsegment
 		
 		\move(1.4 0)\lvec(2 0)
 		\move(1.4 2)\lvec(2 2)
 		\move(2 0)\lvec(2 2)
 		\move(2 2)\lvec(2.4 2.25)
 		\move(2 0)\lvec(2.4 0.25)
 		\move(2.4 0.25)\lvec(2.4 2.25)
 		\move(2.4 2.25)\lvec(1.8 2.25)
 		\move(1.4 0)\lvec(1.4 2)
 		\move(1.4 2)\lvec(1.8 2.25)
 		\htext(1.7 1){$0$}

 		\htext(1.7 -0.7){\fontsize{8}{8}\selectfont{  $1/l  \times 1/2 \times 1$}}
 		
 		\esegment
 		
 		\move(5 0)
 		\bsegment
 		\move(1.4 0)\lvec(2 0)
 		\move(1.4 2)\lvec(2 2)
 		\move(2 0)\lvec(2 2)
 		\move(2 2)\lvec(2.4 2.25)
 		\move(2 0)\lvec(2.4 0.25)
 		\move(2.4 0.25)\lvec(2.4 2.25)
 		\move(2.4 2.25)\lvec(1.8 2.25)
 		\move(1.4 0)\lvec(1.4 2)
 		\move(1.4 2)\lvec(1.8 2.25)
 		\htext(1.7 1){$1$}

 		\htext(1.7 -0.7){\fontsize{8}{8}\selectfont{  $1/l  \times 1/2 \times 1$}}
 		\esegment 	
 		
 		\move(10 0)
 	\bsegment
 	\move(1.4 0)\lvec(2 0)
 	\move(1.4 2)\lvec(2 2)
 	\move(2 0)\lvec(2 2)
 	\move(2 2)\lvec(2.4 2.25)
 	\move(2 0)\lvec(2.4 0.25)
 	\move(2.4 0.25)\lvec(2.4 2.25)
 	\move(2.4 2.25)\lvec(1.8 2.25)
 	\move(1.4 0)\lvec(1.4 2)
 	\move(1.4 2)\lvec(1.8 2.25)
 	\htext(1.7 1){$n$\!-\!$1$}

 	\htext(1.7 -0.7){\fontsize{8}{8}\selectfont{  $1/l  \times 1/2 \times 1$}}
 	\esegment 	
 	
 		\move(15 0)
 		\bsegment
 	\move(1.4 0)\lvec(2 0)
 	\move(1.4 2)\lvec(2 2)
 	\move(2 0)\lvec(2 2)
 	\move(2 2)\lvec(2.4 2.25)
 	\move(2 0)\lvec(2.4 0.25)
 	\move(2.4 0.25)\lvec(2.4 2.25)
 	\move(2.4 2.25)\lvec(1.8 2.25)
 	\move(1.4 0)\lvec(1.4 2)
 	\move(1.4 2)\lvec(1.8 2.25)
 	\htext(1.7 1){$n$}

 	\htext(1.7 -0.7){\fontsize{8}{8}\selectfont{  $1/l  \times 1/2 \times 1$}}
 	\esegment 	
 	
 		\move(20 0)
 	\bsegment
 	\move(1.4 0)\lvec(2 0)
 	
 	\move(1.4 2)\lvec(2 2)
 	
 	\move(2 0)\lvec(2 2)
 	
 	\move(2 2)\lvec(2.8 2.5)
 	
 	\move(2 0)\lvec(2.8 0.5)
 	
 	\move(2.8 2.5)\lvec(2.8 0.5)
 	
 	\move(2.2 2.5)\lvec(2.8 2.5)
 	
 	\move(1.4 0)\lvec(1.4 2)
 	\move(1.4 2)\lvec(2.2 2.5)
 	
 	\htext(1.7 1){$i$}
 	\htext(1.7 -0.7){\fontsize{8}{8}\selectfont{  $1/l  \times 1 \times 1$}}
 	\esegment 	
 		
 	\end{texdraw}
 \end{center}
 
\vskip 5mm 

 (5) $B_{n}^{(1)}\ (2\le i\le n-1)$
 
 \vskip 3mm 
 
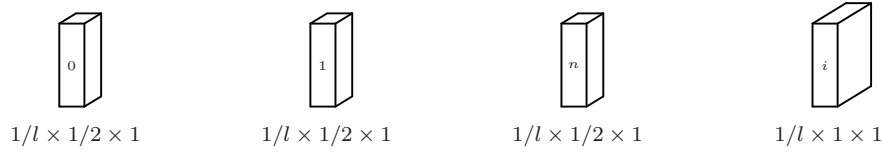
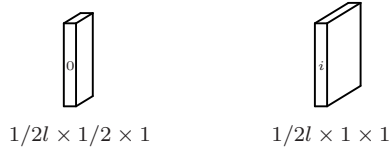
   \begin{center}
 	\begin{texdraw}
 		\fontsize{3}{3}\selectfont
 		\drawdim em
 		\setunitscale 2.3
 		\move(0 0)
 		\bsegment
 		
 		\move(1.4 0)\lvec(2 0)
 		\move(1.4 2)\lvec(2 2)
 		\move(2 0)\lvec(2 2)
 		\move(2 2)\lvec(2.4 2.25)
 		\move(2 0)\lvec(2.4 0.25)
 		\move(2.4 0.25)\lvec(2.4 2.25)
 		\move(2.4 2.25)\lvec(1.8 2.25)
 		\move(1.4 0)\lvec(1.4 2)
 		\move(1.4 2)\lvec(1.8 2.25)
 		\htext(1.7 1){$0$}

 		\htext(1.7 -0.7){\fontsize{8}{8}\selectfont{  $1/l  \times 1/2 \times 1$}}
 		
 		\esegment
 		
 		\move(6 0)
 		\bsegment
 		\move(1.4 0)\lvec(2 0)
 		\move(1.4 2)\lvec(2 2)
 		\move(2 0)\lvec(2 2)
 		\move(2 2)\lvec(2.4 2.25)
 		\move(2 0)\lvec(2.4 0.25)
 		\move(2.4 0.25)\lvec(2.4 2.25)
 		\move(2.4 2.25)\lvec(1.8 2.25)
 		\move(1.4 0)\lvec(1.4 2)
 		\move(1.4 2)\lvec(1.8 2.25)
 		\htext(1.7 1){$1$}

 		\htext(1.7 -0.7){\fontsize{8}{8}\selectfont{  $1/l  \times 1/2 \times 1$}}
 		\esegment 	
 		
 		\move(12 0)
 		\bsegment
 			\move(1.4 0)\lvec(2 0)
 		\move(1.4 2)\lvec(2 2)
 		\move(2 0)\lvec(2 2)
 		\move(2 2)\lvec(2.4 2.25)
 		\move(2 0)\lvec(2.4 0.25)
 		\move(2.4 0.25)\lvec(2.4 2.25)
 		\move(2.4 2.25)\lvec(1.8 2.25)
 		\move(1.4 0)\lvec(1.4 2)
 		\move(1.4 2)\lvec(1.8 2.25)
 		\htext(1.7 1){$n$}

 		\htext(1.7 -0.7){\fontsize{8}{8}\selectfont{  $1/l  \times 1/2 \times 1$}}
 		\esegment 	
 	
 	 		\move(18 0)
 	\bsegment
 	\move(1.4 0)\lvec(2 0)
 	
 	\move(1.4 2)\lvec(2 2)
 	
 	\move(2 0)\lvec(2 2)
 	
 	\move(2 2)\lvec(2.8 2.5)
 	
 	\move(2 0)\lvec(2.8 0.5)
 	
 	\move(2.8 2.5)\lvec(2.8 0.5)
 	
 	\move(2.2 2.5)\lvec(2.8 2.5)
 	
 	\move(1.4 0)\lvec(1.4 2)
 	\move(1.4 2)\lvec(2.2 2.5)
 	
 	\htext(1.7 1){$i$}
 	\htext(1.7 -0.7){\fontsize{8}{8}\selectfont{  $1/l  \times 1 \times 1$}}
 	\esegment	
 	\end{texdraw}
 \end{center}

\vskip 5mm

 (6) $C_{n}^{(1)}\ (1\le i\le n)$
 
 \vskip 3mm 
 
    \begin{center}
 	\begin{texdraw}
 		\fontsize{3}{3}\selectfont
 		\drawdim em
 		\setunitscale 2.3
 		\move(0 0)
 		\bsegment
 		
\move(1.4 0)\lvec(1.7 0)
\move(1.4 2)\lvec(1.7 2)
\move(1.7 0)\lvec(1.7 2)
\move(1.7 2)\lvec(2.1 2.25)
\move(1.7 0)\lvec(2.1 0.25)
\move(2.1 0.25)\lvec(2.1 2.25)
\move(2.1 2.25)\lvec(1.8 2.25)
\move(1.4 0)\lvec(1.4 2)
\move(1.4 2)\lvec(1.8 2.25)
\htext(1.55 1){$0$}

\htext(1.7 -0.7){\fontsize{8}{8}\selectfont{  $1/2l  \times 1/2 \times 1$}}
 		
 		\esegment
 		
 	\move(6 0)
 		\bsegment
	\move(1.4 0)\lvec(1.7 0)

\move(1.4 2)\lvec(1.7 2)

\move(1.7 0)\lvec(1.7 2)

\move(1.7 2)\lvec(2.5 2.5)

\move(1.7 0)\lvec(2.5 0.5)

\move(2.5 2.5)\lvec(2.5 0.5)

\move(2.2 2.5)\lvec(2.5 2.5)

\move(1.4 0)\lvec(1.4 2)
\move(1.4 2)\lvec(2.2 2.5)	
\htext(1.55 1){$i$}
\htext(1.7 -0.7){\fontsize{8}{8}\selectfont{  $1/2l  \times 1 \times 1$}}
 		\esegment	
 	\end{texdraw}
 \end{center}

\vskip 5mm 

For simplicity, we use a planar representation of our blocks. 
The triangle on the upper-left corner represents the block in the front, and the triangle on the lower-right corner represents the block at the back.

\vskip 3mm

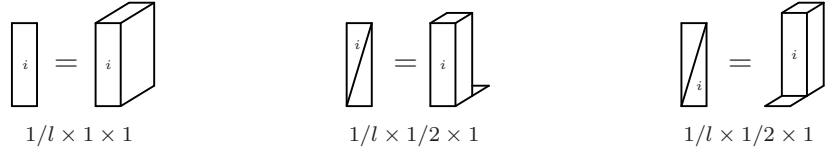
\begin{center}
	\begin{texdraw}
		\fontsize{3}{3}\selectfont
		\drawdim em
		\setunitscale 2.3
		\move(-8 0)
		\bsegment
		\move(0 0)\lvec(-0.6 0)
		\move(0 0)\lvec(0 2)
		\move(-0.6 0)\lvec(-0.6 2)
		
		\move(-0.6 2)\lvec(0 2)
		
		\htext(-0.3 1){$i$}
		\htext(0.7 1){\fontsize{12}{12}\selectfont{$=$}}
		
		\htext(1 -0.7){\fontsize{8}{8}\selectfont{$1/l\times1\times1$}}
		
		\move(1.4 0)\lvec(2 0)
		
		\move(1.4 2)\lvec(2 2)
		
		\move(2 0)\lvec(2 2)
		
		\move(2 2)\lvec(2.8 2.5)
		
		\move(2 0)\lvec(2.8 0.5)
		
		\move(2.8 2.5)\lvec(2.8 0.5)
		
		\move(2.2 2.5)\lvec(2.8 2.5)
		
		\move(1.4 0)\lvec(1.4 2)
		\move(1.4 2)\lvec(2.2 2.5)
		\htext(1.7 1){$i$}
		\esegment

		\move(0 0)
		\bsegment
		\move(0 0)\lvec(-0.6 0)
		\move(0 0)\lvec(0 2)
		\move(-0.6 0)\lvec(-0.6 2)
		
		\move(-0.6 2)\lvec(0 2)
		\move(-0.6 0)\lvec(0 2)

		\htext(-0.35 1.5){$i$}
		\htext(0.8 1){\fontsize{12}{12}\selectfont{$=$}}
		
		\htext(1 -0.7){\fontsize{8}{8}\selectfont{$1/l\times1/2\times1$}}
			
		\move(1.4 0)\lvec(2 0)
		\move(1.4 2)\lvec(2 2)
		\move(2 0)\lvec(2 2)
		\move(2 2)\lvec(2.4 2.25)
		\move(2 0)\lvec(2.8 0.5)
		\move(2.4 0.25)\lvec(2.4 2.25)
		\move(2.4 2.25)\lvec(1.8 2.25)
		\move(1.4 0)\lvec(1.4 2)
		\move(1.4 2)\lvec(1.8 2.25)
		\move(2.4 0.5)\lvec(2.8 0.5)
		\htext(1.7 1){$i$}

		\esegment
		
		\move(8 0)
		\bsegment
		\move(0 0)\lvec(-0.6 0)
		\move(0 0)\lvec(0 2)
		\move(-0.6 0)\lvec(-0.6 2)
		
		\move(-0.6 2)\lvec(0 2)
		\move(-0.6 0)\lvec(0 2)
		
		\htext(-0.18 0.5){$i$}
		
		\htext(0.8 1){\fontsize{12}{12}\selectfont{$=$}}
		
		\htext(1 -0.7){\fontsize{8}{8}\selectfont{$1/l\times1/2\times1$}}
		
		\move(1.4 0)\lvec(2 0)

		\move(2.4 2.25)\lvec(2.8 2.5)
		\move(2 0)\lvec(2.8 0.5)
		\move(1.4 0)\lvec(1.8 0.25)
		\move(2.8 2.5)\lvec(2.8 0.5)
		\move(2.4 0.25)\lvec(2.4 2.25)
		\move(1.8 0.25)\lvec(1.8 2.25)
		\htext(2.1 1.25){$i$}
		\move(2.4 2.25)\lvec(1.8 2.25)
		\move(2.4 0.25)\lvec(1.8 0.25)
		
		\move(2.2 2.5)\lvec(2.8 2.5)

		\move(1.8 2.25)\lvec(2.2 2.5)

		\esegment
	\end{texdraw}
\end{center}  

\vskip 2mm

We will now explain how we stack a block on another block:

\vskip 3mm 

We can place a set $A$ of blocks on top of another set $B$ of blocks if
\begin{itemize}
	\item The width of $B$ is $\frac{i}{l}\ (1\le i\le l)$, the depth of $B$ is $1$ or $\frac{1}{2}$.
	\vskip 2mm
	\item The width and depth of $A$ is less than or equal to the width and depth of $B$.
\end{itemize}

\begin{example}
We consider the case of $l=2$. 

\vskip 2mm

(a) The ways of placing blocks below are allowed.

\vskip 2mm

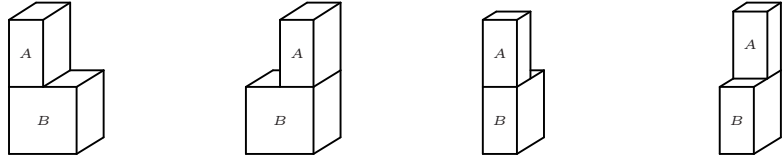
\begin{center}
	\begin{texdraw}
		\fontsize{3}{3}\selectfont
		\drawdim em
		\setunitscale 2.1
		\move(0 0)
		\bsegment		
		\move(0 0)\lvec(2 0)		
		\move(0 0)\lvec(0 2)
		\move(0 2)\lvec(2 2)	
		\move(2 0)\lvec(2 2)
		\move(2 0)\lvec(2.8 0.5)	
		\move(2.8 0.5)\lvec(2.8 2.5)	
		\move(2 2)\lvec(2.8 2.5)
		\move(1 2)\lvec(1.8 2.5)
		\move(1.8 2.5)\lvec(2.8 2.5)				
		\move(1.8 2.5)\lvec(1.8 4.5)				
		\move(1 2)\lvec(1 4)				
		\move(1 4)\lvec(1.8 4.5)				
		\move(0 2)\lvec(0 4)				
		\move(0 4)\lvec(1 4)				
		\move(0 4)\lvec(0.8 4.5)				
		\move(0.8 4.5)\lvec(1.8 4.5)				
		\htext(0.5 3){$A$}				
		\htext(1 1){$B$}				
		\esegment
		
		\move(7 0)
		\bsegment		
		\move(0 0)\lvec(2 0)		
		\move(0 0)\lvec(0 2)
		\move(0 2)\lvec(2 2)	
		\move(2 0)\lvec(2 2)
		\move(2 0)\lvec(2.8 0.5)	
		\move(2.8 0.5)\lvec(2.8 2.5)	
		\move(2 2)\lvec(2.8 2.5)
		\move(0 2)\lvec(0.8 2.5)
		\move(0.8 2.5)	\lvec(1 2.5)			
		\move(2.8 2.5)\lvec(2.8 4.5)				
		\move(2 2)\lvec(2 4)				
		\move(2 4)\lvec(2.8 4.5)				
		\move(1 2)\lvec(1 4)				
		\move(1 4)\lvec(2 4)				
		\move(1 4)\lvec(1.8 4.5)				
		\move(1.8 4.5)\lvec(2.8 4.5)				
		\htext(1.5 3){$A$}				
		\htext(1 1){$B$}				
		\esegment
		
		\move(14 0)
		\bsegment		
		\move(0 0)\lvec(1 0)		
		\move(0 0)\lvec(0 2)
		\move(0 2)\lvec(1 2)	
		\move(1 0)\lvec(1 2)		
		\move(1 0)\lvec(1.8 0.5)
		\move(1 2)\lvec(1.8 2.5)
		\move(1.8 0.5)\lvec(1.8 2.5)				
		\move(1.4 2.25)\lvec(1.4 4.25)				
		\move(1 2)\lvec(1 4)				
		\move(1 4)\lvec(1.4 4.25)				
		\move(0 2)\lvec(0 4)				
		\move(0 4)\lvec(1 4)				
		\move(0 4)\lvec(0.4 4.25)				
		\move(0.4 4.25)\lvec(1.4 4.25)	
		\move(1.4 2.5)\lvec(1.8 2.5)		
		\htext(0.5 3){$A$}				
		\htext(0.5 1){$B$}					
		\esegment	
		
		\move(21 0)
		\bsegment		
		\move(0 0)\lvec(1 0)		
		\move(0 0)\lvec(0 2)
		\move(0 2)\lvec(1 2)	
		\move(1 0)\lvec(1 2)		
		\move(1 0)\lvec(1.8 0.5)
		\move(1 2)\lvec(1.8 2.5)
		\move(1.8 0.5)\lvec(1.8 2.5)				
		\move(1.4 2.25)\lvec(1.4 4.25)				
		\move(1.8 2.5)\lvec(1.8 4.5)				
		\move(0.4 2.25)\lvec(1.4 2.25)			
		\move(0.4 2.25)\lvec(0.4 4.25)	
		\move(0 2)\lvec(0.4 2.25)
		\move(0.4 4.25)\lvec(1.4 4.25)	
		\move(0.4 4.25)\lvec(0.8 4.5)	
		\move(1.4 4.25)\lvec(1.8 4.5)	
		\move(0.8 4.5)\lvec(1.8 4.5)			
		\htext(0.9 3.25){$A$}				
		\htext(0.5 1){$B$}					
		\esegment					
	\end{texdraw}
\end{center}

\vskip 2mm

(b) The ways of placing blocks below are not allowed.

\vskip 2mm

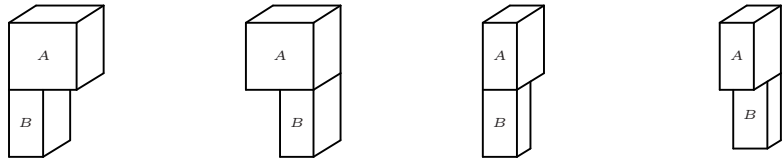
\begin{center}
	\begin{texdraw}
		\fontsize{3}{3}\selectfont
		\drawdim em
		\setunitscale 2.1
		\move(0 0)
		\bsegment		
		\move(0 0)\lvec(1 0)	
		\move(0 0)\lvec(0 4)
		\move(1 0)\lvec(1 2)
		\move(0 4)\lvec(2 4)
		\move(0 2)\lvec(2 2)
		\move(2 2)\lvec(2 4)
		\move(1 0)\lvec(1.8 0.5)
		\move(1.8 0.5)\lvec(1.8 2)
		\move(2 2)\lvec(2.8 2.5)
		\move(2 4)\lvec(2.8 4.5)
		\move(2.8 2.5)\lvec(2.8 4.5)
		\move(0.8 4.5)\lvec(2.8 4.5)
		\move(0 4)\lvec(0.8 4.5)
		\htext(1 3){$A$}				
		\htext(0.5 1){$B$}
		\esegment
		
		\move(7 0)
		\bsegment		
		
		\move(0 2)\lvec(0 4)
		
		\move(0 4)\lvec(2 4)
		\move(0 2)\lvec(2 2)
		\move(2 2)\lvec(2 4)
		\move(2 2)\lvec(2.8 2.5)
		\move(2 4)\lvec(2.8 4.5)
		\move(2.8 2.5)\lvec(2.8 4.5)
		\move(0.8 4.5)\lvec(2.8 4.5)
		\move(0 4)\lvec(0.8 4.5)

		\move(2 0)\lvec(2 2)
		\move(1 0)\lvec(2 0)
		\move(2 0)\lvec(2.8 0.5)
		\move(2.8 0.5)\lvec(2.8 2.5)
		\move(1 0)\lvec(1 2)

		\htext(1 3){$A$}				
		\htext(1.5 1){$B$}			
		\esegment
		
		\move(14 0)
		\bsegment		
		\move(0 0)\lvec(1 0)		
		\move(0 0)\lvec(0 4)
		\move(1 0)\lvec(1 4)	
		\move(0 2)\lvec(1 2)
		\move(0 4)\lvec(1 4)
		\move(1 0)\lvec(1.4 0.25)
		\move(1.4 0.25)\lvec(1.4 2.25)	
		\move(1 2)\lvec(1.8 2.5)
		\move(1.8 2.5)\lvec(1.8 4.5)
		\move(1 4)\lvec(1.8 4.5)
		\move(0 4)\lvec(0.8 4.5)
		\move(0.8 4.5)\lvec(1.8 4.5)

		\htext(0.5 1){$B$}				
		\htext(0.5 3){$A$}							
		\esegment	
		
		\move(21 0)
		\bsegment		
		\move(0 2)\lvec(0 4)		
		\move(1 2)\lvec(1 4)
		\move(1.8 2.5)\lvec(1.8 4.5)
		\move(1.8 0.5)\lvec(1.8 2.5)
		\move(1 2)\lvec(1.8 2.5)
		\move(0 2)\lvec(1 2)	
		\move(0 4)\lvec(1 4)
		\move(1 4)\lvec(1.8 4.5)	
		\move(0 4)\lvec(0.8 4.5)
		\move(0.8 4.5)\lvec(1.8 4.5)
		\move(1.4 0.25)\lvec(1.8 0.5)
		\move(1.4 0.25)\lvec(1.4 2.25)
		\move(0.4 0.25)\lvec(1.4 0.25)
		\move(0.4 0.25)\lvec(0.4 2)	
		\htext(0.9 1.25){$B$}				
		\htext(0.5 3){$A$}			
		\esegment					
	\end{texdraw}
\end{center}

\end{example}

\begin{definition}
A level-$l$ slice is defined to be a set of finitely many blocks stacked in a column of unit depth and $\frac{1}{l}$ or $\frac{1}{2l}$ width. 
\end{definition}

In the following text, a level-$l$ slice is referred to as a slice for short.

\begin{definition}\label{stacking patterns}
	For each type of affine Cartan datum, a level-$l$ Young Column is defined to be a set of finitely many blocks in a column of unit width and unit depth stacked following the patterns given in the Picture $\mathbf{Stacking}\ \mathbf{Patterns}$ such that the number of blocks in the slice is weakly decreasing from right to left.

\vskip 10mm

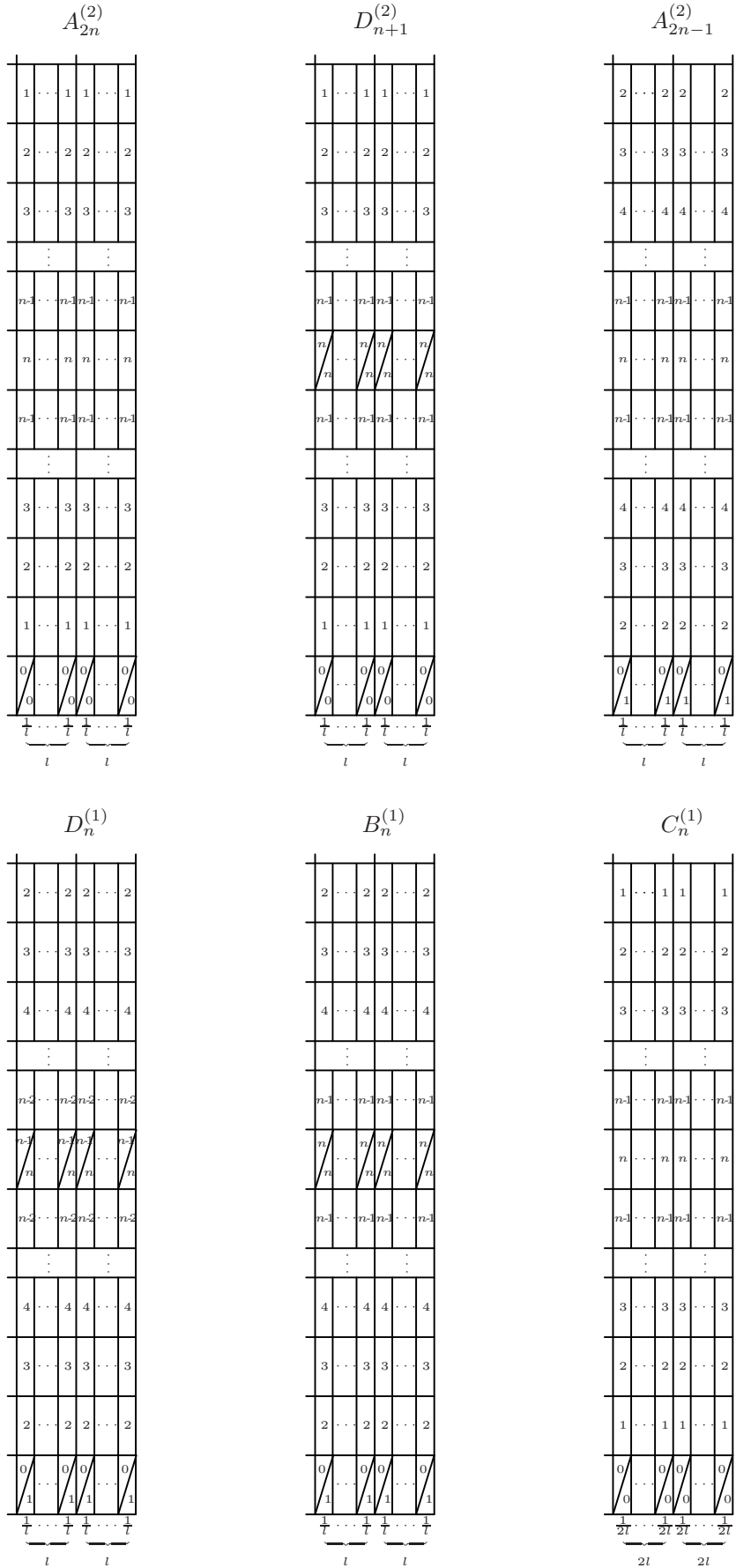
\begin{center}
	\begin{texdraw}
		\fontsize{3}{3}\selectfont
		\drawdim em
		\setunitscale 2.1
		\htext(-13.8 24.5){\htext(2 -1){\fontsize{10}{10}\selectfont{$A_{2n}^{(2)}$}}}	
		\move(-10 0)\rlvec(-4.3 0) \move(-10 2)\rlvec(-4.3 0)\move(-10 4)\rlvec(-4.3 0) \move(-10 6)\rlvec(-4.3 0) \move(-10 8)\rlvec(-4.3 0)
		\move(-10 9)\rlvec(-4.3 0) \move(-10 11)\rlvec(-4.3 0)\move(-10 13)\rlvec(-4.3 0)\move(-10 15)\rlvec(-4.3 0)\move(-10 16)\rlvec(-4.3 0) \move(-10 18)\rlvec(-4.3 0) \move(-10 20)\rlvec(-4.3 0) \move(-10 22)\rlvec(-4.3 0)
		\move(-10 0)\rlvec(0 22.3) \move(-10.6 0)\rlvec(0 8)  \move(-11.4 0)\rlvec(0 8) \move(-12 0)\rlvec(0 22.3) \move(-10.6 9)\rlvec(0 6) \move(-11.4 9)\rlvec(0 6)  \move(-14 0)\rlvec(0 22.3)
		\move(-12.6 0)\rlvec(0 8)\move(-13.4 0)\rlvec(0 8)
		\move(-12.6 9)\rlvec(0 6)\move(-12.6 16)\rlvec(0 6) \move(-13.4 9)\rlvec(0 6)
		\move(-13.4 16)\rlvec(0 6)\move(-10.6 0)\rlvec(0.6 2) \move(-12 0)\rlvec(0.6 2) 
		\move(-10.6 20)\rlvec(0 2)\move(-11.4 20)\rlvec(0 2)
		\move(-10.6 16)\rlvec(0 2)\move(-11.4 16)\rlvec(0 2)
		\move(-10.6 18)\rlvec(0 2)\move(-11.4 18)\rlvec(0 2)
		\move(-12.6 0)\rlvec(0.6 2)\move(-14 0)\rlvec(0.6 2)
		\htext(-10.95 1){$\cdots$}\htext(-10.95 3){$\cdots$}\htext(-10.95 5){$\cdots$}\htext(-10.95 7){$\cdots$}\vtext(-10.95 8.5){$\cdots$}  \htext(-10.95 10){$\cdots$}  \htext(-10.95 12){$\cdots$}  \htext(-10.95 14){$\cdots$}\htext(-10.95 17){$\cdots$}  \htext(-10.95 19){$\cdots$} \vtext(-10.95 15.5){$\cdots$} \vtext(-12.95 15.5){$\cdots$} \htext(-12.95 19){$\cdots$}        \htext(-10.95 21){$\cdots$} 
		\htext(-10.2 0.5){$0$} \htext(-10.4 1.5){$0$} 
		\htext(-11.6 0.5){$0$} \htext(-11.8 1.5){$0$} 
		\htext(-10.3 3){$1$}\htext(-11.7 3){$1$}
		\htext(-10.3 5){$2$}\htext(-11.7 5){$2$}
		\htext(-10.3 7){$3$}\htext(-11.7 7){$3$}
		\htext(-10.3 10){\fontsize{5.0pt}{\baselineskip}\selectfont{$n$}\!-\!\fontsize{5.0pt}{\baselineskip}\selectfont{$1$}}
		\htext(-11.7 10){\fontsize{5.0pt}{\baselineskip}\selectfont{$n$}\!-\!\fontsize{5.0pt}{\baselineskip}\selectfont{$1$}} 
		\htext(-10.3 12){$n$}\htext(-11.7 12){$n$} 
		\htext(-10.3 14){\fontsize{5.0pt}{\baselineskip}\selectfont{$n$}\!-\!\fontsize{5.0pt}{\baselineskip}\selectfont{$1$}}
		\htext(-11.7 14){\fontsize{5.0pt}{\baselineskip}\selectfont{$n$}\!-\!\fontsize{5.0pt}{\baselineskip}\selectfont{$1$}} 
		\htext(-10.3 17){$3$} \htext(-11.7 17){$3$}
		\htext(-10.3 19){$2$} \htext(-11.7 19){$2$}
		\htext(-10.3 21){$1$} \htext(-11.7 21){$1$}
		\htext(-12.95 1){$\cdots$}\htext(-12.95 3){$\cdots$}\htext(-12.95 5){$\cdots$}\htext(-12.95 7){$\cdots$}\vtext(-12.95 8.5){$\cdots$}  \htext(-12.95 10){$\cdots$}  \htext(-12.95 12){$\cdots$}  \htext(-12.95 14){$\cdots$} \htext(-12.95 21){$\cdots$}\htext(-12.95 17){$\cdots$}
		\htext(-10.3 -0.4){$\frac{1}{l}$} \htext(-11.7 -0.4){$\frac{1}{l}$}  \htext(-12.3 -0.4){$\frac{1}{l}$}\htext(-10.95 -0.4){$\cdots$} \htext(-12.95 -0.4){$\cdots$}\htext(-13.7 -0.4){$\frac{1}{l}$}
		\htext(-11 -1){$\underbrace{\rule{3em}{0em}}$}
		\htext(-13 -1){$\underbrace{\rule{3em}{0em}}$}
		\htext(-11 -1.6){$l$}
		\htext(-13 -1.6){$l$}
		\htext(-12.2 0.5){$0$} \htext(-12.4 1.5){$0$} 
		\htext(-13.6 0.5){$0$} \htext(-13.8 1.5){$0$} 
		\htext(-12.3 3){$1$}\htext(-13.7 3){$1$}
		\htext(-12.3 5){$2$}\htext(-13.7 5){$2$}
		\htext(-12.3 7){$3$}\htext(-13.7 7){$3$}
		\htext(-12.3 10){\fontsize{5.0pt}{\baselineskip}\selectfont{$n$}\!-\!\fontsize{5.0pt}{\baselineskip}\selectfont{$1$}}
		\htext(-13.7 10){\fontsize{5.0pt}{\baselineskip}\selectfont{$n$}\!-\!\fontsize{5.0pt}{\baselineskip}\selectfont{$1$}} 
		\htext(-12.3 12){$n$}\htext(-13.7 12){$n$} 
		\htext(-12.3 14){\fontsize{5.0pt}{\baselineskip}\selectfont{$n$}\!-\!\fontsize{5.0pt}{\baselineskip}\selectfont{$1$}}
		\htext(-13.7 14){\fontsize{5.0pt}{\baselineskip}\selectfont{$n$}\!-\!\fontsize{5.0pt}{\baselineskip}\selectfont{$1$}} 
		\htext(-12.3 17){$3$} \htext(-13.7 17){$3$}
		\htext(-12.3 19){$2$} \htext(-13.7 19){$2$}
		\htext(-12.3 21){$1$} \htext(-13.7 21){$1$}
		%
		\htext(-3.8 24.5){\htext(2 -1){\fontsize{10}{10}\selectfont{$D_{n+1}^{(2)}$}}}
		\move(0 0)\rlvec(-4.3 0) \move(0 2)\rlvec(-4.3 0)\move(0 4)\rlvec(-4.3 0) \move(0 6)\rlvec(-4.3 0) \move(0 8)\rlvec(-4.3 0)
		\move(0 9)\rlvec(-4.3 0) \move(0 11)\rlvec(-4.3 0)\move(0 13)\rlvec(-4.3 0)\move(0 15)\rlvec(-4.3 0)\move(0 16)\rlvec(-4.3 0) \move(0 18)\rlvec(-4.3 0) \move(0 20)\rlvec(-4.3 0) \move(0 22)\rlvec(-4.3 0)
		\move(0 0)\rlvec(0 22.3) \move(-0.6 0)\rlvec(0 8)  \move(-1.4 0)\rlvec(0 8) \move(-2 0)\rlvec(0 22.3) \move(-0.6 9)\rlvec(0 6) \move(-1.4 9)\rlvec(0 6)  \move(-4 0)\rlvec(0 22.3)
		\move(-2.6 0)\rlvec(0 8)\move(-3.4 0)\rlvec(0 8)
		\move(-2.6 9)\rlvec(0 6)\move(-2.6 16)\rlvec(0 6) \move(-3.4 9)\rlvec(0 6)
		\move(-3.4 16)\rlvec(0 6)\move(-0.6 0)\rlvec(0.6 2) \move(-2 0)\rlvec(0.6 2) 
		\move(-0.6 20)\rlvec(0 2)\move(-1.4 20)\rlvec(0 2)
		\move(-0.6 16)\rlvec(0 2)\move(-1.4 16)\rlvec(0 2)
		\move(-0.6 18)\rlvec(0 2)\move(-1.4 18)\rlvec(0 2)
		\move(-2.6 0)\rlvec(0.6 2)\move(-4 0)\rlvec(0.6 2)
		\move(-0.6 11)\rlvec(0.6 2)\move(-2 11)\rlvec(0.6 2)
		\move(-2.6 11)\rlvec(0.6 2)\move(-4 11)\rlvec(0.6 2)
		\htext(-0.95 1){$\cdots$}\htext(-0.95 3){$\cdots$}\htext(-0.95 5){$\cdots$}\htext(-0.95 7){$\cdots$}\vtext(-0.95 8.5){$\cdots$}  \htext(-0.95 10){$\cdots$}  \htext(-0.95 12){$\cdots$}  \htext(-0.95 14){$\cdots$}\htext(-0.95 17){$\cdots$}  \htext(-0.95 19){$\cdots$} \vtext(-0.95 15.5){$\cdots$} \vtext(-2.95 15.5){$\cdots$} \htext(-2.95 19){$\cdots$}        \htext(-0.95 21){$\cdots$} 
		\htext(-0.2 0.5){$0$} \htext(-0.4 1.5){$0$} 
		\htext(-1.6 0.5){$0$} \htext(-1.8 1.5){$0$} 
		\htext(-0.2 11.5){$n$} \htext(-0.4 12.5){$n$} 
		\htext(-1.6 11.5){$n$} \htext(-1.8 12.5){$n$}
		\htext(-2.2 11.5){$n$} \htext(-2.4 12.5){$n$} 
		\htext(-3.6 11.5){$n$} \htext(-3.8 12.5){$n$}
		\htext(-0.3 3){$1$}\htext(-1.7 3){$1$}
		\htext(-0.3 5){$2$}\htext(-1.7 5){$2$}
		\htext(-0.3 7){$3$}\htext(-1.7 7){$3$}
		\htext(-0.3 10){\fontsize{5.0pt}{\baselineskip}\selectfont{$n$}\!-\!\fontsize{5.0pt}{\baselineskip}\selectfont{$1$}}
		\htext(-1.7 10){\fontsize{5.0pt}{\baselineskip}\selectfont{$n$}\!-\!\fontsize{5.0pt}{\baselineskip}\selectfont{$1$}} 
		
		\htext(-0.3 14){\fontsize{5.0pt}{\baselineskip}\selectfont{$n$}\!-\!\fontsize{5.0pt}{\baselineskip}\selectfont{$1$}}
		\htext(-1.7 14){\fontsize{5.0pt}{\baselineskip}\selectfont{$n$}\!-\!\fontsize{5.0pt}{\baselineskip}\selectfont{$1$}} 
		\htext(-0.3 17){$3$} \htext(-1.7 17){$3$}
		\htext(-0.3 19){$2$} \htext(-1.7 19){$2$}
		\htext(-0.3 21){$1$} \htext(-1.7 21){$1$}
		\htext(-2.95 1){$\cdots$}\htext(-2.95 3){$\cdots$}\htext(-2.95 5){$\cdots$}\htext(-2.95 7){$\cdots$}\vtext(-2.95 8.5){$\cdots$}  \htext(-2.95 10){$\cdots$}  \htext(-2.95 12){$\cdots$}  \htext(-2.95 14){$\cdots$} \htext(-2.95 21){$\cdots$}\htext(-2.95 17){$\cdots$}
		\htext(-0.3 -0.4){$\frac{1}{l}$} \htext(-1.7 -0.4){$\frac{1}{l}$}  \htext(-2.3 -0.4){$\frac{1}{l}$}\htext(-0.95 -0.4){$\cdots$} \htext(-2.95 -0.4){$\cdots$}\htext(-3.7 -0.4){$\frac{1}{l}$}
		\htext(-1 -1){$\underbrace{\rule{3em}{0em}}$}
		\htext(-3 -1){$\underbrace{\rule{3em}{0em}}$}
		\htext(-1 -1.6){$l$}
		\htext(-3 -1.6){$l$}
		\htext(-2.2 0.5){$0$} \htext(-2.4 1.5){$0$} 
		\htext(-3.6 0.5){$0$} \htext(-3.8 1.5){$0$} 
		\htext(-2.3 3){$1$}\htext(-3.7 3){$1$}
		\htext(-2.3 5){$2$}\htext(-3.7 5){$2$}
		\htext(-2.3 7){$3$}\htext(-3.7 7){$3$}
		\htext(-2.3 10){\fontsize{5.0pt}{\baselineskip}\selectfont{$n$}\!-\!\fontsize{5.0pt}{\baselineskip}\selectfont{$1$}}
		\htext(-3.7 10){\fontsize{5.0pt}{\baselineskip}\selectfont{$n$}\!-\!\fontsize{5.0pt}{\baselineskip}\selectfont{$1$}} 
		
		\htext(-2.3 14){\fontsize{5.0pt}{\baselineskip}\selectfont{$n$}\!-\!\fontsize{5.0pt}{\baselineskip}\selectfont{$1$}}
		\htext(-3.7 14){\fontsize{5.0pt}{\baselineskip}\selectfont{$n$}\!-\!\fontsize{5.0pt}{\baselineskip}\selectfont{$1$}} 
		\htext(-2.3 17){$3$} \htext(-3.7 17){$3$}
		\htext(-2.3 19){$2$} \htext(-3.7 19){$2$}
		\htext(-2.3 21){$1$} \htext(-3.7 21){$1$}
		%
		\htext(6.3 24.5){\htext(2 -1){\fontsize{10}{10}\selectfont{$A_{2n-1}^{(2)}$}}}
		\move(10 0)\rlvec(-4.3 0) \move(10 2)\rlvec(-4.3 0)\move(10 4)\rlvec(-4.3 0) \move(10 6)\rlvec(-4.3 0) \move(10 8)\rlvec(-4.3 0)
		\move(10 9)\rlvec(-4.3 0) \move(10 11)\rlvec(-4.3 0)\move(10 13)\rlvec(-4.3 0)\move(10 15)\rlvec(-4.3 0)\move(10 16)\rlvec(-4.3 0) \move(10 18)\rlvec(-4.3 0) \move(10 20)\rlvec(-4.3 0) \move(10 22)\rlvec(-4.3 0)
		\move(10 0)\rlvec(0 22.3) \move(9.4 0)\rlvec(0 8)  \move(8.6 0)\rlvec(0 8) \move(8 0)\rlvec(0 22.3) \move(9.4 9)\rlvec(0 6) \move(8.6 9)\rlvec(0 6)  \move(6 0)\rlvec(0 22.3)
		\move(7.4 0)\rlvec(0 8)\move(6.6 0)\rlvec(0 8)
		\move(7.4 9)\rlvec(0 6)\move(7.4 16)\rlvec(0 6) \move(6.6 9)\rlvec(0 6)
		\move(6.6 16)\rlvec(0 6)\move(9.4 0)\rlvec(0.6 2) \move(8 0)\rlvec(0.6 2) 
		\move(9.4 20)\rlvec(0 2)\move(8.6 20)\rlvec(0 2)
		\move(9.4 16)\rlvec(0 2)\move(8.6 16)\rlvec(0 2)
		\move(9.4 18)\rlvec(0 2)\move(8.6 18)\rlvec(0 2)
		\move(7.4 0)\rlvec(0.6 2)\move(6 0)\rlvec(0.6 2)
		\htext(9.05 1){$\cdots$}\htext(9.05 3){$\cdots$}\htext(9.05 5){$\cdots$}\htext(9.05 7){$\cdots$}\vtext(9.05 8.5){$\cdots$}  \htext(9.05 10){$\cdots$}  \htext(9.05 12){$\cdots$}  \htext(9.05 14){$\cdots$}\htext(9.05 17){$\cdots$}  \htext(9.05 19){$\cdots$} \vtext(9.05 15.5){$\cdots$} \vtext(7.05 15.5){$\cdots$} \htext(7.05 19){$\cdots$}        \htext(7.05 21){$\cdots$} 
		\htext(9.8 0.5){$1$} \htext(9.6 1.5){$0$} 
		\htext(8.4 0.5){$1$} \htext(8.2 1.5){$0$} 
		\htext(9.7 3){$2$}\htext(8.3 3){$2$}
		\htext(9.7 5){$3$}\htext(8.3 5){$3$}
		\htext(9.7 7){$4$}\htext(8.3 7){$4$}
		\htext(9.7 10){\fontsize{5.0pt}{\baselineskip}\selectfont{$n$}\!-\!\fontsize{5.0pt}{\baselineskip}\selectfont{$1$}}
		\htext(8.3 10){\fontsize{5.0pt}{\baselineskip}\selectfont{$n$}\!-\!\fontsize{5.0pt}{\baselineskip}\selectfont{$1$}} 
		\htext(9.7 12){$n$}\htext(8.3 12){$n$} 
		\htext(9.7 14){\fontsize{5.0pt}{\baselineskip}\selectfont{$n$}\!-\!\fontsize{5.0pt}{\baselineskip}\selectfont{$1$}}
		\htext(8.3 14){\fontsize{5.0pt}{\baselineskip}\selectfont{$n$}\!-\!\fontsize{5.0pt}{\baselineskip}\selectfont{$1$}} 
		\htext(9.7 17){$4$} \htext(8.3 17){$4$}
		\htext(9.7 19){$3$} \htext(8.3 19){$3$}
		\htext(9.7 21){$2$} \htext(8.3 21){$2$}		
		\htext(7.05 1){$\cdots$}\htext(7.05 3){$\cdots$}\htext(7.05 5){$\cdots$}\htext(7.05 7){$\cdots$}\vtext(7.05 8.5){$\cdots$}  \htext(7.05 10){$\cdots$}  \htext(7.05 12){$\cdots$}  \htext(7.05 14){$\cdots$} \htext(7.05 21){$\cdots$}\htext(7.05 17){$\cdots$}
		\htext(9.7 -0.4){$\frac{1}{l}$} \htext(8.3 -0.4){$\frac{1}{l}$}  \htext(7.7 -0.4){$\frac{1}{l}$}\htext(9.05 -0.4){$\cdots$} \htext(7.05 -0.4){$\cdots$}\htext(6.3 -0.4){$\frac{1}{l}$}
		\htext(9 -1){$\underbrace{\rule{3em}{0em}}$}
		\htext(7 -1){$\underbrace{\rule{3em}{0em}}$}
		\htext(9 -1.6){$l$}
		\htext(7 -1.6){$l$}
		\htext(7.8 0.5){$1$} \htext(7.6 1.5){$0$} 
		\htext(6.4 0.5){$1$} \htext(6.2 1.5){$0$} 
		\htext(7.7 3){$2$}\htext(6.3 3){$2$}
		\htext(7.7 5){$3$}\htext(6.3 5){$3$}
		\htext(7.7 7){$4$}\htext(6.3 7){$4$}
		\htext(7.7 10){\fontsize{5.0pt}{\baselineskip}\selectfont{$n$}\!-\!\fontsize{5.0pt}{\baselineskip}\selectfont{$1$}}
		\htext(6.3 10){\fontsize{5.0pt}{\baselineskip}\selectfont{$n$}\!-\!\fontsize{5.0pt}{\baselineskip}\selectfont{$1$}} 
		\htext(7.7 12){$n$}\htext(6.3 12){$n$} 
		\htext(7.7 14){\fontsize{5.0pt}{\baselineskip}\selectfont{$n$}\!-\!\fontsize{5.0pt}{\baselineskip}\selectfont{$1$}}
		\htext(6.3 14){\fontsize{5.0pt}{\baselineskip}\selectfont{$n$}\!-\!\fontsize{5.0pt}{\baselineskip}\selectfont{$1$}} 
		\htext(7.7 17){$4$} \htext(6.3 17){$4$}
		\htext(7.7 19){$3$} \htext(6.3 19){$3$}
		\htext(7.7 21){$2$} \htext(6.3 21){$2$}
		\move(0 -27)
		\bsegment
		\htext(-13.7 51.4){\htext(2 -1){\fontsize{10}{10}\selectfont{$D_n^{(1)}$}}}
		\move(-10 0)\rlvec(-4.3 0) \move(-10 2)\rlvec(-4.3 0)\move(-10 4)\rlvec(-4.3 0) \move(-10 6)\rlvec(-4.3 0) \move(-10 8)\rlvec(-4.3 0)
		\move(-10 9)\rlvec(-4.3 0) \move(-10 11)\rlvec(-4.3 0)\move(-10 13)\rlvec(-4.3 0)\move(-10 15)\rlvec(-4.3 0)\move(-10 16)\rlvec(-4.3 0) \move(-10 18)\rlvec(-4.3 0) \move(-10 20)\rlvec(-4.3 0) \move(-10 22)\rlvec(-4.3 0)
		\move(-10 0)\rlvec(0 22.3) \move(-10.6 0)\rlvec(0 8)  \move(-11.4 0)\rlvec(0 8) \move(-12 0)\rlvec(0 22.3) \move(-10.6 9)\rlvec(0 6) \move(-11.4 9)\rlvec(0 6)  \move(-14 0)\rlvec(0 22.3)
		\move(-12.6 0)\rlvec(0 8)\move(-13.4 0)\rlvec(0 8)
		\move(-12.6 9)\rlvec(0 6)\move(-12.6 16)\rlvec(0 6) \move(-13.4 9)\rlvec(0 6)
		\move(-13.4 16)\rlvec(0 6)\move(-10.6 0)\rlvec(0.6 2) \move(-12 0)\rlvec(0.6 2) 
		\move(-10.6 20)\rlvec(0 2)\move(-11.4 20)\rlvec(0 2)
		\move(-10.6 16)\rlvec(0 2)\move(-11.4 16)\rlvec(0 2)
		\move(-10.6 18)\rlvec(0 2)\move(-11.4 18)\rlvec(0 2)
		\move(-12.6 0)\rlvec(0.6 2)\move(-14 0)\rlvec(0.6 2)
		\move(-10.6 11)\rlvec(0.6 2)\move(-12 11)\rlvec(0.6 2)
		\move(-12.6 11)\rlvec(0.6 2)\move(-14 11)\rlvec(0.6 2)
		\htext(-10.2 11.5){$n$} \htext(-10.38 12.6){\fontsize{5.0pt}{\baselineskip}\selectfont{$n$}\!-\!\fontsize{5.0pt}{\baselineskip}\selectfont{$1$}} 
		\htext(-11.6 11.5){$n$} \htext(-11.78 12.6){\fontsize{5.0pt}{\baselineskip}\selectfont{$n$}\!-\!\fontsize{5.0pt}{\baselineskip}\selectfont{$1$}}
		\htext(-12.2 11.5){$n$} \htext(-12.38 12.6){\fontsize{5.0pt}{\baselineskip}\selectfont{$n$}\!-\!\fontsize{5.0pt}{\baselineskip}\selectfont{$1$}} 
		\htext(-13.6 11.5){$n$} \htext(-13.78 12.6){\fontsize{5.0pt}{\baselineskip}\selectfont{$n$}\!-\!\fontsize{5.0pt}{\baselineskip}\selectfont{$1$}}
		\htext(-10.95 1){$\cdots$}\htext(-10.95 3){$\cdots$}\htext(-10.95 5){$\cdots$}\htext(-10.95 7){$\cdots$}\vtext(-10.95 8.5){$\cdots$}  \htext(-10.95 10){$\cdots$}  \htext(-10.95 12){$\cdots$}  \htext(-10.95 14){$\cdots$}\htext(-10.95 17){$\cdots$}  \htext(-10.95 19){$\cdots$} \vtext(-10.95 15.5){$\cdots$} \vtext(-12.95 15.5){$\cdots$} \htext(-12.95 19){$\cdots$}        \htext(-10.95 21){$\cdots$} 
		\htext(-10.2 0.5){$1$} \htext(-10.4 1.5){$0$} 
		\htext(-11.6 0.5){$1$} \htext(-11.8 1.5){$0$} 
		\htext(-10.3 3){$2$}\htext(-11.7 3){$2$}
		\htext(-10.3 5){$3$}\htext(-11.7 5){$3$}
		\htext(-10.3 7){$4$}\htext(-11.7 7){$4$}
		\htext(-10.3 10){\fontsize{5.0pt}{\baselineskip}\selectfont{$n$}\!-\!\fontsize{5.0pt}{\baselineskip}\selectfont{$2$}}
		\htext(-11.7 10){\fontsize{5.0pt}{\baselineskip}\selectfont{$n$}\!-\!\fontsize{5.0pt}{\baselineskip}\selectfont{$2$}} 
		
		\htext(-10.3 14){\fontsize{5.0pt}{\baselineskip}\selectfont{$n$}\!-\!\fontsize{5.0pt}{\baselineskip}\selectfont{$2$}}
		\htext(-11.7 14){\fontsize{5.0pt}{\baselineskip}\selectfont{$n$}\!-\!\fontsize{5.0pt}{\baselineskip}\selectfont{$2$}} 
		\htext(-10.3 17){$4$} \htext(-11.7 17){$4$}
		\htext(-10.3 19){$3$} \htext(-11.7 19){$3$}
		\htext(-10.3 21){$2$} \htext(-11.7 21){$2$}
		\htext(-12.95 1){$\cdots$}\htext(-12.95 3){$\cdots$}\htext(-12.95 5){$\cdots$}\htext(-12.95 7){$\cdots$}\vtext(-12.95 8.5){$\cdots$}  \htext(-12.95 10){$\cdots$}  \htext(-12.95 12){$\cdots$}  \htext(-12.95 14){$\cdots$} \htext(-12.95 21){$\cdots$}\htext(-12.95 17){$\cdots$}
		\htext(-10.3 -0.4){$\frac{1}{l}$} \htext(-11.7 -0.4){$\frac{1}{l}$}  \htext(-12.3 -0.4){$\frac{1}{l}$}\htext(-10.95 -0.4){$\cdots$} \htext(-12.95 -0.4){$\cdots$}\htext(-13.7 -0.4){$\frac{1}{l}$}
		\htext(-11 -1){$\underbrace{\rule{3em}{0em}}$}
		\htext(-13 -1){$\underbrace{\rule{3em}{0em}}$}
		\htext(-11 -1.6){$l$}
		\htext(-13 -1.6){$l$}
		\htext(-12.2 0.5){$1$} \htext(-12.4 1.5){$0$} 
		\htext(-13.6 0.5){$1$} \htext(-13.8 1.5){$0$} 
		\htext(-12.3 3){$2$}\htext(-13.7 3){$2$}
		\htext(-12.3 5){$3$}\htext(-13.7 5){$3$}
		\htext(-12.3 7){$4$}\htext(-13.7 7){$4$}
		\htext(-12.3 10){\fontsize{5.0pt}{\baselineskip}\selectfont{$n$}\!-\!\fontsize{5.0pt}{\baselineskip}\selectfont{$2$}}
		\htext(-13.7 10){\fontsize{5.0pt}{\baselineskip}\selectfont{$n$}\!-\!\fontsize{5.0pt}{\baselineskip}\selectfont{$2$}} 
		
		\htext(-12.3 14){\fontsize{5.0pt}{\baselineskip}\selectfont{$n$}\!-\!\fontsize{5.0pt}{\baselineskip}\selectfont{$2$}}
		\htext(-13.7 14){\fontsize{5.0pt}{\baselineskip}\selectfont{$n$}\!-\!\fontsize{5.0pt}{\baselineskip}\selectfont{$2$}} 
		\htext(-12.3 17){$4$} \htext(-13.7 17){$4$}
		\htext(-12.3 19){$3$} \htext(-13.7 19){$3$}
		\htext(-12.3 21){$2$} \htext(-13.7 21){$2$}
		%
		
		\htext(-3.7 51.4){\htext(2 -1){\fontsize{10}{10}\selectfont{$B_n^{(1)}$}}}
		
		\htext(-2 -4){\fontsize{12}{12}\selectfont{$\mathbf{Stacking}\  \mathbf{Patterns}$}}
		
		\move(0 0)\rlvec(-4.3 0) \move(0 2)\rlvec(-4.3 0)\move(0 4)\rlvec(-4.3 0) \move(0 6)\rlvec(-4.3 0) \move(0 8)\rlvec(-4.3 0)
		\move(0 9)\rlvec(-4.3 0) \move(0 11)\rlvec(-4.3 0)\move(0 13)\rlvec(-4.3 0)\move(0 15)\rlvec(-4.3 0)\move(0 16)\rlvec(-4.3 0) \move(0 18)\rlvec(-4.3 0) \move(0 20)\rlvec(-4.3 0) \move(0 22)\rlvec(-4.3 0)
		\move(0 0)\rlvec(0 22.3) \move(-0.6 0)\rlvec(0 8)  \move(-1.4 0)\rlvec(0 8) \move(-2 0)\rlvec(0 22.3) \move(-0.6 9)\rlvec(0 6) \move(-1.4 9)\rlvec(0 6)  \move(-4 0)\rlvec(0 22.3)
		\move(-2.6 0)\rlvec(0 8)\move(-3.4 0)\rlvec(0 8)
		\move(-2.6 9)\rlvec(0 6)\move(-2.6 16)\rlvec(0 6) \move(-3.4 9)\rlvec(0 6)
		\move(-3.4 16)\rlvec(0 6)\move(-0.6 0)\rlvec(0.6 2) \move(-2 0)\rlvec(0.6 2) 
		\move(-0.6 20)\rlvec(0 2)\move(-1.4 20)\rlvec(0 2)
		\move(-0.6 16)\rlvec(0 2)\move(-1.4 16)\rlvec(0 2)
		\move(-0.6 18)\rlvec(0 2)\move(-1.4 18)\rlvec(0 2)
		\move(-2.6 0)\rlvec(0.6 2)\move(-4 0)\rlvec(0.6 2)
		\move(-0.6 11)\rlvec(0.6 2)\move(-2 11)\rlvec(0.6 2)
		\move(-2.6 11)\rlvec(0.6 2)\move(-4 11)\rlvec(0.6 2)
		\htext(-0.2 11.5){$n$} \htext(-0.4 12.5){$n$} 
		\htext(-1.6 11.5){$n$} \htext(-1.8 12.5){$n$}
		\htext(-2.2 11.5){$n$} \htext(-2.4 12.5){$n$} 
		\htext(-3.6 11.5){$n$} \htext(-3.8 12.5){$n$}
		\htext(-0.95 1){$\cdots$}\htext(-0.95 3){$\cdots$}\htext(-0.95 5){$\cdots$}\htext(-0.95 7){$\cdots$}\vtext(-0.95 8.5){$\cdots$}  \htext(-0.95 10){$\cdots$}  \htext(-0.95 12){$\cdots$}  \htext(-0.95 14){$\cdots$}\htext(-0.95 17){$\cdots$}  \htext(-0.95 19){$\cdots$} \vtext(-0.95 15.5){$\cdots$} \vtext(-2.95 15.5){$\cdots$} \htext(-2.95 19){$\cdots$}        \htext(-0.95 21){$\cdots$} 
		\htext(-0.2 0.5){$1$} \htext(-0.4 1.5){$0$} 
		\htext(-1.6 0.5){$1$} \htext(-1.8 1.5){$0$} 
		\htext(-0.3 3){$2$}\htext(-1.7 3){$2$}
		\htext(-0.3 5){$3$}\htext(-1.7 5){$3$}
		\htext(-0.3 7){$4$}\htext(-1.7 7){$4$}
		\htext(-0.3 10){\fontsize{5.0pt}{\baselineskip}\selectfont{$n$}\!-\!\fontsize{5.0pt}{\baselineskip}\selectfont{$1$}}
		\htext(-1.7 10){\fontsize{5.0pt}{\baselineskip}\selectfont{$n$}\!-\!\fontsize{5.0pt}{\baselineskip}\selectfont{$1$}} 
		
		\htext(-0.3 14){\fontsize{5.0pt}{\baselineskip}\selectfont{$n$}\!-\!\fontsize{5.0pt}{\baselineskip}\selectfont{$1$}}
		\htext(-1.7 14){\fontsize{5.0pt}{\baselineskip}\selectfont{$n$}\!-\!\fontsize{5.0pt}{\baselineskip}\selectfont{$1$}} 
		\htext(-0.3 17){$4$} \htext(-1.7 17){$4$}
		\htext(-0.3 19){$3$} \htext(-1.7 19){$3$}
		\htext(-0.3 21){$2$} \htext(-1.7 21){$2$}
		\htext(-2.95 1){$\cdots$}\htext(-2.95 3){$\cdots$}\htext(-2.95 5){$\cdots$}\htext(-2.95 7){$\cdots$}\vtext(-2.95 8.5){$\cdots$}  \htext(-2.95 10){$\cdots$}  \htext(-2.95 12){$\cdots$}  \htext(-2.95 14){$\cdots$} \htext(-2.95 21){$\cdots$}\htext(-2.95 17){$\cdots$}
		\htext(-0.3 -0.4){$\frac{1}{l}$} \htext(-1.7 -0.4){$\frac{1}{l}$}  \htext(-2.3 -0.4){$\frac{1}{l}$}\htext(-0.95 -0.4){$\cdots$} \htext(-2.95 -0.4){$\cdots$}\htext(-3.7 -0.4){$\frac{1}{l}$}
		\htext(-1 -1){$\underbrace{\rule{3em}{0em}}$}
		\htext(-3 -1){$\underbrace{\rule{3em}{0em}}$}
		\htext(-1 -1.6){$l$}
		\htext(-3 -1.6){$l$}
		\htext(-2.2 0.5){$1$} \htext(-2.4 1.5){$0$} 
		\htext(-3.6 0.5){$1$} \htext(-3.8 1.5){$0$} 
		\htext(-2.3 3){$2$}\htext(-3.7 3){$2$}
		\htext(-2.3 5){$3$}\htext(-3.7 5){$3$}
		\htext(-2.3 7){$4$}\htext(-3.7 7){$4$}
		\htext(-2.3 10){\fontsize{5.0pt}{\baselineskip}\selectfont{$n$}\!-\!\fontsize{5.0pt}{\baselineskip}\selectfont{$1$}}
		\htext(-3.7 10){\fontsize{5.0pt}{\baselineskip}\selectfont{$n$}\!-\!\fontsize{5.0pt}{\baselineskip}\selectfont{$1$}} 
		
		\htext(-2.3 14){\fontsize{5.0pt}{\baselineskip}\selectfont{$n$}\!-\!\fontsize{5.0pt}{\baselineskip}\selectfont{$1$}}
		\htext(-3.7 14){\fontsize{5.0pt}{\baselineskip}\selectfont{$n$}\!-\!\fontsize{5.0pt}{\baselineskip}\selectfont{$1$}} 
		\htext(-2.3 17){$4$} \htext(-3.7 17){$4$}
		\htext(-2.3 19){$3$} \htext(-3.7 19){$3$}
		\htext(-2.3 21){$2$} \htext(-3.7 21){$2$}

		%
		\htext(6.3 51.4){\htext(2 -1){\fontsize{10}{10}\selectfont{$C_n^{(1)}$}}}
		\move(9.4 0)\lvec(10 2)
		\move(8 0)\lvec(8.6 2)
			\move(7.4 0)\lvec(8 2)
		\move(6 0)\lvec(6.6 2)
		\move(10 0)\rlvec(-4.3 0) \move(10 2)\rlvec(-4.3 0)\move(10 4)\rlvec(-4.3 0) \move(10 6)\rlvec(-4.3 0) \move(10 8)\rlvec(-4.3 0)
		\move(10 9)\rlvec(-4.3 0) \move(10 11)\rlvec(-4.3 0)\move(10 13)\rlvec(-4.3 0)\move(10 15)\rlvec(-4.3 0)\move(10 16)\rlvec(-4.3 0) \move(10 18)\rlvec(-4.3 0) \move(10 20)\rlvec(-4.3 0) \move(10 22)\rlvec(-4.3 0)
		\move(10 0)\rlvec(0 22.3) \move(9.4 0)\rlvec(0 8)  \move(8.6 0)\rlvec(0 8) \move(8 0)\rlvec(0 22.3) \move(9.4 9)\rlvec(0 6) \move(8.6 9)\rlvec(0 6)  \move(6 0)\rlvec(0 22.3)
		\move(7.4 0)\rlvec(0 8)\move(6.6 0)\rlvec(0 8)
		\move(7.4 9)\rlvec(0 6)\move(7.4 16)\rlvec(0 6) \move(6.6 9)\rlvec(0 6)
		\move(6.6 16)\rlvec(0 6)
		\move(9.4 20)\rlvec(0 2)\move(8.6 20)\rlvec(0 2)
		\move(9.4 16)\rlvec(0 2)\move(8.6 16)\rlvec(0 2)
		\move(9.4 18)\rlvec(0 2)\move(8.6 18)\rlvec(0 2)

		\htext(9.05 1){$\cdots$}\htext(9.05 3){$\cdots$}\htext(9.05 5){$\cdots$}\htext(9.05 7){$\cdots$}\vtext(9.05 8.5){$\cdots$}  \htext(9.05 10){$\cdots$}  \htext(9.05 12){$\cdots$}  \htext(9.05 14){$\cdots$}\htext(9.05 17){$\cdots$}  \htext(9.05 19){$\cdots$} \vtext(9.05 15.5){$\cdots$} \vtext(7.05 15.5){$\cdots$} \htext(7.05 19){$\cdots$}        \htext(7.05 21){$\cdots$} 
	    \htext(9.8 0.5){$0$} \htext(9.6 1.5){$0$} 
	    \htext(8.4 0.5){$0$} \htext(8.2 1.5){$0$}
		\htext(9.7 3){$1$}\htext(8.3 3){$1$}
		\htext(9.7 5){$2$}\htext(8.3 5){$2$}
		\htext(9.7 7){$3$}\htext(8.3 7){$3$}
		\htext(9.7 10){\fontsize{5.0pt}{\baselineskip}\selectfont{$n$}\!-\!\fontsize{5.0pt}{\baselineskip}\selectfont{$1$}}
		\htext(8.3 10){\fontsize{5.0pt}{\baselineskip}\selectfont{$n$}\!-\!\fontsize{5.0pt}{\baselineskip}\selectfont{$1$}} 
		\htext(9.7 12){$n$}\htext(8.3 12){$n$} 
		\htext(9.7 14){\fontsize{5.0pt}{\baselineskip}\selectfont{$n$}\!-\!\fontsize{5.0pt}{\baselineskip}\selectfont{$1$}}
		\htext(8.3 14){\fontsize{5.0pt}{\baselineskip}\selectfont{$n$}\!-\!\fontsize{5.0pt}{\baselineskip}\selectfont{$1$}} 
		\htext(9.7 17){$3$} \htext(8.3 17){$3$}
		\htext(9.7 19){$2$} \htext(8.3 19){$2$}
		\htext(9.7 21){$1$} \htext(8.3 21){$1$}		
		\htext(7.05 1){$\cdots$}\htext(7.05 3){$\cdots$}\htext(7.05 5){$\cdots$}\htext(7.05 7){$\cdots$}\vtext(7.05 8.5){$\cdots$}  \htext(7.05 10){$\cdots$}  \htext(7.05 12){$\cdots$}  \htext(7.05 14){$\cdots$} \htext(7.05 21){$\cdots$}\htext(7.05 17){$\cdots$}
		\htext(9.7 -0.4){$\frac{1}{2l}$} \htext(8.3 -0.4){$\frac{1}{2l}$}  \htext(7.7 -0.4){$\frac{1}{2l}$}\htext(9.05 -0.4){$\cdots$} \htext(7.05 -0.4){$\cdots$}\htext(6.3 -0.4){$\frac{1}{2l}$}
		\htext(9 -1){$\underbrace{\rule{3em}{0em}}$}
		\htext(7 -1){$\underbrace{\rule{3em}{0em}}$}
		\htext(9 -1.6){$2l$}
		\htext(7 -1.6){$2l$}
		\htext(7.8 0.5){$0$} \htext(7.6 1.5){$0$} 
        \htext(6.4 0.5){$0$} \htext(6.2 1.5){$0$} 
		\htext(7.7 3){$1$}\htext(6.3 3){$1$}
		\htext(7.7 5){$2$}\htext(6.3 5){$2$}
		\htext(7.7 7){$3$}\htext(6.3 7){$3$}
		\htext(7.7 10){\fontsize{5.0pt}{\baselineskip}\selectfont{$n$}\!-\!\fontsize{5.0pt}{\baselineskip}\selectfont{$1$}}
		\htext(6.3 10){\fontsize{5.0pt}{\baselineskip}\selectfont{$n$}\!-\!\fontsize{5.0pt}{\baselineskip}\selectfont{$1$}} 
		\htext(7.7 12){$n$}\htext(6.3 12){$n$} 
		\htext(7.7 14){\fontsize{5.0pt}{\baselineskip}\selectfont{$n$}\!-\!\fontsize{5.0pt}{\baselineskip}\selectfont{$1$}}
		\htext(6.3 14){\fontsize{5.0pt}{\baselineskip}\selectfont{$n$}\!-\!\fontsize{5.0pt}{\baselineskip}\selectfont{$1$}} 
		\htext(7.7 17){$3$} \htext(6.3 17){$3$}
		\htext(7.7 19){$2$} \htext(6.3 19){$2$}
		\htext(7.7 21){$1$} \htext(6.3 21){$1$}
		\esegment	
	\end{texdraw}
\end{center}
\end{definition}

\vskip 2mm 
\clearpage
\begin{definition}\label{level l delta slice}
For each stacking pattern, the first complete cycle  of  a slice with $\frac{1}{l}$-width (or $\frac{1}{2l}$-width in $C_{n}^{(1)}$-case) from the bottom of a pattern is called a level-$l$ $\delta$-slice.
\end{definition}

The stacking patterns in Definition \ref{stacking patterns} consist of repeated $\delta$-slices extending to upward and left infinitely. In the following text, a level-$l$ $\delta$-slice is referred to as a  $\delta$-slice for short. For simplicity, we could present a $\delta$-slice as a sequence of colors corresponding to it. For example:

\vspace{-1pt}

$$
\ A_{2n}^{(2)}: (0,0,1,2,\cdots, n-1, n, n-1,\cdots,2,1),\quad\ 
$$
$$
\ D_{n+1}^{(2)}: (0,0,1,2,\cdots, n-1, n,n, n-1,\cdots,2,1),
$$
$$
A_{2n-1}^{(2)}: (0,1,2,\cdots, n-1, n, n-1,\cdots,2),\quad\quad\
$$
$$
D_{n}^{(1)}: (0,1,2,\cdots, n-2, n-1, n, n-2,\cdots,2),\
$$
$$
B_{n}^{(1)}: (0,1,2,\cdots, n-1, n, n, n-1,\cdots,2),\quad\ \ \
$$
$$
C_{n}^{(1)}: (0,1,2,\cdots, n-1, n, n-1,\cdots,2,1).\quad\ \ \
$$

\vspace{12pt}

\begin{example}
We list the stacking patterns of types $A_4^{(2)}$, $A_5^{(2)}$, $B_3^{(1)}$ for level-$2$ cases.

\vskip 6mm 
 
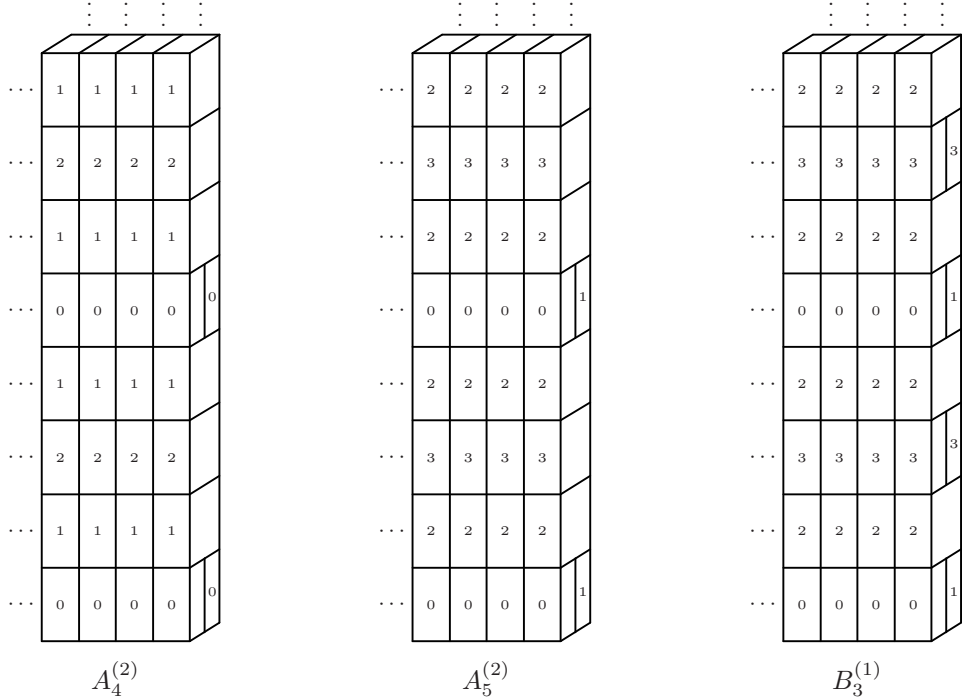
	\begin{center}  	
	\begin{texdraw}	
		\fontsize{5}{5}\selectfont
		\drawdim em
		\setunitscale 2.3
		\move(0 0)\lvec(0 16)	
		\move(0 0)\lvec(4 0)	
		\move(4 0)\lvec(4 16)	
		\move(0 16)\lvec(4 16)	
		\move(1 0)\lvec(1 16)	
		\move(2 0)\lvec(2 16)	
		\move(3 0)\lvec(3 16)	
		\move(0 2)\lvec(4 2)	
		\move(0 4)\lvec(4 4)
		\move(0 6)\lvec(4 6)
		\move(0 8)\lvec(4 8)
		\move(0 10)\lvec(4 10)
		\move(0 12)\lvec(4 12)
		\move(0 14)\lvec(4 14)
		\move(0 16)\lvec(0.8 16.5)
		\move(1 16)\lvec(1.8 16.5)
		\move(2 16)\lvec(2.8 16.5)
		\move(3 16)\lvec(3.8 16.5)
		\move(4 16)\lvec(4.8 16.5)
		\move(4 14)\lvec(4.8 14.5)
		\move(4 12)\lvec(4.8 12.5)
		\move(4 10)\lvec(4.8 10.5)
		\move(4 8)\lvec(4.8 8.5)
		\move(4 6)\lvec(4.8 6.5)
		\move(4 4)\lvec(4.8 4.5)
		\move(4 2)\lvec(4.8 2.5)
		\move(4 0)\lvec(4.8 0.5)
		\move(4.8 0.5)\lvec(4.8 16.5)
		\move(0.8 16.5)\lvec(4.8 16.5)
		\move(4.4 8.25)\lvec(4.4 10.25)
		\move(4.4 0.25)\lvec(4.4 2.25)
		
		\htext(0.5 1){$0$}
		\htext(0.5 3){$1$}
		\htext(0.5 5){$2$}
		\htext(0.5 7){$1$}
		\htext(0.5 9){$0$}
		\htext(0.5 11){$1$}
		\htext(0.5 13){$2$}
		\htext(0.5 15){$1$}
		
		\move(1 0)
		\bsegment
		\htext(0.5 1){$0$}
		\htext(0.5 3){$1$}
		\htext(0.5 5){$2$}
		\htext(0.5 7){$1$}
		\htext(0.5 9){$0$}
		\htext(0.5 11){$1$}
		\htext(0.5 13){$2$}
		\htext(0.5 15){$1$}
		\esegment
		
		\move(2 0)
		\bsegment
		\htext(0.5 1){$0$}
		\htext(0.5 3){$1$}
		\htext(0.5 5){$2$}
		\htext(0.5 7){$1$}
		\htext(0.5 9){$0$}
		\htext(0.5 11){$1$}
		\htext(0.5 13){$2$}
		\htext(0.5 15){$1$}
		\esegment
		
		\move(3 0)
		\bsegment
		\htext(0.5 1){$0$}
		\htext(0.5 3){$1$}
		\htext(0.5 5){$2$}
		\htext(0.5 7){$1$}
		\htext(0.5 9){$0$}
		\htext(0.5 11){$1$}
		\htext(0.5 13){$2$}
		\htext(0.5 15){$1$}
		\esegment
		\htext(4.6 1.36){$0$}	
		\htext(4.6 9.36){$0$}
		\vtext(4.3 17.1){\fontsize{8}{8}\selectfont{$\cdots$}}
		\vtext(3.3 17.1){\fontsize{8}{8}\selectfont{$\cdots$}}
		\vtext(2.3 17.1){\fontsize{8}{8}\selectfont{$\cdots$}}
		\vtext(1.3 17.1){\fontsize{8}{8}\selectfont{$\cdots$}}	
		\htext(-0.5 1){\fontsize{8}{8}\selectfont{$\cdots$}}	
		\htext(-0.5 3){\fontsize{8}{8}\selectfont{$\cdots$}}
		\htext(-0.5 5){\fontsize{8}{8}\selectfont{$\cdots$}}
		\htext(-0.5 7){\fontsize{8}{8}\selectfont{$\cdots$}}
		\htext(-0.5 9){\fontsize{8}{8}\selectfont{$\cdots$}}	
		\htext(-0.5 11){\fontsize{8}{8}\selectfont{$\cdots$}}
		\htext(-0.5 13){\fontsize{8}{8}\selectfont{$\cdots$}}
		\htext(-0.5 15){\fontsize{8}{8}\selectfont{$\cdots$}}
		\htext(2 -1){\fontsize{10}{10}\selectfont{$A_4^{(2)}$}}
		\move(10 0)
		\bsegment
		\move(0 0)\lvec(0 16)	
		\move(0 0)\lvec(4 0)	
		\move(4 0)\lvec(4 16)	
		\move(0 16)\lvec(4 16)	
		\move(1 0)\lvec(1 16)	
		\move(2 0)\lvec(2 16)	
		\move(3 0)\lvec(3 16)	
		\move(0 2)\lvec(4 2)	
		\move(0 4)\lvec(4 4)
		\move(0 6)\lvec(4 6)
		\move(0 8)\lvec(4 8)
		\move(0 10)\lvec(4 10)
		\move(0 12)\lvec(4 12)
		\move(0 14)\lvec(4 14)
		\move(0 16)\lvec(0.8 16.5)
		\move(1 16)\lvec(1.8 16.5)
		\move(2 16)\lvec(2.8 16.5)
		\move(3 16)\lvec(3.8 16.5)
		\move(4 16)\lvec(4.8 16.5)
		\move(4 14)\lvec(4.8 14.5)
		\move(4 12)\lvec(4.8 12.5)
		\move(4 10)\lvec(4.8 10.5)
		\move(4 8)\lvec(4.8 8.5)
		\move(4 6)\lvec(4.8 6.5)
		\move(4 4)\lvec(4.8 4.5)
		\move(4 2)\lvec(4.8 2.5)
		\move(4 0)\lvec(4.8 0.5)
		\move(4.8 0.5)\lvec(4.8 16.5)
		\move(0.8 16.5)\lvec(4.8 16.5)
		\move(4.4 8.25)\lvec(4.4 10.25)
		\move(4.4 0.25)\lvec(4.4 2.25)
		
		\htext(0.5 1){$0$}
		\htext(0.5 3){$2$}
		\htext(0.5 5){$3$}
		\htext(0.5 7){$2$}
		\htext(0.5 9){$0$}
		\htext(0.5 11){$2$}
		\htext(0.5 13){$3$}
		\htext(0.5 15){$2$}
		
		\move(1 0)
		\bsegment
		\htext(0.5 1){$0$}
		\htext(0.5 3){$2$}
		\htext(0.5 5){$3$}
		\htext(0.5 7){$2$}
		\htext(0.5 9){$0$}
		\htext(0.5 11){$2$}
		\htext(0.5 13){$3$}
		\htext(0.5 15){$2$}
		\esegment
		
		\move(2 0)
		\bsegment
		\htext(0.5 1){$0$}
		\htext(0.5 3){$2$}
		\htext(0.5 5){$3$}
		\htext(0.5 7){$2$}
		\htext(0.5 9){$0$}
		\htext(0.5 11){$2$}
		\htext(0.5 13){$3$}
		\htext(0.5 15){$2$}
		\esegment
		
		\move(3 0)
		\bsegment
		\htext(0.5 1){$0$}
		\htext(0.5 3){$2$}
		\htext(0.5 5){$3$}
		\htext(0.5 7){$2$}
		\htext(0.5 9){$0$}
		\htext(0.5 11){$2$}
		\htext(0.5 13){$3$}
		\htext(0.5 15){$2$}
		\esegment
		\htext(4.6 1.36){$1$}	
		\htext(4.6 9.36){$1$}
		\vtext(4.3 17.1){\fontsize{8}{8}\selectfont{$\cdots$}}
		\vtext(3.3 17.1){\fontsize{8}{8}\selectfont{$\cdots$}}
		\vtext(2.3 17.1){\fontsize{8}{8}\selectfont{$\cdots$}}
		\vtext(1.3 17.1){\fontsize{8}{8}\selectfont{$\cdots$}}	
		\htext(-0.5 1){\fontsize{8}{8}\selectfont{$\cdots$}}	
		\htext(-0.5 3){\fontsize{8}{8}\selectfont{$\cdots$}}
		\htext(-0.5 5){\fontsize{8}{8}\selectfont{$\cdots$}}
		\htext(-0.5 7){\fontsize{8}{8}\selectfont{$\cdots$}}
		\htext(-0.5 9){\fontsize{8}{8}\selectfont{$\cdots$}}	
		\htext(-0.5 11){\fontsize{8}{8}\selectfont{$\cdots$}}
		\htext(-0.5 13){\fontsize{8}{8}\selectfont{$\cdots$}}
		\htext(-0.5 15){\fontsize{8}{8}\selectfont{$\cdots$}}
		\htext(2 -1){\fontsize{10}{10}\selectfont{$A_5^{(2)}$}}
		\esegment

		\move(20 0)
		\bsegment
		\move(0 0)\lvec(0 16)	
		\move(0 0)\lvec(4 0)	
		\move(4 0)\lvec(4 16)	
		\move(0 16)\lvec(4 16)	
		\move(1 0)\lvec(1 16)	
		\move(2 0)\lvec(2 16)	
		\move(3 0)\lvec(3 16)	
		\move(0 2)\lvec(4 2)	
		\move(0 4)\lvec(4 4)
		\move(0 6)\lvec(4 6)
		\move(0 8)\lvec(4 8)
		\move(0 10)\lvec(4 10)
		\move(0 12)\lvec(4 12)
		\move(0 14)\lvec(4 14)
		\move(0 16)\lvec(0.8 16.5)
		\move(1 16)\lvec(1.8 16.5)
		\move(2 16)\lvec(2.8 16.5)
		\move(3 16)\lvec(3.8 16.5)
		\move(4 16)\lvec(4.8 16.5)
		\move(4 14)\lvec(4.8 14.5)
		\move(4 12)\lvec(4.8 12.5)
		\move(4 10)\lvec(4.8 10.5)
		\move(4 8)\lvec(4.8 8.5)
		\move(4 6)\lvec(4.8 6.5)
		\move(4 4)\lvec(4.8 4.5)
		\move(4 2)\lvec(4.8 2.5)
		\move(4 0)\lvec(4.8 0.5)
		\move(4.8 0.5)\lvec(4.8 16.5)
		\move(0.8 16.5)\lvec(4.8 16.5)
		\move(4.4 8.25)\lvec(4.4 10.25)
		\move(4.4 0.25)\lvec(4.4 2.25)
		\move(4.4 4.25)\lvec(4.4 6.25)
		\move(4.4 12.25)\lvec(4.4 14.25)
		\htext(0.5 1){$0$}
		\htext(0.5 3){$2$}
		\htext(0.5 5){$3$}
		\htext(0.5 7){$2$}
		\htext(0.5 9){$0$}
		\htext(0.5 11){$2$}
		\htext(0.5 13){$3$}
		\htext(0.5 15){$2$}
		
		\move(1 0)
		\bsegment
		\htext(0.5 1){$0$}
		\htext(0.5 3){$2$}
		\htext(0.5 5){$3$}
		\htext(0.5 7){$2$}
		\htext(0.5 9){$0$}
		\htext(0.5 11){$2$}
		\htext(0.5 13){$3$}
		\htext(0.5 15){$2$}
		\esegment
		
		\move(2 0)
		\bsegment
		\htext(0.5 1){$0$}
		\htext(0.5 3){$2$}
		\htext(0.5 5){$3$}
		\htext(0.5 7){$2$}
		\htext(0.5 9){$0$}
		\htext(0.5 11){$2$}
		\htext(0.5 13){$3$}
		\htext(0.5 15){$2$}
		\esegment
		
		\move(3 0)
		\bsegment
		\htext(0.5 1){$0$}
		\htext(0.5 3){$2$}
		\htext(0.5 5){$3$}
		\htext(0.5 7){$2$}
		\htext(0.5 9){$0$}
		\htext(0.5 11){$2$}
		\htext(0.5 13){$3$}
		\htext(0.5 15){$2$}
		\esegment
		\htext(4.6 1.36){$1$}
		\htext(4.6 5.36){$3$}	
		\htext(4.6 9.36){$1$}
		\htext(4.6 13.36){$3$}

		\vtext(4.3 17.1){\fontsize{8}{8}\selectfont{$\cdots$}}
		\vtext(3.3 17.1){\fontsize{8}{8}\selectfont{$\cdots$}}
		\vtext(2.3 17.1){\fontsize{8}{8}\selectfont{$\cdots$}}
		\vtext(1.3 17.1){\fontsize{8}{8}\selectfont{$\cdots$}}	
		\htext(-0.5 1){\fontsize{8}{8}\selectfont{$\cdots$}}	
		\htext(-0.5 3){\fontsize{8}{8}\selectfont{$\cdots$}}
		\htext(-0.5 5){\fontsize{8}{8}\selectfont{$\cdots$}}
		\htext(-0.5 7){\fontsize{8}{8}\selectfont{$\cdots$}}
		\htext(-0.5 9){\fontsize{8}{8}\selectfont{$\cdots$}}	
		\htext(-0.5 11){\fontsize{8}{8}\selectfont{$\cdots$}}
		\htext(-0.5 13){\fontsize{8}{8}\selectfont{$\cdots$}}
		\htext(-0.5 15){\fontsize{8}{8}\selectfont{$\cdots$}}
		\htext(2 -1){\fontsize{10}{10}\selectfont{$B_3^{(1)}$}}
		\esegment
		
	\end{texdraw}
\end{center}   
	
\end{example}

\vskip 3mm

\begin{example}\label{level-15 Young Column}
A level-$15$ Young column of type $B_{4}^{(1)}$ is given below.

\vskip 2mm

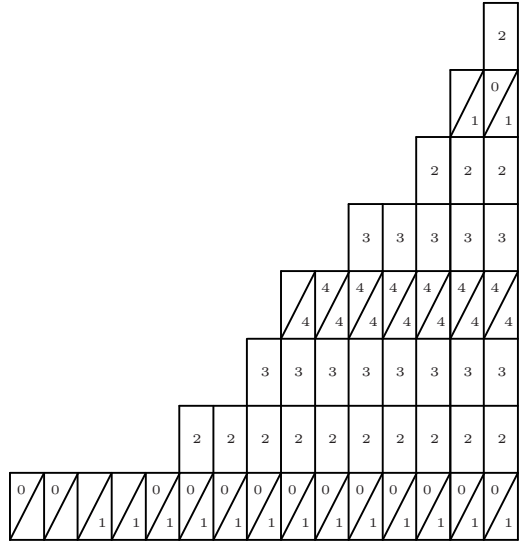
\begin{center}
	\begin{texdraw}
		\fontsize{5}{5}\selectfont
		\drawdim em
		\setunitscale 2.1
		
		\move(0 0)\rlvec(-15 0)	
		\move(0 0)\rlvec(0 12)
		\move(0 2)\rlvec(-15 0)
		\move(-1 0 )\rlvec(0 12)
		\move(-2 0 )\rlvec(0 12)
		\move(-2 0 )\rlvec(0 12)
		\move(-3 0 )\rlvec(0 10)
		\move(-4 0 )\rlvec(0 10)
		\move(-5 0 )\rlvec(0 4)
		\move(-6 0 )\rlvec(0 8)
		\move(-7 0 )\rlvec(0 8)
		\move(-8 0 )\rlvec(0 6)
		\move(-9 0 )\rlvec(0 4)
		\move(-10 0 )\rlvec(0 4)
		\move(-11 0 )\rlvec(0 2)
		\move(-12 0 )\rlvec(0 2)
		\move(-13 0 )\rlvec(0 2)
		\move(-14 0 )\rlvec(0 2)
		\move(-15 0 )\rlvec(0 2)
		\move(0 12)\rlvec(-2 0)
		\move(-3 10)\rlvec(-1 0)
		
		\move(-3 12)\lvec(-2 12)
		\move(-3 10)\lvec(-2 10)	
		\move(-3 8)\lvec(-2 8)
		\move(-3 6)\lvec(-2 6)
		
		\move(-6 8)\rlvec(-1 0)	
		
		\move(0 14)\rlvec(0 2)
		\move(-2 12)\lvec(-2 14)
		\move(-2 12)\lvec(-1 14) 
		\move(-2 14)\lvec(-1 14)
		\move(0 16)\rlvec(-1 0)
		\move(-1 14)\rlvec(0 2)	
		\htext(-0.5 15){$2$}	
		\move(0 4)\rlvec(-10 0)
		\move(-6 6)\rlvec(-2 0)
		\move(-3 6)\rlvec(-1 0)
		\move(-3 8)\rlvec(-1 0)
		\move(0 6)\rlvec(-2 0)
		\move(0 8)\rlvec(-2 0)
		\move(0 10)\rlvec(-2 0)
		\move(-1 0)\rlvec(1 2)
		\move(-2 0)\rlvec(1 2)
		\move(-3 0)\rlvec(1 2)
		\move(-3 10)\rlvec(0 2)
		\move(-5 8)\lvec(-5 10)
		\move(-5 10)\lvec(-4 10)
		\htext(-4.5 9){$3$}	
		\move(-6 6)\rlvec(1 0)
		\move(-6 8)\rlvec(1 0)
		\move(-6 6)\rlvec(1 2)
		\htext(-5.5 5){$3$}

		\move(-4 0)\rlvec(1 2)
		\move(-5 0)\rlvec(1 2)
		\move(-5 4)\rlvec(0 4)
		\move(-4 8)\rlvec(-1 0)
		\move(-4 6)\rlvec(-1 0)
		\htext(-4.5 5){$3$}
		\move(-5 6)\rlvec(1 2)
		\htext(-4.3 6.5){$4$}
		\htext(-4.7 7.5){$4$}
		\htext(-5.3 6.5){$4$}
		\htext(-5.7 7.5){$4$}
		
		\move(-6 0)\rlvec(1 2)
		\move(-7 0)\rlvec(1 2)
		\move(-8 0)\rlvec(1 2)
		\move(-9 0)\rlvec(1 2)
		\move(-10 0)\rlvec(1 2)
		\move(-11 0)\rlvec(1 2)
		\move(-12 0)\rlvec(1 2)
		\move(-13 0)\rlvec(1 2)
		\move(-14 0)\rlvec(1 2)
		\move(-15 0)\rlvec(1 2)
		\move(-7 6)\rlvec(1 2)
		
		\move(-4 6)\rlvec(1 2)
		\move(-2 6)\rlvec(1 2)
		\move(-1 6)\rlvec(1 2)
		\htext(-0.3 0.5){$1$}
		\htext(-0.7 1.5){$0$}
		\htext(-0.7 13.5){$0$}
		\htext(-0.5 3){$2$}
		\htext(-0.5 5){$3$}
		\htext(-0.3 6.5){$4$}
		\htext(-0.7 7.5){$4$}
		\htext(-0.5 9){$3$}
		\htext(-0.5 11){$2$}
		\move(0 12)\rlvec(0 2)
		\move(-1 12)\rlvec(0 2)
		\move(0 14)\rlvec(-1 0)	
		\move(-1 12)\rlvec(1 2)	
		\htext(-0.3 12.5){$1$}
		\htext(-1.3 0.5){$1$}
		\htext(-1.3 12.5){$1$}
		\htext(-1.7 1.5){$0$}
		\htext(-1.5 3){$2$}
		\htext(-1.5 5){$3$}
		\htext(-1.3 6.5){$4$}
		\htext(-1.7 7.5){$4$}
		\htext(-1.5 9){$3$}
		\htext(-1.5 11){$2$}
		
		\htext(-2.3 0.5){$1$}
		\htext(-2.7 1.5){$0$}
		\htext(-2.5 3){$2$}	
		\htext(-2.5 5){$3$}
		\move(-3 6)\lvec(-2 8)	
		\htext(-2.3 6.5){$4$}
		\htext(-2.7 7.5){$4$}	
		\htext(-2.5 9){$3$}	
		\htext(-2.5 11){$2$}				
		\htext(-3.3 0.5){$1$}
		\htext(-3.7 1.5){$0$}
		\htext(-3.5 3){$2$}
		\htext(-3.5 5){$3$}	
		\htext(-3.3 6.5){$4$}
		\htext(-3.7 7.5){$4$}
		\htext(-3.5 9){$3$}
		
		\htext(-4.3 0.5){$1$}
		\htext(-4.7 1.5){$0$}
		\htext(-4.5 3){$2$}
		
		\htext(-5.3 0.5){$1$}
		\htext(-5.7 1.5){$0$}
		\htext(-5.5 3){$2$}
		
		\htext(-6.3 0.5){$1$}
		\htext(-6.7 1.5){$0$}
		\htext(-6.5 3){$2$}
		\htext(-6.5 5){$3$}
		\htext(-6.3 6.5){$4$}

		\htext(-7.3 0.5){$1$}
		\htext(-7.7 1.5){$0$}
		\htext(-7.5 3){$2$}
		\htext(-7.5 5){$3$}

		\htext(-8.3 0.5){$1$}
		\htext(-8.7 1.5){$0$}
		\htext(-8.5 3){$2$}

		\htext(-9.3 0.5){$1$}
		\htext(-9.7 1.5){$0$}
		\htext(-9.5 3){$2$}

		\htext(-10.3 0.5){$1$}
		\htext(-11.3 0.5){$1$}
		\htext(-12.3 0.5){$1$}
		\htext(-10.7 1.5){$0$}
		\htext(-13.7 1.5){$0$}
		\htext(-14.7 1.5){$0$}
	\end{texdraw}
\end{center}

\end{example}

\vskip 3mm

 \subsection{Level-$l$ reduced Young columns}\label{Level-$l$ reduced Young Column}
 \hfill 
 
 \vskip 2mm 
 
  A $\delta$-slice in a level-$l$ Young column is called {\it removable} if there is at least one block on the top of the $\delta$-slice. A level-$l$ Young column is called {\it reduced} if there is no removable $\delta$-slice. We list the level-$l$ reduced Young columns corresponding to the vectors in the perfect crystals as follows:

\vskip 2mm 

We denote the number of slices by $x_i$, $\bar{x}_i$ and $l'$ in the following statements.

\vskip 2mm 
\begin{enumerate}
\item\label{(1)} $A_{2n}^{(2)}$ case :
Level-$l$ reduced Young column $C_{(x_1,\cdots,x_n,\bar{x}_n,\cdots,\bar{x}_1)}$,
where $\sum_{i=1}^{n}(x_i+\bar{x}_i)\le l$, $x_i, \bar{x}_i\ge 0$ and $l'=l-\sum_{i=1}^{n}(x_i+\bar{x}_i)$.

\vskip 2mm 

\begin{center}
	\begin{texdraw}
		\fontsize{3}{3}\selectfont
		\drawdim em
		\setunitscale 2
		\move(0 0)\lvec(2 0)
		\move(0 0)\lvec(0 2)
		\move(0 2)\lvec(2 2)
		\move(0.6 0)\lvec(0.6 2)
		\move(1.4 0)\lvec(1.4 2)
		\move(0 0)\lvec(0.6 2)
		\move(1.4 0)\lvec(2 2)
		\htext(0.4 0.4){$0$}
		\htext(1.8 0.4){$0$}
		\htext(1.05 1){$\cdots$}
		\htext(1 -0.3){$\underbrace{\rule{3.5em}{0em}}$}
		\htext(1 -0.6){$l'$}
		\move(2 0)
		\bsegment
		\move(0 0)\lvec(2 0)
		\move(0 0)\lvec(0 2)
		\move(0 2)\lvec(2 2)
		\move(0.6 0)\lvec(0.6 2)
		\move(1.4 0)\lvec(1.4 2)
		\htext(1.05 1){$\cdots$}
		\move(0 0)\lvec(0.6 2)
		\move(1.4 0)\lvec(2 2)
		\htext(0.2 1.6){$0$}
		\htext(1.6 1.6){$0$}
		\htext(0.4 0.4){$0$}
		\htext(1.8 0.4){$0$}
		\htext(1 -0.3){$\underbrace{\rule{3.5em}{0em}}$}
		\htext(1 -0.7){$x_1$}
		\esegment
		
		\move(4 0)
		\bsegment
		\move(0 0)\lvec(2 0)
		\move(0 0)\lvec(0 2)
		\move(0 2)\lvec(2 2)
		\move(0.6 0)\lvec(0.6 4)
		\move(1.4 0)\lvec(1.4 4)
		\htext(1.05 1){$\cdots$}
		\htext(1.05 3){$\cdots$}
		\move(0 2)\lvec(0 4)
		\move(2 2)\lvec(2 4)
		\move(0 4)\lvec(2 4)
		\move(0 0)\lvec(0.6 2)
		\move(1.4 0)\lvec(2 2)
		\htext(0.2 1.6){$0$}
		\htext(1.6 1.6){$0$}
		\htext(0.4 0.4){$0$}
		\htext(1.8 0.4){$0$}
		\htext(0.3 3){$1$}
		\htext(1.7 3){$1$}
		\htext(1 -0.3){$\underbrace{\rule{3.5em}{0em}}$}
		\htext(1 -0.7){$x_2$}
		\esegment
		
		\move(6 0)
		\bsegment
		\move(0 0)\lvec(2 0)
		\move(0 0)\lvec(0 2)
		\move(0 2)\lvec(2 2)
		\move(2 0)\lvec(2 6)
		\move(0.6 0)\lvec(0.6 6)
		\move(1.4 0)\lvec(1.4 6)
		\htext(1.05 1){$\cdots$}
		\htext(1.05 3){$\cdots$}
		\htext(1.05 5){$\cdots$}
		\move(0 4)\lvec(0 6)
		\move(0 6)\lvec(2 6)
		\move(0 4)\lvec(2 4)
		\htext(2.55 1){\fontsize{8}{8}\selectfont$\cdots$}
		\move(0 0)\lvec(0.6 2)
		\move(1.4 0)\lvec(2 2)
		\htext(0.2 1.6){$0$}
		\htext(1.6 1.6){$0$}
		\htext(0.4 0.4){$0$}
		\htext(1.8 0.4){$0$}
		\htext(0.3 3){$1$}
		\htext(1.7 3){$1$}
		\htext(0.3 5){$2$}
		\htext(1.7 5){$2$}
		\htext(1 -0.3){$\underbrace{\rule{3.5em}{0em}}$}
		\htext(1 -0.7){$x_3$}
		\esegment

		\move(9 0)
		\bsegment
		\move(0 0)\lvec(2 0)
		\move(0 0)\lvec(0 2)
		\move(0 2)\lvec(2 2)
		\move(0 2)\lvec(0 7)
		\move(0 7)\lvec(2 7)
		\move(0 4)\lvec(2 4)
		\move(0 5)\lvec(2 5)
		\move(0.6 0)\lvec(0.6 4)
		\move(1.4 0)\lvec(1.4 4)
		\move(0.6 5)\lvec(0.6 7)
		\move(1.4 5)\lvec(1.4 7)
		\htext(1.05 1){$\cdots$}
		\htext(1.05 3){$\cdots$}
		\vtext(1.05 4.5){$\cdots$}
		\htext(1.05 6){$\cdots$}
		\move(0 0)\lvec(0.6 2)
		\move(1.4 0)\lvec(2 2)
		\htext(0.2 1.6){$0$}
		\htext(1.6 1.6){$0$}
		\htext(0.4 0.4){$0$}
		\htext(1.8 0.4){$0$}
		\htext(0.3 3){$1$}
		\htext(1.7 3){$1$}
		\htext(0.3 6){$i$-$1$}
		\htext(1.7 6){$i$-$1$}
		\htext(1 -0.3){$\underbrace{\rule{3.5em}{0em}}$}
		\htext(1 -0.7){$x_i$}
		\esegment
		
		\move(11 0)
		\bsegment
		\move(0 0)\lvec(2 0)
		\move(0 0)\lvec(0 2)
		\move(0 2)\lvec(2 2)
		\move(0 2)\lvec(0 9)
		\move(2 0)\lvec(2 9)
		\move(0 9)\lvec(2 9)
		\move(0 5)\lvec(2 5)
		\move(0 4)\lvec(2 4)
		\move(0 7)\lvec(2 7)
		\move(0.6 0)\lvec(0.6 4)
		\move(1.4 0)\lvec(1.4 4)
		\move(0.6 5)\lvec(0.6 9)
		\move(1.4 5)\lvec(1.4 9)
		\htext(1.05 1){$\cdots$}
		\htext(1.05 3){$\cdots$}
		\vtext(1.05 4.5){$\cdots$}
		\htext(1.05 6){$\cdots$}
		\htext(1.05 8){$\cdots$}
		\move(0 0)\lvec(0.6 2)
		\move(1.4 0)\lvec(2 2)
		\htext(0.2 1.6){$0$}
		\htext(1.6 1.6){$0$}
		\htext(0.4 0.4){$0$}
		\htext(1.8 0.4){$0$}
		\htext(0.3 3){$1$}
		\htext(1.7 3){$1$}
		\htext(0.3 6){$i$-$1$}
		\htext(1.7 6){$i$-$1$}
		\htext(0.3 8){$i$}
		\htext(1.7 8){$i$}
		\htext(2.55 1){\fontsize{8}{8}\selectfont$\cdots$}
		\htext(1 -0.3){$\underbrace{\rule{3.5em}{0em}}$}
		\htext(1 -0.7){$x_{i \!+\! 1}$}
		\esegment
		
		\move(14 0)
		\bsegment
		\move(0 0)\lvec(2 0)
		\move(0 0)\lvec(0 2)
		\move(0 2)\lvec(2 2)
		\move(0 2)\lvec(0 11)
		\move(0 11)\lvec(2 11)
		
		\move(0 4)\lvec(2 4)
		\move(0 6)\lvec(2 6)
		\move(0 8)\lvec(2 8)
		\move(0 9)\lvec(2 9)
		\move(0.6 0)\lvec(0.6 8)
		\move(1.4 0)\lvec(1.4 8)
		\move(0.6 9)\lvec(0.6 11)
		\move(1.4 9)\lvec(1.4 11)
		
		\htext(1.05 1){$\cdots$}
		\htext(1.05 3){$\cdots$}
		\htext(1.05 5){$\cdots$}
		\htext(1.05 7){$\cdots$}
		\vtext(1.05 8.5){$\cdots$}
		\htext(1.05 10){$\cdots$}
		
		\move(0 0)\lvec(0.6 2)
		\move(1.4 0)\lvec(2 2)
		\htext(0.2 1.6){$0$}
		\htext(1.6 1.6){$0$}
		\htext(0.4 0.4){$0$}
		\htext(1.8 0.4){$0$}
		\htext(0.3 3){$1$}
		\htext(1.7 3){$1$}
		\htext(0.3 5){$2$}
		\htext(1.7 5){$2$}
		\htext(0.3 7){$3$}
		\htext(1.7 7){$3$}
		\htext(0.3 10){$n$\!-\!$1$}
		\htext(1.7 10){$n$\!-\!$1$}
		\htext(1 -0.3){$\underbrace{\rule{3.5em}{0em}}$}
		\htext(1 -0.7){$x_{n}$}
		\esegment
		
		\move(16 0)
		\bsegment
		\move(0 0)\lvec(2 0)
		\move(0 0)\lvec(0 2)
		\move(0 2)\lvec(2 2)
		\move(0 2)\lvec(0 13)
		\move(2 0)\lvec(2 13)
		\move(0 13)\lvec(2 13)
		
		\move(0 4)\lvec(2 4)
		\move(0 6)\lvec(2 6)
		\move(0 8)\lvec(2 8)
		\move(0 9)\lvec(2 9)
		\move(0 11)\lvec(2 11)
		
		\move(0.6 0)\lvec(0.6 8)
		\move(1.4 0)\lvec(1.4 8)
		
		\move(0.6 9)\lvec(0.6 13)
		\move(1.4 9)\lvec(1.4 13)
		
		\htext(1.05 1){$\cdots$}
		\htext(1.05 3){$\cdots$}
		\htext(1.05 5){$\cdots$}
		\htext(1.05 7){$\cdots$}
		\vtext(1.05 8.5){$\cdots$}
		\htext(1.05 10){$\cdots$}
		\htext(1.05 12){$\cdots$}
		\move(0 0)\lvec(0.6 2)
		\move(1.4 0)\lvec(2 2)
		\htext(0.2 1.6){$0$}
		\htext(1.6 1.6){$0$}
		\htext(0.4 0.4){$0$}
		\htext(1.8 0.4){$0$}
		\htext(0.3 3){$1$}
		\htext(1.7 3){$1$}
		\htext(0.3 5){$2$}
		\htext(1.7 5){$2$}
		\htext(0.3 7){$3$}
		\htext(1.7 7){$3$}
		\htext(0.3 10){$n$\!-\!$1$}
		\htext(1.7 10){$n$\!-\!$1$}
		\htext(0.3 12){$n$}
		\htext(1.7 12){$n$}
		\htext(1 -0.3){$\underbrace{\rule{3.5em}{0em}}$}
		\htext(1 -0.7){$\bar{x}_{n}$}
		\htext(2.55 1){\fontsize{8}{8}\selectfont$\cdots$}
		\esegment
		
		\move(19 0)
		\bsegment
		\move(0 0)\lvec(2 0)
		\move(0 0)\lvec(0 2)
		\move(0 2)\lvec(2 2)
		\move(0 2)\lvec(0 14)
		\move(0 14)\lvec(2 14)
		\move(0 12)\lvec(2 12)
		\move(0 11)\lvec(2 11)
		\move(0 9)\lvec(2 9)
		\move(0 8)\lvec(2 8)
		\move(0 6)\lvec(2 6)
		\move(0 4)\lvec(2 4)
		
		\move(0.6 9)\lvec(0.6 11)
		\move(1.4 9)\lvec(1.4 11)
		
		\move(0.6 12)\lvec(0.6 14)
		\move(1.4 12)\lvec(1.4 14)
		
		\move(0.6 0)\lvec(0.6 8)
		\move(1.4 0)\lvec(1.4 8)
		\htext(1.05 1){$\cdots$}
		\htext(1.05 3){$\cdots$}
		\htext(1.05 5){$\cdots$}
		\htext(1.05 7){$\cdots$}
		\vtext(1.05 8.5){$\cdots$}
		\htext(1.05 10){$\cdots$}
		\vtext(1.05 11.5){$\cdots$}
		\htext(1.05 13){$\cdots$}
		\move(0 0)\lvec(0.6 2)
		\move(1.4 0)\lvec(2 2)
		\htext(0.2 1.6){$0$}
		\htext(1.6 1.6){$0$}
		\htext(0.4 0.4){$0$}
		\htext(1.8 0.4){$0$}
		\htext(0.3 3){$1$}
		\htext(1.7 3){$1$}
		\htext(0.3 5){$2$}
		\htext(1.7 5){$2$}
		\htext(0.3 7){$3$}
		\htext(1.7 7){$3$}
		\htext(0.3 10){$n$\!-\!$1$}
		\htext(1.7 10){$n$\!-\!$1$}
		\htext(0.3 13){$i$\!+\!$1$}
		\htext(1.7 13){$i$\!+\!$1$}
		\htext(1 -0.3){$\underbrace{\rule{3.5em}{0em}}$}
		\htext(1 -0.7){$\bar{x}_{i\!+\! 1}$}
		\esegment
		
		\move(21 0)
		\bsegment
		\move(0 0)\lvec(2 0)
		\move(0 0)\lvec(0 2)
		\move(0 2)\lvec(2 2)
		\move(0 2)\lvec(0 16)
		\move(2 0)\lvec(2 16)
		\move(0 16)\lvec(2 16)
		
		\move(0 14)\lvec(2 14)
		\move(0 12)\lvec(2 12)
		\move(0 11)\lvec(2 11)
		\move(0 9)\lvec(2 9)
		\move(0 8)\lvec(2 8)
		\move(0 6)\lvec(2 6)
		\move(0 4)\lvec(2 4)
		
		\move(0.6 9)\lvec(0.6 11)
		\move(1.4 9)\lvec(1.4 11)
		
		\move(0.6 12)\lvec(0.6 16)
		\move(1.4 12)\lvec(1.4 16)
		
		\move(0.6 0)\lvec(0.6 8)
		\move(1.4 0)\lvec(1.4 8)
		\htext(1.05 1){$\cdots$}
		\htext(1.05 3){$\cdots$}
		\htext(1.05 5){$\cdots$}
		\htext(1.05 7){$\cdots$}
		\vtext(1.05 8.5){$\cdots$}
		\htext(1.05 10){$\cdots$}
		\vtext(1.05 11.5){$\cdots$}
		\htext(1.05 13){$\cdots$}
		\htext(1.05 15){$\cdots$}
		\move(0 0)\lvec(0.6 2)
		\move(1.4 0)\lvec(2 2)
		\htext(0.2 1.6){$0$}
		\htext(1.6 1.6){$0$}
		\htext(0.4 0.4){$0$}
		\htext(1.8 0.4){$0$}
		\htext(0.3 3){$1$}
		\htext(1.7 3){$1$}
		\htext(0.3 5){$2$}
		\htext(1.7 5){$2$}
		\htext(0.3 7){$3$}
		\htext(1.7 7){$3$}
		\htext(0.3 10){$n$\!-\!$1$}
		\htext(1.7 10){$n$\!-\!$1$}
		\htext(0.3 13){$i$\!+\!$1$}
		\htext(1.7 13){$i$\!+\!$1$}
		\htext(0.3 15){$i$}
		\htext(1.7 15){$i$}
		\htext(2.55 1){\fontsize{8}{8}\selectfont$\cdots$}
		\htext(1 -0.3){$\underbrace{\rule{3.5em}{0em}}$}
		\htext(1 -0.7){$\bar{x}_{i}$}
		\esegment
		
		\move(24 0)
		\bsegment
		\move(0 0)\lvec(2 0)
		\move(0 0)\lvec(0 2)
		\move(0 2)\lvec(2 2)
		\move(0 2)\lvec(0 18)
		\move(0 4)\lvec(2 4)
		\move(0 6)\lvec(2 6)
		\move(0 8)\lvec(2 8)
		\move(0 9)\lvec(2 9)
		\move(0 11)\lvec(2 11)
		\move(0 13)\lvec(2 13)
		\move(0 15)\lvec(2 15)
		\move(0 16)\lvec(2 16)
		\move(0 18)\lvec(2 18)
		\move(0.6 0)\lvec(0.6 8)
		\move(1.4 0)\lvec(1.4 8)
		\move(0.6 9)\lvec(0.6 15)
		\move(1.4 9)\lvec(1.4 15)
		\move(0.6 16)\lvec(0.6 18)
		\move(1.4 16)\lvec(1.4 18)
		\htext(1.05 1){$\cdots$}
		\htext(1.05 3){$\cdots$}
		\htext(1.05 5){$\cdots$}
		\htext(1.05 7){$\cdots$}
		\vtext(1.05 8.5){$\cdots$}
		\htext(1.05 10){$\cdots$}
		\htext(1.05 12){$\cdots$}
		\htext(1.05 14){$\cdots$}
		\vtext(1.05 15.5){$\cdots$}
		\htext(1.05 17){$\cdots$}
		\move(0 0)\lvec(0.6 2)
		\move(1.4 0)\lvec(2 2)
		\htext(0.2 1.6){$0$}
		\htext(1.6 1.6){$0$}
		\htext(0.4 0.4){$0$}
		\htext(1.8 0.4){$0$}
		\htext(0.3 3){$1$}
		\htext(1.7 3){$1$}
		\htext(0.3 5){$2$}
		\htext(1.7 5){$2$}
		\htext(0.3 7){$3$}
		\htext(1.7 7){$3$}
		\htext(0.3 10){$n$\!-\!$1$}
		\htext(1.7 10){$n$\!-\!$1$}
		\htext(0.3 12){$n$}
		\htext(1.7 12){$n$}
		\htext(0.3 14){$n$\!-\!$1$}
		\htext(1.7 14){$n$\!-\!$1$}
		\htext(0.3 17){$3$}
		\htext(1.7 17){$3$}
		\htext(1 -0.3){$\underbrace{\rule{3.5em}{0em}}$}
		\htext(1 -0.7){$\bar{x}_{3}$}
		\esegment
		
		\move(26 0)
		\bsegment
		\move(0 0)\lvec(2 0)
		\move(0 0)\lvec(0 2)
		\move(0 2)\lvec(2 2)
		\move(0 2)\lvec(0 20)
		\move(0 4)\lvec(2 4)
		\move(0 6)\lvec(2 6)
		\move(0 8)\lvec(2 8)
		\move(0 9)\lvec(2 9)
		\move(0 11)\lvec(2 11)
		\move(0 13)\lvec(2 13)
		\move(0 15)\lvec(2 15)
		\move(0 16)\lvec(2 16)
		\move(0 18)\lvec(2 18)
		\move(0 20)\lvec(2 20)
		
		\move(0.6 0)\lvec(0.6 8)
		\move(1.4 0)\lvec(1.4 8)
		\move(0.6 9)\lvec(0.6 15)
		\move(1.4 9)\lvec(1.4 15)
		\move(0.6 16)\lvec(0.6 20)
		\move(1.4 16)\lvec(1.4 20)
		\htext(1.05 1){$\cdots$}
		\htext(1.05 3){$\cdots$}
		\htext(1.05 5){$\cdots$}
		\htext(1.05 7){$\cdots$}
		\vtext(1.05 8.5){$\cdots$}
		\htext(1.05 10){$\cdots$}
		\htext(1.05 12){$\cdots$}
		\htext(1.05 14){$\cdots$}
		\vtext(1.05 15.5){$\cdots$}
		\htext(1.05 17){$\cdots$}
		\htext(1.05 19){$\cdots$}
		\move(0 0)\lvec(0.6 2)
		\move(1.4 0)\lvec(2 2)
		\htext(0.2 1.6){$0$}
		\htext(1.6 1.6){$0$}
		\htext(0.4 0.4){$0$}
		\htext(1.8 0.4){$0$}
		\htext(0.3 3){$1$}
		\htext(1.7 3){$1$}
		\htext(0.3 5){$2$}
		\htext(1.7 5){$2$}
		\htext(0.3 7){$3$}
		\htext(1.7 7){$3$}
		\htext(0.3 10){$n$\!-\!$1$}
		\htext(1.7 10){$n$\!-\!$1$}
		\htext(0.3 12){$n$}
		\htext(1.7 12){$n$}
		\htext(0.3 14){$n$\!-\!$1$}
		\htext(1.7 14){$n$\!-\!$1$}
		\htext(0.3 17){$3$}
		\htext(1.7 17){$3$}
		\htext(0.3 19){$2$}
		\htext(1.7 19){$2$}
		\htext(1 -0.3){$\underbrace{\rule{3.5em}{0em}}$}
		\htext(1 -0.7){$\bar{x}_{2}$}
		\esegment
		
		\move(28 0)
		\bsegment
		\move(0 0)\lvec(2 0)
		\move(0 0)\lvec(0 2)
		\move(0 2)\lvec(2 2)
		\move(2 0)\lvec(2 22)
		\move(0 2)\lvec(0 22)
		\move(0 22)\lvec(2 22)
		
		\move(0 4)\lvec(2 4)
		\move(0 6)\lvec(2 6)
		\move(0 8)\lvec(2 8)
		\move(0 9)\lvec(2 9)
		\move(0 11)\lvec(2 11)
		\move(0 13)\lvec(2 13)
		\move(0 15)\lvec(2 15)
		\move(0 16)\lvec(2 16)
		\move(0 18)\lvec(2 18)
		\move(0 20)\lvec(2 20)

		\move(0.6 0)\lvec(0.6 8)
		\move(1.4 0)\lvec(1.4 8)
		\move(0.6 9)\lvec(0.6 15)
		\move(1.4 9)\lvec(1.4 15)
		\move(0.6 16)\lvec(0.6 22)
		\move(1.4 16)\lvec(1.4 22)
		\htext(1.05 1){$\cdots$}
		\htext(1.05 3){$\cdots$}
		\htext(1.05 5){$\cdots$}
		\htext(1.05 7){$\cdots$}
		\vtext(1.05 8.5){$\cdots$}
		\htext(1.05 10){$\cdots$}
		\htext(1.05 12){$\cdots$}
		\htext(1.05 14){$\cdots$}
		\vtext(1.05 15.5){$\cdots$}
		\htext(1.05 17){$\cdots$}
		\htext(1.05 19){$\cdots$}
		\htext(1.05 21){$\cdots$}
		\move(0 0)\lvec(0.6 2)
		\move(1.4 0)\lvec(2 2)
		\htext(0.2 1.6){$0$}
		\htext(1.6 1.6){$0$}
		\htext(0.4 0.4){$0$}
		\htext(1.8 0.4){$0$}
		\htext(0.3 3){$1$}
		\htext(1.7 3){$1$}
		\htext(0.3 5){$2$}
		\htext(1.7 5){$2$}
		\htext(0.3 7){$3$}
		\htext(1.7 7){$3$}
		\htext(0.3 10){$n$\!-\!$1$}
		\htext(1.7 10){$n$\!-\!$1$}
		\htext(0.3 12){$n$}
		\htext(1.7 12){$n$}
		\htext(0.3 14){$n$\!-\!$1$}
		\htext(1.7 14){$n$\!-\!$1$}
		\htext(0.3 17){$3$}
		\htext(1.7 17){$3$}
		\htext(0.3 19){$2$}
		\htext(1.7 19){$2$}
		\htext(0.3 21){$1$}
		\htext(1.7 21){$1$}
		\htext(1 -0.3){$\underbrace{\rule{3.5em}{0em}}$}
		\htext(1 -0.7){$\bar{x}_{1}$}
		\esegment
	\end{texdraw}
\end{center}

\vskip 2mm

\item\label{(2)} $D_{n+1}^{(2)}$ case :
Level-$l$ reduced Young column $C_{(x_1,\cdots,x_n,x_0,\bar{x}_n,\cdots,\bar{x}_1)}$,
where $x_0+\sum_{i=1}^{n}(x_i+\bar{x}_i)\le l$, $x_0=0$ or $1$, $x_i, \bar{x}_i\ge 0$ and $l'=l-\sum_{i=1}^{n}(x_i+\bar{x}_i)-x_0$.

\vskip 2mm

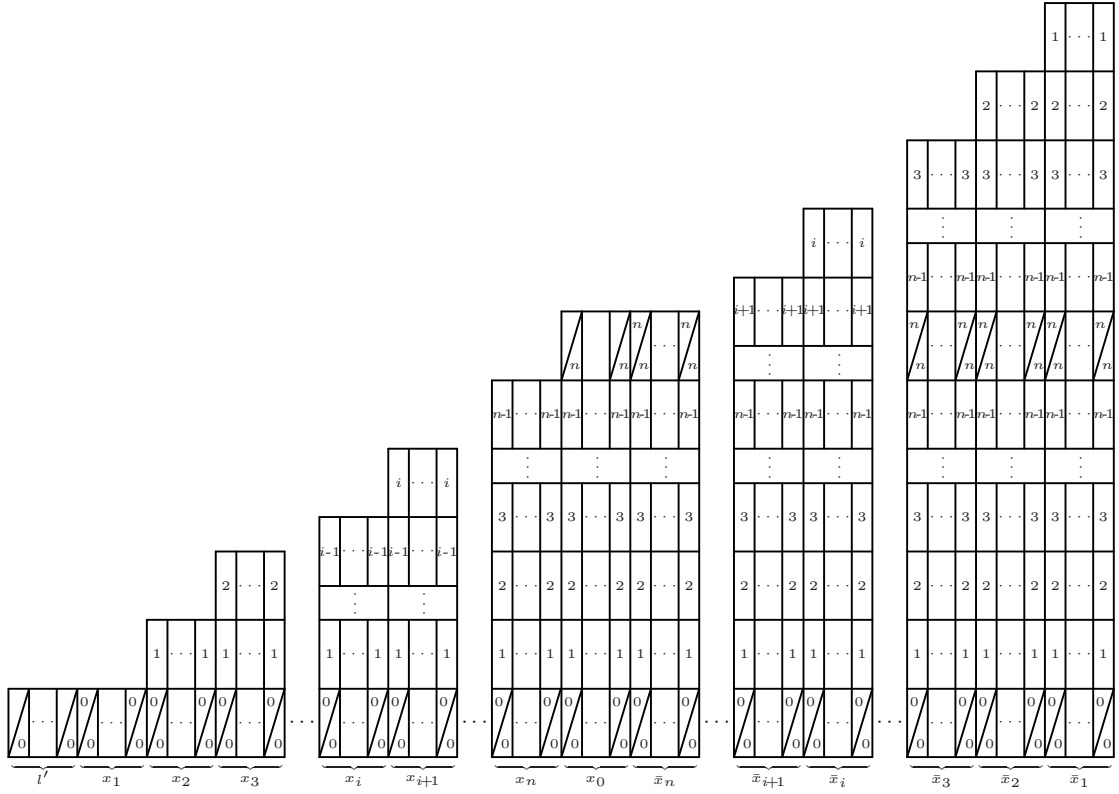
\begin{center}
	\begin{texdraw}
		\fontsize{3}{3}\selectfont
		\drawdim em
		\setunitscale 1.9
		
		\move(0 0)
		\bsegment
		\move(0 0)\lvec(2 0)
		\move(0 0)\lvec(0 2)
		\move(0 2)\lvec(2 2)
		\move(0.6 0)\lvec(0.6 2)
		\move(1.4 0)\lvec(1.4 2)
		\move(0 0)\lvec(0.6 2)
		\move(1.4 0)\lvec(2 2)
		\htext(0.4 0.4){$0$}
		\htext(1.8 0.4){$0$}
		\htext(1.05 1){$\cdots$}
		\htext(1 -0.3){$\underbrace{\rule{3.5em}{0em}}$}
		\htext(1 -0.6){$l'$}
		\esegment
		\move(2 0)
		\bsegment
		\move(0 0)\lvec(2 0)
		\move(0 0)\lvec(0 2)
		\move(0 2)\lvec(2 2)
		\move(0.6 0)\lvec(0.6 2)
		\move(1.4 0)\lvec(1.4 2)
		\htext(1.05 1){$\cdots$}
		\htext(0.2 1.6){$0$}
		\htext(1.6 1.6){$0$}
		\htext(0.4 0.4){$0$}
		\htext(1.8 0.4){$0$}
		\move(0 0)\lvec(0.6 2)
		\move(1.4 0)\lvec(2 2)
		\htext(1 -0.3){$\underbrace{\rule{3.5em}{0em}}$}
		\htext(1 -0.7){$x_1$}
		\esegment
		
		\move(4 0)
		\bsegment
		\move(0 0)\lvec(2 0)
		\move(0 0)\lvec(0 2)
		\move(0 2)\lvec(2 2)
		\move(0.6 0)\lvec(0.6 4)
		\move(1.4 0)\lvec(1.4 4)
		\htext(1.05 1){$\cdots$}
		\htext(1.05 3){$\cdots$}
		\move(0 2)\lvec(0 4)
		\move(2 2)\lvec(2 4)
		\move(0 4)\lvec(2 4)
		\htext(0.2 1.6){$0$}
		\htext(1.6 1.6){$0$}
		\htext(0.4 0.4){$0$}
		\htext(1.8 0.4){$0$}
		\move(0 0)\lvec(0.6 2)
		\move(1.4 0)\lvec(2 2)
		\htext(0.3 3){$1$}
		\htext(1.7 3){$1$}
		\htext(1 -0.3){$\underbrace{\rule{3.5em}{0em}}$}
		\htext(1 -0.7){$x_2$}
		\esegment
		
		\move(6 0)
		\bsegment
		\move(0 0)\lvec(2 0)
		\move(0 0)\lvec(0 2)
		\move(0 2)\lvec(2 2)
		\move(2 0)\lvec(2 6)
		\move(0.6 0)\lvec(0.6 6)
		\move(1.4 0)\lvec(1.4 6)
		\htext(1.05 1){$\cdots$}
		\htext(1.05 3){$\cdots$}
		\htext(1.05 5){$\cdots$}
		\move(0 4)\lvec(0 6)
		\move(0 6)\lvec(2 6)
		\move(0 4)\lvec(2 4)
		\htext(2.55 1){\fontsize{8}{8}\selectfont$\cdots$}
		\htext(0.2 1.6){$0$}
		\htext(1.6 1.6){$0$}
		\htext(0.4 0.4){$0$}
		\htext(1.8 0.4){$0$}
		\move(0 0)\lvec(0.6 2)
		\move(1.4 0)\lvec(2 2)	
		\htext(0.3 3){$1$}
		\htext(1.7 3){$1$}
		\htext(0.3 5){$2$}
		\htext(1.7 5){$2$}
		\htext(1 -0.3){$\underbrace{\rule{3.5em}{0em}}$}
		\htext(1 -0.7){$x_3$}
		\esegment

		\move(9 0)
		\bsegment
		\move(0 0)\lvec(2 0)
		\move(0 0)\lvec(0 2)
		\move(0 2)\lvec(2 2)
		\move(0 2)\lvec(0 7)
		\move(0 7)\lvec(2 7)
		\move(0 4)\lvec(2 4)
		\move(0 5)\lvec(2 5)
		\move(0.6 0)\lvec(0.6 4)
		\move(1.4 0)\lvec(1.4 4)
		\move(0.6 5)\lvec(0.6 7)
		\move(1.4 5)\lvec(1.4 7)
		\htext(1.05 1){$\cdots$}
		\htext(1.05 3){$\cdots$}
		\vtext(1.05 4.5){$\cdots$}
		\htext(1.05 6){$\cdots$}
		\htext(0.2 1.6){$0$}
		\htext(1.6 1.6){$0$}
		\htext(0.4 0.4){$0$}
		\htext(1.8 0.4){$0$}
		\move(0 0)\lvec(0.6 2)
		\move(1.4 0)\lvec(2 2)	
		\htext(0.3 3){$1$}
		\htext(1.7 3){$1$}
		\htext(0.3 6){$i$-$1$}
		\htext(1.7 6){$i$-$1$}
		\htext(1 -0.3){$\underbrace{\rule{3.5em}{0em}}$}
		\htext(1 -0.7){$x_i$}
		\esegment
		
		\move(11 0)
		\bsegment
		\move(0 0)\lvec(2 0)
		\move(0 0)\lvec(0 2)
		\move(0 2)\lvec(2 2)
		\move(0 2)\lvec(0 9)
		\move(2 0)\lvec(2 9)
		\move(0 9)\lvec(2 9)
		\move(0 5)\lvec(2 5)
		\move(0 4)\lvec(2 4)
		\move(0 7)\lvec(2 7)
		\move(0.6 0)\lvec(0.6 4)
		\move(1.4 0)\lvec(1.4 4)
		\move(0.6 5)\lvec(0.6 9)
		\move(1.4 5)\lvec(1.4 9)
		\htext(1.05 1){$\cdots$}
		\htext(1.05 3){$\cdots$}
		\vtext(1.05 4.5){$\cdots$}
		\htext(1.05 6){$\cdots$}
		\htext(1.05 8){$\cdots$}
		\htext(0.2 1.6){$0$}
		\htext(1.6 1.6){$0$}
		\htext(0.4 0.4){$0$}
		\htext(1.8 0.4){$0$}
		\move(0 0)\lvec(0.6 2)
		\move(1.4 0)\lvec(2 2)
		\htext(0.3 3){$1$}
		\htext(1.7 3){$1$}
		\htext(0.3 6){$i$-$1$}
		\htext(1.7 6){$i$-$1$}
		\htext(0.3 8){$i$}
		\htext(1.7 8){$i$}
		\htext(2.55 1){\fontsize{8}{8}\selectfont$\cdots$}
		\htext(1 -0.3){$\underbrace{\rule{3.5em}{0em}}$}
		\htext(1 -0.7){$x_{i \!+\! 1}$}
		\esegment
		
		\move(14 0)
		\bsegment
		\move(0 0)\lvec(2 0)
		\move(0 0)\lvec(0 2)
		\move(0 2)\lvec(2 2)
		\move(0 2)\lvec(0 11)
		\move(0 11)\lvec(2 11)

		\htext(0.2 1.6){$0$}
		\htext(1.6 1.6){$0$}
		\htext(0.4 0.4){$0$}
		\htext(1.8 0.4){$0$}
		\move(0 0)\lvec(0.6 2)
		\move(1.4 0)\lvec(2 2)		
		\move(0 4)\lvec(2 4)
		\move(0 6)\lvec(2 6)
		\move(0 8)\lvec(2 8)
		\move(0 9)\lvec(2 9)
		\move(0.6 0)\lvec(0.6 8)
		\move(1.4 0)\lvec(1.4 8)
		\move(0.6 9)\lvec(0.6 11)
		\move(1.4 9)\lvec(1.4 11)
		
		\htext(1.05 1){$\cdots$}
		\htext(1.05 3){$\cdots$}
		\htext(1.05 5){$\cdots$}
		\htext(1.05 7){$\cdots$}
		\vtext(1.05 8.5){$\cdots$}
		\htext(1.05 10){$\cdots$}

		\htext(0.3 3){$1$}
		\htext(1.7 3){$1$}
		\htext(0.3 5){$2$}
		\htext(1.7 5){$2$}
		\htext(0.3 7){$3$}
		\htext(1.7 7){$3$}
		\htext(0.3 10){$n$\!-\!$1$}
		\htext(1.7 10){$n$\!-\!$1$}
		\htext(1 -0.3){$\underbrace{\rule{3.5em}{0em}}$}
		\htext(1 -0.7){$x_{n}$}
		
		\esegment
		
		\move(16 0)
		\bsegment
		\move(0 0)\lvec(2 0)
		\move(0 0)\lvec(0 2)
		\move(0 2)\lvec(2 2)
		\move(0 2)\lvec(0 11)
		\move(0 11)\lvec(2 11)

		\htext(0.2 1.6){$0$}
		\htext(1.6 1.6){$0$}
		\htext(0.4 0.4){$0$}
		\htext(1.8 0.4){$0$}
		\move(0 0)\lvec(0.6 2)
		\move(1.4 0)\lvec(2 2)		
		\move(0 4)\lvec(2 4)
		\move(0 6)\lvec(2 6)
		\move(0 8)\lvec(2 8)
		\move(0 9)\lvec(2 9)
		\move(0.6 0)\lvec(0.6 8)
		\move(1.4 0)\lvec(1.4 8)
		\move(0.6 9)\lvec(0.6 11)
		\move(1.4 9)\lvec(1.4 11)
		
		\htext(1.05 1){$\cdots$}
		\htext(1.05 3){$\cdots$}
		\htext(1.05 5){$\cdots$}
		\htext(1.05 7){$\cdots$}
		\vtext(1.05 8.5){$\cdots$}
		\htext(1.05 10){$\cdots$}

		\htext(0.3 3){$1$}
		\htext(1.7 3){$1$}
		\htext(0.3 5){$2$}
		\htext(1.7 5){$2$}
		\htext(0.3 7){$3$}
		\htext(1.7 7){$3$}
		\htext(0.3 10){$n$\!-\!$1$}
		\htext(1.7 10){$n$\!-\!$1$}
		\htext(1 -0.3){$\underbrace{\rule{3.5em}{0em}}$}
		\htext(1 -0.7){$x_{0}$}
		\move(0 11)\lvec(0 13)
		\move(0.6 11)\lvec(0.6 13)
		\move(1.4 11)\lvec(1.4 13)
		\move(0 13)\lvec(2 13)
		\move(0 11)\lvec(0.6 13)
		\move(1.4 11)\lvec(2 13)
		\htext(0.4 11.4){$n$}
		\htext(1.8 11.4){$n$}
		\esegment	
		
		\move(18 0)
		\bsegment
		\move(0 0)\lvec(2 0)
		\move(0 0)\lvec(0 2)
		\move(0 2)\lvec(2 2)
		\move(0 2)\lvec(0 13)
		\move(2 0)\lvec(2 13)
		\move(0 13)\lvec(2 13)
		\htext(0.2 1.6){$0$}
		\htext(1.6 1.6){$0$}
		\htext(0.4 0.4){$0$}
		\htext(1.8 0.4){$0$}
		\move(0 0)\lvec(0.6 2)
		\move(1.4 0)\lvec(2 2)		
		\move(0 4)\lvec(2 4)
		\move(0 6)\lvec(2 6)
		\move(0 8)\lvec(2 8)
		\move(0 9)\lvec(2 9)
		\move(0 11)\lvec(2 11)
		
		\move(0.6 0)\lvec(0.6 8)
		\move(1.4 0)\lvec(1.4 8)
		
		\move(0.6 9)\lvec(0.6 13)
		\move(1.4 9)\lvec(1.4 13)
		
		\htext(1.05 1){$\cdots$}
		\htext(1.05 3){$\cdots$}
		\htext(1.05 5){$\cdots$}
		\htext(1.05 7){$\cdots$}
		\vtext(1.05 8.5){$\cdots$}
		\htext(1.05 10){$\cdots$}
		\htext(1.05 12){$\cdots$}

		\htext(0.3 3){$1$}
		\htext(1.7 3){$1$}
		\htext(0.3 5){$2$}
		\htext(1.7 5){$2$}
		\htext(0.3 7){$3$}
		\htext(1.7 7){$3$}
		\htext(0.3 10){$n$\!-\!$1$}
		\htext(1.7 10){$n$\!-\!$1$}
		
		\htext(1 -0.3){$\underbrace{\rule{3.5em}{0em}}$}
		\htext(1 -0.7){$\bar{x}_{n}$}
		\htext(2.55 1){\fontsize{8}{8}\selectfont$\cdots$}
		\move(0 11)\lvec(0.6 13)
		\htext(0.4 11.4){$n$}
		\move(1.4 11)\lvec(2 13)
		\htext(1.8 11.4){$n$}
		\htext(0.2 12.6){$n$}
		\htext(1.6 12.6){$n$}
		\esegment
		
		\move(21 0)
		\bsegment
		\move(0 0)\lvec(2 0)
		\move(0 0)\lvec(0 2)
		\move(0 2)\lvec(2 2)
		\move(0 2)\lvec(0 14)
		\move(0 14)\lvec(2 14)
		\move(0 12)\lvec(2 12)
		\move(0 11)\lvec(2 11)
		\move(0 9)\lvec(2 9)
		\move(0 8)\lvec(2 8)
		\move(0 6)\lvec(2 6)
		\move(0 4)\lvec(2 4)
		
		\move(0.6 9)\lvec(0.6 11)
		\move(1.4 9)\lvec(1.4 11)
		
		\move(0.6 12)\lvec(0.6 14)
		\move(1.4 12)\lvec(1.4 14)
		
		\move(0.6 0)\lvec(0.6 8)
		\move(1.4 0)\lvec(1.4 8)
		\htext(1.05 1){$\cdots$}
		\htext(1.05 3){$\cdots$}
		\htext(1.05 5){$\cdots$}
		\htext(1.05 7){$\cdots$}
		\vtext(1.05 8.5){$\cdots$}
		\htext(1.05 10){$\cdots$}
		\vtext(1.05 11.5){$\cdots$}
		\htext(1.05 13){$\cdots$}
		\htext(0.2 1.6){$0$}
		\htext(1.6 1.6){$0$}
		\htext(0.4 0.4){$0$}
		\htext(1.8 0.4){$0$}
		\move(0 0)\lvec(0.6 2)
		\move(1.4 0)\lvec(2 2)	
		\htext(0.3 3){$1$}
		\htext(1.7 3){$1$}
		\htext(0.3 5){$2$}
		\htext(1.7 5){$2$}
		\htext(0.3 7){$3$}
		\htext(1.7 7){$3$}
		\htext(0.3 10){$n$\!-\!$1$}
		\htext(1.7 10){$n$\!-\!$1$}
		\htext(0.3 13){$i$\!+\!$1$}
		\htext(1.7 13){$i$\!+\!$1$}
		\htext(1 -0.3){$\underbrace{\rule{3.5em}{0em}}$}
		\htext(1 -0.7){$\bar{x}_{i\!+\! 1}$}
		\esegment
		
		\move(23 0)
		\bsegment
		\move(0 0)\lvec(2 0)
		\move(0 0)\lvec(0 2)
		\move(0 2)\lvec(2 2)
		\move(0 2)\lvec(0 16)
		\move(2 0)\lvec(2 16)
		\move(0 16)\lvec(2 16)
		
		\move(0 14)\lvec(2 14)
		\move(0 12)\lvec(2 12)
		\move(0 11)\lvec(2 11)
		\move(0 9)\lvec(2 9)
		\move(0 8)\lvec(2 8)
		\move(0 6)\lvec(2 6)
		\move(0 4)\lvec(2 4)
		
		\move(0.6 9)\lvec(0.6 11)
		\move(1.4 9)\lvec(1.4 11)
		
		\move(0.6 12)\lvec(0.6 16)
		\move(1.4 12)\lvec(1.4 16)
		\htext(0.2 1.6){$0$}
		\htext(1.6 1.6){$0$}
		\htext(0.4 0.4){$0$}
		\htext(1.8 0.4){$0$}
		\move(0 0)\lvec(0.6 2)
		\move(1.4 0)\lvec(2 2)		
		\move(0.6 0)\lvec(0.6 8)
		\move(1.4 0)\lvec(1.4 8)
		\htext(1.05 1){$\cdots$}
		\htext(1.05 3){$\cdots$}
		\htext(1.05 5){$\cdots$}
		\htext(1.05 7){$\cdots$}
		\vtext(1.05 8.5){$\cdots$}
		\htext(1.05 10){$\cdots$}
		\vtext(1.05 11.5){$\cdots$}
		\htext(1.05 13){$\cdots$}
		\htext(1.05 15){$\cdots$}
		
		\htext(0.3 3){$1$}
		\htext(1.7 3){$1$}
		\htext(0.3 5){$2$}
		\htext(1.7 5){$2$}
		\htext(0.3 7){$3$}
		\htext(1.7 7){$3$}
		\htext(0.3 10){$n$\!-\!$1$}
		\htext(1.7 10){$n$\!-\!$1$}
		\htext(0.3 13){$i$\!+\!$1$}
		\htext(1.7 13){$i$\!+\!$1$}
		\htext(0.3 15){$i$}
		\htext(1.7 15){$i$}
		\htext(2.55 1){\fontsize{8}{8}\selectfont$\cdots$}
		\htext(1 -0.3){$\underbrace{\rule{3.5em}{0em}}$}
		\htext(1 -0.7){$\bar{x}_{i}$}
		\esegment
		
		\move(26 0)
		\bsegment
		\move(0 0)\lvec(2 0)
		\move(0 0)\lvec(0 2)
		\move(0 2)\lvec(2 2)
		\move(0 2)\lvec(0 18)
		\move(0 4)\lvec(2 4)
		\move(0 6)\lvec(2 6)
		\move(0 8)\lvec(2 8)
		\move(0 9)\lvec(2 9)
		\move(0 11)\lvec(2 11)
		\move(0 13)\lvec(2 13)
		\move(0 15)\lvec(2 15)
		\move(0 16)\lvec(2 16)
		\move(0 18)\lvec(2 18)
		\move(0.6 0)\lvec(0.6 8)
		\move(1.4 0)\lvec(1.4 8)
		\move(0.6 9)\lvec(0.6 15)
		\move(1.4 9)\lvec(1.4 15)
		\move(0.6 16)\lvec(0.6 18)
		\move(1.4 16)\lvec(1.4 18)
		\htext(1.05 1){$\cdots$}
		\htext(1.05 3){$\cdots$}
		\htext(1.05 5){$\cdots$}
		\htext(1.05 7){$\cdots$}
		\vtext(1.05 8.5){$\cdots$}
		\htext(1.05 10){$\cdots$}
		\htext(1.05 12){$\cdots$}
		\htext(1.05 14){$\cdots$}
		\vtext(1.05 15.5){$\cdots$}
		\htext(1.05 17){$\cdots$}
		\htext(0.2 1.6){$0$}
		\htext(1.6 1.6){$0$}
		\htext(0.4 0.4){$0$}
		\htext(1.8 0.4){$0$}
		\move(0 0)\lvec(0.6 2)
		\move(1.4 0)\lvec(2 2)
		\htext(0.3 3){$1$}
		\htext(1.7 3){$1$}
		\htext(0.3 5){$2$}
		\htext(1.7 5){$2$}
		\htext(0.3 7){$3$}
		\htext(1.7 7){$3$}
		\htext(0.3 10){$n$\!-\!$1$}
		\htext(1.7 10){$n$\!-\!$1$}
		
		\htext(0.3 14){$n$\!-\!$1$}
		\htext(1.7 14){$n$\!-\!$1$}
		\htext(0.3 17){$3$}
		\htext(1.7 17){$3$}
		\htext(1 -0.3){$\underbrace{\rule{3.5em}{0em}}$}
		\htext(1 -0.7){$\bar{x}_{3}$}
		
		\move(0 11)\lvec(0.6 13)
		\htext(0.4 11.4){$n$}
		\move(1.4 11)\lvec(2 13)
		\htext(1.8 11.4){$n$}
		\htext(0.2 12.6){$n$}
		\htext(1.6 12.6){$n$}
		\esegment
		
		\move(28 0)
		\bsegment
		\move(0 0)\lvec(2 0)
		\move(0 0)\lvec(0 2)
		\move(0 2)\lvec(2 2)
		\move(0 2)\lvec(0 20)
		\move(0 4)\lvec(2 4)
		\move(0 6)\lvec(2 6)
		\move(0 8)\lvec(2 8)
		\move(0 9)\lvec(2 9)
		\move(0 11)\lvec(2 11)
		\move(0 13)\lvec(2 13)
		\move(0 15)\lvec(2 15)
		\move(0 16)\lvec(2 16)
		\move(0 18)\lvec(2 18)
		\move(0 20)\lvec(2 20)
		\htext(0.2 1.6){$0$}
		\htext(1.6 1.6){$0$}
		\htext(0.4 0.4){$0$}
		\htext(1.8 0.4){$0$}
		\move(0 0)\lvec(0.6 2)
		\move(1.4 0)\lvec(2 2)		
		\move(0.6 0)\lvec(0.6 8)
		\move(1.4 0)\lvec(1.4 8)
		\move(0.6 9)\lvec(0.6 15)
		\move(1.4 9)\lvec(1.4 15)
		\move(0.6 16)\lvec(0.6 20)
		\move(1.4 16)\lvec(1.4 20)
		\htext(1.05 1){$\cdots$}
		\htext(1.05 3){$\cdots$}
		\htext(1.05 5){$\cdots$}
		\htext(1.05 7){$\cdots$}
		\vtext(1.05 8.5){$\cdots$}
		\htext(1.05 10){$\cdots$}
		\htext(1.05 12){$\cdots$}
		\htext(1.05 14){$\cdots$}
		\vtext(1.05 15.5){$\cdots$}
		\htext(1.05 17){$\cdots$}
		\htext(1.05 19){$\cdots$}
		
		\htext(0.3 3){$1$}
		\htext(1.7 3){$1$}
		\htext(0.3 5){$2$}
		\htext(1.7 5){$2$}
		\htext(0.3 7){$3$}
		\htext(1.7 7){$3$}
		\htext(0.3 10){$n$\!-\!$1$}
		\htext(1.7 10){$n$\!-\!$1$}
		\move(0 11)\lvec(0.6 13)
		\htext(0.4 11.4){$n$}
		\move(1.4 11)\lvec(2 13)
		\htext(1.8 11.4){$n$}
		\htext(0.2 12.6){$n$}
		\htext(1.6 12.6){$n$}
		\htext(0.3 14){$n$\!-\!$1$}
		\htext(1.7 14){$n$\!-\!$1$}
		\htext(0.3 17){$3$}
		\htext(1.7 17){$3$}
		\htext(0.3 19){$2$}
		\htext(1.7 19){$2$}
		\htext(1 -0.3){$\underbrace{\rule{3.5em}{0em}}$}
		\htext(1 -0.7){$\bar{x}_{2}$}
		\esegment
		
		\move(30 0)
		\bsegment
		\move(0 0)\lvec(2 0)
		\move(0 0)\lvec(0 2)
		\move(0 2)\lvec(2 2)
		\move(2 0)\lvec(2 22)
		\move(0 2)\lvec(0 22)
		\move(0 22)\lvec(2 22)
		
		\move(0 4)\lvec(2 4)
		\move(0 6)\lvec(2 6)
		\move(0 8)\lvec(2 8)
		\move(0 9)\lvec(2 9)
		\move(0 11)\lvec(2 11)
		\move(0 13)\lvec(2 13)
		\move(0 15)\lvec(2 15)
		\move(0 16)\lvec(2 16)
		\move(0 18)\lvec(2 18)
		\move(0 20)\lvec(2 20)
		
		\htext(0.2 1.6){$0$}
		\htext(1.6 1.6){$0$}
		\htext(0.4 0.4){$0$}
		\htext(1.8 0.4){$0$}
		\move(0 0)\lvec(0.6 2)
		\move(1.4 0)\lvec(2 2)		
		\move(0.6 0)\lvec(0.6 8)
		\move(1.4 0)\lvec(1.4 8)
		\move(0.6 9)\lvec(0.6 15)
		\move(1.4 9)\lvec(1.4 15)
		\move(0.6 16)\lvec(0.6 22)
		\move(1.4 16)\lvec(1.4 22)
		\htext(1.05 1){$\cdots$}
		\htext(1.05 3){$\cdots$}
		\htext(1.05 5){$\cdots$}
		\htext(1.05 7){$\cdots$}
		\vtext(1.05 8.5){$\cdots$}
		\htext(1.05 10){$\cdots$}
		\htext(1.05 12){$\cdots$}
		\htext(1.05 14){$\cdots$}
		\vtext(1.05 15.5){$\cdots$}
		\htext(1.05 17){$\cdots$}
		\htext(1.05 19){$\cdots$}
		\htext(1.05 21){$\cdots$}
		
		\htext(0.3 3){$1$}
		\htext(1.7 3){$1$}
		\htext(0.3 5){$2$}
		\htext(1.7 5){$2$}
		\htext(0.3 7){$3$}
		\htext(1.7 7){$3$}
		\htext(0.3 10){$n$\!-\!$1$}
		\htext(1.7 10){$n$\!-\!$1$}
		\move(0 11)\lvec(0.6 13)
		\htext(0.4 11.4){$n$}
		\move(1.4 11)\lvec(2 13)
		\htext(1.8 11.4){$n$}
		\htext(0.2 12.6){$n$}
		\htext(1.6 12.6){$n$}
		\htext(0.3 14){$n$\!-\!$1$}
		\htext(1.7 14){$n$\!-\!$1$}
		\htext(0.3 17){$3$}
		\htext(1.7 17){$3$}
		\htext(0.3 19){$2$}
		\htext(1.7 19){$2$}
		\htext(0.3 21){$1$}
		\htext(1.7 21){$1$}
		\htext(1 -0.3){$\underbrace{\rule{3.5em}{0em}}$}
		\htext(1 -0.7){$\bar{x}_{1}$}
		\esegment
	\end{texdraw}
\end{center}

\vskip 2mm

\item\label{(3)} $A_{2n-1}^{(2)}$ case :
Level-$l$ reduced Young column $C_{(x_1,\cdots,x_n,\bar{x}_n,\cdots,\bar{x}_1)}$,
where $\sum_{i=1}^{n}(x_i+\bar{x}_i)=l$ and $x_i, \bar{x}_i\ge 0$.

\vskip 2mm

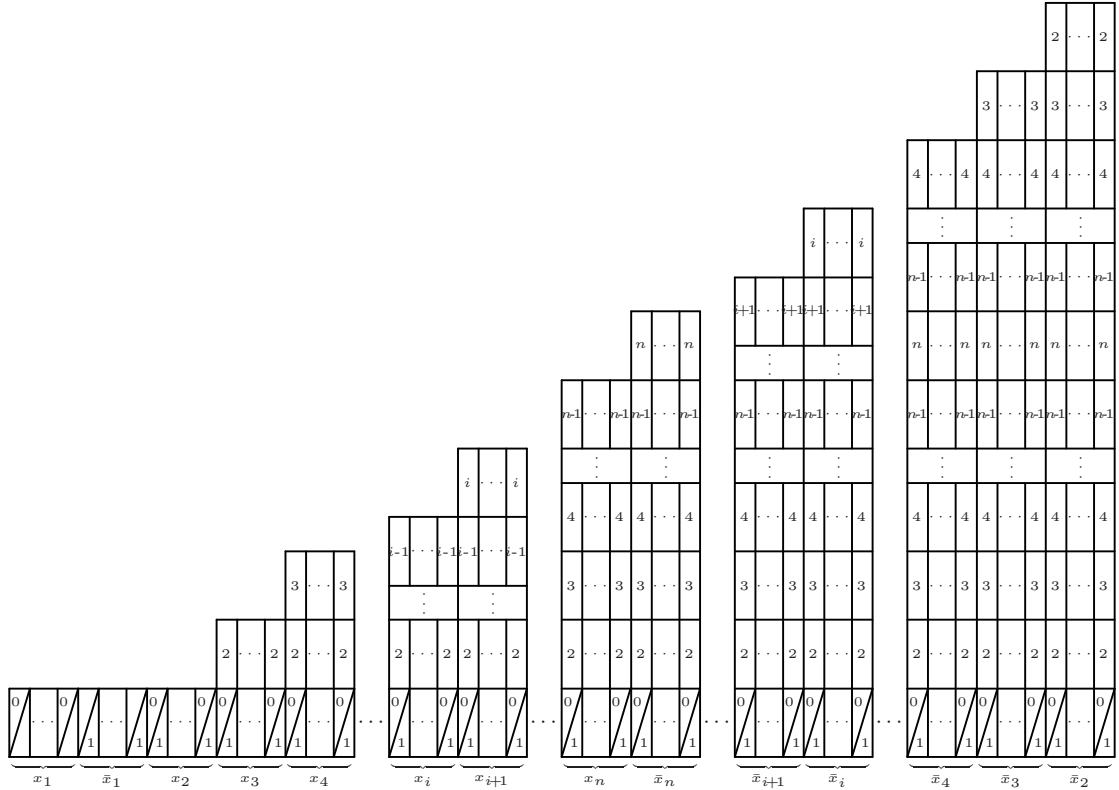
\begin{center}
	\begin{texdraw}
		\fontsize{3}{3}\selectfont
		\drawdim em
		\setunitscale 1.9
		
		\move(-2 0)
		\bsegment
		\move(0 0)\lvec(2 0)
		\move(0 0)\lvec(0 2)
		\move(0 2)\lvec(2 2)
		\move(0.6 0)\lvec(0.6 2)
		\move(1.4 0)\lvec(1.4 2)
		\move(0 0)\lvec(0.6 2)
		\move(1.4 0)\lvec(2 2)
		\htext(0.2 1.6){$0$}
		\htext(1.6 1.6){$0$}
		
		\htext(1.05 1){$\cdots$}
		\htext(1 -0.3){$\underbrace{\rule{3.5em}{0em}}$}
		\htext(1 -0.7){$x_1$}
		\esegment
		
		\move(0 0)
		\bsegment
		\move(0 0)\lvec(2 0)
		\move(0 0)\lvec(0 2)
		\move(0 2)\lvec(2 2)
		\move(0.6 0)\lvec(0.6 2)
		\move(1.4 0)\lvec(1.4 2)
		\move(0 0)\lvec(0.6 2)
		\move(1.4 0)\lvec(2 2)
		\htext(0.4 0.4){$1$}
		\htext(1.8 0.4){$1$}
		\htext(1.05 1){$\cdots$}
		\htext(1 -0.3){$\underbrace{\rule{3.5em}{0em}}$}
		\htext(1 -0.7){$\bar{x}_1$}
		\esegment
		\move(2 0)
		\bsegment
		\move(0 0)\lvec(2 0)
		\move(0 0)\lvec(0 2)
		\move(0 2)\lvec(2 2)
		\move(0.6 0)\lvec(0.6 2)
		\move(1.4 0)\lvec(1.4 2)
		\htext(1.05 1){$\cdots$}
		\htext(0.2 1.6){$0$}
		\htext(1.6 1.6){$0$}
		
		\htext(0.4 0.4){$1$}
		\htext(1.8 0.4){$1$}
		\move(0 0)\lvec(0.6 2)
		\move(1.4 0)\lvec(2 2)
		\htext(1 -0.3){$\underbrace{\rule{3.5em}{0em}}$}
		\htext(1 -0.7){$x_2$}
		\esegment
		
		\move(4 0)
		\bsegment
		\move(0 0)\lvec(2 0)
		\move(0 0)\lvec(0 2)
		\move(0 2)\lvec(2 2)
		\move(0.6 0)\lvec(0.6 4)
		\move(1.4 0)\lvec(1.4 4)
		\htext(1.05 1){$\cdots$}
		\htext(1.05 3){$\cdots$}
		\move(0 2)\lvec(0 4)
		\move(2 2)\lvec(2 4)
		\move(0 4)\lvec(2 4)
		\htext(0.2 1.6){$0$}
		\htext(1.6 1.6){$0$}
		\htext(0.4 0.4){$1$}
		\htext(1.8 0.4){$1$}
		\move(0 0)\lvec(0.6 2)
		\move(1.4 0)\lvec(2 2)
		\htext(0.3 3){$2$}
		\htext(1.7 3){$2$}
		\htext(1 -0.3){$\underbrace{\rule{3.5em}{0em}}$}
		\htext(1 -0.7){$x_3$}
		\esegment
		
		\move(6 0)
		\bsegment
		\move(0 0)\lvec(2 0)
		\move(0 0)\lvec(0 2)
		\move(0 2)\lvec(2 2)
		\move(2 0)\lvec(2 6)
		\move(0.6 0)\lvec(0.6 6)
		\move(1.4 0)\lvec(1.4 6)
		\htext(1.05 1){$\cdots$}
		\htext(1.05 3){$\cdots$}
		\htext(1.05 5){$\cdots$}
		\move(0 4)\lvec(0 6)
		\move(0 6)\lvec(2 6)
		\move(0 4)\lvec(2 4)
		\htext(2.55 1){\fontsize{8}{8}\selectfont$\cdots$}
		\htext(0.2 1.6){$0$}
		\htext(1.6 1.6){$0$}
		\htext(0.4 0.4){$1$}
		\htext(1.8 0.4){$1$}
		\move(0 0)\lvec(0.6 2)
		\move(1.4 0)\lvec(2 2)	
		\htext(0.3 3){$2$}
		\htext(1.7 3){$2$}
		\htext(0.3 5){$3$}
		\htext(1.7 5){$3$}
		\htext(1 -0.3){$\underbrace{\rule{3.5em}{0em}}$}
		\htext(1 -0.7){$x_4$}
		\esegment

		\move(9 0)
		\bsegment
		\move(0 0)\lvec(2 0)
		\move(0 0)\lvec(0 2)
		\move(0 2)\lvec(2 2)
		\move(0 2)\lvec(0 7)
		\move(0 7)\lvec(2 7)
		\move(0 4)\lvec(2 4)
		\move(0 5)\lvec(2 5)
		\move(0.6 0)\lvec(0.6 4)
		\move(1.4 0)\lvec(1.4 4)
		\move(0.6 5)\lvec(0.6 7)
		\move(1.4 5)\lvec(1.4 7)
		\htext(1.05 1){$\cdots$}
		\htext(1.05 3){$\cdots$}
		\vtext(1.05 4.5){$\cdots$}
		\htext(1.05 6){$\cdots$}
		\htext(0.2 1.6){$0$}
		\htext(1.6 1.6){$0$}
		\htext(0.4 0.4){$1$}
		\htext(1.8 0.4){$1$}
		\move(0 0)\lvec(0.6 2)
		\move(1.4 0)\lvec(2 2)	
		\htext(0.3 3){$2$}
		\htext(1.7 3){$2$}
		\htext(0.3 6){$i$-$1$}
		\htext(1.7 6){$i$-$1$}
		\htext(1 -0.3){$\underbrace{\rule{3.5em}{0em}}$}
		\htext(1 -0.7){$x_i$}
		\esegment
		
		\move(11 0)
		\bsegment
		\move(0 0)\lvec(2 0)
		\move(0 0)\lvec(0 2)
		\move(0 2)\lvec(2 2)
		\move(0 2)\lvec(0 9)
		\move(2 0)\lvec(2 9)
		\move(0 9)\lvec(2 9)
		\move(0 5)\lvec(2 5)
		\move(0 4)\lvec(2 4)
		\move(0 7)\lvec(2 7)
		\move(0.6 0)\lvec(0.6 4)
		\move(1.4 0)\lvec(1.4 4)
		\move(0.6 5)\lvec(0.6 9)
		\move(1.4 5)\lvec(1.4 9)
		\htext(1.05 1){$\cdots$}
		\htext(1.05 3){$\cdots$}
		\vtext(1.05 4.5){$\cdots$}
		\htext(1.05 6){$\cdots$}
		\htext(1.05 8){$\cdots$}
		\htext(0.2 1.6){$0$}
		\htext(1.6 1.6){$0$}
		\htext(0.4 0.4){$1$}
		\htext(1.8 0.4){$1$}
		\move(0 0)\lvec(0.6 2)
		\move(1.4 0)\lvec(2 2)
		\htext(0.3 3){$2$}
		\htext(1.7 3){$2$}
		\htext(0.3 6){$i$-$1$}
		\htext(1.7 6){$i$-$1$}
		\htext(0.3 8){$i$}
		\htext(1.7 8){$i$}
		\htext(2.55 1){\fontsize{8}{8}\selectfont$\cdots$}
		\htext(1 -0.3){$\underbrace{\rule{3.5em}{0em}}$}
		\htext(1 -0.7){$x_{i \!+\! 1}$}
		\esegment
		
		\move(14 0)
		\bsegment
		\move(0 0)\lvec(2 0)
		\move(0 0)\lvec(0 2)
		\move(0 2)\lvec(2 2)
		\move(0 2)\lvec(0 11)
		\move(0 11)\lvec(2 11)
		\htext(0.2 1.6){$0$}
		\htext(1.6 1.6){$0$}
		\htext(0.4 0.4){$1$}
		\htext(1.8 0.4){$1$}
		\move(0 0)\lvec(0.6 2)
		\move(1.4 0)\lvec(2 2)		
		\move(0 4)\lvec(2 4)
		\move(0 6)\lvec(2 6)
		\move(0 8)\lvec(2 8)
		\move(0 9)\lvec(2 9)
		\move(0.6 0)\lvec(0.6 8)
		\move(1.4 0)\lvec(1.4 8)
		\move(0.6 9)\lvec(0.6 11)
		\move(1.4 9)\lvec(1.4 11)
		
		\htext(1.05 1){$\cdots$}
		\htext(1.05 3){$\cdots$}
		\htext(1.05 5){$\cdots$}
		\htext(1.05 7){$\cdots$}
		\vtext(1.05 8.5){$\cdots$}
		\htext(1.05 10){$\cdots$}		
		\htext(0.3 3){$2$}
		\htext(1.7 3){$2$}
		\htext(0.3 5){$3$}
		\htext(1.7 5){$3$}
		\htext(0.3 7){$4$}
		\htext(1.7 7){$4$}
		\htext(0.3 10){$n$\!-\!$1$}
		\htext(1.7 10){$n$\!-\!$1$}
		\htext(1 -0.3){$\underbrace{\rule{3.5em}{0em}}$}
		\htext(1 -0.7){$x_{n}$}

		\esegment
		
		\move(16 0)
		\bsegment
		\move(0 0)\lvec(2 0)
		\move(0 0)\lvec(0 2)
		\move(0 2)\lvec(2 2)
		\move(0 2)\lvec(0 13)
		\move(2 0)\lvec(2 13)
		\move(0 13)\lvec(2 13)
		\htext(0.2 1.6){$0$}
		\htext(1.6 1.6){$0$}
		\htext(0.4 0.4){$1$}
		\htext(1.8 0.4){$1$}
		\move(0 0)\lvec(0.6 2)
		\move(1.4 0)\lvec(2 2)		
		\move(0 4)\lvec(2 4)
		\move(0 6)\lvec(2 6)
		\move(0 8)\lvec(2 8)
		\move(0 9)\lvec(2 9)
		\move(0 11)\lvec(2 11)
		
		\move(0.6 0)\lvec(0.6 8)
		\move(1.4 0)\lvec(1.4 8)
		
		\move(0.6 9)\lvec(0.6 13)
		\move(1.4 9)\lvec(1.4 13)
		
		\htext(1.05 1){$\cdots$}
		\htext(1.05 3){$\cdots$}
		\htext(1.05 5){$\cdots$}
		\htext(1.05 7){$\cdots$}
		\vtext(1.05 8.5){$\cdots$}
		\htext(1.05 10){$\cdots$}
		\htext(1.05 12){$\cdots$}

		\htext(0.3 3){$2$}
		\htext(1.7 3){$2$}
		\htext(0.3 5){$3$}
		\htext(1.7 5){$3$}
		\htext(0.3 7){$4$}
		\htext(1.7 7){$4$}
		\htext(0.3 10){$n$\!-\!$1$}
		\htext(1.7 10){$n$\!-\!$1$}
		\htext(0.3 12){$n$}
		\htext(1.7 12){$n$}
		\htext(1 -0.3){$\underbrace{\rule{3.5em}{0em}}$}
		\htext(1 -0.7){$\bar{x}_{n}$}
		\htext(2.55 1){\fontsize{8}{8}\selectfont$\cdots$}

		\esegment
		
		\move(19 0)
		\bsegment
		\move(0 0)\lvec(2 0)
		\move(0 0)\lvec(0 2)
		\move(0 2)\lvec(2 2)
		\move(0 2)\lvec(0 14)
		\move(0 14)\lvec(2 14)
		\move(0 12)\lvec(2 12)
		\move(0 11)\lvec(2 11)
		\move(0 9)\lvec(2 9)
		\move(0 8)\lvec(2 8)
		\move(0 6)\lvec(2 6)
		\move(0 4)\lvec(2 4)
		
		\move(0.6 9)\lvec(0.6 11)
		\move(1.4 9)\lvec(1.4 11)
		
		\move(0.6 12)\lvec(0.6 14)
		\move(1.4 12)\lvec(1.4 14)
		
		\move(0.6 0)\lvec(0.6 8)
		\move(1.4 0)\lvec(1.4 8)
		\htext(1.05 1){$\cdots$}
		\htext(1.05 3){$\cdots$}
		\htext(1.05 5){$\cdots$}
		\htext(1.05 7){$\cdots$}
		\vtext(1.05 8.5){$\cdots$}
		\htext(1.05 10){$\cdots$}
		\vtext(1.05 11.5){$\cdots$}
		\htext(1.05 13){$\cdots$}
		\htext(0.2 1.6){$0$}
		\htext(1.6 1.6){$0$}
		\htext(0.4 0.4){$1$}
		\htext(1.8 0.4){$1$}
		\move(0 0)\lvec(0.6 2)
		\move(1.4 0)\lvec(2 2)	
		\htext(0.3 3){$2$}
		\htext(1.7 3){$2$}
		\htext(0.3 5){$3$}
		\htext(1.7 5){$3$}
		\htext(0.3 7){$4$}
		\htext(1.7 7){$4$}
		\htext(0.3 10){$n$\!-\!$1$}
		\htext(1.7 10){$n$\!-\!$1$}
		\htext(0.3 13){$i$\!+\!$1$}
		\htext(1.7 13){$i$\!+\!$1$}
		\htext(1 -0.3){$\underbrace{\rule{3.5em}{0em}}$}
		\htext(1 -0.7){$\bar{x}_{i\!+\! 1}$}
		\esegment
		
		\move(21 0)
		\bsegment
		\move(0 0)\lvec(2 0)
		\move(0 0)\lvec(0 2)
		\move(0 2)\lvec(2 2)
		\move(0 2)\lvec(0 16)
		\move(2 0)\lvec(2 16)
		\move(0 16)\lvec(2 16)
		
		\move(0 14)\lvec(2 14)
		\move(0 12)\lvec(2 12)
		\move(0 11)\lvec(2 11)
		\move(0 9)\lvec(2 9)
		\move(0 8)\lvec(2 8)
		\move(0 6)\lvec(2 6)
		\move(0 4)\lvec(2 4)
		
		\move(0.6 9)\lvec(0.6 11)
		\move(1.4 9)\lvec(1.4 11)
		
		\move(0.6 12)\lvec(0.6 16)
		\move(1.4 12)\lvec(1.4 16)
		\htext(0.2 1.6){$0$}
		\htext(1.6 1.6){$0$}
		\htext(0.4 0.4){$1$}
		\htext(1.8 0.4){$1$}
		\move(0 0)\lvec(0.6 2)
		\move(1.4 0)\lvec(2 2)		
		\move(0.6 0)\lvec(0.6 8)
		\move(1.4 0)\lvec(1.4 8)
		\htext(1.05 1){$\cdots$}
		\htext(1.05 3){$\cdots$}
		\htext(1.05 5){$\cdots$}
		\htext(1.05 7){$\cdots$}
		\vtext(1.05 8.5){$\cdots$}
		\htext(1.05 10){$\cdots$}
		\vtext(1.05 11.5){$\cdots$}
		\htext(1.05 13){$\cdots$}
		\htext(1.05 15){$\cdots$}
		
		\htext(0.3 3){$2$}
		\htext(1.7 3){$2$}
		\htext(0.3 5){$3$}
		\htext(1.7 5){$3$}
		\htext(0.3 7){$4$}
		\htext(1.7 7){$4$}
		\htext(0.3 10){$n$\!-\!$1$}
		\htext(1.7 10){$n$\!-\!$1$}
		\htext(0.3 13){$i$\!+\!$1$}
		\htext(1.7 13){$i$\!+\!$1$}
		\htext(0.3 15){$i$}
		\htext(1.7 15){$i$}
		\htext(2.55 1){\fontsize{8}{8}\selectfont$\cdots$}
		\htext(1 -0.3){$\underbrace{\rule{3.5em}{0em}}$}
		\htext(1 -0.7){$\bar{x}_{i}$}
		\esegment
		
		\move(24 0)
		\bsegment
		\move(0 0)\lvec(2 0)
		\move(0 0)\lvec(0 2)
		\move(0 2)\lvec(2 2)
		\move(0 2)\lvec(0 18)
		\move(0 4)\lvec(2 4)
		\move(0 6)\lvec(2 6)
		\move(0 8)\lvec(2 8)
		\move(0 9)\lvec(2 9)
		\move(0 11)\lvec(2 11)
		\move(0 13)\lvec(2 13)
		\move(0 15)\lvec(2 15)
		\move(0 16)\lvec(2 16)
		\move(0 18)\lvec(2 18)
		\move(0.6 0)\lvec(0.6 8)
		\move(1.4 0)\lvec(1.4 8)
		\move(0.6 9)\lvec(0.6 15)
		\move(1.4 9)\lvec(1.4 15)
		\move(0.6 16)\lvec(0.6 18)
		\move(1.4 16)\lvec(1.4 18)
		\htext(1.05 1){$\cdots$}
		\htext(1.05 3){$\cdots$}
		\htext(1.05 5){$\cdots$}
		\htext(1.05 7){$\cdots$}
		\vtext(1.05 8.5){$\cdots$}
		\htext(1.05 10){$\cdots$}
		\htext(1.05 12){$\cdots$}
		\htext(1.05 14){$\cdots$}
		\vtext(1.05 15.5){$\cdots$}
		\htext(1.05 17){$\cdots$}
		\htext(0.2 1.6){$0$}
		\htext(1.6 1.6){$0$}
		\htext(0.4 0.4){$1$}
		\htext(1.8 0.4){$1$}
		\move(0 0)\lvec(0.6 2)
		\move(1.4 0)\lvec(2 2)
		\htext(0.3 3){$2$}
		\htext(1.7 3){$2$}
		\htext(0.3 5){$3$}
		\htext(1.7 5){$3$}
		\htext(0.3 7){$4$}
		\htext(1.7 7){$4$}
		\htext(0.3 10){$n$\!-\!$1$}
		\htext(1.7 10){$n$\!-\!$1$}
		\htext(0.3 12){$n$}
		\htext(1.7 12){$n$}
		\htext(0.3 14){$n$\!-\!$1$}
		\htext(1.7 14){$n$\!-\!$1$}
		\htext(0.3 17){$4$}
		\htext(1.7 17){$4$}
		\htext(1 -0.3){$\underbrace{\rule{3.5em}{0em}}$}
		\htext(1 -0.7){$\bar{x}_{4}$}
		\esegment
		
		\move(26 0)
		\bsegment
		\move(0 0)\lvec(2 0)
		\move(0 0)\lvec(0 2)
		\move(0 2)\lvec(2 2)
		\move(0 2)\lvec(0 20)
		\move(0 4)\lvec(2 4)
		\move(0 6)\lvec(2 6)
		\move(0 8)\lvec(2 8)
		\move(0 9)\lvec(2 9)
		\move(0 11)\lvec(2 11)
		\move(0 13)\lvec(2 13)
		\move(0 15)\lvec(2 15)
		\move(0 16)\lvec(2 16)
		\move(0 18)\lvec(2 18)
		\move(0 20)\lvec(2 20)
		\htext(0.2 1.6){$0$}
		\htext(1.6 1.6){$0$}
		\htext(0.4 0.4){$1$}
		\htext(1.8 0.4){$1$}
		\move(0 0)\lvec(0.6 2)
		\move(1.4 0)\lvec(2 2)		
		\move(0.6 0)\lvec(0.6 8)
		\move(1.4 0)\lvec(1.4 8)
		\move(0.6 9)\lvec(0.6 15)
		\move(1.4 9)\lvec(1.4 15)
		\move(0.6 16)\lvec(0.6 20)
		\move(1.4 16)\lvec(1.4 20)
		\htext(1.05 1){$\cdots$}
		\htext(1.05 3){$\cdots$}
		\htext(1.05 5){$\cdots$}
		\htext(1.05 7){$\cdots$}
		\vtext(1.05 8.5){$\cdots$}
		\htext(1.05 10){$\cdots$}
		\htext(1.05 12){$\cdots$}
		\htext(1.05 14){$\cdots$}
		\vtext(1.05 15.5){$\cdots$}
		\htext(1.05 17){$\cdots$}
		\htext(1.05 19){$\cdots$}
		
		\htext(0.3 3){$2$}
		\htext(1.7 3){$2$}
		\htext(0.3 5){$3$}
		\htext(1.7 5){$3$}
		\htext(0.3 7){$4$}
		\htext(1.7 7){$4$}
		\htext(0.3 10){$n$\!-\!$1$}
		\htext(1.7 10){$n$\!-\!$1$}
		\htext(0.3 12){$n$}
		\htext(1.7 12){$n$}
		
		\htext(0.3 14){$n$\!-\!$1$}
		\htext(1.7 14){$n$\!-\!$1$}
		\htext(0.3 17){$4$}
		\htext(1.7 17){$4$}
		\htext(0.3 19){$3$}
		\htext(1.7 19){$3$}
		\htext(1 -0.3){$\underbrace{\rule{3.5em}{0em}}$}
		\htext(1 -0.7){$\bar{x}_{3}$}
		\esegment
		
		\move(28 0)
		\bsegment
		\move(0 0)\lvec(2 0)
		\move(0 0)\lvec(0 2)
		\move(0 2)\lvec(2 2)
		\move(2 0)\lvec(2 22)
		\move(0 2)\lvec(0 22)
		\move(0 22)\lvec(2 22)
		
		\move(0 4)\lvec(2 4)
		\move(0 6)\lvec(2 6)
		\move(0 8)\lvec(2 8)
		\move(0 9)\lvec(2 9)
		\move(0 11)\lvec(2 11)
		\move(0 13)\lvec(2 13)
		\move(0 15)\lvec(2 15)
		\move(0 16)\lvec(2 16)
		\move(0 18)\lvec(2 18)
		\move(0 20)\lvec(2 20)
		
		\htext(0.2 1.6){$0$}
		\htext(1.6 1.6){$0$}
		\htext(0.4 0.4){$1$}
		\htext(1.8 0.4){$1$}
		\move(0 0)\lvec(0.6 2)
		\move(1.4 0)\lvec(2 2)		
		\move(0.6 0)\lvec(0.6 8)
		\move(1.4 0)\lvec(1.4 8)
		\move(0.6 9)\lvec(0.6 15)
		\move(1.4 9)\lvec(1.4 15)
		\move(0.6 16)\lvec(0.6 22)
		\move(1.4 16)\lvec(1.4 22)
		\htext(1.05 1){$\cdots$}
		\htext(1.05 3){$\cdots$}
		\htext(1.05 5){$\cdots$}
		\htext(1.05 7){$\cdots$}
		\vtext(1.05 8.5){$\cdots$}
		\htext(1.05 10){$\cdots$}
		\htext(1.05 12){$\cdots$}
		\htext(1.05 14){$\cdots$}
		\vtext(1.05 15.5){$\cdots$}
		\htext(1.05 17){$\cdots$}
		\htext(1.05 19){$\cdots$}
		\htext(1.05 21){$\cdots$}
		
		\htext(0.3 3){$2$}
		\htext(1.7 3){$2$}
		\htext(0.3 5){$3$}
		\htext(1.7 5){$3$}
		\htext(0.3 7){$4$}
		\htext(1.7 7){$4$}
		\htext(0.3 10){$n$\!-\!$1$}
		\htext(1.7 10){$n$\!-\!$1$}
		\htext(0.3 12){$n$}
		\htext(1.7 12){$n$}

		\htext(0.3 14){$n$\!-\!$1$}
		\htext(1.7 14){$n$\!-\!$1$}
		\htext(0.3 17){$4$}
		\htext(1.7 17){$4$}
		\htext(0.3 19){$3$}
		\htext(1.7 19){$3$}
		\htext(0.3 21){$2$}
		\htext(1.7 21){$2$}
		\htext(1 -0.3){$\underbrace{\rule{3.5em}{0em}}$}
		\htext(1 -0.7){$\bar{x}_{2}$}
		\esegment
	\end{texdraw}
\end{center}

\vskip 2mm

\item\label{(4)} $D_{n}^{(1)}$ case :
Level-$l$ reduced Young column $C_{(x_1,\cdots,x_n,\bar{x}_n,\cdots,\bar{x}_1)}$,
where $x_n=0$ or $\bar{x}_n=0$, $x_i, \bar{x}_i\ge 0$ and $\sum_{i=1}^{n}(x_i+\bar{x}_i)=l$.

\vskip 2mm

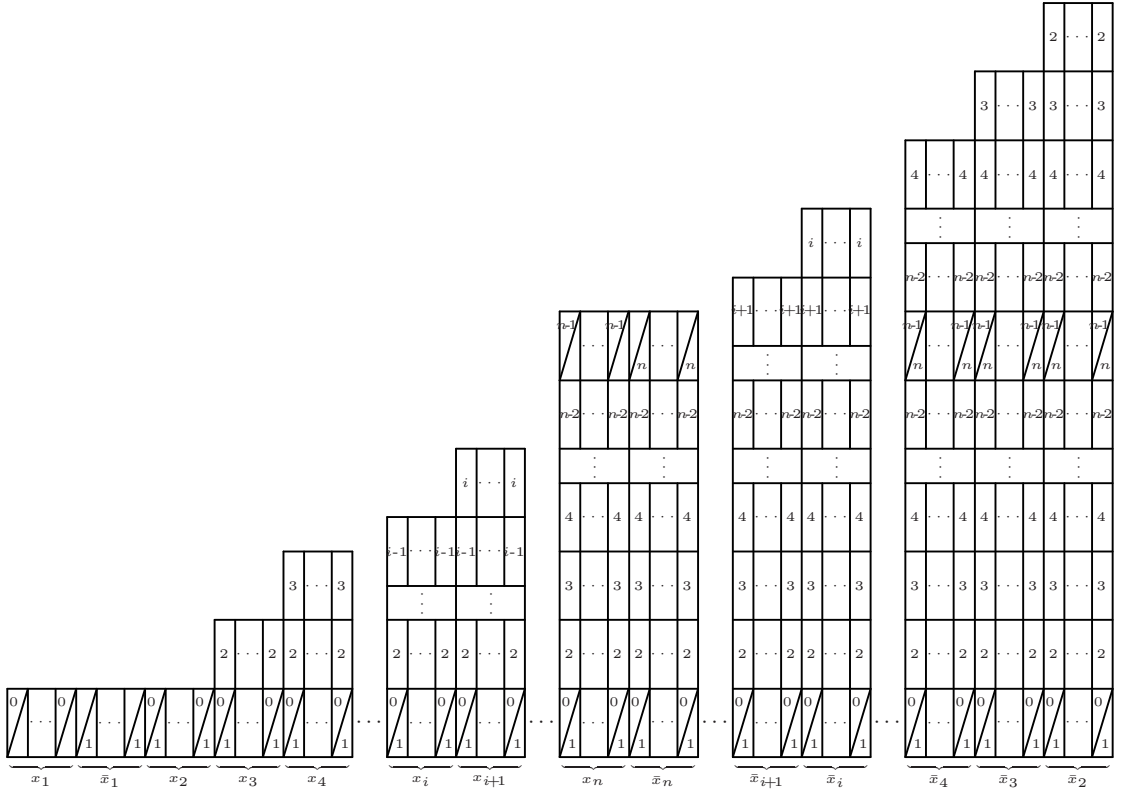
\begin{center}
	\begin{texdraw}
		\fontsize{3}{3}\selectfont
		\drawdim em
		\setunitscale 1.9
		\move(-2 0)
		\bsegment
		\move(0 0)\lvec(2 0)
		\move(0 0)\lvec(0 2)
		\move(0 2)\lvec(2 2)
		\move(0.6 0)\lvec(0.6 2)
		\move(1.4 0)\lvec(1.4 2)
		\move(0 0)\lvec(0.6 2)
		\move(1.4 0)\lvec(2 2)
		\htext(0.2 1.6){$0$}
		\htext(1.6 1.6){$0$}
		\htext(1.05 1){$\cdots$}
		\htext(1 -0.3){$\underbrace{\rule{3.5em}{0em}}$}
		\htext(1 -0.7){$x_1$}
		\esegment
		\move(0 0)
		\bsegment
		\move(0 0)\lvec(2 0)
		\move(0 0)\lvec(0 2)
		\move(0 2)\lvec(2 2)
		\move(0.6 0)\lvec(0.6 2)
		\move(1.4 0)\lvec(1.4 2)
		\move(0 0)\lvec(0.6 2)
		\move(1.4 0)\lvec(2 2)
		\htext(0.4 0.4){$1$}
		\htext(1.8 0.4){$1$}
		\htext(1.05 1){$\cdots$}
		\htext(1 -0.3){$\underbrace{\rule{3.5em}{0em}}$}
		\htext(1 -0.7){$\bar{x}_1$}
		\esegment
		\move(2 0)
		\bsegment
		\move(0 0)\lvec(2 0)
		\move(0 0)\lvec(0 2)
		\move(0 2)\lvec(2 2)
		\move(0.6 0)\lvec(0.6 2)
		\move(1.4 0)\lvec(1.4 2)
		\htext(1.05 1){$\cdots$}
		\htext(0.2 1.6){$0$}
		\htext(1.6 1.6){$0$}
		\htext(0.4 0.4){$1$}
		\htext(1.8 0.4){$1$}
		\move(0 0)\lvec(0.6 2)
		\move(1.4 0)\lvec(2 2)
		\htext(1 -0.3){$\underbrace{\rule{3.5em}{0em}}$}
		\htext(1 -0.7){$x_2$}
		\esegment
		
		\move(4 0)
		\bsegment
		\move(0 0)\lvec(2 0)
		\move(0 0)\lvec(0 2)
		\move(0 2)\lvec(2 2)
		\move(0.6 0)\lvec(0.6 4)
		\move(1.4 0)\lvec(1.4 4)
		\htext(1.05 1){$\cdots$}
		\htext(1.05 3){$\cdots$}
		\move(0 2)\lvec(0 4)
		\move(2 2)\lvec(2 4)
		\move(0 4)\lvec(2 4)
		\htext(0.2 1.6){$0$}
		\htext(1.6 1.6){$0$}
		\htext(0.4 0.4){$1$}
		\htext(1.8 0.4){$1$}
		\move(0 0)\lvec(0.6 2)
		\move(1.4 0)\lvec(2 2)
		\htext(0.3 3){$2$}
		\htext(1.7 3){$2$}
		\htext(1 -0.3){$\underbrace{\rule{3.5em}{0em}}$}
		\htext(1 -0.7){$x_3$}
		\esegment
		
		\move(6 0)
		\bsegment
		\move(0 0)\lvec(2 0)
		\move(0 0)\lvec(0 2)
		\move(0 2)\lvec(2 2)
		\move(2 0)\lvec(2 6)
		\move(0.6 0)\lvec(0.6 6)
		\move(1.4 0)\lvec(1.4 6)
		\htext(1.05 1){$\cdots$}
		\htext(1.05 3){$\cdots$}
		\htext(1.05 5){$\cdots$}
		\move(0 4)\lvec(0 6)
		\move(0 6)\lvec(2 6)
		\move(0 4)\lvec(2 4)
		\htext(2.55 1){\fontsize{8}{8}\selectfont$\cdots$}
		\htext(0.2 1.6){$0$}
		\htext(1.6 1.6){$0$}
		\htext(0.4 0.4){$1$}
		\htext(1.8 0.4){$1$}
		\move(0 0)\lvec(0.6 2)
		\move(1.4 0)\lvec(2 2)	
		\htext(0.3 3){$2$}
		\htext(1.7 3){$2$}
		\htext(0.3 5){$3$}
		\htext(1.7 5){$3$}
		\htext(1 -0.3){$\underbrace{\rule{3.5em}{0em}}$}
		\htext(1 -0.7){$x_4$}
		\esegment

		\move(9 0)
		\bsegment
		\move(0 0)\lvec(2 0)
		\move(0 0)\lvec(0 2)
		\move(0 2)\lvec(2 2)
		\move(0 2)\lvec(0 7)
		\move(0 7)\lvec(2 7)
		\move(0 4)\lvec(2 4)
		\move(0 5)\lvec(2 5)
		\move(0.6 0)\lvec(0.6 4)
		\move(1.4 0)\lvec(1.4 4)
		\move(0.6 5)\lvec(0.6 7)
		\move(1.4 5)\lvec(1.4 7)
		\htext(1.05 1){$\cdots$}
		\htext(1.05 3){$\cdots$}
		\vtext(1.05 4.5){$\cdots$}
		\htext(1.05 6){$\cdots$}
		\htext(0.2 1.6){$0$}
		\htext(1.6 1.6){$0$}
		\htext(0.4 0.4){$1$}
		\htext(1.8 0.4){$1$}
		\move(0 0)\lvec(0.6 2)
		\move(1.4 0)\lvec(2 2)	
		\htext(0.3 3){$2$}
		\htext(1.7 3){$2$}
		\htext(0.3 6){$i$-$1$}
		\htext(1.7 6){$i$-$1$}
		\htext(1 -0.3){$\underbrace{\rule{3.5em}{0em}}$}
		\htext(1 -0.7){$x_i$}
		\esegment
		
		\move(11 0)
		\bsegment
		\move(0 0)\lvec(2 0)
		\move(0 0)\lvec(0 2)
		\move(0 2)\lvec(2 2)
		\move(0 2)\lvec(0 9)
		\move(2 0)\lvec(2 9)
		\move(0 9)\lvec(2 9)
		\move(0 5)\lvec(2 5)
		\move(0 4)\lvec(2 4)
		\move(0 7)\lvec(2 7)
		\move(0.6 0)\lvec(0.6 4)
		\move(1.4 0)\lvec(1.4 4)
		\move(0.6 5)\lvec(0.6 9)
		\move(1.4 5)\lvec(1.4 9)
		\htext(1.05 1){$\cdots$}
		\htext(1.05 3){$\cdots$}
		\vtext(1.05 4.5){$\cdots$}
		\htext(1.05 6){$\cdots$}
		\htext(1.05 8){$\cdots$}
		\htext(0.2 1.6){$0$}
		\htext(1.6 1.6){$0$}
		\htext(0.4 0.4){$1$}
		\htext(1.8 0.4){$1$}
		\move(0 0)\lvec(0.6 2)
		\move(1.4 0)\lvec(2 2)
		\htext(0.3 3){$2$}
		\htext(1.7 3){$2$}
		\htext(0.3 6){$i$-$1$}
		\htext(1.7 6){$i$-$1$}
		\htext(0.3 8){$i$}
		\htext(1.7 8){$i$}
		\htext(2.55 1){\fontsize{8}{8}\selectfont$\cdots$}
		\htext(1 -0.3){$\underbrace{\rule{3.5em}{0em}}$}
		\htext(1 -0.7){$x_{i \!+\! 1}$}
		\esegment
		
		\move(14 0)
		\bsegment
		\move(0 0)\lvec(2 0)
		\move(0 0)\lvec(0 2)
		\move(0 2)\lvec(2 2)
		\move(0 2)\lvec(0 11)
		\move(0 11)\lvec(2 11)
		\htext(0.2 1.6){$0$}
		\htext(1.6 1.6){$0$}
		\htext(0.4 0.4){$1$}
		\htext(1.8 0.4){$1$}
		\move(0 0)\lvec(0.6 2)
		\move(1.4 0)\lvec(2 2)		
		\move(0 4)\lvec(2 4)
		\move(0 6)\lvec(2 6)
		\move(0 8)\lvec(2 8)
		\move(0 9)\lvec(2 9)
		\move(0.6 0)\lvec(0.6 8)
		\move(1.4 0)\lvec(1.4 8)
		\move(0.6 9)\lvec(0.6 11)
		\move(1.4 9)\lvec(1.4 11)
		
		\htext(1.05 1){$\cdots$}
		\htext(1.05 3){$\cdots$}
		\htext(1.05 5){$\cdots$}
		\htext(1.05 7){$\cdots$}
		\vtext(1.05 8.5){$\cdots$}
		\htext(1.05 10){$\cdots$}
		\htext(1.05 12){$\cdots$}
		
		\htext(0.3 3){$2$}
		\htext(1.7 3){$2$}
		\htext(0.3 5){$3$}
		\htext(1.7 5){$3$}
		\htext(0.3 7){$4$}
		\htext(1.7 7){$4$}
		\htext(0.3 10){$n$\!-\!$2$}
		\htext(1.7 10){$n$\!-\!$2$}
		\htext(1 -0.3){$\underbrace{\rule{3.5em}{0em}}$}
		\htext(1 -0.7){$x_{n}$}
		\move(0 11)\lvec(0 13)
		\move(0.6 11)\lvec(0.6 13)
		\move(1.4 11)\lvec(1.4 13)	
		\move(0 13)\lvec(2 13)
		\move(0 11)\lvec(0.6 13)
		
		\move(1.4 11)\lvec(2 13)
		
		\htext(0.23 12.6){$n$\!-\!$1$}
		\htext(1.63 12.6){$n$\!-\!$1$}		
		\esegment
		
		\move(16 0)
		\bsegment
		\move(0 0)\lvec(2 0)
		\move(0 0)\lvec(0 2)
		\move(0 2)\lvec(2 2)
		\move(0 2)\lvec(0 13)
		\move(2 0)\lvec(2 13)
		\move(0 13)\lvec(2 13)
		\htext(0.2 1.6){$0$}
		\htext(1.6 1.6){$0$}
		\htext(0.4 0.4){$1$}
		\htext(1.8 0.4){$1$}
		\move(0 0)\lvec(0.6 2)
		\move(1.4 0)\lvec(2 2)		
		\move(0 4)\lvec(2 4)
		\move(0 6)\lvec(2 6)
		\move(0 8)\lvec(2 8)
		\move(0 9)\lvec(2 9)
		\move(0 11)\lvec(2 11)
		
		\move(0.6 0)\lvec(0.6 8)
		\move(1.4 0)\lvec(1.4 8)
		
		\move(0.6 9)\lvec(0.6 13)
		\move(1.4 9)\lvec(1.4 13)
		
		\htext(1.05 1){$\cdots$}
		\htext(1.05 3){$\cdots$}
		\htext(1.05 5){$\cdots$}
		\htext(1.05 7){$\cdots$}
		\vtext(1.05 8.5){$\cdots$}
		\htext(1.05 10){$\cdots$}
		\htext(1.05 12){$\cdots$}

		\htext(0.3 3){$2$}
		\htext(1.7 3){$2$}
		\htext(0.3 5){$3$}
		\htext(1.7 5){$3$}
		\htext(0.3 7){$4$}
		\htext(1.7 7){$4$}
		\htext(0.3 10){$n$\!-\!$2$}
		\htext(1.7 10){$n$\!-\!$2$}
		
		\htext(1 -0.3){$\underbrace{\rule{3.5em}{0em}}$}
		\htext(1 -0.7){$\bar{x}_{n}$}
		\htext(2.55 1){\fontsize{8}{8}\selectfont$\cdots$}
		\move(0 11)\lvec(0.6 13)
		\htext(0.4 11.4){$n$}
		\move(1.4 11)\lvec(2 13)
		\htext(1.8 11.4){$n$}
		
		\esegment
		
		\move(19 0)
		\bsegment
		\move(0 0)\lvec(2 0)
		\move(0 0)\lvec(0 2)
		\move(0 2)\lvec(2 2)
		\move(0 2)\lvec(0 14)
		\move(0 14)\lvec(2 14)
		\move(0 12)\lvec(2 12)
		\move(0 11)\lvec(2 11)
		\move(0 9)\lvec(2 9)
		\move(0 8)\lvec(2 8)
		\move(0 6)\lvec(2 6)
		\move(0 4)\lvec(2 4)
		
		\move(0.6 9)\lvec(0.6 11)
		\move(1.4 9)\lvec(1.4 11)
		
		\move(0.6 12)\lvec(0.6 14)
		\move(1.4 12)\lvec(1.4 14)
		
		\move(0.6 0)\lvec(0.6 8)
		\move(1.4 0)\lvec(1.4 8)
		\htext(1.05 1){$\cdots$}
		\htext(1.05 3){$\cdots$}
		\htext(1.05 5){$\cdots$}
		\htext(1.05 7){$\cdots$}
		\vtext(1.05 8.5){$\cdots$}
		\htext(1.05 10){$\cdots$}
		\vtext(1.05 11.5){$\cdots$}
		\htext(1.05 13){$\cdots$}
		\htext(0.2 1.6){$0$}
		\htext(1.6 1.6){$0$}
		\htext(0.4 0.4){$1$}
		\htext(1.8 0.4){$1$}
		\move(0 0)\lvec(0.6 2)
		\move(1.4 0)\lvec(2 2)	
		\htext(0.3 3){$2$}
		\htext(1.7 3){$2$}
		\htext(0.3 5){$3$}
		\htext(1.7 5){$3$}
		\htext(0.3 7){$4$}
		\htext(1.7 7){$4$}
		\htext(0.3 10){$n$\!-\!$2$}
		\htext(1.7 10){$n$\!-\!$2$}
		\htext(0.3 13){$i$\!+\!$1$}
		\htext(1.7 13){$i$\!+\!$1$}
		\htext(1 -0.3){$\underbrace{\rule{3.5em}{0em}}$}
		\htext(1 -0.7){$\bar{x}_{i\!+\! 1}$}
		\esegment
		
		\move(21 0)
		\bsegment
		\move(0 0)\lvec(2 0)
		\move(0 0)\lvec(0 2)
		\move(0 2)\lvec(2 2)
		\move(0 2)\lvec(0 16)
		\move(2 0)\lvec(2 16)
		\move(0 16)\lvec(2 16)
		
		\move(0 14)\lvec(2 14)
		\move(0 12)\lvec(2 12)
		\move(0 11)\lvec(2 11)
		\move(0 9)\lvec(2 9)
		\move(0 8)\lvec(2 8)
		\move(0 6)\lvec(2 6)
		\move(0 4)\lvec(2 4)
		
		\move(0.6 9)\lvec(0.6 11)
		\move(1.4 9)\lvec(1.4 11)
		
		\move(0.6 12)\lvec(0.6 16)
		\move(1.4 12)\lvec(1.4 16)
		\htext(0.2 1.6){$0$}
		\htext(1.6 1.6){$0$}
		\htext(0.4 0.4){$1$}
		\htext(1.8 0.4){$1$}
		\move(0 0)\lvec(0.6 2)
		\move(1.4 0)\lvec(2 2)		
		\move(0.6 0)\lvec(0.6 8)
		\move(1.4 0)\lvec(1.4 8)
		\htext(1.05 1){$\cdots$}
		\htext(1.05 3){$\cdots$}
		\htext(1.05 5){$\cdots$}
		\htext(1.05 7){$\cdots$}
		\vtext(1.05 8.5){$\cdots$}
		\htext(1.05 10){$\cdots$}
		\vtext(1.05 11.5){$\cdots$}
		\htext(1.05 13){$\cdots$}
		\htext(1.05 15){$\cdots$}
		
		\htext(0.3 3){$2$}
		\htext(1.7 3){$2$}
		\htext(0.3 5){$3$}
		\htext(1.7 5){$3$}
		\htext(0.3 7){$4$}
		\htext(1.7 7){$4$}
		\htext(0.3 10){$n$\!-\!$2$}
		\htext(1.7 10){$n$\!-\!$2$}
		\htext(0.3 13){$i$\!+\!$1$}
		\htext(1.7 13){$i$\!+\!$1$}
		\htext(0.3 15){$i$}
		\htext(1.7 15){$i$}
		\htext(2.55 1){\fontsize{8}{8}\selectfont$\cdots$}
		\htext(1 -0.3){$\underbrace{\rule{3.5em}{0em}}$}
		\htext(1 -0.7){$\bar{x}_{i}$}
		\esegment
		
		\move(24 0)
		\bsegment
		\move(0 0)\lvec(2 0)
		\move(0 0)\lvec(0 2)
		\move(0 2)\lvec(2 2)
		\move(0 2)\lvec(0 18)
		\move(0 4)\lvec(2 4)
		\move(0 6)\lvec(2 6)
		\move(0 8)\lvec(2 8)
		\move(0 9)\lvec(2 9)
		\move(0 11)\lvec(2 11)
		\move(0 13)\lvec(2 13)
		\move(0 15)\lvec(2 15)
		\move(0 16)\lvec(2 16)
		\move(0 18)\lvec(2 18)
		\move(0.6 0)\lvec(0.6 8)
		\move(1.4 0)\lvec(1.4 8)
		\move(0.6 9)\lvec(0.6 15)
		\move(1.4 9)\lvec(1.4 15)
		\move(0.6 16)\lvec(0.6 18)
		\move(1.4 16)\lvec(1.4 18)
		\htext(1.05 1){$\cdots$}
		\htext(1.05 3){$\cdots$}
		\htext(1.05 5){$\cdots$}
		\htext(1.05 7){$\cdots$}
		\vtext(1.05 8.5){$\cdots$}
		\htext(1.05 10){$\cdots$}
		\htext(1.05 12){$\cdots$}
		\htext(1.05 14){$\cdots$}
		\vtext(1.05 15.5){$\cdots$}
		\htext(1.05 17){$\cdots$}
		\htext(0.2 1.6){$0$}
		\htext(1.6 1.6){$0$}
		\htext(0.4 0.4){$1$}
		\htext(1.8 0.4){$1$}
		\move(0 0)\lvec(0.6 2)
		\move(1.4 0)\lvec(2 2)
		\htext(0.3 3){$2$}
		\htext(1.7 3){$2$}
		\htext(0.3 5){$3$}
		\htext(1.7 5){$3$}
		\htext(0.3 7){$4$}
		\htext(1.7 7){$4$}
		\htext(0.3 10){$n$\!-\!$2$}
		\htext(1.7 10){$n$\!-\!$2$}
		
		\htext(0.3 14){$n$\!-\!$2$}
		\htext(1.7 14){$n$\!-\!$2$}
		\htext(0.3 17){$4$}
		\htext(1.7 17){$4$}
		\htext(1 -0.3){$\underbrace{\rule{3.5em}{0em}}$}
		\htext(1 -0.7){$\bar{x}_{4}$}
		
		\move(0 11)\lvec(0.6 13)
		\htext(0.4 11.4){$n$}
		\move(1.4 11)\lvec(2 13)
		\htext(1.8 11.4){$n$}
		
		\htext(0.23 12.6){$n$\!-\!$1$}
		\htext(1.63 12.6){$n$\!-\!$1$}
		\esegment
		
		\move(26 0)
		\bsegment
		\move(0 0)\lvec(2 0)
		\move(0 0)\lvec(0 2)
		\move(0 2)\lvec(2 2)
		\move(0 2)\lvec(0 20)
		\move(0 4)\lvec(2 4)
		\move(0 6)\lvec(2 6)
		\move(0 8)\lvec(2 8)
		\move(0 9)\lvec(2 9)
		\move(0 11)\lvec(2 11)
		\move(0 13)\lvec(2 13)
		\move(0 15)\lvec(2 15)
		\move(0 16)\lvec(2 16)
		\move(0 18)\lvec(2 18)
		\move(0 20)\lvec(2 20)
		\htext(0.2 1.6){$0$}
		\htext(1.6 1.6){$0$}
		\htext(0.4 0.4){$1$}
		\htext(1.8 0.4){$1$}
		\move(0 0)\lvec(0.6 2)
		\move(1.4 0)\lvec(2 2)		
		\move(0.6 0)\lvec(0.6 8)
		\move(1.4 0)\lvec(1.4 8)
		\move(0.6 9)\lvec(0.6 15)
		\move(1.4 9)\lvec(1.4 15)
		\move(0.6 16)\lvec(0.6 20)
		\move(1.4 16)\lvec(1.4 20)
		\htext(1.05 1){$\cdots$}
		\htext(1.05 3){$\cdots$}
		\htext(1.05 5){$\cdots$}
		\htext(1.05 7){$\cdots$}
		\vtext(1.05 8.5){$\cdots$}
		\htext(1.05 10){$\cdots$}
		\htext(1.05 12){$\cdots$}
		\htext(1.05 14){$\cdots$}
		\vtext(1.05 15.5){$\cdots$}
		\htext(1.05 17){$\cdots$}
		\htext(1.05 19){$\cdots$}
		
		\htext(0.3 3){$2$}
		\htext(1.7 3){$2$}
		\htext(0.3 5){$3$}
		\htext(1.7 5){$3$}
		\htext(0.3 7){$4$}
		\htext(1.7 7){$4$}
		\htext(0.3 10){$n$\!-\!$2$}
		\htext(1.7 10){$n$\!-\!$2$}
		\move(0 11)\lvec(0.6 13)
		\htext(0.4 11.4){$n$}
		\move(1.4 11)\lvec(2 13)
		\htext(1.8 11.4){$n$}
		
		\htext(0.3 14){$n$\!-\!$2$}
		\htext(1.7 14){$n$\!-\!$2$}
		\htext(0.3 17){$4$}
		\htext(1.7 17){$4$}
		\htext(0.3 19){$3$}
		\htext(1.7 19){$3$}
		\htext(1 -0.3){$\underbrace{\rule{3.5em}{0em}}$}
		\htext(1 -0.7){$\bar{x}_{3}$}

		\htext(0.23 12.6){$n$\!-\!$1$}
		\htext(1.63 12.6){$n$\!-\!$1$}
		\esegment
		
		\move(28 0)
		\bsegment
		\move(0 0)\lvec(2 0)
		\move(0 0)\lvec(0 2)
		\move(0 2)\lvec(2 2)
		\move(2 0)\lvec(2 22)
		\move(0 2)\lvec(0 22)
		\move(0 22)\lvec(2 22)
		
		\move(0 4)\lvec(2 4)
		\move(0 6)\lvec(2 6)
		\move(0 8)\lvec(2 8)
		\move(0 9)\lvec(2 9)
		\move(0 11)\lvec(2 11)
		\move(0 13)\lvec(2 13)
		\move(0 15)\lvec(2 15)
		\move(0 16)\lvec(2 16)
		\move(0 18)\lvec(2 18)
		\move(0 20)\lvec(2 20)
		
		\htext(0.2 1.6){$0$}
		\htext(1.6 1.6){$0$}
		\htext(0.4 0.4){$1$}
		\htext(1.8 0.4){$1$}
		\move(0 0)\lvec(0.6 2)
		\move(1.4 0)\lvec(2 2)		
		\move(0.6 0)\lvec(0.6 8)
		\move(1.4 0)\lvec(1.4 8)
		\move(0.6 9)\lvec(0.6 15)
		\move(1.4 9)\lvec(1.4 15)
		\move(0.6 16)\lvec(0.6 22)
		\move(1.4 16)\lvec(1.4 22)
		\htext(1.05 1){$\cdots$}
		\htext(1.05 3){$\cdots$}
		\htext(1.05 5){$\cdots$}
		\htext(1.05 7){$\cdots$}
		\vtext(1.05 8.5){$\cdots$}
		\htext(1.05 10){$\cdots$}
		\htext(1.05 12){$\cdots$}
		\htext(1.05 14){$\cdots$}
		\vtext(1.05 15.5){$\cdots$}
		\htext(1.05 17){$\cdots$}
		\htext(1.05 19){$\cdots$}
		\htext(1.05 21){$\cdots$}
		
		\htext(0.3 3){$2$}
		\htext(1.7 3){$2$}
		\htext(0.3 5){$3$}
		\htext(1.7 5){$3$}
		\htext(0.3 7){$4$}
		\htext(1.7 7){$4$}
		\htext(0.3 10){$n$\!-\!$2$}
		\htext(1.7 10){$n$\!-\!$2$}
		\move(0 11)\lvec(0.6 13)
		\htext(0.4 11.4){$n$}
		\move(1.4 11)\lvec(2 13)
		\htext(1.8 11.4){$n$}
		
		\htext(0.3 14){$n$\!-\!$2$}
		\htext(1.7 14){$n$\!-\!$2$}
		\htext(0.3 17){$4$}
		\htext(1.7 17){$4$}
		\htext(0.3 19){$3$}
		\htext(1.7 19){$3$}
		\htext(0.3 21){$2$}
		\htext(1.7 21){$2$}
		\htext(1 -0.3){$\underbrace{\rule{3.5em}{0em}}$}
		\htext(1 -0.7){$\bar{x}_{2}$}
		
		\htext(0.23 12.6){$n$\!-\!$1$}
		\htext(1.63 12.6){$n$\!-\!$1$}
		\esegment
	\end{texdraw}
\end{center}

\vskip 2mm

\item\label{(5)} $B_{n}^{(1)}$ case :
Level-$l$ reduced Young column $C_{(x_1,\cdots,x_n,x_0,\bar{x}_n,\cdots,\bar{x}_1)}$,
where $x_0=0$ or $1$, $x_i, \bar{x}_i\ge 0$ and $x_0+\sum_{i=1}^{n}(x_i+\bar{x}_i)=l$.

\vskip 2mm

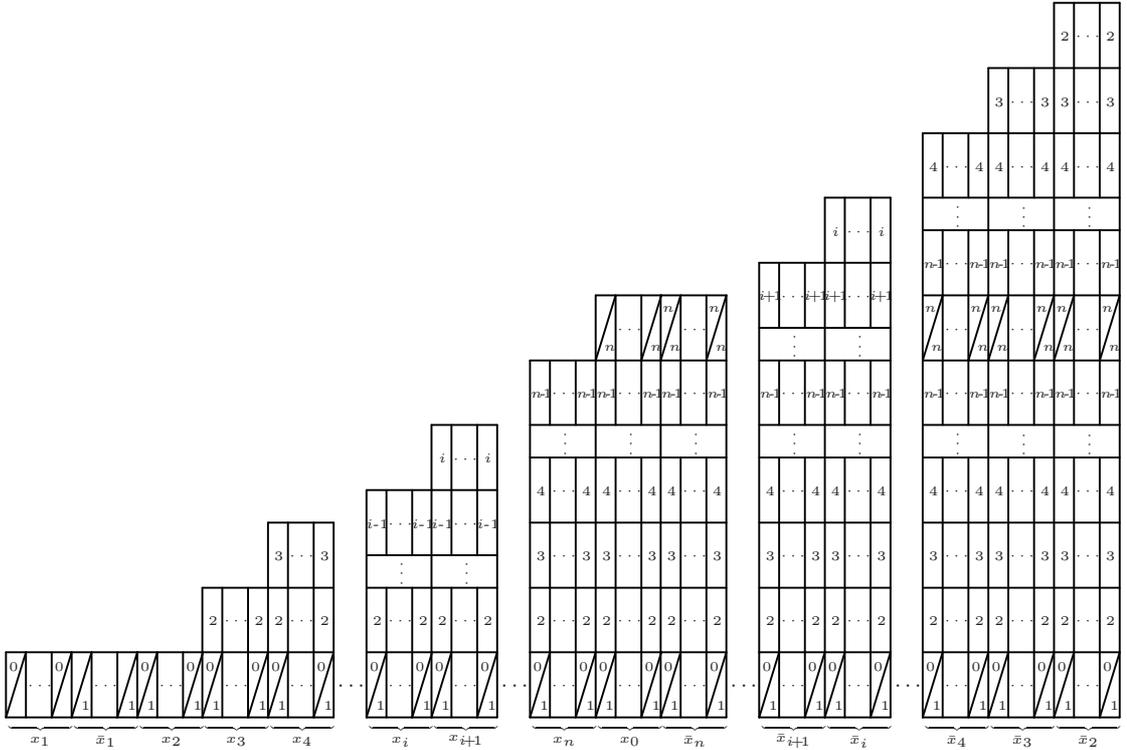
\begin{center}
	\begin{texdraw}
		\fontsize{3}{3}\selectfont
		\drawdim em
		\setunitscale 1.8
		\move(-2 0)
		\bsegment
		\move(0 0)\lvec(2 0)
		\move(0 0)\lvec(0 2)
		\move(0 2)\lvec(2 2)
		\move(0.6 0)\lvec(0.6 2)
		\move(1.4 0)\lvec(1.4 2)
		\move(0 0)\lvec(0.6 2)
		\move(1.4 0)\lvec(2 2)
		\htext(0.2 1.6){$0$}
		\htext(1.6 1.6){$0$}
		\htext(1.05 1){$\cdots$}
		\htext(1 -0.3){$\underbrace{\rule{3.5em}{0em}}$}
		\htext(1 -0.7){$x_1$}
		\esegment
		\move(0 0)
		\bsegment
		\move(0 0)\lvec(2 0)
		\move(0 0)\lvec(0 2)
		\move(0 2)\lvec(2 2)
		\move(0.6 0)\lvec(0.6 2)
		\move(1.4 0)\lvec(1.4 2)
		\move(0 0)\lvec(0.6 2)
		\move(1.4 0)\lvec(2 2)
		\htext(0.4 0.4){$1$}
		\htext(1.8 0.4){$1$}
		\htext(1.05 1){$\cdots$}
		\htext(1 -0.3){$\underbrace{\rule{3.5em}{0em}}$}
		\htext(1 -0.7){$\bar{x}_1$}
		\esegment
		\move(2 0)
		\bsegment
		\move(0 0)\lvec(2 0)
		\move(0 0)\lvec(0 2)
		\move(0 2)\lvec(2 2)
		\move(0.6 0)\lvec(0.6 2)
		\move(1.4 0)\lvec(1.4 2)
		\htext(1.05 1){$\cdots$}
		\htext(0.2 1.6){$0$}
		\htext(1.6 1.6){$0$}
		\htext(0.4 0.4){$1$}
		\htext(1.8 0.4){$1$}
		\move(0 0)\lvec(0.6 2)
		\move(1.4 0)\lvec(2 2)
		\htext(1 -0.3){$\underbrace{\rule{3.5em}{0em}}$}
		\htext(1 -0.7){$x_2$}
		\esegment
		
		\move(4 0)
		\bsegment
		\move(0 0)\lvec(2 0)
		\move(0 0)\lvec(0 2)
		\move(0 2)\lvec(2 2)
		\move(0.6 0)\lvec(0.6 4)
		\move(1.4 0)\lvec(1.4 4)
		\htext(1.05 1){$\cdots$}
		\htext(1.05 3){$\cdots$}
		\move(0 2)\lvec(0 4)
		\move(2 2)\lvec(2 4)
		\move(0 4)\lvec(2 4)
		\htext(0.2 1.6){$0$}
		\htext(1.6 1.6){$0$}
		\htext(0.4 0.4){$1$}
		\htext(1.8 0.4){$1$}
		\move(0 0)\lvec(0.6 2)
		\move(1.4 0)\lvec(2 2)
		\htext(0.3 3){$2$}
		\htext(1.7 3){$2$}
		\htext(1 -0.3){$\underbrace{\rule{3.5em}{0em}}$}
		\htext(1 -0.7){$x_3$}
		\esegment
		
		\move(6 0)
		\bsegment
		\move(0 0)\lvec(2 0)
		\move(0 0)\lvec(0 2)
		\move(0 2)\lvec(2 2)
		\move(2 0)\lvec(2 6)
		\move(0.6 0)\lvec(0.6 6)
		\move(1.4 0)\lvec(1.4 6)
		\htext(1.05 1){$\cdots$}
		\htext(1.05 3){$\cdots$}
		\htext(1.05 5){$\cdots$}
		\move(0 4)\lvec(0 6)
		\move(0 6)\lvec(2 6)
		\move(0 4)\lvec(2 4)
		\htext(2.55 1){\fontsize{8}{8}\selectfont$\cdots$}
		\htext(0.2 1.6){$0$}
		\htext(1.6 1.6){$0$}
		\htext(0.4 0.4){$1$}
		\htext(1.8 0.4){$1$}
		\move(0 0)\lvec(0.6 2)
		\move(1.4 0)\lvec(2 2)	
		\htext(0.3 3){$2$}
		\htext(1.7 3){$2$}
		\htext(0.3 5){$3$}
		\htext(1.7 5){$3$}
		\htext(1 -0.3){$\underbrace{\rule{3.5em}{0em}}$}
		\htext(1 -0.7){$x_4$}
		\esegment

		\move(9 0)
		\bsegment
		\move(0 0)\lvec(2 0)
		\move(0 0)\lvec(0 2)
		\move(0 2)\lvec(2 2)
		\move(0 2)\lvec(0 7)
		\move(0 7)\lvec(2 7)
		\move(0 4)\lvec(2 4)
		\move(0 5)\lvec(2 5)
		\move(0.6 0)\lvec(0.6 4)
		\move(1.4 0)\lvec(1.4 4)
		\move(0.6 5)\lvec(0.6 7)
		\move(1.4 5)\lvec(1.4 7)
		\htext(1.05 1){$\cdots$}
		\htext(1.05 3){$\cdots$}
		\vtext(1.05 4.5){$\cdots$}
		\htext(1.05 6){$\cdots$}
		\htext(0.2 1.6){$0$}
		\htext(1.6 1.6){$0$}
		\htext(0.4 0.4){$1$}
		\htext(1.8 0.4){$1$}
		\move(0 0)\lvec(0.6 2)
		\move(1.4 0)\lvec(2 2)	
		\htext(0.3 3){$2$}
		\htext(1.7 3){$2$}
		\htext(0.3 6){$i$-$1$}
		\htext(1.7 6){$i$-$1$}
		\htext(1 -0.3){$\underbrace{\rule{3.5em}{0em}}$}
		\htext(1 -0.7){$x_i$}
		\esegment
		
		\move(11 0)
		\bsegment
		\move(0 0)\lvec(2 0)
		\move(0 0)\lvec(0 2)
		\move(0 2)\lvec(2 2)
		\move(0 2)\lvec(0 9)
		\move(2 0)\lvec(2 9)
		\move(0 9)\lvec(2 9)
		\move(0 5)\lvec(2 5)
		\move(0 4)\lvec(2 4)
		\move(0 7)\lvec(2 7)
		\move(0.6 0)\lvec(0.6 4)
		\move(1.4 0)\lvec(1.4 4)
		\move(0.6 5)\lvec(0.6 9)
		\move(1.4 5)\lvec(1.4 9)
		\htext(1.05 1){$\cdots$}
		\htext(1.05 3){$\cdots$}
		\vtext(1.05 4.5){$\cdots$}
		\htext(1.05 6){$\cdots$}
		\htext(1.05 8){$\cdots$}
		\htext(0.2 1.6){$0$}
		\htext(1.6 1.6){$0$}
		\htext(0.4 0.4){$1$}
		\htext(1.8 0.4){$1$}
		\move(0 0)\lvec(0.6 2)
		\move(1.4 0)\lvec(2 2)
		\htext(0.3 3){$2$}
		\htext(1.7 3){$2$}
		\htext(0.3 6){$i$-$1$}
		\htext(1.7 6){$i$-$1$}
		\htext(0.3 8){$i$}
		\htext(1.7 8){$i$}
		\htext(2.55 1){\fontsize{8}{8}\selectfont$\cdots$}
		\htext(1 -0.3){$\underbrace{\rule{3.5em}{0em}}$}
		\htext(1 -0.7){$x_{i \!+\! 1}$}
		\esegment
		
		\move(14 0)
		\bsegment
		\move(0 0)\lvec(2 0)
		\move(0 0)\lvec(0 2)
		\move(0 2)\lvec(2 2)
		\move(0 2)\lvec(0 11)
		\move(0 11)\lvec(2 11)
		\htext(0.2 1.6){$0$}
		\htext(1.6 1.6){$0$}
		\htext(0.4 0.4){$1$}
		\htext(1.8 0.4){$1$}
		\move(0 0)\lvec(0.6 2)
		\move(1.4 0)\lvec(2 2)		
		\move(0 4)\lvec(2 4)
		\move(0 6)\lvec(2 6)
		\move(0 8)\lvec(2 8)
		\move(0 9)\lvec(2 9)
		\move(0.6 0)\lvec(0.6 8)
		\move(1.4 0)\lvec(1.4 8)
		\move(0.6 9)\lvec(0.6 11)
		\move(1.4 9)\lvec(1.4 11)
		
		\htext(1.05 1){$\cdots$}
		\htext(1.05 3){$\cdots$}
		\htext(1.05 5){$\cdots$}
		\htext(1.05 7){$\cdots$}
		\vtext(1.05 8.5){$\cdots$}
		\htext(1.05 10){$\cdots$}

		\htext(0.3 3){$2$}
		\htext(1.7 3){$2$}
		\htext(0.3 5){$3$}
		\htext(1.7 5){$3$}
		\htext(0.3 7){$4$}
		\htext(1.7 7){$4$}
		\htext(0.3 10){$n$\!-\!$1$}
		\htext(1.7 10){$n$\!-\!$1$}
		\htext(1 -0.3){$\underbrace{\rule{3.5em}{0em}}$}
		\htext(1 -0.7){$x_{n}$}
		
		\esegment
		
		\move(16 0)
		\bsegment
		\move(0 0)\lvec(2 0)
		\move(0 0)\lvec(0 2)
		\move(0 2)\lvec(2 2)
		\move(0 2)\lvec(0 11)
		\move(0 11)\lvec(2 11)
		\htext(0.2 1.6){$0$}
		\htext(1.6 1.6){$0$}
		\htext(0.4 0.4){$1$}
		\htext(1.8 0.4){$1$}
		\move(0 0)\lvec(0.6 2)
		\move(1.4 0)\lvec(2 2)		
		\move(0 4)\lvec(2 4)
		\move(0 6)\lvec(2 6)
		\move(0 8)\lvec(2 8)
		\move(0 9)\lvec(2 9)
		\move(0.6 0)\lvec(0.6 8)
		\move(1.4 0)\lvec(1.4 8)
		\move(0.6 9)\lvec(0.6 11)
		\move(1.4 9)\lvec(1.4 11)
		
		\htext(1.05 1){$\cdots$}
		\htext(1.05 3){$\cdots$}
		\htext(1.05 5){$\cdots$}
		\htext(1.05 7){$\cdots$}
		\vtext(1.05 8.5){$\cdots$}
		\htext(1.05 10){$\cdots$}
		\htext(1.05 12){$\cdots$}
		
		\htext(0.3 3){$2$}
		\htext(1.7 3){$2$}
		\htext(0.3 5){$3$}
		\htext(1.7 5){$3$}
		\htext(0.3 7){$4$}
		\htext(1.7 7){$4$}
		\htext(0.3 10){$n$\!-\!$1$}
		\htext(1.7 10){$n$\!-\!$1$}
		\htext(1 -0.3){$\underbrace{\rule{3.5em}{0em}}$}
		\htext(1 -0.7){$x_{0}$}
		\move(0 11)\lvec(0 13)
		\move(0.6 11)\lvec(0.6 13)
		\move(1.4 11)\lvec(1.4 13)
		\move(0 13)\lvec(2 13)
		\move(0 11)\lvec(0.6 13)
		\move(1.4 11)\lvec(2 13)
		\htext(0.4 11.4){$n$}
		\htext(1.8 11.4){$n$}	
		\esegment

		\move(18 0)
		\bsegment
		\move(0 0)\lvec(2 0)
		\move(0 0)\lvec(0 2)
		\move(0 2)\lvec(2 2)
		\move(0 2)\lvec(0 13)
		\move(2 0)\lvec(2 13)
		\move(0 13)\lvec(2 13)
		\htext(0.2 1.6){$0$}
		\htext(1.6 1.6){$0$}
		\htext(0.4 0.4){$1$}
		\htext(1.8 0.4){$1$}
		\move(0 0)\lvec(0.6 2)
		\move(1.4 0)\lvec(2 2)		
		\move(0 4)\lvec(2 4)
		\move(0 6)\lvec(2 6)
		\move(0 8)\lvec(2 8)
		\move(0 9)\lvec(2 9)
		\move(0 11)\lvec(2 11)
		
		\move(0.6 0)\lvec(0.6 8)
		\move(1.4 0)\lvec(1.4 8)
		
		\move(0.6 9)\lvec(0.6 13)
		\move(1.4 9)\lvec(1.4 13)
		
		\htext(1.05 1){$\cdots$}
		\htext(1.05 3){$\cdots$}
		\htext(1.05 5){$\cdots$}
		\htext(1.05 7){$\cdots$}
		\vtext(1.05 8.5){$\cdots$}
		\htext(1.05 10){$\cdots$}
		\htext(1.05 12){$\cdots$}

		\htext(0.3 3){$2$}
		\htext(1.7 3){$2$}
		\htext(0.3 5){$3$}
		\htext(1.7 5){$3$}
		\htext(0.3 7){$4$}
		\htext(1.7 7){$4$}
		\htext(0.3 10){$n$\!-\!$1$}
		\htext(1.7 10){$n$\!-\!$1$}
		
		\htext(1 -0.3){$\underbrace{\rule{3.5em}{0em}}$}
		\htext(1 -0.7){$\bar{x}_{n}$}
		\htext(2.55 1){\fontsize{8}{8}\selectfont$\cdots$}
		\move(0 11)\lvec(0.6 13)
		\htext(0.4 11.4){$n$}
		\move(1.4 11)\lvec(2 13)
		\htext(1.8 11.4){$n$}
		\htext(0.2 12.6){$n$}
		\htext(1.6 12.6){$n$}
		\esegment
		
		\move(21 0)
		\bsegment
		\move(0 0)\lvec(2 0)
		\move(0 0)\lvec(0 2)
		\move(0 2)\lvec(2 2)
		\move(0 2)\lvec(0 14)
		\move(0 14)\lvec(2 14)
		\move(0 12)\lvec(2 12)
		\move(0 11)\lvec(2 11)
		\move(0 9)\lvec(2 9)
		\move(0 8)\lvec(2 8)
		\move(0 6)\lvec(2 6)
		\move(0 4)\lvec(2 4)
		
		\move(0.6 9)\lvec(0.6 11)
		\move(1.4 9)\lvec(1.4 11)
		
		\move(0.6 12)\lvec(0.6 14)
		\move(1.4 12)\lvec(1.4 14)
		
		\move(0.6 0)\lvec(0.6 8)
		\move(1.4 0)\lvec(1.4 8)
		\htext(1.05 1){$\cdots$}
		\htext(1.05 3){$\cdots$}
		\htext(1.05 5){$\cdots$}
		\htext(1.05 7){$\cdots$}
		\vtext(1.05 8.5){$\cdots$}
		\htext(1.05 10){$\cdots$}
		\vtext(1.05 11.5){$\cdots$}
		\htext(1.05 13){$\cdots$}
		\htext(0.2 1.6){$0$}
		\htext(1.6 1.6){$0$}
		\htext(0.4 0.4){$1$}
		\htext(1.8 0.4){$1$}
		\move(0 0)\lvec(0.6 2)
		\move(1.4 0)\lvec(2 2)	
		\htext(0.3 3){$2$}
		\htext(1.7 3){$2$}
		\htext(0.3 5){$3$}
		\htext(1.7 5){$3$}
		\htext(0.3 7){$4$}
		\htext(1.7 7){$4$}
		\htext(0.3 10){$n$\!-\!$1$}
		\htext(1.7 10){$n$\!-\!$1$}
		\htext(0.3 13){$i$\!+\!$1$}
		\htext(1.7 13){$i$\!+\!$1$}
		\htext(1 -0.3){$\underbrace{\rule{3.5em}{0em}}$}
		\htext(1 -0.7){$\bar{x}_{i\!+\! 1}$}
		\esegment
		
		\move(23 0)
		\bsegment
		\move(0 0)\lvec(2 0)
		\move(0 0)\lvec(0 2)
		\move(0 2)\lvec(2 2)
		\move(0 2)\lvec(0 16)
		\move(2 0)\lvec(2 16)
		\move(0 16)\lvec(2 16)
		
		\move(0 14)\lvec(2 14)
		\move(0 12)\lvec(2 12)
		\move(0 11)\lvec(2 11)
		\move(0 9)\lvec(2 9)
		\move(0 8)\lvec(2 8)
		\move(0 6)\lvec(2 6)
		\move(0 4)\lvec(2 4)
		
		\move(0.6 9)\lvec(0.6 11)
		\move(1.4 9)\lvec(1.4 11)
		
		\move(0.6 12)\lvec(0.6 16)
		\move(1.4 12)\lvec(1.4 16)
		\htext(0.2 1.6){$0$}
		\htext(1.6 1.6){$0$}
		\htext(0.4 0.4){$1$}
		\htext(1.8 0.4){$1$}
		\move(0 0)\lvec(0.6 2)
		\move(1.4 0)\lvec(2 2)		
		\move(0.6 0)\lvec(0.6 8)
		\move(1.4 0)\lvec(1.4 8)
		\htext(1.05 1){$\cdots$}
		\htext(1.05 3){$\cdots$}
		\htext(1.05 5){$\cdots$}
		\htext(1.05 7){$\cdots$}
		\vtext(1.05 8.5){$\cdots$}
		\htext(1.05 10){$\cdots$}
		\vtext(1.05 11.5){$\cdots$}
		\htext(1.05 13){$\cdots$}
		\htext(1.05 15){$\cdots$}
		
		\htext(0.3 3){$2$}
		\htext(1.7 3){$2$}
		\htext(0.3 5){$3$}
		\htext(1.7 5){$3$}
		\htext(0.3 7){$4$}
		\htext(1.7 7){$4$}
		\htext(0.3 10){$n$\!-\!$1$}
		\htext(1.7 10){$n$\!-\!$1$}
		\htext(0.3 13){$i$\!+\!$1$}
		\htext(1.7 13){$i$\!+\!$1$}
		\htext(0.3 15){$i$}
		\htext(1.7 15){$i$}
		\htext(2.55 1){\fontsize{8}{8}\selectfont$\cdots$}
		\htext(1 -0.3){$\underbrace{\rule{3.5em}{0em}}$}
		\htext(1 -0.7){$\bar{x}_{i}$}
		\esegment
		
		\move(26 0)
		\bsegment
		\move(0 0)\lvec(2 0)
		\move(0 0)\lvec(0 2)
		\move(0 2)\lvec(2 2)
		\move(0 2)\lvec(0 18)
		\move(0 4)\lvec(2 4)
		\move(0 6)\lvec(2 6)
		\move(0 8)\lvec(2 8)
		\move(0 9)\lvec(2 9)
		\move(0 11)\lvec(2 11)
		\move(0 13)\lvec(2 13)
		\move(0 15)\lvec(2 15)
		\move(0 16)\lvec(2 16)
		\move(0 18)\lvec(2 18)
		\move(0.6 0)\lvec(0.6 8)
		\move(1.4 0)\lvec(1.4 8)
		\move(0.6 9)\lvec(0.6 15)
		\move(1.4 9)\lvec(1.4 15)
		\move(0.6 16)\lvec(0.6 18)
		\move(1.4 16)\lvec(1.4 18)
		\htext(1.05 1){$\cdots$}
		\htext(1.05 3){$\cdots$}
		\htext(1.05 5){$\cdots$}
		\htext(1.05 7){$\cdots$}
		\vtext(1.05 8.5){$\cdots$}
		\htext(1.05 10){$\cdots$}
		\htext(1.05 12){$\cdots$}
		\htext(1.05 14){$\cdots$}
		\vtext(1.05 15.5){$\cdots$}
		\htext(1.05 17){$\cdots$}
		\htext(0.2 1.6){$0$}
		\htext(1.6 1.6){$0$}
		\htext(0.4 0.4){$1$}
		\htext(1.8 0.4){$1$}
		\move(0 0)\lvec(0.6 2)
		\move(1.4 0)\lvec(2 2)
		\htext(0.3 3){$2$}
		\htext(1.7 3){$2$}
		\htext(0.3 5){$3$}
		\htext(1.7 5){$3$}
		\htext(0.3 7){$4$}
		\htext(1.7 7){$4$}
		\htext(0.3 10){$n$\!-\!$1$}
		\htext(1.7 10){$n$\!-\!$1$}
		
		\htext(0.3 14){$n$\!-\!$1$}
		\htext(1.7 14){$n$\!-\!$1$}
		\htext(0.3 17){$4$}
		\htext(1.7 17){$4$}
		\htext(1 -0.3){$\underbrace{\rule{3.5em}{0em}}$}
		\htext(1 -0.7){$\bar{x}_{4}$}
		
		\move(0 11)\lvec(0.6 13)
		\htext(0.4 11.4){$n$}
		\move(1.4 11)\lvec(2 13)
		\htext(1.8 11.4){$n$}
		\htext(0.2 12.6){$n$}
		\htext(1.6 12.6){$n$}
		\esegment
		
		\move(28 0)
		\bsegment
		\move(0 0)\lvec(2 0)
		\move(0 0)\lvec(0 2)
		\move(0 2)\lvec(2 2)
		\move(0 2)\lvec(0 20)
		\move(0 4)\lvec(2 4)
		\move(0 6)\lvec(2 6)
		\move(0 8)\lvec(2 8)
		\move(0 9)\lvec(2 9)
		\move(0 11)\lvec(2 11)
		\move(0 13)\lvec(2 13)
		\move(0 15)\lvec(2 15)
		\move(0 16)\lvec(2 16)
		\move(0 18)\lvec(2 18)
		\move(0 20)\lvec(2 20)
		\htext(0.2 1.6){$0$}
		\htext(1.6 1.6){$0$}
		\htext(0.4 0.4){$1$}
		\htext(1.8 0.4){$1$}
		\move(0 0)\lvec(0.6 2)
		\move(1.4 0)\lvec(2 2)		
		\move(0.6 0)\lvec(0.6 8)
		\move(1.4 0)\lvec(1.4 8)
		\move(0.6 9)\lvec(0.6 15)
		\move(1.4 9)\lvec(1.4 15)
		\move(0.6 16)\lvec(0.6 20)
		\move(1.4 16)\lvec(1.4 20)
		\htext(1.05 1){$\cdots$}
		\htext(1.05 3){$\cdots$}
		\htext(1.05 5){$\cdots$}
		\htext(1.05 7){$\cdots$}
		\vtext(1.05 8.5){$\cdots$}
		\htext(1.05 10){$\cdots$}
		\htext(1.05 12){$\cdots$}
		\htext(1.05 14){$\cdots$}
		\vtext(1.05 15.5){$\cdots$}
		\htext(1.05 17){$\cdots$}
		\htext(1.05 19){$\cdots$}
		
		\htext(0.3 3){$2$}
		\htext(1.7 3){$2$}
		\htext(0.3 5){$3$}
		\htext(1.7 5){$3$}
		\htext(0.3 7){$4$}
		\htext(1.7 7){$4$}
		\htext(0.3 10){$n$\!-\!$1$}
		\htext(1.7 10){$n$\!-\!$1$}
		\move(0 11)\lvec(0.6 13)
		\htext(0.4 11.4){$n$}
		\move(1.4 11)\lvec(2 13)
		\htext(1.8 11.4){$n$}
		\htext(0.2 12.6){$n$}
		\htext(1.6 12.6){$n$}
		\htext(0.3 14){$n$\!-\!$1$}
		\htext(1.7 14){$n$\!-\!$1$}
		\htext(0.3 17){$4$}
		\htext(1.7 17){$4$}
		\htext(0.3 19){$3$}
		\htext(1.7 19){$3$}
		\htext(1 -0.3){$\underbrace{\rule{3.5em}{0em}}$}
		\htext(1 -0.7){$\bar{x}_{3}$}
		\esegment
		
		\move(30 0)
		\bsegment
		\move(0 0)\lvec(2 0)
		\move(0 0)\lvec(0 2)
		\move(0 2)\lvec(2 2)
		\move(2 0)\lvec(2 22)
		\move(0 2)\lvec(0 22)
		\move(0 22)\lvec(2 22)
		
		\move(0 4)\lvec(2 4)
		\move(0 6)\lvec(2 6)
		\move(0 8)\lvec(2 8)
		\move(0 9)\lvec(2 9)
		\move(0 11)\lvec(2 11)
		\move(0 13)\lvec(2 13)
		\move(0 15)\lvec(2 15)
		\move(0 16)\lvec(2 16)
		\move(0 18)\lvec(2 18)
		\move(0 20)\lvec(2 20)
		
		\htext(0.2 1.6){$0$}
		\htext(1.6 1.6){$0$}
		\htext(0.4 0.4){$1$}
		\htext(1.8 0.4){$1$}
		\move(0 0)\lvec(0.6 2)
		\move(1.4 0)\lvec(2 2)		
		\move(0.6 0)\lvec(0.6 8)
		\move(1.4 0)\lvec(1.4 8)
		\move(0.6 9)\lvec(0.6 15)
		\move(1.4 9)\lvec(1.4 15)
		\move(0.6 16)\lvec(0.6 22)
		\move(1.4 16)\lvec(1.4 22)
		\htext(1.05 1){$\cdots$}
		\htext(1.05 3){$\cdots$}
		\htext(1.05 5){$\cdots$}
		\htext(1.05 7){$\cdots$}
		\vtext(1.05 8.5){$\cdots$}
		\htext(1.05 10){$\cdots$}
		\htext(1.05 12){$\cdots$}
		\htext(1.05 14){$\cdots$}
		\vtext(1.05 15.5){$\cdots$}
		\htext(1.05 17){$\cdots$}
		\htext(1.05 19){$\cdots$}
		\htext(1.05 21){$\cdots$}
		
		\htext(0.3 3){$2$}
		\htext(1.7 3){$2$}
		\htext(0.3 5){$3$}
		\htext(1.7 5){$3$}
		\htext(0.3 7){$4$}
		\htext(1.7 7){$4$}
		\htext(0.3 10){$n$\!-\!$1$}
		\htext(1.7 10){$n$\!-\!$1$}
		\move(0 11)\lvec(0.6 13)
		\htext(0.4 11.4){$n$}
		\move(1.4 11)\lvec(2 13)
		\htext(1.8 11.4){$n$}
		\htext(0.2 12.6){$n$}
		\htext(1.6 12.6){$n$}
		\htext(0.3 14){$n$\!-\!$1$}
		\htext(1.7 14){$n$\!-\!$1$}
		\htext(0.3 17){$4$}
		\htext(1.7 17){$4$}
		\htext(0.3 19){$3$}
		\htext(1.7 19){$3$}
		\htext(0.3 21){$2$}
		\htext(1.7 21){$2$}
		\htext(1 -0.3){$\underbrace{\rule{3.5em}{0em}}$}
		\htext(1 -0.7){$\bar{x}_{2}$}
		\esegment
	\end{texdraw}
\end{center}

\vskip 2mm

\item\label{(6)} $C_{n}^{(1)}$ case :
Level-$l$ reduced Young column $C_{(x_1,\cdots,x_n,\bar{x}_n,\cdots,\bar{x}_1)}$, where $2l\ge\sum_{i=1}^{n}(x_i+\bar{x}_i)\in 2\mathbf Z$, $x_i, \bar{x}_i\ge 0$ and $l'=2l-\sum_{i=1}^{n}(x_i+\bar{x}_i)$.

\vskip 2mm

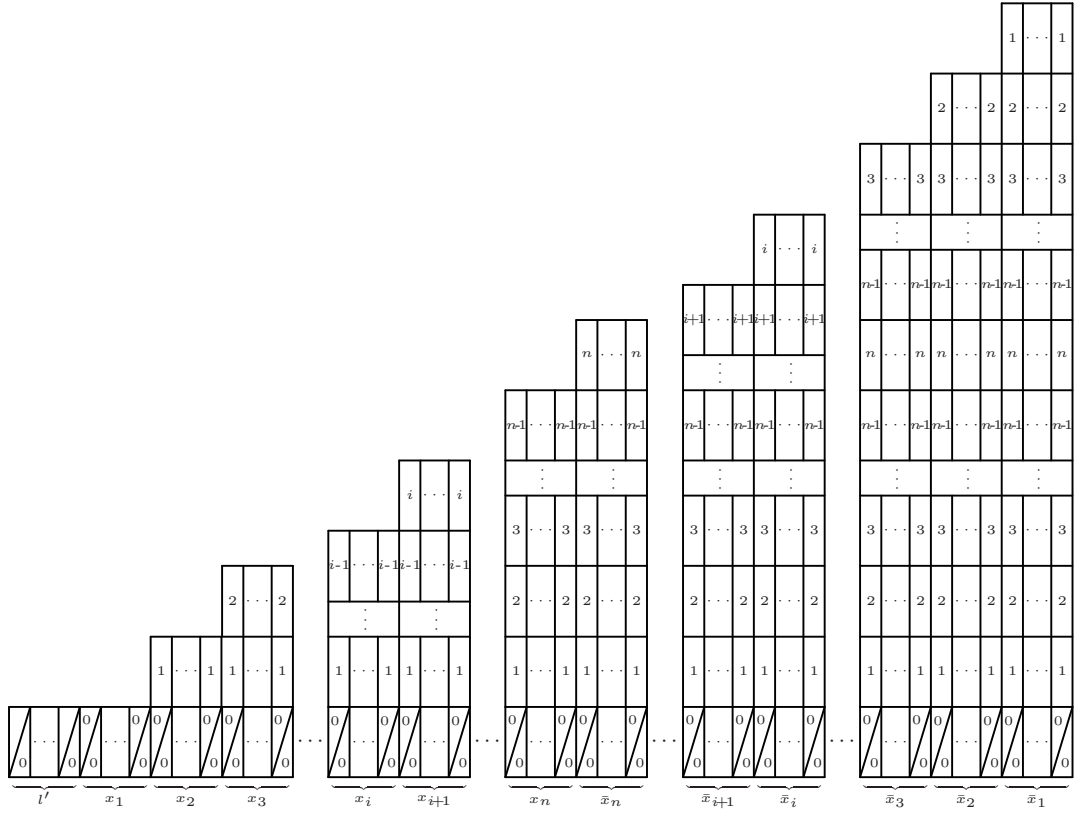
\begin{center}
	\begin{texdraw}
		\fontsize{3}{3}\selectfont
		\drawdim em
		\setunitscale 1.95
		\move(0 0)\lvec(2 0)
		\move(0 0)\lvec(0 2)
		\move(0 2)\lvec(2 2)
		\move(0.6 0)\lvec(0.6 2)
		\move(1.4 0)\lvec(1.4 2)
		\move(0 0)\lvec(0.6 2)
		\move(1.4 0)\lvec(2 2)
		\htext(0.4 0.4){$0$}
		\htext(1.8 0.4){$0$}
		\htext(1.05 1){$\cdots$}
		\htext(1 -0.3){$\underbrace{\rule{3.5em}{0em}}$}
		\htext(1 -0.6){$l'$}
		\move(2 0)
		\bsegment
		\move(0 0)\lvec(2 0)
		\move(0 0)\lvec(0 2)
		\move(0 2)\lvec(2 2)
		\move(0.6 0)\lvec(0.6 2)
		\move(1.4 0)\lvec(1.4 2)
		\htext(0.2 1.6){$0$}
		\htext(1.6 1.6){$0$}
		\htext(0.4 0.4){$0$}
		\htext(1.8 0.4){$0$}
		\move(0 0)\lvec(0.6 2)
		\move(1.4 0)\lvec(2 2)
		
		\htext(1.05 1){$\cdots$}
	    \htext(1 -0.3){$\underbrace{\rule{3.5em}{0em}}$}
		\htext(1 -0.7){$x_1$}
		\esegment
		
		\move(4 0)
		\bsegment
		\move(0 0)\lvec(2 0)
		\move(0 0)\lvec(0 2)
		\move(0 2)\lvec(2 2)
		\move(0.6 0)\lvec(0.6 4)
		\move(1.4 0)\lvec(1.4 4)
		\htext(1.05 1){$\cdots$}
		\htext(1.05 3){$\cdots$}
		\move(0 2)\lvec(0 4)
		\move(2 2)\lvec(2 4)
		\move(0 4)\lvec(2 4)
		\htext(0.2 1.6){$0$}
	    \htext(1.6 1.6){$0$}
	    \htext(0.4 0.4){$0$}
	    \htext(1.8 0.4){$0$}
     	\move(0 0)\lvec(0.6 2)
    	\move(1.4 0)\lvec(2 2)
		\htext(0.3 3){$1$}
		\htext(1.7 3){$1$}
		\htext(1 -0.3){$\underbrace{\rule{3.5em}{0em}}$}
		\htext(1 -0.7){$x_2$}
		\esegment
		
		\move(6 0)
		\bsegment
		\move(0 0)\lvec(2 0)
		\move(0 0)\lvec(0 2)
		\move(0 2)\lvec(2 2)
		\move(2 0)\lvec(2 6)
		\move(0.6 0)\lvec(0.6 6)
		\move(1.4 0)\lvec(1.4 6)
		\htext(1.05 1){$\cdots$}
		\htext(1.05 3){$\cdots$}
		\htext(1.05 5){$\cdots$}
		\move(0 4)\lvec(0 6)
		\move(0 6)\lvec(2 6)
		\move(0 4)\lvec(2 4)
		\htext(2.55 1){\fontsize{8}{8}\selectfont$\cdots$}
	\htext(0.2 1.6){$0$}
	\htext(1.6 1.6){$0$}
	\htext(0.4 0.4){$0$}
	\htext(1.8 0.4){$0$}
	\move(0 0)\lvec(0.6 2)
	\move(1.4 0)\lvec(2 2)
		\htext(0.3 3){$1$}
		\htext(1.7 3){$1$}
		\htext(0.3 5){$2$}
		\htext(1.7 5){$2$}
		\htext(1 -0.3){$\underbrace{\rule{3.5em}{0em}}$}
		\htext(1 -0.7){$x_3$}
		\esegment

		\move(9 0)
		\bsegment
		\move(0 0)\lvec(2 0)
		\move(0 0)\lvec(0 2)
		\move(0 2)\lvec(2 2)
		\move(0 2)\lvec(0 7)
		\move(0 7)\lvec(2 7)
		\move(0 4)\lvec(2 4)
		\move(0 5)\lvec(2 5)
		\move(0.6 0)\lvec(0.6 4)
		\move(1.4 0)\lvec(1.4 4)
		\move(0.6 5)\lvec(0.6 7)
		\move(1.4 5)\lvec(1.4 7)
		\htext(1.05 1){$\cdots$}
		\htext(1.05 3){$\cdots$}
		\vtext(1.05 4.5){$\cdots$}
		\htext(1.05 6){$\cdots$}
	\htext(0.2 1.6){$0$}
	\htext(1.6 1.6){$0$}
	\htext(0.4 0.4){$0$}
	\htext(1.8 0.4){$0$}
	\move(0 0)\lvec(0.6 2)
	\move(1.4 0)\lvec(2 2)
		\htext(0.3 3){$1$}
		\htext(1.7 3){$1$}
		\htext(0.3 6){$i$-$1$}
		\htext(1.7 6){$i$-$1$}
		\htext(1 -0.3){$\underbrace{\rule{3.5em}{0em}}$}
		\htext(1 -0.7){$x_i$}
		\esegment
		
		\move(11 0)
		\bsegment
		\move(0 0)\lvec(2 0)
		\move(0 0)\lvec(0 2)
		\move(0 2)\lvec(2 2)
		\move(0 2)\lvec(0 9)
		\move(2 0)\lvec(2 9)
		\move(0 9)\lvec(2 9)
		\move(0 5)\lvec(2 5)
		\move(0 4)\lvec(2 4)
		\move(0 7)\lvec(2 7)
		\move(0.6 0)\lvec(0.6 4)
		\move(1.4 0)\lvec(1.4 4)
		\move(0.6 5)\lvec(0.6 9)
		\move(1.4 5)\lvec(1.4 9)
		\htext(1.05 1){$\cdots$}
		\htext(1.05 3){$\cdots$}
		\vtext(1.05 4.5){$\cdots$}
		\htext(1.05 6){$\cdots$}
		\htext(1.05 8){$\cdots$}
	\htext(0.2 1.6){$0$}
	\htext(1.6 1.6){$0$}
	\htext(0.4 0.4){$0$}
	\htext(1.8 0.4){$0$}
	\move(0 0)\lvec(0.6 2)
	\move(1.4 0)\lvec(2 2)
		\htext(0.3 3){$1$}
		\htext(1.7 3){$1$}
		\htext(0.3 6){$i$-$1$}
		\htext(1.7 6){$i$-$1$}
		\htext(0.3 8){$i$}
		\htext(1.7 8){$i$}
		\htext(2.55 1){\fontsize{8}{8}\selectfont$\cdots$}
		\htext(1 -0.3){$\underbrace{\rule{3.5em}{0em}}$}
		\htext(1 -0.7){$x_{i \!+\! 1}$}
		\esegment
		
		\move(14 0)
		\bsegment
		\move(0 0)\lvec(2 0)
		\move(0 0)\lvec(0 2)
		\move(0 2)\lvec(2 2)
		\move(0 2)\lvec(0 11)
		\move(0 11)\lvec(2 11)
		
		\move(0 4)\lvec(2 4)
		\move(0 6)\lvec(2 6)
		\move(0 8)\lvec(2 8)
		\move(0 9)\lvec(2 9)
		\move(0.6 0)\lvec(0.6 8)
		\move(1.4 0)\lvec(1.4 8)
		\move(0.6 9)\lvec(0.6 11)
		\move(1.4 9)\lvec(1.4 11)
		
		\htext(1.05 1){$\cdots$}
		\htext(1.05 3){$\cdots$}
		\htext(1.05 5){$\cdots$}
		\htext(1.05 7){$\cdots$}
		\vtext(1.05 8.5){$\cdots$}
		\htext(1.05 10){$\cdots$}
		\htext(0.2 1.6){$0$}
		\htext(1.6 1.6){$0$}
		\htext(0.4 0.4){$0$}
		\htext(1.8 0.4){$0$}
		\move(0 0)\lvec(0.6 2)
		\move(1.4 0)\lvec(2 2)
		\htext(0.3 3){$1$}
		\htext(1.7 3){$1$}
		\htext(0.3 5){$2$}
		\htext(1.7 5){$2$}
		\htext(0.3 7){$3$}
		\htext(1.7 7){$3$}
		\htext(0.3 10){$n$\!-\!$1$}
		\htext(1.7 10){$n$\!-\!$1$}
		\htext(1 -0.3){$\underbrace{\rule{3.5em}{0em}}$}
		\htext(1 -0.7){$x_{n}$}
		\esegment
		
		\move(16 0)
		\bsegment
		\move(0 0)\lvec(2 0)
		\move(0 0)\lvec(0 2)
		\move(0 2)\lvec(2 2)
		\move(0 2)\lvec(0 13)
		\move(2 0)\lvec(2 13)
		\move(0 13)\lvec(2 13)
		
		\move(0 4)\lvec(2 4)
		\move(0 6)\lvec(2 6)
		\move(0 8)\lvec(2 8)
		\move(0 9)\lvec(2 9)
		\move(0 11)\lvec(2 11)
		
		\move(0.6 0)\lvec(0.6 8)
		\move(1.4 0)\lvec(1.4 8)
		
		\move(0.6 9)\lvec(0.6 13)
		\move(1.4 9)\lvec(1.4 13)
		
		\htext(1.05 1){$\cdots$}
		\htext(1.05 3){$\cdots$}
		\htext(1.05 5){$\cdots$}
		\htext(1.05 7){$\cdots$}
		\vtext(1.05 8.5){$\cdots$}
		\htext(1.05 10){$\cdots$}
		\htext(1.05 12){$\cdots$}
	\htext(0.2 1.6){$0$}
	\htext(1.6 1.6){$0$}
	\htext(0.4 0.4){$0$}
	\htext(1.8 0.4){$0$}
	\move(0 0)\lvec(0.6 2)
	\move(1.4 0)\lvec(2 2)
		\htext(0.3 3){$1$}
		\htext(1.7 3){$1$}
		\htext(0.3 5){$2$}
		\htext(1.7 5){$2$}
		\htext(0.3 7){$3$}
		\htext(1.7 7){$3$}
		\htext(0.3 10){$n$\!-\!$1$}
		\htext(1.7 10){$n$\!-\!$1$}
		\htext(0.3 12){$n$}
		\htext(1.7 12){$n$}
		\htext(1 -0.3){$\underbrace{\rule{3.5em}{0em}}$}
		\htext(1 -0.7){$\bar{x}_{n}$}
		\htext(2.55 1){\fontsize{8}{8}\selectfont$\cdots$}
		\esegment
		
		\move(19 0)
		\bsegment
		\move(0 0)\lvec(2 0)
		\move(0 0)\lvec(0 2)
		\move(0 2)\lvec(2 2)
		\move(0 2)\lvec(0 14)
		\move(0 14)\lvec(2 14)
		\move(0 12)\lvec(2 12)
		\move(0 11)\lvec(2 11)
		\move(0 9)\lvec(2 9)
		\move(0 8)\lvec(2 8)
		\move(0 6)\lvec(2 6)
		\move(0 4)\lvec(2 4)
		
		\move(0.6 9)\lvec(0.6 11)
		\move(1.4 9)\lvec(1.4 11)
		
		\move(0.6 12)\lvec(0.6 14)
		\move(1.4 12)\lvec(1.4 14)
		
		\move(0.6 0)\lvec(0.6 8)
		\move(1.4 0)\lvec(1.4 8)
		\htext(1.05 1){$\cdots$}
		\htext(1.05 3){$\cdots$}
		\htext(1.05 5){$\cdots$}
		\htext(1.05 7){$\cdots$}
		\vtext(1.05 8.5){$\cdots$}
		\htext(1.05 10){$\cdots$}
		\vtext(1.05 11.5){$\cdots$}
		\htext(1.05 13){$\cdots$}
\htext(0.2 1.6){$0$}
\htext(1.6 1.6){$0$}
\htext(0.4 0.4){$0$}
\htext(1.8 0.4){$0$}
\move(0 0)\lvec(0.6 2)
\move(1.4 0)\lvec(2 2)
		\htext(0.3 3){$1$}
		\htext(1.7 3){$1$}
		\htext(0.3 5){$2$}
		\htext(1.7 5){$2$}
		\htext(0.3 7){$3$}
		\htext(1.7 7){$3$}
		\htext(0.3 10){$n$\!-\!$1$}
		\htext(1.7 10){$n$\!-\!$1$}
		\htext(0.3 13){$i$\!+\!$1$}
		\htext(1.7 13){$i$\!+\!$1$}
		\htext(1 -0.3){$\underbrace{\rule{3.5em}{0em}}$}
		\htext(1 -0.7){$\bar{x}_{i\!+\! 1}$}
		\esegment
		
		\move(21 0)
		\bsegment
		\move(0 0)\lvec(2 0)
		\move(0 0)\lvec(0 2)
		\move(0 2)\lvec(2 2)
		\move(0 2)\lvec(0 16)
		\move(2 0)\lvec(2 16)
		\move(0 16)\lvec(2 16)
		
		\move(0 14)\lvec(2 14)
		\move(0 12)\lvec(2 12)
		\move(0 11)\lvec(2 11)
		\move(0 9)\lvec(2 9)
		\move(0 8)\lvec(2 8)
		\move(0 6)\lvec(2 6)
		\move(0 4)\lvec(2 4)
		
		\move(0.6 9)\lvec(0.6 11)
		\move(1.4 9)\lvec(1.4 11)
		
		\move(0.6 12)\lvec(0.6 16)
		\move(1.4 12)\lvec(1.4 16)
		
		\move(0.6 0)\lvec(0.6 8)
		\move(1.4 0)\lvec(1.4 8)
		\htext(1.05 1){$\cdots$}
		\htext(1.05 3){$\cdots$}
		\htext(1.05 5){$\cdots$}
		\htext(1.05 7){$\cdots$}
		\vtext(1.05 8.5){$\cdots$}
		\htext(1.05 10){$\cdots$}
		\vtext(1.05 11.5){$\cdots$}
		\htext(1.05 13){$\cdots$}
		\htext(1.05 15){$\cdots$}
\htext(0.2 1.6){$0$}
\htext(1.6 1.6){$0$}
\htext(0.4 0.4){$0$}
\htext(1.8 0.4){$0$}
\move(0 0)\lvec(0.6 2)
\move(1.4 0)\lvec(2 2)
		\htext(0.3 3){$1$}
		\htext(1.7 3){$1$}
		\htext(0.3 5){$2$}
		\htext(1.7 5){$2$}
		\htext(0.3 7){$3$}
		\htext(1.7 7){$3$}
		\htext(0.3 10){$n$\!-\!$1$}
		\htext(1.7 10){$n$\!-\!$1$}
		\htext(0.3 13){$i$\!+\!$1$}
		\htext(1.7 13){$i$\!+\!$1$}
		\htext(0.3 15){$i$}
		\htext(1.7 15){$i$}
		\htext(2.55 1){\fontsize{8}{8}\selectfont$\cdots$}
		\htext(1 -0.3){$\underbrace{\rule{3.5em}{0em}}$}
		\htext(1 -0.7){$\bar{x}_{i}$}
		\esegment
		
		\move(24 0)
		\bsegment
		\move(0 0)\lvec(2 0)
		\move(0 0)\lvec(0 2)
		\move(0 2)\lvec(2 2)
		\move(0 2)\lvec(0 18)
		\move(0 4)\lvec(2 4)
		\move(0 6)\lvec(2 6)
		\move(0 8)\lvec(2 8)
		\move(0 9)\lvec(2 9)
		\move(0 11)\lvec(2 11)
		\move(0 13)\lvec(2 13)
		\move(0 15)\lvec(2 15)
		\move(0 16)\lvec(2 16)
		\move(0 18)\lvec(2 18)
		\move(0.6 0)\lvec(0.6 8)
		\move(1.4 0)\lvec(1.4 8)
		\move(0.6 9)\lvec(0.6 15)
		\move(1.4 9)\lvec(1.4 15)
		\move(0.6 16)\lvec(0.6 18)
		\move(1.4 16)\lvec(1.4 18)
		\htext(1.05 1){$\cdots$}
		\htext(1.05 3){$\cdots$}
		\htext(1.05 5){$\cdots$}
		\htext(1.05 7){$\cdots$}
		\vtext(1.05 8.5){$\cdots$}
		\htext(1.05 10){$\cdots$}
		\htext(1.05 12){$\cdots$}
		\htext(1.05 14){$\cdots$}
		\vtext(1.05 15.5){$\cdots$}
		\htext(1.05 17){$\cdots$}
	\htext(0.2 1.6){$0$}
	\htext(1.6 1.6){$0$}
	\htext(0.4 0.4){$0$}
	\htext(1.8 0.4){$0$}
	\move(0 0)\lvec(0.6 2)
	\move(1.4 0)\lvec(2 2)
		\htext(0.3 3){$1$}
		\htext(1.7 3){$1$}
		\htext(0.3 5){$2$}
		\htext(1.7 5){$2$}
		\htext(0.3 7){$3$}
		\htext(1.7 7){$3$}
		\htext(0.3 10){$n$\!-\!$1$}
		\htext(1.7 10){$n$\!-\!$1$}
		\htext(0.3 12){$n$}
		\htext(1.7 12){$n$}
		\htext(0.3 14){$n$\!-\!$1$}
		\htext(1.7 14){$n$\!-\!$1$}
		\htext(0.3 17){$3$}
		\htext(1.7 17){$3$}
		\htext(1 -0.3){$\underbrace{\rule{3.5em}{0em}}$}
		\htext(1 -0.7){$\bar{x}_{3}$}
		\esegment
		
		\move(26 0)
		\bsegment
		\move(0 0)\lvec(2 0)
		\move(0 0)\lvec(0 2)
		\move(0 2)\lvec(2 2)
		\move(0 2)\lvec(0 20)
		\move(0 4)\lvec(2 4)
		\move(0 6)\lvec(2 6)
		\move(0 8)\lvec(2 8)
		\move(0 9)\lvec(2 9)
		\move(0 11)\lvec(2 11)
		\move(0 13)\lvec(2 13)
		\move(0 15)\lvec(2 15)
		\move(0 16)\lvec(2 16)
		\move(0 18)\lvec(2 18)
		\move(0 20)\lvec(2 20)
		
		\move(0.6 0)\lvec(0.6 8)
		\move(1.4 0)\lvec(1.4 8)
		\move(0.6 9)\lvec(0.6 15)
		\move(1.4 9)\lvec(1.4 15)
		\move(0.6 16)\lvec(0.6 20)
		\move(1.4 16)\lvec(1.4 20)
		\htext(1.05 1){$\cdots$}
		\htext(1.05 3){$\cdots$}
		\htext(1.05 5){$\cdots$}
		\htext(1.05 7){$\cdots$}
		\vtext(1.05 8.5){$\cdots$}
		\htext(1.05 10){$\cdots$}
		\htext(1.05 12){$\cdots$}
		\htext(1.05 14){$\cdots$}
		\vtext(1.05 15.5){$\cdots$}
		\htext(1.05 17){$\cdots$}
		\htext(1.05 19){$\cdots$}
	\htext(0.2 1.6){$0$}
	\htext(1.6 1.6){$0$}
	\htext(0.4 0.4){$0$}
	\htext(1.8 0.4){$0$}
	\move(0 0)\lvec(0.6 2)
	\move(1.4 0)\lvec(2 2)
		\htext(0.3 3){$1$}
		\htext(1.7 3){$1$}
		\htext(0.3 5){$2$}
		\htext(1.7 5){$2$}
		\htext(0.3 7){$3$}
		\htext(1.7 7){$3$}
		\htext(0.3 10){$n$\!-\!$1$}
		\htext(1.7 10){$n$\!-\!$1$}
		\htext(0.3 12){$n$}
		\htext(1.7 12){$n$}
		\htext(0.3 14){$n$\!-\!$1$}
		\htext(1.7 14){$n$\!-\!$1$}
		\htext(0.3 17){$3$}
		\htext(1.7 17){$3$}
		\htext(0.3 19){$2$}
		\htext(1.7 19){$2$}
		\htext(1 -0.3){$\underbrace{\rule{3.5em}{0em}}$}
		\htext(1 -0.7){$\bar{x}_{2}$}
		\esegment
		
		\move(28 0)
		\bsegment
		\move(0 0)\lvec(2 0)
		\move(0 0)\lvec(0 2)
		\move(0 2)\lvec(2 2)
		\move(2 0)\lvec(2 22)
		\move(0 2)\lvec(0 22)
		\move(0 22)\lvec(2 22)
		
		\move(0 4)\lvec(2 4)
		\move(0 6)\lvec(2 6)
		\move(0 8)\lvec(2 8)
		\move(0 9)\lvec(2 9)
		\move(0 11)\lvec(2 11)
		\move(0 13)\lvec(2 13)
		\move(0 15)\lvec(2 15)
		\move(0 16)\lvec(2 16)
		\move(0 18)\lvec(2 18)
		\move(0 20)\lvec(2 20)
		
		\move(0.6 0)\lvec(0.6 8)
		\move(1.4 0)\lvec(1.4 8)
		\move(0.6 9)\lvec(0.6 15)
		\move(1.4 9)\lvec(1.4 15)
		\move(0.6 16)\lvec(0.6 22)
		\move(1.4 16)\lvec(1.4 22)
		\htext(1.05 1){$\cdots$}
		\htext(1.05 3){$\cdots$}
		\htext(1.05 5){$\cdots$}
		\htext(1.05 7){$\cdots$}
		\vtext(1.05 8.5){$\cdots$}
		\htext(1.05 10){$\cdots$}
		\htext(1.05 12){$\cdots$}
		\htext(1.05 14){$\cdots$}
		\vtext(1.05 15.5){$\cdots$}
		\htext(1.05 17){$\cdots$}
		\htext(1.05 19){$\cdots$}
		\htext(1.05 21){$\cdots$}
		\htext(0.2 1.6){$0$}
		\htext(1.6 1.6){$0$}
		\htext(0.4 0.4){$0$}
		\htext(1.8 0.4){$0$}
		\move(0 0)\lvec(0.6 2)
		\move(1.4 0)\lvec(2 2)
	
		\htext(0.3 3){$1$}
		\htext(1.7 3){$1$}
		\htext(0.3 5){$2$}
		\htext(1.7 5){$2$}
		\htext(0.3 7){$3$}
		\htext(1.7 7){$3$}
		\htext(0.3 10){$n$\!-\!$1$}
		\htext(1.7 10){$n$\!-\!$1$}
		\htext(0.3 12){$n$}
		\htext(1.7 12){$n$}
		\htext(0.3 14){$n$\!-\!$1$}
		\htext(1.7 14){$n$\!-\!$1$}
		\htext(0.3 17){$3$}
		\htext(1.7 17){$3$}
		\htext(0.3 19){$2$}
		\htext(1.7 19){$2$}
		\htext(0.3 21){$1$}
		\htext(1.7 21){$1$}
		\htext(1 -0.3){$\underbrace{\rule{3.5em}{0em}}$}
		\htext(1 -0.7){$\bar{x}_{1}$}
		\esegment
	\end{texdraw}
\end{center}
\end{enumerate}

\vskip 2mm

For convenience, we denote 

\begin{equation*}
	*=	
	\begin{cases}
0\ \ \text{for types $A_{2n}^{(2)}$, $D_{n+1}^{(2)}$ and $C_{n}^{(1)}$,}\\
1\ \ \text{for types $A_{2n-1}^{(2)}$, $D_{n}^{(1)}$ and $B_{n}^{(1)}$.}
		\end{cases}
	\end{equation*}

\vskip 2mm

The level-$l$ reduced Young columns defined above should satisfy the following conditions:

\vskip 3mm 

\begin{itemize}
\item For types $A_{2n}^{(2)}\ (1\le i\le n)$, $D_{n+1}^{(2)}\ (1\le i\le n-1)$,  $A_{2n-1}^{(2)}\ (3\le i\le n)$, $D_{n}^{(1)}\ (3\le i\le n-2)$,  $B_{n}^{(1)}\ (3\le i\le n-1)$ and $C_{n}^{(1)}\ (1\le i\le n)$, we choose a part of level-$l$ reduced Young columns 3.2\eqref{(1)}-3.2\eqref{(6)} as in Picture  (\romannumeral1). We could not add $(i-1)$-block or remove $(i-1)$-block on this part if $x_i=\bar{x}_i$.
\vspace{7pt}
\item For types $A_{2n-1}^{(2)}$, $D_{n}^{(1)}$ and $B_{n}^{(1)}$, we choose a part of level-$l$ reduced Young columns 3.2\eqref{(3)}-3.2\eqref{(5)} as in Picture (\romannumeral2). We could not add $0$-block or remove $0$-block, and also not add $1$-block or remove $1$-block on this part if $x_2=\bar{x}_2$.
\vspace{7pt}
\item For types $D_{n+1}^{(2)}$ and $B_{n}^{(1)}$, we choose a part of level-$l$ reduced Young columns 3.2\eqref{(2)} and 3.2\eqref{(5)} as in Picture (\romannumeral3). We could not add $(n-1)$-block or remove $(n-1)$-block on this part if $x_n=\bar{x}_n$.
\end{itemize}

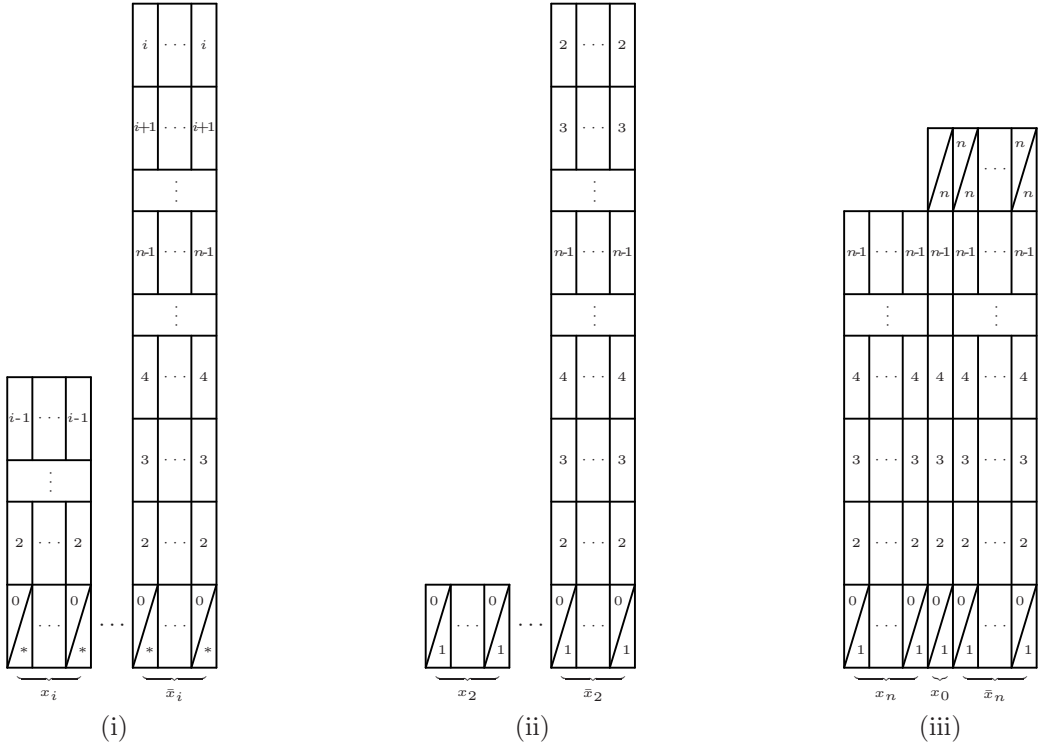
\begin{center}
	\begin{texdraw}
		\fontsize{3}{3}\selectfont
		\drawdim em
		\setunitscale 2.3
		
		\move(0 0)
		\bsegment
		\move(0 0)
		\bsegment
		\move(0 0)\lvec(2 0)
		\move(0 0)\lvec(0 2)
		\move(0 2)\lvec(2 2)
		\move(0 2)\lvec(0 7)
		\move(0 7)\lvec(2 7)
		\move(0 4)\lvec(2 4)
		\move(0 5)\lvec(2 5)
		\move(0.6 0)\lvec(0.6 4)
		\move(1.4 0)\lvec(1.4 4)
		\move(0.6 5)\lvec(0.6 7)
		\move(1.4 5)\lvec(1.4 7)
		\htext(1.05 1){$\cdots$}
		\htext(1.05 3){$\cdots$}
		\vtext(1.05 4.5){$\cdots$}
		\htext(1.05 6){$\cdots$}
		\htext(0.2 1.6){$0$}
		\htext(1.6 1.6){$0$}
		\htext(0.4 0.4){$*$}
		\htext(1.8 0.4){$*$}
		\move(0 0)\lvec(0.6 2)
		\move(1.4 0)\lvec(2 2)	
		\htext(0.3 3){$2$}
		\htext(1.7 3){$2$}
		\htext(0.3 6){$i$-$1$}
		\htext(1.7 6){$i$-$1$}
		\htext(1 -0.3){$\underbrace{\rule{3.5em}{0em}}$}
		
		\htext(1 -0.7){$x_i$}
		\move(2 0)\lvec(2 7)
		\htext(2.55 1){\fontsize{8}{8}\selectfont$\cdots$}
		\htext(2.55 -1.5){\fontsize{10}{10}\selectfont{(\romannumeral1)}}
		\esegment
		\move(3 0)
		\bsegment
		\move(0 0)\lvec(2 0)
		\move(0 0)\lvec(0 2)
		\move(0 2)\lvec(2 2)
		\move(0 2)\lvec(0 16)
		\move(2 0)\lvec(2 16)
		\move(0 16)\lvec(2 16)
		
		\move(0 14)\lvec(2 14)
		\move(0 12)\lvec(2 12)
		\move(0 11)\lvec(2 11)
		\move(0 9)\lvec(2 9)
		\move(0 8)\lvec(2 8)
		\move(0 6)\lvec(2 6)
		\move(0 4)\lvec(2 4)
		
		\move(0.6 9)\lvec(0.6 11)
		\move(1.4 9)\lvec(1.4 11)
		
		\move(0.6 12)\lvec(0.6 16)
		\move(1.4 12)\lvec(1.4 16)
		\htext(0.2 1.6){$0$}
		\htext(1.6 1.6){$0$}
		\htext(0.4 0.4){$*$}
		\htext(1.8 0.4){$*$}
		\move(0 0)\lvec(0.6 2)
		\move(1.4 0)\lvec(2 2)		
		\move(0.6 0)\lvec(0.6 8)
		\move(1.4 0)\lvec(1.4 8)
		\htext(1.05 1){$\cdots$}
		\htext(1.05 3){$\cdots$}
		\htext(1.05 5){$\cdots$}
		\htext(1.05 7){$\cdots$}
		\vtext(1.05 8.5){$\cdots$}
		\htext(1.05 10){$\cdots$}
		\vtext(1.05 11.5){$\cdots$}
		\htext(1.05 13){$\cdots$}
		\htext(1.05 15){$\cdots$}
		
		\htext(0.3 3){$2$}
		\htext(1.7 3){$2$}
		\htext(0.3 5){$3$}
		\htext(1.7 5){$3$}
		\htext(0.3 7){$4$}
		\htext(1.7 7){$4$}
		\htext(0.3 10){$n$\!-\!$1$}
		\htext(1.7 10){$n$\!-\!$1$}
		\htext(0.3 13){$i$\!+\!$1$}
		\htext(1.7 13){$i$\!+\!$1$}
		\htext(0.3 15){$i$}
		\htext(1.7 15){$i$}
		
		\htext(1 -0.3){$\underbrace{\rule{3.5em}{0em}}$}
		\htext(1 -0.7){$\bar{x}_{i}$}
		\esegment		
		\esegment

		\move(10 0)
		\bsegment	
		\move(0 0)
		\bsegment
		\move(0 0)\lvec(2 0)
		\move(0 0)\lvec(0 2)
		\move(0 2)\lvec(2 2)

		\move(0.6 0)\lvec(0.6 2)
		\move(1.4 0)\lvec(1.4 2)
		
		\htext(1.05 1){$\cdots$}
		
		\htext(0.2 1.6){$0$}
		\htext(1.6 1.6){$0$}
		\htext(0.4 0.4){$1$}
		\htext(1.8 0.4){$1$}
		\move(0 0)\lvec(0.6 2)
		\move(1.4 0)\lvec(2 2)

		\htext(1 -0.3){$\underbrace{\rule{3.5em}{0em}}$}
		\htext(1 -0.7){$x_2$}
		\move(2 0)\lvec(2 2)
		\htext(2.55 1){\fontsize{8}{8}\selectfont$\cdots$}
		\htext(2.55 -1.5){\fontsize{10}{10}\selectfont{(\romannumeral2)}}
		\esegment
		\move(3 0)
		\bsegment
		\move(0 0)\lvec(2 0)
		\move(0 0)\lvec(0 2)
		\move(0 2)\lvec(2 2)
		\move(0 2)\lvec(0 16)
		\move(2 0)\lvec(2 16)
		\move(0 16)\lvec(2 16)
		
		\move(0 14)\lvec(2 14)
		\move(0 12)\lvec(2 12)
		\move(0 11)\lvec(2 11)
		\move(0 9)\lvec(2 9)
		\move(0 8)\lvec(2 8)
		\move(0 6)\lvec(2 6)
		\move(0 4)\lvec(2 4)
		
		\move(0.6 9)\lvec(0.6 11)
		\move(1.4 9)\lvec(1.4 11)
		
		\move(0.6 12)\lvec(0.6 16)
		\move(1.4 12)\lvec(1.4 16)
		\htext(0.2 1.6){$0$}
		\htext(1.6 1.6){$0$}
		\htext(0.4 0.4){$1$}
		\htext(1.8 0.4){$1$}
		\move(0 0)\lvec(0.6 2)
		\move(1.4 0)\lvec(2 2)		
		\move(0.6 0)\lvec(0.6 8)
		\move(1.4 0)\lvec(1.4 8)
		\htext(1.05 1){$\cdots$}
		\htext(1.05 3){$\cdots$}
		\htext(1.05 5){$\cdots$}
		\htext(1.05 7){$\cdots$}
		\vtext(1.05 8.5){$\cdots$}
		\htext(1.05 10){$\cdots$}
		\vtext(1.05 11.5){$\cdots$}
		\htext(1.05 13){$\cdots$}
		\htext(1.05 15){$\cdots$}
		
		\htext(0.3 3){$2$}
		\htext(1.7 3){$2$}
		\htext(0.3 5){$3$}
		\htext(1.7 5){$3$}
		\htext(0.3 7){$4$}
		\htext(1.7 7){$4$}
		\htext(0.3 10){$n$\!-\!$1$}
		\htext(1.7 10){$n$\!-\!$1$}
		\htext(0.3 13){$3$}
		\htext(1.7 13){$3$}
		\htext(0.3 15){$2$}
		\htext(1.7 15){$2$}
		
		\htext(1 -0.3){$\underbrace{\rule{3.5em}{0em}}$}
		\htext(1 -0.7){$\bar{x}_{2}$}
		\esegment
		\esegment

		\move(20 0)
		\bsegment	
		\move(0 0)
		\bsegment
		\move(0 0)\lvec(2 0)
		\move(0 0)\lvec(0 2)
		\move(0 2)\lvec(2 2)
		\move(0 2)\lvec(0 11)
		\move(0 11)\lvec(2 11)
		\move(2 0)\lvec(2 13)
		\move(2 0)\lvec(2.6 0)
		\move(2 2)\lvec(2.6 2)
		\move(2 4)\lvec(2.6 4)
		\move(2 6)\lvec(2.6 6)
		\move(2 8)\lvec(2.6 8)
		\move(2 9)\lvec(2.6 9)
		\move(2 11)\lvec(2.6 11)
		\move(2 13)\lvec(2.6 13)
		\move(2 0)\lvec(2.6 2)
		\htext(2.2 1.6){$0$}
		\htext(2.4 0.4){$1$}
		\htext(2.3 3){$2$}
		\htext(2.3 5){$3$}
		\htext(2.3 7){$4$}
		\htext(2.3 10){$n$\!-\!$1$}		
		\move(2 11)\lvec(2.6 13)
		\htext(2.4 11.4){$n$}

		\htext(0.2 1.6){$0$}
		\htext(1.6 1.6){$0$}
		\htext(0.4 0.4){$1$}
		\htext(1.8 0.4){$1$}
		\move(0 0)\lvec(0.6 2)
		\move(1.4 0)\lvec(2 2)		
		\move(0 4)\lvec(2 4)
		\move(0 6)\lvec(2 6)
		\move(0 8)\lvec(2 8)
		\move(0 9)\lvec(2 9)
		\move(0.6 0)\lvec(0.6 8)
		\move(1.4 0)\lvec(1.4 8)
		\move(0.6 9)\lvec(0.6 11)
		\move(1.4 9)\lvec(1.4 11)
		
		\htext(1.05 1){$\cdots$}
		\htext(1.05 3){$\cdots$}
		\htext(1.05 5){$\cdots$}
		\htext(1.05 7){$\cdots$}
		\vtext(1.05 8.5){$\cdots$}
		\htext(1.05 10){$\cdots$}

		\htext(0.3 3){$2$}
		\htext(1.7 3){$2$}
		\htext(0.3 5){$3$}
		\htext(1.7 5){$3$}
		\htext(0.3 7){$4$}
		\htext(1.7 7){$4$}
		\htext(0.3 10){$n$\!-\!$1$}
		\htext(1.7 10){$n$\!-\!$1$}
		\htext(1 -0.3){$\underbrace{\rule{3.5em}{0em}}$}
		\htext(1 -0.7){$x_{n}$}
		\htext(2.3 -0.3){$\underbrace{\rule{0em}{0em}}$}
		\htext(2.3 -0.7){$x_{0}$}
		\htext(2.3 -1.5){\fontsize{10}{10}\selectfont{(\romannumeral3)}}	
		\esegment
		
		\move(2.6 0)
		\bsegment
		\move(0 0)\lvec(2 0)
		\move(0 0)\lvec(0 2)
		\move(0 2)\lvec(2 2)
		\move(0 2)\lvec(0 13)
		\move(2 0)\lvec(2 13)
		\move(0 13)\lvec(2 13)
		\htext(0.2 1.6){$0$}
		\htext(1.6 1.6){$0$}
		\htext(0.4 0.4){$1$}
		\htext(1.8 0.4){$1$}
		\move(0 0)\lvec(0.6 2)
		\move(1.4 0)\lvec(2 2)		
		\move(0 4)\lvec(2 4)
		\move(0 6)\lvec(2 6)
		\move(0 8)\lvec(2 8)
		\move(0 9)\lvec(2 9)
		\move(0 11)\lvec(2 11)
		
		\move(0.6 0)\lvec(0.6 8)
		\move(1.4 0)\lvec(1.4 8)
		
		\move(0.6 9)\lvec(0.6 13)
		\move(1.4 9)\lvec(1.4 13)
		
		\htext(1.05 1){$\cdots$}
		\htext(1.05 3){$\cdots$}
		\htext(1.05 5){$\cdots$}
		\htext(1.05 7){$\cdots$}
		\vtext(1.05 8.5){$\cdots$}
		\htext(1.05 10){$\cdots$}
		\htext(1.05 12){$\cdots$}

		\htext(0.3 3){$2$}
		\htext(1.7 3){$2$}
		\htext(0.3 5){$3$}
		\htext(1.7 5){$3$}
		\htext(0.3 7){$4$}
		\htext(1.7 7){$4$}
		\htext(0.3 10){$n$\!-\!$1$}
		\htext(1.7 10){$n$\!-\!$1$}
		
		\htext(1 -0.3){$\underbrace{\rule{3.5em}{0em}}$}
		\htext(1 -0.7){$\bar{x}_{n}$}
		
		\move(0 11)\lvec(0.6 13)
		\htext(0.4 11.4){$n$}
		\move(1.4 11)\lvec(2 13)
		\htext(1.8 11.4){$n$}
		\htext(0.2 12.6){$n$}
		\htext(1.6 12.6){$n$}
		\esegment
		\esegment
		
	\end{texdraw}	
\end{center}

\vskip 2mm 

We need to make some special rules for the Young columns of type $C_n^{(1)}$.  When we count the number of $0$-blocks which can be added or removed, we should treat two  $0$-blocks of shape $\frac{1}{2l}\times \frac{1}{2}\times 1$ as a whole block.

\vskip 3mm

\begin{example}
A level-$15$ reduced Young column	 of type $B_4^{(1)}$ is given below.	

\vskip 3mm

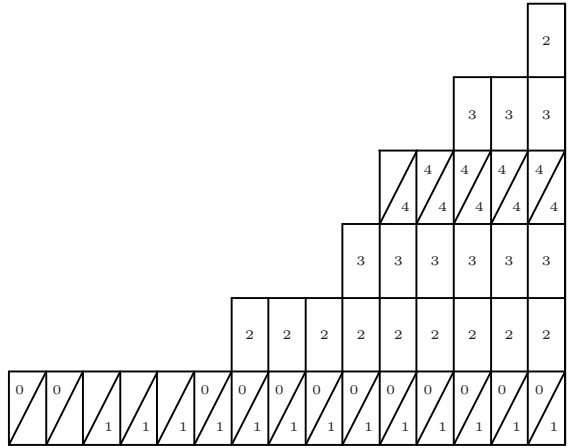
\begin{center}
	\begin{texdraw}
		
		\fontsize{5}{5}\selectfont
		\drawdim em
		\setunitscale 2.3
		
		\move(-2 0)\rlvec(-15 0)	
		
		\move(-2 2)\rlvec(-15 0)
		
		\move(-2 0 )\rlvec(0 12)
		\move(-2 0 )\rlvec(0 12)
		\move(-3 0 )\rlvec(0 10)
		\move(-4 0 )\rlvec(0 10)
		\move(-5 0 )\rlvec(0 4)
		\move(-6 0 )\rlvec(0 8)
		\move(-7 0 )\rlvec(0 8)
		\move(-8 0 )\rlvec(0 6)
		\move(-9 0 )\rlvec(0 4)
		\move(-10 0 )\rlvec(0 4)
		\move(-11 0 )\rlvec(0 2)
		\move(-12 0 )\rlvec(0 2)
		\move(-13 0 )\rlvec(0 2)
		\move(-14 0 )\rlvec(0 2)
		\move(-15 0 )\rlvec(0 2)
		
		\move(-3 10)\rlvec(-1 0)
		
		\move(-3 12)\lvec(-2 12)
		\move(-3 10)\lvec(-2 10)	
		\move(-3 8)\lvec(-2 8)
		\move(-3 6)\lvec(-2 6)
		
		\move(-6 8)\rlvec(-1 0)

		\move(-2 4)\rlvec(-8 0)
		\move(-6 6)\rlvec(-2 0)
		\move(-3 6)\rlvec(-1 0)
		\move(-3 8)\rlvec(-1 0)

		\move(-3 0)\rlvec(1 2)
		\move(-3 10)\rlvec(0 2)
		\move(-5 8)\lvec(-5 10)
		\move(-5 10)\lvec(-4 10)
		\htext(-4.5 9){$3$}	
		\move(-6 6)\rlvec(1 0)
		\move(-6 8)\rlvec(1 0)
		\move(-6 6)\rlvec(1 2)
		\htext(-5.5 5){$3$}

		\move(-4 0)\rlvec(1 2)
		\move(-5 0)\rlvec(1 2)
		\move(-5 4)\rlvec(0 4)
		\move(-4 8)\rlvec(-1 0)
		\move(-4 6)\rlvec(-1 0)
		\htext(-4.5 5){$3$}
		\move(-5 6)\rlvec(1 2)
		\htext(-4.3 6.5){$4$}
		\htext(-4.7 7.5){$4$}
		\htext(-5.3 6.5){$4$}
		\htext(-5.7 7.5){$4$}
		
		\move(-6 0)\rlvec(1 2)
		\move(-7 0)\rlvec(1 2)
		\move(-8 0)\rlvec(1 2)
		\move(-9 0)\rlvec(1 2)
		\move(-10 0)\rlvec(1 2)
		\move(-11 0)\rlvec(1 2)
		\move(-12 0)\rlvec(1 2)
		\move(-13 0)\rlvec(1 2)
		\move(-14 0)\rlvec(1 2)
		\move(-15 0)\rlvec(1 2)
		\move(-7 6)\rlvec(1 2)
		
		\move(-4 6)\rlvec(1 2)
		
		\move(-11 2)\lvec(-11 4)
		\move(-11 4)\lvec(-10 4)
		\move(-16 0)\lvec(-16 2)
		\move(-17 0)\lvec(-17 2)
		\move(-17 0)\lvec(-16 2)
		\move(-16 0)\lvec(-15 2)
		\htext(-10.5 3){$2$}
		\htext(-2.3 0.5){$1$}
		\htext(-2.7 1.5){$0$}
		\htext(-2.5 3){$2$}	
		\htext(-2.5 5){$3$}
		\move(-3 6)\lvec(-2 8)	
		\htext(-2.3 6.5){$4$}
		\htext(-2.7 7.5){$4$}	
		\htext(-2.5 9){$3$}	
		\htext(-2.5 11){$2$}				
		\htext(-3.3 0.5){$1$}
		\htext(-3.7 1.5){$0$}
		\htext(-3.5 3){$2$}
		\htext(-3.5 5){$3$}	
		\htext(-3.3 6.5){$4$}
		\htext(-3.7 7.5){$4$}
		\htext(-3.5 9){$3$}
		
		\htext(-4.3 0.5){$1$}
		\htext(-4.7 1.5){$0$}
		\htext(-4.5 3){$2$}
		
		\htext(-5.3 0.5){$1$}
		\htext(-5.7 1.5){$0$}
		\htext(-5.5 3){$2$}
		
		\htext(-6.3 0.5){$1$}
		\htext(-6.7 1.5){$0$}
		\htext(-6.5 3){$2$}
		\htext(-6.5 5){$3$}
		\htext(-6.3 6.5){$4$}

		\htext(-7.3 0.5){$1$}
		\htext(-7.7 1.5){$0$}
		\htext(-7.5 3){$2$}
		\htext(-7.5 5){$3$}

		\htext(-8.3 0.5){$1$}
		\htext(-8.7 1.5){$0$}
		\htext(-8.5 3){$2$}

		\htext(-9.3 0.5){$1$}
		\htext(-9.7 1.5){$0$}
		\htext(-9.5 3){$2$}
		
		\htext(-11.7 1.5){$0$}
		\htext(-10.3 0.5){$1$}
		\htext(-11.3 0.5){$1$}
		\htext(-12.3 0.5){$1$}
		\htext(-13.3 0.5){$1$}
		\htext(-14.3 0.5){$1$}
		\htext(-10.7 1.5){$0$}
		
		\htext(-15.7 1.5){$0$}
		\htext(-16.7 1.5){$0$}
	\end{texdraw}
\end{center}
\end{example}

Note that the Young column in Example \ref{level-15 Young Column} is not reduced.

\vskip 3mm 

\subsection{The crystal structure on Young columns}\label{The crystal structure}

\hfill

\vskip 2mm 

We denote the set of level-$l$ reduced Young columns of types $A_{2n}^{(2)}$, $D_{n+1}^{(2)}$, $A_{2n-1}^{(2)}$, $D_n^{(1)}$, $B_n^{(1)}$ and $C_n^{(1)}$ by ${\mathcal C}^{(l)}$ uniformly. We shall define the crystal structure on the set ${\mathcal C}^{(l)}$.

\vskip 3mm 

We define the maps $\tilde{\mathfrak f}_i$ ($i\in I$) from ${\mathcal C}^{(l)}$ to ${\mathcal C}^{(l)}$ by adding an $i$-block on level-$l$ reduced Young  columns. We list the definition of $\tilde{\mathfrak f}_i$ for each type as follows:

\vskip 3mm

\begin{enumerate}
	\item\label{1}  For types $A_{2n}^{(2)}\ (1\le i\le n-1)$, $D_{n+1}^{(2)}\ (1\le i\le n-1)$, $C_{n}^{(1)}\ (1\le i\le n-1)$, $A_{2n-1}^{(2)}\ (1< i\le n-1)$, $D_{n}^{(1)}\ (1< i\le n-2)$ and $B_{n}^{(1)}\ (1< i\le n-1)$, 
	 
\vskip 3mm

if $x_{i+1}\ge \bar{x}_{i+1}$, the map $\tilde{\mathfrak f}_i$ is given by

\begin{center}
	\begin{texdraw}
		\fontsize{3}{3}\selectfont
		\drawdim em
		\setunitscale 2.3
		\arrowheadsize l:0.3 w:0.3
		\arrowheadtype t:V
		\move(0 0)
		\bsegment
		\move(0 0)
		\bsegment
		\move(0 0)\lvec(2 0)
		\move(0 0)\lvec(0 2)
		\move(0 2)\lvec(2 2)
		\move(0 2)\lvec(0 7)
		\move(0 7)\lvec(2 7)
		\move(0 4)\lvec(2 4)
		\move(0 5)\lvec(2 5)
		\move(0.6 0)\lvec(0.6 4)
		\move(1.4 0)\lvec(1.4 4)
		\move(0.6 5)\lvec(0.6 7)
		\move(1.4 5)\lvec(1.4 7)
		\htext(1.05 1){$\cdots$}
		\htext(1.05 3){$\cdots$}
		\vtext(1.05 4.5){$\cdots$}
		\htext(1.05 6){$\cdots$}
		\htext(0.2 1.6){$0$}
		\htext(1.6 1.6){$0$}
		\htext(0.4 0.4){$*$}
		\htext(1.8 0.4){$*$}
		\move(0 0)\lvec(0.6 2)
		\move(1.4 0)\lvec(2 2)	
		\htext(0.3 3){$*$\!+\!$1$}
		\htext(1.7 3){$*$\!+\!$1$}
		\htext(0.3 6){$i$-$1$}
		\htext(1.7 6){$i$-$1$}
		\htext(1 -0.3){$\underbrace{\rule{4.5em}{0em}}$}
		\htext(1 -0.7){$x_i$}
		
		\esegment
		
		\move(2 0)
		\bsegment
		\move(0 0)\lvec(2 0)
		\move(0 0)\lvec(0 2)
		\move(0 2)\lvec(2 2)
		\move(0 2)\lvec(0 9)
		\move(2 0)\lvec(2 9)
		\move(0 9)\lvec(2 9)
		\move(0 5)\lvec(2 5)
		\move(0 4)\lvec(2 4)
		\move(0 7)\lvec(2 7)
		\move(0.6 0)\lvec(0.6 4)
		\move(1.4 0)\lvec(1.4 4)
		\move(0.6 5)\lvec(0.6 9)
		\move(1.4 5)\lvec(1.4 9)
		\htext(1.05 1){$\cdots$}
		\htext(1.05 3){$\cdots$}
		\vtext(1.05 4.5){$\cdots$}
		\htext(1.05 6){$\cdots$}
		\htext(1.05 8){$\cdots$}
		\htext(0.2 1.6){$0$}
		\htext(1.6 1.6){$0$}
		\htext(0.4 0.4){$*$}
		\htext(1.8 0.4){$*$}
		\move(0 0)\lvec(0.6 2)
		\move(1.4 0)\lvec(2 2)
		\htext(0.3 3){$*$\!+\!$1$}
		\htext(1.7 3){$*$\!+\!$1$}
		\htext(0.3 6){$i$-$1$}
		\htext(1.7 6){$i$-$1$}
		\htext(0.3 8){$i$}
		\htext(1.7 8){$i$}
		
		\htext(1 -0.3){$\underbrace{\rule{4.5em}{0em}}$}
		\htext(1 -0.7){$x_{i \!+\! 1}$}
		
		\move(2.5 4.5)\avec(4.5 4.5)
		\htext(3.5 5.3){	\fontsize{12}{12}\selectfont$\tilde{\mathfrak f}_i$}
		\esegment
		\esegment

		\move(7 0)
		\bsegment
		\move(0 0)
		\bsegment
		\move(0 0)\lvec(2 0)
		\move(0 0)\lvec(0 2)
		\move(0 2)\lvec(2 2)
		\move(0 2)\lvec(0 7)
		\move(0 7)\lvec(2 7)
		\move(0 4)\lvec(2 4)
		\move(0 5)\lvec(2 5)
		\move(0.6 0)\lvec(0.6 4)
		\move(1.4 0)\lvec(1.4 4)
		\move(0.6 5)\lvec(0.6 7)
		\move(1.4 5)\lvec(1.4 7)
		\htext(1.05 1){$\cdots$}
		\htext(1.05 3){$\cdots$}
		\vtext(1.05 4.5){$\cdots$}
		\htext(1.05 6){$\cdots$}
		\htext(0.2 1.6){$0$}
		\htext(1.6 1.6){$0$}
		\htext(0.4 0.4){$*$}
		\htext(1.8 0.4){$*$}
		\move(0 0)\lvec(0.6 2)
		\move(1.4 0)\lvec(2 2)	
		\htext(0.3 3){$*$\!+\!$1$}
		\htext(1.7 3){$*$\!+\!$1$}
		\htext(0.3 6){$i$-$1$}
		\htext(1.7 6){$i$-$1$}
		\htext(1 -0.3){$\underbrace{\rule{4.5em}{0em}}$}
		\htext(1 -0.7){$x_i$}
		
		\move(1.4 7)\lvec(1.4 9)\lvec(2 9)\lvec(2 7)\lvec(1.4 7) \lfill f:0.8
		\htext(1.7 8){$i$}
		\esegment
		
		\move(2 0)
		\bsegment
		\move(0 0)\lvec(2 0)
		\move(0 0)\lvec(0 2)
		\move(0 2)\lvec(2 2)
		\move(0 2)\lvec(0 9)
		\move(2 0)\lvec(2 9)
		\move(0 9)\lvec(2 9)
		\move(0 5)\lvec(2 5)
		\move(0 4)\lvec(2 4)
		\move(0 7)\lvec(2 7)
		\move(0.6 0)\lvec(0.6 4)
		\move(1.4 0)\lvec(1.4 4)
		\move(0.6 5)\lvec(0.6 9)
		\move(1.4 5)\lvec(1.4 9)
		\htext(1.05 1){$\cdots$}
		\htext(1.05 3){$\cdots$}
		\vtext(1.05 4.5){$\cdots$}
		\htext(1.05 6){$\cdots$}
		\htext(1.05 8){$\cdots$}
		\htext(0.2 1.6){$0$}
		\htext(1.6 1.6){$0$}
		\htext(0.4 0.4){$*$}
		\htext(1.8 0.4){$*$}
		\move(0 0)\lvec(0.6 2)
		\move(1.4 0)\lvec(2 2)
		\htext(0.3 3){$*$\!+\!$1$}
		\htext(1.7 3){$*$\!+\!$1$}
		\htext(0.3 6){$i$-$1$}
		\htext(1.7 6){$i$-$1$}
		\htext(0.3 8){$i$}
		\htext(1.7 8){$i$}
		
		\htext(1 -0.3){$\underbrace{\rule{4.5em}{0em}}$}
		\htext(1 -0.7){$x_{i \!+\! 1}$}
		\esegment
		\esegment	
	\end{texdraw}
\end{center}

\vskip 2mm

if $x_{i+1}< \bar{x}_{i+1}$, the map $\tilde{\mathfrak f}_i$ is given by

\vskip 7mm

\begin{center}
	\begin{texdraw}
		\fontsize{3}{3}\selectfont
		\drawdim em
		\setunitscale 2.3
		\arrowheadsize l:0.3 w:0.3
		\arrowheadtype t:V
		\move(0 0)
		\bsegment
		\move(0 0)
		\bsegment
		\move(0 0)\lvec(2 0)
		\move(0 0)\lvec(0 2)
		\move(0 2)\lvec(2 2)
		\move(0 2)\lvec(0 7)
		\move(0 7)\lvec(2 7)
		\move(0 4)\lvec(2 4)
		\move(0 5)\lvec(2 5)
		\move(0.6 0)\lvec(0.6 4)
		\move(1.4 0)\lvec(1.4 4)
		\move(0.6 5)\lvec(0.6 7)
		\move(1.4 5)\lvec(1.4 7)
		\htext(1.05 1){$\cdots$}
		\htext(1.05 3){$\cdots$}
		\vtext(1.05 4.5){$\cdots$}
		\htext(1.05 6){$\cdots$}
		\htext(0.2 1.6){$0$}
		\htext(1.6 1.6){$0$}
		\htext(0.4 0.4){$*$}
		\htext(1.8 0.4){$*$}
		\move(0 0)\lvec(0.6 2)
		\move(1.4 0)\lvec(2 2)	
		\htext(0.3 3){$*$\!+\!$1$}
		\htext(1.7 3){$*$\!+\!$1$}
		\htext(0.3 6){$i$\!+\!$1$}
		\htext(1.7 6){$i$\!+\!$1$}
		\htext(1 -0.3){$\underbrace{\rule{4.5em}{0em}}$}
		\htext(1 -0.7){$\bar{x}_{i\!+\!1}$}
		
		\esegment
		
		\move(2 0)
		\bsegment
		\move(0 0)\lvec(2 0)
		\move(0 0)\lvec(0 2)
		\move(0 2)\lvec(2 2)
		\move(0 2)\lvec(0 9)
		\move(2 0)\lvec(2 9)
		\move(0 9)\lvec(2 9)
		\move(0 5)\lvec(2 5)
		\move(0 4)\lvec(2 4)
		\move(0 7)\lvec(2 7)
		\move(0.6 0)\lvec(0.6 4)
		\move(1.4 0)\lvec(1.4 4)
		\move(0.6 5)\lvec(0.6 9)
		\move(1.4 5)\lvec(1.4 9)
		\htext(1.05 1){$\cdots$}
		\htext(1.05 3){$\cdots$}
		\vtext(1.05 4.5){$\cdots$}
		\htext(1.05 6){$\cdots$}
		\htext(1.05 8){$\cdots$}
		\htext(0.2 1.6){$0$}
		\htext(1.6 1.6){$0$}
		\htext(0.4 0.4){$*$}
		\htext(1.8 0.4){$*$}
		\move(0 0)\lvec(0.6 2)
		\move(1.4 0)\lvec(2 2)
		\htext(0.3 3){$*$\!+\!$1$}
		\htext(1.7 3){$*$\!+\!$1$}
		\htext(0.3 6){$i$\!+\!$1$}
		\htext(1.7 6){$i$\!+\!$1$}
		\htext(0.3 8){$i$}
		\htext(1.7 8){$i$}
		
		\htext(1 -0.3){$\underbrace{\rule{4.5em}{0em}}$}
		\htext(1 -0.7){$\bar{x}_i$}
		
		\move(2.5 4.5)\avec(4.5 4.5)
		\htext(3.5 5.3){	\fontsize{12}{12}\selectfont$\tilde{\mathfrak f}_i$}
		\esegment
		\esegment

		\move(7 0)
		\bsegment
		\move(0 0)
		\bsegment
		\move(0 0)\lvec(2 0)
		\move(0 0)\lvec(0 2)
		\move(0 2)\lvec(2 2)
		\move(0 2)\lvec(0 7)
		\move(0 7)\lvec(2 7)
		\move(0 4)\lvec(2 4)
		\move(0 5)\lvec(2 5)
		\move(0.6 0)\lvec(0.6 4)
		\move(1.4 0)\lvec(1.4 4)
		\move(0.6 5)\lvec(0.6 7)
		\move(1.4 5)\lvec(1.4 7)
		\htext(1.05 1){$\cdots$}
		\htext(1.05 3){$\cdots$}
		\vtext(1.05 4.5){$\cdots$}
		\htext(1.05 6){$\cdots$}
		\htext(0.2 1.6){$0$}
		\htext(1.6 1.6){$0$}
		\htext(0.4 0.4){$*$}
		\htext(1.8 0.4){$*$}
		\move(0 0)\lvec(0.6 2)
		\move(1.4 0)\lvec(2 2)	
		\htext(0.3 3){$*$\!+\!$1$}
		\htext(1.7 3){$*$\!+\!$1$}
		\htext(0.3 6){$i$\!+\!$1$}
		\htext(1.7 6){$i$\!+\!$1$}
		\htext(1 -0.3){$\underbrace{\rule{4.5em}{0em}}$}
		\htext(1 -0.7){$\bar{x}_{i\!+\!1}$}
		
		\move(1.4 7)\lvec(1.4 9)\lvec(2 9)\lvec(2 7)\lvec(1.4 7) \lfill f:0.8
		\htext(1.7 8){$i$}
		\esegment
		
		\move(2 0)
		\bsegment
		\move(0 0)\lvec(2 0)
		\move(0 0)\lvec(0 2)
		\move(0 2)\lvec(2 2)
		\move(0 2)\lvec(0 9)
		\move(2 0)\lvec(2 9)
		\move(0 9)\lvec(2 9)
		\move(0 5)\lvec(2 5)
		\move(0 4)\lvec(2 4)
		\move(0 7)\lvec(2 7)
		\move(0.6 0)\lvec(0.6 4)
		\move(1.4 0)\lvec(1.4 4)
		\move(0.6 5)\lvec(0.6 9)
		\move(1.4 5)\lvec(1.4 9)
		\htext(1.05 1){$\cdots$}
		\htext(1.05 3){$\cdots$}
		\vtext(1.05 4.5){$\cdots$}
		\htext(1.05 6){$\cdots$}
		\htext(1.05 8){$\cdots$}
		\htext(0.2 1.6){$0$}
		\htext(1.6 1.6){$0$}
		\htext(0.4 0.4){$*$}
		\htext(1.8 0.4){$*$}
		\move(0 0)\lvec(0.6 2)
		\move(1.4 0)\lvec(2 2)
		\htext(0.3 3){$*$\!+\!$1$}
		\htext(1.7 3){$*$\!+\!$1$}
		\htext(0.3 6){$i$\!+\!$1$}
		\htext(1.7 6){$i$\!+\!$1$}
		\htext(0.3 8){$i$}
		\htext(1.7 8){$i$}
		
		\htext(1 -0.3){$\underbrace{\rule{4.5em}{0em}}$}
		\htext(1 -0.7){$\bar{x}_i$}
		\esegment
		\esegment	
	\end{texdraw}
\end{center}

\item\label{2} For types $A_{2n}^{(2)}$ and $D_{n+1}^{(2)}$, 

\vskip 2mm

if $x_1\ge \bar{x}_1$,  the map $\tilde{\mathfrak f}_0$ is given by

\vskip 7mm

\begin{center}
	\begin{texdraw}
		\fontsize{3}{3}\selectfont
		\drawdim em
		\setunitscale 2.3
		\arrowheadsize l:0.3 w:0.3
		\arrowheadtype t:V
		\move(0 0)
		\bsegment
		\move(0 0)\lvec(2 0)
		\move(0 0)\lvec(0 2)
		\move(0 2)\lvec(2 2)
		\move(0.6 0)\lvec(0.6 2)
		\move(1.4 0)\lvec(1.4 2)
		\move(0 0)\lvec(0.6 2)
		\move(1.4 0)\lvec(2 2)
		\htext(0.4 0.4){$0$}
		\htext(1.8 0.4){$0$}
		\htext(1.05 1){$\cdots$}
		\htext(1 -0.3){$\underbrace{\rule{4.5em}{0em}}$}
		\htext(1 -0.6){$l'$}
		\move(2 0)
		\bsegment
		\move(0 0)\lvec(2 0)
		\move(0 0)\lvec(0 2)
		\move(0 2)\lvec(2 2)
		\move(0.6 0)\lvec(0.6 2)
		\move(1.4 0)\lvec(1.4 2)
		\htext(1.05 1){$\cdots$}
		\move(0 0)\lvec(0.6 2)
		\move(1.4 0)\lvec(2 2)
		\move(2 0)\lvec(2 2)
		\htext(0.2 1.6){$0$}
		\htext(1.6 1.6){$0$}
		\htext(0.4 0.4){$0$}
		\htext(1.8 0.4){$0$}
		\htext(1 -0.3){$\underbrace{\rule{4.5em}{0em}}$}
		\htext(1 -0.7){$x_1$}
		\move(2.5 1)\avec(4.5 1)
		\htext(3.5 1.8){	\fontsize{12}{12}\selectfont$\tilde{\mathfrak f}_0$}
		\esegment
		\esegment
		
		\move(7 0)
		\bsegment
		\move(0 0)\lvec(2 0)
		\move(0 0)\lvec(0 2)
		\move(0 2)\lvec(2 2)
		\move(0.6 0)\lvec(0.6 2)
		\move(1.4 0)\lvec(1.4 2)
		\move(0 0)\lvec(0.6 2)
		\move(1.4 0)\lvec(2 2)
		\htext(0.4 0.4){$0$}
		\htext(1.8 0.4){$0$}
		\htext(1.05 1){$\cdots$}
		\htext(1 -0.3){$\underbrace{\rule{4.5em}{0em}}$}
		\htext(1 -0.6){$l'$}
		\htext(1.6 1.6){$0$}
		\move(1.4 0)\lvec(1.4 2)\lvec(2 2)\lvec(1.4 0) \lfill f:0.8
		\move(2 0)
		\bsegment
		\move(0 0)\lvec(2 0)
		\move(0 0)\lvec(0 2)
		\move(0 2)\lvec(2 2)
		\move(0.6 0)\lvec(0.6 2)
		\move(1.4 0)\lvec(1.4 2)
		\htext(1.05 1){$\cdots$}
		\move(0 0)\lvec(0.6 2)
		\move(1.4 0)\lvec(2 2)
		\move(2 0)\lvec(2 2)
		\htext(0.2 1.6){$0$}
		\htext(1.6 1.6){$0$}
		\htext(0.4 0.4){$0$}
		\htext(1.8 0.4){$0$}
		\htext(1 -0.3){$\underbrace{\rule{4.5em}{0em}}$}
		\htext(1 -0.7){$x_1$}
		\esegment
		\esegment
	\end{texdraw}
\end{center}

\vskip 2mm

if $x_1< \bar{x}_1$,  the map $\tilde{\mathfrak f}_0$ is given by

\vskip 2mm

\begin{center}
	\begin{texdraw}
		\fontsize{3}{3}\selectfont
		\drawdim em
		\setunitscale 2.3
		\arrowheadsize l:0.3 w:0.3
		\arrowheadtype t:V
		\move(0 0)
		\bsegment
		\move(-1 0)
		\bsegment
		\move(0 0)\lvec(2 0)
		\move(0 0)\lvec(0 2)
		\move(0 2)\lvec(2 2)
		\move(2 0)\lvec(2 2)
		\move(0.6 0)\lvec(0.6 2)
		\move(1.4 0)\lvec(1.4 2)
		\move(0 0)\lvec(0.6 2)
		\move(1.4 0)\lvec(2 2)
		\htext(0.4 0.4){$0$}
		\htext(1.8 0.4){$0$}
		\htext(1.05 1){$\cdots$}
		\htext(1 -0.3){$\underbrace{\rule{4.5em}{0em}}$}
		\htext(1 -0.6){$l'$}
		\htext(2.55 1){\fontsize{8}{8}\selectfont$\cdots$}
		\esegment
		\move(2 0)
		\bsegment
		\move(0 0)\lvec(2 0)
		\move(0 0)\lvec(0 2)
		\move(0 2)\lvec(2 2)
		\move(0.6 0)\lvec(0.6 2)
		\move(1.4 0)\lvec(1.4 2)
		\htext(1.05 1){$\cdots$}
		\htext(1.05 3){$\cdots$}
		\htext(1.05 5){$\cdots$}
		\vtext(1.05 6.5){$\cdots$}
		\htext(1.05 8){$\cdots$}
		\htext(1.05 10){$\cdots$}
		\move(0 0)\lvec(0.6 2)
		\move(1.4 0)\lvec(2 2)
		\move(2 0)\lvec(2 2)
		\move(2 2)\lvec(2 6)
		\move(1.4 2)\lvec(1.4 6)
		\move(0.6 2)\lvec(0.6 6)
		\move(0 2)\lvec(0 6)
		\move(0 6)\lvec(0 7)
		\move(2 6)\lvec(2 7)
		\move(2 7)\lvec(2 11)
		\move(1.4 7)\lvec(1.4 11)
		\move(0.6 7)\lvec(0.6 11)
		\move(0 7)\lvec(0 11)
		
		\move(0 2)\lvec(2 2)
		\move(0 11)\lvec(2 11)
		\move(0 4)\lvec(2 4)
		\move(0 6)\lvec(2 6)
		\move(0 7)\lvec(2 7)
		\move(0 9)\lvec(2 9)
		\htext(0.3 3){$1$}
		\htext(0.3 5){$2$}
		\htext(0.3 8){$2$}
		\htext(0.3 10){$1$}
		\htext(1.7 3){$1$}
		\htext(1.7 5){$2$}
		\htext(1.7 8){$2$}
		\htext(1.7 10){$1$}
		
		\htext(0.2 1.6){$0$}
		\htext(1.6 1.6){$0$}
		\htext(0.4 0.4){$0$}
		\htext(1.8 0.4){$0$}
		\htext(1 -0.3){$\underbrace{\rule{4.5em}{0em}}$}
		\htext(1 -0.7){$\bar{x}_1$}
		\move(2.5 1)\avec(4.5 1)
		\htext(3.5 1.8){	\fontsize{12}{12}\selectfont$\tilde{\mathfrak f}_0$}
		\esegment
		\esegment
		
		\move(8 0)
		
		\bsegment
		\move(-1 0)
		\bsegment
		\move(0 0)\lvec(2 0)
		\move(0 0)\lvec(0 2)
		\move(0 2)\lvec(2 2)
		\move(2 0)\lvec(2 2)
		\move(0.6 0)\lvec(0.6 2)
		\move(1.4 0)\lvec(1.4 2)
		\move(0 0)\lvec(0.6 2)
		\move(1.4 0)\lvec(2 2)
		\htext(0.4 0.4){$0$}
		\htext(1.8 0.4){$0$}
		\htext(1.05 1){$\cdots$}
		\htext(1 -0.3){$\underbrace{\rule{4.5em}{0em}}$}
		\htext(1 -0.6){$l'$}
		\htext(2.55 1){\fontsize{8}{8}\selectfont$\cdots$}
		\esegment
		\move(2 0)
		\bsegment
		\move(0 0)\lvec(2 0)
		\move(0 0)\lvec(0 2)
		\move(0 2)\lvec(2 2)
		\move(0.6 0)\lvec(0.6 2)
		\move(1.4 0)\lvec(1.4 2)
		\htext(1.05 1){$\cdots$}
		\htext(1.05 3){$\cdots$}
		\htext(1.05 5){$\cdots$}
		\vtext(1.05 6.5){$\cdots$}
		\htext(1.05 8){$\cdots$}
		\htext(1.05 10){$\cdots$}
		\move(0 0)\lvec(0.6 2)
		\move(1.4 0)\lvec(2 2)
		\move(2 0)\lvec(2 2)
		\move(2 2)\lvec(2 6)
		\move(1.4 2)\lvec(1.4 6)
		\move(0.6 2)\lvec(0.6 6)
		\move(0 2)\lvec(0 6)
		\move(0 6)\lvec(0 7)
		\move(2 6)\lvec(2 7)
		\move(2 7)\lvec(2 11)
		\move(1.4 7)\lvec(1.4 11)
		\move(0.6 7)\lvec(0.6 11)
		\move(0 7)\lvec(0 11)
		
		\move(0 2)\lvec(2 2)
		\move(0 11)\lvec(2 11)
		\move(0 4)\lvec(2 4)
		\move(0 6)\lvec(2 6)
		\move(0 7)\lvec(2 7)
		\move(0 9)\lvec(2 9)
		\htext(0.3 3){$1$}
		\htext(0.3 5){$2$}
		\htext(0.3 8){$2$}
		\htext(0.3 10){$1$}
		\htext(1.7 3){$1$}
		\htext(1.7 5){$2$}
		\htext(1.7 8){$2$}
		\htext(1.7 10){$1$}
		
		\htext(0.2 1.6){$0$}
		\htext(1.6 1.6){$0$}
		\htext(0.4 0.4){$0$}
		\htext(1.8 0.4){$0$}
		\htext(1 -0.3){$\underbrace{\rule{4.5em}{0em}}$}
		\htext(1 -0.7){$\bar{x}_1$}
		\move(1.4 11)\lvec(1.4 13)\lvec(2 13) 
		
		\move(2 11)\lvec(2 13)\lvec(1.4 11)\lvec(2 11) \lfill f:0.8
		
		\htext(1.8 11.4){$0$}
		\esegment
		\esegment
	\end{texdraw}
\end{center}
where $l'=l-\sum_{i=1}^{n}(x_i+\bar{x}_i)$ for type $A_{2n}^{(2)}$ and  $l'=l-\sum_{i=1}^{n}(x_i+\bar{x}_i)-x_0$ for type $D_{n+1}^{(2)}$.

\vskip 9mm

\item\label{3} For types $A_{2n-1}^{(2)}$, $D_{n}^{(1)}$ and $B_{n}^{(1)}$, 

\vskip 2mm

if $x_2\ge \bar{x}_2$,  the map $\tilde{\mathfrak f}_1$ is given by

\vskip 10mm

\begin{center}
	\begin{texdraw}
		\fontsize{3}{3}\selectfont
		\drawdim em
		\setunitscale 2.3
		\arrowheadsize l:0.3 w:0.3
		\arrowheadtype t:V
		\move(0 0)
		\bsegment
		\move(-2 0)
		\bsegment
		\move(0 0)\lvec(2 0)
		\move(0 0)\lvec(0 2)
		\move(0 2)\lvec(2 2)
		\move(0.6 0)\lvec(0.6 2)
		\move(1.4 0)\lvec(1.4 2)
		\move(0 0)\lvec(0.6 2)
		\move(1.4 0)\lvec(2 2)
		\htext(0.2 1.6){$0$}
		\htext(1.6 1.6){$0$}
		\htext(1.05 1){$\cdots$}
		\htext(1 -0.3){$\underbrace{\rule{4.5em}{0em}}$}
		\htext(1 -0.7){$x_1$}
		\esegment
		\move(0 0)
		\bsegment
		\move(0 0)\lvec(2 0)
		\move(0 0)\lvec(0 2)
		\move(0 2)\lvec(2 2)
		\move(0.6 0)\lvec(0.6 2)
		\move(1.4 0)\lvec(1.4 2)
		\move(0 0)\lvec(0.6 2)
		\move(1.4 0)\lvec(2 2)
		\htext(0.4 0.4){$1$}
		\htext(1.8 0.4){$1$}
		\htext(1.05 1){$\cdots$}
		\htext(1 -0.3){$\underbrace{\rule{4.5em}{0em}}$}
		\htext(1 -0.7){$\bar{x}_1$}
		\esegment
		\move(2 0)
		\bsegment
		\move(0 0)\lvec(2 0)
		\move(0 0)\lvec(0 2)
		\move(0 2)\lvec(2 2)
		\move(0.6 0)\lvec(0.6 2)
		\move(1.4 0)\lvec(1.4 2)
		\htext(1.05 1){$\cdots$}
		\htext(0.2 1.6){$0$}
		\htext(1.6 1.6){$0$}
		\htext(0.4 0.4){$1$}
		\htext(1.8 0.4){$1$}
		\move(0 0)\lvec(0.6 2)
		\move(1.4 0)\lvec(2 2)
		\move(2 0)\lvec(2 2)
		\htext(1 -0.3){$\underbrace{\rule{4.5em}{0em}}$}
		\htext(1 -0.7){$x_2$}
		
		\move(2.5 1)\avec(4.5 1)
		\htext(3.5 1.8){	\fontsize{12}{12}\selectfont$\tilde{\mathfrak f}_1$}
		\esegment
		\esegment

		\move(9 0)
		\bsegment
		\move(-2 0)
		\bsegment
		\move(0 0)\lvec(2 0)
		\move(0 0)\lvec(0 2)
		\move(0 2)\lvec(2 2)
		\move(0.6 0)\lvec(0.6 2)
		\move(1.4 0)\lvec(1.4 2)
		\move(0 0)\lvec(0.6 2)
		\move(1.4 0)\lvec(2 2)
		\htext(0.2 1.6){$0$}
		\htext(1.6 1.6){$0$}
		\htext(1.05 1){$\cdots$}
		\htext(1 -0.3){$\underbrace{\rule{4.5em}{0em}}$}
		\htext(1 -0.7){$x_1$}
		\htext(1.8 0.4){$1$}
		\move(1.4 0)\lvec(2 0)\lvec(2 2)\lvec(1.4 0)\lfill f:0.8
		\esegment
		\move(0 0)
		\bsegment
		\move(0 0)\lvec(2 0)
		\move(0 0)\lvec(0 2)
		\move(0 2)\lvec(2 2)
		\move(0.6 0)\lvec(0.6 2)
		\move(1.4 0)\lvec(1.4 2)
		\move(0 0)\lvec(0.6 2)
		\move(1.4 0)\lvec(2 2)
		\htext(0.4 0.4){$1$}
		\htext(1.8 0.4){$1$}
		\htext(1.05 1){$\cdots$}
		\htext(1 -0.3){$\underbrace{\rule{4.5em}{0em}}$}
		\htext(1 -0.7){$\bar{x}_1$}
		
		\esegment
		\move(2 0)
		\bsegment
		\move(0 0)\lvec(2 0)
		\move(0 0)\lvec(0 2)
		\move(0 2)\lvec(2 2)
		\move(0.6 0)\lvec(0.6 2)
		\move(1.4 0)\lvec(1.4 2)
		\htext(1.05 1){$\cdots$}
		\htext(0.2 1.6){$0$}
		\htext(1.6 1.6){$0$}
		\htext(0.4 0.4){$1$}
		\htext(1.8 0.4){$1$}
		\move(0 0)\lvec(0.6 2)
		\move(1.4 0)\lvec(2 2)
		\move(2 0)\lvec(2 2)
		\htext(1 -0.3){$\underbrace{\rule{4.5em}{0em}}$}
		\htext(1 -0.7){$x_2$}

		\esegment
		\esegment
	\end{texdraw}
\end{center}

\vskip 2mm

if $x_2< \bar{x}_2$,  the map $\tilde{\mathfrak f}_1$ is given by

\vskip 2mm

\begin{center}
	\begin{texdraw}
		\fontsize{3}{3}\selectfont
		\drawdim em
		\setunitscale 2.3
		\arrowheadsize l:0.3 w:0.3
		\arrowheadtype t:V
		\move(0 0)
		\bsegment
		\move(-1 0)
		\bsegment
		\move(0 0)\lvec(2 0)
		\move(0 0)\lvec(0 2)
		\move(0 2)\lvec(2 2)
		\move(2 0)\lvec(2 2)
		\move(0.6 0)\lvec(0.6 2)
		\move(1.4 0)\lvec(1.4 2)
		\move(0 0)\lvec(0.6 2)
		\move(1.4 0)\lvec(2 2)
		\htext(0.4 0.4){$1$}
		\htext(1.8 0.4){$1$}
		\htext(1.05 1){$\cdots$}
		\htext(1 -0.3){$\underbrace{\rule{4.5em}{0em}}$}
		\htext(1 -0.7){$\bar{x}_1$}
		\htext(2.55 1){\fontsize{8}{8}\selectfont$\cdots$}
		\esegment
		\move(2 0)
		\bsegment
		\move(0 0)\lvec(2 0)
		\move(0 0)\lvec(0 2)
		\move(0 2)\lvec(2 2)
		\move(0.6 0)\lvec(0.6 2)
		\move(1.4 0)\lvec(1.4 2)
		\htext(1.05 1){$\cdots$}
		\htext(1.05 3){$\cdots$}
		\htext(1.05 5){$\cdots$}
		\vtext(1.05 6.5){$\cdots$}
		\htext(1.05 8){$\cdots$}
		\htext(1.05 10){$\cdots$}
		\move(0 0)\lvec(0.6 2)
		\move(1.4 0)\lvec(2 2)
		\move(2 0)\lvec(2 2)
		\move(2 2)\lvec(2 6)
		\move(1.4 2)\lvec(1.4 6)
		\move(0.6 2)\lvec(0.6 6)
		\move(0 2)\lvec(0 6)
		\move(0 6)\lvec(0 7)
		\move(2 6)\lvec(2 7)
		\move(2 7)\lvec(2 11)
		\move(1.4 7)\lvec(1.4 11)
		\move(0.6 7)\lvec(0.6 11)
		\move(0 7)\lvec(0 11)
		
		\move(0 2)\lvec(2 2)
		\move(0 11)\lvec(2 11)
		\move(0 4)\lvec(2 4)
		\move(0 6)\lvec(2 6)
		\move(0 7)\lvec(2 7)
		\move(0 9)\lvec(2 9)
		\htext(0.3 3){$2$}
		\htext(0.3 5){$3$}
		\htext(0.3 8){$3$}
		\htext(0.3 10){$2$}
		\htext(1.7 3){$2$}
		\htext(1.7 5){$3$}
		\htext(1.7 8){$3$}
		\htext(1.7 10){$2$}
		
		\htext(0.2 1.6){$0$}
		\htext(1.6 1.6){$0$}
		\htext(0.4 0.4){$1$}
		\htext(1.8 0.4){$1$}
		\htext(1 -0.3){$\underbrace{\rule{4.5em}{0em}}$}
		\htext(1 -0.7){$\bar{x}_2$}
		\move(2.5 1)\avec(4.5 1)
		\htext(3.5 1.8){	\fontsize{12}{12}\selectfont$\tilde{\mathfrak f}_1$}
		\esegment
		\esegment
		
		\move(8 0)
		\bsegment
		\move(-1 0)
		\bsegment
		\move(0 0)\lvec(2 0)
		\move(0 0)\lvec(0 2)
		\move(0 2)\lvec(2 2)
		\move(2 0)\lvec(2 2)
		\move(0.6 0)\lvec(0.6 2)
		\move(1.4 0)\lvec(1.4 2)
		\move(0 0)\lvec(0.6 2)
		\move(1.4 0)\lvec(2 2)
		\htext(0.4 0.4){$1$}
		\htext(1.8 0.4){$1$}
		\htext(1.05 1){$\cdots$}
		\htext(1 -0.3){$\underbrace{\rule{4.5em}{0em}}$}
		\htext(1 -0.7){$\bar{x}_1$}
		\htext(2.55 1){\fontsize{8}{8}\selectfont$\cdots$}
		\esegment
		\move(2 0)
		\bsegment
		\move(0 0)\lvec(2 0)
		\move(0 0)\lvec(0 2)
		\move(0 2)\lvec(2 2)
		\move(0.6 0)\lvec(0.6 2)
		\move(1.4 0)\lvec(1.4 2)
		\htext(1.05 1){$\cdots$}
		\htext(1.05 3){$\cdots$}
		\htext(1.05 5){$\cdots$}
		\vtext(1.05 6.5){$\cdots$}
		\htext(1.05 8){$\cdots$}
		\htext(1.05 10){$\cdots$}
		\move(0 0)\lvec(0.6 2)
		\move(1.4 0)\lvec(2 2)
		\move(2 0)\lvec(2 2)
		\move(2 2)\lvec(2 6)
		\move(1.4 2)\lvec(1.4 6)
		\move(0.6 2)\lvec(0.6 6)
		\move(0 2)\lvec(0 6)
		\move(0 6)\lvec(0 7)
		\move(2 6)\lvec(2 7)
		\move(2 7)\lvec(2 11)
		\move(1.4 7)\lvec(1.4 11)
		\move(0.6 7)\lvec(0.6 11)
		\move(0 7)\lvec(0 11)
		
		\move(0 2)\lvec(2 2)
		\move(0 11)\lvec(2 11)
		\move(0 4)\lvec(2 4)
		\move(0 6)\lvec(2 6)
		\move(0 7)\lvec(2 7)
		\move(0 9)\lvec(2 9)
		\htext(0.3 3){$2$}
		\htext(0.3 5){$3$}
		\htext(0.3 8){$3$}
		\htext(0.3 10){$2$}
		\htext(1.7 3){$2$}
		\htext(1.7 5){$3$}
		\htext(1.7 8){$3$}
		\htext(1.7 10){$2$}
		
		\htext(0.2 1.6){$0$}
		\htext(1.6 1.6){$0$}
		\htext(0.4 0.4){$1$}
		\htext(1.8 0.4){$1$}
		\htext(1 -0.3){$\underbrace{\rule{4.5em}{0em}}$}
		\htext(1 -0.7){$\bar{x}_2$}

		\move(1.4 11)\lvec(1.4 13)
		\move(1.4 13)\lvec(2 13)
		
		\move(1.4 11)\lvec(2 11)\lvec(2 13)\lvec(1.4 11) \lfill f:0.8
		
		\htext(1.8 11.4){$1$}
		\esegment
		\esegment
	\end{texdraw}
\end{center}

\vskip 5mm

\item\label{4} For types $A_{2n-1}^{(2)}$, $D_{n}^{(1)}$ and $B_{n}^{(1)}$, 

\vskip 2mm

if $x_2\ge \bar{x}_2$,  the map $\tilde{\mathfrak f}_0$ is given by

\vskip 10mm

\begin{center}
	\begin{texdraw}
		\fontsize{3}{3}\selectfont
		\drawdim em
		\setunitscale 2.3
		\arrowheadsize l:0.3 w:0.3
		\arrowheadtype t:V
		\move(0 0)
		\bsegment
		\move(-2 0)
		\bsegment
		\move(0 0)\lvec(2 0)
		\move(0 0)\lvec(0 2)
		\move(0 2)\lvec(2 2)
		\move(0.6 0)\lvec(0.6 2)
		\move(1.4 0)\lvec(1.4 2)
		\move(0 0)\lvec(0.6 2)
		\move(1.4 0)\lvec(2 2)
		\htext(0.2 1.6){$0$}
		\htext(1.6 1.6){$0$}
		\htext(1.05 1){$\cdots$}
		\htext(1 -0.3){$\underbrace{\rule{4.5em}{0em}}$}
		\htext(1 -0.7){$x_1$}
		\esegment
		\move(0 0)
		\bsegment
		\move(0 0)\lvec(2 0)
		\move(0 0)\lvec(0 2)
		\move(0 2)\lvec(2 2)
		\move(0.6 0)\lvec(0.6 2)
		\move(1.4 0)\lvec(1.4 2)
		\move(0 0)\lvec(0.6 2)
		\move(1.4 0)\lvec(2 2)
		\htext(0.4 0.4){$1$}
		\htext(1.8 0.4){$1$}
		\htext(1.05 1){$\cdots$}
		\htext(1 -0.3){$\underbrace{\rule{4.5em}{0em}}$}
		\htext(1 -0.7){$\bar{x}_1$}
		\esegment
		\move(2 0)
		\bsegment
		\move(0 0)\lvec(2 0)
		\move(0 0)\lvec(0 2)
		\move(0 2)\lvec(2 2)
		\move(0.6 0)\lvec(0.6 2)
		\move(1.4 0)\lvec(1.4 2)
		\htext(1.05 1){$\cdots$}
		\htext(0.2 1.6){$0$}
		\htext(1.6 1.6){$0$}
		\htext(0.4 0.4){$1$}
		\htext(1.8 0.4){$1$}
		\move(0 0)\lvec(0.6 2)
		\move(1.4 0)\lvec(2 2)
		\move(2 0)\lvec(2 2)
		\htext(1 -0.3){$\underbrace{\rule{4.5em}{0em}}$}
		\htext(1 -0.7){$x_2$}
		
		\move(2.5 1)\avec(4.5 1)
		\htext(3.5 1.8){	\fontsize{12}{12}\selectfont$\tilde{\mathfrak f}_0$}
		\esegment
		\esegment

		\move(9 0)
		\bsegment
		\move(-2 0)
		\bsegment
		\move(0 0)\lvec(2 0)
		\move(0 0)\lvec(0 2)
		\move(0 2)\lvec(2 2)
		\move(0.6 0)\lvec(0.6 2)
		\move(1.4 0)\lvec(1.4 2)
		\move(0 0)\lvec(0.6 2)
		\move(1.4 0)\lvec(2 2)
		\htext(0.2 1.6){$0$}
		\htext(1.6 1.6){$0$}
		\htext(1.05 1){$\cdots$}
		\htext(1 -0.3){$\underbrace{\rule{4.5em}{0em}}$}
		\htext(1 -0.7){$x_1$}
		\esegment
		\move(0 0)
		\bsegment
		\move(0 0)\lvec(2 0)
		\move(0 0)\lvec(0 2)
		\move(0 2)\lvec(2 2)
		\move(0.6 0)\lvec(0.6 2)
		\move(1.4 0)\lvec(1.4 2)
		\move(0 0)\lvec(0.6 2)
		\move(1.4 0)\lvec(2 2)
		\htext(0.4 0.4){$1$}
		\htext(1.8 0.4){$1$}
		\htext(1.05 1){$\cdots$}
		\htext(1 -0.3){$\underbrace{\rule{4.5em}{0em}}$}
		\htext(1 -0.7){$\bar{x}_1$}
		\esegment
		\move(2 0)
		\bsegment
		\move(0 0)\lvec(2 0)
		\move(0 0)\lvec(0 2)
		\move(0 2)\lvec(2 2)
		\move(0.6 0)\lvec(0.6 2)
		\move(1.4 0)\lvec(1.4 2)
		\htext(1.05 1){$\cdots$}
		\htext(0.2 1.6){$0$}
		\htext(1.6 1.6){$0$}
		\htext(0.4 0.4){$1$}
		\htext(1.8 0.4){$1$}
		\move(0 0)\lvec(0.6 2)
		\move(1.4 0)\lvec(2 2)
		\move(2 0)\lvec(2 2)
		\htext(1 -0.3){$\underbrace{\rule{4.5em}{0em}}$}
		\htext(1 -0.7){$x_2$}
		
		\move(-0.6 0)\lvec(-0.6 2)\lvec(0 2)\lvec(-0.6 0) \lfill f:0.8
		\htext(-0.4 1.6){$0$}
		\esegment
		\esegment
	\end{texdraw}
\end{center}

\vskip 2mm

if $x_2< \bar{x}_2$,  the map $\tilde{\mathfrak f}_0$ is given by

\vskip 2mm

\begin{center}
	\begin{texdraw}
		\fontsize{3}{3}\selectfont
		\drawdim em
		\setunitscale 2.3
		\arrowheadsize l:0.3 w:0.3
		\arrowheadtype t:V
		\move(0 0)
		\bsegment
		\move(-1 0)
		\bsegment
		\move(0 0)\lvec(2 0)
		\move(0 0)\lvec(0 2)
		\move(0 2)\lvec(2 2)
		\move(2 0)\lvec(2 2)
		\move(0.6 0)\lvec(0.6 2)
		\move(1.4 0)\lvec(1.4 2)
		\move(0 0)\lvec(0.6 2)
		\move(1.4 0)\lvec(2 2)
		\htext(0.2 1.6){$0$}
		\htext(1.6 1.6){$0$}
		\htext(1.05 1){$\cdots$}
		\htext(1 -0.3){$\underbrace{\rule{4.5em}{0em}}$}
		\htext(1 -0.7){$x_1$}
		\htext(2.55 1){\fontsize{8}{8}\selectfont$\cdots$}
		\esegment
		\move(2 0)
		\bsegment
		\move(0 0)\lvec(2 0)
		\move(0 0)\lvec(0 2)
		\move(0 2)\lvec(2 2)
		\move(0.6 0)\lvec(0.6 2)
		\move(1.4 0)\lvec(1.4 2)
		\htext(1.05 1){$\cdots$}
		\htext(1.05 3){$\cdots$}
		\htext(1.05 5){$\cdots$}
		\vtext(1.05 6.5){$\cdots$}
		\htext(1.05 8){$\cdots$}
		\htext(1.05 10){$\cdots$}
		\move(0 0)\lvec(0.6 2)
		\move(1.4 0)\lvec(2 2)
		\move(2 0)\lvec(2 2)
		\move(2 2)\lvec(2 6)
		\move(1.4 2)\lvec(1.4 6)
		\move(0.6 2)\lvec(0.6 6)
		\move(0 2)\lvec(0 6)
		\move(0 6)\lvec(0 7)
		\move(2 6)\lvec(2 7)
		\move(2 7)\lvec(2 11)
		\move(1.4 7)\lvec(1.4 11)
		\move(0.6 7)\lvec(0.6 11)
		\move(0 7)\lvec(0 11)
		
		\move(0 2)\lvec(2 2)
		\move(0 11)\lvec(2 11)
		\move(0 4)\lvec(2 4)
		\move(0 6)\lvec(2 6)
		\move(0 7)\lvec(2 7)
		\move(0 9)\lvec(2 9)
		\htext(0.3 3){$2$}
		\htext(0.3 5){$3$}
		\htext(0.3 8){$3$}
		\htext(0.3 10){$2$}
		\htext(1.7 3){$2$}
		\htext(1.7 5){$3$}
		\htext(1.7 8){$3$}
		\htext(1.7 10){$2$}
		
		\htext(0.2 1.6){$0$}
		\htext(1.6 1.6){$0$}
		\htext(0.4 0.4){$1$}
		\htext(1.8 0.4){$1$}
		\htext(1 -0.3){$\underbrace{\rule{4.5em}{0em}}$}
		\htext(1 -0.7){$\bar{x}_2$}
		\move(2.5 1)\avec(4.5 1)
		\htext(3.5 1.8){	\fontsize{12}{12}\selectfont$\tilde{\mathfrak f}_0$}
		\esegment
		\esegment
		
		\move(8 0)
		\bsegment
		\move(-1 0)
		\bsegment
		\move(0 0)\lvec(2 0)
		\move(0 0)\lvec(0 2)
		\move(0 2)\lvec(2 2)
		\move(2 0)\lvec(2 2)
		\move(0.6 0)\lvec(0.6 2)
		\move(1.4 0)\lvec(1.4 2)
		\move(0 0)\lvec(0.6 2)
		\move(1.4 0)\lvec(2 2)
		\htext(0.2 1.6){$0$}
		\htext(1.6 1.6){$0$}
		\htext(1.05 1){$\cdots$}
		\htext(1 -0.3){$\underbrace{\rule{4.5em}{0em}}$}
		\htext(1 -0.7){$x_1$}
		\htext(2.55 1){\fontsize{8}{8}\selectfont$\cdots$}
		\esegment
		\move(2 0)
		\bsegment
		\move(0 0)\lvec(2 0)
		\move(0 0)\lvec(0 2)
		\move(0 2)\lvec(2 2)
		\move(0.6 0)\lvec(0.6 2)
		\move(1.4 0)\lvec(1.4 2)
		\htext(1.05 1){$\cdots$}
		\htext(1.05 3){$\cdots$}
		\htext(1.05 5){$\cdots$}
		\vtext(1.05 6.5){$\cdots$}
		\htext(1.05 8){$\cdots$}
		\htext(1.05 10){$\cdots$}
		\move(0 0)\lvec(0.6 2)
		\move(1.4 0)\lvec(2 2)
		\move(2 0)\lvec(2 2)
		\move(2 2)\lvec(2 6)
		\move(1.4 2)\lvec(1.4 6)
		\move(0.6 2)\lvec(0.6 6)
		\move(0 2)\lvec(0 6)
		\move(0 6)\lvec(0 7)
		\move(2 6)\lvec(2 7)
		\move(2 7)\lvec(2 11)
		\move(1.4 7)\lvec(1.4 11)
		\move(0.6 7)\lvec(0.6 11)
		\move(0 7)\lvec(0 11)
		
		\move(0 2)\lvec(2 2)
		\move(0 11)\lvec(2 11)
		\move(0 4)\lvec(2 4)
		\move(0 6)\lvec(2 6)
		\move(0 7)\lvec(2 7)
		\move(0 9)\lvec(2 9)
		\htext(0.3 3){$2$}
		\htext(0.3 5){$3$}
		\htext(0.3 8){$3$}
		\htext(0.3 10){$2$}
		\htext(1.7 3){$2$}
		\htext(1.7 5){$3$}
		\htext(1.7 8){$3$}
		\htext(1.7 10){$2$}
		
		\htext(0.2 1.6){$0$}
		\htext(1.6 1.6){$0$}
		\htext(0.4 0.4){$1$}
		\htext(1.8 0.4){$1$}
		\htext(1 -0.3){$\underbrace{\rule{4.5em}{0em}}$}
		\htext(1 -0.7){$\bar{x}_2$}

		\move(2 11)\lvec(2 13)
		\move(1.4 11)\lvec(1.4 13)\lvec(2 13)\lvec(1.4 11)\lfill f:0.8
		\htext(1.6 12.6){$0$}
		\esegment
		\esegment
	\end{texdraw}
\end{center}

\item\label{5} For types $D_{n+1}^{(2)}$ and $B_{n}^{(1)}$, 

\vskip 2mm

if $x_0=0$,  the map $\tilde{\mathfrak f}_n$ is given by

\vskip 10mm

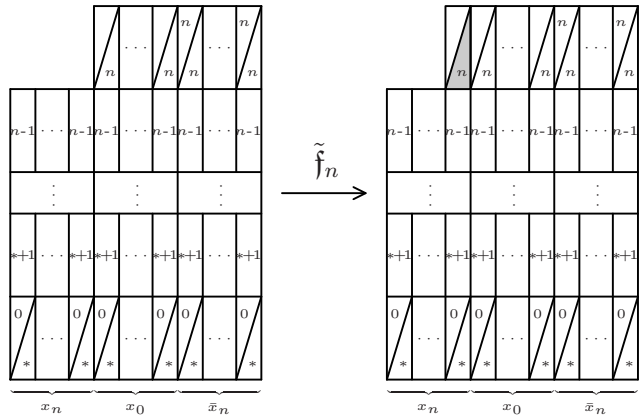
\begin{center}
	\begin{texdraw}
		\fontsize{3}{3}\selectfont
		\drawdim em
		\setunitscale 2.3
		\arrowheadsize l:0.3 w:0.3
		\arrowheadtype t:V
		\move(0 0)
		\bsegment
		\move(0 0)
		\bsegment
		\move(0 0)\lvec(2 0)
		\move(0 0)\lvec(0 2)
		\move(0 2)\lvec(2 2)
		\move(0 2)\lvec(0 7)
		\move(0 7)\lvec(2 7)
		\move(0 4)\lvec(2 4)
		\move(0 5)\lvec(2 5)
		\move(0.6 0)\lvec(0.6 4)
		\move(1.4 0)\lvec(1.4 4)
		\move(0.6 5)\lvec(0.6 7)
		\move(1.4 5)\lvec(1.4 7)
		\htext(1.05 1){$\cdots$}
		\htext(1.05 3){$\cdots$}
		\vtext(1.05 4.5){$\cdots$}
		\htext(1.05 6){$\cdots$}
		\htext(0.2 1.6){$0$}
		\htext(1.6 1.6){$0$}
		\htext(0.4 0.4){$*$}
		\htext(1.8 0.4){$*$}
		\move(0 0)\lvec(0.6 2)
		\move(1.4 0)\lvec(2 2)	
		\htext(0.3 3){$*$\!+\!$1$}
		\htext(1.7 3){$*$\!+\!$1$}
		\htext(0.3 6){$n$-$1$}
		\htext(1.7 6){$n$-$1$}
		\htext(1 -0.3){$\underbrace{\rule{4.5em}{0em}}$}
		\htext(1 -0.7){$x_n$}
		
		\esegment
		
		\move(2 0)
		\bsegment
		\move(0 0)\lvec(2 0)
		\move(0 0)\lvec(0 2)
		\move(0 2)\lvec(2 2)
		\move(0 2)\lvec(0 9)
		
		\move(0 9)\lvec(2 9)
		\move(0 5)\lvec(2 5)
		\move(0 4)\lvec(2 4)
		\move(0 7)\lvec(2 7)
		\move(0.6 0)\lvec(0.6 4)
		\move(1.4 0)\lvec(1.4 4)
		\move(0.6 5)\lvec(0.6 9)
		\move(1.4 5)\lvec(1.4 9)
		\htext(1.05 1){$\cdots$}
		\htext(1.05 3){$\cdots$}
		\vtext(1.05 4.5){$\cdots$}
		\htext(1.05 6){$\cdots$}
		\htext(1.05 8){$\cdots$}
		\htext(0.2 1.6){$0$}
		\htext(1.6 1.6){$0$}
		\htext(0.4 0.4){$*$}
		\htext(1.8 0.4){$*$}
		\move(0 0)\lvec(0.6 2)
		\move(1.4 0)\lvec(2 2)
		\htext(0.3 3){$*$\!+\!$1$}
		\htext(1.7 3){$*$\!+\!$1$}
		\htext(0.3 6){$n$-$1$}
		\htext(1.7 6){$n$-$1$}
		\move(0 7)\lvec(0.6 9)
		\move(1.4 7)\lvec(2 9)
		
		\htext(0.4 7.4){$n$}
		\htext(1.8 7.4){$n$}

		\htext(1 -0.3){$\underbrace{\rule{4.5em}{0em}}$}
		\htext(1 -0.7){$x_0$}

		\esegment
		
		\move(4 0)
		\bsegment
		\move(0 0)\lvec(2 0)
		\move(0 0)\lvec(0 2)
		\move(0 2)\lvec(2 2)
		\move(0 2)\lvec(0 9)
		\move(2 0)\lvec(2 9)
		\move(0 9)\lvec(2 9)
		\move(0 5)\lvec(2 5)
		\move(0 4)\lvec(2 4)
		\move(0 7)\lvec(2 7)
		\move(0.6 0)\lvec(0.6 4)
		\move(1.4 0)\lvec(1.4 4)
		\move(0.6 5)\lvec(0.6 9)
		\move(1.4 5)\lvec(1.4 9)
		\htext(1.05 1){$\cdots$}
		\htext(1.05 3){$\cdots$}
		\vtext(1.05 4.5){$\cdots$}
		\htext(1.05 6){$\cdots$}
		\htext(1.05 8){$\cdots$}
		\htext(0.2 1.6){$0$}
		\htext(1.6 1.6){$0$}
		\htext(0.4 0.4){$*$}
		\htext(1.8 0.4){$*$}
		\move(0 0)\lvec(0.6 2)
		\move(1.4 0)\lvec(2 2)
		\htext(0.3 3){$*$\!+\!$1$}
		\htext(1.7 3){$*$\!+\!$1$}
		\htext(0.3 6){$n$-$1$}
		\htext(1.7 6){$n$-$1$}

		\htext(1 -0.3){$\underbrace{\rule{4.5em}{0em}}$}
		\htext(1 -0.7){$\bar{x}_n$}
		
		\move(2.5 4.5)\avec(4.5 4.5)
		\htext(3.5 5.3){	\fontsize{12}{12}\selectfont$\tilde{\mathfrak f}_n$}
		
		\move(0 7)\lvec(0.6 9)
		\move(1.4 7)\lvec(2 9)
		\htext(0.2 8.6){$n$}
		\htext(1.6 8.6){$n$}
		\htext(0.4 7.4){$n$}
		\htext(1.8 7.4){$n$}
		
		\esegment
		\esegment
		
		\move(9 0)
		\bsegment
		\move(0 0)
		\bsegment
		\move(0 0)\lvec(2 0)
		\move(0 0)\lvec(0 2)
		\move(0 2)\lvec(2 2)
		\move(0 2)\lvec(0 7)
		\move(0 7)\lvec(2 7)
		\move(0 4)\lvec(2 4)
		\move(0 5)\lvec(2 5)
		\move(0.6 0)\lvec(0.6 4)
		\move(1.4 0)\lvec(1.4 4)
		\move(0.6 5)\lvec(0.6 7)
		\move(1.4 5)\lvec(1.4 7)
		\htext(1.05 1){$\cdots$}
		\htext(1.05 3){$\cdots$}
		\vtext(1.05 4.5){$\cdots$}
		\htext(1.05 6){$\cdots$}
		\htext(0.2 1.6){$0$}
		\htext(1.6 1.6){$0$}
		\htext(0.4 0.4){$*$}
		\htext(1.8 0.4){$*$}
		\move(0 0)\lvec(0.6 2)
		\move(1.4 0)\lvec(2 2)	
		\htext(0.3 3){$*$\!+\!$1$}
		\htext(1.7 3){$*$\!+\!$1$}
		\htext(0.3 6){$n$-$1$}
		\htext(1.7 6){$n$-$1$}
		\htext(1 -0.3){$\underbrace{\rule{4.5em}{0em}}$}
		\htext(1 -0.7){$x_n$}
		
		\move(1.4 7)\lvec(2 7)\lvec(2 9)\lvec(1.4 7) \lfill f:0.8
		
		\move(1.4 7)\lvec(1.4 9)\lvec(2 9)
		\htext(1.8 7.4){$n$}
		\esegment
		
		\move(2 0)
		\bsegment
		\move(0 0)\lvec(2 0)
		\move(0 0)\lvec(0 2)
		\move(0 2)\lvec(2 2)
		\move(0 2)\lvec(0 9)
		
		\move(0 9)\lvec(2 9)
		\move(0 5)\lvec(2 5)
		\move(0 4)\lvec(2 4)
		\move(0 7)\lvec(2 7)
		\move(0.6 0)\lvec(0.6 4)
		\move(1.4 0)\lvec(1.4 4)
		\move(0.6 5)\lvec(0.6 9)
		\move(1.4 5)\lvec(1.4 9)
		\htext(1.05 1){$\cdots$}
		\htext(1.05 3){$\cdots$}
		\vtext(1.05 4.5){$\cdots$}
		\htext(1.05 6){$\cdots$}
		\htext(1.05 8){$\cdots$}
		\htext(0.2 1.6){$0$}
		\htext(1.6 1.6){$0$}
		\htext(0.4 0.4){$*$}
		\htext(1.8 0.4){$*$}
		\move(0 0)\lvec(0.6 2)
		\move(1.4 0)\lvec(2 2)
		\htext(0.3 3){$*$\!+\!$1$}
		\htext(1.7 3){$*$\!+\!$1$}
		\htext(0.3 6){$n$-$1$}
		\htext(1.7 6){$n$-$1$}
		
		\htext(1 -0.3){$\underbrace{\rule{4.5em}{0em}}$}
		\htext(1 -0.7){$x_0$}

		\move(0 7)\lvec(0.6 9)
		\move(1.4 7)\lvec(2 9)
		
		\htext(0.4 7.4){$n$}
		\htext(1.8 7.4){$n$}
		\esegment
		
		\move(4 0)
		\bsegment
		\move(0 0)\lvec(2 0)
		\move(0 0)\lvec(0 2)
		\move(0 2)\lvec(2 2)
		\move(0 2)\lvec(0 9)
		\move(2 0)\lvec(2 9)
		\move(0 9)\lvec(2 9)
		\move(0 5)\lvec(2 5)
		\move(0 4)\lvec(2 4)
		\move(0 7)\lvec(2 7)
		\move(0.6 0)\lvec(0.6 4)
		\move(1.4 0)\lvec(1.4 4)
		\move(0.6 5)\lvec(0.6 9)
		\move(1.4 5)\lvec(1.4 9)
		\htext(1.05 1){$\cdots$}
		\htext(1.05 3){$\cdots$}
		\vtext(1.05 4.5){$\cdots$}
		\htext(1.05 6){$\cdots$}
		\htext(1.05 8){$\cdots$}
		\htext(0.2 1.6){$0$}
		\htext(1.6 1.6){$0$}
		\htext(0.4 0.4){$*$}
		\htext(1.8 0.4){$*$}
		\move(0 0)\lvec(0.6 2)
		\move(1.4 0)\lvec(2 2)
		\htext(0.3 3){$*$\!+\!$1$}
		\htext(1.7 3){$*$\!+\!$1$}
		\htext(0.3 6){$n$-$1$}
		\htext(1.7 6){$n$-$1$}
		
		\move(0 7)\lvec(0.6 9)
		\move(1.4 7)\lvec(2 9)
		\htext(0.2 8.6){$n$}
		\htext(1.6 8.6){$n$}
		\htext(0.4 7.4){$n$}
		\htext(1.8 7.4){$n$}	
		
		\htext(1 -0.3){$\underbrace{\rule{4.5em}{0em}}$}
		\htext(1 -0.7){$\bar{x}_n$}

		\esegment
		\esegment	
	\end{texdraw}
\end{center}

\vskip 2mm

if $x_0=1$,  the map $\tilde{\mathfrak f}_n$ is given by

\vskip 10mm

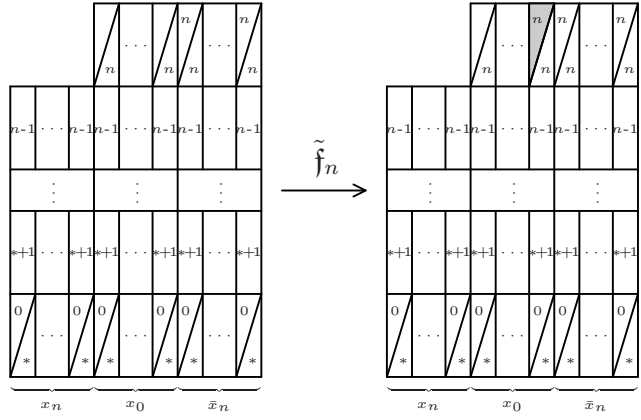
\begin{center}
	\begin{texdraw}
		\fontsize{3}{3}\selectfont
		\drawdim em
		\setunitscale 2.3
		\arrowheadsize l:0.3 w:0.3
		\arrowheadtype t:V
		\move(0 0)
		\bsegment
		\move(0 0)
		\bsegment
		\move(0 0)\lvec(2 0)
		\move(0 0)\lvec(0 2)
		\move(0 2)\lvec(2 2)
		\move(0 2)\lvec(0 7)
		\move(0 7)\lvec(2 7)
		\move(0 4)\lvec(2 4)
		\move(0 5)\lvec(2 5)
		\move(0.6 0)\lvec(0.6 4)
		\move(1.4 0)\lvec(1.4 4)
		\move(0.6 5)\lvec(0.6 7)
		\move(1.4 5)\lvec(1.4 7)
		\htext(1.05 1){$\cdots$}
		\htext(1.05 3){$\cdots$}
		\vtext(1.05 4.5){$\cdots$}
		\htext(1.05 6){$\cdots$}
		\htext(0.2 1.6){$0$}
		\htext(1.6 1.6){$0$}
		\htext(0.4 0.4){$*$}
		\htext(1.8 0.4){$*$}
		\move(0 0)\lvec(0.6 2)
		\move(1.4 0)\lvec(2 2)	
		\htext(0.3 3){$*$\!+\!$1$}
		\htext(1.7 3){$*$\!+\!$1$}
		\htext(0.3 6){$n$-$1$}
		\htext(1.7 6){$n$-$1$}
		\htext(1 -0.3){$\underbrace{\rule{4.5em}{0em}}$}
		\htext(1 -0.7){$x_n$}
		
		\esegment
		
		\move(2 0)
		\bsegment
		\move(0 0)\lvec(2 0)
		\move(0 0)\lvec(0 2)
		\move(0 2)\lvec(2 2)
		\move(0 2)\lvec(0 9)
		
		\move(0 9)\lvec(2 9)
		\move(0 5)\lvec(2 5)
		\move(0 4)\lvec(2 4)
		\move(0 7)\lvec(2 7)
		\move(0.6 0)\lvec(0.6 4)
		\move(1.4 0)\lvec(1.4 4)
		\move(0.6 5)\lvec(0.6 9)
		\move(1.4 5)\lvec(1.4 9)
		\htext(1.05 1){$\cdots$}
		\htext(1.05 3){$\cdots$}
		\vtext(1.05 4.5){$\cdots$}
		\htext(1.05 6){$\cdots$}
		\htext(1.05 8){$\cdots$}
		\htext(0.2 1.6){$0$}
		\htext(1.6 1.6){$0$}
		\htext(0.4 0.4){$*$}
		\htext(1.8 0.4){$*$}
		\move(0 0)\lvec(0.6 2)
		\move(1.4 0)\lvec(2 2)
		\htext(0.3 3){$*$\!+\!$1$}
		\htext(1.7 3){$*$\!+\!$1$}
		\htext(0.3 6){$n$-$1$}
		\htext(1.7 6){$n$-$1$}
		\move(0 7)\lvec(0.6 9)
		\move(1.4 7)\lvec(2 9)
		
		\htext(0.4 7.4){$n$}
		\htext(1.8 7.4){$n$}

		\htext(1 -0.3){$\underbrace{\rule{4.5em}{0em}}$}
		\htext(1 -0.7){$x_0$}

		\esegment
		
		\move(4 0)
		\bsegment
		\move(0 0)\lvec(2 0)
		\move(0 0)\lvec(0 2)
		\move(0 2)\lvec(2 2)
		\move(0 2)\lvec(0 9)
		\move(2 0)\lvec(2 9)
		\move(0 9)\lvec(2 9)
		\move(0 5)\lvec(2 5)
		\move(0 4)\lvec(2 4)
		\move(0 7)\lvec(2 7)
		\move(0.6 0)\lvec(0.6 4)
		\move(1.4 0)\lvec(1.4 4)
		\move(0.6 5)\lvec(0.6 9)
		\move(1.4 5)\lvec(1.4 9)
		\htext(1.05 1){$\cdots$}
		\htext(1.05 3){$\cdots$}
		\vtext(1.05 4.5){$\cdots$}
		\htext(1.05 6){$\cdots$}
		\htext(1.05 8){$\cdots$}
		\htext(0.2 1.6){$0$}
		\htext(1.6 1.6){$0$}
		\htext(0.4 0.4){$*$}
		\htext(1.8 0.4){$*$}
		\move(0 0)\lvec(0.6 2)
		\move(1.4 0)\lvec(2 2)
		\htext(0.3 3){$*$\!+\!$1$}
		\htext(1.7 3){$*$\!+\!$1$}
		\htext(0.3 6){$n$-$1$}
		\htext(1.7 6){$n$-$1$}

		\htext(1 -0.3){$\underbrace{\rule{4.5em}{0em}}$}
		\htext(1 -0.7){$\bar{x}_n$}
		
		\move(2.5 4.5)\avec(4.5 4.5)
		\htext(3.5 5.3){	\fontsize{12}{12}\selectfont$\tilde{\mathfrak f}_n$}
		
		\move(0 7)\lvec(0.6 9)
		\move(1.4 7)\lvec(2 9)
		\htext(0.2 8.6){$n$}
		\htext(1.6 8.6){$n$}
		\htext(0.4 7.4){$n$}
		\htext(1.8 7.4){$n$}
		
		\esegment
		\esegment
		
		\move(9 0)
		\bsegment
		\move(0 0)
		\bsegment
		\move(0 0)\lvec(2 0)
		\move(0 0)\lvec(0 2)
		\move(0 2)\lvec(2 2)
		\move(0 2)\lvec(0 7)
		\move(0 7)\lvec(2 7)
		\move(0 4)\lvec(2 4)
		\move(0 5)\lvec(2 5)
		\move(0.6 0)\lvec(0.6 4)
		\move(1.4 0)\lvec(1.4 4)
		\move(0.6 5)\lvec(0.6 7)
		\move(1.4 5)\lvec(1.4 7)
		\htext(1.05 1){$\cdots$}
		\htext(1.05 3){$\cdots$}
		\vtext(1.05 4.5){$\cdots$}
		\htext(1.05 6){$\cdots$}
		\htext(0.2 1.6){$0$}
		\htext(1.6 1.6){$0$}
		\htext(0.4 0.4){$*$}
		\htext(1.8 0.4){$*$}
		\move(0 0)\lvec(0.6 2)
		\move(1.4 0)\lvec(2 2)	
		\htext(0.3 3){$*$\!+\!$1$}
		\htext(1.7 3){$*$\!+\!$1$}
		\htext(0.3 6){$n$-$1$}
		\htext(1.7 6){$n$-$1$}
		\htext(1 -0.3){$\underbrace{\rule{4.5em}{0em}}$}
		\htext(1 -0.7){$x_n$}

		\esegment
		
		\move(2 0)
		\bsegment
		\move(0 0)\lvec(2 0)
		\move(0 0)\lvec(0 2)
		\move(0 2)\lvec(2 2)
		\move(0 2)\lvec(0 9)
		
		\move(0 9)\lvec(2 9)
		\move(0 5)\lvec(2 5)
		\move(0 4)\lvec(2 4)
		\move(0 7)\lvec(2 7)
		\move(0.6 0)\lvec(0.6 4)
		\move(1.4 0)\lvec(1.4 4)
		\move(0.6 5)\lvec(0.6 9)
		\move(1.4 5)\lvec(1.4 9)
		\htext(1.05 1){$\cdots$}
		\htext(1.05 3){$\cdots$}
		\vtext(1.05 4.5){$\cdots$}
		\htext(1.05 6){$\cdots$}
		\htext(1.05 8){$\cdots$}
		\htext(0.2 1.6){$0$}
		\htext(1.6 1.6){$0$}
		\htext(0.4 0.4){$*$}
		\htext(1.8 0.4){$*$}
		\move(0 0)\lvec(0.6 2)
		\move(1.4 0)\lvec(2 2)
		\htext(0.3 3){$*$\!+\!$1$}
		\htext(1.7 3){$*$\!+\!$1$}
		\htext(0.3 6){$n$-$1$}
		\htext(1.7 6){$n$-$1$}
		
		\htext(1 -0.3){$\underbrace{\rule{4.5em}{0em}}$}
		\htext(1 -0.7){$x_0$}
		\move(1.4 7)\lvec(1.4 9)\lvec(2 9)\lvec(1.4 7) \lfill f:0.8
		\htext(1.6 8.6){$n$}
		
		\move(0 7)\lvec(0.6 9)
		\move(1.4 7)\lvec(2 9)
		
		\htext(0.4 7.4){$n$}
		\htext(1.8 7.4){$n$}

		\esegment
		
		\move(4 0)
		\bsegment
		\move(0 0)\lvec(2 0)
		\move(0 0)\lvec(0 2)
		\move(0 2)\lvec(2 2)
		\move(0 2)\lvec(0 9)
		\move(2 0)\lvec(2 9)
		\move(0 9)\lvec(2 9)
		\move(0 5)\lvec(2 5)
		\move(0 4)\lvec(2 4)
		\move(0 7)\lvec(2 7)
		\move(0.6 0)\lvec(0.6 4)
		\move(1.4 0)\lvec(1.4 4)
		\move(0.6 5)\lvec(0.6 9)
		\move(1.4 5)\lvec(1.4 9)
		\htext(1.05 1){$\cdots$}
		\htext(1.05 3){$\cdots$}
		\vtext(1.05 4.5){$\cdots$}
		\htext(1.05 6){$\cdots$}
		\htext(1.05 8){$\cdots$}
		\htext(0.2 1.6){$0$}
		\htext(1.6 1.6){$0$}
		\htext(0.4 0.4){$*$}
		\htext(1.8 0.4){$*$}
		\move(0 0)\lvec(0.6 2)
		\move(1.4 0)\lvec(2 2)
		\htext(0.3 3){$*$\!+\!$1$}
		\htext(1.7 3){$*$\!+\!$1$}
		\htext(0.3 6){$n$-$1$}
		\htext(1.7 6){$n$-$1$}
		
		\move(0 7)\lvec(0.6 9)
		\move(1.4 7)\lvec(2 9)
		\htext(0.2 8.6){$n$}
		\htext(1.6 8.6){$n$}
		\htext(0.4 7.4){$n$}
		\htext(1.8 7.4){$n$}	
		
		\htext(1 -0.3){$\underbrace{\rule{4.5em}{0em}}$}
		\htext(1 -0.7){$\bar{x}_n$}

		\esegment
		\esegment	
	\end{texdraw}
\end{center}

\vskip 2mm

\item\label{6} For types $A_{2n}^{(2)}$, $C_{n}^{(1)}$, $A_{2n-1}^{(2)}$, the map $\tilde{\mathfrak f}_n$ is given by

\vskip 10mm

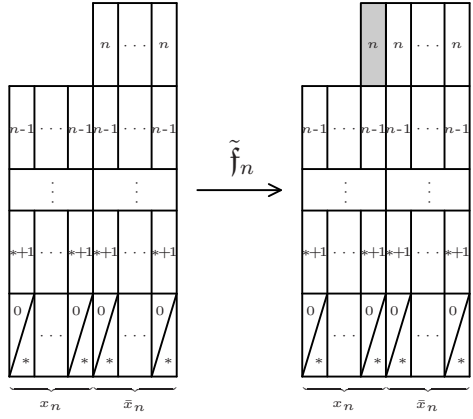
\begin{center}
	\begin{texdraw}
		\fontsize{3}{3}\selectfont
		\drawdim em
		\setunitscale 2.3
		\arrowheadsize l:0.3 w:0.3
		\arrowheadtype t:V
		\move(0 0)
		\bsegment
		\move(0 0)
		\bsegment
		\move(0 0)\lvec(2 0)
		\move(0 0)\lvec(0 2)
		\move(0 2)\lvec(2 2)
		\move(0 2)\lvec(0 7)
		\move(0 7)\lvec(2 7)
		\move(0 4)\lvec(2 4)
		\move(0 5)\lvec(2 5)
		\move(0.6 0)\lvec(0.6 4)
		\move(1.4 0)\lvec(1.4 4)
		\move(0.6 5)\lvec(0.6 7)
		\move(1.4 5)\lvec(1.4 7)
		\htext(1.05 1){$\cdots$}
		\htext(1.05 3){$\cdots$}
		\vtext(1.05 4.5){$\cdots$}
		\htext(1.05 6){$\cdots$}
		\htext(0.2 1.6){$0$}
		\htext(1.6 1.6){$0$}
		\htext(0.4 0.4){$*$}
		\htext(1.8 0.4){$*$}
		\move(0 0)\lvec(0.6 2)
		\move(1.4 0)\lvec(2 2)	
		\htext(0.3 3){$*$\!+\!$1$}
		\htext(1.7 3){$*$\!+\!$1$}
		\htext(0.3 6){$n$-$1$}
		\htext(1.7 6){$n$-$1$}
		\htext(1 -0.3){$\underbrace{\rule{4.5em}{0em}}$}
		\htext(1 -0.7){$x_n$}
		
		\esegment
		
		\move(2 0)
		\bsegment
		\move(0 0)\lvec(2 0)
		\move(0 0)\lvec(0 2)
		\move(0 2)\lvec(2 2)
		\move(0 2)\lvec(0 9)
		\move(2 0)\lvec(2 9)
		\move(0 9)\lvec(2 9)
		\move(0 5)\lvec(2 5)
		\move(0 4)\lvec(2 4)
		\move(0 7)\lvec(2 7)
		\move(0.6 0)\lvec(0.6 4)
		\move(1.4 0)\lvec(1.4 4)
		\move(0.6 5)\lvec(0.6 9)
		\move(1.4 5)\lvec(1.4 9)
		\htext(1.05 1){$\cdots$}
		\htext(1.05 3){$\cdots$}
		\vtext(1.05 4.5){$\cdots$}
		\htext(1.05 6){$\cdots$}
		\htext(1.05 8){$\cdots$}
		\htext(0.2 1.6){$0$}
		\htext(1.6 1.6){$0$}
		\htext(0.4 0.4){$*$}
		\htext(1.8 0.4){$*$}
		\move(0 0)\lvec(0.6 2)
		\move(1.4 0)\lvec(2 2)
		\htext(0.3 3){$*$\!+\!$1$}
		\htext(1.7 3){$*$\!+\!$1$}
		\htext(0.3 6){$n$-$1$}
		\htext(1.7 6){$n$-$1$}
		
		\htext(0.3 8){$n$}
		\htext(1.7 8){$n$}

		\htext(1 -0.3){$\underbrace{\rule{4.5em}{0em}}$}
		\htext(1 -0.7){$\bar{x}_n$}
		\move(2.5 4.5)\avec(4.5 4.5)
		\htext(3.5 5.3){	\fontsize{12}{12}\selectfont$\tilde{\mathfrak f}_n$}
		
		\esegment

		\esegment
		
		\move(7 0)
		\bsegment
		\move(0 0)
		\bsegment
		\move(0 0)\lvec(2 0)
		\move(0 0)\lvec(0 2)
		\move(0 2)\lvec(2 2)
		\move(0 2)\lvec(0 7)
		\move(0 7)\lvec(2 7)
		\move(0 4)\lvec(2 4)
		\move(0 5)\lvec(2 5)
		\move(0.6 0)\lvec(0.6 4)
		\move(1.4 0)\lvec(1.4 4)
		\move(0.6 5)\lvec(0.6 7)
		\move(1.4 5)\lvec(1.4 7)
		\htext(1.05 1){$\cdots$}
		\htext(1.05 3){$\cdots$}
		\vtext(1.05 4.5){$\cdots$}
		\htext(1.05 6){$\cdots$}
		\htext(0.2 1.6){$0$}
		\htext(1.6 1.6){$0$}
		\htext(0.4 0.4){$*$}
		\htext(1.8 0.4){$*$}
		\move(0 0)\lvec(0.6 2)
		\move(1.4 0)\lvec(2 2)	
		\htext(0.3 3){$*$\!+\!$1$}
		\htext(1.7 3){$*$\!+\!$1$}
		\htext(0.3 6){$n$-$1$}
		\htext(1.7 6){$n$-$1$}
		\htext(1 -0.3){$\underbrace{\rule{4.5em}{0em}}$}
		\htext(1 -0.7){$x_n$}
		\move(1.4 7)\lvec(1.4 9)\lvec(2 9)\lvec(2 7)\lvec(1.4 7) \lfill f:0.8
		\htext(1.7 8){$n$}
		\esegment
		
		\move(2 0)
		\bsegment
		\move(0 0)\lvec(2 0)
		\move(0 0)\lvec(0 2)
		\move(0 2)\lvec(2 2)
		\move(0 2)\lvec(0 9)
		\move(2 0)\lvec(2 9)
		\move(0 9)\lvec(2 9)
		\move(0 5)\lvec(2 5)
		\move(0 4)\lvec(2 4)
		\move(0 7)\lvec(2 7)
		\move(0.6 0)\lvec(0.6 4)
		\move(1.4 0)\lvec(1.4 4)
		\move(0.6 5)\lvec(0.6 9)
		\move(1.4 5)\lvec(1.4 9)
		\htext(1.05 1){$\cdots$}
		\htext(1.05 3){$\cdots$}
		\vtext(1.05 4.5){$\cdots$}
		\htext(1.05 6){$\cdots$}
		\htext(1.05 8){$\cdots$}
		\htext(0.2 1.6){$0$}
		\htext(1.6 1.6){$0$}
		\htext(0.4 0.4){$*$}
		\htext(1.8 0.4){$*$}
		\move(0 0)\lvec(0.6 2)
		\move(1.4 0)\lvec(2 2)
		\htext(0.3 3){$*$\!+\!$1$}
		\htext(1.7 3){$*$\!+\!$1$}
		\htext(0.3 6){$n$-$1$}
		\htext(1.7 6){$n$-$1$}
		\htext(0.3 8){$n$}
		\htext(1.7 8){$n$}
		\htext(1 -0.3){$\underbrace{\rule{4.5em}{0em}}$}
		\htext(1 -0.7){$\bar{x}_n$}

		\esegment

		\esegment	
	\end{texdraw}
\end{center}

\item\label{7} For type $D_{n}^{(1)}$, 

\vskip 2mm

if $x_n\ge 0$, $\bar{x}_n=0$, the map $\tilde{\mathfrak f}_{n-1}$ is given by

\vskip 10mm

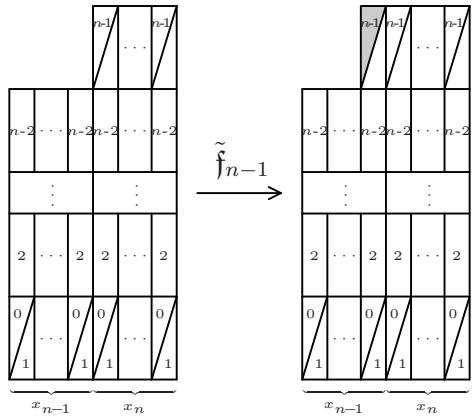
\begin{center}
	\begin{texdraw}
		\fontsize{3}{3}\selectfont
		\drawdim em
		\setunitscale 2.3
		\arrowheadsize l:0.3 w:0.3
		\arrowheadtype t:V
		\move(0 0)
		\bsegment
		\move(0 0)
		\bsegment
		\move(0 0)\lvec(2 0)
		\move(0 0)\lvec(0 2)
		\move(0 2)\lvec(2 2)
		\move(0 2)\lvec(0 7)
		\move(0 7)\lvec(2 7)
		\move(0 4)\lvec(2 4)
		\move(0 5)\lvec(2 5)
		\move(0.6 0)\lvec(0.6 4)
		\move(1.4 0)\lvec(1.4 4)
		\move(0.6 5)\lvec(0.6 7)
		\move(1.4 5)\lvec(1.4 7)
		\htext(1.05 1){$\cdots$}
		\htext(1.05 3){$\cdots$}
		\vtext(1.05 4.5){$\cdots$}
		\htext(1.05 6){$\cdots$}
		\htext(0.2 1.6){$0$}
		\htext(1.6 1.6){$0$}
		\htext(0.4 0.4){$1$}
		\htext(1.8 0.4){$1$}
		\move(0 0)\lvec(0.6 2)
		\move(1.4 0)\lvec(2 2)	
		\htext(0.3 3){$2$}
		\htext(1.7 3){$2$}
		\htext(0.3 6){$n$-$2$}
		\htext(1.7 6){$n$-$2$}
		\htext(1 -0.3){$\underbrace{\rule{4.5em}{0em}}$}
		\htext(1 -0.7){$x_{n\!-\!1}$}
		
		\esegment
		
		\move(2 0)
		\bsegment
		\move(0 0)\lvec(2 0)
		\move(0 0)\lvec(0 2)
		\move(0 2)\lvec(2 2)
		\move(0 2)\lvec(0 9)
		\move(2 0)\lvec(2 9)
		\move(0 9)\lvec(2 9)
		\move(0 5)\lvec(2 5)
		\move(0 4)\lvec(2 4)
		\move(0 7)\lvec(2 7)
		\move(0.6 0)\lvec(0.6 4)
		\move(1.4 0)\lvec(1.4 4)
		\move(0.6 5)\lvec(0.6 9)
		\move(1.4 5)\lvec(1.4 9)
		\htext(1.05 1){$\cdots$}
		\htext(1.05 3){$\cdots$}
		\vtext(1.05 4.5){$\cdots$}
		\htext(1.05 6){$\cdots$}
		\htext(1.05 8){$\cdots$}
		\htext(0.2 1.6){$0$}
		\htext(1.6 1.6){$0$}
		\htext(0.4 0.4){$1$}
		\htext(1.8 0.4){$1$}
		\move(0 0)\lvec(0.6 2)
		\move(1.4 0)\lvec(2 2)
		\htext(0.3 3){$2$}
		\htext(1.7 3){$2$}
		\htext(0.3 6){$n$-$2$}
		\htext(1.7 6){$n$-$2$}
		\move(0 7)\lvec(0.6 9)
		\move(1.4 7)\lvec(2 9)
		\htext(1.62 8.6){$n$\!-\!$1$}
		\htext(0.22 8.6){$n$\!-\!$1$}
		\htext(1 -0.3){$\underbrace{\rule{4.5em}{0em}}$}
		\htext(1 -0.7){$x_n$}
		\move(2.5 4.5)\avec(4.5 4.5)
		\htext(3.5 5.3){	\fontsize{12}{12}\selectfont$\tilde{\mathfrak f}_{n-1}$}
		
		\esegment

		\esegment
		
		\move(7 0)
		\bsegment
		\move(0 0)
		\bsegment
		\move(0 0)\lvec(2 0)
		\move(0 0)\lvec(0 2)
		\move(0 2)\lvec(2 2)
		\move(0 2)\lvec(0 7)
		\move(0 7)\lvec(2 7)
		\move(0 4)\lvec(2 4)
		\move(0 5)\lvec(2 5)
		\move(0.6 0)\lvec(0.6 4)
		\move(1.4 0)\lvec(1.4 4)
		\move(0.6 5)\lvec(0.6 7)
		\move(1.4 5)\lvec(1.4 7)
		\htext(1.05 1){$\cdots$}
		\htext(1.05 3){$\cdots$}
		\vtext(1.05 4.5){$\cdots$}
		\htext(1.05 6){$\cdots$}
		\htext(0.2 1.6){$0$}
		\htext(1.6 1.6){$0$}
		\htext(0.4 0.4){$1$}
		\htext(1.8 0.4){$1$}
		\move(0 0)\lvec(0.6 2)
		\move(1.4 0)\lvec(2 2)	
		\htext(0.3 3){$2$}
		\htext(1.7 3){$2$}
		\htext(0.3 6){$n$-$2$}
		\htext(1.7 6){$n$-$2$}
		\htext(1 -0.3){$\underbrace{\rule{4.5em}{0em}}$}
		\htext(1 -0.7){$x_{n\!-\!1}$}
		\move(1.4 7)\lvec(1.4 9)\lvec(2 9)\lvec(1.4 7) \lfill f:0.8
		\htext(1.62 8.6){$n$\!-\!$1$}
		\esegment
		
		\move(2 0)
		\bsegment
		\move(0 0)\lvec(2 0)
		\move(0 0)\lvec(0 2)
		\move(0 2)\lvec(2 2)
		\move(0 2)\lvec(0 9)
		\move(2 0)\lvec(2 9)
		\move(0 9)\lvec(2 9)
		\move(0 5)\lvec(2 5)
		\move(0 4)\lvec(2 4)
		\move(0 7)\lvec(2 7)
		\move(0.6 0)\lvec(0.6 4)
		\move(1.4 0)\lvec(1.4 4)
		\move(0.6 5)\lvec(0.6 9)
		\move(1.4 5)\lvec(1.4 9)
		\htext(1.05 1){$\cdots$}
		\htext(1.05 3){$\cdots$}
		\vtext(1.05 4.5){$\cdots$}
		\htext(1.05 6){$\cdots$}
		\htext(1.05 8){$\cdots$}
		\htext(0.2 1.6){$0$}
		\htext(1.6 1.6){$0$}
		\htext(0.4 0.4){$1$}
		\htext(1.8 0.4){$1$}
		\move(0 0)\lvec(0.6 2)
		\move(1.4 0)\lvec(2 2)
		\htext(0.3 3){$2$}
		\htext(1.7 3){$2$}
		\htext(0.3 6){$n$-$2$}
		\htext(1.7 6){$n$-$2$}
		\move(0 7)\lvec(0.6 9)
		\move(1.4 7)\lvec(2 9)
		\htext(1.62 8.6){$n$\!-\!$1$}
		\htext(0.22 8.6){$n$\!-\!$1$}
		\htext(1 -0.3){$\underbrace{\rule{4.5em}{0em}}$}
		\htext(1 -0.7){$x_n$}

		\esegment

		\esegment	
	\end{texdraw}
\end{center}

\vskip 2mm

if $x_n= 0$, $\bar{x}_n\ge 1$, the map $\tilde{\mathfrak f}_{n-1}$ is given by

\vskip 10mm

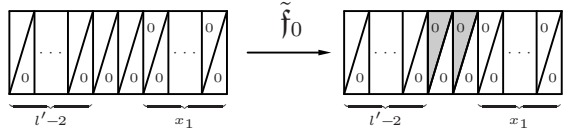
\begin{center}
	\begin{texdraw}
		\fontsize{3}{3}\selectfont
		\drawdim em
		\setunitscale 2.3
		\arrowheadsize l:0.3 w:0.3
		\arrowheadtype t:V
		\move(0 0)
		\bsegment
		\move(0 0)
		\bsegment
		\move(0 0)\lvec(2 0)
		\move(0 0)\lvec(0 2)
		\move(0 2)\lvec(2 2)
		\move(0 2)\lvec(0 7)
		\move(0 7)\lvec(2 7)
		\move(0 4)\lvec(2 4)
		\move(0 5)\lvec(2 5)
		\move(0.6 0)\lvec(0.6 4)
		\move(1.4 0)\lvec(1.4 4)
		\move(0.6 5)\lvec(0.6 7)
		\move(1.4 5)\lvec(1.4 7)
		\htext(1.05 1){$\cdots$}
		\htext(1.05 3){$\cdots$}
		\vtext(1.05 4.5){$\cdots$}
		\htext(1.05 6){$\cdots$}
		\htext(1.05 8){$\cdots$}
		\htext(0.2 1.6){$0$}
		\htext(1.6 1.6){$0$}
		\htext(0.4 0.4){$1$}
		\htext(1.8 0.4){$1$}
		\move(0 0)\lvec(0.6 2)
		\move(1.4 0)\lvec(2 2)	
		\htext(0.3 3){$2$}
		\htext(1.7 3){$2$}
		\htext(0.3 6){$n$-$2$}
		\htext(1.7 6){$n$-$2$}
		\htext(1 -0.3){$\underbrace{\rule{4.5em}{0em}}$}
		\htext(1 -0.7){$\bar{x}_n$}

		\htext(0.4 7.4){$n$}
		\htext(1.8 7.4){$n$}
		
		\move(0 7)\lvec(0 9)
		\move(0.6 7)\lvec(0.6 9)
		\move(1.4 7)\lvec(1.4 9)
		\move(0 9)\lvec(2 9)			
		\move(0 7)\lvec(0.6 9)
		\move(1.4 7)\lvec(2 9)
		\esegment
		
		\move(2 0)
		\bsegment
		\move(0 0)\lvec(2 0)
		\move(0 0)\lvec(0 2)
		\move(0 2)\lvec(2 2)
		\move(0 2)\lvec(0 9)
		\move(2 0)\lvec(2 9)
		\move(0 9)\lvec(2 9)
		\move(0 5)\lvec(2 5)
		\move(0 4)\lvec(2 4)
		\move(0 7)\lvec(2 7)
		\move(0.6 0)\lvec(0.6 4)
		\move(1.4 0)\lvec(1.4 4)
		\move(0.6 5)\lvec(0.6 9)
		\move(1.4 5)\lvec(1.4 9)
		\htext(1.05 1){$\cdots$}
		\htext(1.05 3){$\cdots$}
		\vtext(1.05 4.5){$\cdots$}
		\htext(1.05 6){$\cdots$}
		\htext(1.05 8){$\cdots$}
		\htext(0.2 1.6){$0$}
		\htext(1.6 1.6){$0$}
		\htext(0.4 0.4){$1$}
		\htext(1.8 0.4){$1$}
		\move(0 0)\lvec(0.6 2)
		\move(1.4 0)\lvec(2 2)
		\htext(0.3 3){$2$}
		\htext(1.7 3){$2$}
		\htext(0.3 6){$n$-$2$}
		\htext(1.7 6){$n$-$2$}
		\move(0 7)\lvec(0.6 9)
		\move(1.4 7)\lvec(2 9)
		\htext(1.62 8.6){$n$\!-\!$1$}
		\htext(0.22 8.6){$n$\!-\!$1$}
		\htext(1 -0.3){$\underbrace{\rule{4.5em}{0em}}$}
		\htext(1 -0.7){$\bar{x}_{n\!-\!1}$}
		\move(2.5 4.5)\avec(4.5 4.5)
		\htext(3.5 5.3){	\fontsize{12}{12}\selectfont$\tilde{\mathfrak f}_{n-1}$}

		\htext(0.4 7.4){$n$}
		\htext(1.8 7.4){$n$}
		
		\esegment

		\esegment
		
		\move(7 0)
		\bsegment
		\move(0 0)
		\bsegment
		\move(0 0)\lvec(2 0)
		\move(0 0)\lvec(0 2)
		\move(0 2)\lvec(2 2)
		\move(0 2)\lvec(0 7)
		\move(0 7)\lvec(2 7)
		\move(0 4)\lvec(2 4)
		\move(0 5)\lvec(2 5)
		\move(0.6 0)\lvec(0.6 4)
		\move(1.4 0)\lvec(1.4 4)
		\move(0.6 5)\lvec(0.6 7)
		\move(1.4 5)\lvec(1.4 7)
		\htext(1.05 1){$\cdots$}
		\htext(1.05 3){$\cdots$}
		\vtext(1.05 4.5){$\cdots$}
		\htext(1.05 6){$\cdots$}
		\htext(1.05 8){$\cdots$}
		\htext(0.2 1.6){$0$}
		\htext(1.6 1.6){$0$}
		\htext(0.4 0.4){$1$}
		\htext(1.8 0.4){$1$}
		\move(0 0)\lvec(0.6 2)
		\move(1.4 0)\lvec(2 2)	
		\htext(0.3 3){$2$}
		\htext(1.7 3){$2$}
		\htext(0.3 6){$n$-$2$}
		\htext(1.7 6){$n$-$2$}
		\htext(1 -0.3){$\underbrace{\rule{4.5em}{0em}}$}
		\htext(1 -0.7){$\bar{x}_n$}

		\htext(0.4 7.4){$n$}
		\htext(1.8 7.4){$n$}
		
		\move(0 7)\lvec(0 9)
		\move(0.6 7)\lvec(0.6 9)
		\move(1.4 7)\lvec(1.4 9)
		\move(0 9)\lvec(2 9)			
		\move(0 7)\lvec(0.6 9)
		\move(1.4 7)\lvec(2 9)
		
		\move(1.4 7)\lvec(1.4 9)\lvec(2 9)\lvec(1.4 7)\lfill f:0.8
		\htext(1.62 8.6){$n$\!-\!$1$}
		
		\esegment
		
		\move(2 0)
		\bsegment
		\move(0 0)\lvec(2 0)
		\move(0 0)\lvec(0 2)
		\move(0 2)\lvec(2 2)
		\move(0 2)\lvec(0 9)
		\move(2 0)\lvec(2 9)
		\move(0 9)\lvec(2 9)
		\move(0 5)\lvec(2 5)
		\move(0 4)\lvec(2 4)
		\move(0 7)\lvec(2 7)
		\move(0.6 0)\lvec(0.6 4)
		\move(1.4 0)\lvec(1.4 4)
		\move(0.6 5)\lvec(0.6 9)
		\move(1.4 5)\lvec(1.4 9)
		\htext(1.05 1){$\cdots$}
		\htext(1.05 3){$\cdots$}
		\vtext(1.05 4.5){$\cdots$}
		\htext(1.05 6){$\cdots$}
		\htext(1.05 8){$\cdots$}
		\htext(0.2 1.6){$0$}
		\htext(1.6 1.6){$0$}
		\htext(0.4 0.4){$1$}
		\htext(1.8 0.4){$1$}
		\move(0 0)\lvec(0.6 2)
		\move(1.4 0)\lvec(2 2)
		\htext(0.3 3){$2$}
		\htext(1.7 3){$2$}
		\htext(0.3 6){$n$-$2$}
		\htext(1.7 6){$n$-$2$}
		\move(0 7)\lvec(0.6 9)
		\move(1.4 7)\lvec(2 9)
		\htext(1.62 8.6){$n$\!-\!$1$}
		\htext(0.22 8.6){$n$\!-\!$1$}
		\htext(1 -0.3){$\underbrace{\rule{4.5em}{0em}}$}
		\htext(1 -0.7){$\bar{x}_{n\!-\!1}$}		
		\htext(0.4 7.4){$n$}
		\htext(1.8 7.4){$n$}

		\esegment
		\esegment	
	\end{texdraw}
\end{center}

\vskip 2mm

if $x_n\ge 1$, $\bar{x}_n=0$, the map $\tilde{\mathfrak f}_n$ is given by

\vskip 10mm

\begin{center}
	\begin{texdraw}
		\fontsize{3}{3}\selectfont
		\drawdim em
		\setunitscale 2.3
		\arrowheadsize l:0.3 w:0.3
		\arrowheadtype t:V
		\move(0 0)
		\bsegment
		
		\move(-2 0)
		\bsegment
		\move(0 0)\lvec(2 0)
		\move(0 0)\lvec(0 2)
		\move(0 2)\lvec(2 2)
		\move(0 2)\lvec(0 7)
		\move(0 7)\lvec(2 7)
		\move(0 4)\lvec(2 4)
		\move(0 5)\lvec(2 5)
		\move(0.6 0)\lvec(0.6 4)
		\move(1.4 0)\lvec(1.4 4)
		\move(0.6 5)\lvec(0.6 7)
		\move(1.4 5)\lvec(1.4 7)
		\htext(1.05 1){$\cdots$}
		\htext(1.05 3){$\cdots$}
		\vtext(1.05 4.5){$\cdots$}
		\htext(1.05 6){$\cdots$}
		\htext(1.05 8){$\cdots$}
		\htext(0.2 1.6){$0$}
		\htext(1.6 1.6){$0$}
		\htext(0.4 0.4){$1$}
		\htext(1.8 0.4){$1$}
		\move(0 0)\lvec(0.6 2)
		\move(1.4 0)\lvec(2 2)	
		\htext(0.3 3){$2$}
		\htext(1.7 3){$2$}
		\htext(0.3 6){$n$-$2$}
		\htext(1.7 6){$n$-$2$}
		\htext(1 -0.3){$\underbrace{\rule{4.5em}{0em}}$}
		\htext(1 -0.7){$x_n$}
		
		\htext(1.62 8.6){$n$\!-\!$1$}
		\htext(0.22 8.6){$n$\!-\!$1$}

		\move(0 7)\lvec(0 9)
		\move(0.6 7)\lvec(0.6 9)
		\move(1.4 7)\lvec(1.4 9)
		\move(0 9)\lvec(2 9)			
		\move(0 7)\lvec(0.6 9)
		\move(1.4 7)\lvec(2 9)
		\esegment
		
		\move(0 0)
		\bsegment
		\move(0 0)\lvec(2 0)
		\move(0 0)\lvec(0 2)
		\move(0 2)\lvec(2 2)
		\move(0 2)\lvec(0 7)
		\move(0 7)\lvec(2 7)
		\move(0 4)\lvec(2 4)
		\move(0 5)\lvec(2 5)
		\move(0.6 0)\lvec(0.6 4)
		\move(1.4 0)\lvec(1.4 4)
		\move(0.6 5)\lvec(0.6 7)
		\move(1.4 5)\lvec(1.4 7)
		\htext(1.05 1){$\cdots$}
		\htext(1.05 3){$\cdots$}
		\vtext(1.05 4.5){$\cdots$}
		\htext(1.05 6){$\cdots$}
		\htext(1.05 8){$\cdots$}
		\htext(0.2 1.6){$0$}
		\htext(1.6 1.6){$0$}
		\htext(0.4 0.4){$1$}
		\htext(1.8 0.4){$1$}
		\move(0 0)\lvec(0.6 2)
		\move(1.4 0)\lvec(2 2)	
		\htext(0.3 3){$2$}
		\htext(1.7 3){$2$}
		\htext(0.3 6){$n$-$2$}
		\htext(1.7 6){$n$-$2$}
		\htext(1 -0.3){$\underbrace{\rule{4.5em}{0em}}$}
		\htext(1 -0.7){$\bar{x}_n$}

		\htext(0.4 7.4){$n$}
		\htext(1.8 7.4){$n$}
		
		\move(0 7)\lvec(0 9)
		\move(0.6 7)\lvec(0.6 9)
		\move(1.4 7)\lvec(1.4 9)
		\move(0 9)\lvec(2 9)			
		\move(0 7)\lvec(0.6 9)
		\move(1.4 7)\lvec(2 9)
		\esegment
		
		\move(2 0)
		\bsegment
		\move(0 0)\lvec(2 0)
		\move(0 0)\lvec(0 2)
		\move(0 2)\lvec(2 2)
		\move(0 2)\lvec(0 9)
		\move(2 0)\lvec(2 9)
		\move(0 9)\lvec(2 9)
		\move(0 5)\lvec(2 5)
		\move(0 4)\lvec(2 4)
		\move(0 7)\lvec(2 7)
		\move(0.6 0)\lvec(0.6 4)
		\move(1.4 0)\lvec(1.4 4)
		\move(0.6 5)\lvec(0.6 9)
		\move(1.4 5)\lvec(1.4 9)
		\htext(1.05 1){$\cdots$}
		\htext(1.05 3){$\cdots$}
		\vtext(1.05 4.5){$\cdots$}
		\htext(1.05 6){$\cdots$}
		\htext(1.05 8){$\cdots$}
		\htext(0.2 1.6){$0$}
		\htext(1.6 1.6){$0$}
		\htext(0.4 0.4){$1$}
		\htext(1.8 0.4){$1$}
		\move(0 0)\lvec(0.6 2)
		\move(1.4 0)\lvec(2 2)
		\htext(0.3 3){$2$}
		\htext(1.7 3){$2$}
		\htext(0.3 6){$n$-$2$}
		\htext(1.7 6){$n$-$2$}
		\move(0 7)\lvec(0.6 9)
		\move(1.4 7)\lvec(2 9)
		\htext(1.62 8.6){$n$\!-\!$1$}
		\htext(0.22 8.6){$n$\!-\!$1$}
		\htext(1 -0.3){$\underbrace{\rule{4.5em}{0em}}$}
		\htext(1 -0.7){$\bar{x}_{n\!-\!1}$}
		\move(2.5 4.5)\avec(4.5 4.5)
		\htext(3.5 5.3){	\fontsize{12}{12}\selectfont$\tilde{\mathfrak f}_{n}$}

		\htext(0.4 7.4){$n$}
		\htext(1.8 7.4){$n$}
		
		\esegment

		\esegment
		
		\move(9 0)
		\bsegment
		
		\move(-2 0)
		\bsegment
		\move(0 0)\lvec(2 0)
		\move(0 0)\lvec(0 2)
		\move(0 2)\lvec(2 2)
		\move(0 2)\lvec(0 7)
		\move(0 7)\lvec(2 7)
		\move(0 4)\lvec(2 4)
		\move(0 5)\lvec(2 5)
		\move(0.6 0)\lvec(0.6 4)
		\move(1.4 0)\lvec(1.4 4)
		\move(0.6 5)\lvec(0.6 7)
		\move(1.4 5)\lvec(1.4 7)
		\htext(1.05 1){$\cdots$}
		\htext(1.05 3){$\cdots$}
		\vtext(1.05 4.5){$\cdots$}
		\htext(1.05 6){$\cdots$}
		\htext(1.05 8){$\cdots$}
		\htext(0.2 1.6){$0$}
		\htext(1.6 1.6){$0$}
		\htext(0.4 0.4){$1$}
		\htext(1.8 0.4){$1$}
		\move(0 0)\lvec(0.6 2)
		\move(1.4 0)\lvec(2 2)	
		\htext(0.3 3){$2$}
		\htext(1.7 3){$2$}
		\htext(0.3 6){$n$-$2$}
		\htext(1.7 6){$n$-$2$}
		\htext(1 -0.3){$\underbrace{\rule{4.5em}{0em}}$}
		\htext(1 -0.7){$x_n$}
		
		\htext(1.62 8.6){$n$\!-\!$1$}
		\htext(0.22 8.6){$n$\!-\!$1$}

		\move(0 7)\lvec(0 9)
		\move(0.6 7)\lvec(0.6 9)
		\move(1.4 7)\lvec(1.4 9)
		\move(0 9)\lvec(2 9)			
		\move(0 7)\lvec(0.6 9)

		\htext(1.8 7.4){$n$}
		\move(1.4 7)\lvec(2 9)\lvec(2 7)\lvec(1.4 7)\lfill f:0.8
		
		\esegment
		
		\move(0 0)
		\bsegment
		\move(0 0)\lvec(2 0)
		\move(0 0)\lvec(0 2)
		\move(0 2)\lvec(2 2)
		\move(0 2)\lvec(0 7)
		\move(0 7)\lvec(2 7)
		\move(0 4)\lvec(2 4)
		\move(0 5)\lvec(2 5)
		\move(0.6 0)\lvec(0.6 4)
		\move(1.4 0)\lvec(1.4 4)
		\move(0.6 5)\lvec(0.6 7)
		\move(1.4 5)\lvec(1.4 7)
		\htext(1.05 1){$\cdots$}
		\htext(1.05 3){$\cdots$}
		\vtext(1.05 4.5){$\cdots$}
		\htext(1.05 6){$\cdots$}
		\htext(1.05 8){$\cdots$}
		\htext(0.2 1.6){$0$}
		\htext(1.6 1.6){$0$}
		\htext(0.4 0.4){$1$}
		\htext(1.8 0.4){$1$}
		\move(0 0)\lvec(0.6 2)
		\move(1.4 0)\lvec(2 2)	
		\htext(0.3 3){$2$}
		\htext(1.7 3){$2$}
		\htext(0.3 6){$n$-$2$}
		\htext(1.7 6){$n$-$2$}
		\htext(1 -0.3){$\underbrace{\rule{4.5em}{0em}}$}
		\htext(1 -0.7){$\bar{x}_n$}

		\htext(0.4 7.4){$n$}
		\htext(1.8 7.4){$n$}
		
		\move(0 7)\lvec(0 9)
		\move(0.6 7)\lvec(0.6 9)
		\move(1.4 7)\lvec(1.4 9)
		\move(0 9)\lvec(2 9)			
		\move(0 7)\lvec(0.6 9)
		\move(1.4 7)\lvec(2 9)
		\esegment
		
		\move(2 0)
		\bsegment
		\move(0 0)\lvec(2 0)
		\move(0 0)\lvec(0 2)
		\move(0 2)\lvec(2 2)
		\move(0 2)\lvec(0 9)
		\move(2 0)\lvec(2 9)
		\move(0 9)\lvec(2 9)
		\move(0 5)\lvec(2 5)
		\move(0 4)\lvec(2 4)
		\move(0 7)\lvec(2 7)
		\move(0.6 0)\lvec(0.6 4)
		\move(1.4 0)\lvec(1.4 4)
		\move(0.6 5)\lvec(0.6 9)
		\move(1.4 5)\lvec(1.4 9)
		\htext(1.05 1){$\cdots$}
		\htext(1.05 3){$\cdots$}
		\vtext(1.05 4.5){$\cdots$}
		\htext(1.05 6){$\cdots$}
		\htext(1.05 8){$\cdots$}
		\htext(0.2 1.6){$0$}
		\htext(1.6 1.6){$0$}
		\htext(0.4 0.4){$1$}
		\htext(1.8 0.4){$1$}
		\move(0 0)\lvec(0.6 2)
		\move(1.4 0)\lvec(2 2)
		\htext(0.3 3){$2$}
		\htext(1.7 3){$2$}
		\htext(0.3 6){$n$-$2$}
		\htext(1.7 6){$n$-$2$}
		\move(0 7)\lvec(0.6 9)
		\move(1.4 7)\lvec(2 9)
		\htext(1.62 8.6){$n$\!-\!$1$}
		\htext(0.22 8.6){$n$\!-\!$1$}
		\htext(1 -0.3){$\underbrace{\rule{4.5em}{0em}}$}
		\htext(1 -0.7){$\bar{x}_{n\!-\!1}$}
		
		\htext(0.4 7.4){$n$}
		\htext(1.8 7.4){$n$}
		
		\esegment
		
		\esegment	
	\end{texdraw}
\end{center}

\vskip 2mm

if $x_n=0$, $\bar{x}_n\ge 0$, the map $\tilde{\mathfrak f}_n$ is given by

\vskip 10mm

\begin{center}
	\begin{texdraw}
		\fontsize{3}{3}\selectfont
		\drawdim em
		\setunitscale 2.3
		\arrowheadsize l:0.3 w:0.3
		\arrowheadtype t:V
		\move(0 0)
		\bsegment
		
		\move(-2 0)
		\bsegment
		\move(0 0)\lvec(2 0)
		\move(0 0)\lvec(0 2)
		\move(0 2)\lvec(2 2)
		\move(0 2)\lvec(0 7)
		\move(0 7)\lvec(2 7)
		\move(0 4)\lvec(2 4)
		\move(0 5)\lvec(2 5)
		\move(0.6 0)\lvec(0.6 4)
		\move(1.4 0)\lvec(1.4 4)
		\move(0.6 5)\lvec(0.6 7)
		\move(1.4 5)\lvec(1.4 7)
		\htext(1.05 1){$\cdots$}
		\htext(1.05 3){$\cdots$}
		\vtext(1.05 4.5){$\cdots$}
		\htext(1.05 6){$\cdots$}
		
		\htext(0.2 1.6){$0$}
		\htext(1.6 1.6){$0$}
		\htext(0.4 0.4){$1$}
		\htext(1.8 0.4){$1$}
		\move(0 0)\lvec(0.6 2)
		\move(1.4 0)\lvec(2 2)	
		\htext(0.3 3){$2$}
		\htext(1.7 3){$2$}
		\htext(0.3 6){$n$-$2$}
		\htext(1.7 6){$n$-$2$}
		\htext(1 -0.3){$\underbrace{\rule{4.5em}{0em}}$}
		\htext(1 -0.7){$x_{n\!-\!1}$}
		\esegment
		
		\move(0 0)
		\bsegment
		\move(0 0)\lvec(2 0)
		\move(0 0)\lvec(0 2)
		\move(0 2)\lvec(2 2)
		\move(0 2)\lvec(0 7)
		\move(0 7)\lvec(2 7)
		\move(0 4)\lvec(2 4)
		\move(0 5)\lvec(2 5)
		\move(0.6 0)\lvec(0.6 4)
		\move(1.4 0)\lvec(1.4 4)
		\move(0.6 5)\lvec(0.6 7)
		\move(1.4 5)\lvec(1.4 7)
		\htext(1.05 1){$\cdots$}
		\htext(1.05 3){$\cdots$}
		\vtext(1.05 4.5){$\cdots$}
		\htext(1.05 6){$\cdots$}
		\htext(1.05 8){$\cdots$}
		\htext(0.2 1.6){$0$}
		\htext(1.6 1.6){$0$}
		\htext(0.4 0.4){$1$}
		\htext(1.8 0.4){$1$}
		\move(0 0)\lvec(0.6 2)
		\move(1.4 0)\lvec(2 2)	
		\htext(0.3 3){$2$}
		\htext(1.7 3){$2$}
		\htext(0.3 6){$n$-$2$}
		\htext(1.7 6){$n$-$2$}
		\htext(1 -0.3){$\underbrace{\rule{4.5em}{0em}}$}
		\htext(1 -0.7){$x_n$}

		\move(0 7)\lvec(0 9)
		\move(0.6 7)\lvec(0.6 9)
		\move(1.4 7)\lvec(1.4 9)
		\move(0 9)\lvec(2 9)			
		\move(0 7)\lvec(0.6 9)
		\move(1.4 7)\lvec(2 9)
		\htext(1.62 8.6){$n$\!-\!$1$}
		\htext(0.22 8.6){$n$\!-\!$1$}
		\esegment
		
		\move(2 0)
		\bsegment
		\move(0 0)\lvec(2 0)
		\move(0 0)\lvec(0 2)
		\move(0 2)\lvec(2 2)
		\move(0 2)\lvec(0 9)
		\move(2 0)\lvec(2 9)
		\move(0 9)\lvec(2 9)
		\move(0 5)\lvec(2 5)
		\move(0 4)\lvec(2 4)
		\move(0 7)\lvec(2 7)
		\move(0.6 0)\lvec(0.6 4)
		\move(1.4 0)\lvec(1.4 4)
		\move(0.6 5)\lvec(0.6 9)
		\move(1.4 5)\lvec(1.4 9)
		\htext(1.05 1){$\cdots$}
		\htext(1.05 3){$\cdots$}
		\vtext(1.05 4.5){$\cdots$}
		\htext(1.05 6){$\cdots$}
		\htext(1.05 8){$\cdots$}
		\htext(0.2 1.6){$0$}
		\htext(1.6 1.6){$0$}
		\htext(0.4 0.4){$1$}
		\htext(1.8 0.4){$1$}
		\move(0 0)\lvec(0.6 2)
		\move(1.4 0)\lvec(2 2)
		\htext(0.3 3){$2$}
		\htext(1.7 3){$2$}
		\htext(0.3 6){$n$-$2$}
		\htext(1.7 6){$n$-$2$}
		\move(0 7)\lvec(0.6 9)
		\move(1.4 7)\lvec(2 9)
		
		\htext(1 -0.3){$\underbrace{\rule{4.5em}{0em}}$}
		\htext(1 -0.7){$\bar{x}_n$}
		
		\htext(0.4 7.4){$n$}
		\htext(1.8 7.4){$n$}
		\move(2.5 4.5)\avec(4.5 4.5)
		\htext(3.5 5.3){	\fontsize{12}{12}\selectfont$\tilde{\mathfrak f}_{n}$}
		\esegment
		
		\esegment	
		
		\move(9 0)
		\bsegment
		
		\move(-2 0)
		\bsegment
		\move(0 0)\lvec(2 0)
		\move(0 0)\lvec(0 2)
		\move(0 2)\lvec(2 2)
		\move(0 2)\lvec(0 7)
		\move(0 7)\lvec(2 7)
		\move(0 4)\lvec(2 4)
		\move(0 5)\lvec(2 5)
		\move(0.6 0)\lvec(0.6 4)
		\move(1.4 0)\lvec(1.4 4)
		\move(0.6 5)\lvec(0.6 7)
		\move(1.4 5)\lvec(1.4 7)
		\htext(1.05 1){$\cdots$}
		\htext(1.05 3){$\cdots$}
		\vtext(1.05 4.5){$\cdots$}
		\htext(1.05 6){$\cdots$}
		\htext(0.2 1.6){$0$}
		\htext(1.6 1.6){$0$}
		\htext(0.4 0.4){$1$}
		\htext(1.8 0.4){$1$}
		\move(0 0)\lvec(0.6 2)
		\move(1.4 0)\lvec(2 2)	
		\htext(0.3 3){$2$}
		\htext(1.7 3){$2$}
		\htext(0.3 6){$n$-$2$}
		\htext(1.7 6){$n$-$2$}
		\htext(1 -0.3){$\underbrace{\rule{4.5em}{0em}}$}
		\htext(1 -0.7){$x_{n\!-\!1}$}

		\move(1.4 7)\lvec(1.4 9)
		\move(1.4 9)\lvec(2 9)

		\htext(1.8 7.4){$n$}
		\move(1.4 7)\lvec(2 9)\lvec(2 7)\lvec(1.4 7)\lfill f:0.8
		
		\esegment
		
		\move(0 0)
		\bsegment
		\move(0 0)\lvec(2 0)
		\move(0 0)\lvec(0 2)
		\move(0 2)\lvec(2 2)
		\move(0 2)\lvec(0 7)
		\move(0 7)\lvec(2 7)
		\move(0 4)\lvec(2 4)
		\move(0 5)\lvec(2 5)
		\move(0.6 0)\lvec(0.6 4)
		\move(1.4 0)\lvec(1.4 4)
		\move(0.6 5)\lvec(0.6 7)
		\move(1.4 5)\lvec(1.4 7)
		\htext(1.05 1){$\cdots$}
		\htext(1.05 3){$\cdots$}
		\vtext(1.05 4.5){$\cdots$}
		\htext(1.05 6){$\cdots$}
		\htext(1.05 8){$\cdots$}
		\htext(0.2 1.6){$0$}
		\htext(1.6 1.6){$0$}
		\htext(0.4 0.4){$1$}
		\htext(1.8 0.4){$1$}
		\move(0 0)\lvec(0.6 2)
		\move(1.4 0)\lvec(2 2)	
		\htext(0.3 3){$2$}
		\htext(1.7 3){$2$}
		\htext(0.3 6){$n$-$2$}
		\htext(1.7 6){$n$-$2$}
		\htext(1 -0.3){$\underbrace{\rule{4.5em}{0em}}$}
		\htext(1 -0.7){$x_n$}

		\move(0 7)\lvec(0 9)
		\move(0.6 7)\lvec(0.6 9)
		\move(1.4 7)\lvec(1.4 9)
		\move(0 9)\lvec(2 9)			
		\move(0 7)\lvec(0.6 9)
		\move(1.4 7)\lvec(2 9)
		\htext(1.62 8.6){$n$\!-\!$1$}
		\htext(0.22 8.6){$n$\!-\!$1$}
		\esegment
		
		\move(2 0)
		\bsegment
		\move(0 0)\lvec(2 0)
		\move(0 0)\lvec(0 2)
		\move(0 2)\lvec(2 2)
		\move(0 2)\lvec(0 9)
		\move(2 0)\lvec(2 9)
		\move(0 9)\lvec(2 9)
		\move(0 5)\lvec(2 5)
		\move(0 4)\lvec(2 4)
		\move(0 7)\lvec(2 7)
		\move(0.6 0)\lvec(0.6 4)
		\move(1.4 0)\lvec(1.4 4)
		\move(0.6 5)\lvec(0.6 9)
		\move(1.4 5)\lvec(1.4 9)
		\htext(1.05 1){$\cdots$}
		\htext(1.05 3){$\cdots$}
		\vtext(1.05 4.5){$\cdots$}
		\htext(1.05 6){$\cdots$}
		\htext(1.05 8){$\cdots$}
		\htext(0.2 1.6){$0$}
		\htext(1.6 1.6){$0$}
		\htext(0.4 0.4){$1$}
		\htext(1.8 0.4){$1$}
		\move(0 0)\lvec(0.6 2)
		\move(1.4 0)\lvec(2 2)
		\htext(0.3 3){$2$}
		\htext(1.7 3){$2$}
		\htext(0.3 6){$n$-$2$}
		\htext(1.7 6){$n$-$2$}
		\move(0 7)\lvec(0.6 9)
		\move(1.4 7)\lvec(2 9)
		
		\htext(1 -0.3){$\underbrace{\rule{4.5em}{0em}}$}
		\htext(1 -0.7){$\bar{x}_n$}
		
		\htext(0.4 7.4){$n$}
		\htext(1.8 7.4){$n$}
		
		\esegment
		
		\esegment	
	\end{texdraw}
\end{center}

\vskip 6mm

\item\label{8} For type $C_n^{(1)}$, 

\vskip 2mm

if $x_1\ge \bar{x}_1$, the map $\tilde{\mathfrak f}_0$ is given by

\vskip 10mm

\begin{center}
	\begin{texdraw}
		\fontsize{3}{3}\selectfont
		\drawdim em
		\setunitscale 2.3
		\move(0 0)
		\bsegment
		\move(-1.2 0)
		\bsegment
		\move(0 0)\lvec(2 0)
		\move(0 0)\lvec(0 2)
		\move(0 2)\lvec(2 2)
		\move(0.6 0)\lvec(0.6 2)
		\move(1.4 0)\lvec(1.4 2)
		\move(0 0)\lvec(0.6 2)
		\move(1.4 0)\lvec(2 2)
		\htext(0.4 0.4){$0$}
		\htext(1.8 0.4){$0$}

		\htext(2.4 0.4){$0$}
		\htext(3 0.4){$0$}

		\htext(1.05 1){$\cdots$}
		\htext(1 -0.3){$\underbrace{\rule{4.5em}{0em}}$}
		\htext(1 -0.6){$l'\!\!-\!\!2$}
		
		\move(2 0)\lvec(3.2 0)
		\move(2 2)\lvec(3.2 2)
		\move(2 0)\lvec(2 2)
		\move(2.6 0)\lvec(2.6 2)
		\move(2 0)\lvec(2.6 2)	
		\move(2.6 0)\lvec(3.2 2)			
		\esegment
		\move(2 0)
		\bsegment
		\move(0 0)\lvec(2 0)
		\move(0 0)\lvec(0 2)
		\move(0 2)\lvec(2 2)
		\move(2 0)\lvec(2 2)
		\move(0.6 0)\lvec(0.6 2)
		\move(1.4 0)\lvec(1.4 2)
		\move(0 0)\lvec(0.6 2)
			\move(1.4 0)\lvec(2 2)
			\htext(0.2 1.6){$0$}
		\htext(1.6 1.6){$0$}
		\htext(0.4 0.4){$0$}
		\htext(1.8 0.4){$0$}
		\htext(1.05 1){$\cdots$}
	
		\htext(1 -0.3){$\underbrace{\rule{4.5em}{0em}}$}
		\htext(1 -0.7){$x_1$}
		
		\move(2.5 1)\avec(4.5 1)
		\htext(3.5 1.8){	\fontsize{12}{12}\selectfont$\tilde{\mathfrak f}_{0}$}
		\esegment
		\esegment
		
		\move(8 0)
		\bsegment
		\move(-1.2 0)
		\bsegment
		\move(0 0)\lvec(2 0)
		\move(0 0)\lvec(0 2)
		\move(0 2)\lvec(2 2)
		\move(0.6 0)\lvec(0.6 2)
		\move(1.4 0)\lvec(1.4 2)
		\move(0 0)\lvec(0.6 2)
		\move(1.4 0)\lvec(2 2)
		\htext(0.4 0.4){$0$}
		\htext(1.8 0.4){$0$}

		\htext(2.4 0.4){$0$}
		\htext(3 0.4){$0$}		
		
		\htext(2.2 1.6){$0$}
		\htext(2.8 1.6){$0$}	
		\htext(1.05 1){$\cdots$}
		\htext(1 -0.3){$\underbrace{\rule{4.5em}{0em}}$}
		\htext(1 -0.6){$l'\!\!-\!\!2$}
		
		\move(2 0)\lvec(2 2)\lvec(2.6 2)\lvec(2 0) \lfill f:0.8
		\move(2.6 0)\lvec(2.6 2)\lvec(3.2 2)\lvec(2.6 0) \lfill f:0.8
		
		\move(2 0)\lvec(3.2 0)
		\move(2 2)\lvec(3.2 2)
		\move(2 0)\lvec(2 2)
		\move(2.6 0)\lvec(2.6 2)
		\move(2 0)\lvec(2.6 2)	
		\move(2.6 0)\lvec(3.2 2)			
		\esegment
		\move(2 0)
		\bsegment
		\move(0 0)\lvec(2 0)
		\move(0 0)\lvec(0 2)
		\move(0 2)\lvec(2 2)
		\move(2 0)\lvec(2 2)
		\move(0.6 0)\lvec(0.6 2)
		\move(1.4 0)\lvec(1.4 2)
		\htext(1.05 1){$\cdots$}
		\move(0 0)\lvec(0.6 2)
			\move(1.4 0)\lvec(2 2)
		\htext(0.2 1.6){$0$}
		\htext(1.6 1.6){$0$}
		\htext(0.4 0.4){$0$}
		\htext(1.8 0.4){$0$}
		\htext(1 -0.3){$\underbrace{\rule{4.5em}{0em}}$}
		\htext(1 -0.7){$x_1$}
		\esegment
		\esegment
	\end{texdraw}
\end{center}

where $l'=2l-\sum_{i=1}^{n}(x_i+\bar{x}_i)$.

\vskip 4mm

if $x_1= \bar{x}_1-1$, the map $\tilde{\mathfrak f}_0$ is given by

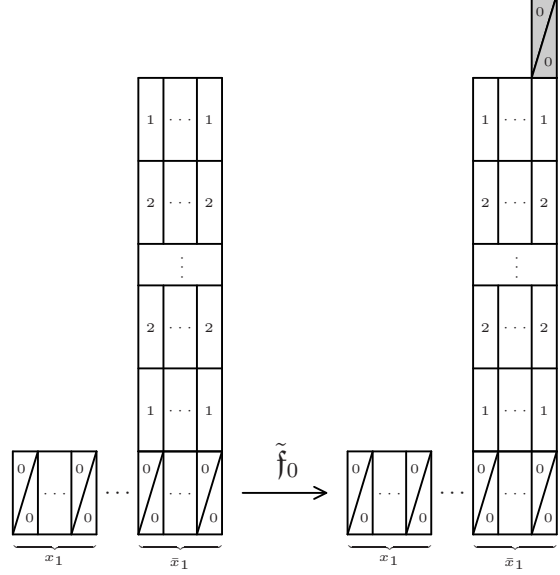
\begin{center}
	\begin{texdraw}
		\fontsize{3}{3}\selectfont
		\drawdim em
		\setunitscale 2.3
		\arrowheadsize l:0.3 w:0.3
		\arrowheadtype t:V
		\move(0 0)
		
		\bsegment

		\move(0 0)
		\bsegment
		\move(0 0)\lvec(2 0)
		\move(0 0)\lvec(0 2)
		\move(0 2)\lvec(2 2)
		\move(2 0)\lvec(2 2)
		\move(0.6 0)\lvec(0.6 2)
		\move(1.4 0)\lvec(1.4 2)
			\move(0 0)\lvec(0.6 2)
		\move(1.4 0)\lvec(2 2)
		\htext(0.2 1.6){$0$}
		\htext(1.6 1.6){$0$}
		\htext(0.4 0.4){$0$}
		\htext(1.8 0.4){$0$}
		\htext(1.05 1){$\cdots$}
		\htext(1 -0.3){$\underbrace{\rule{4.5em}{0em}}$}
		\htext(1 -0.6){$x_1$}
		\htext(2.55 1){\fontsize{8}{8}\selectfont$\cdots$}
		\esegment
		\move(3 0)
		\bsegment
		\move(0 0)\lvec(2 0)
		\move(0 0)\lvec(0 2)
		\move(0 2)\lvec(2 2)
		\move(0.6 0)\lvec(0.6 2)
		\move(1.4 0)\lvec(1.4 2)
		\htext(1.05 1){$\cdots$}
		\htext(1.05 3){$\cdots$}
		\htext(1.05 5){$\cdots$}
		\vtext(1.05 6.5){$\cdots$}
		\htext(1.05 8){$\cdots$}
		\htext(1.05 10){$\cdots$}
			\move(0 0)\lvec(0.6 2)
		\move(1.4 0)\lvec(2 2)
		\htext(0.2 1.6){$0$}
		\htext(1.6 1.6){$0$}
		\htext(0.4 0.4){$0$}
		\htext(1.8 0.4){$0$}
		\move(2 0)\lvec(2 2)
		\move(2 2)\lvec(2 6)
		\move(1.4 2)\lvec(1.4 6)
		\move(0.6 2)\lvec(0.6 6)
		\move(0 2)\lvec(0 6)
		\move(0 6)\lvec(0 7)
		\move(2 6)\lvec(2 7)
		\move(2 7)\lvec(2 11)
		\move(1.4 7)\lvec(1.4 11)
		\move(0.6 7)\lvec(0.6 11)
		\move(0 7)\lvec(0 11)
		
		\move(0 2)\lvec(2 2)
		\move(0 11)\lvec(2 11)
		\move(0 4)\lvec(2 4)
		\move(0 6)\lvec(2 6)
		\move(0 7)\lvec(2 7)
		\move(0 9)\lvec(2 9)
		\htext(0.3 3){$1$}
		\htext(0.3 5){$2$}
		\htext(0.3 8){$2$}
		\htext(0.3 10){$1$}
		\htext(1.7 3){$1$}
		\htext(1.7 5){$2$}
		\htext(1.7 8){$2$}
		\htext(1.7 10){$1$}

		\htext(1 -0.3){$\underbrace{\rule{4.5em}{0em}}$}
		\htext(1 -0.7){$\bar{x}_1$}
	
		\move(2.5 1)\avec(4.5 1)
		\htext(3.5 1.8){	\fontsize{12}{12}\selectfont$\tilde{\mathfrak f}_{0}$}
		
		\esegment
		\esegment
		
		\move(8 0)
		
		\bsegment
		
		\move(0 0)
		\bsegment
		\move(0 0)\lvec(2 0)
		\move(0 0)\lvec(0 2)
		\move(0 2)\lvec(2 2)
		\move(2 0)\lvec(2 2)
		\move(0.6 0)\lvec(0.6 2)
		\move(1.4 0)\lvec(1.4 2)
			\move(0 0)\lvec(0.6 2)
		\move(1.4 0)\lvec(2 2)
		\htext(0.2 1.6){$0$}
		\htext(1.6 1.6){$0$}
		\htext(0.4 0.4){$0$}
		\htext(1.8 0.4){$0$}
		\htext(1.05 1){$\cdots$}
		\htext(1 -0.3){$\underbrace{\rule{4.5em}{0em}}$}
		\htext(1 -0.6){$x_1$}
		\htext(2.55 1){\fontsize{8}{8}\selectfont$\cdots$}
		\esegment
		\move(3 0)
		\bsegment
		\move(0 0)\lvec(2 0)
		\move(0 0)\lvec(0 2)
		\move(0 2)\lvec(2 2)
			\move(0 0)\lvec(0.6 2)
		\move(1.4 0)\lvec(2 2)
		\htext(0.2 1.6){$0$}
		\htext(1.6 1.6){$0$}
		\htext(0.4 0.4){$0$}
		\htext(1.8 0.4){$0$}
		
		\htext(1.6 12.6){$0$}
			\htext(1.8 11.4){$0$}
		\move(0.6 0)\lvec(0.6 2)
		\move(1.4 0)\lvec(1.4 2)
		\htext(1.05 1){$\cdots$}
		\htext(1.05 3){$\cdots$}
		\htext(1.05 5){$\cdots$}
		\vtext(1.05 6.5){$\cdots$}
		\htext(1.05 8){$\cdots$}
		\htext(1.05 10){$\cdots$}
	
		\move(2 0)\lvec(2 2)
		\move(2 2)\lvec(2 6)
		\move(1.4 2)\lvec(1.4 6)
		\move(0.6 2)\lvec(0.6 6)
		\move(0 2)\lvec(0 6)
		\move(0 6)\lvec(0 7)
		\move(2 6)\lvec(2 7)
		\move(2 7)\lvec(2 11)
		\move(1.4 7)\lvec(1.4 11)
		\move(0.6 7)\lvec(0.6 11)
		\move(0 7)\lvec(0 11)
		
		\move(0 2)\lvec(2 2)
		\move(0 11)\lvec(2 11)
		\move(0 4)\lvec(2 4)
		\move(0 6)\lvec(2 6)
		\move(0 7)\lvec(2 7)
		\move(0 9)\lvec(2 9)
		\htext(0.3 3){$1$}
		\htext(0.3 5){$2$}
		\htext(0.3 8){$2$}
		\htext(0.3 10){$1$}
		\htext(1.7 3){$1$}
		\htext(1.7 5){$2$}
		\htext(1.7 8){$2$}
		\htext(1.7 10){$1$}

		\htext(1 -0.3){$\underbrace{\rule{4.5em}{0em}}$}
		\htext(1 -0.7){$\bar{x}_1$}
		\move(1.4 11)\lvec(1.4 13)\lvec(2 13) 	\lvec(1.4 11)
		\lfill f:0.8
	\move(2 11)\lvec(2 13)\lvec(1.4 11)\lvec(2 11)\lfill f:0.8

		\esegment
		\esegment
	\end{texdraw}
\end{center}

\vskip 4mm

if $x_1\le \bar{x}_1-2$, the map $\tilde{\mathfrak f}_0$ is given by

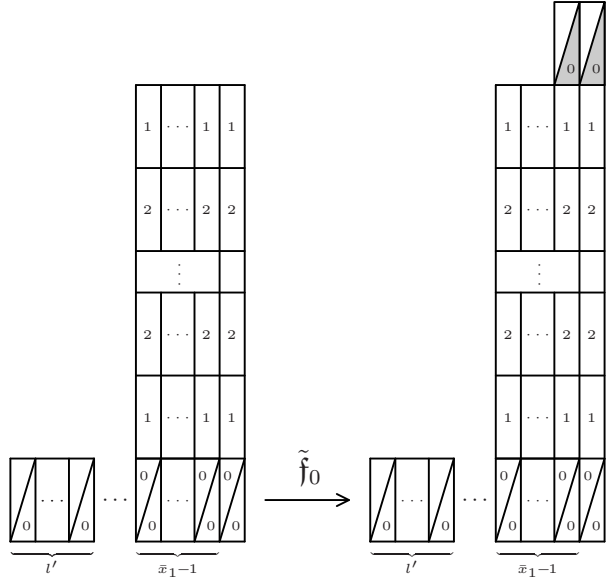
\begin{center}
	\begin{texdraw}
		\fontsize{3}{3}\selectfont
		\drawdim em
		\setunitscale 2.3
		\arrowheadsize l:0.3 w:0.3
		\arrowheadtype t:V
		\move(0 0)
		
		\bsegment
		
		\move(0 0)
		\bsegment
		\move(0 0)\lvec(2 0)
		\move(0 0)\lvec(0 2)
		\move(0 2)\lvec(2 2)
		\move(2 0)\lvec(2 2)
		\move(0.6 0)\lvec(0.6 2)
		\move(1.4 0)\lvec(1.4 2)
		\move(0 0)\lvec(0.6 2)
		\move(1.4 0)\lvec(2 2)
		\htext(0.4 0.4){$0$}
		\htext(1.8 0.4){$0$}
		\htext(1.05 1){$\cdots$}
		\htext(1 -0.3){$\underbrace{\rule{4.5em}{0em}}$}
		\htext(1 -0.6){$l'$}
		\htext(1.05 1){$\cdots$}
		\htext(1 -0.3){$\underbrace{\rule{4.5em}{0em}}$}
		
		\htext(2.55 1){\fontsize{8}{8}\selectfont$\cdots$}
		\esegment
		\move(3 0)
		\bsegment
		\move(0 0)\lvec(2 0)
		\move(0 0)\lvec(0 2)
		\move(0 2)\lvec(2 2)
		\move(0.6 0)\lvec(0.6 2)
		\move(1.4 0)\lvec(1.4 2)
		\htext(1.05 1){$\cdots$}
		\htext(1.05 3){$\cdots$}
		\htext(1.05 5){$\cdots$}
		\vtext(1.05 6.5){$\cdots$}
		\htext(1.05 8){$\cdots$}
		\htext(1.05 10){$\cdots$}
			
		\move(0 0)\lvec(0.6 2)
		\move(1.4 0)\lvec(2 2)
		\htext(0.2 1.6){$0$}
		\htext(1.6 1.6){$0$}
		\htext(0.4 0.4){$0$}
		\htext(1.8 0.4){$0$}
		\move(2 0)\lvec(2 2)
		\move(2 2)\lvec(2 6)
		\move(1.4 2)\lvec(1.4 6)
		\move(0.6 2)\lvec(0.6 6)
		\move(0 2)\lvec(0 6)
		\move(0 6)\lvec(0 7)
		\move(2 6)\lvec(2 7)
		\move(2 7)\lvec(2 11)
		\move(1.4 7)\lvec(1.4 11)
		\move(0.6 7)\lvec(0.6 11)
		\move(0 7)\lvec(0 11)
		
		\move(0 2)\lvec(2 2)
		\move(0 11)\lvec(2 11)
		\move(0 4)\lvec(2 4)
		\move(0 6)\lvec(2 6)
		\move(0 7)\lvec(2 7)
		\move(0 9)\lvec(2 9)
		\htext(0.3 3){$1$}
		\htext(0.3 5){$2$}
		\htext(0.3 8){$2$}
		\htext(0.3 10){$1$}
		\htext(1.7 3){$1$}
		\htext(1.7 5){$2$}
		\htext(1.7 8){$2$}
		\htext(1.7 10){$1$}

		\htext(1 -0.3){$\underbrace{\rule{4.5em}{0em}}$}
		\htext(1 -0.7){$\bar{x}_1\!\!-\!\!1$}
		
		\move(2.6 0)\lvec(2.6 11)
		\move(2 0)\lvec(2.6 0)
		\move(2 2)\lvec(2.6 2)
		\move(2 4)\lvec(2.6 4)
		\move(2 6)\lvec(2.6 6)	
		\move(2 7)\lvec(2.6 7)	
		\move(2 9)\lvec(2.6 9)	
		\move(2 11)\lvec(2.6 11)

		\move(2 0)\lvec(2.6 2)
	
		\htext(2.2 1.6){$0$}
	
		\htext(2.4 0.4){$0$}
		
		\htext(2.3 3){$1$}
		\htext(2.3 5){$2$}
		\htext(2.3 8){$2$}
		\htext(2.3 10){$1$}	
		\move(3.1 1)\avec(5.1 1)
		\htext(4.1 1.8){	\fontsize{12}{12}\selectfont$\tilde{\mathfrak f}_{0}$}	
		\esegment
		\esegment
		
		\move(8.6 0)
		
		\bsegment
		
		\move(0 0)
		\bsegment
		\move(0 0)\lvec(2 0)
		\move(0 0)\lvec(0 2)
		\move(0 2)\lvec(2 2)
		\move(2 0)\lvec(2 2)
		\move(0.6 0)\lvec(0.6 2)
		\move(1.4 0)\lvec(1.4 2)
		\move(0 0)\lvec(0.6 2)
		\move(1.4 0)\lvec(2 2)
		\htext(0.4 0.4){$0$}
		\htext(1.8 0.4){$0$}
		\htext(1.05 1){$\cdots$}
		\htext(1 -0.3){$\underbrace{\rule{4.5em}{0em}}$}
		\htext(1 -0.6){$l'$}
		\htext(2.55 1){\fontsize{8}{8}\selectfont$\cdots$}
		\esegment
		\move(3 0)
		\bsegment
		\move(0 0)\lvec(2 0)
		\move(0 0)\lvec(0 2)
		\move(0 2)\lvec(2 2)
		\move(0.6 0)\lvec(0.6 2)
		\move(1.4 0)\lvec(1.4 2)
		\htext(1.05 1){$\cdots$}
		\htext(1.05 3){$\cdots$}
		\htext(1.05 5){$\cdots$}
		\vtext(1.05 6.5){$\cdots$}
		\htext(1.05 8){$\cdots$}
		\htext(1.05 10){$\cdots$}
	  \move(0 0)\lvec(0.6 2)
	\move(1.4 0)\lvec(2 2)
	\htext(0.2 1.6){$0$}
	\htext(1.6 1.6){$0$}
	\htext(0.4 0.4){$0$}
	\htext(1.8 0.4){$0$}
	\move(2 0)\lvec(2.6 2)
	\htext(2.2 1.6){$0$}
	\htext(2.4 0.4){$0$}
		\move(2 0)\lvec(2 2)
		\move(2 2)\lvec(2 6)
		\move(1.4 2)\lvec(1.4 6)
		\move(0.6 2)\lvec(0.6 6)
		\move(0 2)\lvec(0 6)
		\move(0 6)\lvec(0 7)
		\move(2 6)\lvec(2 7)
		\move(2 7)\lvec(2 11)
		\move(1.4 7)\lvec(1.4 11)
		\move(0.6 7)\lvec(0.6 11)
		\move(0 7)\lvec(0 11)
		
		\move(0 2)\lvec(2 2)
		\move(0 11)\lvec(2 11)
		\move(0 4)\lvec(2 4)
		\move(0 6)\lvec(2 6)
		\move(0 7)\lvec(2 7)
		\move(0 9)\lvec(2 9)
		\htext(0.3 3){$1$}
		\htext(0.3 5){$2$}
		\htext(0.3 8){$2$}
		\htext(0.3 10){$1$}
		\htext(1.7 3){$1$}
		\htext(1.7 5){$2$}
		\htext(1.7 8){$2$}
		\htext(1.7 10){$1$}

		\htext(1 -0.3){$\underbrace{\rule{4.5em}{0em}}$}
		\htext(1 -0.7){$\bar{x}_1\!\!-\!\!1$}
		\move(1.4 11)\lvec(1.4 13)\lvec(2 13) \lvec(2 11)\lvec(1.4 11)
		
		\move(2.6 0)\lvec(2.6 13)
		\move(2 0)\lvec(2.6 0)
		\move(2 2)\lvec(2.6 2)
		\move(2 4)\lvec(2.6 4)
		\move(2 6)\lvec(2.6 6)	
		\move(2 7)\lvec(2.6 7)	
		\move(2 9)\lvec(2.6 9)	
		\move(2 11)\lvec(2.6 11)				
		\move(2 13)\lvec(2.6 13)
	
		\htext(2.3 3){$1$}
		\htext(2.3 5){$2$}
		\htext(2.3 8){$2$}
		\htext(2.3 10){$1$}

		\move(1.4 11)\lvec(2 11)\lvec(2 13)\lvec(1.4 11) \lfill f:0.8
		\move(2 11)\lvec(2.6 11)\lvec(2.6 13)\lvec(2 11) \lfill f:0.8
		\htext(1.8 11.4){$0$}
		\htext(2.4 11.4){$0$}
		\esegment
		\esegment
	\end{texdraw}
\end{center}

where $l'=2l-\sum_{i=1}^{n}(x_i+\bar{x}_i)$.
\end{enumerate}

\vskip 3mm

\begin{remark}
After we perform the  $\tilde{\mathfrak f}_i$-action on the level-$l$ reduced Young columns, we remove all removable $\delta$-slices and move some slices in the Young columns to keep the shapes of the level-$l$ reduced Young columns  in Section \ref{Level-$l$ reduced Young Column}.
\end{remark}

\vskip 1mm

\begin{example}
We consider the case $B_{4}^{(1)}$. Let  $l=15$ and $b=(2,1,2,2,1,1,1,2,3)$. The action of $\tilde{\mathfrak f}_{1}$ on the reduced Young Column $C_b$ is given below.
\vskip 2mm
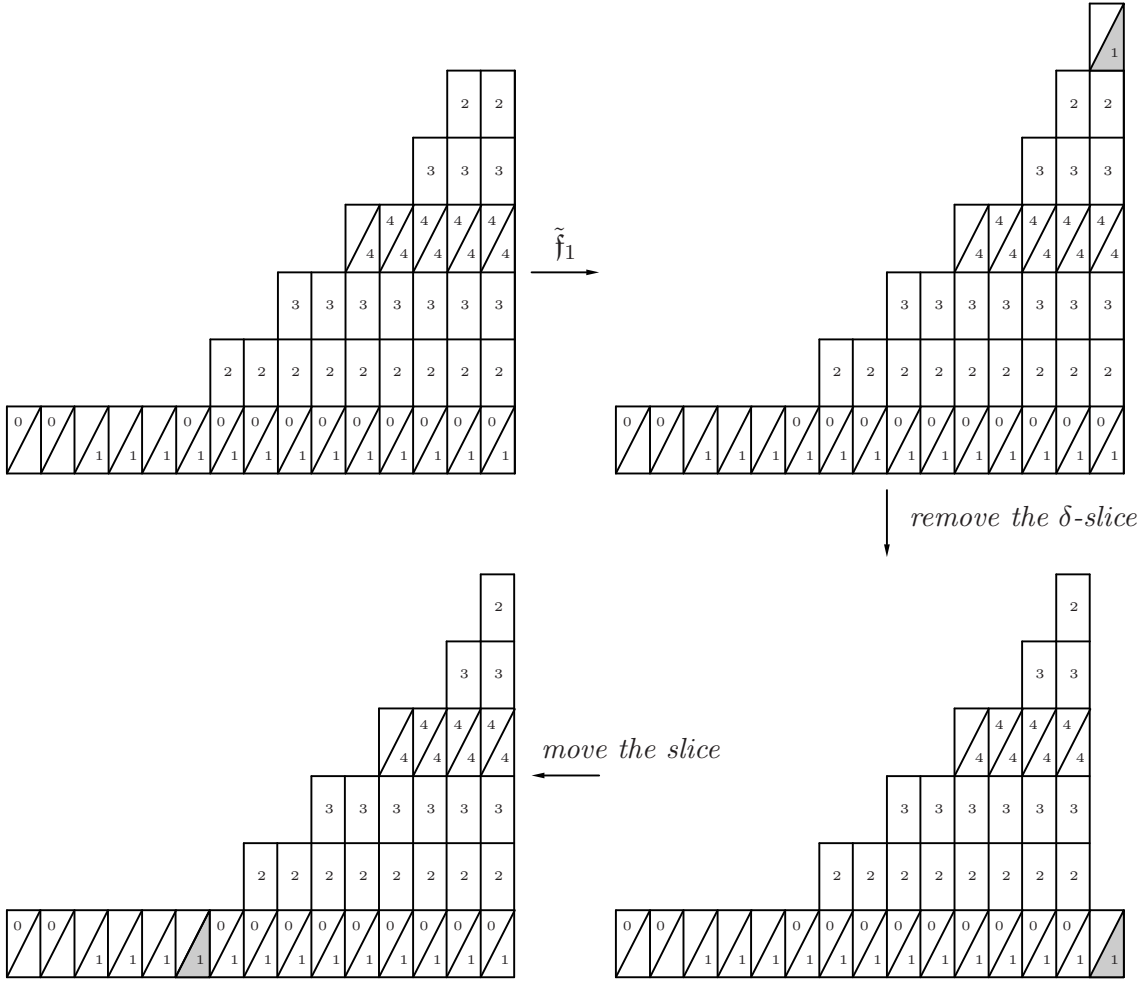
\begin{center}
	\begin{texdraw}
		\fontsize{5}{5}\selectfont
		\drawdim em
		\setunitscale 2.1
		\move(2 0)
		\bsegment		 
		\move(-2 0)\rlvec(-15 0)	
		
		\move(-2 2)\rlvec(-15 0)
		
		\move(-2 0 )\rlvec(0 12)
		\move(-2 0 )\rlvec(0 12)
		\move(-3 0 )\rlvec(0 10)
		\move(-4 0 )\rlvec(0 10)
		\move(-5 0 )\rlvec(0 4)
		\move(-6 0 )\rlvec(0 8)
		\move(-7 0 )\rlvec(0 8)
		\move(-8 0 )\rlvec(0 6)
		\move(-9 0 )\rlvec(0 4)
		\move(-10 0 )\rlvec(0 4)
		\move(-11 0 )\rlvec(0 2)
		\move(-12 0 )\rlvec(0 2)
		\move(-13 0 )\rlvec(0 2)
		\move(-14 0 )\rlvec(0 2)
		\move(-15 0 )\rlvec(0 2)
		
		\move(-3 10)\rlvec(-1 0)
		
		\move(-3 12)\lvec(-2 12)
		\move(-3 10)\lvec(-2 10)	
		\move(-3 8)\lvec(-2 8)
		\move(-3 6)\lvec(-2 6)
		
		\move(-6 8)\rlvec(-1 0)

		\move(-2 4)\rlvec(-8 0)
		\move(-6 6)\rlvec(-2 0)
		\move(-3 6)\rlvec(-1 0)
		\move(-3 8)\rlvec(-1 0)

		\move(-3 0)\rlvec(1 2)
		\move(-3 10)\rlvec(0 2)
		\move(-5 8)\lvec(-5 10)
		\move(-5 10)\lvec(-4 10)
		\htext(-4.5 9){$3$}	
		\move(-6 6)\rlvec(1 0)
		\move(-6 8)\rlvec(1 0)
		\move(-6 6)\rlvec(1 2)
		\htext(-5.5 5){$3$}

		\move(-4 0)\rlvec(1 2)
		\move(-5 0)\rlvec(1 2)
		\move(-5 4)\rlvec(0 4)
		\move(-4 8)\rlvec(-1 0)
		\move(-4 6)\rlvec(-1 0)
		\htext(-4.5 5){$3$}
		\move(-5 6)\rlvec(1 2)
		\htext(-4.3 6.5){$4$}
		\htext(-4.7 7.5){$4$}
		\htext(-5.3 6.5){$4$}
		\htext(-5.7 7.5){$4$}
		
		\move(-6 0)\rlvec(1 2)
		\move(-7 0)\rlvec(1 2)
		\move(-8 0)\rlvec(1 2)
		\move(-9 0)\rlvec(1 2)
		\move(-10 0)\rlvec(1 2)
		\move(-11 0)\rlvec(1 2)
		\move(-12 0)\rlvec(1 2)
		\move(-13 0)\rlvec(1 2)
		\move(-14 0)\rlvec(1 2)
		\move(-15 0)\rlvec(1 2)
		\move(-7 6)\rlvec(1 2)
		
		\move(-4 6)\rlvec(1 2)
		
		\move(-11 2)\lvec(-11 4)
		\move(-11 4)\lvec(-10 4)
		\move(-16 0)\lvec(-16 2)
		\move(-17 0)\lvec(-17 2)
		\move(-17 0)\lvec(-16 2)
		\move(-16 0)\lvec(-15 2)
		\htext(-10.5 3){$2$}
		\htext(-2.3 0.5){$1$}
		\htext(-2.7 1.5){$0$}
		\htext(-2.5 3){$2$}	
		\htext(-2.5 5){$3$}
		\move(-3 6)\lvec(-2 8)	
		\htext(-2.3 6.5){$4$}
		\htext(-2.7 7.5){$4$}	
		\htext(-2.5 9){$3$}	
		\htext(-2.5 11){$2$}				
		\htext(-3.3 0.5){$1$}
		\htext(-3.7 1.5){$0$}
		\htext(-3.5 3){$2$}
		\htext(-3.5 5){$3$}	
		\htext(-3.3 6.5){$4$}
		\htext(-3.7 7.5){$4$}
		\htext(-3.5 9){$3$}
		
		\htext(-4.3 0.5){$1$}
		\htext(-4.7 1.5){$0$}
		\htext(-4.5 3){$2$}
		
		\htext(-5.3 0.5){$1$}
		\htext(-5.7 1.5){$0$}
		\htext(-5.5 3){$2$}
		
		\htext(-6.3 0.5){$1$}
		\htext(-6.7 1.5){$0$}
		\htext(-6.5 3){$2$}
		\htext(-6.5 5){$3$}
		\htext(-6.3 6.5){$4$}

		\htext(-7.3 0.5){$1$}
		\htext(-7.7 1.5){$0$}
		\htext(-7.5 3){$2$}
		\htext(-7.5 5){$3$}

		\htext(-8.3 0.5){$1$}
		\htext(-8.7 1.5){$0$}
		\htext(-8.5 3){$2$}

		\htext(-9.3 0.5){$1$}
		\htext(-9.7 1.5){$0$}
		\htext(-9.5 3){$2$}
		
		\htext(-11.7 1.5){$0$}
		\htext(-10.3 0.5){$1$}
		\htext(-11.3 0.5){$1$}
		\htext(-12.3 0.5){$1$}
		\htext(-13.3 0.5){$1$}
		\htext(-14.3 0.5){$1$}
		\htext(-10.7 1.5){$0$}
		
		\htext(-15.7 1.5){$0$}
		\htext(-16.7 1.5){$0$}
		\move(-4 10)\lvec(-4 12)
		\move(-4 12)\lvec(-3 12)
		\htext(-3.5 11){$2$}
		\move(-9 4)\lvec(-9 6)	
		\move(-9 6)\lvec(-8 6)
		\htext(-8.5 5){$3$}	
		\move(-3 12)\lvec(-3 14)
		\move(-3 14)\lvec(-2 14)
		
		\move(-3 12)\lvec(-2 14)\lvec(-2 12)\lvec(-3 12)\lfill f:0.8
		\htext(-2.3 12.5){$1$}
		\esegment
		
		\move(-16 0)
		\bsegment
		\move(-2 0)\rlvec(-15 0)	
		
		\move(-2 2)\rlvec(-15 0)
		
		\move(-2 0 )\rlvec(0 12)
		\move(-2 0 )\rlvec(0 12)
		\move(-3 0 )\rlvec(0 10)
		\move(-4 0 )\rlvec(0 10)
		\move(-5 0 )\rlvec(0 4)
		\move(-6 0 )\rlvec(0 8)
		\move(-7 0 )\rlvec(0 8)
		\move(-8 0 )\rlvec(0 6)
		\move(-9 0 )\rlvec(0 4)
		\move(-10 0 )\rlvec(0 4)
		\move(-11 0 )\rlvec(0 2)
		\move(-12 0 )\rlvec(0 2)
		\move(-13 0 )\rlvec(0 2)
		\move(-14 0 )\rlvec(0 2)
		\move(-15 0 )\rlvec(0 2)
		
		\move(-3 10)\rlvec(-1 0)
		
		\move(-3 12)\lvec(-2 12)
		\move(-3 10)\lvec(-2 10)	
		\move(-3 8)\lvec(-2 8)
		\move(-3 6)\lvec(-2 6)
		
		\move(-6 8)\rlvec(-1 0)

		\move(-2 4)\rlvec(-8 0)
		\move(-6 6)\rlvec(-2 0)
		\move(-3 6)\rlvec(-1 0)
		\move(-3 8)\rlvec(-1 0)

		\move(-3 0)\rlvec(1 2)
		\move(-3 10)\rlvec(0 2)
		\move(-5 8)\lvec(-5 10)
		\move(-5 10)\lvec(-4 10)
		\htext(-4.5 9){$3$}	
		\move(-6 6)\rlvec(1 0)
		\move(-6 8)\rlvec(1 0)
		\move(-6 6)\rlvec(1 2)
		\htext(-5.5 5){$3$}

		\move(-4 0)\rlvec(1 2)
		\move(-5 0)\rlvec(1 2)
		\move(-5 4)\rlvec(0 4)
		\move(-4 8)\rlvec(-1 0)
		\move(-4 6)\rlvec(-1 0)
		\htext(-4.5 5){$3$}
		\move(-5 6)\rlvec(1 2)
		\htext(-4.3 6.5){$4$}
		\htext(-4.7 7.5){$4$}
		\htext(-5.3 6.5){$4$}
		\htext(-5.7 7.5){$4$}
		
		\move(-6 0)\rlvec(1 2)
		\move(-7 0)\rlvec(1 2)
		\move(-8 0)\rlvec(1 2)
		\move(-9 0)\rlvec(1 2)
		\move(-10 0)\rlvec(1 2)
		\move(-11 0)\rlvec(1 2)
		\move(-12 0)\rlvec(1 2)
		\move(-13 0)\rlvec(1 2)
		\move(-14 0)\rlvec(1 2)
		\move(-15 0)\rlvec(1 2)
		\move(-7 6)\rlvec(1 2)
		
		\move(-4 6)\rlvec(1 2)
		
		\move(-11 2)\lvec(-11 4)
		\move(-11 4)\lvec(-10 4)
		\move(-16 0)\lvec(-16 2)
		\move(-17 0)\lvec(-17 2)
		\move(-17 0)\lvec(-16 2)
		\move(-16 0)\lvec(-15 2)
		\htext(-10.5 3){$2$}
		\htext(-2.3 0.5){$1$}
		\htext(-2.7 1.5){$0$}
		\htext(-2.5 3){$2$}	
		\htext(-2.5 5){$3$}
		\move(-3 6)\lvec(-2 8)	
		\htext(-2.3 6.5){$4$}
		\htext(-2.7 7.5){$4$}	
		\htext(-2.5 9){$3$}	
		\htext(-2.5 11){$2$}				
		\htext(-3.3 0.5){$1$}
		\htext(-3.7 1.5){$0$}
		\htext(-3.5 3){$2$}
		\htext(-3.5 5){$3$}	
		\htext(-3.3 6.5){$4$}
		\htext(-3.7 7.5){$4$}
		\htext(-3.5 9){$3$}
		
		\htext(-4.3 0.5){$1$}
		\htext(-4.7 1.5){$0$}
		\htext(-4.5 3){$2$}
		
		\htext(-5.3 0.5){$1$}
		\htext(-5.7 1.5){$0$}
		\htext(-5.5 3){$2$}
		
		\htext(-6.3 0.5){$1$}
		\htext(-6.7 1.5){$0$}
		\htext(-6.5 3){$2$}
		\htext(-6.5 5){$3$}
		\htext(-6.3 6.5){$4$}

		\htext(-7.3 0.5){$1$}
		\htext(-7.7 1.5){$0$}
		\htext(-7.5 3){$2$}
		\htext(-7.5 5){$3$}

		\htext(-8.3 0.5){$1$}
		\htext(-8.7 1.5){$0$}
		\htext(-8.5 3){$2$}

		\htext(-9.3 0.5){$1$}
		\htext(-9.7 1.5){$0$}
		\htext(-9.5 3){$2$}
		
		\htext(-11.7 1.5){$0$}
		\htext(-10.3 0.5){$1$}
		\htext(-11.3 0.5){$1$}
		\htext(-12.3 0.5){$1$}
		\htext(-13.3 0.5){$1$}
		\htext(-14.3 0.5){$1$}
		\htext(-10.7 1.5){$0$}
		
		\htext(-15.7 1.5){$0$}
		\htext(-16.7 1.5){$0$}
		\move(-4 10)\lvec(-4 12)
		\move(-4 12)\lvec(-3 12)
		\htext(-3.5 11){$2$}
		\move(-9 4)\lvec(-9 6)	
		\move(-9 6)\lvec(-8 6)
		\htext(-8.5 5){$3$}				
		\esegment	
		
		\move(2 -15)
		\bsegment
		\move(-2 0)\rlvec(-15 0)	
		
		\move(-2 2)\rlvec(-15 0)
		
		\move(-2 0 )\rlvec(0 2)
		
		\move(-3 0 )\rlvec(0 10)
		\move(-4 0 )\rlvec(0 10)
		\move(-5 0 )\rlvec(0 4)
		\move(-6 0 )\rlvec(0 8)
		\move(-7 0 )\rlvec(0 8)
		\move(-8 0 )\rlvec(0 6)
		\move(-9 0 )\rlvec(0 4)
		\move(-10 0 )\rlvec(0 4)
		\move(-11 0 )\rlvec(0 2)
		\move(-12 0 )\rlvec(0 2)
		\move(-13 0 )\rlvec(0 2)
		\move(-14 0 )\rlvec(0 2)
		\move(-15 0 )\rlvec(0 2)
		
		\move(-3 10)\rlvec(-1 0)
		
		\move(-6 8)\rlvec(-1 0)	
		\move(-3 4)\rlvec(-7 0)
		\move(-6 6)\rlvec(-2 0)
		\move(-3 6)\rlvec(-1 0)
		\move(-3 8)\rlvec(-1 0)

		\move(-3 0)\lvec(-2 2)\lvec(-2 0)\lvec(-3 0) \lfill f:0.8

		\move(-3 10)\rlvec(0 2)
		\move(-5 8)\lvec(-5 10)
		\move(-5 10)\lvec(-4 10)
		\htext(-4.5 9){$3$}	
		\move(-6 6)\rlvec(1 0)
		\move(-6 8)\rlvec(1 0)
		\move(-6 6)\rlvec(1 2)
		\htext(-5.5 5){$3$}

		\move(-4 0)\rlvec(1 2)
		\move(-5 0)\rlvec(1 2)
		\move(-5 4)\rlvec(0 4)
		\move(-4 8)\rlvec(-1 0)
		\move(-4 6)\rlvec(-1 0)
		\htext(-4.5 5){$3$}
		\move(-5 6)\rlvec(1 2)
		\htext(-4.3 6.5){$4$}
		\htext(-4.7 7.5){$4$}
		\htext(-5.3 6.5){$4$}
		\htext(-5.7 7.5){$4$}
		
		\move(-6 0)\rlvec(1 2)
		\move(-7 0)\rlvec(1 2)
		\move(-8 0)\rlvec(1 2)
		\move(-9 0)\rlvec(1 2)
		\move(-10 0)\rlvec(1 2)
		\move(-11 0)\rlvec(1 2)
		\move(-12 0)\rlvec(1 2)
		\move(-13 0)\rlvec(1 2)
		\move(-14 0)\rlvec(1 2)
		\move(-15 0)\rlvec(1 2)
		\move(-7 6)\rlvec(1 2)
		
		\move(-4 6)\rlvec(1 2)
		
		\move(-11 2)\lvec(-11 4)
		\move(-11 4)\lvec(-10 4)
		\move(-16 0)\lvec(-16 2)
		\move(-17 0)\lvec(-17 2)
		\move(-17 0)\lvec(-16 2)
		\move(-16 0)\lvec(-15 2)
		\htext(-10.5 3){$2$}
		\htext(-2.3 0.5){$1$}

		\htext(-3.3 0.5){$1$}
		\htext(-3.7 1.5){$0$}
		\htext(-3.5 3){$2$}
		\htext(-3.5 5){$3$}	
		\htext(-3.3 6.5){$4$}
		\htext(-3.7 7.5){$4$}
		\htext(-3.5 9){$3$}
		
		\htext(-4.3 0.5){$1$}
		\htext(-4.7 1.5){$0$}
		\htext(-4.5 3){$2$}
		
		\htext(-5.3 0.5){$1$}
		\htext(-5.7 1.5){$0$}
		\htext(-5.5 3){$2$}
		
		\htext(-6.3 0.5){$1$}
		\htext(-6.7 1.5){$0$}
		\htext(-6.5 3){$2$}
		\htext(-6.5 5){$3$}
		\htext(-6.3 6.5){$4$}

		\htext(-7.3 0.5){$1$}
		\htext(-7.7 1.5){$0$}
		\htext(-7.5 3){$2$}
		\htext(-7.5 5){$3$}

		\htext(-8.3 0.5){$1$}
		\htext(-8.7 1.5){$0$}
		\htext(-8.5 3){$2$}

		\htext(-9.3 0.5){$1$}
		\htext(-9.7 1.5){$0$}
		\htext(-9.5 3){$2$}
		
		\htext(-11.7 1.5){$0$}
		\htext(-10.3 0.5){$1$}
		\htext(-11.3 0.5){$1$}
		\htext(-12.3 0.5){$1$}
		\htext(-13.3 0.5){$1$}
		\htext(-14.3 0.5){$1$}
		\htext(-10.7 1.5){$0$}
		
		\htext(-15.7 1.5){$0$}
		\htext(-16.7 1.5){$0$}
		\move(-4 10)\lvec(-4 12)
		\move(-4 12)\lvec(-3 12)
		\htext(-3.5 11){$2$}
		\move(-9 4)\lvec(-9 6)	
		\move(-9 6)\lvec(-8 6)
		\htext(-8.5 5){$3$}

		\esegment

		\move(-15 -15)
		\bsegment
		\move(-3 0)\rlvec(-14 0)	
		
		\move(-3 2)\rlvec(-14 0)
		
		\move(-3 0 )\rlvec(0 10)
		\move(-4 0 )\rlvec(0 10)
		\move(-5 0 )\rlvec(0 4)
		\move(-6 0 )\rlvec(0 8)
		\move(-7 0 )\rlvec(0 8)
		\move(-8 0 )\rlvec(0 6)
		\move(-9 0 )\rlvec(0 4)
		\move(-10 0 )\rlvec(0 4)
		\move(-11 0 )\rlvec(0 2)
		\move(-12 0 )\rlvec(0 2)
		\move(-13 0 )\rlvec(0 2)
		\move(-14 0 )\rlvec(0 2)
		\move(-15 0 )\rlvec(0 2)
		
		\move(-3 10)\rlvec(-1 0)
		
		\move(-6 8)\rlvec(-1 0)	
		\move(-3 4)\rlvec(-7 0)
		\move(-6 6)\rlvec(-2 0)
		\move(-3 6)\rlvec(-1 0)
		\move(-3 8)\rlvec(-1 0)

		\move(-18 0)\lvec(-17 0)
		\move(-18 0)\lvec(-18 2)
		\move(-18 2)\lvec(-17 2)
		\move(-18 0)\lvec(-17 2)
		\htext(-16.7 1.5){$0$}
		\htext(-17.7 1.5){$0$}

		\move(-13 0)\lvec(-12 2)\lvec(-12 0)\lvec(-13 0) \lfill f:0.8
		\htext(-12.3 0.5){$1$}

		\move(-3 10)\rlvec(0 2)
		\move(-5 8)\lvec(-5 10)
		\move(-5 10)\lvec(-4 10)
		\htext(-4.5 9){$3$}	
		\move(-6 6)\rlvec(1 0)
		\move(-6 8)\rlvec(1 0)
		\move(-6 6)\rlvec(1 2)
		\htext(-5.5 5){$3$}

		\move(-4 0)\rlvec(1 2)
		\move(-5 0)\rlvec(1 2)
		\move(-5 4)\rlvec(0 4)
		\move(-4 8)\rlvec(-1 0)
		\move(-4 6)\rlvec(-1 0)
		\htext(-4.5 5){$3$}
		\move(-5 6)\rlvec(1 2)
		\htext(-4.3 6.5){$4$}
		\htext(-4.7 7.5){$4$}
		\htext(-5.3 6.5){$4$}
		\htext(-5.7 7.5){$4$}
		
		\move(-6 0)\rlvec(1 2)
		\move(-7 0)\rlvec(1 2)
		\move(-8 0)\rlvec(1 2)
		\move(-9 0)\rlvec(1 2)
		\move(-10 0)\rlvec(1 2)
		\move(-11 0)\rlvec(1 2)
		\move(-12 0)\rlvec(1 2)
		\move(-13 0)\rlvec(1 2)
		\move(-14 0)\rlvec(1 2)
		\move(-15 0)\rlvec(1 2)
		\move(-7 6)\rlvec(1 2)
		
		\move(-4 6)\rlvec(1 2)
		
		\move(-11 2)\lvec(-11 4)
		\move(-11 4)\lvec(-10 4)
		\move(-16 0)\lvec(-16 2)
		\move(-17 0)\lvec(-17 2)
		\move(-17 0)\lvec(-16 2)
		\move(-16 0)\lvec(-15 2)
		\htext(-10.5 3){$2$}

		\htext(-3.3 0.5){$1$}
		\htext(-3.7 1.5){$0$}
		\htext(-3.5 3){$2$}
		\htext(-3.5 5){$3$}	
		\htext(-3.3 6.5){$4$}
		\htext(-3.7 7.5){$4$}
		\htext(-3.5 9){$3$}
		
		\htext(-4.3 0.5){$1$}
		\htext(-4.7 1.5){$0$}
		\htext(-4.5 3){$2$}
		
		\htext(-5.3 0.5){$1$}
		\htext(-5.7 1.5){$0$}
		\htext(-5.5 3){$2$}
		
		\htext(-6.3 0.5){$1$}
		\htext(-6.7 1.5){$0$}
		\htext(-6.5 3){$2$}
		\htext(-6.5 5){$3$}
		\htext(-6.3 6.5){$4$}

		\htext(-7.3 0.5){$1$}
		\htext(-7.7 1.5){$0$}
		\htext(-7.5 3){$2$}
		\htext(-7.5 5){$3$}

		\htext(-8.3 0.5){$1$}
		\htext(-8.7 1.5){$0$}
		\htext(-8.5 3){$2$}

		\htext(-9.3 0.5){$1$}
		\htext(-9.7 1.5){$0$}
		\htext(-9.5 3){$2$}
		
		\htext(-11.7 1.5){$0$}
		\htext(-10.3 0.5){$1$}
		\htext(-11.3 0.5){$1$}
		\htext(-13.3 0.5){$1$}
		\htext(-14.3 0.5){$1$}
		\htext(-15.3 0.5){$1$}
		\htext(-10.7 1.5){$0$}
		
		\move(-4 10)\lvec(-4 12)
		\move(-4 12)\lvec(-3 12)
		\htext(-3.5 11){$2$}
		\move(-9 4)\lvec(-9 6)	
		\move(-9 6)\lvec(-8 6)
		\htext(-8.5 5){$3$}	
		\esegment				
		
		\move(-17.5 6)\ravec(2 0)\htext(-16.5 6.8){\fontsize{11}{11}\selectfont$\tilde{\mathfrak f}_{1}$}	
		
		\move(-7 -0.5)\ravec(0 -2)\htext(-3 -1.3){\fontsize{11}{11}\selectfont{remove the $\delta$-slice}}	
		
		\move(-15.5 -9)\ravec(-2 0)\htext(-14.6 -8.2){\fontsize{11}{11}\selectfont{move the slice}}			
	\end{texdraw}
\end{center}

	\end{example}

\vskip 3mm

Let $C\in\mathcal C^{(l)}$ be a level-$l$ reduced Young column. We define $\tilde{\mathfrak e}_iC=C'$ ($i\in I$) by removing the $i$-block in $C$ such that $\tilde{\mathfrak f}_iC'=C$. We denote the number of $i$-blocks which can be removed and added in $C$ by $\epsilon_i(C)$ and $\phi_i(C)$, respectively. We define $\wt(C)=\sum_{i=0}^{n}(\phi_i(C)-\epsilon_i(C))\Lambda_i$.

\vskip 2mm

According to the definitions of $\tilde{\mathfrak e}_i$, $\tilde{\mathfrak f}_i$, $\epsilon_i$, $\phi_i$ ($i\in I$) and $\wt$ on ${\mathcal C}^{(l)}$, we  have the following theorem:

\begin{theorem}\label{perfect crystal isomorphism}
For the quantum affine algebras of types $A_{2n}^{(2)}$, $D_{n+1}^{(2)}$, $A_{2n-1}^{(2)}$, $D_{n}^{(1)}$, $B_{n}^{(1)}$ and $C_{n}^{(1)}$, there is a crystal isomorphism 
$$
\Phi: {\mathcal C}^{(l)}\to\mathcal B^{(l)},\ \ C_{b}\to b.
$$	
	\end{theorem}

\begin{proof}
We now focus on the type $B_{n}^{(1)}$. Let $C_b$ be a level-$l$ reduced Young column in $\mathcal C^{(l)}$.	We assume $\tilde{\mathfrak f}_iC_b=C_{b'}$. According to the definitions of Kashiwara operators $\tilde{f}_i$ on $\mathcal B^{(l)}$ in \eqref{Kashiwara operators1}, \eqref{Kashiwara operators2} and the maps $\tilde{\mathfrak f}_i$ of type $B_{n}^{(1)}$ in 3.3\eqref{1}, 3.3\eqref{3}-3.3\eqref{5}, we have $\tilde{f}_ib=b'$. Hence $\Phi\tilde{\mathfrak f}_i=\tilde{f}_i\Phi$. 

\vskip 2mm

We can verify the commutative relation  $\Phi\tilde{\mathfrak e}_i=\tilde{e}_i\Phi$ by the same approach. The surjectivity of $\Phi$ is natural. The injectivity of $\Phi$ follows 
from the definitions of level-$l$ reduced Young columns. The proofs of other types  are similar.
\end{proof}
		
\vskip 5mm 

\section{Higher level Young walls}\label{Level-$l$ Young Wall}

\vskip 3mm 

In this section, we introduce the notions of ground-state walls, Young walls and reduced Young
walls. We show that the set of reduced Young
walls is isomorphic to $B(\lambda)$ as $U_{q}'({\mathfrak g})$-crystals. 

\vskip 2mm 

\subsection{Ground-state columns and ground-state walls}

\hfill

\vskip 2mm

The \textit{ground-state column} is a special case of Young column. The \textit{ground-state wall}  is constructed from the ground-state columns by extending to the left infinitely. Let $\lambda=\sum_{i=0}^{n}k_i\Lambda_i$ be a level $l$ dominant integral weight and let $b$ or $b'$ be the minimal vectors in ${\mathcal B}^{(l)}$. We list the constructions of the ground-state columns and ground-state walls for each type as follows: 

\vskip 2mm

\begin{enumerate}
	\item
	For type $A_{2n}^{(2)}$,  we have
	\begin{equation*}
	\langle\lambda,c\rangle=k_0+2k_1+\cdots+2k_{n-1}+2k_n=l.
	\end{equation*}
	Let $b=(k_1,\cdots,k_n,k_n,\cdots,k_1)$. Then the ground-state column is $C_b$ and the ground-state wall is
	\begin{equation*}	
	 Y_{\lambda}=(\cdots,C_b,C_b,C_b).
	 \end{equation*}
	 
	 \vskip 2mm 
	 
	\item 
    For type $D_{n+1}^{(2)}$,  we have
	\begin{equation*}
	\langle\lambda,c\rangle=k_0+2k_1+\cdots+2k_{n-1}+k_n=l.
\end{equation*}
	Let $b=(x_1,\cdots,x_n,x_0,\bar{x}_n,\cdots,\bar{x}_1)$ such that
	\begin{equation*}
	\begin{aligned}
		x_0&=\begin{cases}
			0 &\text{if $k_n$ is even,}\\
			1 &\text{if $k_n$ is odd,}
		\end{cases}\\
		x_i&=\bar{x}_i=k_i\  \ \text{for $i=1,2,\cdots,n-1$,}\\
		x_n&=\bar{x}_n=\frac{1}{2}(k_n-x_0).
	\end{aligned}
    \end{equation*}
	Then the ground-state column is $C_b$ and the ground-state wall is 
	\begin{equation*}
	Y_{\lambda}=(\cdots,C_b,C_b,C_b).
	\end{equation*}

\vskip 2mm 

	\item 
	For type $C_{n}^{(1)}$,  we have
	\begin{equation*}
	\langle\lambda,c\rangle=k_0+k_1+\cdots+k_{n-1}+k_n=l.
\end{equation*}
	Let $b=(k_1,\cdots,k_n,k_n,\cdots,k_1)$. Then the ground-state column is $C_b$ and the ground-state wall is
	\begin{equation*}
		 Y_{\lambda}=(\cdots,C_b,C_b,C_b).
\end{equation*}	
	
	\vskip 2mm
	
\item 
For type $D_{n}^{(1)}$,  we have
\begin{equation*}
\langle\lambda,c\rangle=k_0+k_1+2k_2+\cdots+2k_{n-2}+k_{n-1}+k_n=l.
\end{equation*}
Let $b=(x_1,\cdots,x_n,\bar{x}_n,\cdots,\bar{x}_1)$ such that
\begin{equation*}
\begin{aligned}
	x_1&=k_0,\  \bar{x}_1=k_1,\\
	x_i&=\bar{x}_i=k_i\ \  \text{for $i=2,\cdots,n-2$,}\\
	x_{n-1}&=\bar{x}_{n-1}={\rm min}(k_{n-1}, k_n),\\
	x_{n}&={(k_{n-1}-k_n)}_{+},\\
	\bar{x}_{n}&={(k_n-k_{n-1})}_{+}
\end{aligned}
\end{equation*}
and let  $b'=(y_1,\cdots,y_n,\bar{y}_n,\cdots,\bar{y}_1)$ such that
\begin{equation*}
\begin{aligned}
	y_1&=k_1,\  \bar{y}_1=k_0,\\
	y_i&=\bar{y}_i=k_i\ \  \text{for $i=2,\cdots,n-2$,}\\
	y_{n-1}&=\bar{y}_{n-1}={\rm min}(k_{n-1}, k_n),\\
	y_{n}&={(k_{n-1}-k_n)}_{+},\\
	\bar{y}_{n}&={(k_n-k_{n-1})}_{+}.
\end{aligned}
\end{equation*}
Then the ground-state columns are $C_b$ and $C_{b'}$ and the ground-state wall is
\begin{equation*}
Y_{\lambda}=(\cdots,C_b,C_{b'},C_b,C_{b'}).
\end{equation*}

\vskip 2mm 

\item 
For type $B_{n}^{(1)}$,  we have
\begin{equation*}
\langle\lambda,c\rangle=k_0+k_1+2k_2+\cdots+2k_{n-1}+k_n=l.
\end{equation*}
Let $b=(x_1,\cdots,x_n,x_0,\bar{x}_n,\cdots,\bar{x}_1)$ such that
\begin{equation*}
\begin{aligned}
	x_0&=\begin{cases}
		0 &\text{if $k_n$ is even,}\\
		1 &\text{if $k_n$ is odd,}
	\end{cases}\\
	x_1&=k_0,\ \ \bar{x}_1=k_1,\\
	x_i&=\bar{x}_i=k_i\  \ \text{for $i=2,\cdots,n-1$,}\\
	x_n&=\bar{x}_n=\frac{1}{2}(k_n-x_0)
\end{aligned}
\end{equation*}
and let  $b'=(y_1,\cdots,y_n,y_0,\bar{y}_n,\cdots,\bar{y}_1)$ such that
\begin{equation*}
\begin{aligned}
	y_0&=\begin{cases}
		0 &\text{if $k_n$ is even,}\\
		1 &\text{if $k_n$ is odd,}
	\end{cases}\\
	y_1&=k_1,\ \ \bar{y}_1=k_0,\\
	y_i&=\bar{y}_i=k_i\  \ \text{for $i=2,\cdots,n-1$,}\\
	y_n&=\bar{y}_n=\frac{1}{2}(k_n-y_0).
\end{aligned}
\end{equation*}
Then the  ground-state columns are $C_b$ and $C_{b'}$ and the ground-state wall is 
\begin{equation*}
Y_{\lambda}=(\cdots,C_b,C_{b'},C_b,C_{b'}).
\end{equation*}

\vskip 2mm 

\item 
For type $A_{2n-1}^{(2)}$,  we have
\begin{equation*}
\langle\lambda,c\rangle=k_0+k_1+2k_2+\cdots+2k_{n-1}+2k_n=l.
\end{equation*}
Let $b=(x_1,\cdots,x_n,\bar{x}_n,\cdots,\bar{x}_1)$ such that
\begin{equation*}
\begin{aligned}
	x_1&=k_0,\  \bar{x}_1=k_1,\\
	x_i&=\bar{x}_i=k_i\ \  \text{for $i=2,\cdots,n$}
\end{aligned}
\end{equation*}
and let $b'=(y_1,\cdots,y_n,\bar{y}_n,\cdots,\bar{y}_1)$ such that
\begin{equation*}
\begin{aligned}
	y_1&=k_1,\  \bar{y}_1=k_0,\\
	y_i&=\bar{y}_i=k_i\ \  \text{for $i=2,\cdots,n$.}
\end{aligned}
\end{equation*}
Then the ground-state columns are $C_b$ and $C_{b'}$ and the ground-state wall is 
\begin{equation*}
Y_{\lambda}=(\cdots,C_b,C_{b'},C_b,C_{b'}).
\end{equation*}
\end{enumerate}

\vskip 2mm 

\begin{remark}
The ground-state columns and ground-state walls constructed above correspond to the minimal vectors and ground-state paths in \cite{KMN2}.
	\end{remark}

\vskip 1mm

In the following examples, we 
list the ground-state walls of types $A_{4}^{(2)}$, $A_{5}^{(2)}$ and $B_3^{(1)}$ for the level-$2$ case.

\begin{example}
	For type $A_{4}^{(2)}$, the level-$2$ dominant integral weights are $2\Lambda_0$, $\Lambda_1$, $\Lambda_2$ and the corresponding minimal  vectors are
	
\begin{equation*}
	b_1=(0,0,0,0),\ b_2=(1,0,0,1),\ b_3=(0,1,1,0).
\end{equation*}

\vskip 4mm

The ground-state walls corresponding to these vectors are

\vskip 8mm

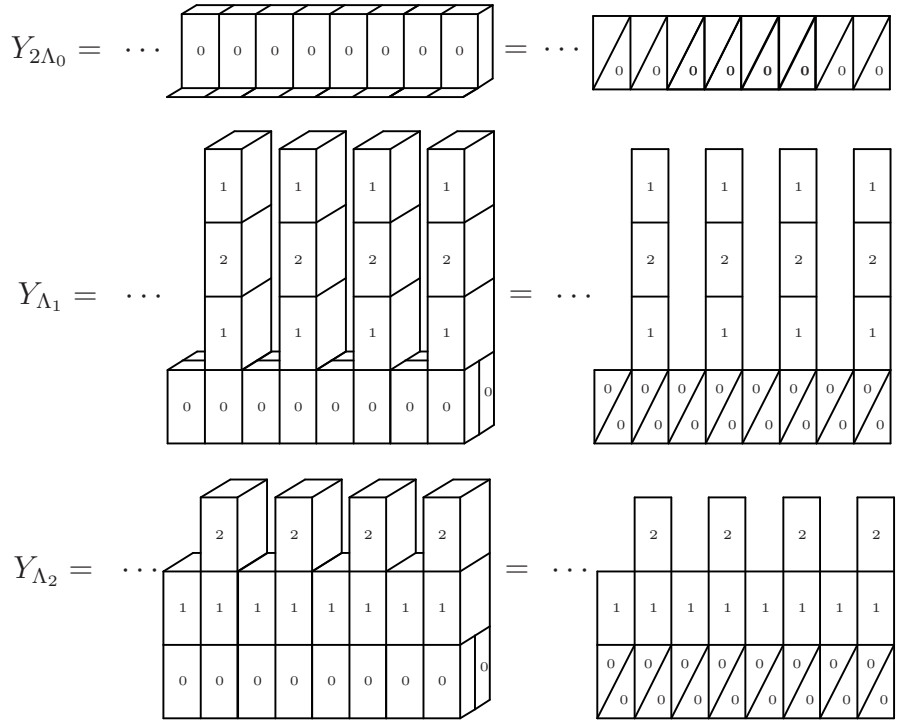
\begin{center}
	\begin{texdraw}
		\fontsize{3}{3}\selectfont
		\drawdim em
		\setunitscale 2.3
		
		\move(-0.2 0)
		\bsegment
		\move(0 0)\lvec(0.4 0.25)
		\move(0.4 0.25)\lvec(2.4 0.25)
		\move(0.4 0.25)\lvec(0.4 2.25)
		\move(0 0)\lvec(2 0)
		\move(2.4 0.25)\lvec(2 0)
		\move(2 0)\lvec(2.4 0.25)
		\move(0.4 2.25)\lvec(0.8 2.5)
		\move(0.4 2.25)\lvec(2.4 2.25)
		\move(2.4 2.25)\lvec(2.8 2.5)
		\move(2.8 2.5)\lvec(2.8 0.5)
		\move(2.4 0.25)\lvec(2.8 0.5)
		\move(2.4 0.25)\lvec(2.4 2.25)
		\move(0.8 2.5)\lvec(2.8 2.5)
		\move(1.4 2.25)\lvec(1.8 2.5)
		\move(1.4 0.25)\lvec(1.4 2.25)
		\move(1 0)\lvec(1.4 0.25)
		\htext(1.9 1.25){$0$}\htext(0.9 1.25){$0$}
		\esegment
		
		\move(-2.2 0)
		\bsegment
		\move(0 0)\lvec(0.4 0.25)
		\move(0.4 0.25)\lvec(2.4 0.25)
		\move(0.4 0.25)\lvec(0.4 2.25)
		\move(0 0)\lvec(2 0)
		\move(2.4 0.25)\lvec(2 0)
		\move(2 0)\lvec(2.4 0.25)
		\move(0.4 2.25)\lvec(0.8 2.5)
		\move(0.4 2.25)\lvec(2.4 2.25)
		\move(2.4 2.25)\lvec(2.8 2.5)
		
		\move(0.8 2.5)\lvec(2.8 2.5)
		\move(1.4 2.25)\lvec(1.8 2.5)
		\move(1.4 0.25)\lvec(1.4 2.25)
		\move(1 0)\lvec(1.4 0.25)
		\htext(1.9 1.25){$0$}\htext(0.9 1.25){$0$}
		\esegment
		
		\move(-4.2 0)
		\bsegment
		\move(0 0)\lvec(0.4 0.25)
		\move(0.4 0.25)\lvec(2.4 0.25)
		\move(0.4 0.25)\lvec(0.4 2.25)
		\move(0 0)\lvec(2 0)
		\move(2.4 0.25)\lvec(2 0)
		
		\move(0.4 2.25)\lvec(0.8 2.5)
		\move(0.4 2.25)\lvec(2.4 2.25)

		\move(0.8 2.5)\lvec(2.8 2.5)
		\move(1.4 2.25)\lvec(1.8 2.5)
		\move(1.4 0.25)\lvec(1.4 2.25)
		\move(1 0)\lvec(1.4 0.25)
		\htext(1.9 1.25){$0$}\htext(0.9 1.25){$0$}
		\esegment
		
		\move(-6.2 0)
		\bsegment
		\move(0 0)\lvec(0.4 0.25)
		\move(0.4 0.25)\lvec(2.4 0.25)
		\move(0.4 0.25)\lvec(0.4 2.25)
		\move(0 0)\lvec(2 0)
		\move(2.4 0.25)\lvec(2 0)
		
		\move(0.4 2.25)\lvec(0.8 2.5)
		\move(0.4 2.25)\lvec(2.4 2.25)

		\move(0.8 2.5)\lvec(2.8 2.5)
		\move(1.4 2.25)\lvec(1.8 2.5)
		\move(1.4 0.25)\lvec(1.4 2.25)
		\move(1 0)\lvec(1.4 0.25)
		\htext(1.9 1.25){$0$}\htext(0.9 1.25){$0$}
		\esegment
		\htext(-6.8 1.25){\fontsize{12}{12}\selectfont{$\cdots$}}
		
		\htext(-9.2 1.25){\fontsize{12}{12}\selectfont{$Y_{2\Lambda_0}=$}}
		
		\htext(3.2 1.25){\fontsize{12}{12}\selectfont{$=$}}
		
		\move(5.3 0.2)
		\bsegment
		\move(0 0)\lvec(0 2)
		\move(0 0)\lvec(4 0)
		\move(0 2)\lvec(4 2)
		\move(4 0)\lvec(4 2)
		\move(1 0)\lvec(1 2)
		\move(2 0)\lvec(2 2)
		\move(3 0)\lvec(3 2)
		\move(0 0)\lvec(1 2)
		\move(1 0)\lvec(2 2)
		\move(2 0)\lvec(3 2)
		\move(3 0)\lvec(4 2)
		\htext(0.7 0.5){$0$}
		\htext(1.7 0.5){$0$}
		\htext(2.7 0.5){$0$}
		\htext(3.7 0.5){$0$}
		\esegment
		\move(7.3 0.2)
		\bsegment
		\move(0 0)\lvec(4 0)
		\move(0 2)\lvec(4 2)
		\move(4 0)\lvec(4 2)
		\move(1 0)\lvec(1 2)
		\move(2 0)\lvec(2 2)
		\move(3 0)\lvec(3 2)
		\move(0 0)\lvec(1 2)
		\move(1 0)\lvec(2 2)
		\move(2 0)\lvec(3 2)
		\move(3 0)\lvec(4 2)
		\htext(0.7 0.5){$0$}
		\htext(1.7 0.5){$0$}
		\htext(2.7 0.5){$0$}
		\htext(3.7 0.5){$0$}
		\esegment
		\move(9.3 0.2)
		\bsegment
		\move(0 0)\lvec(4 0)
		\move(0 2)\lvec(4 2)
		\move(4 0)\lvec(4 2)
		\move(1 0)\lvec(1 2)
		\move(2 0)\lvec(2 2)
		\move(3 0)\lvec(3 2)
		\move(0 0)\lvec(1 2)
		\move(1 0)\lvec(2 2)
		\move(2 0)\lvec(3 2)
		\move(3 0)\lvec(4 2)
		\htext(0.7 0.5){$0$}
		\htext(1.7 0.5){$0$}
		\htext(2.7 0.5){$0$}
		\htext(3.7 0.5){$0$}
		\esegment
		\htext(4.5 1.29){\fontsize{12}{12}\selectfont{$\cdots$}}
	\end{texdraw}
\end{center}

\vskip 3mm

\begin{center}
	\begin{texdraw}
		\fontsize{3}{3}\selectfont
		\drawdim em
		\setunitscale 2.3
		\move(0 0)\lvec(1 0)\move(1 0)\lvec(2 0)
		\move(0 0)\lvec(0 2)\move(1 0)\lvec(1 2)
		\move(2 0)\lvec(2 2)\move(2 0)\lvec(2.4 0.25)	
		\move(2.4 0.25)\lvec(2.8 0.5)	\move(2.4 0.25)\lvec(2.4 2.25)\move(2.8 0.5)\lvec(2.8 2.5)
		\move(0 2)\lvec(2 2)\move(2 2)\lvec(2.8 2.5)
		\move(0.4  2.25)\lvec(1 2.25)\move(0.8  2.5)\lvec(1 2.5)\move(1  2)\lvec(1 8)\move(2  2)\lvec(2 8)	\move(2.8  2.5)\lvec(2.8 8.5)\move(1  4)\lvec(2 4)\move(2  4)\lvec(2.8 4.5)\move(1  6)\lvec(2 6)\move(2  6)\lvec(2.8 6.5)\move(1  8)\lvec(2 8)\move(2  8)\lvec(2.8 8.5)\move(1  8)\lvec(1.8 8.5)
		\move(1.8  8.5)\lvec(2.8 8.5)
		\htext(1.5 1){$0$}\htext(1.5 3){$1$}\htext(1.5 5){$2$}\htext(1.5 7){$1$}\htext(0.5 1){$0$}
		\htext(2.63 1.4){$0$}
		
		\move(-2 0)
		\bsegment
		\move(0 0)\lvec(1 0)\move(1 0)\lvec(2 0)
		\move(0 0)\lvec(0 2)\move(1 0)\lvec(1 2)
		\move(2 0)\lvec(2 2)	
		
		\move(0 2)\lvec(2 2)\move(2 2)\lvec(2.8 2.5)
		\move(0.4  2.25)\lvec(1 2.25)\move(0.8  2.5)\lvec(1 2.5)\move(1  2)\lvec(1 8)\move(2  2)\lvec(2 8)	\move(2.8  2.5)\lvec(2.8 8.5)\move(1  4)\lvec(2 4)\move(2  4)\lvec(2.8 4.5)\move(1  6)\lvec(2 6)\move(2  6)\lvec(2.8 6.5)\move(1  8)\lvec(2 8)\move(2  8)\lvec(2.8 8.5)\move(1  8)\lvec(1.8 8.5)
		\move(1.8  8.5)\lvec(2.8 8.5)
		\move(0 2)\lvec(0.8 2.5)
		\htext(1.5 1){$0$}\htext(1.5 3){$1$}\htext(1.5 5){$2$}\htext(1.5 7){$1$}\htext(0.5 1){$0$}
		\esegment
		
		\move(-4 0)
		\bsegment
		\move(0 0)\lvec(1 0)\move(1 0)\lvec(2 0)
		\move(0 0)\lvec(0 2)\move(1 0)\lvec(1 2)

		\move(0 2)\lvec(2 2)\move(2 2)\lvec(2.8 2.5)
		\move(0.4  2.25)\lvec(1 2.25)\move(0.8  2.5)\lvec(1 2.5)\move(1  2)\lvec(1 8)\move(2  2)\lvec(2 8)	\move(2.8  2.5)\lvec(2.8 8.5)\move(1  4)\lvec(2 4)\move(2  4)\lvec(2.8 4.5)\move(1  6)\lvec(2 6)\move(2  6)\lvec(2.8 6.5)\move(1  8)\lvec(2 8)\move(2  8)\lvec(2.8 8.5)\move(1  8)\lvec(1.8 8.5)
		\move(1.8  8.5)\lvec(2.8 8.5)
		\move(0 2)\lvec(0.8 2.5)
		\htext(1.5 1){$0$}\htext(1.5 3){$1$}\htext(1.5 5){$2$}\htext(1.5 7){$1$}\htext(0.5 1){$0$}
		\esegment
		
		\move(-6 0)
		\bsegment
		\move(0 0)\lvec(1 0)\move(1 0)\lvec(2 0)
		\move(0 0)\lvec(0 2)\move(1 0)\lvec(1 2)

		\move(0 2)\lvec(2 2)\move(2 2)\lvec(2.8 2.5)
		\move(0.4  2.25)\lvec(1 2.25)\move(0.8  2.5)\lvec(1 2.5)\move(1  2)\lvec(1 8)\move(2  2)\lvec(2 8)	\move(2.8  2.5)\lvec(2.8 8.5)\move(1  4)\lvec(2 4)\move(2  4)\lvec(2.8 4.5)\move(1  6)\lvec(2 6)\move(2  6)\lvec(2.8 6.5)\move(1  8)\lvec(2 8)\move(2  8)\lvec(2.8 8.5)\move(1  8)\lvec(1.8 8.5)
		\move(1.8  8.5)\lvec(2.8 8.5)
		\move(0 2)\lvec(0.8 2.5)
		\htext(1.5 1){$0$}\htext(1.5 3){$1$}\htext(1.5 5){$2$}\htext(1.5 7){$1$}\htext(0.5 1){$0$}
		\esegment
		
		\move(5.5 0)
		\bsegment
		\move(0 0)\lvec(1 0)\move(1 0)\lvec(2 0)
		\move(0 0)\lvec(0 2)\move(1 0)\lvec(1 2)
		\move(2 0)\lvec(2 2)\move(0 2)\lvec(2 2)
		\move(2 2)\lvec(2 8)\move(1 2)\lvec(1 8)
		\move(1 8)\lvec(2 8)\move(1 4)\lvec(2 4)
		\move(1 6)\lvec(2 6)\move(0 0)\lvec(1 2)
		\move(1 0)\lvec(2 2)
		\htext(0.7 0.5){$0$}\htext(1.7 0.5){$0$}
		\htext(0.35 1.5){$0$}\htext(1.35 1.5){$0$}
		\htext(1.5 3){$1$}\htext(1.5 5){$2$}\htext(1.5 7){$1$}
		\esegment
		\move(7.5 0)
		\bsegment
		\move(0 0)\lvec(1 0)\move(1 0)\lvec(2 0)
		\move(1 0)\lvec(1 2)
		\move(2 0)\lvec(2 2)\move(0 2)\lvec(2 2)
		\move(2 2)\lvec(2 8)\move(1 2)\lvec(1 8)
		\move(1 8)\lvec(2 8)\move(1 4)\lvec(2 4)
		\move(1 6)\lvec(2 6)\move(0 0)\lvec(1 2)
		\move(1 0)\lvec(2 2)
		\htext(0.7 0.5){$0$}\htext(1.7 0.5){$0$}
		\htext(0.35 1.5){$0$}\htext(1.35 1.5){$0$}
		\htext(1.5 3){$1$}\htext(1.5 5){$2$}\htext(1.5 7){$1$}
		\esegment
		\move(9.5 0)
		\bsegment
		\move(0 0)\lvec(1 0)\move(1 0)\lvec(2 0)
		\move(1 0)\lvec(1 2)
		\move(2 0)\lvec(2 2)\move(0 2)\lvec(2 2)
		\move(2 2)\lvec(2 8)\move(1 2)\lvec(1 8)
		\move(1 8)\lvec(2 8)\move(1 4)\lvec(2 4)
		\move(1 6)\lvec(2 6)\move(0 0)\lvec(1 2)
		\move(1 0)\lvec(2 2)
		\htext(0.7 0.5){$0$}\htext(1.7 0.5){$0$}
		\htext(0.35 1.5){$0$}\htext(1.35 1.5){$0$}
		\htext(1.5 3){$1$}\htext(1.5 5){$2$}\htext(1.5 7){$1$}
		\esegment
		\move(11.5 0)
		\bsegment
		\move(0 0)\lvec(1 0)\move(1 0)\lvec(2 0)
		\move(1 0)\lvec(1 2)
		\move(2 0)\lvec(2 2)\move(0 2)\lvec(2 2)
		\move(2 2)\lvec(2 8)\move(1 2)\lvec(1 8)
		\move(1 8)\lvec(2 8)\move(1 4)\lvec(2 4)
		\move(1 6)\lvec(2 6)\move(0 0)\lvec(1 2)
		\move(1 0)\lvec(2 2)
		\htext(0.7 0.5){$0$}\htext(1.7 0.5){$0$}
		\htext(0.35 1.5){$0$}\htext(1.35 1.5){$0$}
		\htext(1.5 3){$1$}\htext(1.5 5){$2$}\htext(1.5 7){$1$}
		\esegment
		\htext(3.5 4){\fontsize{12}{12}\selectfont{$=$}}
		\htext(5 4){\fontsize{12}{12}\selectfont{$\cdots$}}
		\htext(-6.6 4){\fontsize{12}{12}\selectfont{$\cdots$}}
		\htext(-9 4){\fontsize{12}{12}\selectfont{$Y_{\Lambda_1}=$}}
	\end{texdraw}
\end{center}

\vskip 3mm

\begin{center}
	\begin{texdraw}
		\fontsize{3}{3}\selectfont
		\drawdim em
		\setunitscale 2.3
		\move(0 0)\lvec(1 0)\move(1 0)\lvec(2 0)
		\move(0 0)\lvec(0 2)\move(1 0)\lvec(1 2)
		\move(2 0)\lvec(2 2)\move(2 0)\lvec(2.4 0.25)	
		\move(2.4 0.25)\lvec(2.8 0.5)
		\move(2.4 0.25)\lvec(2.4 2.25)
		\move(2.8 0.5)\lvec(2.8 2.5)
		\move(0 4)\lvec(1 4)
		\move(0.8 4.5)\lvec(1 4.5)
		\move(0 2)\lvec(2 2)
		\move(2 2)\lvec(2.8 2.5)
		\move(1  2)\lvec(1 6)
		\move(2  2)\lvec(2 6)	
		\move(2.8  2.5)\lvec(2.8 6.5)
		\move(1  4)\lvec(2 4)
		\move(2  4)\lvec(2.8 4.5)
		\move(1  6)\lvec(2 6)
		\move(2  6)\lvec(2.8 6.5)
		\move(1  6)\lvec(1.8 6.5)
		\move(1.8 6.5)\lvec(2.8 6.5)
		\htext(1.5 1){$0$}\htext(1.5 3){$1$}\htext(1.5 5){$2$}\htext(0.5 1){$0$}\htext(0.5 3){$1$}
		\htext(2.63 1.4){$0$}
		
		\move(-2 0)
		\bsegment
		\move(0 0)\lvec(1 0)\move(1 0)\lvec(2 0)
		\move(0 0)\lvec(0 2)\move(1 0)\lvec(1 2)
		\move(2 0)\lvec(2 2)	
		
		\move(0 2)\lvec(2 2)
		\move(1  2)\lvec(1 6)\move(2  2)\lvec(2 6)	\move(2.8  4.5)\lvec(2.8 6.5)\move(1  4)\lvec(2 4)\move(2  4)\lvec(2.8 4.5)\move(1  6)\lvec(2 6)\move(2  6)\lvec(2.8 6.5)
		\move(1 6)\lvec(1.8 6.5)
		\move(1.8 6.5)\lvec(2.8 6.5)
		\move(0 2)\lvec(0 4)\move(0 4)\lvec(1 4)
		\move(0 4)\lvec(0.8 4.5)\move(0.8 4.5)\lvec(1 4.5)
		\htext(1.5 1){$0$}\htext(1.5 3){$1$}\htext(1.5 5){$2$}\htext(0.5 1){$0$}\htext(0.5 3){$1$}
		\esegment
		
		\move(-4 0)
		\bsegment
		\move(0 0)\lvec(1 0)\move(1 0)\lvec(2 0)
		\move(0 0)\lvec(0 2)\move(1 0)\lvec(1 2)
		\move(2 0)\lvec(2 2)	
		
		\move(0 2)\lvec(2 2)
		\move(1  2)\lvec(1 6)\move(2  2)\lvec(2 6)	\move(2.8  4.5)\lvec(2.8 6.5)\move(1  4)\lvec(2 4)\move(2  4)\lvec(2.8 4.5)\move(1  6)\lvec(2 6)\move(2  6)\lvec(2.8 6.5)
		\move(1 6)\lvec(1.8 6.5)
		\move(1.8 6.5)\lvec(2.8 6.5)
		\move(0 2)\lvec(0 4)\move(0 4)\lvec(1 4)
		\move(0 4)\lvec(0.8 4.5)\move(0.8 4.5)\lvec(1 4.5)
		\htext(1.5 1){$0$}\htext(1.5 3){$1$}\htext(1.5 5){$2$}\htext(0.5 1){$0$}\htext(0.5 3){$1$}
		\esegment
		
		\move(-6 0)
		\bsegment
		\move(0 0)\lvec(1 0)\move(1 0)\lvec(2 0)
		\move(0 0)\lvec(0 2)\move(1 0)\lvec(1 2)
		\move(2 0)\lvec(2 2)	
		
		\move(0 2)\lvec(2 2)
		\move(1  2)\lvec(1 6)\move(2  2)\lvec(2 6)	\move(2.8  4.5)\lvec(2.8 6.5)\move(1  4)\lvec(2 4)\move(2  4)\lvec(2.8 4.5)\move(1  6)\lvec(2 6)\move(2  6)\lvec(2.8 6.5)
		\move(1 6)\lvec(1.8 6.5)
		\move(1.8 6.5)\lvec(2.8 6.5)
		\move(0 2)\lvec(0 4)\move(0 4)\lvec(1 4)
		\move(0 4)\lvec(0.8 4.5)\move(0.8 4.5)\lvec(1 4.5)
		\htext(1.5 1){$0$}\htext(1.5 3){$1$}\htext(1.5 5){$2$}\htext(0.5 1){$0$}\htext(0.5 3){$1$}
		\esegment
		
		\move(5.7 0)
		\bsegment
		\move(0 0)\lvec(1 0)\move(1 0)\lvec(2 0)
		\move(0 0)\lvec(0 2)\move(1 0)\lvec(1 2)
		\move(2 0)\lvec(2 2)\move(0 2)\lvec(2 2)
		\move(2 2)\lvec(2 6)\move(1 2)\lvec(1 6)
		\move(1 4)\lvec(2 4)
		\move(1 6)\lvec(2 6)\move(0 0)\lvec(1 2)
		\move(1 0)\lvec(2 2)
		\move(0 4)\lvec(1 4)\move(0 2)\lvec(0 4)
		\htext(0.7 0.5){$0$}\htext(1.7 0.5){$0$}
		\htext(0.35 1.5){$0$}\htext(1.35 1.5){$0$}\htext(0.5 3){$1$}
		\htext(1.5 3){$1$}\htext(1.5 5){$2$}
		\esegment
		\move(7.7 0)
		\bsegment
		\move(0 0)\lvec(1 0)\move(1 0)\lvec(2 0)
		\move(1 0)\lvec(1 2)
		\move(2 0)\lvec(2 2)\move(0 2)\lvec(2 2)
		\move(2 2)\lvec(2 6)\move(1 2)\lvec(1 6)
		\move(1 4)\lvec(2 4)
		\move(1 6)\lvec(2 6)\move(0 0)\lvec(1 2)
		\move(1 0)\lvec(2 2)\move(0 4)\lvec(1 4)
		\htext(0.7 0.5){$0$}\htext(1.7 0.5){$0$}
		\htext(0.35 1.5){$0$}\htext(1.35 1.5){$0$}\htext(0.5 3){$1$}
		\htext(1.5 3){$1$}\htext(1.5 5){$2$}
		\esegment
		\move(9.7 0)
		\bsegment
		\move(0 0)\lvec(1 0)\move(1 0)\lvec(2 0)
		\move(1 0)\lvec(1 2)
		\move(2 0)\lvec(2 2)\move(0 2)\lvec(2 2)
		\move(2 2)\lvec(2 6)\move(1 2)\lvec(1 6)
		\move(1 4)\lvec(2 4)
		\move(1 6)\lvec(2 6)\move(0 0)\lvec(1 2)
		\move(1 0)\lvec(2 2)\move(0 4)\lvec(1 4)
		\htext(0.7 0.5){$0$}\htext(1.7 0.5){$0$}
		\htext(0.35 1.5){$0$}\htext(1.35 1.5){$0$}\htext(0.5 3){$1$}
		\htext(1.5 3){$1$}\htext(1.5 5){$2$}
		\esegment
		\move(11.7 0)
		\bsegment
		\move(0 0)\lvec(1 0)\move(1 0)\lvec(2 0)
		\move(1 0)\lvec(1 2)
		\move(2 0)\lvec(2 2)\move(0 2)\lvec(2 2)
		\move(2 2)\lvec(2 6)\move(1 2)\lvec(1 6)
		\move(1 4)\lvec(2 4)
		\move(1 6)\lvec(2 6)\move(0 0)\lvec(1 2)
		\move(1 0)\lvec(2 2)\move(0 4)\lvec(1 4)
		\htext(0.7 0.5){$0$}\htext(1.7 0.5){$0$}
		\htext(0.35 1.5){$0$}\htext(1.35 1.5){$0$}\htext(0.5 3){$1$}
		\htext(1.5 3){$1$}\htext(1.5 5){$2$}
		\esegment
		\htext(3.5 4){\fontsize{12}{12}\selectfont{$=$}}
		\htext(5 4){\fontsize{12}{12}\selectfont{$\cdots$}}
		\htext(-6.6 4){\fontsize{12}{12}\selectfont{$\cdots$}}
		\htext(-9 4){\fontsize{12}{12}\selectfont{$Y_{\Lambda_2}=$}}
	\end{texdraw}
\end{center}

	\end{example}
	
	\vskip 2mm

\begin{example}
	For type $A_{5}^{(2)}$, the level-$2$ dominant integral weights are $2\Lambda_0$, $2\Lambda_1$, $\Lambda_0+\Lambda_1$, $\Lambda_2$, $\Lambda_3$ and the corresponding minimal vectors are
	
	\begin{equation*}
\begin{aligned}
	&b_1=(2,0,0,0,0,0), b'_1=(0,0,0,0,0,2),\\ &b_2=(0,0,0,0,0,2),\ b'_2=(2,0,0,0,0,0),\\
	 &b_3=b'_3=(1,0,0,0,0,1),\\
    &b_4=b'_4=(0,1,0,0,1,0),\\
    &b_5=b'_5=(0,0,1,1,0,0). 
\end{aligned}
\end{equation*}

\vskip 3mm

The ground-state walls corresponding to these vectors are	

\vskip 3mm

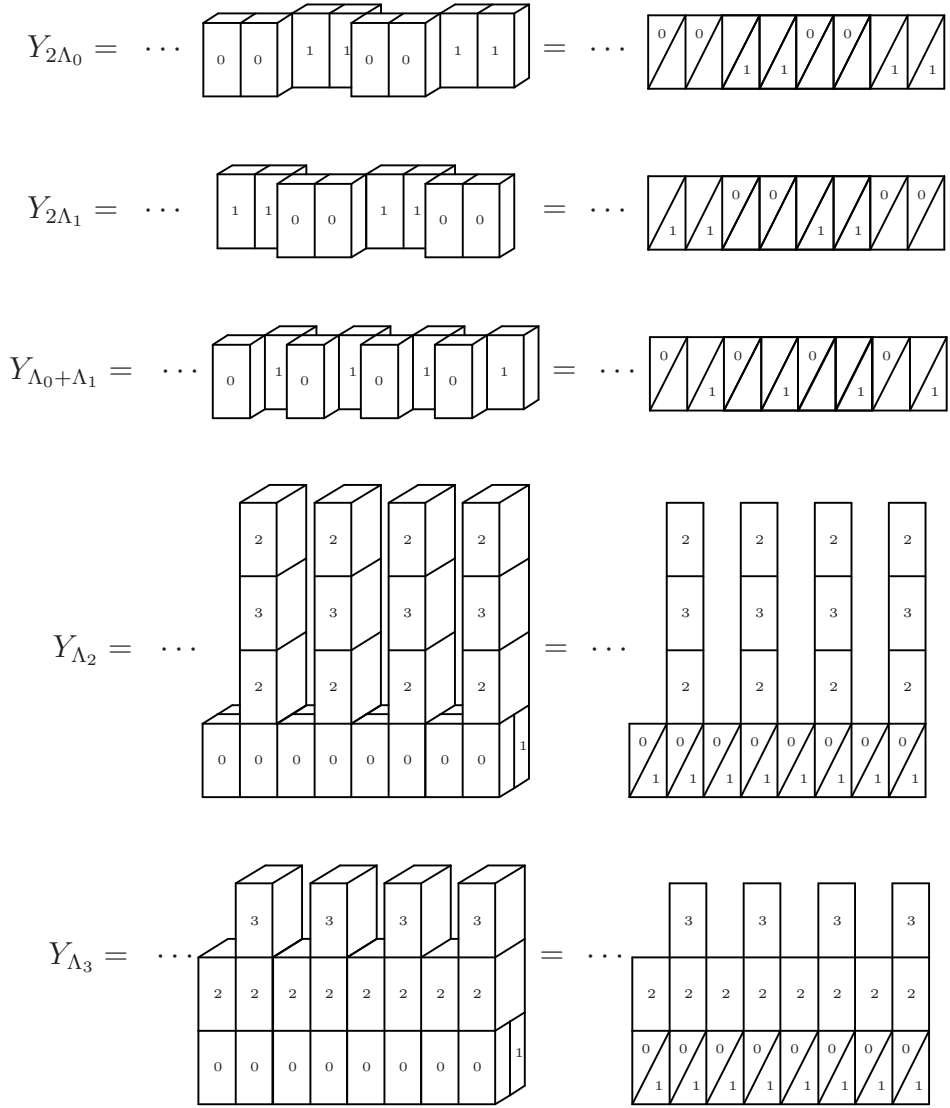
\begin{center}
	\begin{texdraw}
		\fontsize{3}{3}\selectfont
		\drawdim em
		\setunitscale 2.3
		\move(0 0)\lvec(0 2)
		\move(0 2)\lvec(0.4 2.25)
		\move(0 0)\lvec(2 0)
		\move(0 2)\lvec(2 2)
		\move(0.4 2.25)\lvec(2.4 2.25)
		\move(1 0)\lvec(1 2)
		\move(1 2)\lvec(1.4 2.25)
		\move(2 0)\lvec(2 2)
		\move(2 2)\lvec(2.4 2.25)
		\move(2 0)\lvec(2.4 0.25)
		\move(2.4 0.25)\lvec(2.4 2.25)
		\move(2.4 0.25)\lvec(4.4 0.25)
		\move(3.4 0.25)\lvec(3.4 2.25)	
		\move(4.4 0.25)\lvec(4.4 2.25)
		\move(2.4 2.25)\lvec(4.4 2.25)
		\move(2.4 2.25)\lvec(2.8 2.5)
		\move(4.4 2.25)\lvec(4.8 2.5)
		\move(2.8 2.5)\lvec(4.8 2.5)
		\move(3.4 2.25)\lvec(3.8 2.5)
		\move(4.8 0.5)\lvec(4.8 2.5)	
		\move(4.4 0.25)\lvec(4.8 0.5)
		\htext(0.5 1){$0$}
		\htext(1.5 1){$0$}
		\htext(2.9 1.25){$1$}
		\htext(3.9 1.25){$1$}	
		
		\move(-4 0)	
		\bsegment
		\move(0 0)\lvec(0 2)
		\move(0 2)\lvec(0.4 2.25)
		\move(0 0)\lvec(2 0)
		\move(0 2)\lvec(2 2)
		\move(0.4 2.25)\lvec(2.4 2.25)
		\move(1 0)\lvec(1 2)
		\move(1 2)\lvec(1.4 2.25)
		\move(2 0)\lvec(2 2)
		\move(2 2)\lvec(2.4 2.25)
		\move(2 0)\lvec(2.4 0.25)
		\move(2.4 0.25)\lvec(2.4 2.25)
		\move(2.4 0.25)\lvec(4 0.25)
		\move(3.4 0.25)\lvec(3.4 2.25)	
		
		\move(2.4 2.25)\lvec(4.4 2.25)
		\move(2.4 2.25)\lvec(2.8 2.5)
		\move(4.4 2.25)\lvec(4.8 2.5)
		\move(2.8 2.5)\lvec(4.8 2.5)
		\move(3.4 2.25)\lvec(3.8 2.5)
		\move(4.8 2.25)\lvec(4.8 2.5)
		
		\htext(0.5 1){$0$}
		\htext(1.5 1){$0$}
		\htext(2.9 1.25){$1$}
		\htext(3.9 1.25){$1$}
		\esegment
		
		\move(8 0.2)
		\bsegment
		\move(0 0)\lvec(0 2)
		\move(0 0)\lvec(4 0)
		\move(0 2)\lvec(4 2)
		\move(4 0)\lvec(4 2)
		\move(1 0)\lvec(1 2)
		\move(2 0)\lvec(2 2)
		\move(3 0)\lvec(3 2)
		\move(0 0)\lvec(1 2)
		\move(1 0)\lvec(2 2)
		\move(2 0)\lvec(3 2)
		\move(3 0)\lvec(4 2)
		\htext(0.4 1.5){$0$}
		\htext(1.4 1.5){$0$}
		
		\esegment
		\move(10 0.2)
		\bsegment
		\move(0 0)\lvec(4 0)
		\move(0 2)\lvec(4 2)
		\move(4 0)\lvec(4 2)
		\move(1 0)\lvec(1 2)
		\move(2 0)\lvec(2 2)
		\move(3 0)\lvec(3 2)
		\move(0 0)\lvec(1 2)
		\move(1 0)\lvec(2 2)
		\move(2 0)\lvec(3 2)
		\move(3 0)\lvec(4 2)
		\htext(0.7 0.5){$1$}
		\htext(1.7 0.5){$1$}
		
		\esegment
		\move(12 0.2)
		\bsegment
		\move(0 0)\lvec(4 0)
		\move(0 2)\lvec(4 2)
		\move(4 0)\lvec(4 2)
		\move(1 0)\lvec(1 2)
		\move(2 0)\lvec(2 2)
		\move(3 0)\lvec(3 2)
		\move(0 0)\lvec(1 2)
		\move(1 0)\lvec(2 2)
		\move(2 0)\lvec(3 2)
		\move(3 0)\lvec(4 2)
		\htext(0.4 1.5){$0$}
		\htext(1.4 1.5){$0$}
		\htext(2.7 0.5){$1$}
		\htext(3.7 0.5){$1$}
		\esegment
		\htext(5.5 1.25){\fontsize{12}{12}\selectfont{$=$}}
		\htext(7 1.25){\fontsize{12}{12}\selectfont{$\cdots$}}
		\htext(-5 1.25){\fontsize{12}{12}\selectfont{$\cdots$}}
		
		\htext(-7.5 1.25){\fontsize{12}{12}\selectfont{$Y_{2\Lambda_0}=$}}
	\end{texdraw}
\end{center}

\vskip 3mm

\begin{center}
	\begin{texdraw}
		\fontsize{3}{3}\selectfont
		\drawdim em
		\setunitscale 2.3
		\move(-2 0)	
		\bsegment
		\move(0 0)\lvec(0 2)
		\move(0 2)\lvec(0.4 2.25)
		\move(0 0)\lvec(2 0)
		\move(0 2)\lvec(2 2)
		\move(0.4 2.25)\lvec(2.4 2.25)
		\move(1 0)\lvec(1 2)
		\move(1 2)\lvec(1.4 2.25)
		\move(2 0)\lvec(2 2)
		\move(2 2)\lvec(2.4 2.25)
		\move(2 0)\lvec(2.4 0.25)
		\move(2.4 0.25)\lvec(2.4 2.25)
		\htext(0.5 1){$0$}
		\htext(1.5 1){$0$}
		\esegment

		\move(-6 0)	
		\bsegment
		\move(2.4 0.25)\lvec(2.4 2.25)
		\move(2.4 0.25)\lvec(4 0.25)
		\move(3.4 0.25)\lvec(3.4 2.25)	
		
		\move(2.4 2.25)\lvec(4.4 2.25)
		\move(2.4 2.25)\lvec(2.8 2.5)
		\move(4.4 2.25)\lvec(4.8 2.5)
		\move(2.8 2.5)\lvec(4.8 2.5)
		\move(3.4 2.25)\lvec(3.8 2.5)
		\move(4.8 2.25)\lvec(4.8 2.5)
		\htext(2.9 1.25){$1$}
		\htext(3.8 1.25){$1$}
		\esegment
		
		\move(-2 0)	
		\bsegment
		\move(2.4 0.25)\lvec(2.4 2.25)
		\move(2.4 0.25)\lvec(4 0.25)
		\move(3.4 0.25)\lvec(3.4 2.25)	
		
		\move(2.4 2.25)\lvec(4.4 2.25)
		\move(2.4 2.25)\lvec(2.8 2.5)
		\move(4.4 2.25)\lvec(4.8 2.5)
		\move(2.8 2.5)\lvec(4.8 2.5)
		\move(3.4 2.25)\lvec(3.8 2.5)
		\move(4.8 2.25)\lvec(4.8 2.5)
		\htext(2.9 1.25){$1$}
		\htext(3.8 1.25){$1$}
		\esegment
		
		\move(2 0)	
		\bsegment
		\move(0 0)\lvec(0 2)
		\move(0 2)\lvec(0.4 2.25)
		\move(0 0)\lvec(2 0)
		\move(0 2)\lvec(2 2)
		\move(0.4 2.25)\lvec(2.4 2.25)
		\move(1 0)\lvec(1 2)
		\move(1 2)\lvec(1.4 2.25)
		\move(2 0)\lvec(2 2)
		\move(2 2)\lvec(2.4 2.25)
		\move(2 0)\lvec(2.4 0.25)
		\move(2.4 0.25)\lvec(2.4 2.25)
		\htext(0.5 1){$0$}
		\htext(1.5 1){$0$}
		\esegment
		
		\move(8 0.2)
		\bsegment
		\move(0 0)\lvec(0 2)
		\move(0 0)\lvec(4 0)
		\move(0 2)\lvec(4 2)
		\move(4 0)\lvec(4 2)
		\move(1 0)\lvec(1 2)
		\move(2 0)\lvec(2 2)
		\move(3 0)\lvec(3 2)
		\move(0 0)\lvec(1 2)
		\move(1 0)\lvec(2 2)
		\move(2 0)\lvec(3 2)
		\move(3 0)\lvec(4 2)
		\htext(0.7 0.5){$1$}
		\htext(1.7 0.5){$1$}
		
		\esegment
		\move(10 0.2)
		\bsegment
		\move(0 0)\lvec(4 0)
		\move(0 2)\lvec(4 2)
		\move(4 0)\lvec(4 2)
		\move(1 0)\lvec(1 2)
		\move(2 0)\lvec(2 2)
		\move(3 0)\lvec(3 2)
		\move(0 0)\lvec(1 2)
		\move(1 0)\lvec(2 2)
		\move(2 0)\lvec(3 2)
		\move(3 0)\lvec(4 2)
		\htext(0.4 1.5){$0$}
		\htext(1.4 1.5){$0$}
		
		\esegment
		\move(12 0.2)
		\bsegment
		\move(0 0)\lvec(4 0)
		\move(0 2)\lvec(4 2)
		\move(4 0)\lvec(4 2)
		\move(1 0)\lvec(1 2)
		\move(2 0)\lvec(2 2)
		\move(3 0)\lvec(3 2)
		\move(0 0)\lvec(1 2)
		\move(1 0)\lvec(2 2)
		\move(2 0)\lvec(3 2)
		\move(3 0)\lvec(4 2)
		\htext(2.4 1.5){$0$}
		\htext(3.4 1.5){$0$}
		\htext(0.7 0.5){$1$}
		\htext(1.7 0.5){$1$}
		\esegment
		\htext(5.5 1.25){\fontsize{12}{12}\selectfont{$=$}}
		\htext(7 1.25){\fontsize{12}{12}\selectfont{$\cdots$}}
		\htext(-5 1.25){\fontsize{12}{12}\selectfont{$\cdots$}}
		
		\htext(-7.5 1.25){\fontsize{12}{12}\selectfont{$Y_{2\Lambda_1}=$}}
	\end{texdraw}
\end{center}

\vskip 3mm

\begin{center}
	\begin{texdraw}
		\fontsize{3}{3}\selectfont
		\drawdim em
		\setunitscale 2.3
		\move(2 0)
		\bsegment	
		\move(0 0)\lvec(1 0)
		\move(1 0)\lvec(1 2)
		\move(0 0)\lvec(0 2)
		\move(0 2)\lvec(1 2)
		\move(0 2)\lvec(0.4 2.25)
		\move(1 2)\lvec(1.4 2.25)
		\move(0.4 2.25)\lvec(1.4 2.25)
		\move(1 0)\lvec(1.4 0.25)
		\move(1.4 0.25)\lvec(1.4 2.25)
		\move(1.4 0.25)\lvec(2.4 0.25)
		\move(2.4 0.25)\lvec(2.8 0.5)
		\move(2.4 0.25)\lvec(2.4 2.25)
		\move(2.8 0.5)\lvec(2.8 2.5)
		\move(2.4 2.25)\lvec(2.8 2.5)
		\move(1.4 2.25)\lvec(2.4 2.25)
		\move(1.4 2.25)\lvec(1.8 2.5)
		\move(1.8 2.5)\lvec(2.8 2.5)
		\htext(0.5 1){$0$}	
		\htext(1.9 1.25){$1$}
		\esegment
		
		\move(0 0)
		\bsegment
		\move(0 0)\lvec(1 0)
		\move(1 0)\lvec(1 2)
		\move(0 0)\lvec(0 2)
		\move(0 2)\lvec(1 2)
		\move(0 2)\lvec(0.4 2.25)
		\move(1 2)\lvec(1.4 2.25)
		\move(0.4 2.25)\lvec(1.4 2.25)
		\move(1 0)\lvec(1.4 0.25)
		\move(1.4 0.25)\lvec(1.4 2.25)
		\move(1.4 0.25)\lvec(2 0.25)

		\move(2.8 2.25)\lvec(2.8 2.5)
		\move(2.4 2.25)\lvec(2.8 2.5)
		\move(1.4 2.25)\lvec(2.4 2.25)
		\move(1.4 2.25)\lvec(1.8 2.5)
		\move(1.8 2.5)\lvec(2.8 2.5)
		\htext(0.5 1){$0$}	
		\htext(1.8 1.25){$1$}
		\esegment
		
		\move(-2 0)
		\bsegment
		\move(0 0)\lvec(1 0)
		\move(1 0)\lvec(1 2)
		\move(0 0)\lvec(0 2)
		\move(0 2)\lvec(1 2)
		\move(0 2)\lvec(0.4 2.25)
		\move(1 2)\lvec(1.4 2.25)
		\move(0.4 2.25)\lvec(1.4 2.25)
		\move(1 0)\lvec(1.4 0.25)
		\move(1.4 0.25)\lvec(1.4 2.25)
		\move(1.4 0.25)\lvec(2 0.25)

		\move(2.8 2.25)\lvec(2.8 2.5)
		\move(2.4 2.25)\lvec(2.8 2.5)
		\move(1.4 2.25)\lvec(2.4 2.25)
		\move(1.4 2.25)\lvec(1.8 2.5)
		\move(1.8 2.5)\lvec(2.8 2.5)
		\htext(0.5 1){$0$}	
		\htext(1.8 1.25){$1$}
		\esegment
		
		\move(-4 0)
		\bsegment
		\move(0 0)\lvec(1 0)
		\move(1 0)\lvec(1 2)
		\move(0 0)\lvec(0 2)
		\move(0 2)\lvec(1 2)
		\move(0 2)\lvec(0.4 2.25)
		\move(1 2)\lvec(1.4 2.25)
		\move(0.4 2.25)\lvec(1.4 2.25)
		\move(1 0)\lvec(1.4 0.25)
		\move(1.4 0.25)\lvec(1.4 2.25)
		\move(1.4 0.25)\lvec(2 0.25)

		\move(2.8 2.25)\lvec(2.8 2.5)
		\move(2.4 2.25)\lvec(2.8 2.5)
		\move(1.4 2.25)\lvec(2.4 2.25)
		\move(1.4 2.25)\lvec(1.8 2.5)
		\move(1.8 2.5)\lvec(2.8 2.5)
		\htext(0.5 1){$0$}	
		\htext(1.8 1.25){$1$}
		\esegment
		
		\move(7.8 0.2)
		\bsegment
		\move(0 0)\lvec(0 2)
		\move(0 0)\lvec(4 0)
		\move(0 2)\lvec(4 2)
		\move(4 0)\lvec(4 2)
		\move(1 0)\lvec(1 2)
		\move(2 0)\lvec(2 2)
		\move(3 0)\lvec(3 2)
		\move(0 0)\lvec(1 2)
		\move(1 0)\lvec(2 2)
		\move(2 0)\lvec(3 2)
		\move(3 0)\lvec(4 2)
		\htext(1.7 0.5){$1$}
		\htext(3.7 0.5){$1$}
		
		\esegment
		\move(9.8 0.2)
		\bsegment
		\move(0 0)\lvec(4 0)
		\move(0 2)\lvec(4 2)
		\move(4 0)\lvec(4 2)
		\move(1 0)\lvec(1 2)
		\move(2 0)\lvec(2 2)
		\move(3 0)\lvec(3 2)
		\move(0 0)\lvec(1 2)
		\move(1 0)\lvec(2 2)
		\move(2 0)\lvec(3 2)
		\move(3 0)\lvec(4 2)
		\htext(-1.6 1.5){$0$}
		\htext(0.4 1.5){$0$}
		
		\esegment
		\move(11.8 0.2)
		\bsegment
		\move(0 0)\lvec(4 0)
		\move(0 2)\lvec(4 2)
		\move(4 0)\lvec(4 2)
		\move(1 0)\lvec(1 2)
		\move(2 0)\lvec(2 2)
		\move(3 0)\lvec(3 2)
		\move(0 0)\lvec(1 2)
		\move(1 0)\lvec(2 2)
		\move(2 0)\lvec(3 2)
		\move(3 0)\lvec(4 2)
		\htext(2.4 1.5){$0$}
		\htext(0.4 1.5){$0$}
		\htext(1.7 0.5){$1$}
		\htext(3.7 0.5){$1$}
		\esegment
		\htext(5.5 1.25){\fontsize{12}{12}\selectfont{$=$}}
		\htext(7 1.25){\fontsize{12}{12}\selectfont{$\cdots$}}
		\htext(-4.8 1.25){\fontsize{12}{12}\selectfont{$\cdots$}}
		
		\htext(-7.8 1.25){\fontsize{12}{12}\selectfont{$Y_{\Lambda_0+\Lambda_1}=$}}	
		
	\end{texdraw}
	
\end{center}
	
\vskip 3mm

\begin{center}
	\begin{texdraw}
		\fontsize{3}{3}\selectfont
		\drawdim em
		\setunitscale 2.3
		\move(0 0)\lvec(1 0)\move(1 0)\lvec(2 0)
		\move(0 0)\lvec(0 2)\move(1 0)\lvec(1 2)
		\move(2 0)\lvec(2 2)\move(2 0)\lvec(2.4 0.25)	
		\move(2.4 0.25)\lvec(2.8 0.5)	\move(2.4 0.25)\lvec(2.4 2.25)\move(2.8 0.5)\lvec(2.8 2.5)
		\move(0 2)\lvec(2 2)\move(2 2)\lvec(2.8 2.5)
		\move(0.4  2.25)\lvec(1 2.25)\move(0.8  2.5)\lvec(1 2.5)\move(1  2)\lvec(1 8)\move(2  2)\lvec(2 8)	\move(2.8  2.5)\lvec(2.8 8.5)\move(1  4)\lvec(2 4)\move(2  4)\lvec(2.8 4.5)\move(1  6)\lvec(2 6)\move(2  6)\lvec(2.8 6.5)\move(1  8)\lvec(2 8)\move(2  8)\lvec(2.8 8.5)\move(1  8)\lvec(1.8 8.5)
		\move(1.8  8.5)\lvec(2.8 8.5)
		\htext(1.5 1){$0$}\htext(1.5 3){$2$}\htext(1.5 5){$3$}\htext(1.5 7){$2$}\htext(0.5 1){$0$}
		\htext(2.63 1.4){$1$}
		
		\move(-2 0)
		\bsegment
		\move(0 0)\lvec(1 0)\move(1 0)\lvec(2 0)
		\move(0 0)\lvec(0 2)\move(1 0)\lvec(1 2)
		\move(2 0)\lvec(2 2)	
		
		\move(0 2)\lvec(2 2)\move(2 2)\lvec(2.8 2.5)
		\move(0.4  2.25)\lvec(1 2.25)\move(0.8  2.5)\lvec(1 2.5)\move(1  2)\lvec(1 8)\move(2  2)\lvec(2 8)	\move(2.8  2.5)\lvec(2.8 8.5)\move(1  4)\lvec(2 4)\move(2  4)\lvec(2.8 4.5)\move(1  6)\lvec(2 6)\move(2  6)\lvec(2.8 6.5)\move(1  8)\lvec(2 8)\move(2  8)\lvec(2.8 8.5)\move(1  8)\lvec(1.8 8.5)
		\move(1.8  8.5)\lvec(2.8 8.5)
		\move(0 2)\lvec(0.8 2.5)
		\htext(1.5 1){$0$}\htext(1.5 3){$2$}\htext(1.5 5){$3$}\htext(1.5 7){$2$}\htext(0.5 1){$0$}
		\esegment
		
		\move(-4 0)
		\bsegment
		\move(0 0)\lvec(1 0)\move(1 0)\lvec(2 0)
		\move(0 0)\lvec(0 2)\move(1 0)\lvec(1 2)

		\move(0 2)\lvec(2 2)\move(2 2)\lvec(2.8 2.5)
		\move(0.4  2.25)\lvec(1 2.25)\move(0.8  2.5)\lvec(1 2.5)\move(1  2)\lvec(1 8)\move(2  2)\lvec(2 8)	\move(2.8  2.5)\lvec(2.8 8.5)\move(1  4)\lvec(2 4)\move(2  4)\lvec(2.8 4.5)\move(1  6)\lvec(2 6)\move(2  6)\lvec(2.8 6.5)\move(1  8)\lvec(2 8)\move(2  8)\lvec(2.8 8.5)\move(1  8)\lvec(1.8 8.5)
		\move(1.8  8.5)\lvec(2.8 8.5)
		\move(0 2)\lvec(0.8 2.5)
		\htext(1.5 1){$0$}\htext(1.5 3){$2$}\htext(1.5 5){$3$}\htext(1.5 7){$2$}\htext(0.5 1){$0$}
		\esegment
		
		\move(-6 0)
		\bsegment
		\move(0 0)\lvec(1 0)\move(1 0)\lvec(2 0)
		\move(0 0)\lvec(0 2)\move(1 0)\lvec(1 2)

		\move(0 2)\lvec(2 2)\move(2 2)\lvec(2.8 2.5)
		\move(0.4  2.25)\lvec(1 2.25)\move(0.8  2.5)\lvec(1 2.5)\move(1  2)\lvec(1 8)\move(2  2)\lvec(2 8)	\move(2.8  2.5)\lvec(2.8 8.5)\move(1  4)\lvec(2 4)\move(2  4)\lvec(2.8 4.5)\move(1  6)\lvec(2 6)\move(2  6)\lvec(2.8 6.5)\move(1  8)\lvec(2 8)\move(2  8)\lvec(2.8 8.5)\move(1  8)\lvec(1.8 8.5)
		\move(1.8  8.5)\lvec(2.8 8.5)
		\move(0 2)\lvec(0.8 2.5)
		\htext(1.5 1){$0$}\htext(1.5 3){$2$}\htext(1.5 5){$3$}\htext(1.5 7){$2$}\htext(0.5 1){$0$}
		\esegment
		
		\move(5.5 0)
		\bsegment
		\move(0 0)\lvec(1 0)\move(1 0)\lvec(2 0)
		\move(0 0)\lvec(0 2)\move(1 0)\lvec(1 2)
		\move(2 0)\lvec(2 2)\move(0 2)\lvec(2 2)
		\move(2 2)\lvec(2 8)\move(1 2)\lvec(1 8)
		\move(1 8)\lvec(2 8)\move(1 4)\lvec(2 4)
		\move(1 6)\lvec(2 6)\move(0 0)\lvec(1 2)
		\move(1 0)\lvec(2 2)
		\htext(0.7 0.5){$1$}\htext(1.7 0.5){$1$}
		\htext(0.35 1.5){$0$}\htext(1.35 1.5){$0$}
		\htext(1.5 3){$2$}\htext(1.5 5){$3$}\htext(1.5 7){$2$}
		\esegment
		\move(7.5 0)
		\bsegment
		\move(0 0)\lvec(1 0)\move(1 0)\lvec(2 0)
		\move(1 0)\lvec(1 2)
		\move(2 0)\lvec(2 2)\move(0 2)\lvec(2 2)
		\move(2 2)\lvec(2 8)\move(1 2)\lvec(1 8)
		\move(1 8)\lvec(2 8)\move(1 4)\lvec(2 4)
		\move(1 6)\lvec(2 6)\move(0 0)\lvec(1 2)
		\move(1 0)\lvec(2 2)
		\htext(0.7 0.5){$1$}\htext(1.7 0.5){$1$}
		\htext(0.35 1.5){$0$}\htext(1.35 1.5){$0$}
		\htext(1.5 3){$2$}\htext(1.5 5){$3$}\htext(1.5 7){$2$}
		\esegment
		\move(9.5 0)
		\bsegment
		\move(0 0)\lvec(1 0)\move(1 0)\lvec(2 0)
		\move(1 0)\lvec(1 2)
		\move(2 0)\lvec(2 2)\move(0 2)\lvec(2 2)
		\move(2 2)\lvec(2 8)\move(1 2)\lvec(1 8)
		\move(1 8)\lvec(2 8)\move(1 4)\lvec(2 4)
		\move(1 6)\lvec(2 6)\move(0 0)\lvec(1 2)
		\move(1 0)\lvec(2 2)
		\htext(0.7 0.5){$1$}\htext(1.7 0.5){$1$}
		\htext(0.35 1.5){$0$}\htext(1.35 1.5){$0$}
		\htext(1.5 3){$2$}\htext(1.5 5){$3$}\htext(1.5 7){$2$}
		\esegment
		\move(11.5 0)
		\bsegment
		\move(0 0)\lvec(1 0)\move(1 0)\lvec(2 0)
		\move(1 0)\lvec(1 2)
		\move(2 0)\lvec(2 2)\move(0 2)\lvec(2 2)
		\move(2 2)\lvec(2 8)\move(1 2)\lvec(1 8)
		\move(1 8)\lvec(2 8)\move(1 4)\lvec(2 4)
		\move(1 6)\lvec(2 6)\move(0 0)\lvec(1 2)
		\move(1 0)\lvec(2 2)
		\htext(0.7 0.5){$1$}\htext(1.7 0.5){$1$}
		\htext(0.35 1.5){$0$}\htext(1.35 1.5){$0$}
		\htext(1.5 3){$2$}\htext(1.5 5){$3$}\htext(1.5 7){$2$}
		\esegment
		\htext(3.5 4){\fontsize{12}{12}\selectfont{$=$}}
		\htext(5 4){\fontsize{12}{12}\selectfont{$\cdots$}}
		\htext(-6.6 4){\fontsize{12}{12}\selectfont{$\cdots$}}
		\htext(-9 4){\fontsize{12}{12}\selectfont{$Y_{\Lambda_2}=$}}
	\end{texdraw}
\end{center}

\vskip 3mm

\begin{center}
	\begin{texdraw}
		\fontsize{3}{3}\selectfont
		\drawdim em
		\setunitscale 2.3
		\move(0 0)\lvec(1 0)\move(1 0)\lvec(2 0)
		\move(0 0)\lvec(0 2)\move(1 0)\lvec(1 2)
		\move(2 0)\lvec(2 2)\move(2 0)\lvec(2.4 0.25)	
		\move(2.4 0.25)\lvec(2.8 0.5)
		\move(2.4 0.25)\lvec(2.4 2.25)
		\move(2.8 0.5)\lvec(2.8 2.5)
		\move(0 4)\lvec(1 4)
		\move(0.8 4.5)\lvec(1 4.5)
		\move(0 2)\lvec(2 2)
		\move(2 2)\lvec(2.8 2.5)
		\move(1  2)\lvec(1 6)
		\move(2  2)\lvec(2 6)	
		\move(2.8  2.5)\lvec(2.8 6.5)
		\move(1  4)\lvec(2 4)
		\move(2  4)\lvec(2.8 4.5)
		\move(1  6)\lvec(2 6)
		\move(2  6)\lvec(2.8 6.5)
		\move(1  6)\lvec(1.8 6.5)
		\move(1.8 6.5)\lvec(2.8 6.5)
		\htext(1.5 1){$0$}\htext(1.5 3){$2$}\htext(1.5 5){$3$}\htext(0.5 1){$0$}\htext(0.5 3){$2$}
		\htext(2.63 1.4){$1$}
		
		\move(-2 0)
		\bsegment
		\move(0 0)\lvec(1 0)\move(1 0)\lvec(2 0)
		\move(0 0)\lvec(0 2)\move(1 0)\lvec(1 2)
		\move(2 0)\lvec(2 2)	
		
		\move(0 2)\lvec(2 2)
		\move(1  2)\lvec(1 6)\move(2  2)\lvec(2 6)	\move(2.8  4.5)\lvec(2.8 6.5)\move(1  4)\lvec(2 4)\move(2  4)\lvec(2.8 4.5)\move(1  6)\lvec(2 6)\move(2  6)\lvec(2.8 6.5)
		\move(1 6)\lvec(1.8 6.5)
		\move(1.8 6.5)\lvec(2.8 6.5)
		\move(0 2)\lvec(0 4)\move(0 4)\lvec(1 4)
		\move(0 4)\lvec(0.8 4.5)\move(0.8 4.5)\lvec(1 4.5)
		\htext(1.5 1){$0$}\htext(1.5 3){$2$}\htext(1.5 5){$3$}\htext(0.5 1){$0$}\htext(0.5 3){$2$}
		\esegment
		
		\move(-4 0)
		\bsegment
		\move(0 0)\lvec(1 0)\move(1 0)\lvec(2 0)
		\move(0 0)\lvec(0 2)\move(1 0)\lvec(1 2)
		\move(2 0)\lvec(2 2)	
		
		\move(0 2)\lvec(2 2)
		\move(1  2)\lvec(1 6)\move(2  2)\lvec(2 6)	\move(2.8  4.5)\lvec(2.8 6.5)\move(1  4)\lvec(2 4)\move(2  4)\lvec(2.8 4.5)\move(1  6)\lvec(2 6)\move(2  6)\lvec(2.8 6.5)
		\move(1 6)\lvec(1.8 6.5)
		\move(1.8 6.5)\lvec(2.8 6.5)
		\move(0 2)\lvec(0 4)\move(0 4)\lvec(1 4)
		\move(0 4)\lvec(0.8 4.5)\move(0.8 4.5)\lvec(1 4.5)
		\htext(1.5 1){$0$}\htext(1.5 3){$2$}\htext(1.5 5){$3$}\htext(0.5 1){$0$}\htext(0.5 3){$2$}
		\esegment
		
		\move(-6 0)
		\bsegment
		\move(0 0)\lvec(1 0)\move(1 0)\lvec(2 0)
		\move(0 0)\lvec(0 2)\move(1 0)\lvec(1 2)
		\move(2 0)\lvec(2 2)	
		
		\move(0 2)\lvec(2 2)
		\move(1  2)\lvec(1 6)\move(2  2)\lvec(2 6)	\move(2.8  4.5)\lvec(2.8 6.5)\move(1  4)\lvec(2 4)\move(2  4)\lvec(2.8 4.5)\move(1  6)\lvec(2 6)\move(2  6)\lvec(2.8 6.5)
		\move(1 6)\lvec(1.8 6.5)
		\move(1.8 6.5)\lvec(2.8 6.5)
		\move(0 2)\lvec(0 4)\move(0 4)\lvec(1 4)
		\move(0 4)\lvec(0.8 4.5)\move(0.8 4.5)\lvec(1 4.5)
		\htext(1.5 1){$0$}\htext(1.5 3){$2$}\htext(1.5 5){$3$}\htext(0.5 1){$0$}\htext(0.5 3){$2$}
		\esegment
		
		\move(5.7 0)
		\bsegment
		\move(0 0)\lvec(1 0)\move(1 0)\lvec(2 0)
		\move(0 0)\lvec(0 2)\move(1 0)\lvec(1 2)
		\move(2 0)\lvec(2 2)\move(0 2)\lvec(2 2)
		\move(2 2)\lvec(2 6)\move(1 2)\lvec(1 6)
		\move(1 4)\lvec(2 4)
		\move(1 6)\lvec(2 6)\move(0 0)\lvec(1 2)
		\move(1 0)\lvec(2 2)
		\move(0 4)\lvec(1 4)\move(0 2)\lvec(0 4)
		\htext(0.7 0.5){$1$}\htext(1.7 0.5){$1$}
		\htext(0.35 1.5){$0$}\htext(1.35 1.5){$0$}\htext(0.5 3){$2$}
		\htext(1.5 3){$2$}\htext(1.5 5){$3$}
		\esegment
		\move(7.7 0)
		\bsegment
		\move(0 0)\lvec(1 0)\move(1 0)\lvec(2 0)
		\move(1 0)\lvec(1 2)
		\move(2 0)\lvec(2 2)\move(0 2)\lvec(2 2)
		\move(2 2)\lvec(2 6)\move(1 2)\lvec(1 6)
		\move(1 4)\lvec(2 4)
		\move(1 6)\lvec(2 6)\move(0 0)\lvec(1 2)
		\move(1 0)\lvec(2 2)\move(0 4)\lvec(1 4)
		\htext(0.7 0.5){$1$}\htext(1.7 0.5){$1$}
		\htext(0.35 1.5){$0$}\htext(1.35 1.5){$0$}\htext(0.5 3){$2$}
		\htext(1.5 3){$2$}\htext(1.5 5){$3$}
		\esegment
		\move(9.7 0)
		\bsegment
		\move(0 0)\lvec(1 0)\move(1 0)\lvec(2 0)
		\move(1 0)\lvec(1 2)
		\move(2 0)\lvec(2 2)\move(0 2)\lvec(2 2)
		\move(2 2)\lvec(2 6)\move(1 2)\lvec(1 6)
		\move(1 4)\lvec(2 4)
		\move(1 6)\lvec(2 6)\move(0 0)\lvec(1 2)
		\move(1 0)\lvec(2 2)\move(0 4)\lvec(1 4)
		\htext(0.7 0.5){$1$}\htext(1.7 0.5){$1$}
		\htext(0.35 1.5){$0$}\htext(1.35 1.5){$0$}\htext(0.5 3){$2$}
		\htext(1.5 3){$2$}\htext(1.5 5){$3$}
		\esegment
		\move(11.7 0)
		\bsegment
		\move(0 0)\lvec(1 0)\move(1 0)\lvec(2 0)
		\move(1 0)\lvec(1 2)
		\move(2 0)\lvec(2 2)\move(0 2)\lvec(2 2)
		\move(2 2)\lvec(2 6)\move(1 2)\lvec(1 6)
		\move(1 4)\lvec(2 4)
		\move(1 6)\lvec(2 6)\move(0 0)\lvec(1 2)
		\move(1 0)\lvec(2 2)\move(0 4)\lvec(1 4)
		\htext(0.7 0.5){$1$}\htext(1.7 0.5){$1$}
		\htext(0.35 1.5){$0$}\htext(1.35 1.5){$0$}\htext(0.5 3){$2$}
		\htext(1.5 3){$2$}\htext(1.5 5){$3$}
		\esegment
		\htext(3.5 4){\fontsize{12}{12}\selectfont{$=$}}
		\htext(5 4){\fontsize{12}{12}\selectfont{$\cdots$}}
		\htext(-6.6 4){\fontsize{12}{12}\selectfont{$\cdots$}}
		\htext(-9 4){\fontsize{12}{12}\selectfont{$Y_{\Lambda_3}=$}}
	\end{texdraw}
\end{center}

\end{example}

\vskip 2mm 

\begin{example}
	For type $B_{3}^{(1)}$, the level-$2$ dominant integral weights are $2\Lambda_0$, $2\Lambda_1$, $\Lambda_0+\Lambda_1$, $\Lambda_2$, $2\Lambda_3$, $\Lambda_0+\Lambda_3$, $\Lambda_1+\Lambda_3$ and the corresponding minimal vectors are
	
	\vskip 3mm
	
	\begin{equation*}
		\begin{aligned}
			&b_1=(2,0,0,0,0,0,0),\ b'_1=(0,0,0,0,0,0,2),\\ &b_2=(0,0,0,0,0,0,2),\ b'_2=(2,0,0,0,0,0,0),\\
			&b_3=b'_3=(1,0,0,0,0,0,1),\\
			&b_4=b'_4=(0,1,0,0,0,1,0),\\
			&b_5=b'_5=(0,0,1,0,1,0,0),\\
			&b_6=(1,0,0,1,0,0,0),\ b'_6=(0,0,0,1,0,0,1),\\
			&b_7=(0,0,0,1,0,0,1),\ b'_7=(1,0,0,1,0,0,0).\\
		\end{aligned}
	\end{equation*}

	\vskip 3mm
	
	The ground-state walls corresponding to these vectors are	
	
\vskip 7mm

	\begin{center}
	\begin{texdraw}
		\fontsize{3}{3}\selectfont
		\drawdim em
		\setunitscale 2.3
		\move(0 0)\lvec(0 2)
		\move(0 2)\lvec(0.4 2.25)
		\move(0 0)\lvec(2 0)
		\move(0 2)\lvec(2 2)
		\move(0.4 2.25)\lvec(2.4 2.25)
		\move(1 0)\lvec(1 2)
		\move(1 2)\lvec(1.4 2.25)
		\move(2 0)\lvec(2 2)
		\move(2 2)\lvec(2.4 2.25)
		\move(2 0)\lvec(2.4 0.25)
		\move(2.4 0.25)\lvec(2.4 2.25)
		\move(2.4 0.25)\lvec(4.4 0.25)
		\move(3.4 0.25)\lvec(3.4 2.25)	
		\move(4.4 0.25)\lvec(4.4 2.25)
		\move(2.4 2.25)\lvec(4.4 2.25)
		\move(2.4 2.25)\lvec(2.8 2.5)
		\move(4.4 2.25)\lvec(4.8 2.5)
		\move(2.8 2.5)\lvec(4.8 2.5)
		\move(3.4 2.25)\lvec(3.8 2.5)
		\move(4.8 0.5)\lvec(4.8 2.5)	
		\move(4.4 0.25)\lvec(4.8 0.5)
		\htext(0.5 1){$0$}
		\htext(1.5 1){$0$}
		\htext(2.9 1.25){$1$}
		\htext(3.9 1.25){$1$}	
		
		\move(-4 0)	
		\bsegment
		\move(0 0)\lvec(0 2)
		\move(0 2)\lvec(0.4 2.25)
		\move(0 0)\lvec(2 0)
		\move(0 2)\lvec(2 2)
		\move(0.4 2.25)\lvec(2.4 2.25)
		\move(1 0)\lvec(1 2)
		\move(1 2)\lvec(1.4 2.25)
		\move(2 0)\lvec(2 2)
		\move(2 2)\lvec(2.4 2.25)
		\move(2 0)\lvec(2.4 0.25)
		\move(2.4 0.25)\lvec(2.4 2.25)
		\move(2.4 0.25)\lvec(4 0.25)
		\move(3.4 0.25)\lvec(3.4 2.25)	
		
		\move(2.4 2.25)\lvec(4.4 2.25)
		\move(2.4 2.25)\lvec(2.8 2.5)
		\move(4.4 2.25)\lvec(4.8 2.5)
		\move(2.8 2.5)\lvec(4.8 2.5)
		\move(3.4 2.25)\lvec(3.8 2.5)
		\move(4.8 2.25)\lvec(4.8 2.5)
		
		\htext(0.5 1){$0$}
		\htext(1.5 1){$0$}
		\htext(2.9 1.25){$1$}
		\htext(3.9 1.25){$1$}
		\esegment
		
		\move(8 0.2)
		\bsegment
		\move(0 0)\lvec(0 2)
		\move(0 0)\lvec(4 0)
		\move(0 2)\lvec(4 2)
		\move(4 0)\lvec(4 2)
		\move(1 0)\lvec(1 2)
		\move(2 0)\lvec(2 2)
		\move(3 0)\lvec(3 2)
		\move(0 0)\lvec(1 2)
		\move(1 0)\lvec(2 2)
		\move(2 0)\lvec(3 2)
		\move(3 0)\lvec(4 2)
		\htext(0.4 1.5){$0$}
		\htext(1.4 1.5){$0$}
		
		\esegment
		\move(10 0.2)
		\bsegment
		\move(0 0)\lvec(4 0)
		\move(0 2)\lvec(4 2)
		\move(4 0)\lvec(4 2)
		\move(1 0)\lvec(1 2)
		\move(2 0)\lvec(2 2)
		\move(3 0)\lvec(3 2)
		\move(0 0)\lvec(1 2)
		\move(1 0)\lvec(2 2)
		\move(2 0)\lvec(3 2)
		\move(3 0)\lvec(4 2)
		\htext(0.7 0.5){$1$}
		\htext(1.7 0.5){$1$}
		
		\esegment
		\move(12 0.2)
		\bsegment
		\move(0 0)\lvec(4 0)
		\move(0 2)\lvec(4 2)
		\move(4 0)\lvec(4 2)
		\move(1 0)\lvec(1 2)
		\move(2 0)\lvec(2 2)
		\move(3 0)\lvec(3 2)
		\move(0 0)\lvec(1 2)
		\move(1 0)\lvec(2 2)
		\move(2 0)\lvec(3 2)
		\move(3 0)\lvec(4 2)
		\htext(0.4 1.5){$0$}
		\htext(1.4 1.5){$0$}
		\htext(2.7 0.5){$1$}
		\htext(3.7 0.5){$1$}
		\esegment
		\htext(5.5 1.25){\fontsize{12}{12}\selectfont{$=$}}
		\htext(7 1.25){\fontsize{12}{12}\selectfont{$\cdots$}}
		\htext(-5 1.25){\fontsize{12}{12}\selectfont{$\cdots$}}
		
		\htext(-7.5 1.25){\fontsize{12}{12}\selectfont{$Y_{2\Lambda_0}=$}}
	\end{texdraw}
\end{center}

\vskip 3mm

\begin{center}
	\begin{texdraw}
		\fontsize{3}{3}\selectfont
		\drawdim em
		\setunitscale 2.3
		\move(-2 0)	
		\bsegment
		\move(0 0)\lvec(0 2)
		\move(0 2)\lvec(0.4 2.25)
		\move(0 0)\lvec(2 0)
		\move(0 2)\lvec(2 2)
		\move(0.4 2.25)\lvec(2.4 2.25)
		\move(1 0)\lvec(1 2)
		\move(1 2)\lvec(1.4 2.25)
		\move(2 0)\lvec(2 2)
		\move(2 2)\lvec(2.4 2.25)
		\move(2 0)\lvec(2.4 0.25)
		\move(2.4 0.25)\lvec(2.4 2.25)
		\htext(0.5 1){$0$}
		\htext(1.5 1){$0$}
		\esegment

		\move(-6 0)	
		\bsegment
		\move(2.4 0.25)\lvec(2.4 2.25)
		\move(2.4 0.25)\lvec(4 0.25)
		\move(3.4 0.25)\lvec(3.4 2.25)	
		
		\move(2.4 2.25)\lvec(4.4 2.25)
		\move(2.4 2.25)\lvec(2.8 2.5)
		\move(4.4 2.25)\lvec(4.8 2.5)
		\move(2.8 2.5)\lvec(4.8 2.5)
		\move(3.4 2.25)\lvec(3.8 2.5)
		\move(4.8 2.25)\lvec(4.8 2.5)
		\htext(2.9 1.25){$1$}
		\htext(3.8 1.25){$1$}
		\esegment
		
		\move(-2 0)	
		\bsegment
		\move(2.4 0.25)\lvec(2.4 2.25)
		\move(2.4 0.25)\lvec(4 0.25)
		\move(3.4 0.25)\lvec(3.4 2.25)	
		
		\move(2.4 2.25)\lvec(4.4 2.25)
		\move(2.4 2.25)\lvec(2.8 2.5)
		\move(4.4 2.25)\lvec(4.8 2.5)
		\move(2.8 2.5)\lvec(4.8 2.5)
		\move(3.4 2.25)\lvec(3.8 2.5)
		\move(4.8 2.25)\lvec(4.8 2.5)
		\htext(2.9 1.25){$1$}
		\htext(3.8 1.25){$1$}
		\esegment
		
		\move(2 0)	
		\bsegment
		\move(0 0)\lvec(0 2)
		\move(0 2)\lvec(0.4 2.25)
		\move(0 0)\lvec(2 0)
		\move(0 2)\lvec(2 2)
		\move(0.4 2.25)\lvec(2.4 2.25)
		\move(1 0)\lvec(1 2)
		\move(1 2)\lvec(1.4 2.25)
		\move(2 0)\lvec(2 2)
		\move(2 2)\lvec(2.4 2.25)
		\move(2 0)\lvec(2.4 0.25)
		\move(2.4 0.25)\lvec(2.4 2.25)
		\htext(0.5 1){$0$}
		\htext(1.5 1){$0$}
		\esegment
		
		\move(8 0.2)
		\bsegment
		\move(0 0)\lvec(0 2)
		\move(0 0)\lvec(4 0)
		\move(0 2)\lvec(4 2)
		\move(4 0)\lvec(4 2)
		\move(1 0)\lvec(1 2)
		\move(2 0)\lvec(2 2)
		\move(3 0)\lvec(3 2)
		\move(0 0)\lvec(1 2)
		\move(1 0)\lvec(2 2)
		\move(2 0)\lvec(3 2)
		\move(3 0)\lvec(4 2)
		\htext(0.7 0.5){$1$}
		\htext(1.7 0.5){$1$}
		
		\esegment
		\move(10 0.2)
		\bsegment
		\move(0 0)\lvec(4 0)
		\move(0 2)\lvec(4 2)
		\move(4 0)\lvec(4 2)
		\move(1 0)\lvec(1 2)
		\move(2 0)\lvec(2 2)
		\move(3 0)\lvec(3 2)
		\move(0 0)\lvec(1 2)
		\move(1 0)\lvec(2 2)
		\move(2 0)\lvec(3 2)
		\move(3 0)\lvec(4 2)
		\htext(0.4 1.5){$0$}
		\htext(1.4 1.5){$0$}
		
		\esegment
		\move(12 0.2)
		\bsegment
		\move(0 0)\lvec(4 0)
		\move(0 2)\lvec(4 2)
		\move(4 0)\lvec(4 2)
		\move(1 0)\lvec(1 2)
		\move(2 0)\lvec(2 2)
		\move(3 0)\lvec(3 2)
		\move(0 0)\lvec(1 2)
		\move(1 0)\lvec(2 2)
		\move(2 0)\lvec(3 2)
		\move(3 0)\lvec(4 2)
		\htext(2.4 1.5){$0$}
		\htext(3.4 1.5){$0$}
		\htext(0.7 0.5){$1$}
		\htext(1.7 0.5){$1$}
		\esegment
		\htext(5.5 1.25){\fontsize{12}{12}\selectfont{$=$}}
		\htext(7 1.25){\fontsize{12}{12}\selectfont{$\cdots$}}
		\htext(-5 1.25){\fontsize{12}{12}\selectfont{$\cdots$}}
		
		\htext(-7.5 1.25){\fontsize{12}{12}\selectfont{$Y_{2\Lambda_1}=$}}
	\end{texdraw}
\end{center}

\vskip 3mm

\begin{center}
	\begin{texdraw}
		\fontsize{3}{3}\selectfont
		\drawdim em
		\setunitscale 2.3
		\move(2 0)
		\bsegment	
		\move(0 0)\lvec(1 0)
		\move(1 0)\lvec(1 2)
		\move(0 0)\lvec(0 2)
		\move(0 2)\lvec(1 2)
		\move(0 2)\lvec(0.4 2.25)
		\move(1 2)\lvec(1.4 2.25)
		\move(0.4 2.25)\lvec(1.4 2.25)
		\move(1 0)\lvec(1.4 0.25)
		\move(1.4 0.25)\lvec(1.4 2.25)
		\move(1.4 0.25)\lvec(2.4 0.25)
		\move(2.4 0.25)\lvec(2.8 0.5)
		\move(2.4 0.25)\lvec(2.4 2.25)
		\move(2.8 0.5)\lvec(2.8 2.5)
		\move(2.4 2.25)\lvec(2.8 2.5)
		\move(1.4 2.25)\lvec(2.4 2.25)
		\move(1.4 2.25)\lvec(1.8 2.5)
		\move(1.8 2.5)\lvec(2.8 2.5)
		\htext(0.5 1){$0$}	
		\htext(1.9 1.25){$1$}
		\esegment
		
		\move(0 0)
		\bsegment
		\move(0 0)\lvec(1 0)
		\move(1 0)\lvec(1 2)
		\move(0 0)\lvec(0 2)
		\move(0 2)\lvec(1 2)
		\move(0 2)\lvec(0.4 2.25)
		\move(1 2)\lvec(1.4 2.25)
		\move(0.4 2.25)\lvec(1.4 2.25)
		\move(1 0)\lvec(1.4 0.25)
		\move(1.4 0.25)\lvec(1.4 2.25)
		\move(1.4 0.25)\lvec(2 0.25)

		\move(2.8 2.25)\lvec(2.8 2.5)
		\move(2.4 2.25)\lvec(2.8 2.5)
		\move(1.4 2.25)\lvec(2.4 2.25)
		\move(1.4 2.25)\lvec(1.8 2.5)
		\move(1.8 2.5)\lvec(2.8 2.5)
		\htext(0.5 1){$0$}	
		\htext(1.8 1.25){$1$}
		\esegment
		
		\move(-2 0)
		\bsegment
		\move(0 0)\lvec(1 0)
		\move(1 0)\lvec(1 2)
		\move(0 0)\lvec(0 2)
		\move(0 2)\lvec(1 2)
		\move(0 2)\lvec(0.4 2.25)
		\move(1 2)\lvec(1.4 2.25)
		\move(0.4 2.25)\lvec(1.4 2.25)
		\move(1 0)\lvec(1.4 0.25)
		\move(1.4 0.25)\lvec(1.4 2.25)
		\move(1.4 0.25)\lvec(2 0.25)

		\move(2.8 2.25)\lvec(2.8 2.5)
		\move(2.4 2.25)\lvec(2.8 2.5)
		\move(1.4 2.25)\lvec(2.4 2.25)
		\move(1.4 2.25)\lvec(1.8 2.5)
		\move(1.8 2.5)\lvec(2.8 2.5)
		\htext(0.5 1){$0$}	
		\htext(1.8 1.25){$1$}
		\esegment
		
		\move(-4 0)
		\bsegment
		\move(0 0)\lvec(1 0)
		\move(1 0)\lvec(1 2)
		\move(0 0)\lvec(0 2)
		\move(0 2)\lvec(1 2)
		\move(0 2)\lvec(0.4 2.25)
		\move(1 2)\lvec(1.4 2.25)
		\move(0.4 2.25)\lvec(1.4 2.25)
		\move(1 0)\lvec(1.4 0.25)
		\move(1.4 0.25)\lvec(1.4 2.25)
		\move(1.4 0.25)\lvec(2 0.25)

		\move(2.8 2.25)\lvec(2.8 2.5)
		\move(2.4 2.25)\lvec(2.8 2.5)
		\move(1.4 2.25)\lvec(2.4 2.25)
		\move(1.4 2.25)\lvec(1.8 2.5)
		\move(1.8 2.5)\lvec(2.8 2.5)
		\htext(0.5 1){$0$}	
		\htext(1.8 1.25){$1$}
		\esegment
		
		\move(7.8 0.2)
		\bsegment
		\move(0 0)\lvec(0 2)
		\move(0 0)\lvec(4 0)
		\move(0 2)\lvec(4 2)
		\move(4 0)\lvec(4 2)
		\move(1 0)\lvec(1 2)
		\move(2 0)\lvec(2 2)
		\move(3 0)\lvec(3 2)
		\move(0 0)\lvec(1 2)
		\move(1 0)\lvec(2 2)
		\move(2 0)\lvec(3 2)
		\move(3 0)\lvec(4 2)
		\htext(1.7 0.5){$1$}
		\htext(3.7 0.5){$1$}
		
		\esegment
		\move(9.8 0.2)
		\bsegment
		\move(0 0)\lvec(4 0)
		\move(0 2)\lvec(4 2)
		\move(4 0)\lvec(4 2)
		\move(1 0)\lvec(1 2)
		\move(2 0)\lvec(2 2)
		\move(3 0)\lvec(3 2)
		\move(0 0)\lvec(1 2)
		\move(1 0)\lvec(2 2)
		\move(2 0)\lvec(3 2)
		\move(3 0)\lvec(4 2)
		\htext(-1.6 1.5){$0$}
		\htext(0.4 1.5){$0$}
		
		\esegment
		\move(11.8 0.2)
		\bsegment
		\move(0 0)\lvec(4 0)
		\move(0 2)\lvec(4 2)
		\move(4 0)\lvec(4 2)
		\move(1 0)\lvec(1 2)
		\move(2 0)\lvec(2 2)
		\move(3 0)\lvec(3 2)
		\move(0 0)\lvec(1 2)
		\move(1 0)\lvec(2 2)
		\move(2 0)\lvec(3 2)
		\move(3 0)\lvec(4 2)
		\htext(2.4 1.5){$0$}
		\htext(0.4 1.5){$0$}
		\htext(1.7 0.5){$1$}
		\htext(3.7 0.5){$1$}
		\esegment
		\htext(5.5 1.25){\fontsize{12}{12}\selectfont{$=$}}
		\htext(7 1.25){\fontsize{12}{12}\selectfont{$\cdots$}}
		\htext(-4.8 1.25){\fontsize{12}{12}\selectfont{$\cdots$}}
		
		\htext(-7.8 1.25){\fontsize{12}{12}\selectfont{$Y_{\Lambda_0+\Lambda_1}=$}}	
		
	\end{texdraw}
	
\end{center}

\vskip 3mm

\begin{center}
	\begin{texdraw}
		\fontsize{3}{3}\selectfont
		\drawdim em
		\setunitscale 2.3
		\move(0 0)\lvec(1 0)\move(1 0)\lvec(2 0)
		\move(0 0)\lvec(0 2)\move(1 0)\lvec(1 2)
		\move(2 0)\lvec(2 2)\move(2 0)\lvec(2.4 0.25)	
		\move(2.4 0.25)\lvec(2.8 0.5)	\move(2.4 0.25)\lvec(2.4 2.25)\move(2.8 0.5)\lvec(2.8 2.5)
		\move(0 2)\lvec(2 2)\move(2 2)\lvec(2.8 2.5)
		\move(0.4  2.25)\lvec(1 2.25)\move(0.8  2.5)\lvec(1 2.5)\move(1  2)\lvec(1 8)\move(2  2)\lvec(2 8)	\move(2.8  2.5)\lvec(2.8 8.5)\move(1  4)\lvec(2 4)\move(2  4)\lvec(2.8 4.5)\move(1  6)\lvec(2 6)\move(2  6)\lvec(2.8 6.5)\move(1  8)\lvec(2 8)\move(2  8)\lvec(2.8 8.5)\move(1  8)\lvec(1.8 8.5)
		\move(1.8  8.5)\lvec(2.8 8.5)
		\htext(1.5 1){$0$}\htext(1.5 3){$2$}\htext(1.5 5){$3$}\htext(1.5 7){$2$}\htext(0.5 1){$0$}
		\htext(2.63 1.4){$1$}
		\move(2.4  4.25)\lvec(2.4  6.25)
		\htext(2.62 5.4){$3$}
		
		\move(-2 0)
		\bsegment
		\move(0 0)\lvec(1 0)\move(1 0)\lvec(2 0)
		\move(0 0)\lvec(0 2)\move(1 0)\lvec(1 2)
		\move(2 0)\lvec(2 2)	
		
		\move(0 2)\lvec(2 2)\move(2 2)\lvec(2.8 2.5)
		\move(0.4  2.25)\lvec(1 2.25)\move(0.8  2.5)\lvec(1 2.5)\move(1  2)\lvec(1 8)\move(2  2)\lvec(2 8)	\move(2.8  2.5)\lvec(2.8 8.5)\move(1  4)\lvec(2 4)\move(2  4)\lvec(2.8 4.5)\move(1  6)\lvec(2 6)\move(2  6)\lvec(2.8 6.5)\move(1  8)\lvec(2 8)\move(2  8)\lvec(2.8 8.5)\move(1  8)\lvec(1.8 8.5)
		\move(1.8  8.5)\lvec(2.8 8.5)
		\move(0 2)\lvec(0.8 2.5)
		\move(2.4  4.25)\lvec(2.4  6.25)
		\htext(2.62 5.4){$3$}
		\htext(1.5 1){$0$}\htext(1.5 3){$2$}\htext(1.5 5){$3$}\htext(1.5 7){$2$}\htext(0.5 1){$0$}
		\esegment
		
		\move(-4 0)
		\bsegment
		\move(0 0)\lvec(1 0)\move(1 0)\lvec(2 0)
		\move(0 0)\lvec(0 2)\move(1 0)\lvec(1 2)

		\move(0 2)\lvec(2 2)\move(2 2)\lvec(2.8 2.5)
		\move(0.4  2.25)\lvec(1 2.25)\move(0.8  2.5)\lvec(1 2.5)\move(1  2)\lvec(1 8)\move(2  2)\lvec(2 8)	\move(2.8  2.5)\lvec(2.8 8.5)\move(1  4)\lvec(2 4)\move(2  4)\lvec(2.8 4.5)\move(1  6)\lvec(2 6)\move(2  6)\lvec(2.8 6.5)\move(1  8)\lvec(2 8)\move(2  8)\lvec(2.8 8.5)\move(1  8)\lvec(1.8 8.5)
		\move(1.8  8.5)\lvec(2.8 8.5)
		\move(0 2)\lvec(0.8 2.5)
		\move(2.4  4.25)\lvec(2.4  6.25)
		\htext(2.62 5.4){$3$}
		\htext(1.5 1){$0$}\htext(1.5 3){$2$}\htext(1.5 5){$3$}\htext(1.5 7){$2$}\htext(0.5 1){$0$}
		\esegment
		
		\move(-6 0)
		\bsegment
		\move(0 0)\lvec(1 0)\move(1 0)\lvec(2 0)
		\move(0 0)\lvec(0 2)\move(1 0)\lvec(1 2)

		\move(0 2)\lvec(2 2)\move(2 2)\lvec(2.8 2.5)
		\move(0.4  2.25)\lvec(1 2.25)\move(0.8  2.5)\lvec(1 2.5)\move(1  2)\lvec(1 8)\move(2  2)\lvec(2 8)	\move(2.8  2.5)\lvec(2.8 8.5)\move(1  4)\lvec(2 4)\move(2  4)\lvec(2.8 4.5)\move(1  6)\lvec(2 6)\move(2  6)\lvec(2.8 6.5)\move(1  8)\lvec(2 8)\move(2  8)\lvec(2.8 8.5)\move(1  8)\lvec(1.8 8.5)
		\move(1.8  8.5)\lvec(2.8 8.5)
		\move(0 2)\lvec(0.8 2.5)
		\move(2.4  4.25)\lvec(2.4  6.25)
		\htext(2.62 5.4){$3$}
		\htext(1.5 1){$0$}\htext(1.5 3){$2$}\htext(1.5 5){$3$}\htext(1.5 7){$2$}\htext(0.5 1){$0$}
		\esegment
		
		\move(5.5 0)
		\bsegment
		\move(0 0)\lvec(1 0)\move(1 0)\lvec(2 0)
		\move(0 0)\lvec(0 2)\move(1 0)\lvec(1 2)
		\move(2 0)\lvec(2 2)\move(0 2)\lvec(2 2)
		\move(2 2)\lvec(2 8)\move(1 2)\lvec(1 8)
		\move(1 8)\lvec(2 8)\move(1 4)\lvec(2 4)
		\move(1 6)\lvec(2 6)\move(0 0)\lvec(1 2)
		\move(1 0)\lvec(2 2)
		\htext(0.7 0.5){$1$}\htext(1.7 0.5){$1$}
		\htext(0.35 1.5){$0$}\htext(1.35 1.5){$0$}
		\htext(1.5 3){$2$}
		\move(1 4)\lvec(2 6)
		\htext(1.7 4.5){$3$}
		\htext(1.35 5.5){$3$}
		\htext(1.5 7){$2$}
		\esegment
		\move(7.5 0)
		\bsegment
		\move(0 0)\lvec(1 0)\move(1 0)\lvec(2 0)
		\move(1 0)\lvec(1 2)
		\move(2 0)\lvec(2 2)\move(0 2)\lvec(2 2)
		\move(2 2)\lvec(2 8)\move(1 2)\lvec(1 8)
		\move(1 8)\lvec(2 8)\move(1 4)\lvec(2 4)
		\move(1 6)\lvec(2 6)\move(0 0)\lvec(1 2)
		\move(1 0)\lvec(2 2)
		\htext(0.7 0.5){$1$}\htext(1.7 0.5){$1$}
		\htext(0.35 1.5){$0$}\htext(1.35 1.5){$0$}
		\htext(1.5 3){$2$}\move(1 4)\lvec(2 6)
		\htext(1.7 4.5){$3$}
		\htext(1.35 5.5){$3$}\htext(1.5 7){$2$}
		\esegment
		\move(9.5 0)
		\bsegment
		\move(0 0)\lvec(1 0)\move(1 0)\lvec(2 0)
		\move(1 0)\lvec(1 2)
		\move(2 0)\lvec(2 2)\move(0 2)\lvec(2 2)
		\move(2 2)\lvec(2 8)\move(1 2)\lvec(1 8)
		\move(1 8)\lvec(2 8)\move(1 4)\lvec(2 4)
		\move(1 6)\lvec(2 6)\move(0 0)\lvec(1 2)
		\move(1 0)\lvec(2 2)
		\htext(0.7 0.5){$1$}\htext(1.7 0.5){$1$}
		\htext(0.35 1.5){$0$}\htext(1.35 1.5){$0$}
		\htext(1.5 3){$2$}\move(1 4)\lvec(2 6)
		\htext(1.7 4.5){$3$}
		\htext(1.35 5.5){$3$}\htext(1.5 7){$2$}
		\esegment
		\move(11.5 0)
		\bsegment
		\move(0 0)\lvec(1 0)\move(1 0)\lvec(2 0)
		\move(1 0)\lvec(1 2)
		\move(2 0)\lvec(2 2)\move(0 2)\lvec(2 2)
		\move(2 2)\lvec(2 8)\move(1 2)\lvec(1 8)
		\move(1 8)\lvec(2 8)\move(1 4)\lvec(2 4)
		\move(1 6)\lvec(2 6)\move(0 0)\lvec(1 2)
		\move(1 0)\lvec(2 2)
		\htext(0.7 0.5){$1$}\htext(1.7 0.5){$1$}
		\htext(0.35 1.5){$0$}\htext(1.35 1.5){$0$}
		\htext(1.5 3){$2$}\move(1 4)\lvec(2 6)
		\htext(1.7 4.5){$3$}
		\htext(1.35 5.5){$3$}\htext(1.5 7){$2$}
		\esegment
		\htext(3.5 4){\fontsize{12}{12}\selectfont{$=$}}
		\htext(5 4){\fontsize{12}{12}\selectfont{$\cdots$}}
		\htext(-6.6 4){\fontsize{12}{12}\selectfont{$\cdots$}}
		\htext(-9 4){\fontsize{12}{12}\selectfont{$Y_{\Lambda_2}=$}}
	\end{texdraw}
\end{center}

\vskip 3mm

\begin{center}
	\begin{texdraw}
		\fontsize{3}{3}\selectfont
		\drawdim em
		\setunitscale 2.3
		\move(0 0)\lvec(1 0)\move(1 0)\lvec(2 0)
		\move(0 0)\lvec(0 2)\move(1 0)\lvec(1 2)
		\move(2 0)\lvec(2 2)\move(2 0)\lvec(2.4 0.25)	
		\move(2.4 0.25)\lvec(2.8 0.5)
		\move(2.4 0.25)\lvec(2.4 2.25)
		\move(2.8 0.5)\lvec(2.8 2.5)
		\move(0 4)\lvec(1 4)
		\move(0.8 4.5)\lvec(1 4.5)
		\move(0 2)\lvec(2 2)
		\move(2 2)\lvec(2.8 2.5)
		\move(1  2)\lvec(1 6)
		\move(2  2)\lvec(2 6)	
		\move(2.8  2.5)\lvec(2.8 6.5)
		\move(1  4)\lvec(2 4)
		\move(2  4)\lvec(2.8 4.5)
		\move(1  6)\lvec(2 6)
		\move(2  6)\lvec(2.8 6.5)
		\move(1  6)\lvec(1.8 6.5)
		\move(1.8 6.5)\lvec(2.8 6.5)
		\move(2.4  4.25)\lvec(2.4  6.25)
		\move(1.4  6.25)\lvec(2.4  6.25)
		\htext(2.62 5.4){$3$}
		\htext(1.5 1){$0$}\htext(1.5 3){$2$}\htext(1.5 5){$3$}\htext(0.5 1){$0$}\htext(0.5 3){$2$}
		\htext(2.63 1.4){$1$}
		
		\move(-2 0)
		\bsegment
		\move(0 0)\lvec(1 0)\move(1 0)\lvec(2 0)
		\move(0 0)\lvec(0 2)\move(1 0)\lvec(1 2)
		\move(2 0)\lvec(2 2)	
		
		\move(0 2)\lvec(2 2)
		\move(1  2)\lvec(1 6)\move(2  2)\lvec(2 6)	\move(2.8  4.5)\lvec(2.8 6.5)\move(1  4)\lvec(2 4)\move(2  4)\lvec(2.8 4.5)\move(1  6)\lvec(2 6)\move(2  6)\lvec(2.8 6.5)
		\move(1 6)\lvec(1.8 6.5)
		\move(1.8 6.5)\lvec(2.8 6.5)
		\move(0 2)\lvec(0 4)\move(0 4)\lvec(1 4)
		\move(0 4)\lvec(0.8 4.5)\move(0.8 4.5)\lvec(1 4.5)
		\move(2.4  4.25)\lvec(2.4  6.25)
		\move(1.4  6.25)\lvec(2.4  6.25)
		\htext(2.62 5.4){$3$}
		\htext(1.5 1){$0$}\htext(1.5 3){$2$}\htext(1.5 5){$3$}\htext(0.5 1){$0$}\htext(0.5 3){$2$}
		\esegment
		
		\move(-4 0)
		\bsegment
		\move(0 0)\lvec(1 0)\move(1 0)\lvec(2 0)
		\move(0 0)\lvec(0 2)\move(1 0)\lvec(1 2)
		\move(2 0)\lvec(2 2)	
		
		\move(0 2)\lvec(2 2)
		\move(1  2)\lvec(1 6)\move(2  2)\lvec(2 6)	\move(2.8  4.5)\lvec(2.8 6.5)\move(1  4)\lvec(2 4)\move(2  4)\lvec(2.8 4.5)\move(1  6)\lvec(2 6)\move(2  6)\lvec(2.8 6.5)
		\move(1 6)\lvec(1.8 6.5)
		\move(1.8 6.5)\lvec(2.8 6.5)
		\move(0 2)\lvec(0 4)\move(0 4)\lvec(1 4)
		\move(0 4)\lvec(0.8 4.5)\move(0.8 4.5)\lvec(1 4.5)
		\move(2.4  4.25)\lvec(2.4  6.25)
		\move(1.4  6.25)\lvec(2.4  6.25)
		\htext(2.62 5.4){$3$}
		\htext(1.5 1){$0$}\htext(1.5 3){$2$}\htext(1.5 5){$3$}\htext(0.5 1){$0$}\htext(0.5 3){$2$}
		\esegment
		
		\move(-6 0)
		\bsegment
		\move(0 0)\lvec(1 0)\move(1 0)\lvec(2 0)
		\move(0 0)\lvec(0 2)\move(1 0)\lvec(1 2)
		\move(2 0)\lvec(2 2)	
		
		\move(0 2)\lvec(2 2)
		\move(1  2)\lvec(1 6)\move(2  2)\lvec(2 6)	\move(2.8  4.5)\lvec(2.8 6.5)\move(1  4)\lvec(2 4)\move(2  4)\lvec(2.8 4.5)\move(1  6)\lvec(2 6)\move(2  6)\lvec(2.8 6.5)
		\move(1 6)\lvec(1.8 6.5)
		\move(1.8 6.5)\lvec(2.8 6.5)
		\move(0 2)\lvec(0 4)\move(0 4)\lvec(1 4)
		\move(0 4)\lvec(0.8 4.5)\move(0.8 4.5)\lvec(1 4.5)
		\move(2.4  4.25)\lvec(2.4  6.25)
		\move(1.4  6.25)\lvec(2.4  6.25)
		\htext(2.62 5.4){$3$}
		\htext(1.5 1){$0$}\htext(1.5 3){$2$}\htext(1.5 5){$3$}\htext(0.5 1){$0$}\htext(0.5 3){$2$}
		\esegment
		
		\move(5.7 0)
		\bsegment
		\move(0 0)\lvec(1 0)\move(1 0)\lvec(2 0)
		\move(0 0)\lvec(0 2)\move(1 0)\lvec(1 2)
		\move(2 0)\lvec(2 2)\move(0 2)\lvec(2 2)
		\move(2 2)\lvec(2 6)\move(1 2)\lvec(1 6)
		\move(1 4)\lvec(2 4)
		\move(1 6)\lvec(2 6)\move(0 0)\lvec(1 2)
		\move(1 0)\lvec(2 2)
		\move(0 4)\lvec(1 4)\move(0 2)\lvec(0 4)
		\htext(0.7 0.5){$1$}\htext(1.7 0.5){$1$}
		\htext(0.35 1.5){$0$}\htext(1.35 1.5){$0$}\htext(0.5 3){$2$}
		\htext(1.5 3){$2$}\move(1 4)\lvec(2 6)
		\htext(1.7 4.5){$3$}
		\htext(1.35 5.5){$3$}
		\esegment
		\move(7.7 0)
		\bsegment
		\move(0 0)\lvec(1 0)\move(1 0)\lvec(2 0)
		\move(1 0)\lvec(1 2)
		\move(2 0)\lvec(2 2)\move(0 2)\lvec(2 2)
		\move(2 2)\lvec(2 6)\move(1 2)\lvec(1 6)
		\move(1 4)\lvec(2 4)
		\move(1 6)\lvec(2 6)\move(0 0)\lvec(1 2)
		\move(1 0)\lvec(2 2)\move(0 4)\lvec(1 4)
		\htext(0.7 0.5){$1$}\htext(1.7 0.5){$1$}
		\htext(0.35 1.5){$0$}\htext(1.35 1.5){$0$}\htext(0.5 3){$2$}
		\htext(1.5 3){$2$}\move(1 4)\lvec(2 6)
		\htext(1.7 4.5){$3$}
		\htext(1.35 5.5){$3$}
		\esegment
		\move(9.7 0)
		\bsegment
		\move(0 0)\lvec(1 0)\move(1 0)\lvec(2 0)
		\move(1 0)\lvec(1 2)
		\move(2 0)\lvec(2 2)\move(0 2)\lvec(2 2)
		\move(2 2)\lvec(2 6)\move(1 2)\lvec(1 6)
		\move(1 4)\lvec(2 4)
		\move(1 6)\lvec(2 6)\move(0 0)\lvec(1 2)
		\move(1 0)\lvec(2 2)\move(0 4)\lvec(1 4)
		\htext(0.7 0.5){$1$}\htext(1.7 0.5){$1$}
		\htext(0.35 1.5){$0$}\htext(1.35 1.5){$0$}\htext(0.5 3){$2$}
		\htext(1.5 3){$2$}\move(1 4)\lvec(2 6)
		\htext(1.7 4.5){$3$}
		\htext(1.35 5.5){$3$}
		\esegment
		\move(11.7 0)
		\bsegment
		\move(0 0)\lvec(1 0)\move(1 0)\lvec(2 0)
		\move(1 0)\lvec(1 2)
		\move(2 0)\lvec(2 2)\move(0 2)\lvec(2 2)
		\move(2 2)\lvec(2 6)\move(1 2)\lvec(1 6)
		\move(1 4)\lvec(2 4)
		\move(1 6)\lvec(2 6)\move(0 0)\lvec(1 2)
		\move(1 0)\lvec(2 2)\move(0 4)\lvec(1 4)
		\htext(0.7 0.5){$1$}\htext(1.7 0.5){$1$}
		\htext(0.35 1.5){$0$}\htext(1.35 1.5){$0$}\htext(0.5 3){$2$}
		\htext(1.5 3){$2$}\move(1 4)\lvec(2 6)
		\htext(1.7 4.5){$3$}
		\htext(1.35 5.5){$3$}
		\esegment
		\htext(3.5 4){\fontsize{12}{12}\selectfont{$=$}}
		\htext(5 4){\fontsize{12}{12}\selectfont{$\cdots$}}
		\htext(-6.6 4){\fontsize{12}{12}\selectfont{$\cdots$}}
		\htext(-9 4){\fontsize{12}{12}\selectfont{$Y_{2\Lambda_3}=$}}
	\end{texdraw}
\end{center}

\vskip 3mm

\begin{center}
	\begin{texdraw}
		\fontsize{3}{3}\selectfont
		\drawdim em
		\setunitscale 2.3
		\move(-0.9 0)
		\bsegment
		\move(0 0)\lvec(0 2)
		\move(0 0)\lvec(1 0)
		\move(1 0)\lvec(2 0)
		\move(1 0)\lvec(1 2)
		\move(2 0)\lvec(2 2)
		\move(0 2)\lvec(2 2)
		\move(0 2)\lvec(0.4 2.25)
		\move(0.4 2.25)\lvec(1 2.25)
		\move(1 2)\lvec(1 4)
		\move(2 2)\lvec(2 4)
		\move(1 4)\lvec(2 4)
		\move(1 4)\lvec(1.4 4.25)
		\move(2 4)\lvec(2.4 4.25)
		\move(1.4 4.25)\lvec(2.4 4.25)
		\move(2.4 4.25)\lvec(2.8 4.5)
		\move(2.4 4.25)\lvec(2.4 6.25)
		\move(1.4 4.25)\lvec(1.4 6.25)
		\move(1.4 6.25)\lvec(2.4 6.25)
		\move(2.4 6.25)\lvec(2.8 6.5)
		\move(1.4 6.25)\lvec(1.8 6.5)
		\move(1.8 6.5)\lvec(2.8 6.5)
		\move(2.8 2.5)\lvec(2.8 6.5)
		\move(2 2)\lvec(2.8 2.5)
		\move(2 0)\lvec(2.4 0.25)
		\move(2.4 0.25)\lvec(3 0.25)
		\move(2.4 2.25)\lvec(3 2.25)
		\move(2.8 2.5)\lvec(3 2.5)
		\move(2.4 0.25)\lvec(2.4 2.25)
		\move(3 0)\lvec(3 4)
		\move(3 0)\lvec(4 0)
		\move(4 0)\lvec(4 4)
		\move(3 2)\lvec(4 2)
		\move(3 4)\lvec(4 4)
		\move(4 0)\lvec(4.4 0.25)
		\move(4.4 0.25)\lvec(4.8 0.5)
		\move(4.4 0.25)\lvec(4.4 2.25)
		\move(4.8 0.5)\lvec(4.8 2.5)
		\move(4.8 2.5)\lvec(4.8 6.5)
		\move(4 2)\lvec(4.8 2.5)
		\move(4 4)\lvec(4.8 4.5)
		\move(3 4)\lvec(3.4 4.25)
		\move(3.4 4.25)\lvec(4.4 4.25)
		\move(4.4 4.25)\lvec(4.4 6.25)
		\move(3.4 4.25)\lvec(3.4 6.25)		
		\move(3.4 6.25)\lvec(4.4 6.25)	
		\move(4.4 6.25)\lvec(4.8 6.5)
		\move(3.4 6.25)\lvec(3.8 6.5)
		\move(3.8 6.5)\lvec(4.8 6.5)
		\htext(0.5 1){$0$}
		\htext(1.5 1){$0$}
		\htext(3.5 1){$0$}
		\htext(1.5 3){$2$}
		\htext(3.5 3){$2$}
		\htext(1.9 5.25){$3$}
		\htext(3.9 5.25){$3$}
		\htext(2.8 1.25){$1$}
		\htext(4.65 1.375){$1$}
		\esegment
		
		\move(-4.9 0)
		\bsegment
		\move(0 0)\lvec(0 2)
		\move(0 0)\lvec(1 0)
		\move(1 0)\lvec(2 0)
		\move(1 0)\lvec(1 2)
		\move(2 0)\lvec(2 2)
		\move(0 2)\lvec(2 2)
		\move(0 2)\lvec(0.4 2.25)
		\move(0.4 2.25)\lvec(1 2.25)
		\move(1 2)\lvec(1 4)
		\move(2 2)\lvec(2 4)
		\move(1 4)\lvec(2 4)
		\move(1 4)\lvec(1.4 4.25)
		\move(2 4)\lvec(2.4 4.25)
		\move(1.4 4.25)\lvec(2.4 4.25)
		\move(2.4 4.25)\lvec(2.8 4.5)
		\move(2.4 4.25)\lvec(2.4 6.25)
		\move(1.4 4.25)\lvec(1.4 6.25)
		\move(1.4 6.25)\lvec(2.4 6.25)
		\move(2.4 6.25)\lvec(2.8 6.5)
		\move(1.4 6.25)\lvec(1.8 6.5)
		\move(1.8 6.5)\lvec(2.8 6.5)
		\move(2.8 2.5)\lvec(2.8 6.5)
		\move(2 2)\lvec(2.8 2.5)
		\move(2 0)\lvec(2.4 0.25)
		\move(2.4 0.25)\lvec(3 0.25)
		\move(2.4 2.25)\lvec(3 2.25)
		\move(2.8 2.5)\lvec(3 2.5)
		\move(2.4 0.25)\lvec(2.4 2.25)
		\move(3 0)\lvec(3 4)
		\move(3 0)\lvec(4 0)
		\move(4 0)\lvec(4 4)
		\move(3 2)\lvec(4 2)
		\move(3 4)\lvec(4 4)
		\move(4.8 2.25)\lvec(4.8 2.5)
		\move(4.8 2.5)\lvec(4.8 6.5)
		\move(4 2)\lvec(4.8 2.5)
		\move(4 4)\lvec(4.8 4.5)
		\move(3 4)\lvec(3.4 4.25)
		\move(3.4 4.25)\lvec(4.4 4.25)
		\move(4.4 4.25)\lvec(4.4 6.25)
		\move(3.4 4.25)\lvec(3.4 6.25)		
		\move(3.4 6.25)\lvec(4.4 6.25)	
		\move(4.4 6.25)\lvec(4.8 6.5)
		\move(3.4 6.25)\lvec(3.8 6.5)
		\move(3.8 6.5)\lvec(4.8 6.5)
		\htext(0.5 1){$0$}
		\htext(1.5 1){$0$}
		\htext(3.5 1){$0$}
		\htext(1.5 3){$2$}
		\htext(3.5 3){$2$}
		\htext(1.9 5.25){$3$}
		\htext(3.9 5.25){$3$}
		\htext(2.8 1.25){$1$}
		\esegment

		\move(6.8 0)
		\bsegment
		\move(0 0)\lvec(1 0)\move(1 0)\lvec(2 0)
		\move(0 0)\lvec(0 2)\move(1 0)\lvec(1 2)
		\move(2 0)\lvec(2 2)\move(0 2)\lvec(2 2)
		\move(2 2)\lvec(2 6)\move(1 2)\lvec(1 6)
		\move(1 4)\lvec(2 4)
		\move(1 6)\lvec(2 6)\move(0 0)\lvec(1 2)
		\move(1 0)\lvec(2 2)
		
		\htext(1.7 0.5){$1$}
		\htext(0.35 1.5){$0$}\htext(1.35 1.5){$0$}
		\htext(1.5 3){$2$}\move(1 4)\lvec(2 6)
		\htext(1.7 4.5){$3$}
		
		\esegment
		
		\move(8.8 0)
		\bsegment
		\move(0 0)\lvec(1 0)\move(1 0)\lvec(2 0)
		\move(0 0)\lvec(0 2)\move(1 0)\lvec(1 2)
		\move(2 0)\lvec(2 2)\move(0 2)\lvec(2 2)
		\move(2 2)\lvec(2 6)\move(1 2)\lvec(1 6)
		\move(1 4)\lvec(2 4)
		\move(1 6)\lvec(2 6)\move(0 0)\lvec(1 2)
		\move(1 0)\lvec(2 2)
		
		\htext(0.7 0.5){$1$}\htext(1.7 0.5){$1$}
		\htext(1.35 1.5){$0$}
		\htext(1.5 3){$2$}\move(1 4)\lvec(2 6)
		\htext(1.7 4.5){$3$}
		
		\esegment
		
		\move(10.8 0)
		\bsegment
		\move(0 0)\lvec(1 0)\move(1 0)\lvec(2 0)
		\move(0 0)\lvec(0 2)\move(1 0)\lvec(1 2)
		\move(2 0)\lvec(2 2)\move(0 2)\lvec(2 2)
		\move(2 2)\lvec(2 6)\move(1 2)\lvec(1 6)
		\move(1 4)\lvec(2 4)
		\move(1 6)\lvec(2 6)\move(0 0)\lvec(1 2)
		\move(1 0)\lvec(2 2)
		
		\htext(1.7 0.5){$1$}
		\htext(0.35 1.5){$0$}\htext(1.35 1.5){$0$}
		\htext(1.5 3){$2$}\move(1 4)\lvec(2 6)
		\htext(1.7 4.5){$3$}
		
		\esegment
		\move(12.8 0)
		\bsegment
		\move(0 0)\lvec(1 0)\move(1 0)\lvec(2 0)
		\move(0 0)\lvec(0 2)\move(1 0)\lvec(1 2)
		\move(2 0)\lvec(2 2)\move(0 2)\lvec(2 2)
		\move(2 2)\lvec(2 6)\move(1 2)\lvec(1 6)
		\move(1 4)\lvec(2 4)
		\move(1 6)\lvec(2 6)\move(0 0)\lvec(1 2)
		\move(1 0)\lvec(2 2)
		
		\htext(0.7 0.5){$1$}\htext(1.7 0.5){$1$}
		\htext(1.35 1.5){$0$}
		\htext(1.5 3){$2$}\move(1 4)\lvec(2 6)
		\htext(1.7 4.5){$3$}
		
		\esegment
		\htext(5.5 4){\fontsize{12}{12}\selectfont{$=\cdots$}}

		\htext(-7.9 4){\fontsize{12}{12}\selectfont{$Y_{\Lambda_0+\Lambda_3}=\cdots$}}	
		
	\end{texdraw}
\end{center}

\vskip 3mm

\begin{center}
	\begin{texdraw}
		\fontsize{3}{3}\selectfont
		\drawdim em
		\setunitscale 2.3
		\move(-0.9 0)
		\bsegment
		\move(2 0)\lvec(2 2)
		\move(2 0)\lvec(2 0)
		\move(1 0)\lvec(2 0)
		\move(1 0)\lvec(1 2)
		\move(2 0)\lvec(2 2)
		\move(1 2)\lvec(2 2)
		\move(0 2)\lvec(0.4 2.25)
		\move(2.4 2.25)\lvec(3 2.25)
		\move(1 2)\lvec(1 4)
		\move(2 2)\lvec(2 4)
		\move(1 4)\lvec(2 4)
		\move(1 4)\lvec(1.4 4.25)
		\move(2 4)\lvec(2.4 4.25)
		\move(1.4 4.25)\lvec(2.4 4.25)
		\move(2.4 4.25)\lvec(2.8 4.5)
		\move(2.4 4.25)\lvec(2.4 6.25)
		\move(1.4 4.25)\lvec(1.4 6.25)
		\move(1.4 6.25)\lvec(2.4 6.25)
		\move(2.4 6.25)\lvec(2.8 6.5)
		\move(1.4 6.25)\lvec(1.8 6.5)
		\move(1.8 6.5)\lvec(2.8 6.5)
		\move(2 0)\lvec(3 0)
		\move(2 2)\lvec(3 2)
		\move(2.8 2.5)\lvec(2.8 6.5)
		\move(2 2)\lvec(2.8 2.5)
		\move(0 0)\lvec(0.4 0.25)
		\move(0.4 0.25)\lvec(1 0.25)
		\move(0.4 2.25)\lvec(1 2.25)
		\move(0.8 2.5)\lvec(1 2.5)
		\move(0.4 0.25)\lvec(0.4 2.25)
		\move(3 0)\lvec(3 4)
		\move(3 0)\lvec(4 0)
		\move(4 0)\lvec(4 4)
		\move(3 2)\lvec(4 2)
		\move(3 4)\lvec(4 4)
		\move(4 0)\lvec(4.4 0.25)
		\move(4.4 0.25)\lvec(4.8 0.5)
		\move(4.4 0.25)\lvec(4.4 2.25)
		\move(4.8 0.5)\lvec(4.8 2.5)
		\move(4.8 2.5)\lvec(4.8 6.5)
		\move(4 2)\lvec(4.8 2.5)
		\move(4 4)\lvec(4.8 4.5)
		\move(3 4)\lvec(3.4 4.25)
		\move(3.4 4.25)\lvec(4.4 4.25)
		\move(4.4 4.25)\lvec(4.4 6.25)
		\move(3.4 4.25)\lvec(3.4 6.25)		
		\move(3.4 6.25)\lvec(4.4 6.25)	
		\move(4.4 6.25)\lvec(4.8 6.5)
		\move(3.4 6.25)\lvec(3.8 6.5)
		\move(3.8 6.5)\lvec(4.8 6.5)
		\move(2.8 2.25)\lvec(2.8 2.5)
		\htext(2.5 1){$0$}
		\htext(1.5 1){$0$}
		\htext(3.5 1){$0$}
		\htext(1.5 3){$2$}
		\htext(3.5 3){$2$}
		\htext(1.9 5.25){$3$}
		\htext(3.9 5.25){$3$}
		\htext(0.8 1.25){$1$}
		\htext(4.65 1.375){$1$}
		\esegment
		
		\move(-4.9 0)
		\bsegment
		\move(2 0)\lvec(2 2)
		\move(2 0)\lvec(2 0)
		\move(1 0)\lvec(2 0)
		\move(1 0)\lvec(1 2)
		\move(2 0)\lvec(2 2)
		\move(1 2)\lvec(2 2)
		
		\move(2.4 2.25)\lvec(3 2.25)
		\move(1 2)\lvec(1 4)
		\move(2 2)\lvec(2 4)
		\move(1 4)\lvec(2 4)
		\move(1 4)\lvec(1.4 4.25)
		\move(2 4)\lvec(2.4 4.25)
		\move(1.4 4.25)\lvec(2.4 4.25)
		\move(2.4 4.25)\lvec(2.8 4.5)
		\move(2.4 4.25)\lvec(2.4 6.25)
		\move(1.4 4.25)\lvec(1.4 6.25)
		\move(1.4 6.25)\lvec(2.4 6.25)
		\move(2.4 6.25)\lvec(2.8 6.5)
		\move(1.4 6.25)\lvec(1.8 6.5)
		\move(1.8 6.5)\lvec(2.8 6.5)
		\move(2 0)\lvec(3 0)
		\move(2 2)\lvec(3 2)
		\move(2.8 2.5)\lvec(2.8 6.5)
		\move(2 2)\lvec(2.8 2.5)
		\move(0.4 2.25)\lvec(0.8 2.5)
		\move(0.4 0.25)\lvec(1 0.25)
		\move(0.4 2.25)\lvec(1 2.25)
		\move(0.8 2.5)\lvec(1 2.5)
		\move(0.4 0.25)\lvec(0.4 2.25)
		\move(3 0)\lvec(3 4)
		\move(3 0)\lvec(4 0)
		\move(4 0)\lvec(4 4)
		\move(3 2)\lvec(4 2)
		\move(3 4)\lvec(4 4)
		\move(4 0)\lvec(4.4 0.25)
		
		\move(2.8 2.25)\lvec(2.8 2.5)
		
		\move(4.8 2.5)\lvec(4.8 6.5)
		\move(4 2)\lvec(4.8 2.5)
		\move(4 4)\lvec(4.8 4.5)
		\move(3 4)\lvec(3.4 4.25)
		\move(3.4 4.25)\lvec(4.4 4.25)
		\move(4.4 4.25)\lvec(4.4 6.25)
		\move(3.4 4.25)\lvec(3.4 6.25)		
		\move(3.4 6.25)\lvec(4.4 6.25)	
		\move(4.4 6.25)\lvec(4.8 6.5)
		\move(3.4 6.25)\lvec(3.8 6.5)
		\move(3.8 6.5)\lvec(4.8 6.5)
		\htext(2.5 1){$0$}
		\htext(1.5 1){$0$}
		\htext(3.5 1){$0$}
		\htext(1.5 3){$2$}
		\htext(3.5 3){$2$}
		\htext(1.9 5.25){$3$}
		\htext(3.9 5.25){$3$}
		\htext(0.8 1.25){$1$}
		\esegment

		\move(6.8 0)
		\bsegment
		\move(0 0)\lvec(1 0)\move(1 0)\lvec(2 0)
		\move(0 0)\lvec(0 2)\move(1 0)\lvec(1 2)
		\move(2 0)\lvec(2 2)\move(0 2)\lvec(2 2)
		\move(2 2)\lvec(2 6)\move(1 2)\lvec(1 6)
		\move(1 4)\lvec(2 4)
		\move(1 6)\lvec(2 6)\move(0 0)\lvec(1 2)
		\move(1 0)\lvec(2 2)
		
		\htext(1.7 0.5){$1$}
		\htext(0.7 0.5){$1$}
		\htext(1.35 1.5){$0$}
		\htext(1.5 3){$2$}
		\move(1 4)\lvec(2 6)
		\htext(1.7 4.5){$3$}
		
		\esegment
		
		\move(8.8 0)
		\bsegment
		\move(0 0)\lvec(1 0)\move(1 0)\lvec(2 0)
		\move(0 0)\lvec(0 2)\move(1 0)\lvec(1 2)
		\move(2 0)\lvec(2 2)\move(0 2)\lvec(2 2)
		\move(2 2)\lvec(2 6)\move(1 2)\lvec(1 6)
		\move(1 4)\lvec(2 4)
		\move(1 6)\lvec(2 6)\move(0 0)\lvec(1 2)
		\move(1 0)\lvec(2 2)
		
		\htext(1.7 0.5){$1$}
		\htext(1.35 1.5){$0$}
		\htext(0.35 1.5){$0$}
		\htext(1.5 3){$2$}\move(1 4)\lvec(2 6)
		\htext(1.7 4.5){$3$}
		
		\esegment
		
		\move(10.8 0)
		\bsegment
		\move(0 0)\lvec(1 0)\move(1 0)\lvec(2 0)
		\move(0 0)\lvec(0 2)\move(1 0)\lvec(1 2)
		\move(2 0)\lvec(2 2)\move(0 2)\lvec(2 2)
		\move(2 2)\lvec(2 6)\move(1 2)\lvec(1 6)
		\move(1 4)\lvec(2 4)
		\move(1 6)\lvec(2 6)\move(0 0)\lvec(1 2)
		\move(1 0)\lvec(2 2)
		
		\htext(1.7 0.5){$1$}
		\htext(0.7 0.5){$1$}
		\htext(1.35 1.5){$0$}
		\htext(1.5 3){$2$}\move(1 4)\lvec(2 6)
		\htext(1.7 4.5){$3$}
		
		\esegment
		\move(12.8 0)
		\bsegment
		\move(0 0)\lvec(1 0)\move(1 0)\lvec(2 0)
		\move(0 0)\lvec(0 2)\move(1 0)\lvec(1 2)
		\move(2 0)\lvec(2 2)\move(0 2)\lvec(2 2)
		\move(2 2)\lvec(2 6)\move(1 2)\lvec(1 6)
		\move(1 4)\lvec(2 4)
		\move(1 6)\lvec(2 6)\move(0 0)\lvec(1 2)
		\move(1 0)\lvec(2 2)
		
		\htext(1.7 0.5){$1$}
		\htext(1.35 1.5){$0$}
		\htext(0.35 1.5){$0$}
		\htext(1.5 3){$2$}\move(1 4)\lvec(2 6)
		\htext(1.7 4.5){$3$}
		
		\esegment
		\htext(5.5 4){\fontsize{12}{12}\selectfont{$=\cdots$}}

		\htext(-7.9 4){\fontsize{12}{12}\selectfont{$Y_{\Lambda_1+\Lambda_3}=\cdots$}}	
		
	\end{texdraw}
\end{center}

\end{example}

\vskip 5mm 

\subsection{New Young wall realization of the crystal  $B(\lambda)$}\label{Construction of Young Wall}

\hfill

\vskip 2mm 

A \textit{level-$l$ Young wall} consists of infinitely many level-$l$ Young columns which are extended to the left. We denote a level-$l$ Young wall by $Y=(\cdots,C_2,C_1,C_0)$, where $C_i\ (i\ge 0)$ are  level-$l$ Young columns.

\vskip 2mm

The rules for building Young walls are given as follows:

\vskip 2mm

\begin{enumerate}
	\item The Young walls should be built on top of one of the ground-state walls.
	\vspace{6pt}
	\item The colored blocks should be stacked by the stacking patterns in Definition \ref{stacking patterns}.
	\vspace{6pt}
	\item  No block with unit depth can be placed on top of a block with half-unit depth.
	\vspace{6pt}
	\item There is no $i$-block added on the $k$-th Young column if $\phi_i(C_k)\le\epsilon_i(C_{k-1})$ ($k\ge 1$).
	\vspace{6pt}
	\item For two adjacent Young columns $C_k$ and $C_{k-1}$, the height of the $i$-th slice in $C_{k-1}$ is not less than  the height of the $(i+l)$-th slice in $C_k$.
\end{enumerate}

\vskip 2mm 

We say that a slice in a Young wall $Y$ has a {\it removable $\delta$-slice} if we may remove a $\delta$-slice and still get a Young wall. A Young wall $Y$ is called {\it reduced} if there is no removable $\delta$-slice in $Y$. 
We always assume that the ground-state walls are reduced. 

\vskip 2mm

\begin{example}
We give an example of level-$2$ reduced Young wall built on the ground-state wall $Y_{2\Lambda_0}$ of type $B_3^{(1)}$.

\vskip 4mm

\begin{center}
	\begin{texdraw}
		\fontsize{3}{3}\selectfont
		\drawdim em
		\setunitscale 2.3
		\move(-6 0)
		\bsegment
		\move(0 0)\lvec(1 0)	
		\move(1 0)\lvec(2 0)
		\move(1 0)\lvec(1 2)
		\move(0 0)\lvec(0 2)
		\move(2 0)\lvec(2 2)
		\move(0 2)\lvec(1 2)
		\move(1 2)\lvec(2 2)
		\move(0 2)\lvec(0.4 2.25)
		\move(1 2)\lvec(1.4 2.25)
		\move(2 2)\lvec(2.4 2.25)
		\move(0.4 2.25)\lvec(2.4 2.25)
		\move(2 0)\lvec(2.4 0.25)
		\move(2.4 0.25)\lvec(2.4 2.25)
		\move(2.4 0.25)\lvec(3 0.25)
		\move(2.4 2.25)\lvec(3.4 2.25)
		\move(2.8 2.5)\lvec(3.8 2.5)
		\move(2.4 2.25)\lvec(2.8 2.5)
		\move(3.4 2.25)\lvec(3.8 2.5)
		\move(3 0)\lvec(3 2)
		\move(3 2)\lvec(3.4 2.25)
		\move(3 0)\lvec(4 0)
		\move(4 0)\lvec(4 2)
		\move(3 2)\lvec(4 2)
		\move(3.4 2.25)\lvec(4.4 2.25)
		\move(3.8 2.5)\lvec(4.8 2.5)
		\move(4 2)\lvec(4.8 2.5)
		\move(4 2)\lvec(4.8 2.5)
		\move(4.8 2.25)\lvec(4.8 2.5)
		\move(4.4 2.25)\lvec(5 2.25)
		\move(4 2)\lvec(5 2)
		\move(4 0)\lvec(5 0)
		\move(5 0)\lvec(5 4)
		\move(6 0)\lvec(6 4)
		\move(5 0)\lvec(6 0)
		\move(5 2)\lvec(6 2)
		\move(5 4)\lvec(6 4)
		\move(5 4)\lvec(5.8 4.5)
		\move(5.8 4.5)\lvec(6.8 4.5)
		\move(6 4)\lvec(6.8 4.5)
		\move(6 2)\lvec(6.8 2.5)
		\move(6.8 2.5)\lvec(6.8 4.5)
		\move(6 0)\lvec(6.4 0.25)
		\move(6.4 0.25)\lvec(6.4 2.25)
		\move(6.4 0.25)\lvec(7 0.25)
		\move(6.4 2.25)\lvec(7 2.25)
		\move(6.8 2.5)\lvec(7 2.5)
		\move(7 0)\lvec(7 6)
		\move(7 0)\lvec(8 0)
		\move(8 0)\lvec(8 6)
		\move(8 0)\lvec(8.8 0.5)
		\move(8.8 0.5)\lvec(8.8 6.5)
		\move(7 2)\lvec(8 2)
		\move(7 4)\lvec(8 4)
		\move(7 6)\lvec(8 6)
		\move(8 6)\lvec(8.8 6.5)
		\move(7 6)\lvec(7.8 6.5)
		\move(7.8 6.5)\lvec(8.8 6.5)
		\move(8 4)\lvec(8.8 4.5)
		\move(8 2)\lvec(8.8 2.5)
		\move(8.4 0.25)\lvec(8.4 2.25)
		\move(8.4 4.25)\lvec(8.4 6.25)
		\move(7.4 6.25)\lvec(8.4 6.25)
		\htext(0.5 1){$0$}
		\htext(1.5 1){$0$}
		\htext(3.5 1){$0$}
		\htext(4.5 1){$0$}
		\htext(5.5 1){$0$}
		\htext(7.5 1){$0$}
		\htext(2.7 1.375){$1$}
		\htext(6.7 1.375){$1$}
		\htext(8.6 1.375){$1$}
		\htext(7.5 3){$2$}
		\htext(7.5 5){$3$}
		\htext(8.57 5.375){$3$}
		\htext(5.5 3){$2$}
		\esegment

		\move(5.7 0)
		\bsegment
		\move(0 0)\lvec(1 0)\move(1 0)\lvec(2 0)
		\move(0 0)\lvec(0 2)\move(1 0)\lvec(1 2)
		\move(2 0)\lvec(2 2)\move(0 2)\lvec(2 2)

		\move(0 0)\lvec(1 2)
		\move(1 0)\lvec(2 2)

		\htext(0.35 1.5){$0$}\htext(1.35 1.5){$0$}
		
		\esegment
		
		\move(7.7 0)
		\bsegment
		\move(0 0)\lvec(1 0)\move(1 0)\lvec(2 0)
		\move(1 0)\lvec(1 2)
		\move(2 0)\lvec(2 2)\move(0 2)\lvec(2 2)

		\move(0 0)\lvec(1 2)
		\move(1 0)\lvec(2 2)
		\htext(0.7 0.5){$1$}\htext(1.7 0.5){$1$}
		\htext(1.35 1.5){$0$}
		\esegment

		\move(9.7 0)
		\bsegment
		\move(0 0)\lvec(1 0)\move(1 0)\lvec(2 0)
		\move(1 0)\lvec(1 2)
		\move(2 0)\lvec(2 2)\move(0 2)\lvec(2 2)
		\move(2 2)\lvec(2 4)\move(1 2)\lvec(1 4)
		\move(1 4)\lvec(2 4)
		\move(0 0)\lvec(1 2)
		\move(1 0)\lvec(2 2)
		\htext(1.7 0.5){$1$}
		\htext(0.35 1.5){$0$}\htext(1.35 1.5){$0$}
		\htext(1.5 3){$2$}
		
		\esegment
		\move(11.7 0)
		\bsegment
		\move(0 0)\lvec(1 0)\move(1 0)\lvec(2 0)
		\move(1 0)\lvec(1 2)
		\move(2 0)\lvec(2 2)\move(0 2)\lvec(2 2)
		\move(2 2)\lvec(2 6)\move(1 2)\lvec(1 6)
		\move(1 4)\lvec(2 4)
		\move(1 6)\lvec(2 6)\move(0 0)\lvec(1 2)
		\move(1 0)\lvec(2 2)
		\htext(0.7 0.5){$1$}\htext(1.7 0.5){$1$}
		\htext(1.35 1.5){$0$}
		\htext(1.5 3){$2$}\move(1 4)\lvec(2 6)
		\htext(1.7 4.5){$3$}
		\htext(1.35 5.5){$3$}
		\esegment
		\htext(4.1 3){\fontsize{12}{12}\selectfont{$=$}}
	\end{texdraw}
\end{center}

\end{example}

\vskip 4mm 

\begin{remark}
The removable $\delta$-slices in Young walls are different from the removable $\delta$-slices in Young columns.  For types $A_{2n}^{(2)}$, $D_{n+1}^{(2)}$ and $C_n^{(1)}$, we make a $\delta$-slice 
only when we add a $0$-block. 
For types $A_{2n-1}^{(2)}$, $D_{n}^{(1)}$ and $B_n^{(1)}$, we can make a removable $\delta$-slice only when we add both $0$-block and $1$-block together. 
\end{remark}

\vskip 2mm 

We define the {\it $i$-signature} of a Young column $C$ to be the pair of integers 
 ${\rm Sign}_{i}(C)=(\epsilon_i(C),\phi_i(C))$. To express the $i$-signature, we will also use a sequence of $\epsilon_i(C)$ many $-$'s followed by $\phi_i(C)$ many $+$'s. 

\vskip 2mm 

Let $Y=(\cdots,C_2,C_1,C_0)$ be a Young Wall.  Cancel out all  $(+,-)$-pairs in this sequence $(\cdots,{\rm Sign}_i(C_2),{\rm Sign}_i(C_1),{\rm Sign}_i(C_0))$. The remaining finite sequence of $-$'s followed by 
$+$'s is called the {\it $i$-signature} of $Y$.  We denote the $i$-signature of $Y$ by ${\rm Sign}_i(Y)$.

\vskip 2mm

We define $\tilde{E}_i Y$ $(i\in I)$ to be the Young wall obtained by removing an $i$-block  from the Young column of $Y$ that corresponds to the rightmost $-$ in ${\rm Sign}_i(Y)$. We define $\tilde{E}_i Y = 0$ 
if there is no $-$ in ${\rm Sign}_i(Y)$.

\vskip 2mm

We define $\tilde{F}_i Y$ $(i\in I)$ to be the Young wall obtained by adding an $i$-block to the Young column of $Y$ that corresponds to the leftmost $+$ in ${\rm Sign}_i(Y)$.
We define $\tilde{F}_i Y = 0$  if there is no $+$ in ${\rm Sign}_i(Y)$.

\vskip 3mm

Let $\mathbf{Y}(\lambda)$ be the set of all level-$l$ reduced Young walls. We define the maps

\begin{equation*}
	\wt : \mathbf{Y}(\lambda) \longrightarrow P,\quad 
	\varepsilon_i : \mathbf{Y}(\lambda) \longrightarrow \mathbf Z, \quad 
	\varphi_i : \mathbf{Y}(\lambda) \longrightarrow \mathbf Z
\end{equation*}
by

$$
\! \! \! \! \! \! \! \!   \wt(Y) = \lambda - \sum_{i\in I} k_i \alpha_i,\qquad\qquad\qquad\qquad\qquad\qquad
$$

$$
\varepsilon_i(Y) = \text{the number of $-$'s in the $i$-signature of $Y$,}
$$

$$
\varphi_i(Y) = \text{the number of $+$'s in the $i$-signature of $Y$,}
$$
$$
$$
where $k_i$ is the number of $i$-blocks in $Y$ that have been added to
the ground-state wall $Y_\lambda$.

\vskip 3mm

\begin{theorem}\label{crystal structure on Young walls}
	The maps $\wt: \mathbf{Y}(\lambda) \rightarrow P$,
	$\tilde{E}_i,\tilde{F}_i : \mathbf{Y}(\lambda) \rightarrow \mathbf{Y}(\lambda)\cup \{0\}$ 
	and $\varepsilon_i,\varphi_i : \mathbf{Y}(\lambda) \rightarrow\mathbf Z$ define a crystal
	structure on $\mathbf{Y}(\lambda)$.
\end{theorem}

\begin{proof}
	We only need to prove that  $\mathbf{Y}(\lambda) \cup \{0\}$ is closed under the maps 
 $\tilde{F}_i$ $(i \in I)$. 
	 Let $Y=(\cdots, C_2, C_1,C_0)$ be a level-$l$ reduced Young wall. 
	 
	\vskip 2mm
	
	For types $A_{2n}^{(2)}$ and $D_{n+1}^{(2)}$, only $\tilde{F}_0$-action can produce $\delta$-slices. 

	\vskip 2mm	
	
	Suppose $\tilde{F}_0$ acts on the Young column $C_k$ in $Y$. Then $\phi_0(C_{k+1})\le\epsilon_0(C_{k})$ and $x_1<\bar{x}_1$ in $C_k$.  Thus we have

\vskip -1mm
	\begin{equation*}
	 2l-(\phi_0(C_{k+1})-2{(\bar{x}_1'-x_1')}_{+})\ge 2l-\phi_0(C_{k+1})\ge2l-\epsilon_0(C_{k}).
     \end{equation*}  
\vskip 2mm

	If we can remove the  $\delta$-slice in $C_k$, we obtain

	\begin{equation*}
	2l-(\phi_0(C_{k+1})-2{(\bar{x}_1'-x_1')}_{+})\ge 2l-\phi_0(C_{k+1})>2l-\epsilon_0(C_{k})-2,
	\end{equation*}
\vskip 3mm
\noindent which means the number of $0$-blocks in $C_{k+1}$ is bigger than the number of $0$-blocks in $C_{k}$. This is a contradiction to the tensor product rule. 
	Hence $\tilde{F}_0Y$ is reduced. 
	
	\vskip 2mm

	For types $A_{2n-1}^{(2)}$, $D_{n}^{(1)}$ and $B_{n}^{(1)}$, only the simultaneous action of $\tilde{F}_0$ and $\tilde{F}_1$ can produce removable $\delta$-slices. 
	
	\vskip 2mm	
	
	Suppose  $\tilde{F}_0$ and $\tilde{F}_1$ act on the Young column $C_k$. Then $\phi_0(C_{k+1})\le\epsilon_0(C_{k})$, $\phi_1(C_{k+1})\le\epsilon_1(C_{k})$ and $x_2<\bar{x}_2$ in $C_k$. 
	Thus we obtain
	\vskip -1mm
$$
l-(\phi_0(C_{k+1})-{(\bar{x}_2'-x_2')}_{+})\ge l-\phi_0(C_{k+1})\ge l-\epsilon_0(C_k),
$$
\vskip -2mm
$$
l-(\phi_1(C_{k+1})-{(\bar{x}_2'-x_2')}_{+})\ge l-\phi_1(C_{k+1})\ge l-\epsilon_1(C_k),
$$
which implies the number of $0$-blocks and $1$-blocks in $C_{k+1}$ is no less than 
	the number of $0$-blocks and $1$-blocks in $C_{k}$.

\vskip 2mm

	If we remove a $\delta$-slice in $C_k$, the number of $0$-blocks and $1$-blocks in $C_k$ will be reduced by two. This is also a contradiction to the tensor product rule. Hence $\tilde{F}_0Y$ and  $\tilde{F}_1Y$ are reduced.
	
	\vskip 2mm

	For type $C_{n}^{(1)}$, only  $\tilde{F}_0$-action can produce $\delta$-slices. Let $b=(x_1,\cdots,x_n,\bar{x}_n,\cdots,\bar{x}_1)$ and $b=(x_1',\cdots,x_n',\bar{x}_n',\cdots,\bar{x}_1')$ denote the vectors in the perfect crystal corresponding to the Young columns $C_k$ and $C_{k+1}$, respectively.  
	
	\vskip 2mm

	Suppose  $\tilde{F}_0$ act on the Young column $C_k$ and $x_1=\bar{x}_1-1$ or $x_1\le\bar{x}_1-2$ in $C_k$.  Then $\phi_0(C_{k+1})\le\epsilon_0(C_{k})$ and we obtain the following inequality
	\vskip -1mm
	
	\begin{equation}\label{type C}
	\sum_{i=1}^{n}(x_i+\bar{x}_i)\le \sum_{i=1}^{n}(x_i'+\bar{x}_i')-2{(\bar{x}_1'-x_1')}_{+}.
	\end{equation}

	\vskip 2mm
	
	The  number of $0$-blocks in $C_{k}$ is $2l+\sum_{i=1}^{n}(x_i+\bar{x}_i)$ and the number of $0$-blocks in $C_{k+1}$ is $2l+\sum_{i=1}^{n}(x'_i+\bar{x}'_i)$.
	It follows that

	\begin{equation*}
	2l+\sum_{i=1}^{n}(x'_i+\bar{x}'_i)\le 2l+\sum_{i=1}^{n}(x_i+\bar{x}_i).
	\end{equation*}

	\vskip 2mm
	
	Assume that we can remove a $\delta$-slice and still get a Young wall. 
	
	\vskip 2mm 

	If $x_1=\bar{x}_1-1$ and remove a $\delta$-slice, then we must have 

	\vskip -1mm	
	
\begin{equation*}
		2l+\sum_{i=1}^{n}(x'_i+\bar{x}'_i)\le 2l+\sum_{i=1}^{n}(x_i+\bar{x}_i)-2.
\end{equation*}

	\vskip 1mm
	
    If $x_1\le\bar{x}_1-2$ and remove two $\delta$-slices, then we must have
    
 	\vskip 1mm   
     
\begin{equation*}
	2l+\sum_{i=1}^{n}(x'_i+\bar{x}'_i)\le 2l+\sum_{i=1}^{n}(x_i+\bar{x}_i)-4.
\end{equation*}

	\vskip 2mm
	
	Both cases will contradict to the inequality \eqref{type C}. Hence $\tilde{F}_0Y$ is reduced.
\end{proof}

\vskip 2mm

\begin{theorem}\label{main theorem}
There exists a  $U_q'(\mathfrak g)$-crystal isomorphism 
\vskip -1mm
\begin{equation*}
\Psi: \mathbf{Y}(\lambda) \stackrel{\sim} \longrightarrow B(\lambda)
	\quad \text{given by} \quad 
	Y_{\lambda} \longmapsto u_{\lambda}.
\end{equation*}		

	\end{theorem}
\vskip 0.2mm	
\begin{proof}
Since $B(\lambda)\cong \mathcal P(\lambda)$, we have only to prove $\mathbf{Y}(\lambda)\cong\mathcal P(\lambda)$. Thanks to  the crystal isomorphism $\Phi$ in Theorem \ref{perfect crystal isomorphism}, we may define a map $\Psi:\mathbf{Y}(\lambda)\to \mathcal P(\lambda)$ by

\vskip -1mm 

\begin{equation*}
Y=(\cdots, C_2, C_1, C_0)\longmapsto P_Y=(\cdots, \Phi(C_2), \Phi(C_1), \Phi(C_0)).
\end{equation*}

\vskip 2mm 

Clearly, the map $\Psi$ is a bijection. Using the tensor product rule, it is straightforward to 
check that $\Psi$ commutes with the Kashiwara operators; i.e., 

\vskip -1mm 
\begin{equation*}
\tilde{e}_{i} \circ \Psi = \Psi \circ \tilde{E}_{i},  \quad \tilde{f}_{i} \circ \Psi = \Psi \circ \tilde{F}_{i} \ \ \ \text{for all} \ i \in I.
\end{equation*}
\end{proof}

In the following examples, the arrows represent the action of Kashiwara operators $\tilde{f}_i\ (i\in I)$.

\newpage

\begin{example}
The top part of the  crystal $\mathbf{Y}(2\Lambda_0)$ for $B_{3}^{(1)}$ is given below. 

\vskip 15mm

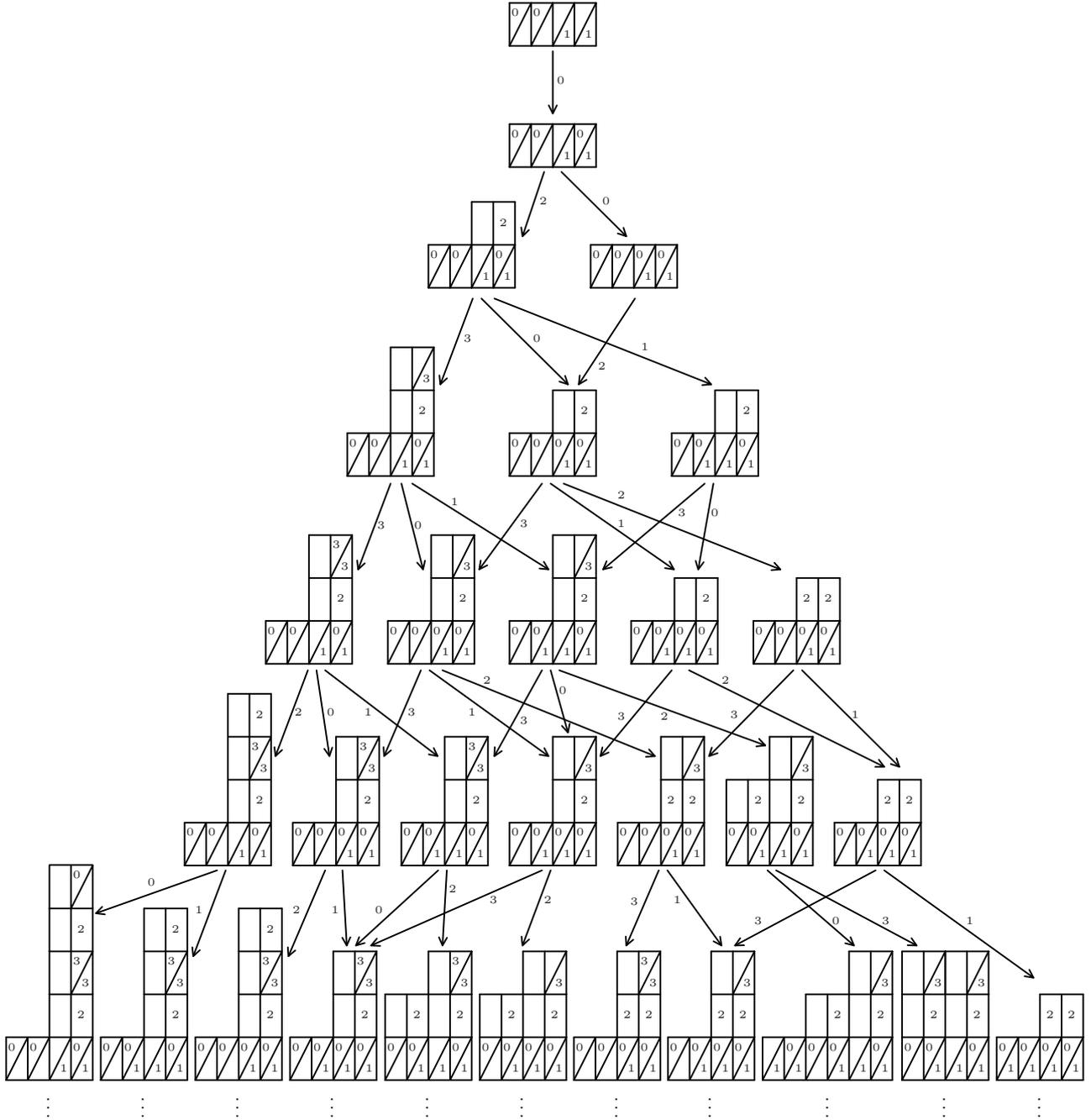
\begin{center}
	\begin{texdraw}
		\drawdim em
		\arrowheadsize l:0.3 w:0.3
		\arrowheadtype t:V
		\fontsize{4}{4}\selectfont
		\drawdim em
		\setunitscale 2
		%
		\move(0 9)
		\bsegment
		\setsegscale 0.8
		\move(0 0)\rlvec(0 2)\move(0 0)\rlvec(-2 0)\move(-2 0)\rlvec(0 2)\move(-2 2)\rlvec(2 0)\move(0 0)\rlvec(2 0)\move(2 0)\rlvec(0 2)\move(2 2)\rlvec(-2 0)\move(-1 0)\rlvec(0 2)\move(1 0)\rlvec(0 2)	\move(-2 0)\rlvec(1 2) \move(-1 0)\rlvec(1 2) \move(0 0)\rlvec(1 2) \move(1 0)\rlvec(1 2) 
		\htext(-1.7 1.5){$0$}\htext(-0.7 1.5){$0$}\htext(0.7 0.5){$1$}\htext(1.7 0.5){$1$}
		
		\esegment
		\move(0 4.5)
		\bsegment
		\setsegscale 0.8
		
		\move(0 0)\rlvec(0 2)\move(0 0)\rlvec(-2 0)\move(-2 0)\rlvec(0 2)\move(-2 2)\rlvec(2 0)\move(0 0)\rlvec(2 0)\move(2 0)\rlvec(0 2)\move(2 2)\rlvec(-2 0)\move(-1 0)\rlvec(0 2)\move(1 0)\rlvec(0 2)	\move(-2 0)\rlvec(1 2) \move(-1 0)\rlvec(1 2) \move(0 0)\rlvec(1 2) \move(1 0)\rlvec(1 2) 
		\htext(-1.7 1.5){$0$}\htext(-0.7 1.5){$0$}
		\htext(1.3 1.5){$0$}
		
		\htext(0.7 0.5){$1$}\htext(1.7 0.5){$1$}
		
		\esegment
		
		\move(0 0)
		\bsegment
		\move(-3 0)
		\bsegment
		\setsegscale 0.8
		
		\move(0 0)\rlvec(0 2)\move(0 0)\rlvec(-2 0)\move(-2 0)\rlvec(0 2)\move(-2 2)\rlvec(2 0)\move(0 0)\rlvec(2 0)\move(2 0)\rlvec(0 2)\move(2 2)\rlvec(-2 0)\move(-1 0)\rlvec(0 2)\move(1 0)\rlvec(0 2)	\move(-2 0)\rlvec(1 2) \move(-1 0)\rlvec(1 2) \move(0 0)\rlvec(1 2) \move(1 0)\rlvec(1 2) 
		\htext(-1.7 1.5){$0$}\htext(-0.7 1.5){$0$}\htext(1.3 1.5){$0$}\htext(0.7 0.5){$1$}\htext(1.7 0.5){$1$}
		
		\move(0 2)\rlvec(0 2)\move(0 4)\rlvec(2 0)\move(2 2)\rlvec(0 2)
		\move(1 2)\rlvec(0 2)\htext(1.5 3){$2$}
		
		\esegment
		\move(3 0)
		\bsegment
		\setsegscale 0.8
		
		\move(0 0)\rlvec(0 2)\move(0 0)\rlvec(-2 0)\move(-2 0)\rlvec(0 2)\move(-2 2)\rlvec(2 0)\move(0 0)\rlvec(2 0)\move(2 0)\rlvec(0 2)\move(2 2)\rlvec(-2 0)\move(-1 0)\rlvec(0 2)\move(1 0)\rlvec(0 2)	\move(-2 0)\rlvec(1 2) \move(-1 0)\rlvec(1 2) \move(0 0)\rlvec(1 2) \move(1 0)\rlvec(1 2) 
		\htext(-1.7 1.5){$0$}\htext(-0.7 1.5){$0$}\htext(0.3 1.5){$0$}\htext(1.3 1.5){$0$}\htext(0.7 0.5){$1$}\htext(1.7 0.5){$1$}
		
		\esegment
		\esegment
		
		\move(0 -7)
		\bsegment
		\move(-6 0)
		\bsegment
		\setsegscale 0.8
		
		\move(0 0)\rlvec(0 2)\move(0 0)\rlvec(-2 0)\move(-2 0)\rlvec(0 2)\move(-2 2)\rlvec(2 0)\move(0 0)\rlvec(2 0)\move(2 0)\rlvec(0 2)\move(2 2)\rlvec(-2 0)\move(-1 0)\rlvec(0 2)\move(1 0)\rlvec(0 2)	\move(-2 0)\rlvec(1 2) \move(-1 0)\rlvec(1 2) \move(0 0)\rlvec(1 2) \move(1 0)\rlvec(1 2) 
		\htext(-1.7 1.5){$0$}\htext(-0.7 1.5){$0$}\htext(1.3 1.5){$0$}\htext(0.7 0.5){$1$}\htext(1.7 0.5){$1$}
		
		\move(0 2)\rlvec(0 2)\move(0 4)\rlvec(2 0)\move(2 2)\rlvec(0 2)
		\move(1 2)\rlvec(0 2)\htext(1.5 3){$2$}\htext(1.7 4.5){$3$}
		
		\move(0 4)\rlvec(0 2)\move(0 6)\rlvec(2 0)\move(2 4)\rlvec(0 2)\move(1 4)\rlvec(0 2)\move(1 4)\rlvec(1 2)
		
		\esegment
		\move(0 0)
		\bsegment
		\setsegscale 0.8
		
		\move(0 0)\rlvec(0 2)\move(0 0)\rlvec(-2 0)\move(-2 0)\rlvec(0 2)\move(-2 2)\rlvec(2 0)\move(0 0)\rlvec(2 0)\move(2 0)\rlvec(0 2)\move(2 2)\rlvec(-2 0)\move(-1 0)\rlvec(0 2)\move(1 0)\rlvec(0 2)	\move(-2 0)\rlvec(1 2) \move(-1 0)\rlvec(1 2) \move(0 0)\rlvec(1 2) \move(1 0)\rlvec(1 2) 
		\htext(-1.7 1.5){$0$}\htext(-0.7 1.5){$0$}\htext(0.3 1.5){$0$}\htext(1.3 1.5){$0$}\htext(0.7 0.5){$1$}\htext(1.7 0.5){$1$}
		
		\move(0 2)\rlvec(0 2)\move(0 4)\rlvec(2 0)\move(2 2)\rlvec(0 2)
		\move(1 2)\rlvec(0 2)\htext(1.5 3){$2$}
		
		\esegment
		\move(6 0)
		\bsegment
		\setsegscale 0.8
		
		\move(0 0)\rlvec(0 2)\move(0 0)\rlvec(-2 0)\move(-2 0)\rlvec(0 2)\move(-2 2)\rlvec(2 0)\move(0 0)\rlvec(2 0)\move(2 0)\rlvec(0 2)\move(2 2)\rlvec(-2 0)\move(-1 0)\rlvec(0 2)\move(1 0)\rlvec(0 2)	\move(-2 0)\rlvec(1 2) \move(-1 0)\rlvec(1 2) \move(0 0)\rlvec(1 2) \move(1 0)\rlvec(1 2) 
		\htext(-1.7 1.5){$0$}\htext(-0.7 1.5){$0$}\htext(1.3 1.5){$0$}\htext(-0.3 0.5){$1$}\htext(0.7 0.5){$1$}\htext(1.7 0.5){$1$}
		
		\move(0 2)\rlvec(0 2)\move(0 4)\rlvec(2 0)\move(2 2)\rlvec(0 2)
		\move(1 2)\rlvec(0 2)\htext(1.5 3){$2$}
		
		\esegment	
		
		\esegment

		\move(0 -14)
		\bsegment
		\move(-9 0)
		\bsegment
		\setsegscale 0.8
		\move(0 0)\rlvec(0 2)\move(0 0)\rlvec(-2 0)\move(-2 0)\rlvec(0 2)\move(-2 2)\rlvec(2 0)\move(0 0)\rlvec(2 0)\move(2 0)\rlvec(0 2)\move(2 2)\rlvec(-2 0)\move(-1 0)\rlvec(0 2)\move(1 0)\rlvec(0 2)	\move(-2 0)\rlvec(1 2) \move(-1 0)\rlvec(1 2) \move(0 0)\rlvec(1 2) \move(1 0)\rlvec(1 2) 
		\htext(-1.7 1.5){$0$}\htext(-0.7 1.5){$0$}\htext(1.3 1.5){$0$}\htext(0.7 0.5){$1$}\htext(1.7 0.5){$1$}
		
		\move(0 2)\rlvec(0 2)\move(0 4)\rlvec(2 0)\move(2 2)\rlvec(0 2)
		\move(1 2)\rlvec(0 2)\htext(1.5 3){$2$}\htext(1.7 4.5){$3$}\htext(1.3 5.5){$3$}
		
		\move(0 4)\rlvec(0 2)\move(0 6)\rlvec(2 0)\move(2 4)\rlvec(0 2)\move(1 4)\rlvec(0 2)\move(1 4)\rlvec(1 2)
		\esegment
		\move(-4.5 0)
		\bsegment
		\setsegscale 0.8
		\move(0 0)\rlvec(0 2)\move(0 0)\rlvec(-2 0)\move(-2 0)\rlvec(0 2)\move(-2 2)\rlvec(2 0)\move(0 0)\rlvec(2 0)\move(2 0)\rlvec(0 2)\move(2 2)\rlvec(-2 0)\move(-1 0)\rlvec(0 2)\move(1 0)\rlvec(0 2)	\move(-2 0)\rlvec(1 2) \move(-1 0)\rlvec(1 2) \move(0 0)\rlvec(1 2) \move(1 0)\rlvec(1 2) 
		\htext(-1.7 1.5){$0$}\htext(-0.7 1.5){$0$}\htext(0.3 1.5){$0$}\htext(1.3 1.5){$0$}\htext(0.7 0.5){$1$}\htext(1.7 0.5){$1$}
		
		\move(0 2)\rlvec(0 2)\move(0 4)\rlvec(2 0)\move(2 2)\rlvec(0 2)
		\move(1 2)\rlvec(0 2)\htext(1.5 3){$2$}\htext(1.7 4.5){$3$}
		
		\move(0 4)\rlvec(0 2)\move(0 6)\rlvec(2 0)\move(2 4)\rlvec(0 2)\move(1 4)\rlvec(0 2)\move(1 4)\rlvec(1 2)
		\esegment
		\move(0 0)
		\bsegment
		\setsegscale 0.8
		\move(0 0)\rlvec(0 2)\move(0 0)\rlvec(-2 0)\move(-2 0)\rlvec(0 2)\move(-2 2)\rlvec(2 0)\move(0 0)\rlvec(2 0)\move(2 0)\rlvec(0 2)\move(2 2)\rlvec(-2 0)\move(-1 0)\rlvec(0 2)\move(1 0)\rlvec(0 2)	\move(-2 0)\rlvec(1 2) \move(-1 0)\rlvec(1 2) \move(0 0)\rlvec(1 2) \move(1 0)\rlvec(1 2) 
		\htext(-1.7 1.5){$0$}\htext(-0.7 1.5){$0$}
		\htext(1.3 1.5){$0$}\htext(-0.3 0.5){$1$}\htext(0.7 0.5){$1$}\htext(1.7 0.5){$1$}
		
		\move(0 2)\rlvec(0 2)\move(0 4)\rlvec(2 0)\move(2 2)\rlvec(0 2)
		\move(1 2)\rlvec(0 2)\htext(1.5 3){$2$}\htext(1.7 4.5){$3$}
		
		\move(0 4)\rlvec(0 2)\move(0 6)\rlvec(2 0)\move(2 4)\rlvec(0 2)\move(1 4)\rlvec(0 2)\move(1 4)\rlvec(1 2)
		\esegment
		\move(4.5 0)
		\bsegment
		\setsegscale 0.8
		\move(0 0)\rlvec(0 2)\move(0 0)\rlvec(-2 0)\move(-2 0)\rlvec(0 2)\move(-2 2)\rlvec(2 0)\move(0 0)\rlvec(2 0)\move(2 0)\rlvec(0 2)\move(2 2)\rlvec(-2 0)\move(-1 0)\rlvec(0 2)\move(1 0)\rlvec(0 2)	\move(-2 0)\rlvec(1 2) \move(-1 0)\rlvec(1 2) \move(0 0)\rlvec(1 2) \move(1 0)\rlvec(1 2) 
		\htext(-1.7 1.5){$0$}\htext(-0.7 1.5){$0$}\htext(0.3 1.5){$0$}
		\htext(1.3 1.5){$0$}\htext(-0.3 0.5){$1$}\htext(0.7 0.5){$1$}\htext(1.7 0.5){$1$}
		
		\move(0 2)\rlvec(0 2)\move(0 4)\rlvec(2 0)\move(2 2)\rlvec(0 2)
		\move(1 2)\rlvec(0 2)\htext(1.5 3){$2$}
		
		\esegment
		\move(9 0)
		\bsegment
		\setsegscale 0.8
		\move(0 0)\rlvec(0 2)\move(0 0)\rlvec(-2 0)\move(-2 0)\rlvec(0 2)\move(-2 2)\rlvec(2 0)\move(0 0)\rlvec(2 0)\move(2 0)\rlvec(0 2)\move(2 2)\rlvec(-2 0)\move(-1 0)\rlvec(0 2)\move(1 0)\rlvec(0 2)	\move(-2 0)\rlvec(1 2) \move(-1 0)\rlvec(1 2) \move(0 0)\rlvec(1 2) \move(1 0)\rlvec(1 2) 
		\htext(-1.7 1.5){$0$}\htext(-0.7 1.5){$0$}\htext(0.3 1.5){$0$}
		\htext(1.3 1.5){$0$}\htext(0.7 0.5){$1$}\htext(1.7 0.5){$1$}
		
		\move(0 2)\rlvec(0 2)\move(0 4)\rlvec(2 0)\move(2 2)\rlvec(0 2)
		\move(1 2)\rlvec(0 2)\htext(0.5 3){$2$}
		\htext(1.5 3){$2$}
		\esegment
		
		\esegment
		\move(0 -21.5)
		\bsegment
		\move(-12 0)
		\bsegment
		\setsegscale 0.8
		\move(0 0)\rlvec(0 2)\move(0 0)\rlvec(-2 0)\move(-2 0)\rlvec(0 2)\move(-2 2)\rlvec(2 0)\move(0 0)\rlvec(2 0)\move(2 0)\rlvec(0 2)\move(2 2)\rlvec(-2 0)\move(-1 0)\rlvec(0 2)\move(1 0)\rlvec(0 2)	\move(-2 0)\rlvec(1 2) \move(-1 0)\rlvec(1 2) \move(0 0)\rlvec(1 2) \move(1 0)\rlvec(1 2) 
		\htext(-1.7 1.5){$0$}\htext(-0.7 1.5){$0$}\htext(1.3 1.5){$0$}\htext(0.7 0.5){$1$}\htext(1.7 0.5){$1$}
		
		\move(0 2)\rlvec(0 2)\move(0 4)\rlvec(2 0)\move(2 2)\rlvec(0 2)
		\move(1 2)\rlvec(0 2)\htext(1.5 3){$2$}\htext(1.7 4.5){$3$}\htext(1.3 5.5){$3$}
		
		\move(0 4)\rlvec(0 2)\move(0 6)\rlvec(2 0)\move(2 4)\rlvec(0 2)\move(1 4)\rlvec(0 2)\move(1 4)\rlvec(1 2)
		
		\move(0 6)\rlvec(0 2)\move(0 8)\rlvec(2 0)\move(2 6)\rlvec(0 2)\move(1 6)\rlvec(0 2)\htext(1.5 7){$2$}
		\esegment
		\move(-8 0)
		\bsegment
		\setsegscale 0.8
		\move(0 0)\rlvec(0 2)\move(0 0)\rlvec(-2 0)\move(-2 0)\rlvec(0 2)\move(-2 2)\rlvec(2 0)\move(0 0)\rlvec(2 0)\move(2 0)\rlvec(0 2)\move(2 2)\rlvec(-2 0)\move(-1 0)\rlvec(0 2)\move(1 0)\rlvec(0 2)	\move(-2 0)\rlvec(1 2) \move(-1 0)\rlvec(1 2) \move(0 0)\rlvec(1 2) \move(1 0)\rlvec(1 2) 
		\htext(-1.7 1.5){$0$}\htext(-0.7 1.5){$0$}\htext(0.3 1.5){$0$}\htext(1.3 1.5){$0$}\htext(0.7 0.5){$1$}\htext(1.7 0.5){$1$}
		
		\move(0 2)\rlvec(0 2)\move(0 4)\rlvec(2 0)\move(2 2)\rlvec(0 2)
		\move(1 2)\rlvec(0 2)\htext(1.5 3){$2$}\htext(1.7 4.5){$3$}\htext(1.3 5.5){$3$}
		
		\move(0 4)\rlvec(0 2)\move(0 6)\rlvec(2 0)\move(2 4)\rlvec(0 2)\move(1 4)\rlvec(0 2)\move(1 4)\rlvec(1 2)
		\esegment
		\move(-4 0)
		\bsegment
		\setsegscale 0.8
		
		\move(0 0)\rlvec(0 2)\move(0 0)\rlvec(-2 0)\move(-2 0)\rlvec(0 2)\move(-2 2)\rlvec(2 0)\move(0 0)\rlvec(2 0)\move(2 0)\rlvec(0 2)\move(2 2)\rlvec(-2 0)\move(-1 0)\rlvec(0 2)\move(1 0)\rlvec(0 2)	\move(-2 0)\rlvec(1 2) \move(-1 0)\rlvec(1 2) \move(0 0)\rlvec(1 2) \move(1 0)\rlvec(1 2) 
		\htext(-1.7 1.5){$0$}\htext(-0.7 1.5){$0$}\htext(1.3 1.5){$0$}\htext(-0.3 0.5){$1$}\htext(0.7 0.5){$1$}\htext(1.7 0.5){$1$}
		
		\move(0 2)\rlvec(0 2)\move(0 4)\rlvec(2 0)\move(2 2)\rlvec(0 2)
		\move(1 2)\rlvec(0 2)\htext(1.5 3){$2$}\htext(1.7 4.5){$3$}\htext(1.3 5.5){$3$}
		
		\move(0 4)\rlvec(0 2)\move(0 6)\rlvec(2 0)\move(2 4)\rlvec(0 2)\move(1 4)\rlvec(0 2)\move(1 4)\rlvec(1 2)
		\esegment
		\move(0 0)
		\bsegment
		\setsegscale 0.8
		
		\move(0 0)\rlvec(0 2)\move(0 0)\rlvec(-2 0)\move(-2 0)\rlvec(0 2)\move(-2 2)\rlvec(2 0)\move(0 0)\rlvec(2 0)\move(2 0)\rlvec(0 2)\move(2 2)\rlvec(-2 0)\move(-1 0)\rlvec(0 2)\move(1 0)\rlvec(0 2)	\move(-2 0)\rlvec(1 2) \move(-1 0)\rlvec(1 2) \move(0 0)\rlvec(1 2) \move(1 0)\rlvec(1 2) 
		\htext(-1.7 1.5){$0$}\htext(-0.7 1.5){$0$}\htext(0.3 1.5){$0$}\htext(1.3 1.5){$0$}\htext(-0.3 0.5){$1$}\htext(0.7 0.5){$1$}\htext(1.7 0.5){$1$}
		
		\move(0 2)\rlvec(0 2)\move(0 4)\rlvec(2 0)\move(2 2)\rlvec(0 2)
		\move(1 2)\rlvec(0 2)\htext(1.5 3){$2$}\htext(1.7 4.5){$3$}
		
		\move(0 4)\rlvec(0 2)\move(0 6)\rlvec(2 0)\move(2 4)\rlvec(0 2)\move(1 4)\rlvec(0 2)\move(1 4)\rlvec(1 2)
		
		\esegment
		\move(4 0)
		\bsegment
		\setsegscale 0.8
		
		\move(0 0)\rlvec(0 2)\move(0 0)\rlvec(-2 0)\move(-2 0)\rlvec(0 2)\move(-2 2)\rlvec(2 0)\move(0 0)\rlvec(2 0)\move(2 0)\rlvec(0 2)\move(2 2)\rlvec(-2 0)\move(-1 0)\rlvec(0 2)\move(1 0)\rlvec(0 2)	\move(-2 0)\rlvec(1 2) \move(-1 0)\rlvec(1 2) \move(0 0)\rlvec(1 2) \move(1 0)\rlvec(1 2) 
		\htext(-1.7 1.5){$0$}\htext(-0.7 1.5){$0$}\htext(0.3 1.5){$0$}\htext(1.3 1.5){$0$}\htext(0.7 0.5){$1$}\htext(1.7 0.5){$1$}
		
		\move(0 2)\rlvec(0 2)\move(0 4)\rlvec(2 0)\move(2 2)\rlvec(0 2)
		\move(1 2)\rlvec(0 2)\htext(0.5 3){$2$}\htext(1.5 3){$2$}\htext(1.7 4.5){$3$}
		
		\move(0 4)\rlvec(0 2)\move(0 6)\rlvec(2 0)\move(2 4)\rlvec(0 2)\move(1 4)\rlvec(0 2)\move(1 4)\rlvec(1 2)	
		
		\esegment
		\move(8 0)
		\bsegment
		\setsegscale 0.8
		
		\move(0 0)\rlvec(0 2)\move(0 0)\rlvec(-2 0)\move(-2 0)\rlvec(0 2)\move(-2 2)\rlvec(2 0)\move(0 0)\rlvec(2 0)\move(2 0)\rlvec(0 2)\move(2 2)\rlvec(-2 0)\move(-1 0)\rlvec(0 2)\move(1 0)\rlvec(0 2)	\move(-2 0)\rlvec(1 2) \move(-1 0)\rlvec(1 2) \move(0 0)\rlvec(1 2) \move(1 0)\rlvec(1 2) 
		\htext(-1.7 1.5){$0$}\htext(-0.7 1.5){$0$}\htext(1.3 1.5){$0$}\htext(-0.3 0.5){$1$}\htext(0.7 0.5){$1$}\htext(1.7 0.5){$1$}
		
		\move(0 2)\rlvec(0 2)\move(0 4)\rlvec(2 0)\move(2 2)\rlvec(0 2)
		\move(1 2)\rlvec(0 2)\htext(1.5 3){$2$}\htext(1.7 4.5){$3$}
		
		\move(0 4)\rlvec(0 2)\move(0 6)\rlvec(2 0)\move(2 4)\rlvec(0 2)\move(1 4)\rlvec(0 2)\move(1 4)\rlvec(1 2)
		
		\move(-2 2)\rlvec(0 2)	\move(-2 4)\rlvec(2 0) \move(-1 2)\rlvec(0 2)
		\htext(-0.5 3){$2$}
		\esegment
		\move(12 0)
		\bsegment
		\setsegscale 0.8
		
		\move(0 0)\rlvec(0 2)\move(0 0)\rlvec(-2 0)\move(-2 0)\rlvec(0 2)\move(-2 2)\rlvec(2 0)\move(0 0)\rlvec(2 0)\move(2 0)\rlvec(0 2)\move(2 2)\rlvec(-2 0)\move(-1 0)\rlvec(0 2)\move(1 0)\rlvec(0 2)	\move(-2 0)\rlvec(1 2) \move(-1 0)\rlvec(1 2) \move(0 0)\rlvec(1 2) \move(1 0)\rlvec(1 2) 
		\htext(-1.7 1.5){$0$}\htext(-0.7 1.5){$0$}\htext(0.3 1.5){$0$}\htext(1.3 1.5){$0$}\htext(-0.3 0.5){$1$}\htext(0.7 0.5){$1$}\htext(1.7 0.5){$1$}
		
		\move(0 2)\rlvec(0 2)\move(0 4)\rlvec(2 0)\move(2 2)\rlvec(0 2)
		\move(1 2)\rlvec(0 2)\htext(0.5 3){$2$} \htext(1.5 3){$2$}
		\esegment

		\esegment
		\move(0 -29.5)
		\bsegment
		\move(-18.6 0)
		\bsegment
		\setsegscale 0.8
		\move(0 0)\rlvec(0 2)\move(0 0)\rlvec(-2 0)\move(-2 0)\rlvec(0 2)\move(-2 2)\rlvec(2 0)\move(0 0)\rlvec(2 0)\move(2 0)\rlvec(0 2)\move(2 2)\rlvec(-2 0)\move(-1 0)\rlvec(0 2)\move(1 0)\rlvec(0 2)	\move(-2 0)\rlvec(1 2) \move(-1 0)\rlvec(1 2) \move(0 0)\rlvec(1 2) \move(1 0)\rlvec(1 2) 
		\htext(-1.7 1.5){$0$}\htext(-0.7 1.5){$0$}\htext(1.3 1.5){$0$}\htext(0.7 0.5){$1$}\htext(1.7 0.5){$1$}
		
		\move(0 2)\rlvec(0 2)\move(0 4)\rlvec(2 0)\move(2 2)\rlvec(0 2)
		\move(1 2)\rlvec(0 2)\htext(1.5 3){$2$}\htext(1.7 4.5){$3$}\htext(1.3 5.5){$3$}
		
		\move(0 4)\rlvec(0 2)\move(0 6)\rlvec(2 0)\move(2 4)\rlvec(0 2)\move(1 4)\rlvec(0 2)\move(1 4)\rlvec(1 2)
		
		\move(0 6)\rlvec(0 2)\move(0 8)\rlvec(2 0)\move(2 6)\rlvec(0 2)\move(1 6)\rlvec(0 2)\htext(1.5 7){$2$}
		
		\move(0 8)\rlvec(0 2)\move(0 10)\rlvec(2 0)\move(2 8)\rlvec(0 2)\move(1 8)\rlvec(0 2)\move(1 8)\rlvec(1 2)\htext(1.3 9.5){$0$}
		\esegment
		\move(-15.1 0)
		\bsegment
		\setsegscale 0.8
		\move(0 0)\rlvec(0 2)\move(0 0)\rlvec(-2 0)\move(-2 0)\rlvec(0 2)\move(-2 2)\rlvec(2 0)\move(0 0)\rlvec(2 0)\move(2 0)\rlvec(0 2)\move(2 2)\rlvec(-2 0)\move(-1 0)\rlvec(0 2)\move(1 0)\rlvec(0 2)	\move(-2 0)\rlvec(1 2) \move(-1 0)\rlvec(1 2) \move(0 0)\rlvec(1 2) \move(1 0)\rlvec(1 2) 
		\htext(-1.7 1.5){$0$}\htext(-0.7 1.5){$0$}\htext(1.3 1.5){$0$}\htext(-0.3 0.5){$1$}\htext(0.7 0.5){$1$}\htext(1.7 0.5){$1$}
		
		\move(0 2)\rlvec(0 2)\move(0 4)\rlvec(2 0)\move(2 2)\rlvec(0 2)
		\move(1 2)\rlvec(0 2)\htext(1.5 3){$2$}\htext(1.7 4.5){$3$}\htext(1.3 5.5){$3$}
		
		\move(0 4)\rlvec(0 2)\move(0 6)\rlvec(2 0)\move(2 4)\rlvec(0 2)\move(1 4)\rlvec(0 2)\move(1 4)\rlvec(1 2)
		
		\move(0 6)\rlvec(0 2)\move(0 8)\rlvec(2 0)\move(2 6)\rlvec(0 2)\move(1 6)\rlvec(0 2)\htext(1.5 7){$2$}
		\esegment
		\move(-11.6 0)
		\bsegment
		\setsegscale 0.8
		
		\move(0 0)\rlvec(0 2)\move(0 0)\rlvec(-2 0)\move(-2 0)\rlvec(0 2)\move(-2 2)\rlvec(2 0)\move(0 0)\rlvec(2 0)\move(2 0)\rlvec(0 2)\move(2 2)\rlvec(-2 0)\move(-1 0)\rlvec(0 2)\move(1 0)\rlvec(0 2)	\move(-2 0)\rlvec(1 2) \move(-1 0)\rlvec(1 2) \move(0 0)\rlvec(1 2) \move(1 0)\rlvec(1 2) 
		\htext(-1.7 1.5){$0$}\htext(-0.7 1.5){$0$}\htext(0.3 1.5){$0$}\htext(1.3 1.5){$0$}\htext(0.7 0.5){$1$}\htext(1.7 0.5){$1$}
		
		\move(0 2)\rlvec(0 2)\move(0 4)\rlvec(2 0)\move(2 2)\rlvec(0 2)
		\move(1 2)\rlvec(0 2)\htext(1.5 3){$2$}\htext(1.7 4.5){$3$}\htext(1.3 5.5){$3$}
		
		\move(0 4)\rlvec(0 2)\move(0 6)\rlvec(2 0)\move(2 4)\rlvec(0 2)\move(1 4)\rlvec(0 2)\move(1 4)\rlvec(1 2)
		
		\move(0 6)\rlvec(0 2)\move(0 8)\rlvec(2 0)\move(2 6)\rlvec(0 2)\move(1 6)\rlvec(0 2)\htext(1.5 7){$2$}
		\esegment
		\move(-8.1 0)
		\bsegment
		\setsegscale 0.8
		
		\move(0 0)\rlvec(0 2)\move(0 0)\rlvec(-2 0)\move(-2 0)\rlvec(0 2)\move(-2 2)\rlvec(2 0)\move(0 0)\rlvec(2 0)\move(2 0)\rlvec(0 2)\move(2 2)\rlvec(-2 0)\move(-1 0)\rlvec(0 2)\move(1 0)\rlvec(0 2)	\move(-2 0)\rlvec(1 2) \move(-1 0)\rlvec(1 2) \move(0 0)\rlvec(1 2) \move(1 0)\rlvec(1 2) 
		\htext(-1.7 1.5){$0$}\htext(-0.7 1.5){$0$}\htext(0.3 1.5){$0$}\htext(1.3 1.5){$0$}\htext(-0.3 0.5){$1$}\htext(0.7 0.5){$1$}\htext(1.7 0.5){$1$}
		
		\move(0 2)\rlvec(0 2)\move(0 4)\rlvec(2 0)\move(2 2)\rlvec(0 2)
		\move(1 2)\rlvec(0 2)\htext(1.5 3){$2$}\htext(1.7 4.5){$3$}\htext(1.3 5.5){$3$}
		
		\move(0 4)\rlvec(0 2)\move(0 6)\rlvec(2 0)\move(2 4)\rlvec(0 2)\move(1 4)\rlvec(0 2)\move(1 4)\rlvec(1 2)	
		
		\esegment
		\move(-4.6 0)
		\bsegment
		\setsegscale 0.8
		
		\move(0 0)\rlvec(0 2)\move(0 0)\rlvec(-2 0)\move(-2 0)\rlvec(0 2)\move(-2 2)\rlvec(2 0)\move(0 0)\rlvec(2 0)\move(2 0)\rlvec(0 2)\move(2 2)\rlvec(-2 0)\move(-1 0)\rlvec(0 2)\move(1 0)\rlvec(0 2)	\move(-2 0)\rlvec(1 2) \move(-1 0)\rlvec(1 2) \move(0 0)\rlvec(1 2) \move(1 0)\rlvec(1 2) 
		\htext(-1.7 1.5){$0$}\htext(-0.7 1.5){$0$}\htext(1.3 1.5){$0$}\htext(-0.3 0.5){$1$}\htext(0.7 0.5){$1$}\htext(1.7 0.5){$1$}
		
		\move(0 2)\rlvec(0 2)\move(0 4)\rlvec(2 0)\move(2 2)\rlvec(0 2)
		\move(1 2)\rlvec(0 2)\htext(1.5 3){$2$}\htext(1.7 4.5){$3$}
		\htext(1.3 5.5){$3$}
		
		\move(0 4)\rlvec(0 2)\move(0 6)\rlvec(2 0)\move(2 4)\rlvec(0 2)\move(1 4)\rlvec(0 2)\move(1 4)\rlvec(1 2)
		
		\move(-2 2)\rlvec(0 2)	\move(-2 4)\rlvec(2 0) \move(-1 2)\rlvec(0 2)
		\htext(-0.5 3){$2$}
		
		\esegment
		\move(-1.1 0)
		\bsegment
		\setsegscale 0.8
		
		\move(0 0)\rlvec(0 2)\move(0 0)\rlvec(-2 0)\move(-2 0)\rlvec(0 2)\move(-2 2)\rlvec(2 0)\move(0 0)\rlvec(2 0)\move(2 0)\rlvec(0 2)\move(2 2)\rlvec(-2 0)\move(-1 0)\rlvec(0 2)\move(1 0)\rlvec(0 2)	\move(-2 0)\rlvec(1 2) \move(-1 0)\rlvec(1 2) \move(0 0)\rlvec(1 2) \move(1 0)\rlvec(1 2) 
		\htext(-1.7 1.5){$0$}\htext(-0.7 1.5){$0$}\htext(0.3 1.5){$0$}\htext(1.3 1.5){$0$}\htext(-0.3 0.5){$1$}\htext(0.7 0.5){$1$}\htext(1.7 0.5){$1$}
		
		\move(0 2)\rlvec(0 2)\move(0 4)\rlvec(2 0)\move(2 2)\rlvec(0 2)
		\move(1 2)\rlvec(0 2)\htext(1.5 3){$2$}\htext(1.7 4.5){$3$}
		
		\move(0 4)\rlvec(0 2)\move(0 6)\rlvec(2 0)\move(2 4)\rlvec(0 2)\move(1 4)\rlvec(0 2)\move(1 4)\rlvec(1 2)
		
		\move(-2 2)\rlvec(0 2)	\move(-2 4)\rlvec(2 0) \move(-1 2)\rlvec(0 2)
		\htext(-0.5 3){$2$}
		
		\esegment
		\move(2.37 0)
		\bsegment
		\setsegscale 0.8
		
		\move(0 0)\rlvec(0 2)\move(0 0)\rlvec(-2 0)\move(-2 0)\rlvec(0 2)\move(-2 2)\rlvec(2 0)\move(0 0)\rlvec(2 0)\move(2 0)\rlvec(0 2)\move(2 2)\rlvec(-2 0)\move(-1 0)\rlvec(0 2)\move(1 0)\rlvec(0 2)	\move(-2 0)\rlvec(1 2) \move(-1 0)\rlvec(1 2) \move(0 0)\rlvec(1 2) \move(1 0)\rlvec(1 2) 
		\htext(-1.7 1.5){$0$}\htext(-0.7 1.5){$0$}\htext(0.3 1.5){$0$}\htext(1.3 1.5){$0$}\htext(0.7 0.5){$1$}\htext(1.7 0.5){$1$}
		
		\move(0 2)\rlvec(0 2)\move(0 4)\rlvec(2 0)\move(2 2)\rlvec(0 2)
		\move(1 2)\rlvec(0 2)\htext(0.5 3){$2$}\htext(1.5 3){$2$}\htext(1.7 4.5){$3$}
		\htext(1.3 5.5){$3$}
		
		\move(0 4)\rlvec(0 2)\move(0 6)\rlvec(2 0)\move(2 4)\rlvec(0 2)\move(1 4)\rlvec(0 2)\move(1 4)\rlvec(1 2)

		\esegment
		\move(5.85 0)
		\bsegment
		\setsegscale 0.8
		
		\move(0 0)\rlvec(0 2)\move(0 0)\rlvec(-2 0)\move(-2 0)\rlvec(0 2)\move(-2 2)\rlvec(2 0)\move(0 0)\rlvec(2 0)\move(2 0)\rlvec(0 2)\move(2 2)\rlvec(-2 0)\move(-1 0)\rlvec(0 2)\move(1 0)\rlvec(0 2)	\move(-2 0)\rlvec(1 2) \move(-1 0)\rlvec(1 2) \move(0 0)\rlvec(1 2) \move(1 0)\rlvec(1 2) 
		\htext(-1.7 1.5){$0$}\htext(-0.7 1.5){$0$}\htext(0.3 1.5){$0$}\htext(1.3 1.5){$0$}\htext(-0.3 0.5){$1$}\htext(0.7 0.5){$1$}\htext(1.7 0.5){$1$}
		
		\move(0 2)\rlvec(0 2)\move(0 4)\rlvec(2 0)\move(2 2)\rlvec(0 2)
		\move(1 2)\rlvec(0 2)\htext(0.5 3){$2$}\htext(1.5 3){$2$}\htext(1.7 4.5){$3$}
		
		\move(0 4)\rlvec(0 2)\move(0 6)\rlvec(2 0)\move(2 4)\rlvec(0 2)\move(1 4)\rlvec(0 2)\move(1 4)\rlvec(1 2)

		\esegment
		
		\move(10.95 0)
		\bsegment
		\setsegscale 0.8
		
		\move(0 0)\rlvec(0 2)\move(0 0)\rlvec(-2 0)\move(-2 0)\rlvec(0 2)\move(-2 2)\rlvec(2 0)\move(0 0)\rlvec(2 0)\move(2 0)\rlvec(0 2)\move(2 2)\rlvec(-2 0)\move(-1 0)\rlvec(0 2)\move(1 0)\rlvec(0 2)	\move(-2 0)\rlvec(1 2) \move(-1 0)\rlvec(1 2) \move(0 0)\rlvec(1 2) \move(1 0)\rlvec(1 2) 
		\htext(-1.7 1.5){$0$}\htext(-0.7 1.5){$0$}\htext(1.3 1.5){$0$}\htext(-0.3 0.5){$1$}\htext(0.7 0.5){$1$}\htext(1.7 0.5){$1$}
		
		\move(0 2)\rlvec(0 2)\move(0 4)\rlvec(2 0)\move(2 2)\rlvec(0 2)
		\move(1 2)\rlvec(0 2)\htext(1.5 3){$2$}\htext(1.7 4.5){$3$}
		
		\move(0 4)\rlvec(0 2)\move(0 6)\rlvec(2 0)\move(2 4)\rlvec(0 2)\move(1 4)\rlvec(0 2)\move(1 4)\rlvec(1 2)
		
		\move(-2 2)\rlvec(0 2)	\move(-2 4)\rlvec(2 0) \move(-1 2)\rlvec(0 2)
		\htext(-0.5 3){$2$}
		
		\move(-4 0)\rlvec(2 0)\move(-4 0)\rlvec(0 2)\move(-4 2)\rlvec(2 0)\move(-3 0)\rlvec(0 2)\move(-4 0)\rlvec(1 2)\move(-3 0)\rlvec(1 2)
		\htext(-3.3 0.5){$1$}\htext(-2.3 0.5){$1$}\htext(-2.7 1.5){$0$}
		\esegment
		\move(14.5 0)
		\bsegment
		\setsegscale 0.8
		
		\move(0 0)\rlvec(0 2)\move(0 0)\rlvec(-2 0)\move(-2 0)\rlvec(0 2)\move(-2 2)\rlvec(2 0)\move(0 0)\rlvec(2 0)\move(2 0)\rlvec(0 2)\move(2 2)\rlvec(-2 0)\move(-1 0)\rlvec(0 2)\move(1 0)\rlvec(0 2)	\move(-2 0)\rlvec(1 2) \move(-1 0)\rlvec(1 2) \move(0 0)\rlvec(1 2) \move(1 0)\rlvec(1 2) 
		\htext(-1.7 1.5){$0$}\htext(-0.7 1.5){$0$}\htext(1.3 1.5){$0$}\htext(-0.3 0.5){$1$}\htext(0.7 0.5){$1$}\htext(1.7 0.5){$1$}
		
		\move(0 2)\rlvec(0 2)\move(0 4)\rlvec(2 0)\move(2 2)\rlvec(0 2)
		\move(1 2)\rlvec(0 2)\htext(1.5 3){$2$}\htext(1.7 4.5){$3$}
		
		\move(0 4)\rlvec(0 2)\move(0 6)\rlvec(2 0)\move(2 4)\rlvec(0 2)\move(1 4)\rlvec(0 2)\move(1 4)\rlvec(1 2)
		
		\move(-2 2)\rlvec(0 2)	\move(-2 4)\rlvec(2 0) \move(-1 2)\rlvec(0 2)
		\htext(-0.5 3){$2$}
		
		\move(-2 4)\rlvec(0 2)\move(-2 6)\rlvec(2 0)\move(0 4)\rlvec(0 2)\move(-1 4)\rlvec(0 2)\move(-1 4)\rlvec(1 2)\htext(-0.3 4.5){$3$}
		
		\esegment
		\move(18 0)
		\bsegment
		\setsegscale 0.8
		
		\move(0 0)\rlvec(0 2)\move(0 0)\rlvec(-2 0)\move(-2 0)\rlvec(0 2)\move(-2 2)\rlvec(2 0)\move(0 0)\rlvec(2 0)\move(2 0)\rlvec(0 2)\move(2 2)\rlvec(-2 0)\move(-1 0)\rlvec(0 2)\move(1 0)\rlvec(0 2)	\move(-2 0)\rlvec(1 2) \move(-1 0)\rlvec(1 2) \move(0 0)\rlvec(1 2) \move(1 0)\rlvec(1 2) 
		\htext(-1.7 1.5){$0$}\htext(-0.7 1.5){$0$}\htext(0.3 1.5){$0$}\htext(1.3 1.5){$0$}\htext(-1.3 0.5){$1$}\htext(-0.3 0.5){$1$}\htext(0.7 0.5){$1$}\htext(1.7 0.5){$1$}
		
		\move(0 2)\rlvec(0 2)\move(0 4)\rlvec(2 0)\move(2 2)\rlvec(0 2)
		\move(1 2)\rlvec(0 2)\htext(0.5 3){$2$} \htext(1.5 3){$2$}	
		
		\esegment
		
		\esegment	
		\move(-1.2 0)
		\bsegment
		\setsegscale 0.8
		\move(1.5 11)\ravec(0 -2.9)\htext(1.9 9.6){$0$}
		
		\move(1.1 5.4)\ravec(-1 -3)\htext(1.1 4){$2$}
		
		\move(1.9 5.4)\ravec(3 -3)\htext(4 4){$0$}
		
		\move(-2.2 -0.5)\ravec(-1.5 -4)\htext(-2.4 -2.4){$3$}
		\move(-1.8 -0.5)\ravec(4 -4)\htext(0.8 -2.4){$0$}
		\move(-1.2 -0.5)\ravec(10 -4)\htext(5.8 -2.8){$1$}
		\move(5.3 -0.5)\ravec(-2.6 -4)\htext(3.8 -3.7){$2$}
		\move(-6 -9.1)\ravec(-1.5 -4)\htext(-6.4 -11.1){$3$}
		\move(-5.5 -9.1)\ravec(1 -4)\htext(-4.7 -11.1){$0$}
		\move(-5 -9.1)\ravec(6.3 -4)\htext(-3 -10){$1$}
		\move(1 -9.1)\ravec(-2.9 -4)\htext(0.2 -11){$3$}
		\move(1.4 -9.1)\ravec(5.7 -4)\htext(4.7 -11){$1$}
		\move(2 -9.1)\ravec(10 -4)\htext(4.7 -9.7){$2$}
		\move(8.5 -9.1)\ravec(-4.7 -4)\htext(7.5 -10.5){$3$}
		\move(8.9 -9.1)\ravec(-0.7 -4)\htext(9 -10.5){$0$}
		\move(-9.8 -17.8)\ravec(-1.5 -4)\htext(-10.2 -19.8){$2$}
		\move(-9.4 -17.8)\ravec(0.6 -4)\htext(-8.7 -19.8){$0$}
		\move(-9 -17.8)\ravec(5.2 -4)\htext(-7 -19.8){$1$}
		\move(-4.6 -17.8)\ravec(-1.7 -4)\htext(-5 -19.8){$3$}
		\move(-4.2 -17.8)\ravec(5.5 -4)\htext(-2.2 -19.8){$1$}
		\move(1 -17.8)\ravec(-2.2 -4)\htext(0.2 -20.2){$3$}
		\move(-3.6 -17.8)\ravec(9.8 -4)\htext(-1.5 -18.3){$2$}
		\move(1.4 -17.8)\ravec(0.8 -2.9)\htext(2 -18.8){$0$}
		\move(1.8 -17.8)\ravec(9.5 -3.5)\htext(6.7 -20){$2$}
		\move(7 -17.8)\ravec(-3.3 -4)\htext(4.7 -20){$3$}
		\move(7.8 -17.8)\ravec(9 -4.5)\htext(9.5 -18.3){$2$}
		\move(12.6 -17.8)\ravec(-3.9 -4)\htext(9.9 -19.9){$3$}
		\move(13 -17.8)\ravec(4.5 -4.5)\htext(15.5 -19.9){$1$}
		\move(-14 -27.1)\ravec(-5.6 -2)\htext(-17 -27.7){$0$}
		\move(-13.6 -27.1)\ravec(-1.5 -4)\htext(-14.8 -29){$1$}
		\move(-9 -27.1)\ravec(-1.7 -4)\htext(-10.3 -29){$2$}
		\move(-8.2 -27.1)\ravec(0.2 -3.5)\htext(-8.5 -29){$1$}
		\move(-3.8 -27.1)\ravec(-3.8 -3.5)\htext(-6.5 -29){$0$}
		\move(-3.4 -27.1)\ravec(-0.2 -3.5)\htext(-3.1 -28){$2$}
		\move(1 -27.1)\ravec(-7.9 -3.5)\htext(-1.2 -28.5){$3$}
		\move(1.4 -27.1)\ravec(-1.3 -3.5)\htext(1.3 -28.5){$2$}
		\move(6.4 -27.1)\ravec(-1.5 -3.5)\htext(5.3 -28.6){$3$}
		\move(6.8 -27.1)\ravec(2.5 -3.5)\htext(7.3 -28.5){$1$}
		\move(11.4 -27.1)\ravec(4 -3.5)\htext(14.6 -29.5){$0$}
		\move(11.8 -27.1)\ravec(6.5 -3.5)\htext(16.9 -29.5){$3$}
		\move(16.4 -27.1)\ravec(-6.5 -3.5)\htext(11 -29.5){$3$}
		\move(16.8 -27.1)\ravec(6.9 -5)\htext(20.8 -29.5){$1$}
		\esegment
		
		\vtext(-18.6 -30.5){\fontsize{8}{8}\selectfont{$\cdots$}}
		\vtext(-15.1 -30.5){\fontsize{8}{8}\selectfont{$\cdots$}}
		\vtext(-11.6 -30.5){\fontsize{8}{8}\selectfont{$\cdots$}}
		\vtext(-8.1 -30.5){\fontsize{8}{8}\selectfont{$\cdots$}}
		\vtext(-4.6 -30.5){\fontsize{8}{8}\selectfont{$\cdots$}}
		\vtext(-1.1 -30.5){\fontsize{8}{8}\selectfont{$\cdots$}}
		\vtext(2.4 -30.5){\fontsize{8}{8}\selectfont{$\cdots$}}
		\vtext(5.85 -30.5){\fontsize{8}{8}\selectfont{$\cdots$}}
		\vtext(10.18 -30.5){\fontsize{8}{8}\selectfont{$\cdots$}}
		\vtext(14.5 -30.5){\fontsize{8}{8}\selectfont{$\cdots$}}
		\vtext(18.02 -30.5){\fontsize{8}{8}\selectfont{$\cdots$}}
	\end{texdraw}
\end{center}

\end{example}

\clearpage
\begin{example}
The top part of the  crystal $\mathbf{Y}(2\Lambda_0)$ for $A_{2}^{(2)}$ is given below.

\vskip 15mm

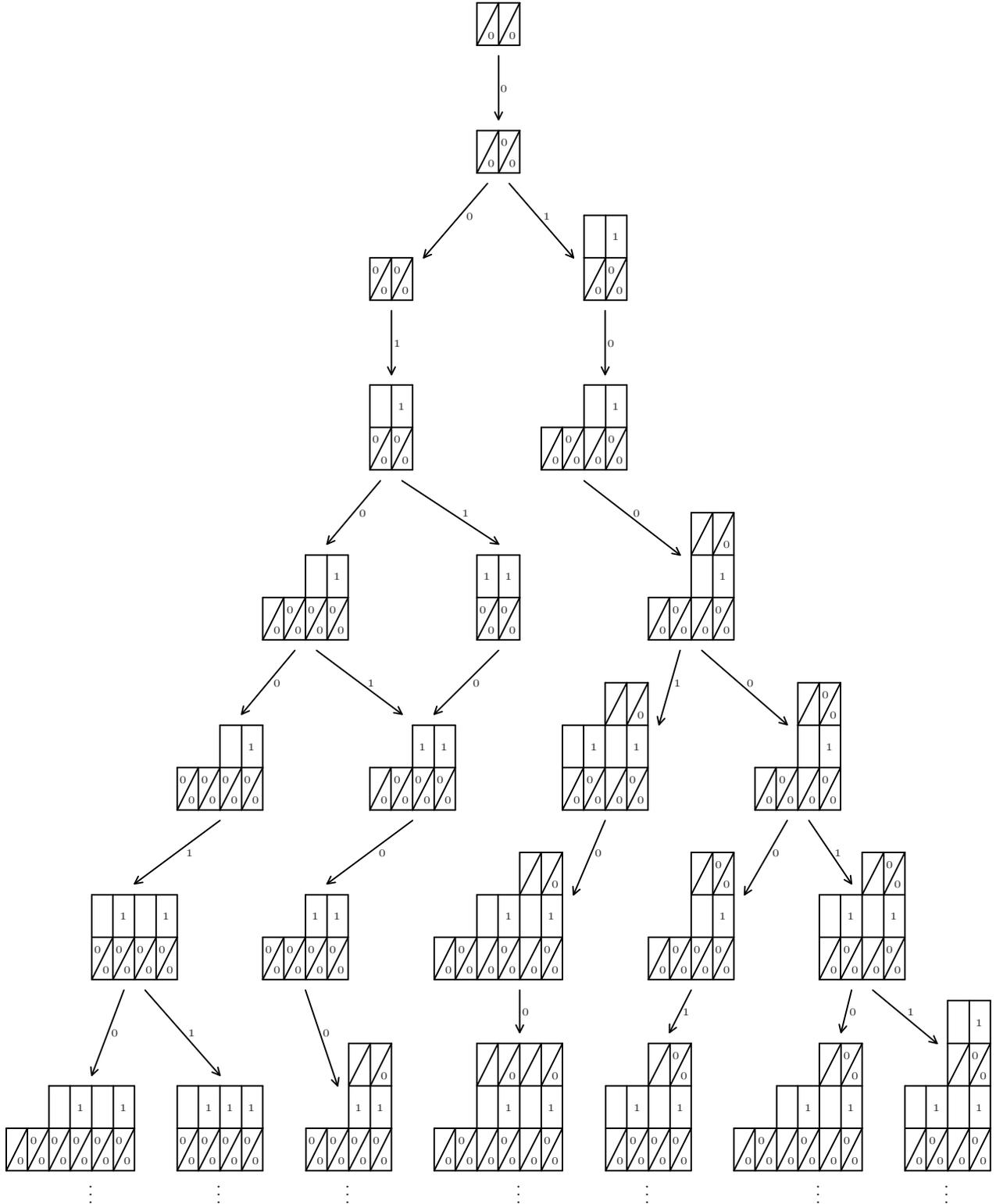
\begin{center}
	\begin{texdraw}
		\drawdim em
		\arrowheadsize l:0.3 w:0.3
		\arrowheadtype t:V
		\fontsize{4}{4}\selectfont
		\drawdim em
		\setunitscale 1.7
		\move(0 5)
		\bsegment
		\move(0 0)\lvec(2 0)
		\move(0 0)\lvec(0 2)
		\move(0 2)\lvec(2 2)
		\move(2 0)\lvec(2 2)
		\move(1 0)\lvec(1 2)
		\move(0 0)\lvec(1 2)
		\move(1 0)\lvec(2 2)
		\htext(0.7 0.5){$0$}
		\htext(1.7 0.5){$0$}
		
		\esegment
		\move(0 -1)
		\bsegment
		\move(0 0)\lvec(2 0)
		\move(0 0)\lvec(0 2)
		\move(0 2)\lvec(2 2)
		\move(2 0)\lvec(2 2)
		\move(1 0)\lvec(1 2)
		\move(0 0)\lvec(1 2)
		\move(1 0)\lvec(2 2)
		\htext(0.7 0.5){$0$}
		\htext(1.7 0.5){$0$}
		\htext(1.31 1.45){$0$}
		\esegment
		\move(0 -7)
		\bsegment
		\move(5 0)
		\bsegment
		\move(0 0)\lvec(2 0)
		\move(0 0)\lvec(0 2)
		\move(0 2)\lvec(2 2)
		\move(2 0)\lvec(2 2)
		\move(1 0)\lvec(1 2)
		\move(0 0)\lvec(1 2)
		\move(1 0)\lvec(2 2)
		\htext(0.7 0.5){$0$}
		\htext(1.7 0.5){$0$}
		\htext(1.31 1.45){$0$}
		\move(0 2)\lvec(0 4)
		\move(2 2)\lvec(2 4)
		\move(0 4)\lvec(2 4)
		\move(1 2)\lvec(1 4)
		\htext(1.5 3){$1$}
		\esegment
		
		\move(-5 0)
		\bsegment
		\move(0 0)\lvec(2 0)
		\move(0 0)\lvec(0 2)
		\move(0 2)\lvec(2 2)
		\move(2 0)\lvec(2 2)
		\move(1 0)\lvec(1 2)
		\move(0 0)\lvec(1 2)
		\move(1 0)\lvec(2 2)
		\htext(0.7 0.5){$0$}
		\htext(1.7 0.5){$0$}
		\htext(1.31 1.45){$0$}
		\htext(0.31 1.45){$0$}
		\esegment
		
		\esegment
		\move(0 -15)
		\bsegment
		\move(5 0)
		\bsegment
		\move(0 0)\lvec(2 0)
		\move(0 0)\lvec(0 2)
		\move(0 2)\lvec(2 2)
		\move(2 0)\lvec(2 2)
		\move(1 0)\lvec(1 2)
		\move(0 0)\lvec(1 2)
		\move(1 0)\lvec(2 2)
		\move(0 2)\lvec(0 4)
		\move(2 2)\lvec(2 4)
		\move(0 4)\lvec(2 4)
		\move(1 2)\lvec(1 4)
		
		\htext(0.7 0.5){$0$}
		\htext(1.7 0.5){$0$}
		\htext(1.31 1.45){$0$}
		
		\move(-2 0)\lvec(0 0)
		\move(-2 0)\lvec(-2 2)
		\move(-2 2)\lvec(0 2)
		\move(-1 0)\lvec(-1 2)
		\move(-1 0)\lvec(0 2)
		\move(-2 0)\lvec(-1 2)
		\htext(-1.3 0.5){$0$}
		\htext(-0.3 0.5){$0$}
		\htext(-0.69 1.45){$0$}
		\htext(1.5 3){$1$}
		\esegment
		
		\move(-5 0)
		\bsegment
		\move(0 0)\lvec(2 0)
		\move(0 0)\lvec(0 2)
		\move(0 2)\lvec(2 2)
		\move(2 0)\lvec(2 2)
		\move(1 0)\lvec(1 2)
		\move(0 0)\lvec(1 2)
		\move(1 0)\lvec(2 2)
		\move(0 2)\lvec(0 4)
		\move(2 2)\lvec(2 4)
		\move(0 4)\lvec(2 4)
		\move(1 2)\lvec(1 4)
		\htext(0.7 0.5){$0$}
		\htext(1.7 0.5){$0$}
		\htext(1.31 1.45){$0$}
		\htext(1.5 3){$1$}
		\htext(0.31 1.45){$0$}
		\esegment
		
		\esegment
		
		\move(0 -23)
		\bsegment
		
		\move(-8 0)
		\bsegment
		\move(0 0)\lvec(2 0)
		\move(0 0)\lvec(0 2)
		\move(0 2)\lvec(2 2)
		\move(2 0)\lvec(2 2)
		\move(1 0)\lvec(1 2)
		\move(0 0)\lvec(1 2)
		\move(1 0)\lvec(2 2)
		\move(0 2)\lvec(0 4)
		\move(2 2)\lvec(2 4)
		\move(0 4)\lvec(2 4)
		\move(1 2)\lvec(1 4)
		\htext(0.7 0.5){$0$}
		\htext(1.7 0.5){$0$}
		\htext(1.31 1.45){$0$}
		\htext(0.31 1.45){$0$}
		\move(-2 0)\lvec(0 0)
		\move(-2 0)\lvec(-2 2)
		\move(-2 2)\lvec(0 2)
		\move(-1 0)\lvec(-1 2)
		\move(-1 0)\lvec(0 2)
		\move(-2 0)\lvec(-1 2)
		\htext(-1.3 0.5){$0$}
		\htext(-0.3 0.5){$0$}
		\htext(-0.69 1.45){$0$}
		\htext(1.5 3){$1$}
		\esegment
		
		\move(0 0)
		\bsegment
		\move(0 0)\lvec(2 0)
		\move(0 0)\lvec(0 2)
		\move(0 2)\lvec(2 2)
		\move(2 0)\lvec(2 2)
		\move(1 0)\lvec(1 2)
		\move(0 0)\lvec(1 2)
		\move(1 0)\lvec(2 2)
		\move(0 2)\lvec(0 4)
		\move(2 2)\lvec(2 4)
		\move(0 4)\lvec(2 4)
		\move(1 2)\lvec(1 4)
		\htext(0.7 0.5){$0$}
		\htext(1.7 0.5){$0$}
		\htext(1.31 1.45){$0$}
		\htext(1.5 3){$1$}
		\htext(0.5 3){$1$}
		\htext(0.31 1.45){$0$}
		\esegment
		
		\move(10 0)
		\bsegment
		\move(0 0)\lvec(2 0)
		\move(0 0)\lvec(0 2)
		\move(0 2)\lvec(2 2)
		\move(2 0)\lvec(2 2)
		\move(1 0)\lvec(1 2)
		\move(0 0)\lvec(1 2)
		\move(1 0)\lvec(2 2)
		\move(0 2)\lvec(0 4)
		\move(2 2)\lvec(2 4)
		\move(0 4)\lvec(2 4)
		\move(1 2)\lvec(1 4)
		\move(0 4)\lvec(0 6)
		\move(2 4)\lvec(2 6)
		\move(0 6)\lvec(2 6)
		\move(1 4)\lvec(1 6)
		\htext(0.7 0.5){$0$}
		\htext(1.7 0.5){$0$}
		\htext(1.31 1.45){$0$}
		\htext(1.7 4.5){$0$}
		\move(1 4)\lvec(2 6)
		\move(0 4)\lvec(1 6)
		\move(-2 0)\lvec(0 0)
		\move(-2 0)\lvec(-2 2)
		\move(-2 2)\lvec(0 2)
		\move(-1 0)\lvec(-1 2)
		\move(-1 0)\lvec(0 2)
		\move(-2 0)\lvec(-1 2)
		\htext(-1.3 0.5){$0$}
		\htext(-0.3 0.5){$0$}
		\htext(-0.69 1.45){$0$}
		\htext(1.5 3){$1$}
		\esegment
		
		\esegment
		\move(0 -31)
		\bsegment
		
		\move(-12 0)
		\bsegment
		\move(0 0)\lvec(2 0)
		\move(0 0)\lvec(0 2)
		\move(0 2)\lvec(2 2)
		\move(2 0)\lvec(2 2)
		\move(1 0)\lvec(1 2)
		\move(0 0)\lvec(1 2)
		\move(1 0)\lvec(2 2)
		\move(0 2)\lvec(0 4)
		\move(2 2)\lvec(2 4)
		\move(0 4)\lvec(2 4)
		\move(1 2)\lvec(1 4)
		\htext(0.7 0.5){$0$}
		\htext(1.7 0.5){$0$}
		\htext(1.31 1.45){$0$}
		\htext(0.31 1.45){$0$}
		\move(-2 0)\lvec(0 0)
		\move(-2 0)\lvec(-2 2)
		\move(-2 2)\lvec(0 2)
		\move(-1 0)\lvec(-1 2)
		\move(-1 0)\lvec(0 2)
		\move(-2 0)\lvec(-1 2)
		\htext(-1.3 0.5){$0$}
		\htext(-0.3 0.5){$0$}
		\htext(-0.69 1.45){$0$}
		\htext(-1.69 1.45){$0$}
		\htext(1.5 3){$1$}
		\esegment
		
		\move(-3 0)
		\bsegment
		\move(0 0)\lvec(2 0)
		\move(0 0)\lvec(0 2)
		\move(0 2)\lvec(2 2)
		\move(2 0)\lvec(2 2)
		\move(1 0)\lvec(1 2)
		\move(0 0)\lvec(1 2)
		\move(1 0)\lvec(2 2)
		\move(0 2)\lvec(0 4)
		\move(2 2)\lvec(2 4)
		\move(0 4)\lvec(2 4)
		\move(1 2)\lvec(1 4)
		\htext(0.7 0.5){$0$}
		\htext(1.7 0.5){$0$}
		\htext(1.31 1.45){$0$}
		\htext(0.31 1.45){$0$}
		\move(-2 0)\lvec(0 0)
		\move(-2 0)\lvec(-2 2)
		\move(-2 2)\lvec(0 2)
		\move(-1 0)\lvec(-1 2)
		\move(-1 0)\lvec(0 2)
		\move(-2 0)\lvec(-1 2)
		\htext(-1.3 0.5){$0$}
		\htext(-0.3 0.5){$0$}
		\htext(-0.69 1.45){$0$}
		
		\htext(1.5 3){$1$}
		\htext(0.5 3){$1$}
		\esegment
		
		\move(6 0)
		\bsegment
		\move(0 0)\lvec(2 0)
		\move(0 0)\lvec(0 2)
		\move(0 2)\lvec(2 2)
		\move(2 0)\lvec(2 2)
		\move(1 0)\lvec(1 2)
		\move(0 0)\lvec(1 2)
		\move(1 0)\lvec(2 2)
		\move(0 2)\lvec(0 4)
		\move(2 2)\lvec(2 4)
		\move(0 4)\lvec(2 4)
		\move(1 2)\lvec(1 4)
		\move(0 4)\lvec(0 6)
		\move(2 4)\lvec(2 6)
		\move(0 6)\lvec(2 6)
		\move(1 4)\lvec(1 6)
		\htext(0.7 0.5){$0$}
		\htext(1.7 0.5){$0$}
		\htext(1.31 1.45){$0$}
		\htext(1.7 4.5){$0$}
		\move(1 4)\lvec(2 6)
		\move(0 4)\lvec(1 6)
		\move(-2 0)\lvec(0 0)
		\move(-2 0)\lvec(-2 2)
		\move(-2 2)\lvec(0 2)
		\move(-1 0)\lvec(-1 2)
		\move(-1 0)\lvec(0 2)
		\move(-2 0)\lvec(-1 2)
		\move(-1 2)\lvec(-1 4)
		\move(-2 2)\lvec(-2 4)
		\move(-2 4)\lvec(0 4)
		\htext(-1.3 0.5){$0$}
		\htext(-0.3 0.5){$0$}
		\htext(-0.69 1.45){$0$}
		\htext(1.5 3){$1$}
		\htext(-0.5 3){$1$}
		\esegment
		
		\move(15 0)
		\bsegment
		\move(0 0)\lvec(2 0)
		\move(0 0)\lvec(0 2)
		\move(0 2)\lvec(2 2)
		\move(2 0)\lvec(2 2)
		\move(1 0)\lvec(1 2)
		\move(0 0)\lvec(1 2)
		\move(1 0)\lvec(2 2)
		\move(0 2)\lvec(0 4)
		\move(2 2)\lvec(2 4)
		\move(0 4)\lvec(2 4)
		\move(1 2)\lvec(1 4)
		\move(0 4)\lvec(0 6)
		\move(2 4)\lvec(2 6)
		\move(0 6)\lvec(2 6)
		\move(1 4)\lvec(1 6)
		\htext(0.7 0.5){$0$}
		\htext(1.7 0.5){$0$}
		\htext(1.31 1.45){$0$}
		\htext(1.7 4.5){$0$}
		\move(1 4)\lvec(2 6)
		\move(0 4)\lvec(1 6)
		\move(-2 0)\lvec(0 0)
		\move(-2 0)\lvec(-2 2)
		\move(-2 2)\lvec(0 2)
		\move(-1 0)\lvec(-1 2)
		\move(-1 0)\lvec(0 2)
		\move(-2 0)\lvec(-1 2)
		\htext(-1.3 0.5){$0$}
		\htext(-0.3 0.5){$0$}
		\htext(-0.69 1.45){$0$}
		\htext(1.5 3){$1$}
		\htext(1.31 5.45){$0$}
		\esegment

		\esegment
		
		\move(0 -39)
		\bsegment
		
		\move(-16 0)
		\bsegment
		\move(0 0)\lvec(2 0)
		\move(0 0)\lvec(0 2)
		\move(0 2)\lvec(2 2)
		\move(2 0)\lvec(2 2)
		\move(1 0)\lvec(1 2)
		\move(0 0)\lvec(1 2)
		\move(1 0)\lvec(2 2)
		\move(0 2)\lvec(0 4)
		\move(2 2)\lvec(2 4)
		\move(0 4)\lvec(2 4)
		\move(1 2)\lvec(1 4)
		\move(-2 2)\lvec(-2 4)
		\move(-2 4)\lvec(0 4)
		\move(-1 2)\lvec(-1 4)
		\htext(0.7 0.5){$0$}
		\htext(1.7 0.5){$0$}
		\htext(1.31 1.45){$0$}
		\htext(0.31 1.45){$0$}
		\move(-2 0)\lvec(0 0)
		\move(-2 0)\lvec(-2 2)
		\move(-2 2)\lvec(0 2)
		\move(-1 0)\lvec(-1 2)
		\move(-1 0)\lvec(0 2)
		\move(-2 0)\lvec(-1 2)
		\htext(-1.3 0.5){$0$}
		\htext(-0.3 0.5){$0$}
		\htext(-0.69 1.45){$0$}
		\htext(-1.69 1.45){$0$}
		\htext(1.5 3){$1$}
		\htext(-0.5 3){$1$}
		\esegment
		
		\move(-8 0)
		\bsegment
		\move(0 0)\lvec(2 0)
		\move(0 0)\lvec(0 2)
		\move(0 2)\lvec(2 2)
		\move(2 0)\lvec(2 2)
		\move(1 0)\lvec(1 2)
		\move(0 0)\lvec(1 2)
		\move(1 0)\lvec(2 2)
		\move(0 2)\lvec(0 4)
		\move(2 2)\lvec(2 4)
		\move(0 4)\lvec(2 4)
		\move(1 2)\lvec(1 4)
		\htext(0.7 0.5){$0$}
		\htext(1.7 0.5){$0$}
		\htext(1.31 1.45){$0$}
		\htext(0.31 1.45){$0$}
		\move(-2 0)\lvec(0 0)
		\move(-2 0)\lvec(-2 2)
		\move(-2 2)\lvec(0 2)
		\move(-1 0)\lvec(-1 2)
		\move(-1 0)\lvec(0 2)
		\move(-2 0)\lvec(-1 2)
		\htext(-1.3 0.5){$0$}
		\htext(-0.3 0.5){$0$}
		\htext(-0.69 1.45){$0$}
		\htext(-1.69 1.45){$0$}
		\htext(1.5 3){$1$}
		\htext(0.5 3){$1$}
		\esegment
		
		\move(2 0)
		\bsegment
		\move(0 0)\lvec(2 0)
		\move(0 0)\lvec(0 2)
		\move(0 2)\lvec(2 2)
		\move(2 0)\lvec(2 2)
		\move(1 0)\lvec(1 2)
		\move(0 0)\lvec(1 2)
		\move(1 0)\lvec(2 2)
		\move(0 2)\lvec(0 4)
		\move(2 2)\lvec(2 4)
		\move(0 4)\lvec(2 4)
		\move(1 2)\lvec(1 4)
		\move(0 4)\lvec(0 6)
		\move(2 4)\lvec(2 6)
		\move(0 6)\lvec(2 6)
		\move(1 4)\lvec(1 6)
		\htext(0.7 0.5){$0$}
		\htext(1.7 0.5){$0$}
		\htext(1.31 1.45){$0$}
		\htext(1.7 4.5){$0$}
		\move(1 4)\lvec(2 6)
		\move(0 4)\lvec(1 6)
		\move(-2 0)\lvec(0 0)
		\move(-2 0)\lvec(-2 2)
		\move(-2 2)\lvec(0 2)
		\move(-1 0)\lvec(-1 2)
		\move(-1 0)\lvec(0 2)
		\move(-2 0)\lvec(-1 2)
		\move(-1 2)\lvec(-1 4)
		\move(-2 2)\lvec(-2 4)
		\move(-2 4)\lvec(0 4)
		\move(-4 0)\lvec(-2 0)
		\move(-4 0)\lvec(-4 2)
		\move(-4 2)\lvec(-2 2)
		\move(-3 0)\lvec(-3 2)
		\move(-4 0)\lvec(-3 2)
		\move(-3 0)\lvec(-2 2)
		\htext(-3.3 0.5){$0$}
		\htext(-2.3 0.5){$0$}
		\htext(-2.69 1.45){$0$}
		\htext(-1.3 0.5){$0$}
		\htext(-0.3 0.5){$0$}
		\htext(-0.69 1.45){$0$}
		\htext(1.5 3){$1$}
		\htext(-0.5 3){$1$}
		\esegment
		
		\move(10 0)
		\bsegment
		\move(0 0)\lvec(2 0)
		\move(0 0)\lvec(0 2)
		\move(0 2)\lvec(2 2)
		\move(2 0)\lvec(2 2)
		\move(1 0)\lvec(1 2)
		\move(0 0)\lvec(1 2)
		\move(1 0)\lvec(2 2)
		\move(0 2)\lvec(0 4)
		\move(2 2)\lvec(2 4)
		\move(0 4)\lvec(2 4)
		\move(1 2)\lvec(1 4)
		\move(0 4)\lvec(0 6)
		\move(2 4)\lvec(2 6)
		\move(0 6)\lvec(2 6)
		\move(1 4)\lvec(1 6)
		\htext(0.7 0.5){$0$}
		\htext(1.7 0.5){$0$}
		\htext(1.31 1.45){$0$}
		\htext(0.31 1.45){$0$}
		\htext(1.7 4.5){$0$}
		\move(1 4)\lvec(2 6)
		\move(0 4)\lvec(1 6)
		\move(-2 0)\lvec(0 0)
		\move(-2 0)\lvec(-2 2)
		\move(-2 2)\lvec(0 2)
		\move(-1 0)\lvec(-1 2)
		\move(-1 0)\lvec(0 2)
		\move(-2 0)\lvec(-1 2)
		\htext(-1.3 0.5){$0$}
		\htext(-0.3 0.5){$0$}
		\htext(-0.69 1.45){$0$}
		\htext(1.5 3){$1$}
		\htext(1.31 5.45){$0$}
		\esegment
		
		\move(18 0)
		\bsegment
		\move(0 0)\lvec(2 0)
		\move(0 0)\lvec(0 2)
		\move(0 2)\lvec(2 2)
		\move(2 0)\lvec(2 2)
		\move(1 0)\lvec(1 2)
		\move(0 0)\lvec(1 2)
		\move(1 0)\lvec(2 2)
		\move(0 2)\lvec(0 4)
		\move(2 2)\lvec(2 4)
		\move(0 4)\lvec(2 4)
		\move(1 2)\lvec(1 4)
		\move(0 4)\lvec(0 6)
		\move(2 4)\lvec(2 6)
		\move(0 6)\lvec(2 6)
		\move(1 4)\lvec(1 6)
		\htext(0.7 0.5){$0$}
		\htext(1.7 0.5){$0$}
		\htext(1.31 1.45){$0$}
		\htext(1.31 5.45){$0$}
		\htext(1.7 4.5){$0$}
		\move(1 4)\lvec(2 6)
		\move(0 4)\lvec(1 6)
		\move(-2 0)\lvec(0 0)
		\move(-2 0)\lvec(-2 2)
		\move(-2 2)\lvec(0 2)
		\move(-1 0)\lvec(-1 2)
		\move(-1 0)\lvec(0 2)
		\move(-2 0)\lvec(-1 2)
		\move(-1 2)\lvec(-1 4)
		\move(-2 2)\lvec(-2 4)
		\move(-2 4)\lvec(0 4)
		\htext(-1.3 0.5){$0$}
		\htext(-0.3 0.5){$0$}
		\htext(-0.69 1.45){$0$}
		\htext(1.5 3){$1$}
		\htext(-0.5 3){$1$}
		\esegment
		
		\esegment
		\move(0 -48)
		\bsegment
		
		\move(-18 0)
		\bsegment
		\move(0 0)\lvec(2 0)
		\move(0 0)\lvec(0 2)
		\move(0 2)\lvec(2 2)
		\move(2 0)\lvec(2 2)
		\move(1 0)\lvec(1 2)
		\move(0 0)\lvec(1 2)
		\move(1 0)\lvec(2 2)
		\move(0 2)\lvec(0 4)
		\move(2 2)\lvec(2 4)
		\move(0 4)\lvec(2 4)
		\move(1 2)\lvec(1 4)
		\move(-2 2)\lvec(-2 4)
		\move(-2 4)\lvec(0 4)
		\move(-1 2)\lvec(-1 4)
		\htext(0.7 0.5){$0$}
		\htext(1.7 0.5){$0$}
		\htext(1.31 1.45){$0$}
		\htext(0.31 1.45){$0$}
		\move(-2 0)\lvec(0 0)
		\move(-2 0)\lvec(-2 2)
		\move(-2 2)\lvec(0 2)
		\move(-1 0)\lvec(-1 2)
		\move(-1 0)\lvec(0 2)
		\move(-2 0)\lvec(-1 2)
		\htext(-1.3 0.5){$0$}
		\htext(-0.3 0.5){$0$}
		\htext(-0.69 1.45){$0$}
		\htext(-1.69 1.45){$0$}
		\htext(1.5 3){$1$}
		\htext(-0.5 3){$1$}
		\move(-4 0)\lvec(-2 0)
		\move(-4 0)\lvec(-4 2)
		\move(-4 2)\lvec(0 2)
		\move(-3 0)\lvec(-3 2)
		\move(-4 0)\lvec(-3 2)
		\move(-3 0)\lvec(-2 2)
		\htext(-3.3 0.5){$0$}
		\htext(-2.3 0.5){$0$}
		\htext(-2.69 1.45){$0$}
		
		\vtext(0 -1){\fontsize{8}{8}\selectfont$\cdots$}
		\esegment
		
		\move(-12 0)
		\bsegment
		\move(0 0)\lvec(2 0)
		\move(0 0)\lvec(0 2)
		\move(0 2)\lvec(2 2)
		\move(2 0)\lvec(2 2)
		\move(1 0)\lvec(1 2)
		\move(0 0)\lvec(1 2)
		\move(1 0)\lvec(2 2)
		\move(0 2)\lvec(0 4)
		\move(2 2)\lvec(2 4)
		\move(0 4)\lvec(2 4)
		\move(1 2)\lvec(1 4)
		\move(-2 2)\lvec(-2 4)
		\move(-2 4)\lvec(0 4)
		\move(-1 2)\lvec(-1 4)
		
		\htext(0.7 0.5){$0$}
		\htext(1.7 0.5){$0$}
		\htext(1.31 1.45){$0$}
		\htext(0.31 1.45){$0$}
		\move(-2 0)\lvec(0 0)
		\move(-2 0)\lvec(-2 2)
		\move(-2 2)\lvec(0 2)
		\move(-1 0)\lvec(-1 2)
		\move(-1 0)\lvec(0 2)
		\move(-2 0)\lvec(-1 2)
		\htext(-1.3 0.5){$0$}
		\htext(-0.3 0.5){$0$}
		\htext(-0.69 1.45){$0$}
		\htext(-1.69 1.45){$0$}
		\htext(1.5 3){$1$}
		\htext(0.5 3){$1$}
		\htext(-0.5 3){$1$}
		
		\vtext(0 -1){\fontsize{8}{8}\selectfont$\cdots$}
		\esegment
		
		\move(-6 0)
		\bsegment
		\move(0 0)\lvec(2 0)
		\move(0 0)\lvec(0 2)
		\move(0 2)\lvec(2 2)
		\move(2 0)\lvec(2 2)
		\move(1 0)\lvec(1 2)
		\move(0 0)\lvec(1 2)
		\move(1 0)\lvec(2 2)
		\move(0 2)\lvec(0 4)
		\move(2 2)\lvec(2 4)
		\move(0 4)\lvec(2 4)
		\move(1 2)\lvec(1 4)
		\move(0 4)\lvec(0 6)
		\move(2 4)\lvec(2 6)
		\move(0 6)\lvec(2 6)
		\move(1 4)\lvec(1 6)
		\htext(0.7 0.5){$0$}
		\htext(1.7 0.5){$0$}
		\htext(1.31 1.45){$0$}
		\htext(0.31 1.45){$0$}
		\htext(1.7 4.5){$0$}
		\move(1 4)\lvec(2 6)
		\move(0 4)\lvec(1 6)
		\move(-2 0)\lvec(0 0)
		\move(-2 0)\lvec(-2 2)
		\move(-2 2)\lvec(0 2)
		\move(-1 0)\lvec(-1 2)
		\move(-1 0)\lvec(0 2)
		\move(-2 0)\lvec(-1 2)
		\htext(-1.3 0.5){$0$}
		\htext(-0.3 0.5){$0$}
		\htext(-0.69 1.45){$0$}
		\htext(-1.69 1.45){$0$}
		\htext(1.5 3){$1$}
		\htext(0.5 3){$1$}
		
		\vtext(0 -1){\fontsize{8}{8}\selectfont$\cdots$}
		\esegment
		
		\move(2 0)
		\bsegment
		
		\move(0 0)\lvec(2 0)
		\move(0 0)\lvec(0 2)
		\move(0 2)\lvec(2 2)
		\move(2 0)\lvec(2 2)
		\move(1 0)\lvec(1 2)
		\move(0 0)\lvec(1 2)
		\move(1 0)\lvec(2 2)
		\move(0 2)\lvec(0 4)
		\move(2 2)\lvec(2 4)
		\move(0 4)\lvec(2 4)
		\move(1 2)\lvec(1 4)
		\move(0 4)\lvec(0 6)
		\move(2 4)\lvec(2 6)
		\move(0 6)\lvec(2 6)
		\move(1 4)\lvec(1 6)
		\htext(0.7 0.5){$0$}
		\htext(1.7 0.5){$0$}
		\htext(1.31 1.45){$0$}
		\htext(1.7 4.5){$0$}
		\htext(-0.3 4.5){$0$}
		\move(-2 4)\lvec(-2 6)
		\move(-2 6)\lvec(0 6)
		\move(-1 4)\lvec(-1 6)
		\move(-1 4)\lvec(0 6)
		\move(-2 4)\lvec(-1 6)
		\move(1 4)\lvec(2 6)
		\move(0 4)\lvec(1 6)
		\move(-2 0)\lvec(0 0)
		\move(-2 0)\lvec(-2 2)
		\move(-2 2)\lvec(0 2)
		\move(-1 0)\lvec(-1 2)
		\move(-1 0)\lvec(0 2)
		\move(-2 0)\lvec(-1 2)
		\move(-1 2)\lvec(-1 4)
		\move(-2 2)\lvec(-2 4)
		\move(-2 4)\lvec(0 4)
		\move(-4 0)\lvec(-2 0)
		\move(-4 0)\lvec(-4 2)
		\move(-4 2)\lvec(-2 2)
		\move(-3 0)\lvec(-3 2)
		\move(-4 0)\lvec(-3 2)
		\move(-3 0)\lvec(-2 2)
		\htext(-3.3 0.5){$0$}
		\htext(-2.3 0.5){$0$}
		\htext(-2.69 1.45){$0$}
		\htext(-1.3 0.5){$0$}
		\htext(-0.3 0.5){$0$}
		\htext(-0.69 1.45){$0$}
		\htext(1.5 3){$1$}
		\htext(-0.5 3){$1$}

		\vtext(0 -1){\fontsize{8}{8}\selectfont$\cdots$}
		
		\esegment
		
		\move(8 0)
		\bsegment
		\move(0 0)\lvec(2 0)
		\move(0 0)\lvec(0 2)
		\move(0 2)\lvec(2 2)
		\move(2 0)\lvec(2 2)
		\move(1 0)\lvec(1 2)
		\move(0 0)\lvec(1 2)
		\move(1 0)\lvec(2 2)
		\move(0 2)\lvec(0 4)
		\move(2 2)\lvec(2 4)
		\move(0 4)\lvec(2 4)
		\move(1 2)\lvec(1 4)
		\move(0 4)\lvec(0 6)
		\move(2 4)\lvec(2 6)
		\move(0 6)\lvec(2 6)
		\move(1 4)\lvec(1 6)
		\htext(0.7 0.5){$0$}
		\htext(1.7 0.5){$0$}
		\htext(1.31 1.45){$0$}
		\htext(0.31 1.45){$0$}
		\htext(1.31 5.45){$0$}
		\htext(1.7 4.5){$0$}
		\move(1 4)\lvec(2 6)
		\move(0 4)\lvec(1 6)
		\move(-2 0)\lvec(0 0)
		\move(-2 0)\lvec(-2 2)
		\move(-2 2)\lvec(0 2)
		\move(-1 0)\lvec(-1 2)
		\move(-1 0)\lvec(0 2)
		\move(-2 0)\lvec(-1 2)
		\move(-1 2)\lvec(-1 4)
		\move(-2 2)\lvec(-2 4)
		\move(-2 4)\lvec(0 4)
		\htext(-1.3 0.5){$0$}
		\htext(-0.3 0.5){$0$}
		\htext(-0.69 1.45){$0$}
		\htext(1.5 3){$1$}
		\htext(-0.5 3){$1$}
		
		\vtext(0 -1){\fontsize{8}{8}\selectfont$\cdots$}
		\esegment
		
		\move(16 0)
		\bsegment
		\move(0 0)\lvec(2 0)
		\move(0 0)\lvec(0 2)
		\move(0 2)\lvec(2 2)
		\move(2 0)\lvec(2 2)
		\move(1 0)\lvec(1 2)
		\move(0 0)\lvec(1 2)
		\move(1 0)\lvec(2 2)
		\move(0 2)\lvec(0 4)
		\move(2 2)\lvec(2 4)
		\move(0 4)\lvec(2 4)
		\move(1 2)\lvec(1 4)
		\move(0 4)\lvec(0 6)
		\move(2 4)\lvec(2 6)
		\move(0 6)\lvec(2 6)
		\move(1 4)\lvec(1 6)
		\htext(0.7 0.5){$0$}
		\htext(1.7 0.5){$0$}
		\htext(1.31 1.45){$0$}
		\htext(1.31 5.45){$0$}
		\htext(1.7 4.5){$0$}
		\move(1 4)\lvec(2 6)
		\move(0 4)\lvec(1 6)
		\move(-2 0)\lvec(0 0)
		\move(-2 0)\lvec(-2 2)
		\move(-2 2)\lvec(0 2)
		\move(-1 0)\lvec(-1 2)
		\move(-1 0)\lvec(0 2)
		\move(-2 0)\lvec(-1 2)
		\move(-1 2)\lvec(-1 4)
		\move(-2 2)\lvec(-2 4)
		\move(-2 4)\lvec(0 4)
		\move(-4 0)\lvec(-2 0)
		\move(-4 0)\lvec(-4 2)
		\move(-4 2)\lvec(-2 2)
		\move(-3 0)\lvec(-3 2)
		\move(-4 0)\lvec(-3 2)
		\move(-3 0)\lvec(-2 2)
		\htext(-3.3 0.5){$0$}
		\htext(-2.3 0.5){$0$}
		\htext(-2.69 1.45){$0$}
		\htext(-1.3 0.5){$0$}
		\htext(-0.3 0.5){$0$}
		\htext(-0.69 1.45){$0$}
		\htext(1.5 3){$1$}
		\htext(-0.5 3){$1$}

		\vtext(0 -1){\fontsize{8}{8}\selectfont$\cdots$}
		\esegment

		\move(22 0)
		\bsegment
		\move(0 0)\lvec(2 0)
		\move(0 0)\lvec(0 2)
		\move(0 2)\lvec(2 2)
		\move(2 0)\lvec(2 2)
		\move(1 0)\lvec(1 2)
		\move(0 0)\lvec(1 2)
		\move(1 0)\lvec(2 2)
		\move(0 2)\lvec(0 4)
		\move(2 2)\lvec(2 4)
		\move(0 4)\lvec(2 4)
		\move(1 2)\lvec(1 4)
		\move(0 4)\lvec(0 6)
		\move(2 4)\lvec(2 6)
		\move(0 6)\lvec(2 6)
		\move(1 4)\lvec(1 6)
		\htext(0.7 0.5){$0$}
		\htext(1.7 0.5){$0$}
		\htext(1.31 1.45){$0$}
		\htext(1.31 5.45){$0$}
		\htext(1.7 4.5){$0$}
		\move(1 4)\lvec(2 6)
		\move(0 4)\lvec(1 6)
		\move(-2 0)\lvec(0 0)
		\move(-2 0)\lvec(-2 2)
		\move(-2 2)\lvec(0 2)
		\move(-1 0)\lvec(-1 2)
		\move(-1 0)\lvec(0 2)
		\move(-2 0)\lvec(-1 2)
		\move(-1 2)\lvec(-1 4)
		\move(-2 2)\lvec(-2 4)
		\move(-2 4)\lvec(0 4)
		\move(0 6)\lvec(0 8)
		\move(1 6)\lvec(1 8)
		\move(2 6)\lvec(2 8)
		\move(0 8)\lvec(2 8)
		\htext(1.5 7){$1$}
		\htext(-1.3 0.5){$0$}
		\htext(-0.3 0.5){$0$}
		\htext(-0.69 1.45){$0$}
		\htext(1.5 3){$1$}
		\htext(-0.5 3){$1$}

		\vtext(0 -1){\fontsize{8}{8}\selectfont$\cdots$}
		\esegment
		\esegment
		\move(1 4.5)\avec(1 1.5)\htext(1.3 3){$0$}
		\move(0.5 -1.5)\avec(-2.5 -5)\htext(-0.3 -3){$0$}
		\move(1.5 -1.5)\avec(4.5 -5)\htext(3.3 -3){$1$}
		\move(-4 -7.5)\avec(-4 -10.5)\htext(-3.7 -9){$1$}
		\move(6 -7.5)\avec(6 -10.5)\htext(6.3 -9){$0$}
		\move(-4.5 -15.5)\avec(-7 -18.5)\htext(-5.3 -17){$0$}
		\move(-3.5 -15.5)\avec(1 -18.5)\htext(-0.5 -17){$1$}
		\move(5 -15.5)\avec(9.5 -19)\htext(7.5 -17){$0$}
		\move(-8.5 -23.5)\avec(-11 -26.5)\htext(-9.3 -25){$0$}
		\move(-7.5 -23.5)\avec(-3.5 -26.5)\htext(-4.9 -25){$1$}
		\move(1 -23.5)\avec(-2 -26.5)\htext(0 -25){$0$}
		\move(9.5 -23.5)\avec(8.5 -27)\htext(9.4 -25){$1$}
		\move(10.5 -23.5)\avec(14.5 -27)\htext(12.8 -25){$0$}
		\move(-12 -31.5)\avec(-16 -34.5)\htext(-13.4 -33){$1$}
		\move(-3 -31.5)\avec(-7 -34.5)\htext(-4.4 -33){$0$}
		\move(6 -31.5)\avec(4.5 -35)\htext(5.7 -33){$0$}
		\move(14.5 -31.5)\avec(12.5 -35)\htext(14 -33){$0$}
		\move(15.5 -31.5)\avec(17.5 -34.5)\htext(16.9 -33){$1$}
		\move(-16.5 -39.5)\avec(-18 -43.5)\htext(-16.9 -41.5){$0$}
		\move(-15.5 -39.5)\avec(-12 -43.5)\htext(-13.3 -41.5){$1$}
		\move(-8 -39.5)\avec(-6.5 -44)\htext(-7 -41.5){$0$}
		\move(2 -39.5)\avec(2 -41.5)\htext(2.3 -40.5){$0$}
		\move(10 -39.5)\avec(9 -41.5)\htext(9.8 -40.5){$1$}
		\move(17.5 -39.5)\avec(17 -41.5)\htext(17.6 -40.5){$0$}
		\move(18.5 -39.5)\avec(21.5 -42)\htext(20.3 -40.5){$1$}
	\end{texdraw}
\end{center}

\end{example}

\clearpage
\begin{example}
	The top part of the  crystal $\mathbf{Y}(2\Lambda_0)$ for $D_{3}^{(2)}$ is given below.	

\vskip 15mm

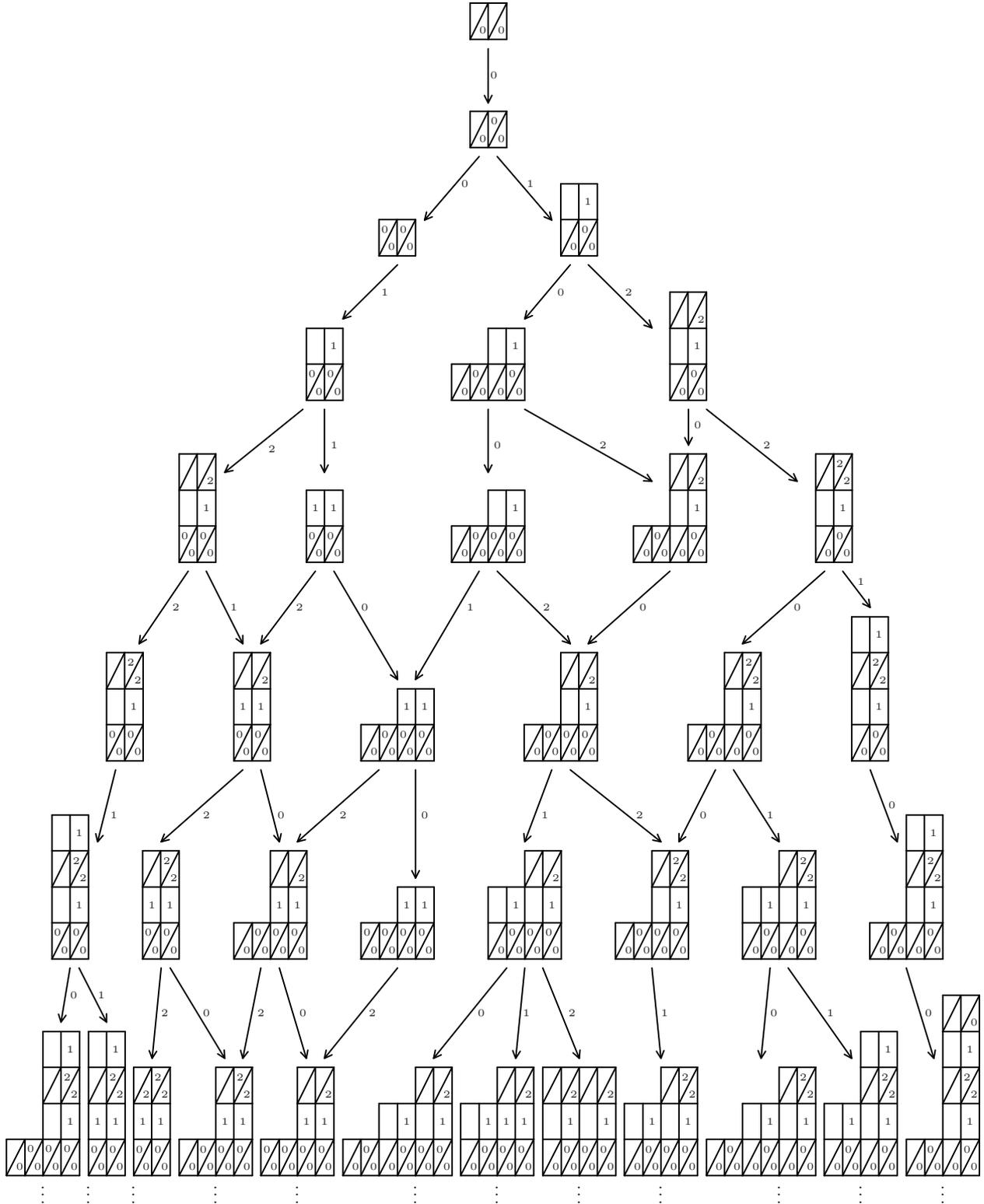
	\begin{center}
	\begin{texdraw}
		\drawdim em
		\arrowheadsize l:0.3 w:0.3
		\arrowheadtype t:V
		\fontsize{4}{4}\selectfont
		\drawdim em
		\setunitscale 1.45
		\move(0 5)
		\bsegment
		\move(3 0)
		\bsegment
		\move(0 0)\lvec(2 0)
		\move(0 0)\lvec(0 2)
		\move(0 2)\lvec(2 2)
		\move(2 0)\lvec(2 2)
		\move(1 0)\lvec(1 2)
		\move(0 0)\lvec(1 2)
		\move(1 0)\lvec(2 2)
		\htext(0.7 0.5){$0$}
		\htext(1.7 0.5){$0$}
		\esegment
		\esegment
		\move(0 -1)
		\bsegment
		\move(3 0)
		\bsegment	
		\move(0 0)\lvec(2 0)
		\move(0 0)\lvec(0 2)
		\move(0 2)\lvec(2 2)
		\move(2 0)\lvec(2 2)
		\move(1 0)\lvec(1 2)
		\move(0 0)\lvec(1 2)
		\move(1 0)\lvec(2 2)
		\htext(0.7 0.5){$0$}
		\htext(1.7 0.5){$0$}
		\htext(1.31 1.45){$0$}
		\esegment
		\esegment
		\move(0 -7)
		\bsegment
		\move(8 0)
		\bsegment
		\move(0 0)\lvec(2 0)
		\move(0 0)\lvec(0 2)
		\move(0 2)\lvec(2 2)
		\move(2 0)\lvec(2 2)
		\move(1 0)\lvec(1 2)
		\move(0 0)\lvec(1 2)
		\move(1 0)\lvec(2 2)
		\htext(0.7 0.5){$0$}
		\htext(1.7 0.5){$0$}
		\htext(1.31 1.45){$0$}
		\move(0 2)\lvec(0 4)
		\move(2 2)\lvec(2 4)
		\move(0 4)\lvec(2 4)
		\move(1 2)\lvec(1 4)
		\htext(1.5 3){$1$}
		\esegment
		
		\move(-2 0)
		\bsegment
		\move(0 0)\lvec(2 0)
		\move(0 0)\lvec(0 2)
		\move(0 2)\lvec(2 2)
		\move(2 0)\lvec(2 2)
		\move(1 0)\lvec(1 2)
		\move(0 0)\lvec(1 2)
		\move(1 0)\lvec(2 2)
		\htext(0.7 0.5){$0$}
		\htext(1.7 0.5){$0$}
		\htext(1.31 1.45){$0$}
		\htext(0.31 1.45){$0$}
		\esegment
		
		\esegment
		\move(0 -15)
		\bsegment
		\move(-6 0)
		\bsegment
		\move(0 0)\lvec(2 0)
		\move(0 0)\lvec(0 2)
		\move(0 2)\lvec(2 2)
		\move(2 0)\lvec(2 2)
		\move(1 0)\lvec(1 2)
		\move(0 0)\lvec(1 2)
		\move(1 0)\lvec(2 2)
		\move(0 2)\lvec(0 4)
		\move(2 2)\lvec(2 4)
		\move(0 4)\lvec(2 4)
		\move(1 2)\lvec(1 4)
		\htext(0.7 0.5){$0$}
		\htext(1.7 0.5){$0$}
		\htext(1.31 1.45){$0$}
		\htext(0.31 1.45){$0$}
		
		\htext(1.5 3){$1$}
		\esegment
		
		\move(4 0)
		\bsegment
		\move(0 0)\lvec(2 0)
		\move(0 0)\lvec(0 2)
		\move(0 2)\lvec(2 2)
		\move(2 0)\lvec(2 2)
		\move(1 0)\lvec(1 2)
		\move(0 0)\lvec(1 2)
		\move(1 0)\lvec(2 2)
		\move(0 2)\lvec(0 4)
		\move(2 2)\lvec(2 4)
		\move(0 4)\lvec(2 4)
		\move(1 2)\lvec(1 4)
		\htext(0.7 0.5){$0$}
		\htext(1.7 0.5){$0$}
		\htext(1.31 1.45){$0$}
		
		\move(-2 0)\lvec(0 0)
		\move(-2 0)\lvec(-2 2)
		\move(-2 2)\lvec(0 2)
		\move(-1 0)\lvec(-1 2)
		\move(-1 0)\lvec(0 2)
		\move(-2 0)\lvec(-1 2)
		\htext(-1.3 0.5){$0$}
		\htext(-0.3 0.5){$0$}
		\htext(-0.69 1.45){$0$}
		\htext(1.5 3){$1$}
		\esegment
		
		\move(14 0)
		\bsegment
		\move(0 0)\lvec(2 0)
		\move(0 0)\lvec(0 2)
		\move(0 2)\lvec(2 2)
		\move(2 0)\lvec(2 2)
		\move(1 0)\lvec(1 2)
		\move(0 0)\lvec(1 2)
		\move(1 0)\lvec(2 2)
		\move(0 2)\lvec(0 4)
		\move(2 2)\lvec(2 4)
		\move(0 4)\lvec(2 4)
		\move(1 2)\lvec(1 4)
		\move(0 4)\lvec(0 6)
		\move(2 4)\lvec(2 6)
		\move(0 6)\lvec(2 6)
		\move(1 4)\lvec(1 6)
		\htext(0.7 0.5){$0$}
		\htext(1.7 0.5){$0$}
		\htext(1.31 1.45){$0$}
		\htext(1.7 4.5){$2$}
		\move(1 4)\lvec(2 6)
		\move(0 4)\lvec(1 6)

		\htext(1.5 3){$1$}
		\esegment
		
		\esegment
		
		\move(0 -24)
		\bsegment
		
		\move(-13 0)
		\bsegment
		\move(0 0)\lvec(2 0)
		\move(0 0)\lvec(0 2)
		\move(0 2)\lvec(2 2)
		\move(2 0)\lvec(2 2)
		\move(1 0)\lvec(1 2)
		\move(0 0)\lvec(1 2)
		\move(1 0)\lvec(2 2)
		\move(0 2)\lvec(0 4)
		\move(2 2)\lvec(2 4)
		\move(0 4)\lvec(2 4)
		\move(1 2)\lvec(1 4)
		\move(0 4)\lvec(0 6)
		\move(2 4)\lvec(2 6)
		\move(0 6)\lvec(2 6)
		\move(1 4)\lvec(1 6)
		\htext(0.7 0.5){$0$}
		\htext(1.7 0.5){$0$}
		\htext(1.31 1.45){$0$}
		\htext(0.31 1.45){$0$}
		\htext(1.7 4.5){$2$}
		\move(1 4)\lvec(2 6)
		\move(0 4)\lvec(1 6)
		
		\htext(1.5 3){$1$}
		
		\esegment
		
		\move(-6 0)
		\bsegment
		\move(0 0)\lvec(2 0)
		\move(0 0)\lvec(0 2)
		\move(0 2)\lvec(2 2)
		\move(2 0)\lvec(2 2)
		\move(1 0)\lvec(1 2)
		\move(0 0)\lvec(1 2)
		\move(1 0)\lvec(2 2)
		\move(0 2)\lvec(0 4)
		\move(2 2)\lvec(2 4)
		\move(0 4)\lvec(2 4)
		\move(1 2)\lvec(1 4)
		\htext(0.7 0.5){$0$}
		\htext(1.7 0.5){$0$}
		\htext(1.31 1.45){$0$}
		\htext(0.31 1.45){$0$}
		
		\htext(1.5 3){$1$}
		\htext(0.5 3){$1$}
		\esegment
		
		\move(4 0)
		\bsegment
		\move(0 0)\lvec(2 0)
		\move(0 0)\lvec(0 2)
		\move(0 2)\lvec(2 2)
		\move(2 0)\lvec(2 2)
		\move(1 0)\lvec(1 2)
		\move(0 0)\lvec(1 2)
		\move(1 0)\lvec(2 2)
		\move(0 2)\lvec(0 4)
		\move(2 2)\lvec(2 4)
		\move(0 4)\lvec(2 4)
		\move(1 2)\lvec(1 4)
		\htext(0.7 0.5){$0$}
		\htext(1.7 0.5){$0$}
		\htext(1.31 1.45){$0$}
		\htext(0.31 1.45){$0$}
		\move(-2 0)\lvec(0 0)
		\move(-2 0)\lvec(-2 2)
		\move(-2 2)\lvec(0 2)
		\move(-1 0)\lvec(-1 2)
		\move(-1 0)\lvec(0 2)
		\move(-2 0)\lvec(-1 2)
		\htext(-1.3 0.5){$0$}
		\htext(-0.3 0.5){$0$}
		\htext(-0.69 1.45){$0$}
		
		\htext(1.5 3){$1$}
		
		\esegment
		
		\move(14 0)
		\bsegment
		\move(0 0)\lvec(2 0)
		\move(0 0)\lvec(0 2)
		\move(0 2)\lvec(2 2)
		\move(2 0)\lvec(2 2)
		\move(1 0)\lvec(1 2)
		\move(0 0)\lvec(1 2)
		\move(1 0)\lvec(2 2)
		\move(0 2)\lvec(0 4)
		\move(2 2)\lvec(2 4)
		\move(0 4)\lvec(2 4)
		\move(1 2)\lvec(1 4)
		\move(0 4)\lvec(0 6)
		\move(2 4)\lvec(2 6)
		\move(0 6)\lvec(2 6)
		\move(1 4)\lvec(1 6)
		\htext(0.7 0.5){$0$}
		\htext(1.7 0.5){$0$}
		\htext(1.31 1.45){$0$}
		
		\htext(1.7 4.5){$2$}
		\move(1 4)\lvec(2 6)
		\move(0 4)\lvec(1 6)
		\move(-2 0)\lvec(0 0)
		\move(-2 0)\lvec(-2 2)
		\move(-2 2)\lvec(0 2)
		\move(-1 0)\lvec(-1 2)
		\move(-1 0)\lvec(0 2)
		\move(-2 0)\lvec(-1 2)
		\htext(-1.3 0.5){$0$}
		\htext(-0.3 0.5){$0$}
		\htext(-0.69 1.45){$0$}
		\htext(1.5 3){$1$}
		
		\esegment
		
		\move(22 0)
		\bsegment
		\move(0 0)\lvec(2 0)
		\move(0 0)\lvec(0 2)
		\move(0 2)\lvec(2 2)
		\move(2 0)\lvec(2 2)
		\move(1 0)\lvec(1 2)
		\move(0 0)\lvec(1 2)
		\move(1 0)\lvec(2 2)
		\move(0 2)\lvec(0 4)
		\move(2 2)\lvec(2 4)
		\move(0 4)\lvec(2 4)
		\move(1 2)\lvec(1 4)
		\move(0 4)\lvec(0 6)
		\move(2 4)\lvec(2 6)
		\move(0 6)\lvec(2 6)
		\move(1 4)\lvec(1 6)
		\htext(0.7 0.5){$0$}
		\htext(1.7 0.5){$0$}
		\htext(1.31 1.45){$0$}
		\htext(1.31 5.45){$2$}
		\htext(1.7 4.5){$2$}
		\move(1 4)\lvec(2 6)
		\move(0 4)\lvec(1 6)
		\htext(1.5 3){$1$}
		
		\esegment

		\esegment
		\move(0 -35)
		\bsegment
		
		\move(-17 0)
		\bsegment
		\move(0 0)\lvec(2 0)
		\move(0 0)\lvec(0 2)
		\move(0 2)\lvec(2 2)
		\move(2 0)\lvec(2 2)
		\move(1 0)\lvec(1 2)
		\move(0 0)\lvec(1 2)
		\move(1 0)\lvec(2 2)
		\move(0 2)\lvec(0 4)
		\move(2 2)\lvec(2 4)
		\move(0 4)\lvec(2 4)
		\move(1 2)\lvec(1 4)
		\move(0 4)\lvec(0 6)
		\move(2 4)\lvec(2 6)
		\move(0 6)\lvec(2 6)
		\move(1 4)\lvec(1 6)
		\htext(0.7 0.5){$0$}
		\htext(1.7 0.5){$0$}
		\htext(1.31 1.45){$0$}
		\htext(0.31 1.45){$0$}
		\htext(1.31 5.45){$2$}
		\htext(1.7 4.5){$2$}
		\move(1 4)\lvec(2 6)
		\move(0 4)\lvec(1 6)
		
		\htext(1.5 3){$1$}
		
		\esegment
		
		\move(-10 0)
		\bsegment
		\move(0 0)\lvec(2 0)
		\move(0 0)\lvec(0 2)
		\move(0 2)\lvec(2 2)
		\move(2 0)\lvec(2 2)
		\move(1 0)\lvec(1 2)
		\move(0 0)\lvec(1 2)
		\move(1 0)\lvec(2 2)
		\move(0 2)\lvec(0 4)
		\move(2 2)\lvec(2 4)
		\move(0 4)\lvec(2 4)
		\move(1 2)\lvec(1 4)
		\move(0 4)\lvec(0 6)
		\move(2 4)\lvec(2 6)
		\move(0 6)\lvec(2 6)
		\move(1 4)\lvec(1 6)
		\htext(0.7 0.5){$0$}
		\htext(1.7 0.5){$0$}
		\htext(1.31 1.45){$0$}
		\htext(0.31 1.45){$0$}
		
		\htext(1.7 4.5){$2$}
		\move(1 4)\lvec(2 6)
		\move(0 4)\lvec(1 6)
		\htext(0.5 3){$1$}
		\htext(1.5 3){$1$}
		\esegment
		
		\move(-1 0)
		\bsegment
		\move(0 0)\lvec(2 0)
		\move(0 0)\lvec(0 2)
		\move(0 2)\lvec(2 2)
		\move(2 0)\lvec(2 2)
		\move(1 0)\lvec(1 2)
		\move(0 0)\lvec(1 2)
		\move(1 0)\lvec(2 2)
		\move(0 2)\lvec(0 4)
		\move(2 2)\lvec(2 4)
		\move(0 4)\lvec(2 4)
		\move(1 2)\lvec(1 4)
		\htext(0.7 0.5){$0$}
		\htext(1.7 0.5){$0$}
		\htext(1.31 1.45){$0$}
		\htext(0.31 1.45){$0$}
		\move(-2 0)\lvec(0 0)
		\move(-2 0)\lvec(-2 2)
		\move(-2 2)\lvec(0 2)
		\move(-1 0)\lvec(-1 2)
		\move(-1 0)\lvec(0 2)
		\move(-2 0)\lvec(-1 2)
		\htext(-1.3 0.5){$0$}
		\htext(-0.3 0.5){$0$}
		\htext(-0.69 1.45){$0$}
		\htext(0.5 3){$1$}
		\htext(1.5 3){$1$}
		
		\esegment
		
		\move(8 0)
		\bsegment
		\move(0 0)\lvec(2 0)
		\move(0 0)\lvec(0 2)
		\move(0 2)\lvec(2 2)
		\move(2 0)\lvec(2 2)
		\move(1 0)\lvec(1 2)
		\move(0 0)\lvec(1 2)
		\move(1 0)\lvec(2 2)
		\move(0 2)\lvec(0 4)
		\move(2 2)\lvec(2 4)
		\move(0 4)\lvec(2 4)
		\move(1 2)\lvec(1 4)
		\move(0 4)\lvec(0 6)
		\move(2 4)\lvec(2 6)
		\move(0 6)\lvec(2 6)
		\move(1 4)\lvec(1 6)
		\htext(0.7 0.5){$0$}
		\htext(1.7 0.5){$0$}
		\htext(1.31 1.45){$0$}
		\htext(0.31 1.45){$0$}
		\htext(1.7 4.5){$2$}
		\move(1 4)\lvec(2 6)
		\move(0 4)\lvec(1 6)
		\move(-2 0)\lvec(0 0)
		\move(-2 0)\lvec(-2 2)
		\move(-2 2)\lvec(0 2)
		\move(-1 0)\lvec(-1 2)
		\move(-1 0)\lvec(0 2)
		\move(-2 0)\lvec(-1 2)
		\htext(-1.3 0.5){$0$}
		\htext(-0.3 0.5){$0$}
		\htext(-0.69 1.45){$0$}
		\htext(1.5 3){$1$}
		
		\esegment
		
		\move(17 0)
		\bsegment
		\move(0 0)\lvec(2 0)
		\move(0 0)\lvec(2 0)
		\move(0 0)\lvec(0 2)
		\move(0 2)\lvec(2 2)
		\move(2 0)\lvec(2 2)
		\move(1 0)\lvec(1 2)
		\move(0 0)\lvec(1 2)
		\move(1 0)\lvec(2 2)
		\move(0 2)\lvec(0 4)
		\move(2 2)\lvec(2 4)
		\move(0 4)\lvec(2 4)
		\move(1 2)\lvec(1 4)
		\move(0 4)\lvec(0 6)
		\move(2 4)\lvec(2 6)
		\move(0 6)\lvec(2 6)
		\move(1 4)\lvec(1 6)
		\htext(0.7 0.5){$0$}
		\htext(1.7 0.5){$0$}
		\htext(1.31 1.45){$0$}
		\htext(1.31 5.45){$2$}
		\htext(1.7 4.5){$2$}
		\move(1 4)\lvec(2 6)
		\move(0 4)\lvec(1 6)
		\move(-2 0)\lvec(0 0)
		\move(-2 0)\lvec(-2 2)
		\move(-2 2)\lvec(0 2)
		\move(-1 0)\lvec(-1 2)
		\move(-1 0)\lvec(0 2)
		\move(-2 0)\lvec(-1 2)
		\htext(-1.3 0.5){$0$}
		\htext(-0.3 0.5){$0$}
		\htext(-0.69 1.45){$0$}
		\htext(1.5 3){$1$}
		\esegment	
		
		\move(24 0)
		\bsegment
		\move(0 0)\lvec(2 0)
		\move(0 0)\lvec(0 2)
		\move(0 2)\lvec(2 2)
		\move(2 0)\lvec(2 2)
		\move(1 0)\lvec(1 2)
		\move(0 0)\lvec(1 2)
		\move(1 0)\lvec(2 2)
		\move(0 2)\lvec(0 4)
		\move(2 2)\lvec(2 4)
		\move(0 4)\lvec(2 4)
		\move(1 2)\lvec(1 4)
		\move(0 4)\lvec(0 6)
		\move(2 4)\lvec(2 6)
		\move(0 6)\lvec(2 6)
		\move(1 4)\lvec(1 6)
		\htext(0.7 0.5){$0$}
		\htext(1.7 0.5){$0$}
		\htext(1.31 1.45){$0$}
		\htext(1.31 5.45){$2$}
		\htext(1.7 4.5){$2$}
		\move(1 4)\lvec(2 6)
		\move(0 4)\lvec(1 6)
		
		\move(0 6)\lvec(0 8)
		\move(1 6)\lvec(1 8)
		\move(2 6)\lvec(2 8)
		\move(0 8)\lvec(2 8)
		\htext(1.5 7){$1$}
		
		\htext(1.5 3){$1$}
		
		\esegment		
		\esegment
		
		\move(0 -46)
		\bsegment
		
		\move(-20 0)
		\bsegment
		\move(0 0)\lvec(2 0)
		\move(0 0)\lvec(0 2)
		\move(0 2)\lvec(2 2)
		\move(2 0)\lvec(2 2)
		\move(1 0)\lvec(1 2)
		\move(0 0)\lvec(1 2)
		\move(1 0)\lvec(2 2)
		\move(0 2)\lvec(0 4)
		\move(2 2)\lvec(2 4)
		\move(0 4)\lvec(2 4)
		\move(1 2)\lvec(1 4)
		\move(0 4)\lvec(0 6)
		\move(2 4)\lvec(2 6)
		\move(0 6)\lvec(2 6)
		\move(1 4)\lvec(1 6)
		\htext(0.7 0.5){$0$}
		\htext(1.7 0.5){$0$}
		\htext(1.31 1.45){$0$}
		\htext(0.31 1.45){$0$}
		\htext(1.31 5.45){$2$}
		\htext(1.7 4.5){$2$}
		\move(1 4)\lvec(2 6)
		\move(0 4)\lvec(1 6)
		
		\move(0 6)\lvec(0 8)
		\move(1 6)\lvec(1 8)
		\move(2 6)\lvec(2 8)
		\move(0 8)\lvec(2 8)
		\htext(1.5 7){$1$}
		
		\htext(1.5 3){$1$}
		\esegment
		
		\move(-15 0)
		\bsegment
		\move(0 0)\lvec(2 0)
		\move(0 0)\lvec(0 2)
		\move(0 2)\lvec(2 2)
		\move(2 0)\lvec(2 2)
		\move(1 0)\lvec(1 2)
		\move(0 0)\lvec(1 2)
		\move(1 0)\lvec(2 2)
		\move(0 2)\lvec(0 4)
		\move(2 2)\lvec(2 4)
		\move(0 4)\lvec(2 4)
		\move(1 2)\lvec(1 4)
		\move(0 4)\lvec(0 6)
		\move(2 4)\lvec(2 6)
		\move(0 6)\lvec(2 6)
		\move(1 4)\lvec(1 6)
		\htext(0.7 0.5){$0$}
		\htext(1.7 0.5){$0$}
		\htext(1.31 1.45){$0$}
		\htext(0.31 1.45){$0$}
		\htext(1.7 4.5){$2$}
		\htext(1.31 5.45){$2$}
		\move(1 4)\lvec(2 6)
		\move(0 4)\lvec(1 6)
		\htext(0.5 3){$1$}
		\htext(1.5 3){$1$}
		\esegment
		
		\move(-8 0)
		\bsegment
		\move(0 0)\lvec(2 0)
		\move(0 0)\lvec(0 2)
		\move(0 2)\lvec(2 2)
		\move(2 0)\lvec(2 2)
		\move(1 0)\lvec(1 2)
		\move(0 0)\lvec(1 2)
		\move(1 0)\lvec(2 2)
		\move(0 2)\lvec(0 4)
		\move(2 2)\lvec(2 4)
		\move(0 4)\lvec(2 4)
		\move(1 2)\lvec(1 4)
		\move(0 4)\lvec(0 6)
		\move(2 4)\lvec(2 6)
		\move(0 6)\lvec(2 6)
		\move(1 4)\lvec(1 6)
		\htext(0.7 0.5){$0$}
		\htext(1.7 0.5){$0$}
		\htext(1.31 1.45){$0$}
		\htext(0.31 1.45){$0$}
		\htext(1.7 4.5){$2$}
		\move(1 4)\lvec(2 6)
		\move(0 4)\lvec(1 6)
		\move(-2 0)\lvec(0 0)
		\move(-2 0)\lvec(-2 2)
		\move(-2 2)\lvec(0 2)
		\move(-1 0)\lvec(-1 2)
		\move(-1 0)\lvec(0 2)
		\move(-2 0)\lvec(-1 2)
		\htext(-1.3 0.5){$0$}
		\htext(-0.3 0.5){$0$}
		\htext(-0.69 1.45){$0$}
		
		\htext(1.5 3){$1$}
		\htext(0.5 3){$1$}
		\esegment
		
		\move(-1 0)
		\bsegment
		
		\move(0 0)\lvec(2 0)
		\move(0 0)\lvec(0 2)
		\move(0 2)\lvec(2 2)
		\move(2 0)\lvec(2 2)
		\move(1 0)\lvec(1 2)
		\move(0 0)\lvec(1 2)
		\move(1 0)\lvec(2 2)
		\move(0 2)\lvec(0 4)
		\move(2 2)\lvec(2 4)
		\move(0 4)\lvec(2 4)
		\move(1 2)\lvec(1 4)
		\htext(0.7 0.5){$0$}
		\htext(1.7 0.5){$0$}
		\htext(1.31 1.45){$0$}
		\htext(0.31 1.45){$0$}
		\move(-2 0)\lvec(0 0)
		\move(-2 0)\lvec(-2 2)
		\move(-2 2)\lvec(0 2)
		\move(-1 0)\lvec(-1 2)
		\move(-1 0)\lvec(0 2)
		\move(-2 0)\lvec(-1 2)
		\htext(-1.3 0.5){$0$}
		\htext(-0.3 0.5){$0$}
		\htext(-0.69 1.45){$0$}
		\htext(-1.69 1.45){$0$}
		\htext(0.5 3){$1$}
		\htext(1.5 3){$1$}
		\esegment
		
		\move(6 0)
		\bsegment
		\move(0 0)\lvec(2 0)
		\move(0 0)\lvec(0 2)
		\move(0 2)\lvec(2 2)
		\move(2 0)\lvec(2 2)
		\move(1 0)\lvec(1 2)
		\move(0 0)\lvec(1 2)
		\move(1 0)\lvec(2 2)
		\move(0 2)\lvec(0 4)
		\move(2 2)\lvec(2 4)
		\move(0 4)\lvec(2 4)
		\move(1 2)\lvec(1 4)
		\move(0 4)\lvec(0 6)
		\move(2 4)\lvec(2 6)
		\move(0 6)\lvec(2 6)
		\move(1 4)\lvec(1 6)
		\htext(0.7 0.5){$0$}
		\htext(1.7 0.5){$0$}
		\htext(1.31 1.45){$0$}
		\htext(0.31 1.45){$0$}
		
		\htext(1.7 4.5){$2$}
		\move(1 4)\lvec(2 6)
		\move(0 4)\lvec(1 6)
		\move(-2 0)\lvec(0 0)
		\move(-2 0)\lvec(-2 2)
		\move(-2 2)\lvec(0 2)
		\move(-1 0)\lvec(-1 2)
		\move(-1 0)\lvec(0 2)
		\move(-2 0)\lvec(-1 2)
		\move(-1 2)\lvec(-1 4)
		\move(-2 2)\lvec(-2 4)
		\move(-2 4)\lvec(0 4)
		\htext(-1.3 0.5){$0$}
		\htext(-0.3 0.5){$0$}
		\htext(-0.69 1.45){$0$}
		\htext(1.5 3){$1$}
		\htext(-0.5 3){$1$}
		\esegment
		
		\move(13 0)
		\bsegment
		\move(0 0)\lvec(2 0)
		\move(0 0)\lvec(0 2)
		\move(0 2)\lvec(2 2)
		\move(2 0)\lvec(2 2)
		\move(1 0)\lvec(1 2)
		\move(0 0)\lvec(1 2)
		\move(1 0)\lvec(2 2)
		\move(0 2)\lvec(0 4)
		\move(2 2)\lvec(2 4)
		\move(0 4)\lvec(2 4)
		\move(1 2)\lvec(1 4)
		\move(0 4)\lvec(0 6)
		\move(2 4)\lvec(2 6)
		\move(0 6)\lvec(2 6)
		\move(1 4)\lvec(1 6)
		\htext(0.7 0.5){$0$}
		\htext(1.7 0.5){$0$}
		\htext(1.31 1.45){$0$}
		\htext(0.31 1.45){$0$}
		\htext(1.31 5.45){$2$}
		\htext(1.7 4.5){$2$}
		\move(1 4)\lvec(2 6)
		\move(0 4)\lvec(1 6)
		\move(-2 0)\lvec(0 0)
		\move(-2 0)\lvec(-2 2)
		\move(-2 2)\lvec(0 2)
		\move(-1 0)\lvec(-1 2)
		\move(-1 0)\lvec(0 2)
		\move(-2 0)\lvec(-1 2)
		
		\htext(-1.3 0.5){$0$}
		\htext(-0.3 0.5){$0$}
		\htext(-0.69 1.45){$0$}
		\htext(1.5 3){$1$}
		
		\esegment

		\move(20 0)
		\bsegment
		\move(0 0)\lvec(2 0)
		\move(0 0)\lvec(0 2)
		\move(0 2)\lvec(2 2)
		\move(2 0)\lvec(2 2)
		\move(1 0)\lvec(1 2)
		\move(0 0)\lvec(1 2)
		\move(1 0)\lvec(2 2)
		\move(0 2)\lvec(0 4)
		\move(2 2)\lvec(2 4)
		\move(0 4)\lvec(2 4)
		\move(1 2)\lvec(1 4)
		\move(0 4)\lvec(0 6)
		\move(2 4)\lvec(2 6)
		\move(0 6)\lvec(2 6)
		\move(1 4)\lvec(1 6)
		\htext(0.7 0.5){$0$}
		\htext(1.7 0.5){$0$}
		\htext(1.31 1.45){$0$}
		
		\htext(1.31 5.45){$2$}
		\htext(1.7 4.5){$2$}
		\move(1 4)\lvec(2 6)
		\move(0 4)\lvec(1 6)
		\move(-2 0)\lvec(0 0)
		\move(-2 0)\lvec(-2 2)
		\move(-2 2)\lvec(0 2)
		\move(-1 0)\lvec(-1 2)
		\move(-1 0)\lvec(0 2)
		\move(-2 0)\lvec(-1 2)
		\move(-1 2)\lvec(-1 4)
		\move(-2 2)\lvec(-2 4)
		\move(-2 4)\lvec(0 4)
		\htext(-1.3 0.5){$0$}
		\htext(-0.3 0.5){$0$}
		\htext(-0.69 1.45){$0$}
		\htext(1.5 3){$1$}
		\htext(-0.5 3){$1$}
		\esegment
		
		\move(27 0)
		\bsegment
		\move(0 0)\lvec(2 0)
		\move(0 0)\lvec(0 2)
		\move(0 2)\lvec(2 2)
		\move(2 0)\lvec(2 2)
		\move(1 0)\lvec(1 2)
		\move(0 0)\lvec(1 2)
		\move(1 0)\lvec(2 2)
		\move(0 2)\lvec(0 4)
		\move(2 2)\lvec(2 4)
		\move(0 4)\lvec(2 4)
		\move(1 2)\lvec(1 4)
		\move(0 4)\lvec(0 6)
		\move(2 4)\lvec(2 6)
		\move(0 6)\lvec(2 6)
		\move(1 4)\lvec(1 6)
		\htext(0.7 0.5){$0$}
		\htext(1.7 0.5){$0$}
		\htext(1.31 1.45){$0$}
		\htext(1.31 5.45){$2$}
		\htext(1.7 4.5){$2$}
		\move(1 4)\lvec(2 6)
		\move(0 4)\lvec(1 6)
		\move(-2 0)\lvec(0 0)
		\move(-2 0)\lvec(-2 2)
		\move(-2 2)\lvec(0 2)
		\move(-1 0)\lvec(-1 2)
		\move(-1 0)\lvec(0 2)
		\move(-2 0)\lvec(-1 2)
		
		\move(0 6)\lvec(0 8)
		\move(1 6)\lvec(1 8)
		\move(2 6)\lvec(2 8)
		\move(0 8)\lvec(2 8)
		\htext(1.5 7){$1$}
		\htext(-1.3 0.5){$0$}
		\htext(-0.3 0.5){$0$}
		\htext(-0.69 1.45){$0$}
		\htext(1.5 3){$1$}
		
		\esegment
		
		\esegment
		\move(0 -58)
		\bsegment
		\move(-20.5 0)
		\bsegment
		\move(0 0)\lvec(2 0)
		\move(0 0)\lvec(0 2)
		\move(0 2)\lvec(2 2)
		\move(2 0)\lvec(2 2)
		\move(1 0)\lvec(1 2)
		\move(0 0)\lvec(1 2)
		\move(1 0)\lvec(2 2)
		\move(0 2)\lvec(0 4)
		\move(2 2)\lvec(2 4)
		\move(0 4)\lvec(2 4)
		\move(1 2)\lvec(1 4)
		\move(0 4)\lvec(0 6)
		\move(2 4)\lvec(2 6)
		\move(0 6)\lvec(2 6)
		\move(1 4)\lvec(1 6)
		\htext(0.7 0.5){$0$}
		\htext(1.7 0.5){$0$}
		\htext(1.31 1.45){$0$}
		\htext(0.31 1.45){$0$}
		\htext(1.31 5.45){$2$}
		\htext(1.7 4.5){$2$}
		\move(1 4)\lvec(2 6)
		\move(0 4)\lvec(1 6)
		
		\move(0 6)\lvec(0 8)
		\move(1 6)\lvec(1 8)
		\move(2 6)\lvec(2 8)
		\move(0 8)\lvec(2 8)
		\htext(1.5 7){$1$}
		
		\htext(1.5 3){$1$}
		\move(-2 0)\lvec(0 0)
		\move(-2 0)\lvec(-2 2)
		\move(-2 2)\lvec(0 2)
		\move(-1 0)\lvec(-1 2)
		\move(-1 0)\lvec(0 2)
		\move(-2 0)\lvec(-1 2)
		\htext(-1.3 0.5){$0$}
		\htext(-0.3 0.5){$0$}
		\htext(-0.69 1.45){$0$}
		
		\vtext(0 -1){\fontsize{8}{8}\selectfont$\cdots$}
		\esegment
		
		\move(-18 0)
		\bsegment
		\move(0 0)\lvec(2 0)
		\move(0 0)\lvec(0 2)
		\move(0 2)\lvec(2 2)
		\move(2 0)\lvec(2 2)
		\move(1 0)\lvec(1 2)
		\move(0 0)\lvec(1 2)
		\move(1 0)\lvec(2 2)
		\move(0 2)\lvec(0 4)
		\move(2 2)\lvec(2 4)
		\move(0 4)\lvec(2 4)
		\move(1 2)\lvec(1 4)
		\move(0 4)\lvec(0 6)
		\move(2 4)\lvec(2 6)
		\move(0 6)\lvec(2 6)
		\move(1 4)\lvec(1 6)
		\htext(0.7 0.5){$0$}
		\htext(1.7 0.5){$0$}
		\htext(1.31 1.45){$0$}
		\htext(0.31 1.45){$0$}
		\htext(1.31 5.45){$2$}
		\htext(1.7 4.5){$2$}
		\move(1 4)\lvec(2 6)
		\move(0 4)\lvec(1 6)
		
		\move(0 6)\lvec(0 8)
		\move(1 6)\lvec(1 8)
		\move(2 6)\lvec(2 8)
		\move(0 8)\lvec(2 8)
		\htext(1.5 7){$1$}
		\htext(0.5 3){$1$}
		\htext(1.5 3){$1$}
		
		\vtext(0 -1){\fontsize{8}{8}\selectfont$\cdots$}
		\esegment
		
		\move(-15.5 0)
		\bsegment
		\move(0 0)\lvec(2 0)
		\move(0 0)\lvec(0 2)
		\move(0 2)\lvec(2 2)
		\move(2 0)\lvec(2 2)
		\move(1 0)\lvec(1 2)
		\move(0 0)\lvec(1 2)
		\move(1 0)\lvec(2 2)
		\move(0 2)\lvec(0 4)
		\move(2 2)\lvec(2 4)
		\move(0 4)\lvec(2 4)
		\move(1 2)\lvec(1 4)
		\move(0 4)\lvec(0 6)
		\move(2 4)\lvec(2 6)
		\move(0 6)\lvec(2 6)
		\move(1 4)\lvec(1 6)
		\htext(0.7 0.5){$0$}
		\htext(1.7 0.5){$0$}
		\htext(1.31 1.45){$0$}
		\htext(0.31 1.45){$0$}
		\htext(1.7 4.5){$2$}
		\htext(0.7 4.5){$2$}
		\move(1 4)\lvec(2 6)
		\move(0 4)\lvec(1 6)
		\htext(1.31 5.45){$2$}
		
		\htext(1.5 3){$1$}
		\htext(0.5 3){$1$}
		
		\vtext(0 -1){\fontsize{8}{8}\selectfont$\cdots$}
		\esegment
		
		\move(-11 0)
		\bsegment
		\move(0 0)\lvec(2 0)
		\move(0 0)\lvec(0 2)
		\move(0 2)\lvec(2 2)
		\move(2 0)\lvec(2 2)
		\move(1 0)\lvec(1 2)
		\move(0 0)\lvec(1 2)
		\move(1 0)\lvec(2 2)
		\move(0 2)\lvec(0 4)
		\move(2 2)\lvec(2 4)
		\move(0 4)\lvec(2 4)
		\move(1 2)\lvec(1 4)
		\move(0 4)\lvec(0 6)
		\move(2 4)\lvec(2 6)
		\move(0 6)\lvec(2 6)
		\move(1 4)\lvec(1 6)
		\htext(0.7 0.5){$0$}
		\htext(1.7 0.5){$0$}
		\htext(1.31 1.45){$0$}
		\htext(0.31 1.45){$0$}
		\htext(1.7 4.5){$2$}
		
		\move(1 4)\lvec(2 6)
		\move(0 4)\lvec(1 6)
		\htext(1.31 5.45){$2$}
		
		\htext(1.5 3){$1$}
		\htext(0.5 3){$1$}
		
		\move(-2 0)\lvec(0 0)
		\move(-2 0)\lvec(-2 2)
		\move(-2 2)\lvec(0 2)
		\move(-1 0)\lvec(-1 2)
		\move(-1 0)\lvec(0 2)
		\move(-2 0)\lvec(-1 2)
		\htext(-1.3 0.5){$0$}
		\htext(-0.3 0.5){$0$}
		\htext(-0.69 1.45){$0$}
		
		\vtext(0 -1){\fontsize{8}{8}\selectfont$\cdots$}
		\esegment

		\move(-6.5 0)
		\bsegment
		
		\move(0 0)\lvec(2 0)
		\move(0 0)\lvec(0 2)
		\move(0 2)\lvec(2 2)
		\move(2 0)\lvec(2 2)
		\move(1 0)\lvec(1 2)
		\move(0 0)\lvec(1 2)
		\move(1 0)\lvec(2 2)
		\move(0 2)\lvec(0 4)
		\move(2 2)\lvec(2 4)
		\move(0 4)\lvec(2 4)
		\move(1 2)\lvec(1 4)
		\htext(0.7 0.5){$0$}
		\htext(1.7 0.5){$0$}
		\htext(1.31 1.45){$0$}
		\htext(0.31 1.45){$0$}
		\move(-2 0)\lvec(0 0)
		\move(-2 0)\lvec(-2 2)
		\move(-2 2)\lvec(0 2)
		\move(-1 0)\lvec(-1 2)
		\move(-1 0)\lvec(0 2)
		\move(-2 0)\lvec(-1 2)
		\htext(-1.3 0.5){$0$}
		\htext(-0.3 0.5){$0$}
		\htext(-0.69 1.45){$0$}
		\htext(-1.69 1.45){$0$}
		\htext(0.5 3){$1$}
		\htext(1.5 3){$1$}
		\move(0 4)\lvec(0 6)
		\move(2 4)\lvec(2 6)
		\move(0 6)\lvec(2 6)
		\move(1 4)\lvec(1 6)
		\htext(1.7 4.5){$2$}
		\move(1 4)\lvec(2 6)
		\move(0 4)\lvec(1 6)
		
		\vtext(0 -1){\fontsize{8}{8}\selectfont$\cdots$}
		\esegment
		
		\move(0 0)
		\bsegment
		\move(0 0)\lvec(2 0)
		\move(0 0)\lvec(0 2)
		\move(0 2)\lvec(2 2)
		\move(2 0)\lvec(2 2)
		\move(1 0)\lvec(1 2)
		\move(0 0)\lvec(1 2)
		\move(1 0)\lvec(2 2)
		\move(0 2)\lvec(0 4)
		\move(2 2)\lvec(2 4)
		\move(0 4)\lvec(2 4)
		\move(1 2)\lvec(1 4)
		\move(0 4)\lvec(0 6)
		\move(2 4)\lvec(2 6)
		\move(0 6)\lvec(2 6)
		\move(1 4)\lvec(1 6)
		\htext(0.7 0.5){$0$}
		\htext(1.7 0.5){$0$}
		\htext(1.31 1.45){$0$}
		\htext(0.31 1.45){$0$}
		
		\htext(1.7 4.5){$2$}
		\move(1 4)\lvec(2 6)
		\move(0 4)\lvec(1 6)
		\move(-2 0)\lvec(0 0)
		\move(-2 0)\lvec(-2 2)
		\move(-2 2)\lvec(0 2)
		\move(-1 0)\lvec(-1 2)
		\move(-1 0)\lvec(0 2)
		\move(-2 0)\lvec(-1 2)
		\move(-1 2)\lvec(-1 4)
		\move(-2 2)\lvec(-2 4)
		\move(-2 4)\lvec(0 4)
		\htext(-1.3 0.5){$0$}
		\htext(-0.3 0.5){$0$}
		\htext(-0.69 1.45){$0$}
		\htext(1.5 3){$1$}
		\htext(-0.5 3){$1$}
		
		\move(-4 0)\lvec(-2 0)
		\move(-4 0)\lvec(-4 2)
		\move(-4 2)\lvec(-2 2)
		\move(-3 0)\lvec(-3 2)
		\move(-3 0)\lvec(-2 2)
		\move(-4 0)\lvec(-3 2)
		\htext(-3.3 0.5){$0$}
		\htext(-2.3 0.5){$0$}
		\htext(-2.69 1.45){$0$}
		
		\vtext(0 -1){\fontsize{8}{8}\selectfont$\cdots$}
		\esegment
		
		\move(4.5 0)
		\bsegment
		\move(0 0)\lvec(2 0)
		\move(0 0)\lvec(0 2)
		\move(0 2)\lvec(2 2)
		\move(2 0)\lvec(2 2)
		\move(1 0)\lvec(1 2)
		\move(0 0)\lvec(1 2)
		\move(1 0)\lvec(2 2)
		\move(0 2)\lvec(0 4)
		\move(2 2)\lvec(2 4)
		\move(0 4)\lvec(2 4)
		\move(1 2)\lvec(1 4)
		\move(0 4)\lvec(0 6)
		\move(2 4)\lvec(2 6)
		\move(0 6)\lvec(2 6)
		\move(1 4)\lvec(1 6)
		\htext(0.7 0.5){$0$}
		\htext(1.7 0.5){$0$}
		\htext(1.31 1.45){$0$}
		\htext(0.31 1.45){$0$}
		
		\htext(1.7 4.5){$2$}
		\move(1 4)\lvec(2 6)
		\move(0 4)\lvec(1 6)
		\move(-2 0)\lvec(0 0)
		\move(-2 0)\lvec(-2 2)
		\move(-2 2)\lvec(0 2)
		\move(-1 0)\lvec(-1 2)
		\move(-1 0)\lvec(0 2)
		\move(-2 0)\lvec(-1 2)
		\move(-1 2)\lvec(-1 4)
		\move(-2 2)\lvec(-2 4)
		\move(-2 4)\lvec(0 4)
		\htext(-1.3 0.5){$0$}
		\htext(-0.3 0.5){$0$}
		\htext(-0.69 1.45){$0$}
		\htext(1.5 3){$1$}
		\htext(0.5 3){$1$}
		\htext(-0.5 3){$1$}
		
		\vtext(0 -1){\fontsize{8}{8}\selectfont$\cdots$}
		\esegment
		
		\move(9 0)
		\bsegment
		\move(0 0)\lvec(2 0)
		\move(0 0)\lvec(0 2)
		\move(0 2)\lvec(2 2)
		\move(2 0)\lvec(2 2)
		\move(1 0)\lvec(1 2)
		\move(0 0)\lvec(1 2)
		\move(1 0)\lvec(2 2)
		\move(0 2)\lvec(0 4)
		\move(2 2)\lvec(2 4)
		\move(0 4)\lvec(2 4)
		\move(1 2)\lvec(1 4)
		\move(0 4)\lvec(0 6)
		\move(2 4)\lvec(2 6)
		\move(0 6)\lvec(2 6)
		\move(1 4)\lvec(1 6)
		\htext(0.7 0.5){$0$}
		\htext(1.7 0.5){$0$}
		\htext(1.31 1.45){$0$}
		\htext(0.31 1.45){$0$}
		
		\htext(1.7 4.5){$2$}
		\htext(-0.3 4.5){$2$}
		\move(1 4)\lvec(2 6)
		\move(0 4)\lvec(1 6)
		\move(-2 0)\lvec(0 0)
		\move(-2 0)\lvec(-2 2)
		\move(-2 2)\lvec(0 2)
		\move(-1 0)\lvec(-1 2)
		\move(-1 0)\lvec(0 2)
		\move(-2 0)\lvec(-1 2)
		\move(-1 2)\lvec(-1 4)
		\move(-2 2)\lvec(-2 4)
		\move(-2 4)\lvec(0 4)
		\htext(-1.3 0.5){$0$}
		\htext(-0.3 0.5){$0$}
		\htext(-0.69 1.45){$0$}
		\htext(1.5 3){$1$}
		
		\htext(-0.5 3){$1$}
		\move(-2 4)\lvec(-2 6)
		\move(-2 6)\lvec(0 6)
		\move(-1 4)\lvec(-1 6)
		\move(-2 4)\lvec(-1 6) 
		\move(-1 4)\lvec(0 6) 
		
		\vtext(0 -1){\fontsize{8}{8}\selectfont$\cdots$}          
		\esegment

		\move(13.5 0)
		\bsegment
		\move(0 0)\lvec(2 0)
		\move(0 0)\lvec(0 2)
		\move(0 2)\lvec(2 2)
		\move(2 0)\lvec(2 2)
		\move(1 0)\lvec(1 2)
		\move(0 0)\lvec(1 2)
		\move(1 0)\lvec(2 2)
		\move(0 2)\lvec(0 4)
		\move(2 2)\lvec(2 4)
		\move(0 4)\lvec(2 4)
		\move(1 2)\lvec(1 4)
		\move(0 4)\lvec(0 6)
		\move(2 4)\lvec(2 6)
		\move(0 6)\lvec(2 6)
		\move(1 4)\lvec(1 6)
		\htext(0.7 0.5){$0$}
		\htext(1.7 0.5){$0$}
		\htext(1.31 1.45){$0$}
		\htext(0.31 1.45){$0$}
		\htext(1.31 5.45){$2$}
		\htext(1.7 4.5){$2$}
		\move(1 4)\lvec(2 6)
		\move(0 4)\lvec(1 6)
		\move(-2 0)\lvec(0 0)
		\move(-2 0)\lvec(-2 2)
		\move(-2 2)\lvec(0 2)
		\move(-1 0)\lvec(-1 2)
		\move(-1 0)\lvec(0 2)
		\move(-2 0)\lvec(-1 2)
		\move(-1 2)\lvec(-1 4)
		\move(-2 2)\lvec(-2 4)
		\move(-2 4)\lvec(0 4)
		\htext(-1.3 0.5){$0$}
		\htext(-0.3 0.5){$0$}
		\htext(-0.69 1.45){$0$}
		\htext(1.5 3){$1$}
		\htext(-0.5 3){$1$}
		
		\vtext(0 -1){\fontsize{8}{8}\selectfont$\cdots$}
		\esegment
		
		\move(20 0)
		\bsegment
		\move(0 0)\lvec(2 0)
		\move(0 0)\lvec(0 2)
		\move(0 2)\lvec(2 2)
		\move(2 0)\lvec(2 2)
		\move(1 0)\lvec(1 2)
		\move(0 0)\lvec(1 2)
		\move(1 0)\lvec(2 2)
		\move(0 2)\lvec(0 4)
		\move(2 2)\lvec(2 4)
		\move(0 4)\lvec(2 4)
		\move(1 2)\lvec(1 4)
		\move(0 4)\lvec(0 6)
		\move(2 4)\lvec(2 6)
		\move(0 6)\lvec(2 6)
		\move(1 4)\lvec(1 6)
		\htext(0.7 0.5){$0$}
		\htext(1.7 0.5){$0$}
		\htext(1.31 1.45){$0$}
		
		\htext(1.31 5.45){$2$}
		\htext(1.7 4.5){$2$}
		\move(1 4)\lvec(2 6)
		\move(0 4)\lvec(1 6)
		\move(-2 0)\lvec(0 0)
		\move(-2 0)\lvec(-2 2)
		\move(-2 2)\lvec(0 2)
		\move(-1 0)\lvec(-1 2)
		\move(-1 0)\lvec(0 2)
		\move(-2 0)\lvec(-1 2)
		\move(-1 2)\lvec(-1 4)
		\move(-2 2)\lvec(-2 4)
		\move(-2 4)\lvec(0 4)
		\htext(-1.3 0.5){$0$}
		\htext(-0.3 0.5){$0$}
		\htext(-0.69 1.45){$0$}
		\htext(1.5 3){$1$}
		\htext(-0.5 3){$1$}
		\move(-4 0)\lvec(-2 0)
		\move(-4 0)\lvec(-4 2)
		\move(-4 2)\lvec(-2 2)
		\move(-3 0)\lvec(-3 2)
		\move(-3 0)\lvec(-2 2)
		\move(-4 0)\lvec(-3 2)
		\htext(-3.3 0.5){$0$}
		\htext(-2.3 0.5){$0$}
		\htext(-2.69 1.45){$0$}
		
		\vtext(0 -1){\fontsize{8}{8}\selectfont$\cdots$}
		\esegment
		
		\move(24.5 0)
		\bsegment
		\move(0 0)\lvec(2 0)
		\move(0 0)\lvec(0 2)
		\move(0 2)\lvec(2 2)
		\move(2 0)\lvec(2 2)
		\move(1 0)\lvec(1 2)
		\move(0 0)\lvec(1 2)
		\move(1 0)\lvec(2 2)
		\move(0 2)\lvec(0 4)
		\move(2 2)\lvec(2 4)
		\move(0 4)\lvec(2 4)
		\move(1 2)\lvec(1 4)
		\move(0 4)\lvec(0 6)
		\move(2 4)\lvec(2 6)
		\move(0 6)\lvec(2 6)
		\move(1 4)\lvec(1 6)
		\htext(0.7 0.5){$0$}
		\htext(1.7 0.5){$0$}
		\htext(1.31 1.45){$0$}
		
		\htext(1.31 5.45){$2$}
		\htext(1.7 4.5){$2$}
		\move(1 4)\lvec(2 6)
		\move(0 4)\lvec(1 6)
		\move(-2 0)\lvec(0 0)
		\move(-2 0)\lvec(-2 2)
		\move(-2 2)\lvec(0 2)
		\move(-1 0)\lvec(-1 2)
		\move(-1 0)\lvec(0 2)
		\move(-2 0)\lvec(-1 2)
		\move(-1 2)\lvec(-1 4)
		\move(-2 2)\lvec(-2 4)
		\move(-2 4)\lvec(0 4)
		\htext(-1.3 0.5){$0$}
		\htext(-0.3 0.5){$0$}
		\htext(-0.69 1.45){$0$}
		\htext(1.5 3){$1$}
		\htext(-0.5 3){$1$}
		\move(0 6)\lvec(0 8)
		\move(1 6)\lvec(1 8)
		\move(2 6)\lvec(2 8)
		\move(0 8)\lvec(2 8) 
		\htext(1.5 7){$1$} 
		
		\vtext(0 -1){\fontsize{8}{8}\selectfont$\cdots$} 
		\esegment
		
		\move(29 0)
		\bsegment
		\move(0 0)\lvec(2 0)
		\move(0 0)\lvec(0 2)
		\move(0 2)\lvec(2 2)
		\move(2 0)\lvec(2 2)
		\move(1 0)\lvec(1 2)
		\move(0 0)\lvec(1 2)
		\move(1 0)\lvec(2 2)
		\move(0 2)\lvec(0 4)
		\move(2 2)\lvec(2 4)
		\move(0 4)\lvec(2 4)
		\move(1 2)\lvec(1 4)
		\move(0 4)\lvec(0 6)
		\move(2 4)\lvec(2 6)
		\move(0 6)\lvec(2 6)
		\move(1 4)\lvec(1 6)
		\htext(0.7 0.5){$0$}
		\htext(1.7 0.5){$0$}
		\htext(1.7 8.5){$0$}
		\htext(1.31 1.45){$0$}
		
		\htext(1.31 5.45){$2$}
		\htext(1.7 4.5){$2$}
		\move(1 4)\lvec(2 6)
		\move(0 4)\lvec(1 6)
		\move(-2 0)\lvec(0 0)
		\move(-2 0)\lvec(-2 2)
		\move(-2 2)\lvec(0 2)
		\move(-1 0)\lvec(-1 2)
		\move(-1 0)\lvec(0 2)
		\move(-2 0)\lvec(-1 2)
		\move(0 8)\lvec(0 10)
		\move(1 8)\lvec(1 10)
		\move(2 8)\lvec(2 10)
		\move(0 10)\lvec(2 10)
		\move(0 8)\lvec(1 10)  
		\move(1 8)\lvec(2 10)
		
		\htext(-1.3 0.5){$0$}
		\htext(-0.3 0.5){$0$}
		\htext(-0.69 1.45){$0$}
		\htext(1.5 3){$1$}
		
		\move(0 6)\lvec(0 8)
		\move(1 6)\lvec(1 8)
		\move(2 6)\lvec(2 8)
		\move(0 8)\lvec(2 8) 
		\htext(1.5 7){$1$}  
		
		\vtext(0 -1){\fontsize{8}{8}\selectfont$\cdots$}
		\esegment
		
		\esegment
		\move(4 4.5)\avec(4 1.5)\htext(4.3 3){$0$}
		\move(3.5 -1.5)\avec(0.5 -5)\htext(2.7 -3){$0$}
		\move(4.5 -1.5)\avec(7.5 -5)\htext(6.3 -3){$1$}
		\move(-1 -7.5)\avec(-4 -10.5)\htext(-1.7 -9){$1$}
		\move(8.5 -7.5)\avec(6 -10.5)\htext(8 -9){$0$}
		\move(9.5 -7.5)\avec(13 -11)\htext(11.7 -9){$2$}
		\move(-6.2 -15.5)\avec(-10.5 -19)\htext(-7.9 -17.7){$2$}
		\move(-5 -15.5)\avec(-5 -19)\htext(-4.5 -17.5){$1$}
		\move(4 -15.5)\avec(4 -19)\htext(4.5 -17.5){$0$}
		\move(6 -15.5)\avec(13 -19.5)\htext(10.3 -17.5){$2$}
		\move(15 -15.5)\avec(15 -17.5)\htext(15.5 -16.4){$0$}
		\move(16 -15.5)\avec(21 -19.5)\htext(19.3 -17.5){$2$}	
		\move(-12.5 -24.5)\avec(-15.2 -28.5)\htext(-13.2 -26.5){$2$}
		\move(-11.5 -24.5)\avec(-9.5 -28.5)\htext(-10 -26.5){$1$}	
		
		\move(-5.5 -24.5)\avec(-8.5 -28.5)\htext(-6.4 -26.5){$2$}	
		
		\move(-4.5 -24.5)\avec(-1 -30.5)\htext(-2.8 -26.5){$0$}	
		
		\move(3.5 -24.5)\avec(0 -30.5)\htext(3 -26.5){$1$}
		
		\move(4.5 -24.5)\avec(8.5 -28.5)\htext(7.2 -26.5){$2$}
		
		\move(14 -24.5)\avec(9.5 -28.5)\htext(12.5 -26.5){$0$}
		
		\move(22.5 -24.5)\avec(18 -28.5)\htext(21 -26.5){$0$}
		
		\move(23.5 -24.5)\avec(25 -26.5)\htext(24.5 -25.1){$1$}
		
		\move(-16.5 -35.5)\avec(-17.5 -39.5)\htext(-16.6 -38){$1$}
		
		\move(-9.5 -35.5)\avec(-14 -39.5)\htext(-11.5 -38){$2$}
		
		\move(-8.5 -35.5)\avec(-7.5 -39.5)\htext(-7.4 -38){$0$}
		
		\move(-2 -35.5)\avec(-6.5 -39.5)\htext(-4 -38){$2$}
		
		\move(0 -35.5)\avec(0 -41.5)\htext(0.5 -38){$0$}
		
		\move(7.5 -35.5)\avec(6 -39.5)\htext(7.1 -38){$1$}
		
		\move(8.5 -35.5)\avec(13.5 -39.5)\htext(12.3 -38){$2$}
		
		\move(16.5 -35.5)\avec(14.5 -39.5)\htext(15.8 -38){$0$}	
		
		\move(17.5 -35.5)\avec(20 -39.5)\htext(19.5 -38){$1$}
		
		\move(25 -35.5)\avec(26.5 -39.5)\htext(26.2 -37.5){$0$}	
		\move(-19 -46.5)\avec(-19.5 -49.5)\htext(-18.8 -48){$0$}
		
		\move(-18.5 -46.5)\avec(-17 -49.5)\htext(-17.3 -48){$1$}
		
		\move(-14 -46.5)\avec(-14.5 -51.5)\htext(-13.8 -49){$2$}
		
		\move(-13.5 -46.5)\avec(-10.5 -51.5)\htext(-11.5 -49){$0$}
		
		\move(-8.5 -46.5)\avec(-9.5 -51.5)\htext(-8.5 -49){$2$}
		
		\move(-7.5 -46.5)\avec(-6 -51.5)\htext(-6.2 -49){$0$}
		
		\move(-1 -46.5)\avec(-5 -51.5)\htext(-2.4 -49){$2$}
		
		\move(5 -46.5)\avec(1 -51.5)\htext(3.6 -49){$0$}
		
		\move(6 -46.5)\avec(5.5 -51.5)\htext(6.1 -49){$1$}
		
		\move(7 -46.5)\avec(9 -51.5)\htext(8.6 -49){$2$}
		
		\move(13 -46.5)\avec(13.5 -51.5)\htext(13.7 -49){$1$}
		
		\move(19.5 -46.5)\avec(19 -51.5)\htext(19.7 -49){$0$}
		
		\move(20.5 -46.5)\avec(24 -51.5)\htext(22.8 -49){$1$}
		
		\move(27 -46.5)\avec(28.5 -51.5)\htext(28.2 -49){$0$}
	\end{texdraw}
\end{center}

\end{example}

\clearpage
\begin{example}
	The top part of the  crystal $\mathbf{Y}(2\Lambda_0)$ for $C_{2}^{(1)}$ is given below.	

\vskip 15mm

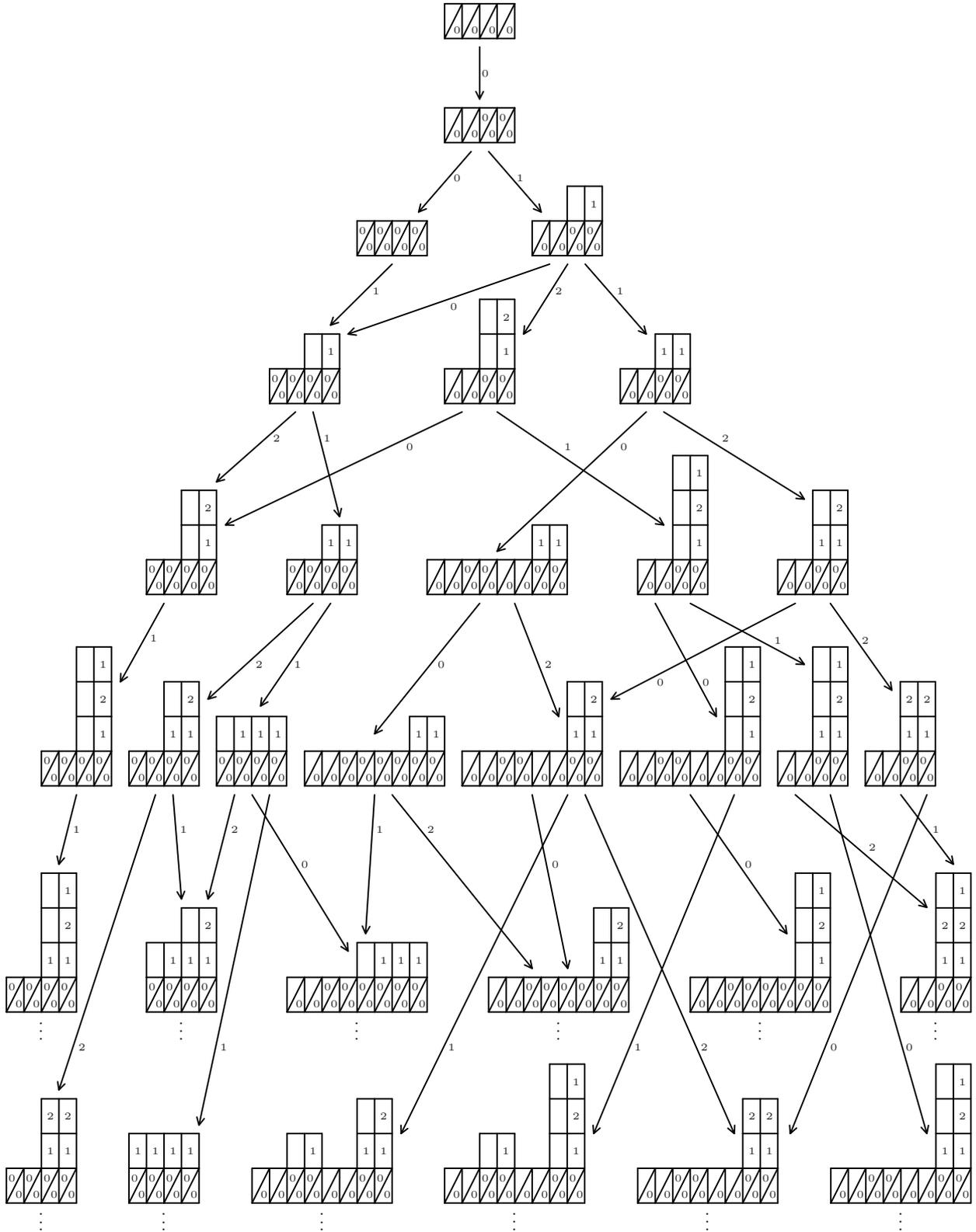
	\begin{center}
	\begin{texdraw}
		\drawdim em
		\arrowheadsize l:0.3 w:0.3
		\arrowheadtype t:V
		\fontsize{4}{4}\selectfont
		\drawdim em
		\setunitscale 1.4
		\move(0 5)
		\bsegment
		\move(3 0)
		\bsegment
		\move(0 0)\lvec(2 0)
		\move(0 0)\lvec(0 2)
		\move(0 2)\lvec(2 2)
		\move(2 0)\lvec(2 2)
		\move(1 0)\lvec(1 2)
		\move(0 0)\lvec(1 2)
		\move(1 0)\lvec(2 2)
		\move(2 0)\lvec(4 0)
		\move(4 0)\lvec(4 2)
		\move(2 2)\lvec(4 2)
		\move(3 0)\lvec(3 2)
		\move(2 0)\lvec(3 2)
		\move(3 0)\lvec(4 2)
		\htext(2.7 0.5){$0$}
		\htext(3.7 0.5){$0$}				
		\htext(0.7 0.5){$0$}
		\htext(1.7 0.5){$0$}
		\esegment
		\esegment
		\move(0 -1)
		\bsegment
		\move(3 0)
		\bsegment	
		\move(0 0)\lvec(2 0)
		\move(0 0)\lvec(0 2)
		\move(0 2)\lvec(2 2)
		\move(2 0)\lvec(2 2)
		\move(1 0)\lvec(1 2)
		\move(0 0)\lvec(1 2)
		\move(1 0)\lvec(2 2)
		\move(2 0)\lvec(4 0)
		\move(4 0)\lvec(4 2)
		\move(2 2)\lvec(4 2)
		\move(3 0)\lvec(3 2)
		\move(2 0)\lvec(3 2)
		\move(3 0)\lvec(4 2)
		\htext(2.7 0.5){$0$}
		\htext(3.7 0.5){$0$}				
		\htext(0.7 0.5){$0$}
		\htext(1.7 0.5){$0$}
		
		\htext(3.31 1.45){$0$}
		\htext(2.31 1.45){$0$}
		\esegment
		\esegment
		\move(0 -7.5)
		\bsegment
		\move(8 0)
		\bsegment
		\move(0 0)\lvec(2 0)
		\move(0 0)\lvec(0 2)
		\move(0 2)\lvec(2 2)
		\move(2 0)\lvec(2 2)
		\move(1 0)\lvec(1 2)
		\move(0 0)\lvec(1 2)
		\move(1 0)\lvec(2 2)
		\move(2 0)\lvec(4 0)
		\move(4 0)\lvec(4 2)
		\move(2 2)\lvec(4 2)
		\move(3 0)\lvec(3 2)
		\move(2 0)\lvec(3 2)
		\move(3 0)\lvec(4 2)
		\htext(2.7 0.5){$0$}
		\htext(3.7 0.5){$0$}				
		\htext(0.7 0.5){$0$}
		\htext(1.7 0.5){$0$}
		\htext(3.5 3){$1$}
		\htext(3.31 1.45){$0$}
		\htext(2.31 1.45){$0$}
		\move(2 2)\lvec(2 4)
		\move(3 2)\lvec(3 4)
		\move(4 2)\lvec(4 4)
		\move(2 4)\lvec(4 4)
		\esegment
		
		\move(-2 0)
		\bsegment
		\move(0 0)\lvec(2 0)
		\move(0 0)\lvec(0 2)
		\move(0 2)\lvec(2 2)
		\move(2 0)\lvec(2 2)
		\move(1 0)\lvec(1 2)
		\move(0 0)\lvec(1 2)
		\move(1 0)\lvec(2 2)
		\move(2 0)\lvec(4 0)
		\move(4 0)\lvec(4 2)
		\move(2 2)\lvec(4 2)
		\move(3 0)\lvec(3 2)
		\move(2 0)\lvec(3 2)
		\move(3 0)\lvec(4 2)
		\htext(2.7 0.5){$0$}
		\htext(3.7 0.5){$0$}				
		\htext(0.7 0.5){$0$}
		\htext(1.7 0.5){$0$}
		\htext(1.31 1.45){$0$}
		\htext(0.31 1.45){$0$}
		\htext(3.31 1.45){$0$}
		\htext(2.31 1.45){$0$}
		\esegment
		
		\esegment
		\move(0 -16)
		\bsegment
		\move(-7 0)
		\bsegment
		\move(0 0)\lvec(2 0)
		\move(0 0)\lvec(0 2)
		\move(0 2)\lvec(2 2)
		\move(2 0)\lvec(2 2)
		\move(1 0)\lvec(1 2)
		\move(0 0)\lvec(1 2)
		\move(1 0)\lvec(2 2)
		\move(2 0)\lvec(4 0)
		\move(4 0)\lvec(4 2)
		\move(2 2)\lvec(4 2)
		\move(3 0)\lvec(3 2)
		\move(2 0)\lvec(3 2)
		\move(3 0)\lvec(4 2)
		\htext(2.7 0.5){$0$}
		\htext(3.7 0.5){$0$}				
		\htext(0.7 0.5){$0$}
		\htext(1.7 0.5){$0$}
		\htext(3.5 3){$1$}
		\htext(1.31 1.45){$0$}
		\htext(0.31 1.45){$0$}
		\htext(3.31 1.45){$0$}
		\htext(2.31 1.45){$0$}
		\move(2 2)\lvec(2 4)
		\move(3 2)\lvec(3 4)
		\move(4 2)\lvec(4 4)
		\move(2 4)\lvec(4 4)
		\esegment
		
		\move(3 0)
		\bsegment
		\move(0 0)\lvec(2 0)
		\move(0 0)\lvec(0 2)
		\move(0 2)\lvec(2 2)
		\move(2 0)\lvec(2 2)
		\move(1 0)\lvec(1 2)
		\move(0 0)\lvec(1 2)
		\move(1 0)\lvec(2 2)
		\move(2 0)\lvec(4 0)
		\move(4 0)\lvec(4 2)
		\move(2 2)\lvec(4 2)
		\move(3 0)\lvec(3 2)
		\move(2 0)\lvec(3 2)
		\move(3 0)\lvec(4 2)
		\htext(2.7 0.5){$0$}
		\htext(3.7 0.5){$0$}				
		\htext(0.7 0.5){$0$}
		\htext(1.7 0.5){$0$}
		\htext(3.5 3){$1$}
		
		\htext(3.31 1.45){$0$}
		\htext(2.31 1.45){$0$}
		\move(2 2)\lvec(2 4)
		\move(3 2)\lvec(3 4)
		\move(4 2)\lvec(4 4)
		\move(2 4)\lvec(4 4)
		\move(2 4)\lvec(2 6)
		\move(3 4)\lvec(3 6)
		\move(4 4)\lvec(4 6)
		\move(2 6)\lvec(4 6)
		\htext(3.5 5){$2$}
		\esegment
		
		\move(13 0)
		\bsegment
		\move(0 0)\lvec(2 0)
		\move(0 0)\lvec(0 2)
		\move(0 2)\lvec(2 2)
		\move(2 0)\lvec(2 2)
		\move(1 0)\lvec(1 2)
		\move(0 0)\lvec(1 2)
		\move(1 0)\lvec(2 2)
		\move(2 0)\lvec(4 0)
		\move(4 0)\lvec(4 2)
		\move(2 2)\lvec(4 2)
		\move(3 0)\lvec(3 2)
		\move(2 0)\lvec(3 2)
		\move(3 0)\lvec(4 2)
		\htext(2.7 0.5){$0$}
		\htext(3.7 0.5){$0$}				
		\htext(0.7 0.5){$0$}
		\htext(1.7 0.5){$0$}
		\htext(3.5 3){$1$}
		\htext(2.5 3){$1$}
		\htext(3.31 1.45){$0$}
		\htext(2.31 1.45){$0$}
		\move(2 2)\lvec(2 4)
		\move(3 2)\lvec(3 4)
		\move(4 2)\lvec(4 4)
		\move(2 4)\lvec(4 4)
		\esegment
		
		\esegment
		
		\move(0 -27)
		\bsegment
		
		\move(-14 0)
		\bsegment
		
		\move(0 0)\lvec(2 0)
		\move(0 0)\lvec(0 2)
		\move(0 2)\lvec(2 2)
		\move(2 0)\lvec(2 2)
		\move(1 0)\lvec(1 2)
		\move(0 0)\lvec(1 2)
		\move(1 0)\lvec(2 2)
		\move(2 0)\lvec(4 0)
		\move(4 0)\lvec(4 2)
		\move(2 2)\lvec(4 2)
		\move(3 0)\lvec(3 2)
		\move(2 0)\lvec(3 2)
		\move(3 0)\lvec(4 2)
		\htext(2.7 0.5){$0$}
		\htext(3.7 0.5){$0$}				
		\htext(0.7 0.5){$0$}
		\htext(1.7 0.5){$0$}
		\htext(3.5 3){$1$}
		\htext(1.31 1.45){$0$}
		\htext(0.31 1.45){$0$}
		\htext(3.31 1.45){$0$}
		\htext(2.31 1.45){$0$}
		\move(2 2)\lvec(2 4)
		\move(3 2)\lvec(3 4)
		\move(4 2)\lvec(4 4)
		\move(2 4)\lvec(4 4)
		\move(2 4)\lvec(2 6)
		\move(3 4)\lvec(3 6)
		\move(4 4)\lvec(4 6)
		\move(2 6)\lvec(4 6)
		\htext(3.5 5){$2$}		
		\esegment
		
		\move(-6 0)
		\bsegment
		\move(0 0)\lvec(2 0)
		\move(0 0)\lvec(0 2)
		\move(0 2)\lvec(2 2)
		\move(2 0)\lvec(2 2)
		\move(1 0)\lvec(1 2)
		\move(0 0)\lvec(1 2)
		\move(1 0)\lvec(2 2)
		\move(2 0)\lvec(4 0)
		\move(4 0)\lvec(4 2)
		\move(2 2)\lvec(4 2)
		\move(3 0)\lvec(3 2)
		\move(2 0)\lvec(3 2)
		\move(3 0)\lvec(4 2)
		\htext(2.7 0.5){$0$}
		\htext(3.7 0.5){$0$}				
		\htext(0.7 0.5){$0$}
		\htext(1.7 0.5){$0$}
		\htext(3.5 3){$1$}
		\htext(2.5 3){$1$}
		\htext(1.31 1.45){$0$}
		\htext(0.31 1.45){$0$}
		\htext(3.31 1.45){$0$}
		\htext(2.31 1.45){$0$}
		\move(2 2)\lvec(2 4)
		\move(3 2)\lvec(3 4)
		\move(4 2)\lvec(4 4)
		\move(2 4)\lvec(4 4)

		\esegment
		
		\move(6 0)
		\bsegment
		\move(0 0)\lvec(2 0)
		\move(0 0)\lvec(0 2)
		\move(0 2)\lvec(2 2)
		\move(2 0)\lvec(2 2)
		\move(1 0)\lvec(1 2)
		\move(0 0)\lvec(1 2)
		\move(1 0)\lvec(2 2)
		\move(2 0)\lvec(4 0)
		\move(4 0)\lvec(4 2)
		\move(2 2)\lvec(4 2)
		\move(3 0)\lvec(3 2)
		\move(2 0)\lvec(3 2)
		\move(3 0)\lvec(4 2)
		\htext(2.7 0.5){$0$}
		\htext(3.7 0.5){$0$}				
		\htext(0.7 0.5){$0$}
		\htext(1.7 0.5){$0$}
		\htext(3.5 3){$1$}
		\htext(2.5 3){$1$}
		\move(-4 0)\lvec(0 0)
		\move(-4 2)\lvec(0 2)
		\move(-4 0)\lvec(-4 2)
		\move(-3 0)\lvec(-3 2)
		\move(-2 0)\lvec(-2 2)
		\move(-1 0)\lvec(-1 2)
		\move(-4 0)\lvec(-3 2)
		\move(-3 0)\lvec(-2 2)
		\move(-2 0)\lvec(-1 2)
		\move(-1 0)\lvec(0 2)
		\htext(-0.69 1.45){$0$}
		\htext(-1.69 1.45){$0$}
		\htext(-1.3 0.5){$0$}
		\htext(-0.3 0.5){$0$}				
		\htext(-3.3 0.5){$0$}
		\htext(-2.3 0.5){$0$}
		
		\htext(3.31 1.45){$0$}
		\htext(2.31 1.45){$0$}
		\move(2 2)\lvec(2 4)
		\move(3 2)\lvec(3 4)
		\move(4 2)\lvec(4 4)
		\move(2 4)\lvec(4 4)
		
		\esegment
		
		\move(14 0)
		\bsegment
		
		\move(0 0)\lvec(2 0)
		\move(0 0)\lvec(0 2)
		\move(0 2)\lvec(2 2)
		\move(2 0)\lvec(2 2)
		\move(1 0)\lvec(1 2)
		\move(0 0)\lvec(1 2)
		\move(1 0)\lvec(2 2)
		\move(2 0)\lvec(4 0)
		\move(4 0)\lvec(4 2)
		\move(2 2)\lvec(4 2)
		\move(3 0)\lvec(3 2)
		\move(2 0)\lvec(3 2)
		\move(3 0)\lvec(4 2)
		\htext(2.7 0.5){$0$}
		\htext(3.7 0.5){$0$}				
		\htext(0.7 0.5){$0$}
		\htext(1.7 0.5){$0$}
		\htext(3.5 3){$1$}
		
		\htext(3.31 1.45){$0$}
		\htext(2.31 1.45){$0$}
		\move(2 2)\lvec(2 4)
		\move(3 2)\lvec(3 4)
		\move(4 2)\lvec(4 4)
		\move(2 4)\lvec(4 4)
		\move(2 4)\lvec(2 6)
		\move(3 4)\lvec(3 6)
		\move(4 4)\lvec(4 6)
		\move(2 6)\lvec(4 6)
		\htext(3.5 5){$2$}
		\htext(3.5 7){$1$}
		\move(2 6)\lvec(2 8)
		\move(3 6)\lvec(3 8)
		\move(4 6)\lvec(4 8)
		\move(2 8)\lvec(4 8)

		\esegment
		
		\move(22 0)
		\bsegment
		\move(0 0)\lvec(2 0)
		\move(0 0)\lvec(0 2)
		\move(0 2)\lvec(2 2)
		\move(2 0)\lvec(2 2)
		\move(1 0)\lvec(1 2)
		\move(0 0)\lvec(1 2)
		\move(1 0)\lvec(2 2)
		\move(2 0)\lvec(4 0)
		\move(4 0)\lvec(4 2)
		\move(2 2)\lvec(4 2)
		\move(3 0)\lvec(3 2)
		\move(2 0)\lvec(3 2)
		\move(3 0)\lvec(4 2)
		\htext(2.7 0.5){$0$}
		\htext(3.7 0.5){$0$}				
		\htext(0.7 0.5){$0$}
		\htext(1.7 0.5){$0$}
		\htext(3.5 3){$1$}
		\htext(2.5 3){$1$}
		\htext(3.31 1.45){$0$}
		\htext(2.31 1.45){$0$}
		\move(2 2)\lvec(2 4)
		\move(3 2)\lvec(3 4)
		\move(4 2)\lvec(4 4)
		\move(2 4)\lvec(4 4)
		\move(2 4)\lvec(2 6)
		\move(3 4)\lvec(3 6)
		\move(4 4)\lvec(4 6)
		\move(2 6)\lvec(4 6)
		\htext(3.5 5){$2$}
		\esegment

		\esegment
		\move(0 -38)
		\bsegment
		
		\move(-20 0)
		\bsegment
		\move(0 0)\lvec(2 0)
		\move(0 0)\lvec(0 2)
		\move(0 2)\lvec(2 2)
		\move(2 0)\lvec(2 2)
		\move(1 0)\lvec(1 2)
		\move(0 0)\lvec(1 2)
		\move(1 0)\lvec(2 2)
		\move(2 0)\lvec(4 0)
		\move(4 0)\lvec(4 2)
		\move(2 2)\lvec(4 2)
		\move(3 0)\lvec(3 2)
		\move(2 0)\lvec(3 2)
		\move(3 0)\lvec(4 2)
		\htext(2.7 0.5){$0$}
		\htext(3.7 0.5){$0$}				
		\htext(0.7 0.5){$0$}
		\htext(1.7 0.5){$0$}
		\htext(3.5 3){$1$}
		\htext(1.31 1.45){$0$}
		\htext(0.31 1.45){$0$}
		\htext(3.31 1.45){$0$}
		\htext(2.31 1.45){$0$}
		\move(2 2)\lvec(2 4)
		\move(3 2)\lvec(3 4)
		\move(4 2)\lvec(4 4)
		\move(2 4)\lvec(4 4)
		\move(2 4)\lvec(2 6)
		\move(3 4)\lvec(3 6)
		\move(4 4)\lvec(4 6)
		\move(2 6)\lvec(4 6)
		\htext(3.5 5){$2$}
		\htext(3.5 7){$1$}
		\move(2 6)\lvec(2 8)
		\move(3 6)\lvec(3 8)
		\move(4 6)\lvec(4 8)
		\move(2 8)\lvec(4 8)	
		
		\esegment
		
		\move(-15 0)
		\bsegment
		\move(0 0)\lvec(2 0)
		\move(0 0)\lvec(0 2)
		\move(0 2)\lvec(2 2)
		\move(2 0)\lvec(2 2)
		\move(1 0)\lvec(1 2)
		\move(0 0)\lvec(1 2)
		\move(1 0)\lvec(2 2)
		\move(2 0)\lvec(4 0)
		\move(4 0)\lvec(4 2)
		\move(2 2)\lvec(4 2)
		\move(3 0)\lvec(3 2)
		\move(2 0)\lvec(3 2)
		\move(3 0)\lvec(4 2)
		\htext(2.7 0.5){$0$}
		\htext(3.7 0.5){$0$}				
		\htext(0.7 0.5){$0$}
		\htext(1.7 0.5){$0$}
		\htext(3.5 3){$1$}
		\htext(2.5 3){$1$}
		\htext(1.31 1.45){$0$}
		\htext(0.31 1.45){$0$}
		\htext(3.31 1.45){$0$}
		\htext(2.31 1.45){$0$}
		\move(2 2)\lvec(2 4)
		\move(3 2)\lvec(3 4)
		\move(4 2)\lvec(4 4)
		\move(2 4)\lvec(4 4)
		\move(2 4)\lvec(2 6)
		\move(3 4)\lvec(3 6)
		\move(4 4)\lvec(4 6)
		\move(2 6)\lvec(4 6)
		\htext(3.5 5){$2$}
		\esegment
		
		\move(-10 0)
		\bsegment
		
		\move(0 0)\lvec(2 0)
		\move(0 0)\lvec(0 2)
		\move(0 2)\lvec(2 2)
		\move(2 0)\lvec(2 2)
		\move(1 0)\lvec(1 2)
		\move(0 0)\lvec(1 2)
		\move(1 0)\lvec(2 2)
		\move(2 0)\lvec(4 0)
		\move(4 0)\lvec(4 2)
		\move(2 2)\lvec(4 2)
		\move(3 0)\lvec(3 2)
		\move(2 0)\lvec(3 2)
		\move(3 0)\lvec(4 2)
		\htext(2.7 0.5){$0$}
		\htext(3.7 0.5){$0$}				
		\htext(0.7 0.5){$0$}
		\htext(1.7 0.5){$0$}
		\htext(3.5 3){$1$}
		\htext(2.5 3){$1$}
		\htext(1.5 3){$1$}
		\htext(1.31 1.45){$0$}
		\htext(0.31 1.45){$0$}
		\htext(3.31 1.45){$0$}
		\htext(2.31 1.45){$0$}
		\move(0 2)\lvec(0 4)
		\move(1 2)\lvec(1 4)
		\move(0 4)\lvec(2 4)		
		\move(2 2)\lvec(2 4)
		\move(3 2)\lvec(3 4)
		\move(4 2)\lvec(4 4)
		\move(2 4)\lvec(4 4)

		\esegment
		
		\move(-1 0)
		\bsegment
		\move(0 0)\lvec(2 0)
		\move(0 0)\lvec(0 2)
		\move(0 2)\lvec(2 2)
		\move(2 0)\lvec(2 2)
		\move(1 0)\lvec(1 2)
		\move(0 0)\lvec(1 2)
		\move(1 0)\lvec(2 2)
		\move(2 0)\lvec(4 0)
		\move(4 0)\lvec(4 2)
		\move(2 2)\lvec(4 2)
		\move(3 0)\lvec(3 2)
		\move(2 0)\lvec(3 2)
		\move(3 0)\lvec(4 2)
		\htext(2.7 0.5){$0$}
		\htext(3.7 0.5){$0$}				
		\htext(0.7 0.5){$0$}
		\htext(1.7 0.5){$0$}
		\htext(3.5 3){$1$}
		\htext(2.5 3){$1$}
		\move(-4 0)\lvec(0 0)
		\move(-4 2)\lvec(0 2)
		\move(-4 0)\lvec(-4 2)
		\move(-3 0)\lvec(-3 2)
		\move(-2 0)\lvec(-2 2)
		\move(-1 0)\lvec(-1 2)
		\move(-4 0)\lvec(-3 2)
		\move(-3 0)\lvec(-2 2)
		\move(-2 0)\lvec(-1 2)
		\move(-1 0)\lvec(0 2)
		\htext(-0.69 1.45){$0$}
		\htext(-1.69 1.45){$0$}
		\htext(-1.3 0.5){$0$}
		\htext(-0.3 0.5){$0$}				
		\htext(-3.3 0.5){$0$}
		\htext(-2.3 0.5){$0$}
		\htext(1.31 1.45){$0$}
		\htext(0.31 1.45){$0$}
		\htext(3.31 1.45){$0$}
		\htext(2.31 1.45){$0$}
		\move(2 2)\lvec(2 4)
		\move(3 2)\lvec(3 4)
		\move(4 2)\lvec(4 4)
		\move(2 4)\lvec(4 4)

		\esegment
		
		\move(8 0)
		\bsegment
		\move(0 0)\lvec(2 0)
		\move(0 0)\lvec(0 2)
		\move(0 2)\lvec(2 2)
		\move(2 0)\lvec(2 2)
		\move(1 0)\lvec(1 2)
		\move(0 0)\lvec(1 2)
		\move(1 0)\lvec(2 2)
		\move(2 0)\lvec(4 0)
		\move(4 0)\lvec(4 2)
		\move(2 2)\lvec(4 2)
		\move(3 0)\lvec(3 2)
		\move(2 0)\lvec(3 2)
		\move(3 0)\lvec(4 2)
		\htext(2.7 0.5){$0$}
		\htext(3.7 0.5){$0$}				
		\htext(0.7 0.5){$0$}
		\htext(1.7 0.5){$0$}
		\htext(3.5 3){$1$}
		\htext(3.5 5){$2$}
		\htext(2.5 3){$1$}
		\move(-4 0)\lvec(0 0)
		\move(-4 2)\lvec(0 2)
		\move(-4 0)\lvec(-4 2)
		\move(-3 0)\lvec(-3 2)
		\move(-2 0)\lvec(-2 2)
		\move(-1 0)\lvec(-1 2)
		\move(-4 0)\lvec(-3 2)
		\move(-3 0)\lvec(-2 2)
		\move(-2 0)\lvec(-1 2)
		\move(-1 0)\lvec(0 2)
		\htext(-0.69 1.45){$0$}
		\htext(-1.69 1.45){$0$}
		\htext(-1.3 0.5){$0$}
		\htext(-0.3 0.5){$0$}				
		\htext(-3.3 0.5){$0$}
		\htext(-2.3 0.5){$0$}
		
		\htext(3.31 1.45){$0$}
		\htext(2.31 1.45){$0$}
		\move(2 2)\lvec(2 4)
		\move(3 2)\lvec(3 4)
		\move(4 2)\lvec(4 4)
		\move(2 4)\lvec(4 4)
		\move(2 4)\lvec(2 6)
		\move(3 4)\lvec(3 6)
		\move(4 4)\lvec(4 6)
		
		\move(2 6)\lvec(4 6)
		\esegment

		\move(17 0)
		\bsegment
		\move(0 0)\lvec(2 0)
		\move(0 0)\lvec(0 2)
		\move(0 2)\lvec(2 2)
		\move(2 0)\lvec(2 2)
		\move(1 0)\lvec(1 2)
		\move(0 0)\lvec(1 2)
		\move(1 0)\lvec(2 2)
		\move(2 0)\lvec(4 0)
		\move(4 0)\lvec(4 2)
		\move(2 2)\lvec(4 2)
		\move(3 0)\lvec(3 2)
		\move(2 0)\lvec(3 2)
		\move(3 0)\lvec(4 2)
		\htext(2.7 0.5){$0$}
		\htext(3.7 0.5){$0$}				
		\htext(0.7 0.5){$0$}
		\htext(1.7 0.5){$0$}
		\htext(3.5 3){$1$}
		\htext(3.5 5){$2$}
		\htext(3.5 7){$1$}
		\move(-4 0)\lvec(0 0)
		\move(-4 2)\lvec(0 2)
		\move(-4 0)\lvec(-4 2)
		\move(-3 0)\lvec(-3 2)
		\move(-2 0)\lvec(-2 2)
		\move(-1 0)\lvec(-1 2)
		\move(-4 0)\lvec(-3 2)
		\move(-3 0)\lvec(-2 2)
		\move(-2 0)\lvec(-1 2)
		\move(-1 0)\lvec(0 2)
		\htext(-0.69 1.45){$0$}
		\htext(-1.69 1.45){$0$}
		\htext(-1.3 0.5){$0$}
		\htext(-0.3 0.5){$0$}				
		\htext(-3.3 0.5){$0$}
		\htext(-2.3 0.5){$0$}
		
		\htext(3.31 1.45){$0$}
		\htext(2.31 1.45){$0$}
		\move(2 2)\lvec(2 4)
		\move(3 2)\lvec(3 4)
		\move(4 2)\lvec(4 4)
		\move(2 4)\lvec(4 4)
		\move(2 4)\lvec(2 8)
		\move(3 4)\lvec(3 8)
		\move(4 4)\lvec(4 8)
		\move(2 8)\lvec(4 8)
		\move(2 6)\lvec(4 6)
		\esegment

		\move(22 0)
		\bsegment
		\move(0 0)\lvec(2 0)
		\move(0 0)\lvec(0 2)
		\move(0 2)\lvec(2 2)
		\move(2 0)\lvec(2 2)
		\move(1 0)\lvec(1 2)
		\move(0 0)\lvec(1 2)
		\move(1 0)\lvec(2 2)
		\move(2 0)\lvec(4 0)
		\move(4 0)\lvec(4 2)
		\move(2 2)\lvec(4 2)
		\move(3 0)\lvec(3 2)
		\move(2 0)\lvec(3 2)
		\move(3 0)\lvec(4 2)
		\htext(2.7 0.5){$0$}
		\htext(3.7 0.5){$0$}				
		\htext(0.7 0.5){$0$}
		\htext(1.7 0.5){$0$}
		\htext(3.5 3){$1$}
		\htext(3.5 5){$2$}
		\htext(3.5 7){$1$}
		\htext(2.5 3){$1$}
		\htext(3.31 1.45){$0$}
		\htext(2.31 1.45){$0$}
		\move(2 2)\lvec(2 4)
		\move(3 2)\lvec(3 4)
		\move(4 2)\lvec(4 4)
		\move(2 4)\lvec(4 4)
		\move(2 4)\lvec(2 8)
		\move(3 4)\lvec(3 8)
		\move(4 4)\lvec(4 8)
		\move(2 8)\lvec(4 8)
		\move(2 6)\lvec(4 6)
		
		\esegment
		
		\move(27 0)
		\bsegment
		\move(0 0)\lvec(2 0)
		\move(0 0)\lvec(0 2)
		\move(0 2)\lvec(2 2)
		\move(2 0)\lvec(2 2)
		\move(1 0)\lvec(1 2)
		\move(0 0)\lvec(1 2)
		\move(1 0)\lvec(2 2)
		\move(2 0)\lvec(4 0)
		\move(4 0)\lvec(4 2)
		\move(2 2)\lvec(4 2)
		\move(3 0)\lvec(3 2)
		\move(2 0)\lvec(3 2)
		\move(3 0)\lvec(4 2)
		\htext(2.7 0.5){$0$}
		\htext(3.7 0.5){$0$}				
		\htext(0.7 0.5){$0$}
		\htext(1.7 0.5){$0$}
		\htext(3.5 3){$1$}
		\htext(3.5 5){$2$}
		\htext(2.5 5){$2$}
		\htext(2.5 3){$1$}
		\htext(3.31 1.45){$0$}
		\htext(2.31 1.45){$0$}
		\move(2 2)\lvec(2 4)
		\move(3 2)\lvec(3 4)
		\move(4 2)\lvec(4 4)
		\move(2 4)\lvec(4 4)
		\move(2 4)\lvec(2 6)
		\move(3 4)\lvec(3 6)
		\move(4 4)\lvec(4 6)
		
		\move(2 6)\lvec(4 6)
		
		\esegment						
		\esegment
		
		\move(0 -51)
		\bsegment
		
		\move(-22 0)
		\bsegment
		\move(0 0)\lvec(2 0)
		\move(0 0)\lvec(0 2)
		\move(0 2)\lvec(2 2)
		\move(2 0)\lvec(2 2)
		\move(1 0)\lvec(1 2)
		\move(0 0)\lvec(1 2)
		\move(1 0)\lvec(2 2)
		\move(2 0)\lvec(4 0)
		\move(4 0)\lvec(4 2)
		\move(2 2)\lvec(4 2)
		\move(3 0)\lvec(3 2)
		\move(2 0)\lvec(3 2)
		\move(3 0)\lvec(4 2)
		\htext(2.7 0.5){$0$}
		\htext(3.7 0.5){$0$}				
		\htext(0.7 0.5){$0$}
		\htext(1.7 0.5){$0$}
		\htext(3.5 3){$1$}
		\htext(2.5 3){$1$}
		\htext(1.31 1.45){$0$}
		\htext(0.31 1.45){$0$}
		\htext(3.31 1.45){$0$}
		\htext(2.31 1.45){$0$}
		\move(2 2)\lvec(2 4)
		\move(3 2)\lvec(3 4)
		\move(4 2)\lvec(4 4)
		\move(2 4)\lvec(4 4)
		\move(2 4)\lvec(2 6)
		\move(3 4)\lvec(3 6)
		\move(4 4)\lvec(4 6)
		\move(2 6)\lvec(4 6)
		\htext(3.5 5){$2$}
		\htext(3.5 7){$1$}
		\move(2 6)\lvec(2 8)
		\move(3 6)\lvec(3 8)
		\move(4 6)\lvec(4 8)
		\move(2 8)\lvec(4 8)
		\vtext(2 -1){\fontsize{8}{8}\selectfont$\cdots$}	
		\esegment

		\move(-14 0)
		\bsegment
		\move(0 0)\lvec(2 0)
		\move(0 0)\lvec(0 2)
		\move(0 2)\lvec(2 2)
		\move(2 0)\lvec(2 2)
		\move(1 0)\lvec(1 2)
		\move(0 0)\lvec(1 2)
		\move(1 0)\lvec(2 2)
		\move(2 0)\lvec(4 0)
		\move(4 0)\lvec(4 2)
		\move(2 2)\lvec(4 2)
		\move(3 0)\lvec(3 2)
		\move(2 0)\lvec(3 2)
		\move(3 0)\lvec(4 2)
		\htext(2.7 0.5){$0$}
		\htext(3.7 0.5){$0$}				
		\htext(0.7 0.5){$0$}
		\htext(1.7 0.5){$0$}
		\htext(3.5 3){$1$}
		\htext(2.5 3){$1$}
		\htext(1.5 3){$1$}
		\htext(1.31 1.45){$0$}
		\htext(0.31 1.45){$0$}
		\htext(3.31 1.45){$0$}
		\htext(2.31 1.45){$0$}
		\move(0 2)\lvec(0 4)
		\move(1 2)\lvec(1 4)
		\move(0 4)\lvec(2 4)	
		\move(2 2)\lvec(2 4)
		\move(3 2)\lvec(3 4)
		\move(4 2)\lvec(4 4)
		\move(2 4)\lvec(4 4)
		\move(2 4)\lvec(2 6)
		\move(3 4)\lvec(3 6)
		\move(4 4)\lvec(4 6)
		\move(2 6)\lvec(4 6)
		\htext(3.5 5){$2$}
		
		\vtext(2 -1){\fontsize{8}{8}\selectfont$\cdots$}
		\esegment
		
		\move(-2 0)
		\bsegment
		
		\move(0 0)\lvec(2 0)
		\move(0 0)\lvec(0 2)
		\move(0 2)\lvec(2 2)
		\move(2 0)\lvec(2 2)
		\move(1 0)\lvec(1 2)
		\move(0 0)\lvec(1 2)
		\move(1 0)\lvec(2 2)
		\move(2 0)\lvec(4 0)
		\move(4 0)\lvec(4 2)
		\move(2 2)\lvec(4 2)
		\move(3 0)\lvec(3 2)
		\move(2 0)\lvec(3 2)
		\move(3 0)\lvec(4 2)
		\htext(2.7 0.5){$0$}
		\htext(3.7 0.5){$0$}				
		\htext(0.7 0.5){$0$}
		\htext(1.7 0.5){$0$}
		\htext(3.5 3){$1$}
		\htext(2.5 3){$1$}
		\htext(1.5 3){$1$}
		\move(-4 0)\lvec(0 0)
		\move(-4 2)\lvec(0 2)
		\move(-4 0)\lvec(-4 2)
		\move(-3 0)\lvec(-3 2)
		\move(-2 0)\lvec(-2 2)
		\move(-1 0)\lvec(-1 2)
		\move(-4 0)\lvec(-3 2)
		\move(-3 0)\lvec(-2 2)
		\move(-2 0)\lvec(-1 2)
		\move(-1 0)\lvec(0 2)
		\htext(-0.69 1.45){$0$}
		\htext(-1.69 1.45){$0$}
		\htext(-1.3 0.5){$0$}
		\htext(-0.3 0.5){$0$}				
		\htext(-3.3 0.5){$0$}
		\htext(-2.3 0.5){$0$}
		\htext(1.31 1.45){$0$}
		\htext(0.31 1.45){$0$}
		\htext(3.31 1.45){$0$}
		\htext(2.31 1.45){$0$}
		\move(2 2)\lvec(2 4)
		\move(3 2)\lvec(3 4)
		\move(4 2)\lvec(4 4)
		\move(2 4)\lvec(4 4)
		\move(0 2)\lvec(0 4)
		\move(1 2)\lvec(1 4)
		\move(0 4)\lvec(2 4)
		\vtext(0 -1){\fontsize{8}{8}\selectfont$\cdots$}
		\esegment
		
		\move(9.5 0)
		\bsegment
		\move(0 0)\lvec(2 0)
		\move(0 0)\lvec(0 2)
		\move(0 2)\lvec(2 2)
		\move(2 0)\lvec(2 2)
		\move(1 0)\lvec(1 2)
		\move(0 0)\lvec(1 2)
		\move(1 0)\lvec(2 2)
		\move(2 0)\lvec(4 0)
		\move(4 0)\lvec(4 2)
		\move(2 2)\lvec(4 2)
		\move(3 0)\lvec(3 2)
		\move(2 0)\lvec(3 2)
		\move(3 0)\lvec(4 2)
		\htext(2.7 0.5){$0$}
		\htext(3.7 0.5){$0$}				
		\htext(0.7 0.5){$0$}
		\htext(1.7 0.5){$0$}
		\htext(3.5 3){$1$}
		\htext(3.5 5){$2$}
		\htext(2.5 3){$1$}
		\move(-4 0)\lvec(0 0)
		\move(-4 2)\lvec(0 2)
		\move(-4 0)\lvec(-4 2)
		\move(-3 0)\lvec(-3 2)
		\move(-2 0)\lvec(-2 2)
		\move(-1 0)\lvec(-1 2)
		\move(-4 0)\lvec(-3 2)
		\move(-3 0)\lvec(-2 2)
		\move(-2 0)\lvec(-1 2)
		\move(-1 0)\lvec(0 2)
		\htext(-0.69 1.45){$0$}
		\htext(-1.69 1.45){$0$}
		\htext(-1.3 0.5){$0$}
		\htext(-0.3 0.5){$0$}				
		\htext(-3.3 0.5){$0$}
		\htext(-2.3 0.5){$0$}
		\htext(1.31 1.45){$0$}
		\htext(0.31 1.45){$0$}
		\htext(3.31 1.45){$0$}
		\htext(2.31 1.45){$0$}
		\move(2 2)\lvec(2 4)
		\move(3 2)\lvec(3 4)
		\move(4 2)\lvec(4 4)
		\move(2 4)\lvec(4 4)
		\move(2 4)\lvec(2 6)
		\move(3 4)\lvec(3 6)
		\move(4 4)\lvec(4 6)
		\move(2 6)\lvec(4 6)
		
		\vtext(0 -1){\fontsize{8}{8}\selectfont$\cdots$}
		\esegment
		
		\move(21 0)
		\bsegment
		\move(0 0)\lvec(2 0)
		\move(0 0)\lvec(0 2)
		\move(0 2)\lvec(2 2)
		\move(2 0)\lvec(2 2)
		\move(1 0)\lvec(1 2)
		\move(0 0)\lvec(1 2)
		\move(1 0)\lvec(2 2)
		\move(2 0)\lvec(4 0)
		\move(4 0)\lvec(4 2)
		\move(2 2)\lvec(4 2)
		\move(3 0)\lvec(3 2)
		\move(2 0)\lvec(3 2)
		\move(3 0)\lvec(4 2)
		\htext(2.7 0.5){$0$}
		\htext(3.7 0.5){$0$}				
		\htext(0.7 0.5){$0$}
		\htext(1.7 0.5){$0$}
		\htext(3.5 3){$1$}
		\htext(3.5 5){$2$}
		\htext(3.5 7){$1$}

		\move(-4 0)\lvec(0 0)
		\move(-4 2)\lvec(0 2)
		\move(-4 0)\lvec(-4 2)
		\move(-3 0)\lvec(-3 2)
		\move(-2 0)\lvec(-2 2)
		\move(-1 0)\lvec(-1 2)
		\move(-4 0)\lvec(-3 2)
		\move(-3 0)\lvec(-2 2)
		\move(-2 0)\lvec(-1 2)
		\move(-1 0)\lvec(0 2)
		\htext(-0.69 1.45){$0$}
		\htext(-1.69 1.45){$0$}
		\htext(1.31 1.45){$0$}
		\htext(0.31 1.45){$0$}
		
		\htext(-1.3 0.5){$0$}
		\htext(-0.3 0.5){$0$}				
		\htext(-3.3 0.5){$0$}
		\htext(-2.3 0.5){$0$}
		
		\htext(3.31 1.45){$0$}
		\htext(2.31 1.45){$0$}
		\move(2 2)\lvec(2 4)
		\move(3 2)\lvec(3 4)
		\move(4 2)\lvec(4 4)
		\move(2 4)\lvec(4 4)
		\move(2 4)\lvec(2 8)
		\move(3 4)\lvec(3 8)
		\move(4 4)\lvec(4 8)
		\move(2 8)\lvec(4 8)
		\move(2 6)\lvec(4 6)
	    \vtext(0 -1){\fontsize{8}{8}\selectfont$\cdots$}
		\esegment
		
		\move(29 0)
		\bsegment
		\move(0 0)\lvec(2 0)
		\move(0 0)\lvec(0 2)
		\move(0 2)\lvec(2 2)
		\move(2 0)\lvec(2 2)
		\move(1 0)\lvec(1 2)
		\move(0 0)\lvec(1 2)
		\move(1 0)\lvec(2 2)
		\move(2 0)\lvec(4 0)
		\move(4 0)\lvec(4 2)
		\move(2 2)\lvec(4 2)
		\move(3 0)\lvec(3 2)
		\move(2 0)\lvec(3 2)
		\move(3 0)\lvec(4 2)
		\htext(2.7 0.5){$0$}
		\htext(3.7 0.5){$0$}				
		\htext(0.7 0.5){$0$}
		\htext(1.7 0.5){$0$}
		\htext(3.5 3){$1$}
		\htext(3.5 5){$2$}
		\htext(2.5 5){$2$}
		\htext(3.5 7){$1$}
		\htext(2.5 3){$1$}
		\htext(3.31 1.45){$0$}
		\htext(2.31 1.45){$0$}
		\move(2 2)\lvec(2 4)
		\move(3 2)\lvec(3 4)
		\move(4 2)\lvec(4 4)
		\move(2 4)\lvec(4 4)
		\move(2 4)\lvec(2 8)
		\move(3 4)\lvec(3 8)
		\move(4 4)\lvec(4 8)
		\move(2 8)\lvec(4 8)
		\move(2 6)\lvec(4 6)
		\vtext(2 -1){\fontsize{8}{8}\selectfont$\cdots$}
		\esegment

		\esegment
		\move(0 -62)
		\bsegment
		\move(-22 0)
		\bsegment
		\move(0 0)\lvec(2 0)
		\move(0 0)\lvec(0 2)
		\move(0 2)\lvec(2 2)
		\move(2 0)\lvec(2 2)
		\move(1 0)\lvec(1 2)
		\move(0 0)\lvec(1 2)
		\move(1 0)\lvec(2 2)
		\move(2 0)\lvec(4 0)
		\move(4 0)\lvec(4 2)
		\move(2 2)\lvec(4 2)
		\move(3 0)\lvec(3 2)
		\move(2 0)\lvec(3 2)
		\move(3 0)\lvec(4 2)
		\htext(2.7 0.5){$0$}
		\htext(3.7 0.5){$0$}				
		\htext(0.7 0.5){$0$}
		\htext(1.7 0.5){$0$}
		\htext(3.5 3){$1$}
		\htext(2.5 3){$1$}
		\htext(1.31 1.45){$0$}
		\htext(0.31 1.45){$0$}
		\htext(3.31 1.45){$0$}
		\htext(2.31 1.45){$0$}
		\move(2 2)\lvec(2 4)
		\move(3 2)\lvec(3 4)
		\move(4 2)\lvec(4 4)
		\move(2 4)\lvec(4 4)
		\move(2 4)\lvec(2 6)
		\move(3 4)\lvec(3 6)
		\move(4 4)\lvec(4 6)
		\move(2 6)\lvec(4 6)
		\htext(3.5 5){$2$}
		\htext(2.5 5){$2$}
		
		\vtext(2 -1){\fontsize{8}{8}\selectfont$\cdots$}
		\esegment

		\move(-15 0)
		\bsegment
		\move(0 0)\lvec(2 0)
		\move(0 0)\lvec(0 2)
		\move(0 2)\lvec(2 2)
		\move(2 0)\lvec(2 2)
		\move(1 0)\lvec(1 2)
		\move(0 0)\lvec(1 2)
		\move(1 0)\lvec(2 2)
		\move(2 0)\lvec(4 0)
		\move(4 0)\lvec(4 2)
		\move(2 2)\lvec(4 2)
		\move(3 0)\lvec(3 2)
		\move(2 0)\lvec(3 2)
		\move(3 0)\lvec(4 2)
		\htext(2.7 0.5){$0$}
		\htext(3.7 0.5){$0$}				
		\htext(0.7 0.5){$0$}
		\htext(1.7 0.5){$0$}
		\htext(3.5 3){$1$}
		\htext(2.5 3){$1$}
		\htext(1.5 3){$1$}
		\htext(0.5 3){$1$}
		\htext(1.31 1.45){$0$}
		\htext(0.31 1.45){$0$}
		\htext(3.31 1.45){$0$}
		\htext(2.31 1.45){$0$}
		\move(0 2)\lvec(0 4)
		\move(1 2)\lvec(1 4)
		\move(0 4)\lvec(2 4)		
		\move(2 2)\lvec(2 4)
		\move(3 2)\lvec(3 4)
		\move(4 2)\lvec(4 4)
		\move(2 4)\lvec(4 4)
		
		\vtext(2 -1){\fontsize{8}{8}\selectfont$\cdots$}	
		\esegment
		
		\move(-4 0)
		\bsegment

		\move(0 0)\lvec(2 0)
		\move(0 0)\lvec(0 2)
		\move(0 2)\lvec(2 2)
		\move(2 0)\lvec(2 2)
		\move(1 0)\lvec(1 2)
		\move(0 0)\lvec(1 2)
		\move(1 0)\lvec(2 2)
		\move(2 0)\lvec(4 0)
		\move(4 0)\lvec(4 2)
		\move(2 2)\lvec(4 2)
		\move(3 0)\lvec(3 2)
		\move(2 0)\lvec(3 2)
		\move(3 0)\lvec(4 2)
		\htext(2.7 0.5){$0$}
		\htext(3.7 0.5){$0$}				
		\htext(0.7 0.5){$0$}
		\htext(1.7 0.5){$0$}
		\htext(3.5 3){$1$}
		\htext(3.5 5){$2$}
		\htext(2.5 3){$1$}
		\move(-4 0)\lvec(0 0)
		\move(-4 2)\lvec(0 2)
		\move(-4 0)\lvec(-4 2)
		\move(-3 0)\lvec(-3 2)
		\move(-2 0)\lvec(-2 2)
		\move(-1 0)\lvec(-1 2)
		\move(-4 0)\lvec(-3 2)
		\move(-3 0)\lvec(-2 2)
		\move(-2 0)\lvec(-1 2)
		\move(-1 0)\lvec(0 2)
		\htext(-0.69 1.45){$0$}
		\htext(-1.69 1.45){$0$}
		\htext(-1.3 0.5){$0$}
		\htext(-0.3 0.5){$0$}				
		\htext(-3.3 0.5){$0$}
		\htext(-2.3 0.5){$0$}
		
		\htext(3.31 1.45){$0$}
		\htext(2.31 1.45){$0$}
		\move(2 2)\lvec(2 4)
		\move(3 2)\lvec(3 4)
		\move(4 2)\lvec(4 4)
		\move(2 4)\lvec(4 4)
		\move(2 4)\lvec(2 6)
		\move(3 4)\lvec(3 6)
		\move(4 4)\lvec(4 6)
		\move(2 6)\lvec(4 6)
		\move(-2 2)\lvec(-2 4)
		\move(-1 2)\lvec(-1 4)
		\move(0 2)\lvec(0 4)
		\move(-2 4)\lvec(0 4)
		\htext(-0.5 3){$1$}

		\vtext(0 -1){\fontsize{8}{8}\selectfont$\cdots$}		
		
		\esegment
		
		\move(7 0)
		\bsegment
		\move(0 0)\lvec(2 0)
		\move(0 0)\lvec(0 2)
		\move(0 2)\lvec(2 2)
		\move(2 0)\lvec(2 2)
		\move(1 0)\lvec(1 2)
		\move(0 0)\lvec(1 2)
		\move(1 0)\lvec(2 2)
		\move(2 0)\lvec(4 0)
		\move(4 0)\lvec(4 2)
		\move(2 2)\lvec(4 2)
		\move(3 0)\lvec(3 2)
		\move(2 0)\lvec(3 2)
		\move(3 0)\lvec(4 2)
		\htext(2.7 0.5){$0$}
		\htext(3.7 0.5){$0$}				
		\htext(0.7 0.5){$0$}
		\htext(1.7 0.5){$0$}
		\htext(3.5 3){$1$}
		\htext(-0.5 3){$1$}
		\htext(3.5 5){$2$}
		\htext(3.5 7){$1$}
		\move(-4 0)\lvec(0 0)
		\move(-4 2)\lvec(0 2)
		\move(-4 0)\lvec(-4 2)
		\move(-3 0)\lvec(-3 2)
		\move(-2 0)\lvec(-2 2)
		\move(-1 0)\lvec(-1 2)
		\move(-4 0)\lvec(-3 2)
		\move(-3 0)\lvec(-2 2)
		\move(-2 0)\lvec(-1 2)
		\move(-1 0)\lvec(0 2)
		\htext(-0.69 1.45){$0$}
		\htext(-1.69 1.45){$0$}
		\htext(-1.3 0.5){$0$}
		\htext(-0.3 0.5){$0$}				
		\htext(-3.3 0.5){$0$}
		\htext(-2.3 0.5){$0$}
		
		\htext(3.31 1.45){$0$}
		\htext(2.31 1.45){$0$}
		\move(2 2)\lvec(2 4)
		\move(3 2)\lvec(3 4)
		\move(4 2)\lvec(4 4)
		\move(2 4)\lvec(4 4)
		\move(2 4)\lvec(2 8)
		\move(3 4)\lvec(3 8)
		\move(4 4)\lvec(4 8)
		\move(2 8)\lvec(4 8)
		\move(2 6)\lvec(4 6)
		\move(-2 2)\lvec(-2 4)
		\move(-1 2)\lvec(-1 4)
		\move(0 2)\lvec(0 4)
		\move(-2 4)\lvec(0 4)	
		
		\vtext(0 -1){\fontsize{8}{8}\selectfont$\cdots$}	
		\esegment
		
		\move(18 0)
		\bsegment
		\move(0 0)\lvec(2 0)
		\move(0 0)\lvec(0 2)
		\move(0 2)\lvec(2 2)
		\move(2 0)\lvec(2 2)
		\move(1 0)\lvec(1 2)
		\move(0 0)\lvec(1 2)
		\move(1 0)\lvec(2 2)
		\move(2 0)\lvec(4 0)
		\move(4 0)\lvec(4 2)
		\move(2 2)\lvec(4 2)
		\move(3 0)\lvec(3 2)
		\move(2 0)\lvec(3 2)
		\move(3 0)\lvec(4 2)
		\htext(2.7 0.5){$0$}
		\htext(3.7 0.5){$0$}				
		\htext(0.7 0.5){$0$}
		\htext(1.7 0.5){$0$}
		\htext(3.5 3){$1$}
		\htext(3.5 5){$2$}
		\htext(2.5 5){$2$}
		\htext(2.5 3){$1$}
		\move(-4 0)\lvec(0 0)
		\move(-4 2)\lvec(0 2)
		\move(-4 0)\lvec(-4 2)
		\move(-3 0)\lvec(-3 2)
		\move(-2 0)\lvec(-2 2)
		\move(-1 0)\lvec(-1 2)
		\move(-4 0)\lvec(-3 2)
		\move(-3 0)\lvec(-2 2)
		\move(-2 0)\lvec(-1 2)
		\move(-1 0)\lvec(0 2)
		\htext(-0.69 1.45){$0$}
		\htext(-1.69 1.45){$0$}
		\htext(-1.3 0.5){$0$}
		\htext(-0.3 0.5){$0$}				
		\htext(-3.3 0.5){$0$}
		\htext(-2.3 0.5){$0$}
		
		\htext(3.31 1.45){$0$}
		\htext(2.31 1.45){$0$}
		\move(2 2)\lvec(2 4)
		\move(3 2)\lvec(3 4)
		\move(4 2)\lvec(4 4)
		\move(2 4)\lvec(4 4)
		\move(2 4)\lvec(2 6)
		\move(3 4)\lvec(3 6)
		\move(4 4)\lvec(4 6)
		\move(2 6)\lvec(4 6)
		
		\vtext(0 -1){\fontsize{8}{8}\selectfont$\cdots$}	
		\esegment

		\move(29 0)
		\bsegment
		\move(0 0)\lvec(2 0)
		\move(0 0)\lvec(0 2)
		\move(0 2)\lvec(2 2)
		\move(2 0)\lvec(2 2)
		\move(1 0)\lvec(1 2)
		\move(0 0)\lvec(1 2)
		\move(1 0)\lvec(2 2)
		\move(2 0)\lvec(4 0)
		\move(4 0)\lvec(4 2)
		\move(2 2)\lvec(4 2)
		\move(3 0)\lvec(3 2)
		\move(2 0)\lvec(3 2)
		\move(3 0)\lvec(4 2)
		\htext(2.7 0.5){$0$}
		\htext(3.7 0.5){$0$}				
		\htext(0.7 0.5){$0$}
		\htext(1.7 0.5){$0$}
		\htext(3.5 3){$1$}
		\htext(2.5 3){$1$}
		\htext(3.5 5){$2$}
		\htext(3.5 7){$1$}
		\move(-4 0)\lvec(0 0)
		\move(-4 2)\lvec(0 2)
		\move(-4 0)\lvec(-4 2)
		\move(-3 0)\lvec(-3 2)
		\move(-2 0)\lvec(-2 2)
		\move(-1 0)\lvec(-1 2)
		\move(-4 0)\lvec(-3 2)
		\move(-3 0)\lvec(-2 2)
		\move(-2 0)\lvec(-1 2)
		\move(-1 0)\lvec(0 2)
		\htext(-0.69 1.45){$0$}
		\htext(-1.69 1.45){$0$}
		\htext(-1.3 0.5){$0$}
		\htext(-0.3 0.5){$0$}				
		\htext(-3.3 0.5){$0$}
		\htext(-2.3 0.5){$0$}
		
		\htext(3.31 1.45){$0$}
		\htext(2.31 1.45){$0$}
		\move(2 2)\lvec(2 4)
		\move(3 2)\lvec(3 4)
		\move(4 2)\lvec(4 4)
		\move(2 4)\lvec(4 4)
		\move(2 4)\lvec(2 8)
		\move(3 4)\lvec(3 8)
		\move(4 4)\lvec(4 8)
		\move(2 8)\lvec(4 8)
		\move(2 6)\lvec(4 6)
		
		\vtext(0 -1){\fontsize{8}{8}\selectfont$\cdots$}
		\esegment
		
		\esegment
		
		\move(5 4.5)\avec(5 1.5)\htext(5.3 3){$0$}
		\move(4.5 -1.5)\avec(1.5 -5)\htext(3.7 -3){$0$}
		\move(5.5 -1.5)\avec(8.5 -5)\htext(7.3 -3){$1$}
		\move(0 -8)\avec(-3.5 -11.5)\htext(-0.9 -9.5){$1$}
		\move(9 -8)\avec(-2.5 -12)\htext(3.5 -10.4){$0$}
		\move(10 -8)\avec(7.5 -12)\htext(9.5 -9.5){$2$}
		\move(11 -8)\avec(14.5 -12)\htext(13 -9.5){$1$}
		\move(-5.5 -16.5)\avec(-10 -20.5)\htext(-6.6 -18){$2$}

		\move(-4.5 -16.5)\avec(-3 -22.5)\htext(-3.7 -18){$1$}
		
		\move(4 -16.5)\avec(-9.5 -23)\htext(1 -18.5){$0$}

		\move(6 -16.5)\avec(15.5 -23)\htext(10 -18.5){$1$}

		\move(14.5 -16.5)\avec(6 -24.5)\htext(13.2 -18.5){$0$}

		\move(15.5 -16.5)\avec(23.5 -21.5)\htext(19 -18){$2$}	
		\move(-13 -27.5)\avec(-15.5 -32)\htext(-13.6 -29.5){$1$}

		\move(-4.5 -27.5)\avec(-10.5 -33)\htext(-7.6 -31){$2$}	
		
		\move(-3.5 -27.5)\avec(-7.5 -33.4)\htext(-5.4 -31){$1$}
		
		\move(5 -27.5)\avec(-1 -35)\htext(2.8 -31){$0$}
		
		\move(7 -27.5)\avec(9.5 -34)\htext(8.9 -31){$2$}
		
		\move(15 -27.5)\avec(18.5 -34)\htext(17.9 -32){$0$}
		
		\move(17 -27.5)\avec(23.5 -31)\htext(22 -29.6){$1$}
		
		\move(23 -27.5)\avec(12.5 -33)\htext(15.3 -32){$0$}
		
		\move(25 -27.5)\avec(28.5 -32.5)\htext(27 -29.6){$2$}		
		
		\move(-18 -38.5)\avec(-19 -42.5)\htext(-18 -40.5){$1$}
		
		\move(-13.5 -38.5)\avec(-19 -55.5)\htext(-17.7 -53){$2$}
		
		\move(-12.5 -38.5)\avec(-12 -44.5)\htext(-11.9 -40.5){$1$}
		
		\move(-9 -38.5)\avec(-10.5 -44.5)\htext(-9 -40.5){$2$}	
		
		\move(-8 -38.5)\avec(-2.5 -47.5)\htext(-5 -42.5){$0$}		
		
		\move(-7 -38.5)\avec(-11 -57.5)\htext(-9.6 -53){$1$}	
		
		\move(-1 -38.5)\avec(-1.5 -46.5)\htext(-0.7 -40.5){$1$}	
		
		\move(0 -38.5)\avec(8 -48.5)\htext(2.2 -40.5){$2$}			
		
		\move(10 -38.5)\avec(0.5 -58)\htext(3.4 -53){$1$}
		
		\move(8 -38.5)\avec(10 -48.5)\htext(9.3 -42.5){$0$}	
	
		\move(11 -38.5)\avec(19.5 -58)\htext(17.8 -53){$2$}	
		
		\move(19.5 -38.5)\avec(11.5 -58)\htext(14 -53){$1$}
		
		\move(17 -38.5)\avec(22.5 -46.5)\htext(20.3 -42.5){$0$}	
		
		\move(30.5 -38.5)\avec(22.7 -58)\htext(25.2 -53){$0$}	
		
		\move(23 -38.5)\avec(30.5 -45)\htext(27.4 -41.5){$2$}
		
		\move(29 -38.5)\avec(32 -42.5)\htext(31 -40.5){$1$}	
		
		\move(25 -38.5)\avec(30.5 -58)\htext(29.5 -53){$0$}				
	\end{texdraw}
\end{center}

\end{example}

\clearpage
\section{New Young wall realization of the cryatal $B(\infty)$}

\vskip 2mm

Let $B(\infty)$ denote the crystal of the negative half $U_{q}^{-}(\mathfrak g)$ of $U_{q}(\mathfrak g)$. In this section, we will give a new Young wall realization of  $B(\infty)$. 

\vskip 2mm 

Let $(B_{\infty}, b_{\infty})$ be the crystal  limit of coherent family of
perfect crystals $\{{\mathcal B}^{(l)}\}_{l\ge 1}$ \cite[Lemma 4.3]{KKM}. One could find an explicit description of $B_{\infty}$ in \cite[Section 5]{KKM}. Then there is a $U_{q}'(\mathfrak g)$-crystal  isomorphism 

\begin{equation}\label{B(infty) crystal isomorphism}
	B(\infty)\cong B(\infty)\otimes	B_{\infty},\quad u_{\infty}\mapsto u_{\infty}\otimes b_{\infty},
\end{equation}
where $u_{\infty}$ is the element in $B(\infty)$ corresponding to $1\in U^{-}_q(\mathfrak g)$.

\vskip 2mm

By applying the isomorphism \eqref{B(infty) crystal isomorphism} repeatedly, we have the crystal isomorphism as follows:

\begin{equation*}\label{B(infty) crystal isomorphism repeat}
	\psi_k:	B(\infty)\cong B(\infty)\otimes	B^{k}_{\infty},\quad u_{\infty}\mapsto u_{\infty}\otimes  \underbrace{b_{\infty}\otimes\cdots\otimes b_{\infty}}_{k}.
\end{equation*}

The infinite sequence $(\cdots, b_{\infty}, b_{\infty}, b_{\infty})$ in $B^{\otimes \infty}_{\infty}$ is called the \textit{ground-state path}. A {\it path} in $B^{\otimes \infty}_{\infty}$ is an infinite sequence $p={(p_k)}_{k\ge 1}$ such that $p_k=b_{\infty}$ for $k\gg 0$. Let ${\mathcal P}(\infty)$ denote the set of all paths in $B^{\otimes \infty}_{\infty}$.
\vskip 2mm
The crystal $B(\infty)$ has the following path realization.

\begin{theorem} [\cite{KKM}]
	There exists a $U_{q}'(\mathfrak g)$-crystal isomorphism 
	
	\begin{equation*}
		B(\infty)\longrightarrow  {\mathcal P}(\infty)
		\ \ \text{given by} \ \ b\to \cdots\otimes p_k\otimes\cdots\otimes p_2\otimes p_1,
	\end{equation*}
	where $\psi_k(b)=u_{\infty}\otimes p_k\otimes\cdots p_1$ for $k\gg 0$.
\end{theorem}

For all  $i\in I$, the following colored blocks will be used to construct Young columns corresponding to the vectors in $B_{\infty}$.

\vskip 5mm

\begin{center}
	\begin{texdraw}
		\fontsize{3}{3}\selectfont
		\drawdim em
		\setunitscale 2.3
		\move(-8 0)
		\bsegment
		\move(0 0)\lvec(-0.6 0)
		\move(0 0)\lvec(0 2)
		\move(-0.6 0)\lvec(-0.6 2)
		
		\move(-0.6 2)\lvec(0 2)
		
		\htext(-0.3 1){$i$}
		\htext(0.7 1){\fontsize{12}{12}\selectfont{$=$}}
		
		\htext(1 -0.7){\fontsize{8}{8}\selectfont{$1/\infty\times1\times1$}}
		
		\move(1.4 0)\lvec(2 0)
		
		\move(1.4 2)\lvec(2 2)
		
		\move(2 0)\lvec(2 2)
		
		\move(2 2)\lvec(2.8 2.5)
		
		\move(2 0)\lvec(2.8 0.5)
		
		\move(2.8 2.5)\lvec(2.8 0.5)
		
		\move(2.2 2.5)\lvec(2.8 2.5)
		
		\move(1.4 0)\lvec(1.4 2)
		\move(1.4 2)\lvec(2.2 2.5)
		\htext(1.7 1){$i$}
		\esegment

		\move(0 0)
		\bsegment
		\move(0 0)\lvec(-0.6 0)
		\move(0 0)\lvec(0 2)
		\move(-0.6 0)\lvec(-0.6 2)
		
		\move(-0.6 2)\lvec(0 2)
		\move(-0.6 0)\lvec(0 2)

		\htext(-0.35 1.5){$i$}
		\htext(0.8 1){\fontsize{12}{12}\selectfont{$=$}}
		
		\htext(1 -0.7){\fontsize{8}{8}\selectfont{$1/\infty\times1/2\times1$}}
		
		\move(1.4 0)\lvec(2 0)
		\move(1.4 2)\lvec(2 2)
		\move(2 0)\lvec(2 2)
		\move(2 2)\lvec(2.4 2.25)
		\move(2 0)\lvec(2.8 0.5)
		\move(2.4 0.25)\lvec(2.4 2.25)
		\move(2.4 2.25)\lvec(1.8 2.25)
		\move(1.4 0)\lvec(1.4 2)
		\move(1.4 2)\lvec(1.8 2.25)
		\move(2.4 0.5)\lvec(2.8 0.5)
		\htext(1.7 1){$i$}

		\esegment
		
		\move(8 0)
		\bsegment
		\move(0 0)\lvec(-0.6 0)
		\move(0 0)\lvec(0 2)
		\move(-0.6 0)\lvec(-0.6 2)
		
		\move(-0.6 2)\lvec(0 2)
		\move(-0.6 0)\lvec(0 2)
		
		\htext(-0.18 0.5){$i$}
		
		\htext(0.8 1){\fontsize{12}{12}\selectfont{$=$}}
		
		\htext(1 -0.7){\fontsize{8}{8}\selectfont{$1/\infty\times1/2\times1$}}
		
		\move(1.4 0)\lvec(2 0)

		\move(2.4 2.25)\lvec(2.8 2.5)
		\move(2 0)\lvec(2.8 0.5)
		\move(1.4 0)\lvec(1.8 0.25)
		\move(2.8 2.5)\lvec(2.8 0.5)
		\move(2.4 0.25)\lvec(2.4 2.25)
		\move(1.8 0.25)\lvec(1.8 2.25)
		\htext(2.1 1.25){$i$}
		\move(2.4 2.25)\lvec(1.8 2.25)
		\move(2.4 0.25)\lvec(1.8 0.25)
		
		\move(2.2 2.5)\lvec(2.8 2.5)

		\move(1.8 2.25)\lvec(2.2 2.5)

		\esegment
	\end{texdraw}
\end{center}   

\vskip 4mm

Here, the notation $\frac{1}{\infty}$ is just a formal symbol. We can think of it as a constant.

\vskip 3mm 

To define the notion of Young columns for $B(\infty)$, we need to make some changes of the
definitions in  Section \ref{Level-$l$ reduced Young Column}. 

\vskip 3mm

\begin{enumerate}
\item For type $A_{2n}^{(2)}$, let $l'=-\sum_{i=1}^{n}(x_i+\bar{x}_i)$ and $x_i,\bar{x}_i\in\mathbf Z$.
\vspace{6pt}
\item For type $D_{n+1}^{(2)}$, let $l'=-\sum_{i=1}^{n}(x_i+\bar{x}_i)-x_0$, $x_0=0$ or $1$ and  $x_i,\bar{x}_i\in\mathbf Z$.
\vspace{6pt}
\item For type $A_{2n-1}^{(2)}$, let $\sum_{i=1}^{n}(x_i+\bar{x}_i)=0$ and $x_i,\bar{x}_i\in\mathbf Z$.
\vspace{6pt}
\item For type $D_n^{(1)}$, let $\sum_{i=1}^{n}(x_i+\bar{x}_i)=0$, $x_n=0$ or $\bar{x}_n=0$ and $x_i,\bar{x}_i\in\mathbf Z$.
\vspace{6pt}
\item For type $B_{n}^{(1)}$, let $x_0+\sum_{i=1}^{n}(x_i+\bar{x}_i)=0$, $x_0=0$ or $1$ and $x_i,\bar{x}_i\in\mathbf Z$.
\vspace{6pt}
\item For type $C_{n}^{(1)}$, let $l'=-\sum_{i=1}^{n}(x_i+\bar{x}_i)$, $x_0=0$ and $x_i,\bar{x}_i\in\mathbf Z$.
\end{enumerate}

\vskip 3mm 

The modified Young columns are called the \textit{virtual Young columns}. In this case, we allow {\it negative} integers to count the number of blocks or slices that appear in a virtual Young column. 

\vskip 2mm 

Let $\widehat{\mathbf Z}_{\geq 0}={\mathbf Z}_{\geq 0}\cup \{\infty\}$. We define the notion of  \textit{extended Young columns}  as follows. We use $u_i,v_i, \bar{u}_i,\bar{v}_i$ to represent the number of slices in the blank areas. We use the rightmost column to mark the colors of the blocks in a Young column.

\vskip 2mm

For type $A_{2n}^{(2)}$, let $C_{(x_1,\cdots,x_n,\bar{x}_n,\cdots,\bar{x}_1)}$ be a virtual Young column. We choose $u_i,v_i, \bar{u}_i,\bar{v}_i\in \widehat{\mathbf Z}_{\geq 0}$ such that 

\begin{equation*}
\bar{u}_i-\bar{v}_i=\bar{x}_i,\  u_i-v_i=x_i,\  u_{n+1}-v_{n+1}=-\sum_{i=1}^{n}(x_i+\bar{x}_i).
	\end{equation*}
	
Then the corresponding extended Young column is given as

\vskip 10mm

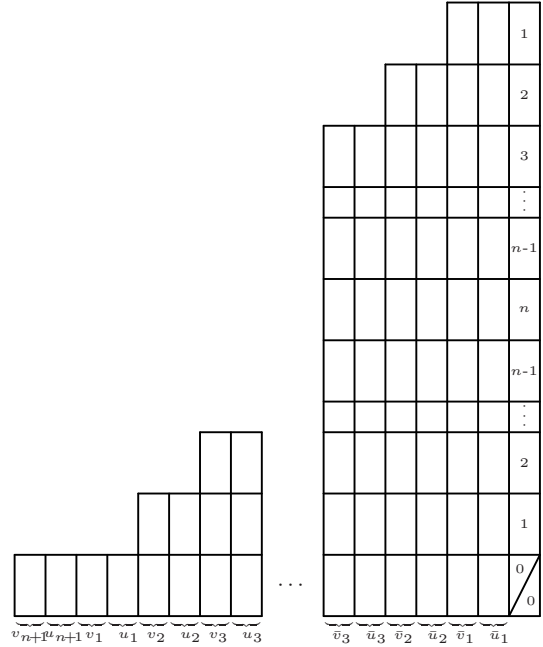
\begin{center}
	\begin{texdraw}
		\fontsize{3}{3}\selectfont
		\drawdim em
		\setunitscale 1.7
		\move(0 0)
		\bsegment
		\move(0 0)\lvec(0 2)
		\move(0 0)\lvec(2 0)
		\move(0 2)\lvec(2 2)
		\move(1 0)\lvec(1 2)
		\htext(0.5 -0.3){$\underbrace{\rule{1.5em}{0em}}$}
		\htext(0.5 -0.7){$v_{n\!+\!1}$}
		
		\htext(1.5 -0.3){$\underbrace{\rule{1.5em}{0em}}$}
		\htext(1.6 -0.7){$u_{n\!+\!1}$}
		\esegment
		
		\move(2 0)
		\bsegment
		\move(0 0)\lvec(0 2)
		\move(0 0)\lvec(2 0)
		\move(0 2)\lvec(2 2)
		\move(1 0)\lvec(1 2)
		\htext(0.5 -0.3){$\underbrace{\rule{1.5em}{0em}}$}
		\htext(0.6 -0.7){$v_{1}$}
		
		\htext(1.5 -0.3){$\underbrace{\rule{1.5em}{0em}}$}
		\htext(1.7 -0.7){$u_{1}$}
		\esegment
		
		\move(4 0)
		\bsegment
		\move(0 0)\lvec(0 2)
		\move(0 0)\lvec(2 0)
		\move(0 2)\lvec(2 2)
		\move(1 0)\lvec(1 2)
		
		\move(0 2)\lvec(0 4)
		\move(1 2)\lvec(1 4)
		\move(0 4)\lvec(2 4)
		
		\htext(0.5 -0.3){$\underbrace{\rule{1.5em}{0em}}$}
		\htext(0.6 -0.7){$v_{2}$}
		
		\htext(1.5 -0.3){$\underbrace{\rule{1.5em}{0em}}$}
		\htext(1.7 -0.7){$u_{2}$}
		\esegment
		
		\move(6 0)
		\bsegment
		\move(0 0)\lvec(0 2)
		\move(0 0)\lvec(2 0)
		\move(0 2)\lvec(2 2)
		\move(1 0)\lvec(1 2)
		
		\move(0 2)\lvec(0 4)
		\move(1 2)\lvec(1 4)
		\move(0 4)\lvec(2 4)
		
		\move(0 4)\lvec(0 6)
		\move(1 4)\lvec(1 6)
		\move(2 0)\lvec(2 6)
		\move(0 6)\lvec(2 6)
		\htext(0.5 -0.3){$\underbrace{\rule{1.5em}{0em}}$}
		\htext(0.6 -0.7){$v_{3}$}
		
		\htext(1.5 -0.3){$\underbrace{\rule{1.5em}{0em}}$}
		\htext(1.7 -0.7){$u_{3}$}
		\htext(3 1){\fontsize{8}{8}\selectfont$\cdots$}
		\esegment
		
		\move(10 0)
		\bsegment
		\move(0 0)\lvec(0 2)
		\move(0 0)\lvec(2 0)
		\move(0 2)\lvec(2 2)
		\move(1 0)\lvec(1 2)
		
		\move(0 2)\lvec(0 4)
		\move(1 2)\lvec(1 4)
		\move(0 4)\lvec(2 4)
		
		\move(0 4)\lvec(0 6)
		\move(1 4)\lvec(1 6)
		
		\move(0 6)\lvec(2 6)
		
		\move(0 6)\lvec(0 16)
		\move(1 6)\lvec(1 16)
		\move(0 16)\lvec(2 16)
		\move(0 14)\lvec(2 14)
		\move(0 13)\lvec(2 13)
		\move(0 11)\lvec(2 11)
		\move(0 9)\lvec(2 9)
		\move(0 7)\lvec(2 7)
		\htext(0.5 -0.3){$\underbrace{\rule{1.5em}{0em}}$}
		\htext(0.6 -0.7){$\bar{v}_{3}$}
		
		\htext(1.5 -0.3){$\underbrace{\rule{1.5em}{0em}}$}
		\htext(1.7 -0.7){$\bar{u}_{3}$}
		
		\esegment
		
		\move(12 0)
		\bsegment
		\move(0 0)\lvec(0 2)
		\move(0 0)\lvec(2 0)
		\move(0 2)\lvec(2 2)
		\move(1 0)\lvec(1 2)
		
		\move(0 2)\lvec(0 4)
		\move(1 2)\lvec(1 4)
		\move(0 4)\lvec(2 4)
		
		\move(0 4)\lvec(0 6)
		\move(1 4)\lvec(1 6)
		
		\move(0 6)\lvec(2 6)
		
		\move(0 6)\lvec(0 16)
		\move(1 6)\lvec(1 16)
		\move(0 16)\lvec(2 16)
		\move(0 14)\lvec(2 14)
		\move(0 13)\lvec(2 13)
		\move(0 11)\lvec(2 11)
		\move(0 9)\lvec(2 9)
		\move(0 7)\lvec(2 7)
		
		\move(0 16)\lvec(0 18)
		\move(1 16)\lvec(1 18)
		\move(0 18)\lvec(2 18)
		
		\htext(0.5 -0.3){$\underbrace{\rule{1.5em}{0em}}$}
		\htext(0.6 -0.7){$\bar{v}_{2}$}
		
		\htext(1.5 -0.3){$\underbrace{\rule{1.5em}{0em}}$}
		\htext(1.7 -0.7){$\bar{u}_{2}$}
		\esegment
		
		\move(14 0)
		\bsegment
		\move(0 0)\lvec(0 2)
		\move(0 0)\lvec(2 0)
		\move(0 2)\lvec(2 2)
		\move(1 0)\lvec(1 2)
		
		\move(0 2)\lvec(0 4)
		\move(1 2)\lvec(1 4)
		\move(0 4)\lvec(2 4)
		
		\move(0 4)\lvec(0 6)
		\move(1 4)\lvec(1 6)
		
		\move(0 6)\lvec(2 6)
		
		\move(0 6)\lvec(0 16)
		\move(1 6)\lvec(1 16)
		\move(0 16)\lvec(2 16)
		\move(0 14)\lvec(2 14)
		\move(0 13)\lvec(2 13)
		\move(0 11)\lvec(2 11)
		\move(0 9)\lvec(2 9)
		\move(0 7)\lvec(2 7)
		
		\move(0 16)\lvec(0 18)
		\move(1 16)\lvec(1 18)
		\move(0 18)\lvec(2 18)
		\move(0 18)\lvec(0 20)
		\move(1 18)\lvec(1 20)
		\move(2 0)\lvec(2 20)
		\move(0 20)\lvec(2 20)
		
		\htext(0.5 -0.3){$\underbrace{\rule{1.5em}{0em}}$}
		\htext(0.6 -0.7){$\bar{v}_{1}$}
		
		\htext(1.5 -0.3){$\underbrace{\rule{1.5em}{0em}}$}
		\htext(1.7 -0.7){$\bar{u}_{1}$}
		\esegment

		\move(16 0)
		\bsegment
		\move(0 0)\lvec(1 0)
		\move(0 2)\lvec(1 2)
		\move(0 4)\lvec(1 4)
		\move(0 6)\lvec(1 6)
		\move(0 7)\lvec(1 7)
		\move(0 9)\lvec(1 9)
		\move(0 11)\lvec(1 11)
		\move(0 13)\lvec(1 13)
		\move(0 14)\lvec(1 14)
		\move(0 16)\lvec(1 16)
		\move(0 18)\lvec(1 18)
		\move(0 20)\lvec(1 20)
		\move(1 0)\lvec(1 20)
		
		\move(0 0)\lvec(1 2)
		
		\htext(0.7 0.5){$0$}	
		\htext(0.35 1.55){$0$}
		\htext(0.5 3){$1$}
		\htext(0.5 5){$2$}
		\vtext(0.5 6.57){$\cdots$}
		\htext(0.5 8){$n$-$1$}
		\htext(0.5 10){$n$}
		\htext(0.5 12){$n$-$1$}
		\vtext(0.5 13.57){$\cdots$}
		\htext(0.5 15){$3$}
		\htext(0.5 17){$2$}
		\htext(0.5 19){$1$}
		\esegment
	\end{texdraw}
\end{center}

\vskip 2mm

For type $D_{n+1}^{(2)}$, let $C_{(x_1,\cdots,x_n,x_0,\bar{x}_n,\cdots,\bar{x}_1)}$ be a virtual Young column. We choose $u_i,v_i, \bar{u}_i,\bar{v}_i\in \widehat{\mathbf Z}_{\geq 0}$ such that 
\begin{equation*}
	\bar{u}_i-\bar{v}_i=\bar{x}_i,\  u_i-v_i=x_i,\  u_{n+1}-v_{n+1}=-\sum_{i=1}^{n}(x_i+\bar{x}_i)-x_0.
\end{equation*}

Then the corresponding extended Young column is given as 

\vskip 10mm

\begin{center}
	\begin{texdraw}
		\fontsize{3}{3}\selectfont
		\drawdim em
		\setunitscale 1.7
		\move(0 0)
		\bsegment
		\move(0 0)\lvec(0 2)
		\move(0 0)\lvec(2 0)
		\move(0 2)\lvec(2 2)
		\move(1 0)\lvec(1 2)
		\htext(0.5 -0.3){$\underbrace{\rule{1.5em}{0em}}$}
		\htext(0.5 -0.7){$v_{n\!+\!1}$}
		
		\htext(1.5 -0.3){$\underbrace{\rule{1.5em}{0em}}$}
		\htext(1.6 -0.7){$u_{n\!+\!1}$}
		\esegment
		
		\move(2 0)
		\bsegment
		\move(0 0)\lvec(0 2)
		\move(0 0)\lvec(2 0)
		\move(0 2)\lvec(2 2)
		\move(1 0)\lvec(1 2)
		\htext(0.5 -0.3){$\underbrace{\rule{1.5em}{0em}}$}
		\htext(0.6 -0.7){$v_{1}$}
		
		\htext(1.5 -0.3){$\underbrace{\rule{1.5em}{0em}}$}
		\htext(1.7 -0.7){$u_{1}$}
		\esegment
		
		\move(4 0)
		\bsegment
		\move(0 0)\lvec(0 2)
		\move(0 0)\lvec(2 0)
		\move(0 2)\lvec(2 2)
		\move(1 0)\lvec(1 2)
		
		\move(0 2)\lvec(0 4)
		\move(1 2)\lvec(1 4)
		\move(0 4)\lvec(2 4)
		
		\htext(0.5 -0.3){$\underbrace{\rule{1.5em}{0em}}$}
		\htext(0.6 -0.7){$v_{2}$}
		
		\htext(1.5 -0.3){$\underbrace{\rule{1.5em}{0em}}$}
		\htext(1.7 -0.7){$u_{2}$}
		\esegment
		
		\move(6 0)
		\bsegment
		\move(0 0)\lvec(0 2)
		\move(0 0)\lvec(2 0)
		\move(0 2)\lvec(2 2)
		\move(1 0)\lvec(1 2)
		
		\move(0 2)\lvec(0 4)
		\move(1 2)\lvec(1 4)
		\move(0 4)\lvec(2 4)
		
		\move(0 4)\lvec(0 6)
		\move(1 4)\lvec(1 6)
		\move(2 0)\lvec(2 6)
		\move(0 6)\lvec(2 6)
		\htext(0.5 -0.3){$\underbrace{\rule{1.5em}{0em}}$}
		\htext(0.6 -0.7){$v_{3}$}
		
		\htext(1.5 -0.3){$\underbrace{\rule{1.5em}{0em}}$}
		\htext(1.7 -0.7){$u_{3}$}
		\htext(3 1){\fontsize{8}{8}\selectfont$\cdots$}
		\esegment
		
		\move(10 0)
		\bsegment
		\move(0 0)\lvec(0 2)
		\move(0 0)\lvec(2 0)
		\move(0 2)\lvec(2 2)
		\move(1 0)\lvec(1 2)
		
		\move(0 2)\lvec(0 4)
		\move(1 2)\lvec(1 4)
		\move(0 4)\lvec(2 4)
		
		\move(0 4)\lvec(0 6)
		\move(1 4)\lvec(1 6)
		
		\move(0 6)\lvec(2 6)
		
		\move(0 6)\lvec(0 16)
		\move(1 6)\lvec(1 16)
		\move(0 16)\lvec(2 16)
		\move(0 14)\lvec(2 14)
		\move(0 13)\lvec(2 13)
		\move(0 11)\lvec(2 11)
		\move(0 9)\lvec(2 9)
		\move(0 7)\lvec(2 7)
		\htext(0.5 -0.3){$\underbrace{\rule{1.5em}{0em}}$}
		\htext(0.6 -0.7){$\bar{v}_{3}$}
		
		\htext(1.5 -0.3){$\underbrace{\rule{1.5em}{0em}}$}
		\htext(1.7 -0.7){$\bar{u}_{3}$}
		
		\esegment
		
		\move(12 0)
		\bsegment
		\move(0 0)\lvec(0 2)
		\move(0 0)\lvec(2 0)
		\move(0 2)\lvec(2 2)
		\move(1 0)\lvec(1 2)
		
		\move(0 2)\lvec(0 4)
		\move(1 2)\lvec(1 4)
		\move(0 4)\lvec(2 4)
		
		\move(0 4)\lvec(0 6)
		\move(1 4)\lvec(1 6)
		
		\move(0 6)\lvec(2 6)
		
		\move(0 6)\lvec(0 16)
		\move(1 6)\lvec(1 16)
		\move(0 16)\lvec(2 16)
		\move(0 14)\lvec(2 14)
		\move(0 13)\lvec(2 13)
		\move(0 11)\lvec(2 11)
		\move(0 9)\lvec(2 9)
		\move(0 7)\lvec(2 7)
		
		\move(0 16)\lvec(0 18)
		\move(1 16)\lvec(1 18)
		\move(0 18)\lvec(2 18)
		
		\htext(0.5 -0.3){$\underbrace{\rule{1.5em}{0em}}$}
		\htext(0.6 -0.7){$\bar{v}_{2}$}
		
		\htext(1.5 -0.3){$\underbrace{\rule{1.5em}{0em}}$}
		\htext(1.7 -0.7){$\bar{u}_{2}$}
		\esegment
		
		\move(14 0)
		\bsegment
		\move(0 0)\lvec(0 2)
		\move(0 0)\lvec(2 0)
		\move(0 2)\lvec(2 2)
		\move(1 0)\lvec(1 2)
		
		\move(0 2)\lvec(0 4)
		\move(1 2)\lvec(1 4)
		\move(0 4)\lvec(2 4)
		
		\move(0 4)\lvec(0 6)
		\move(1 4)\lvec(1 6)
		
		\move(0 6)\lvec(2 6)
		
		\move(0 6)\lvec(0 16)
		\move(1 6)\lvec(1 16)
		\move(0 16)\lvec(2 16)
		\move(0 14)\lvec(2 14)
		\move(0 13)\lvec(2 13)
		\move(0 11)\lvec(2 11)
		\move(0 9)\lvec(2 9)
		\move(0 7)\lvec(2 7)
		
		\move(0 16)\lvec(0 18)
		\move(1 16)\lvec(1 18)
		\move(0 18)\lvec(2 18)
		\move(0 18)\lvec(0 20)
		\move(1 18)\lvec(1 20)
		\move(2 0)\lvec(2 20)
		\move(0 20)\lvec(2 20)
		
		\htext(0.5 -0.3){$\underbrace{\rule{1.5em}{0em}}$}
		\htext(0.6 -0.7){$\bar{v}_{1}$}
		
		\htext(1.5 -0.3){$\underbrace{\rule{1.5em}{0em}}$}
		\htext(1.7 -0.7){$\bar{u}_{1}$}
		\esegment

		\move(16 0)
		\bsegment
		\move(0 0)\lvec(1 0)
		\move(0 2)\lvec(1 2)
		\move(0 4)\lvec(1 4)
		\move(0 6)\lvec(1 6)
		\move(0 7)\lvec(1 7)
		\move(0 9)\lvec(1 9)
		\move(0 11)\lvec(1 11)
		\move(0 13)\lvec(1 13)
		\move(0 14)\lvec(1 14)
		\move(0 16)\lvec(1 16)
		\move(0 18)\lvec(1 18)
		\move(0 20)\lvec(1 20)
		\move(1 0)\lvec(1 20)
		
		\move(0 0)\lvec(1 2)
		
		\htext(0.7 0.5){$0$}	
		\htext(0.35 1.55){$0$}
		\htext(0.5 3){$1$}
		\htext(0.5 5){$2$}
		\vtext(0.5 6.57){$\cdots$}
		\htext(0.5 8){$n$-$1$}
		
		\move(0 9)\lvec(1 11)
		\htext(0.7 9.5){$n$}	
		\htext(0.35 10.55){$n$}
		
		\htext(0.5 12){$n$-$1$}
		\vtext(0.5 13.57){$\cdots$}
		\htext(0.5 15){$3$}
		\htext(0.5 17){$2$}
		\htext(0.5 19){$1$}
		\esegment
	\end{texdraw}
\end{center}

\vskip 6mm

For type $A_{2n-1}^{(2)}$, let $C_{(x_1,\cdots,x_n,\bar{x}_n,\cdots,\bar{x}_1)}$ be a virtual Young column. We choose $u_i,v_i, \bar{u}_i,\bar{v}_i\in \widehat{\mathbf Z}_{\geq 0}$ such that 

\begin{equation*}
	\bar{u}_i-\bar{v}_i=\bar{x}_i,\  u_i-v_i=x_i.
\end{equation*}

Then the corresponding extended Young Column is given as

\vskip 10mm

\begin{center}
	\begin{texdraw}
		\fontsize{3}{3}\selectfont
		\drawdim em
		\setunitscale 1.7
		\move(-2 0)
		\bsegment
		\move(0 0)\lvec(0 2)
		\move(0 0)\lvec(2 0)
		\move(0 2)\lvec(2 2)
		\move(1 0)\lvec(1 2)
		\htext(0.5 -0.3){$\underbrace{\rule{1.5em}{0em}}$}
		\htext(0.6 -0.7){$v_{1}$}
		
		\htext(1.5 -0.3){$\underbrace{\rule{1.5em}{0em}}$}
		\htext(1.7 -0.7){$u_{1}$}
		\esegment
		\move(0 0)
		\bsegment
		\move(0 0)\lvec(0 2)
		\move(0 0)\lvec(2 0)
		\move(0 2)\lvec(2 2)
		\move(1 0)\lvec(1 2)
		\htext(0.5 -0.3){$\underbrace{\rule{1.5em}{0em}}$}
		\htext(0.6 -0.7){$\bar{v}_{1}$}
		
		\htext(1.5 -0.3){$\underbrace{\rule{1.5em}{0em}}$}
		\htext(1.7 -0.7){$\bar{u}_{1}$}
		\esegment
		
		\move(2 0)
		\bsegment
		\move(0 0)\lvec(0 2)
		\move(0 0)\lvec(2 0)
		\move(0 2)\lvec(2 2)
		\move(1 0)\lvec(1 2)
		\htext(0.5 -0.3){$\underbrace{\rule{1.5em}{0em}}$}
		\htext(0.6 -0.7){$v_{2}$}
		
		\htext(1.5 -0.3){$\underbrace{\rule{1.5em}{0em}}$}
		\htext(1.7 -0.7){$u_{2}$}
		\esegment
		
		\move(4 0)
		\bsegment
		\move(0 0)\lvec(0 2)
		\move(0 0)\lvec(2 0)
		\move(0 2)\lvec(2 2)
		\move(1 0)\lvec(1 2)
		
		\move(0 2)\lvec(0 4)
		\move(1 2)\lvec(1 4)
		\move(0 4)\lvec(2 4)
		
		\htext(0.5 -0.3){$\underbrace{\rule{1.5em}{0em}}$}
		\htext(0.6 -0.7){$v_{3}$}
		
		\htext(1.5 -0.3){$\underbrace{\rule{1.5em}{0em}}$}
		\htext(1.7 -0.7){$u_{3}$}
		\esegment
		
		\move(6 0)
		\bsegment
		\move(0 0)\lvec(0 2)
		\move(0 0)\lvec(2 0)
		\move(0 2)\lvec(2 2)
		\move(1 0)\lvec(1 2)
		
		\move(0 2)\lvec(0 4)
		\move(1 2)\lvec(1 4)
		\move(0 4)\lvec(2 4)
		
		\move(0 4)\lvec(0 6)
		\move(1 4)\lvec(1 6)
		\move(2 0)\lvec(2 6)
		\move(0 6)\lvec(2 6)
		\htext(0.5 -0.3){$\underbrace{\rule{1.5em}{0em}}$}
		\htext(0.6 -0.7){$v_{4}$}
		
		\htext(1.5 -0.3){$\underbrace{\rule{1.5em}{0em}}$}
		\htext(1.7 -0.7){$u_{4}$}
		\htext(3 1){\fontsize{8}{8}\selectfont$\cdots$}
		\esegment
		
		\move(10 0)
		\bsegment
		\move(0 0)\lvec(0 2)
		\move(0 0)\lvec(2 0)
		\move(0 2)\lvec(2 2)
		\move(1 0)\lvec(1 2)
		
		\move(0 2)\lvec(0 4)
		\move(1 2)\lvec(1 4)
		\move(0 4)\lvec(2 4)
		
		\move(0 4)\lvec(0 6)
		\move(1 4)\lvec(1 6)
		
		\move(0 6)\lvec(2 6)
		
		\move(0 6)\lvec(0 16)
		\move(1 6)\lvec(1 16)
		\move(0 16)\lvec(2 16)
		\move(0 14)\lvec(2 14)
		\move(0 13)\lvec(2 13)
		\move(0 11)\lvec(2 11)
		\move(0 9)\lvec(2 9)
		\move(0 7)\lvec(2 7)
		\htext(0.5 -0.3){$\underbrace{\rule{1.5em}{0em}}$}
		\htext(0.6 -0.7){$\bar{v}_{4}$}
		
		\htext(1.5 -0.3){$\underbrace{\rule{1.5em}{0em}}$}
		\htext(1.7 -0.7){$\bar{u}_{4}$}
		
		\esegment
		
		\move(12 0)
		\bsegment
		\move(0 0)\lvec(0 2)
		\move(0 0)\lvec(2 0)
		\move(0 2)\lvec(2 2)
		\move(1 0)\lvec(1 2)
		
		\move(0 2)\lvec(0 4)
		\move(1 2)\lvec(1 4)
		\move(0 4)\lvec(2 4)
		
		\move(0 4)\lvec(0 6)
		\move(1 4)\lvec(1 6)
		
		\move(0 6)\lvec(2 6)
		
		\move(0 6)\lvec(0 16)
		\move(1 6)\lvec(1 16)
		\move(0 16)\lvec(2 16)
		\move(0 14)\lvec(2 14)
		\move(0 13)\lvec(2 13)
		\move(0 11)\lvec(2 11)
		\move(0 9)\lvec(2 9)
		\move(0 7)\lvec(2 7)
		
		\move(0 16)\lvec(0 18)
		\move(1 16)\lvec(1 18)
		\move(0 18)\lvec(2 18)
		
		\htext(0.5 -0.3){$\underbrace{\rule{1.5em}{0em}}$}
		\htext(0.6 -0.7){$\bar{v}_{3}$}
		
		\htext(1.5 -0.3){$\underbrace{\rule{1.5em}{0em}}$}
		\htext(1.7 -0.7){$\bar{u}_{3}$}
		\esegment
		
		\move(14 0)
		\bsegment
		\move(0 0)\lvec(0 2)
		\move(0 0)\lvec(2 0)
		\move(0 2)\lvec(2 2)
		\move(1 0)\lvec(1 2)
		
		\move(0 2)\lvec(0 4)
		\move(1 2)\lvec(1 4)
		\move(0 4)\lvec(2 4)
		
		\move(0 4)\lvec(0 6)
		\move(1 4)\lvec(1 6)
		
		\move(0 6)\lvec(2 6)
		
		\move(0 6)\lvec(0 16)
		\move(1 6)\lvec(1 16)
		\move(0 16)\lvec(2 16)
		\move(0 14)\lvec(2 14)
		\move(0 13)\lvec(2 13)
		\move(0 11)\lvec(2 11)
		\move(0 9)\lvec(2 9)
		\move(0 7)\lvec(2 7)
		
		\move(0 16)\lvec(0 18)
		\move(1 16)\lvec(1 18)
		\move(0 18)\lvec(2 18)
		\move(0 18)\lvec(0 20)
		\move(1 18)\lvec(1 20)
		\move(2 0)\lvec(2 20)
		\move(0 20)\lvec(2 20)
		
		\htext(0.5 -0.3){$\underbrace{\rule{1.5em}{0em}}$}
		\htext(0.6 -0.7){$\bar{v}_{2}$}
		
		\htext(1.5 -0.3){$\underbrace{\rule{1.5em}{0em}}$}
		\htext(1.7 -0.7){$\bar{u}_{2}$}
		\esegment

		\move(16 0)
		\bsegment
		\move(0 0)\lvec(1 0)
		\move(0 2)\lvec(1 2)
		\move(0 4)\lvec(1 4)
		\move(0 6)\lvec(1 6)
		\move(0 7)\lvec(1 7)
		\move(0 9)\lvec(1 9)
		\move(0 11)\lvec(1 11)
		\move(0 13)\lvec(1 13)
		\move(0 14)\lvec(1 14)
		\move(0 16)\lvec(1 16)
		\move(0 18)\lvec(1 18)
		\move(0 20)\lvec(1 20)
		\move(1 0)\lvec(1 20)
		
		\move(0 0)\lvec(1 2)
		
		\htext(0.7 0.5){$1$}	
		\htext(0.35 1.55){$0$}
		\htext(0.5 3){$2$}
		\htext(0.5 5){$3$}
		\vtext(0.5 6.57){$\cdots$}
		\htext(0.5 8){$n$-$1$}
		\htext(0.5 10){$n$}
		\htext(0.5 12){$n$-$1$}
		\vtext(0.5 13.57){$\cdots$}
		\htext(0.5 15){$4$}
		\htext(0.5 17){$3$}
		\htext(0.5 19){$2$}
		\esegment
	\end{texdraw}
\end{center}

\vskip 2mm 

For type $D_{n}^{(1)}$, let $C_{(x_1,\cdots,x_n,\bar{x}_n,\cdots,\bar{x}_1)}$ be a virtual Young column. We choose $u_i,v_i, \bar{u}_i,\bar{v}_i\in \widehat{\mathbf Z}_{\geq 0}$ such that 

\begin{equation*}
	\bar{u}_i-\bar{v}_i=\bar{x}_i,\  u_i-v_i=x_i.
\end{equation*}

Then the corresponding  extended Young column is given as 

\vskip 10mm

\begin{center}
	\begin{texdraw}
		\fontsize{3}{3}\selectfont
		\drawdim em
		\setunitscale 1.7
		\move(-2 0)
		\bsegment
		\move(0 0)\lvec(0 2)
		\move(0 0)\lvec(2 0)
		\move(0 2)\lvec(2 2)
		\move(1 0)\lvec(1 2)
		\htext(0.5 -0.3){$\underbrace{\rule{1.5em}{0em}}$}
		\htext(0.6 -0.7){$v_{1}$}
		
		\htext(1.5 -0.3){$\underbrace{\rule{1.5em}{0em}}$}
		\htext(1.7 -0.7){$u_{1}$}
		\esegment
		\move(0 0)
		\bsegment
		\move(0 0)\lvec(0 2)
		\move(0 0)\lvec(2 0)
		\move(0 2)\lvec(2 2)
		\move(1 0)\lvec(1 2)
		\htext(0.5 -0.3){$\underbrace{\rule{1.5em}{0em}}$}
		\htext(0.6 -0.7){$\bar{v}_{1}$}
		
		\htext(1.5 -0.3){$\underbrace{\rule{1.5em}{0em}}$}
		\htext(1.7 -0.7){$\bar{u}_{1}$}
		\esegment
		
		\move(2 0)
		\bsegment
		\move(0 0)\lvec(0 2)
		\move(0 0)\lvec(2 0)
		\move(0 2)\lvec(2 2)
		\move(1 0)\lvec(1 2)
		\htext(0.5 -0.3){$\underbrace{\rule{1.5em}{0em}}$}
		\htext(0.6 -0.7){$v_{2}$}
		
		\htext(1.5 -0.3){$\underbrace{\rule{1.5em}{0em}}$}
		\htext(1.7 -0.7){$u_{2}$}
		\esegment
		
		\move(4 0)
		\bsegment
		\move(0 0)\lvec(0 2)
		\move(0 0)\lvec(2 0)
		\move(0 2)\lvec(2 2)
		\move(1 0)\lvec(1 2)
		
		\move(0 2)\lvec(0 4)
		\move(1 2)\lvec(1 4)
		\move(0 4)\lvec(2 4)
		
		\htext(0.5 -0.3){$\underbrace{\rule{1.5em}{0em}}$}
		\htext(0.6 -0.7){$v_{3}$}
		
		\htext(1.5 -0.3){$\underbrace{\rule{1.5em}{0em}}$}
		\htext(1.7 -0.7){$u_{3}$}
		\esegment
		
		\move(6 0)
		\bsegment
		\move(0 0)\lvec(0 2)
		\move(0 0)\lvec(2 0)
		\move(0 2)\lvec(2 2)
		\move(1 0)\lvec(1 2)
		
		\move(0 2)\lvec(0 4)
		\move(1 2)\lvec(1 4)
		\move(0 4)\lvec(2 4)
		
		\move(0 4)\lvec(0 6)
		\move(1 4)\lvec(1 6)
		\move(2 0)\lvec(2 6)
		\move(0 6)\lvec(2 6)
		\htext(0.5 -0.3){$\underbrace{\rule{1.5em}{0em}}$}
		\htext(0.6 -0.7){$v_{4}$}
		
		\htext(1.5 -0.3){$\underbrace{\rule{1.5em}{0em}}$}
		\htext(1.7 -0.7){$u_{4}$}
		\htext(3 1){\fontsize{8}{8}\selectfont$\cdots$}
		\esegment
		
		\move(10 0)
		\bsegment
		\move(0 0)\lvec(0 2)
		\move(0 0)\lvec(2 0)
		\move(0 2)\lvec(2 2)
		\move(1 0)\lvec(1 2)
		
		\move(0 2)\lvec(0 4)
		\move(1 2)\lvec(1 4)
		\move(0 4)\lvec(2 4)
		
		\move(0 4)\lvec(0 6)
		\move(1 4)\lvec(1 6)
		
		\move(0 6)\lvec(2 6)
		
		\move(0 6)\lvec(0 16)
		\move(1 6)\lvec(1 16)
		\move(0 16)\lvec(2 16)
		\move(0 14)\lvec(2 14)
		\move(0 13)\lvec(2 13)
		\move(0 11)\lvec(2 11)
		\move(0 9)\lvec(2 9)
		\move(0 7)\lvec(2 7)
		\htext(0.5 -0.3){$\underbrace{\rule{1.5em}{0em}}$}
		\htext(0.6 -0.7){$\bar{v}_{4}$}
		
		\htext(1.5 -0.3){$\underbrace{\rule{1.5em}{0em}}$}
		\htext(1.7 -0.7){$\bar{u}_{4}$}
		
		\esegment
		
		\move(12 0)
		\bsegment
		\move(0 0)\lvec(0 2)
		\move(0 0)\lvec(2 0)
		\move(0 2)\lvec(2 2)
		\move(1 0)\lvec(1 2)
		
		\move(0 2)\lvec(0 4)
		\move(1 2)\lvec(1 4)
		\move(0 4)\lvec(2 4)
		
		\move(0 4)\lvec(0 6)
		\move(1 4)\lvec(1 6)
		
		\move(0 6)\lvec(2 6)
		
		\move(0 6)\lvec(0 16)
		\move(1 6)\lvec(1 16)
		\move(0 16)\lvec(2 16)
		\move(0 14)\lvec(2 14)
		\move(0 13)\lvec(2 13)
		\move(0 11)\lvec(2 11)
		\move(0 9)\lvec(2 9)
		\move(0 7)\lvec(2 7)
		
		\move(0 16)\lvec(0 18)
		\move(1 16)\lvec(1 18)
		\move(0 18)\lvec(2 18)
		
		\htext(0.5 -0.3){$\underbrace{\rule{1.5em}{0em}}$}
		\htext(0.6 -0.7){$\bar{v}_{3}$}
		
		\htext(1.5 -0.3){$\underbrace{\rule{1.5em}{0em}}$}
		\htext(1.7 -0.7){$\bar{u}_{3}$}
		\esegment
		
		\move(14 0)
		\bsegment
		\move(0 0)\lvec(0 2)
		\move(0 0)\lvec(2 0)
		\move(0 2)\lvec(2 2)
		\move(1 0)\lvec(1 2)
		
		\move(0 2)\lvec(0 4)
		\move(1 2)\lvec(1 4)
		\move(0 4)\lvec(2 4)
		
		\move(0 4)\lvec(0 6)
		\move(1 4)\lvec(1 6)
		
		\move(0 6)\lvec(2 6)
		
		\move(0 6)\lvec(0 16)
		\move(1 6)\lvec(1 16)
		\move(0 16)\lvec(2 16)
		\move(0 14)\lvec(2 14)
		\move(0 13)\lvec(2 13)
		\move(0 11)\lvec(2 11)
		\move(0 9)\lvec(2 9)
		\move(0 7)\lvec(2 7)
		
		\move(0 16)\lvec(0 18)
		\move(1 16)\lvec(1 18)
		\move(0 18)\lvec(2 18)
		\move(0 18)\lvec(0 20)
		\move(1 18)\lvec(1 20)
		\move(2 0)\lvec(2 20)
		\move(0 20)\lvec(2 20)
		
		\htext(0.5 -0.3){$\underbrace{\rule{1.5em}{0em}}$}
		\htext(0.6 -0.7){$\bar{v}_{2}$}
		
		\htext(1.5 -0.3){$\underbrace{\rule{1.5em}{0em}}$}
		\htext(1.7 -0.7){$\bar{u}_{2}$}
		\esegment

		\move(16 0)
		\bsegment
		\move(0 0)\lvec(1 0)
		\move(0 2)\lvec(1 2)
		\move(0 4)\lvec(1 4)
		\move(0 6)\lvec(1 6)
		\move(0 7)\lvec(1 7)
		\move(0 9)\lvec(1 9)
		\move(0 11)\lvec(1 11)
		\move(0 13)\lvec(1 13)
		\move(0 14)\lvec(1 14)
		\move(0 16)\lvec(1 16)
		\move(0 18)\lvec(1 18)
		\move(0 20)\lvec(1 20)
		\move(1 0)\lvec(1 20)
		
		\move(0 0)\lvec(1 2)
		
		\htext(0.7 0.5){$1$}	
		\htext(0.35 1.55){$0$}
		\htext(0.5 3){$2$}
		\htext(0.5 5){$3$}
		\vtext(0.5 6.57){$\cdots$}
		\htext(0.5 8){$n$-$2$}
		\move(0 9)\lvec(1 11)
		\htext(0.7 9.5){$n$}	
		\htext(0.35 10.55){$n$-$1$}
		
		\htext(0.5 12){$n$-$2$}
		\vtext(0.5 13.57){$\cdots$}
		\htext(0.5 15){$4$}
		\htext(0.5 17){$3$}
		\htext(0.5 19){$2$}
		\esegment
	\end{texdraw}
\end{center}

\vskip 2mm 

For type $B_{n}^{(1)}$, let $C_{(x_1,\cdots,x_n,x_0,\bar{x}_n,\cdots,\bar{x}_1)}$ be a virtual Young column. We choose $u_i,v_i, \bar{u}_i,\bar{v}_i\in \widehat{\mathbf Z}_{\geq 0}$ such that 

\begin{equation*}
	\bar{u}_i-\bar{v}_i=\bar{x}_i,\  u_i-v_i=x_i.
\end{equation*}

Then the corresponding  extended Young column is given as 

\vskip 10mm

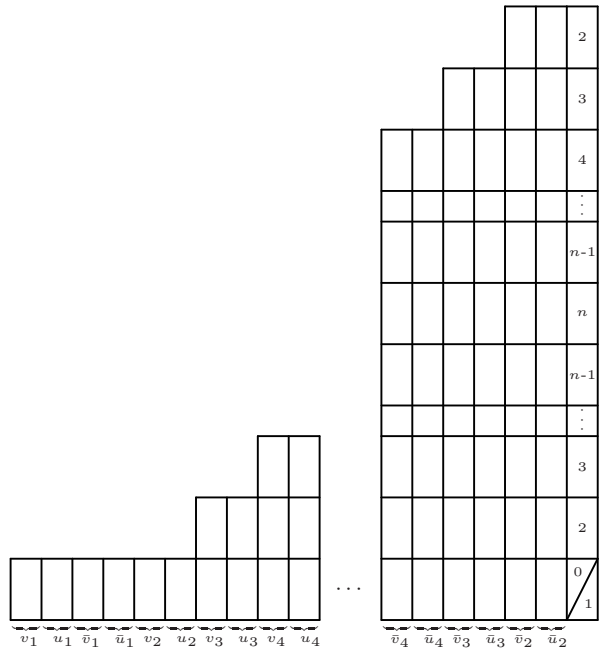
\begin{center}
	\begin{texdraw}
		\fontsize{3}{3}\selectfont
		\drawdim em
		\setunitscale 1.7
		\move(-2 0)
		\bsegment
		\move(0 0)\lvec(0 2)
		\move(0 0)\lvec(2 0)
		\move(0 2)\lvec(2 2)
		\move(1 0)\lvec(1 2)
		\htext(0.5 -0.3){$\underbrace{\rule{1.5em}{0em}}$}
		\htext(0.6 -0.7){$v_{1}$}
		
		\htext(1.5 -0.3){$\underbrace{\rule{1.5em}{0em}}$}
		\htext(1.7 -0.7){$u_{1}$}
		\esegment
		\move(0 0)
		\bsegment
		\move(0 0)\lvec(0 2)
		\move(0 0)\lvec(2 0)
		\move(0 2)\lvec(2 2)
		\move(1 0)\lvec(1 2)
		\htext(0.5 -0.3){$\underbrace{\rule{1.5em}{0em}}$}
		\htext(0.6 -0.7){$\bar{v}_{1}$}
		
		\htext(1.5 -0.3){$\underbrace{\rule{1.5em}{0em}}$}
		\htext(1.7 -0.7){$\bar{u}_{1}$}
		\esegment
		
		\move(2 0)
		\bsegment
		\move(0 0)\lvec(0 2)
		\move(0 0)\lvec(2 0)
		\move(0 2)\lvec(2 2)
		\move(1 0)\lvec(1 2)
		\htext(0.5 -0.3){$\underbrace{\rule{1.5em}{0em}}$}
		\htext(0.6 -0.7){$v_{2}$}
		
		\htext(1.5 -0.3){$\underbrace{\rule{1.5em}{0em}}$}
		\htext(1.7 -0.7){$u_{2}$}
		\esegment
		
		\move(4 0)
		\bsegment
		\move(0 0)\lvec(0 2)
		\move(0 0)\lvec(2 0)
		\move(0 2)\lvec(2 2)
		\move(1 0)\lvec(1 2)
		
		\move(0 2)\lvec(0 4)
		\move(1 2)\lvec(1 4)
		\move(0 4)\lvec(2 4)
		
		\htext(0.5 -0.3){$\underbrace{\rule{1.5em}{0em}}$}
		\htext(0.6 -0.7){$v_{3}$}
		
		\htext(1.5 -0.3){$\underbrace{\rule{1.5em}{0em}}$}
		\htext(1.7 -0.7){$u_{3}$}
		\esegment
		
		\move(6 0)
		\bsegment
		\move(0 0)\lvec(0 2)
		\move(0 0)\lvec(2 0)
		\move(0 2)\lvec(2 2)
		\move(1 0)\lvec(1 2)
		
		\move(0 2)\lvec(0 4)
		\move(1 2)\lvec(1 4)
		\move(0 4)\lvec(2 4)
		
		\move(0 4)\lvec(0 6)
		\move(1 4)\lvec(1 6)
		\move(2 0)\lvec(2 6)
		\move(0 6)\lvec(2 6)
		\htext(0.5 -0.3){$\underbrace{\rule{1.5em}{0em}}$}
		\htext(0.6 -0.7){$v_{4}$}
		
		\htext(1.5 -0.3){$\underbrace{\rule{1.5em}{0em}}$}
		\htext(1.7 -0.7){$u_{4}$}
		\htext(3 1){\fontsize{8}{8}\selectfont$\cdots$}
		\esegment
		
		\move(10 0)
		\bsegment
		\move(0 0)\lvec(0 2)
		\move(0 0)\lvec(2 0)
		\move(0 2)\lvec(2 2)
		\move(1 0)\lvec(1 2)
		
		\move(0 2)\lvec(0 4)
		\move(1 2)\lvec(1 4)
		\move(0 4)\lvec(2 4)
		
		\move(0 4)\lvec(0 6)
		\move(1 4)\lvec(1 6)
		
		\move(0 6)\lvec(2 6)
		
		\move(0 6)\lvec(0 16)
		\move(1 6)\lvec(1 16)
		\move(0 16)\lvec(2 16)
		\move(0 14)\lvec(2 14)
		\move(0 13)\lvec(2 13)
		\move(0 11)\lvec(2 11)
		\move(0 9)\lvec(2 9)
		\move(0 7)\lvec(2 7)
		\htext(0.5 -0.3){$\underbrace{\rule{1.5em}{0em}}$}
		\htext(0.6 -0.7){$\bar{v}_{4}$}
		
		\htext(1.5 -0.3){$\underbrace{\rule{1.5em}{0em}}$}
		\htext(1.7 -0.7){$\bar{u}_{4}$}
		
		\esegment
		
		\move(12 0)
		\bsegment
		\move(0 0)\lvec(0 2)
		\move(0 0)\lvec(2 0)
		\move(0 2)\lvec(2 2)
		\move(1 0)\lvec(1 2)
		
		\move(0 2)\lvec(0 4)
		\move(1 2)\lvec(1 4)
		\move(0 4)\lvec(2 4)
		
		\move(0 4)\lvec(0 6)
		\move(1 4)\lvec(1 6)
		
		\move(0 6)\lvec(2 6)
		
		\move(0 6)\lvec(0 16)
		\move(1 6)\lvec(1 16)
		\move(0 16)\lvec(2 16)
		\move(0 14)\lvec(2 14)
		\move(0 13)\lvec(2 13)
		\move(0 11)\lvec(2 11)
		\move(0 9)\lvec(2 9)
		\move(0 7)\lvec(2 7)
		
		\move(0 16)\lvec(0 18)
		\move(1 16)\lvec(1 18)
		\move(0 18)\lvec(2 18)
		
		\htext(0.5 -0.3){$\underbrace{\rule{1.5em}{0em}}$}
		\htext(0.6 -0.7){$\bar{v}_{3}$}
		
		\htext(1.5 -0.3){$\underbrace{\rule{1.5em}{0em}}$}
		\htext(1.7 -0.7){$\bar{u}_{3}$}
		\esegment
		
		\move(14 0)
		\bsegment
		\move(0 0)\lvec(0 2)
		\move(0 0)\lvec(2 0)
		\move(0 2)\lvec(2 2)
		\move(1 0)\lvec(1 2)
		
		\move(0 2)\lvec(0 4)
		\move(1 2)\lvec(1 4)
		\move(0 4)\lvec(2 4)
		
		\move(0 4)\lvec(0 6)
		\move(1 4)\lvec(1 6)
		
		\move(0 6)\lvec(2 6)
		
		\move(0 6)\lvec(0 16)
		\move(1 6)\lvec(1 16)
		\move(0 16)\lvec(2 16)
		\move(0 14)\lvec(2 14)
		\move(0 13)\lvec(2 13)
		\move(0 11)\lvec(2 11)
		\move(0 9)\lvec(2 9)
		\move(0 7)\lvec(2 7)
		
		\move(0 16)\lvec(0 18)
		\move(1 16)\lvec(1 18)
		\move(0 18)\lvec(2 18)
		\move(0 18)\lvec(0 20)
		\move(1 18)\lvec(1 20)
		\move(2 0)\lvec(2 20)
		\move(0 20)\lvec(2 20)
		
		\htext(0.5 -0.3){$\underbrace{\rule{1.5em}{0em}}$}
		\htext(0.6 -0.7){$\bar{v}_{2}$}
		
		\htext(1.5 -0.3){$\underbrace{\rule{1.5em}{0em}}$}
		\htext(1.7 -0.7){$\bar{u}_{2}$}
		\esegment

		\move(16 0)
		\bsegment
		\move(0 0)\lvec(1 0)
		\move(0 2)\lvec(1 2)
		\move(0 4)\lvec(1 4)
		\move(0 6)\lvec(1 6)
		\move(0 7)\lvec(1 7)
		\move(0 9)\lvec(1 9)
		\move(0 11)\lvec(1 11)
		\move(0 13)\lvec(1 13)
		\move(0 14)\lvec(1 14)
		\move(0 16)\lvec(1 16)
		\move(0 18)\lvec(1 18)
		\move(0 20)\lvec(1 20)
		\move(1 0)\lvec(1 20)
		
		\move(0 0)\lvec(1 2)
		
		\htext(0.7 0.5){$1$}	
		\htext(0.35 1.55){$0$}
		\htext(0.5 3){$2$}
		\htext(0.5 5){$3$}
		\vtext(0.5 6.57){$\cdots$}
		\htext(0.5 8){$n$-$1$}
		
		\htext(0.5 10){$n$}

		\htext(0.5 12){$n$-$1$}
		\vtext(0.5 13.57){$\cdots$}
		\htext(0.5 15){$4$}
		\htext(0.5 17){$3$}
		\htext(0.5 19){$2$}
		\esegment
	\end{texdraw}
\end{center}

\vskip 2mm

For type $C_{n}^{(1)}$, let $C_{(x_1,\cdots,x_n,\bar{x}_n,\cdots,\bar{x}_1)}$ be a virtual Young column. We choose $u_i,v_i, \bar{u}_i,\bar{v}_i \in \widehat{\mathbf Z}_{\geq 0}$ such that
 
\begin{equation*}
	\bar{u}_i-\bar{v}_i=\bar{x}_i,\  u_i-v_i=x_i,\  u_{n+1}-v_{n+1}=-\sum_{i=1}^{n}(x_i+\bar{x}_i).
\end{equation*}

Then the corresponding extended Young column is given as

\vskip 10mm

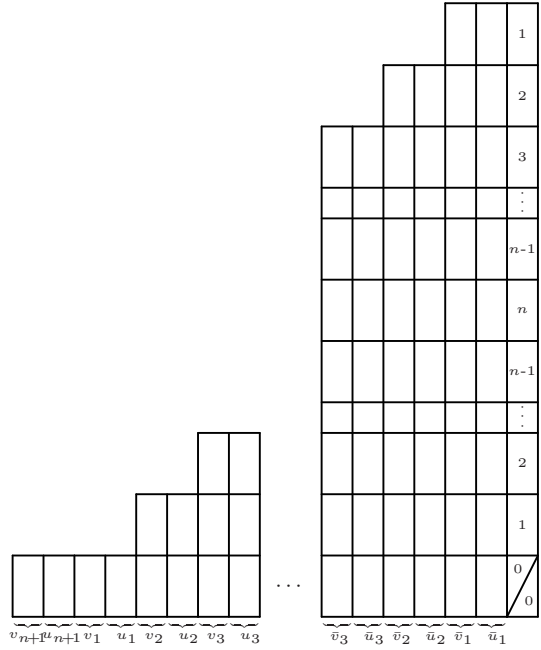
\begin{center}
	\begin{texdraw}
		\fontsize{3}{3}\selectfont
		\drawdim em
		\setunitscale 1.7
		\move(0 0)
		\bsegment
		\move(0 0)\lvec(0 2)
		\move(0 0)\lvec(2 0)
		\move(0 2)\lvec(2 2)
		\move(1 0)\lvec(1 2)
		\htext(0.5 -0.3){$\underbrace{\rule{1.5em}{0em}}$}
		\htext(0.5 -0.7){$v_{n\!+\!1}$}
		
		\htext(1.5 -0.3){$\underbrace{\rule{1.5em}{0em}}$}
		\htext(1.6 -0.7){$u_{n\!+\!1}$}
		\esegment
		
		\move(2 0)
		\bsegment
		\move(0 0)\lvec(0 2)
		\move(0 0)\lvec(2 0)
		\move(0 2)\lvec(2 2)
		\move(1 0)\lvec(1 2)
		\htext(0.5 -0.3){$\underbrace{\rule{1.5em}{0em}}$}
		\htext(0.6 -0.7){$v_{1}$}
		
		\htext(1.5 -0.3){$\underbrace{\rule{1.5em}{0em}}$}
		\htext(1.7 -0.7){$u_{1}$}
		\esegment
		
		\move(4 0)
		\bsegment
		\move(0 0)\lvec(0 2)
		\move(0 0)\lvec(2 0)
		\move(0 2)\lvec(2 2)
		\move(1 0)\lvec(1 2)
		
		\move(0 2)\lvec(0 4)
		\move(1 2)\lvec(1 4)
		\move(0 4)\lvec(2 4)
		
		\htext(0.5 -0.3){$\underbrace{\rule{1.5em}{0em}}$}
		\htext(0.6 -0.7){$v_{2}$}
		
		\htext(1.5 -0.3){$\underbrace{\rule{1.5em}{0em}}$}
		\htext(1.7 -0.7){$u_{2}$}
		\esegment
		
		\move(6 0)
		\bsegment
		\move(0 0)\lvec(0 2)
		\move(0 0)\lvec(2 0)
		\move(0 2)\lvec(2 2)
		\move(1 0)\lvec(1 2)
		
		\move(0 2)\lvec(0 4)
		\move(1 2)\lvec(1 4)
		\move(0 4)\lvec(2 4)
		
		\move(0 4)\lvec(0 6)
		\move(1 4)\lvec(1 6)
		\move(2 0)\lvec(2 6)
		\move(0 6)\lvec(2 6)
		\htext(0.5 -0.3){$\underbrace{\rule{1.5em}{0em}}$}
		\htext(0.6 -0.7){$v_{3}$}
		
		\htext(1.5 -0.3){$\underbrace{\rule{1.5em}{0em}}$}
		\htext(1.7 -0.7){$u_{3}$}
		\htext(3 1){\fontsize{8}{8}\selectfont$\cdots$}
		\esegment
		
		\move(10 0)
		\bsegment
		\move(0 0)\lvec(0 2)
		\move(0 0)\lvec(2 0)
		\move(0 2)\lvec(2 2)
		\move(1 0)\lvec(1 2)
		
		\move(0 2)\lvec(0 4)
		\move(1 2)\lvec(1 4)
		\move(0 4)\lvec(2 4)
		
		\move(0 4)\lvec(0 6)
		\move(1 4)\lvec(1 6)
		
		\move(0 6)\lvec(2 6)
		
		\move(0 6)\lvec(0 16)
		\move(1 6)\lvec(1 16)
		\move(0 16)\lvec(2 16)
		\move(0 14)\lvec(2 14)
		\move(0 13)\lvec(2 13)
		\move(0 11)\lvec(2 11)
		\move(0 9)\lvec(2 9)
		\move(0 7)\lvec(2 7)
		\htext(0.5 -0.3){$\underbrace{\rule{1.5em}{0em}}$}
		\htext(0.6 -0.7){$\bar{v}_{3}$}
		
		\htext(1.5 -0.3){$\underbrace{\rule{1.5em}{0em}}$}
		\htext(1.7 -0.7){$\bar{u}_{3}$}
		
		\esegment
		
		\move(12 0)
		\bsegment
		\move(0 0)\lvec(0 2)
		\move(0 0)\lvec(2 0)
		\move(0 2)\lvec(2 2)
		\move(1 0)\lvec(1 2)
		
		\move(0 2)\lvec(0 4)
		\move(1 2)\lvec(1 4)
		\move(0 4)\lvec(2 4)
		
		\move(0 4)\lvec(0 6)
		\move(1 4)\lvec(1 6)
		
		\move(0 6)\lvec(2 6)
		
		\move(0 6)\lvec(0 16)
		\move(1 6)\lvec(1 16)
		\move(0 16)\lvec(2 16)
		\move(0 14)\lvec(2 14)
		\move(0 13)\lvec(2 13)
		\move(0 11)\lvec(2 11)
		\move(0 9)\lvec(2 9)
		\move(0 7)\lvec(2 7)
		
		\move(0 16)\lvec(0 18)
		\move(1 16)\lvec(1 18)
		\move(0 18)\lvec(2 18)
		
		\htext(0.5 -0.3){$\underbrace{\rule{1.5em}{0em}}$}
		\htext(0.6 -0.7){$\bar{v}_{2}$}
		
		\htext(1.5 -0.3){$\underbrace{\rule{1.5em}{0em}}$}
		\htext(1.7 -0.7){$\bar{u}_{2}$}
		\esegment
		
		\move(14 0)
		\bsegment
		\move(0 0)\lvec(0 2)
		\move(0 0)\lvec(2 0)
		\move(0 2)\lvec(2 2)
		\move(1 0)\lvec(1 2)
		
		\move(0 2)\lvec(0 4)
		\move(1 2)\lvec(1 4)
		\move(0 4)\lvec(2 4)
		
		\move(0 4)\lvec(0 6)
		\move(1 4)\lvec(1 6)
		
		\move(0 6)\lvec(2 6)
		
		\move(0 6)\lvec(0 16)
		\move(1 6)\lvec(1 16)
		\move(0 16)\lvec(2 16)
		\move(0 14)\lvec(2 14)
		\move(0 13)\lvec(2 13)
		\move(0 11)\lvec(2 11)
		\move(0 9)\lvec(2 9)
		\move(0 7)\lvec(2 7)
		
		\move(0 16)\lvec(0 18)
		\move(1 16)\lvec(1 18)
		\move(0 18)\lvec(2 18)
		\move(0 18)\lvec(0 20)
		\move(1 18)\lvec(1 20)
		\move(2 0)\lvec(2 20)
		\move(0 20)\lvec(2 20)
		
		\htext(0.5 -0.3){$\underbrace{\rule{1.5em}{0em}}$}
		\htext(0.6 -0.7){$\bar{v}_{1}$}
		
		\htext(1.5 -0.3){$\underbrace{\rule{1.5em}{0em}}$}
		\htext(1.7 -0.7){$\bar{u}_{1}$}
		\esegment

		\move(16 0)
		\bsegment
		\move(0 0)\lvec(1 0)
		\move(0 2)\lvec(1 2)
		\move(0 4)\lvec(1 4)
		\move(0 6)\lvec(1 6)
		\move(0 7)\lvec(1 7)
		\move(0 9)\lvec(1 9)
		\move(0 11)\lvec(1 11)
		\move(0 13)\lvec(1 13)
		\move(0 14)\lvec(1 14)
		\move(0 16)\lvec(1 16)
		\move(0 18)\lvec(1 18)
		\move(0 20)\lvec(1 20)
		\move(1 0)\lvec(1 20)
		
		\move(0 0)\lvec(1 2)
		
		\htext(0.7 0.5){$0$}	
		\htext(0.35 1.55){$0$}
		\htext(0.5 3){$1$}
		\htext(0.5 5){$2$}
		\vtext(0.5 6.57){$\cdots$}
		\htext(0.5 8){$n$-$1$}
		\htext(0.5 10){$n$}
		\htext(0.5 12){$n$-$1$}
		\vtext(0.5 13.57){$\cdots$}
		\htext(0.5 15){$3$}
		\htext(0.5 17){$2$}
		\htext(0.5 19){$1$}
		\esegment
	\end{texdraw}
\end{center}

\vskip 4mm 

In the extended Young columns defined above, the vertical line separating $u_i(\bar{u}_i)$ and $v_i(\bar{v}_i)$ is called a \textit{virtual line}.

\begin{remark}
The stacking patterns for extended Young columns  follow the patterns given in Definition \ref{stacking patterns}.
	\end{remark}

We define the action of Kashiwara operators on extended Young columns in the same way as in Section \ref{The crystal structure}. But we only add or remove blocks on the right side of each virtual line.

\vskip 2mm

 We can read off the same virtual Young column from different  extended Young columns. 
We say that two extended Young columns are {\it equivalent} if they correspond to the same virtual Young column. The extended Young columns given below form an equivalence class of extended Young columns.

\vskip 2mm 

In each type, set $u_i=v_i=\bar{u}_i=\bar{v}_i=\infty$. The extended Young columns in this case are called the {\it ground-state  extended Young columns}. The virtual Young column corresponding to the ground-state  extended Young column is an empty Young column, which can be regarded as the {\it ground-state virtual Young column}. 

\vskip 2mm 

We can define the notions of virtual Young walls, reduced virtual Young walls,  extended Young walls,  reduced extended Young walls, $i$-signatures and Kashiwara operators in a similar way as in Section \ref{Construction of Young Wall}. Practically, we usually consider the Kashiwara operators acting on  extended Young walls rather than  virtual Young walls.

\vskip 2mm 

Let $Y_{\infty}$ denote the extended Young wall consisting of infinitely many ground-state extended Young columns extending to the left.

\vskip 2mm

For a virtual Young wall or an extended Young wall $Y$, we define

\vskip 3mm

\begin{enumerate}
	\item $\wt(Y) = -\sum_{i\in I}k_i \alpha_i$, where $k_i$ is the number of $i$-blocks added on $Y_{\infty}$,
	\vspace{6pt}
	\item $\varepsilon_i (Y)$ is the number of $-$'s in the $i$-signature of $Y$,
	\vspace{6pt}	
	\item $\varphi_i (Y)$ is the number of $+$'s in the $i$-signature of $Y$.
\end{enumerate}

\vskip 2mm

Let ${\mathcal Y}(\infty)$ and $\widehat{\mathcal Y}(\infty)$ denote the set of all reduced virtual Young walls and reduced extended Young walls, respectively.
Let $Y_{\emptyset}$ be the virtual Young wall consisting of infinitely many empty Young columns extending to the left. 

\vskip 2mm

By a similar argument as in the proof of Theorem \ref{main theorem}, we have the following theorem.

\begin{theorem} \hfill

\vskip 2mm 

(a) The sets ${\mathcal Y}(\infty)$ and $\widehat{\mathcal Y}(\infty)$ are $U_q'(\mathfrak g)$-crystals.

\vskip 2mm 

(b) There exist $U_{q}'(\mathfrak g)$-crystal isomorphisms
\begin{align*}
{\mathcal Y}(\infty)&\stackrel{\sim}{\longrightarrow} \widehat{\mathcal Y}(\infty)\stackrel{\sim}{\longrightarrow} B(\infty)\\
 Y_{\emptyset}\ \ \! \  &\longmapsto\ \ \!  Y_{\infty}\quad \!\!  \longmapsto \   u_{\infty}.
 \end{align*}
\end{theorem}

\begin{proof}
For part (a), we follow a similar approach in the proof of Theorem \ref{crystal structure on Young walls}. We need to prove that the set $\widehat{\mathcal Y}(\infty) \cup \{0\}$ is closed under the maps 
$\tilde{F}_i$ $(i \in I)$. 
Let $Y=(\cdots, C_2, C_1,C_0)$ be a reduced extended Young wall in $\widehat{\mathcal Y}(\infty)$. 

\vskip 2mm

	For  type $A_{2n}^{(2)}$, only $\tilde{F}_0$-action can produce $\delta$-slices. Assume $\tilde{F}_0$ acts on the extended										 Young column $C_k$ in $Y$, then we have
\begin{equation*}
\phi_0(C_{k+1})\le\epsilon_0(C_{k}),\quad x_1<\bar{x}_1.	
	\end{equation*}

 Thus we have
\begin{equation*}
	-(\phi_0(C_{k+1})-2{(\bar{x}_1'-x_1')}_{+})\ge -\phi_0(C_{k+1})\ge -\epsilon_0(C_{k}).
\end{equation*}  

If we remove the  $\delta$-slice in $C_k$, we obtain
\begin{equation*}
	-(\phi_0(C_{k+1})-2{(\bar{x}_1'-x_1')}_{+})\ge -\phi_0(C_{k+1})>-\epsilon_0(C_{k})-2,
\end{equation*}
which means the number of $0$-blocks in $C_{k+1}$ is bigger than the number of blocks in $C_{k}$. This will be contradict to the tensor product rule. 
Hence $\tilde{F}_0Y$ is reduced. 

\vskip 2mm

The crystal structure on $\widehat{\mathcal Y}(\infty)$ of other types can also be proved in a similar way. Since virtual Young walls can be obtained from extended Young walls,  the crystal structure on ${\mathcal Y}(\infty)$ follows from the crystal structure on $\widehat{\mathcal Y}(\infty)$.

\vskip 2mm

For part (b), it is easy to see that the crystal ${\mathcal Y}(\infty)$ is isomorphic to the crystal $\widehat{\mathcal Y}(\infty)$. We can prove that the crystal $\widehat{\mathcal Y}(\infty)$ and $B(\infty)$ are isomorphic using a similar  approach in the proof of Theorem \ref{main theorem}.
\end{proof}

\begin{example}
We illustrate the top part of the crystal $B(\infty)$ for type $D_{3}^{(2)}$ in terms of reduced virtual Young walls. For each reduced virtual Young wall, the number below the slice represents the number of slices.
	
\vskip 4mm

\begin{center}
	\begin{texdraw}
		\drawdim em
		\arrowheadsize l:0.3 w:0.3
		\arrowheadtype t:V
		\fontsize{4}{4}\selectfont
		\drawdim em
		\setunitscale 1.2
		\move(0 11)	
		\bsegment
		\htext(21 0){\fontsize{12}{12}\selectfont{$Y_{\emptyset}$}}
		\esegment
		\move(8 0)	
		\bsegment
		\move(0 0)\lvec(2 0)
		\move(0 0)\lvec(0 2)
		\move(1 0)\lvec(1 2)
		\move(0 2)\lvec(2 2)
		\move(2 0)\lvec(2 2)
		\move(1 4)\lvec(2 4)
		\move(0 0)\lvec(1 2)
		\move(1 0)\lvec(2 2)
		\move(1 2)\lvec(1 4)
		\move(2 2)\lvec(2 4)
		
		\htext(1.5 3){$1$}
		\htext(0.7 0.5){$0$}
		\htext(0.35 1.5){$0$}
		\htext(1.7 0.5){$0$}
		\htext(1.35 1.5){$0$}
		\htext(0.5 -0.5){-$1$}
		\htext(1.5 -0.5){$1$}
		\esegment
		
		\move(20 0)	
		\bsegment
		\move(0 0)\lvec(2 0)
		\move(0 0)\lvec(0 2)
		\move(1 0)\lvec(1 2)
		\move(0 2)\lvec(2 2)
		\move(2 0)\lvec(2 2)
		\move(0 2)\lvec(1 2)
		\move(0 0)\lvec(1 2)
		\move(1 0)\lvec(2 2)
		\htext(0.7 0.5){$0$}
		
		\htext(1.7 0.5){$0$}
		\htext(1.35 1.5){$0$}
		\htext(0.5 -0.5){-$1$}
		\htext(1.5 -0.5){$1$}
		\esegment
		
		\move(32 0)	
		\bsegment
		\move(0 0)\lvec(2 0)
		\move(0 0)\lvec(0 4)
		\move(1 0)\lvec(1 4)
		\move(0 2)\lvec(2 2)
		\move(2 0)\lvec(2 2)
		\move(0 4)\lvec(1 4)
		\move(0 0)\lvec(1 2)
		\move(1 0)\lvec(2 2)
		\move(1 4)\lvec(2 4)
		\move(2 2)\lvec(2 6)
		\move(1 4)\lvec(1 6)
		\move(1 6)\lvec(2 6)
		\move(1 4)\lvec(2 6)
		\htext(0.5 3){$1$}
		\htext(1.5 3){$1$}
		\htext(0.7 0.5){$0$}
		\htext(0.35 1.5){$0$}
		\htext(1.7 0.5){$0$}
		\htext(1.35 1.5){$0$}
		\htext(0.5 -0.5){-$1$}
		\htext(1.5 -0.5){$1$}
		\htext(1.7 4.5){$2$}
		
		\esegment
		
		
		\move(0 -12)
		\bsegment
		\move(-4 0)
		\bsegment
		\move(0 0)\lvec(2 0)
		\move(0 0)\lvec(0 2)
		\move(1 0)\lvec(1 2)
		\move(0 2)\lvec(2 2)
		\move(2 0)\lvec(2 2)
		\move(0 2)\lvec(1 2)
		\move(0 0)\lvec(1 2)
		\move(1 0)\lvec(2 2)
		
		\move(2 2)\lvec(2 4)
		\move(2 4)\lvec(3 4)
		\move(2 0)\lvec(4 0)
		\move(4 0)\lvec(4 8)
		\move(3 0)\lvec(3 2)
		\move(2 2)\lvec(4 2)
		\move(3 2)\lvec(3 8)
		\move(3 4)\lvec(4 4)
		\move(3 6)\lvec(4 6)
		\move(3 8)\lvec(4 8)
		\move(2 0)\lvec(3 2)
		\move(3 0)\lvec(4 2)
		\move(3 4)\lvec(4 6)
		\htext(3.7 4.5){$2$}
		\htext(3.35 5.5){$2$}

		\htext(2.7 0.5){$0$}
		\htext(2.35 1.5){$0$}
		\htext(3.7 0.5){$0$}
		\htext(3.35 1.5){$0$}
		
		\htext(2.5 3){$1$}
		\htext(3.5 3){$1$}
		\htext(3.5 7){$1$}
		
		\htext(0.7 0.5){$0$}
		\htext(1.7 0.5){$0$}
		\htext(1.35 1.5){$0$}
		\htext(0.5 -0.5){$1$}
		\htext(1.5 -0.5){-$1$}
		\htext(2.5 -0.5){$1$}
		\htext(3.5 -0.5){-$1$}
		\esegment
		\move(4 0)
		\bsegment
		\move(0 0)\lvec(2 0)
		\move(0 0)\lvec(0 2)
		\move(1 0)\lvec(1 2)
		\move(0 2)\lvec(2 2)
		\move(2 0)\lvec(2 2)
		\move(0 2)\lvec(1 2)
		\move(0 0)\lvec(1 2)
		\move(1 0)\lvec(2 2)
		\htext(0.7 0.5){$0$}
		\htext(1.5 3){$1$}
		\htext(1.7 0.5){$0$}
		\htext(1.35 1.5){$0$}
		\htext(0.35 1.5){$0$}
		\htext(0.5 -0.5){-$2$}
		\htext(1.5 -0.5){$2$}
		\move(2 2)\lvec(2 4)
		\move(1 2)\lvec(1 4)
		\move(1 4)\lvec(2 4)
		
		\esegment

		\move(10 0)
		\bsegment
		\move(0 0)\lvec(2 0)
		\move(0 0)\lvec(0 2)
		\move(1 0)\lvec(1 2)
		\move(0 2)\lvec(2 2)
		\move(2 0)\lvec(2 2)
		\move(0 2)\lvec(1 2)
		\move(0 0)\lvec(1 2)
		\move(1 0)\lvec(2 2)
		\htext(0.7 0.5){$0$}
		\htext(1.5 3){$1$}
		\htext(1.7 0.5){$0$}
		\htext(1.35 1.5){$0$}
		\htext(0.35 1.5){$0$}
		\htext(0.5 -0.5){-$1$}
		\htext(1.5 -0.5){$1$}
		\move(1 4)\lvec(1 6)
		\move(1 2)\lvec(1 4)
		\move(1 4)\lvec(2 4)
		\move(2 2)\lvec(2 6)
		\move(2 4)\lvec(2 6)
		\move(1 6)\lvec(2 6)
		\move(1 4)\lvec(2 6)
		\htext(1.7 4.5){$2$}
		
		\esegment
		
		\move(16 0)
		\bsegment
		\move(0 0)\lvec(2 0)
		\move(0 0)\lvec(0 2)
		\move(1 0)\lvec(1 2)
		\move(0 2)\lvec(2 2)
		\move(2 0)\lvec(2 2)
		\move(0 2)\lvec(1 2)
		\move(0 0)\lvec(1 2)
		\move(1 0)\lvec(2 2)
		\move(1 2)\lvec(1 4)
		\move(2 2)\lvec(2 4)
		\move(1 4)\lvec(2 4)
		\htext(0.7 0.5){$0$}
		\htext(1.5 3){$1$}
		\htext(1.7 0.5){$0$}
		\htext(1.35 1.5){$0$}
		\htext(0.5 -0.5){-$1$}
		\htext(1.5 -0.5){$1$}
		
		\esegment
		
		\move(22 0)
		\bsegment
		\move(0 0)\lvec(2 0)
		\move(0 0)\lvec(0 2)
		\move(1 0)\lvec(1 2)
		\move(0 2)\lvec(2 2)
		\move(2 0)\lvec(2 2)
		\move(0 2)\lvec(1 2)
		\move(0 0)\lvec(1 2)
		\move(1 0)\lvec(2 2)
		\htext(0.7 0.5){$0$}
		
		\htext(1.7 0.5){$0$}
		\htext(1.35 1.5){$0$}
		\htext(0.5 -0.5){-$2$}
		\htext(1.5 -0.5){$2$}
		
		\esegment
		
		\move(28 0)
		\bsegment
		\move(0 0)\lvec(2 0)
		\move(0 0)\lvec(0 2)
		\move(1 0)\lvec(1 2)
		\move(0 2)\lvec(2 2)
		\move(2 0)\lvec(2 2)
		\move(0 2)\lvec(1 2)
		\move(0 0)\lvec(1 2)
		\move(1 0)\lvec(2 2)
		
		\move(2 2)\lvec(2 4)
		\move(3 0)\lvec(3 6)
		\move(4 0)\lvec(4 6)
		\move(3 6)\lvec(4 6)
		\move(2 0)\lvec(4 0)
		\move(2 2)\lvec(4 2)
		\move(2 4)\lvec(4 4)

		\move(2 0)\lvec(3 2)
		\move(3 0)\lvec(4 2)
		\move(3 4)\lvec(4 6)
		\htext(2.5 3){$1$}
		\htext(3.5 3){$1$}
		\htext(3.7 4.5){$2$}
		
		\htext(0.7 0.5){$0$}
		\htext(2.7 0.5){$0$}
		\htext(2.35 1.5){$0$}
		\htext(3.7 0.5){$0$}
		\htext(3.35 1.5){$0$}
		\htext(1.7 0.5){$0$}
		\htext(1.35 1.5){$0$}
		\htext(0.5 -0.5){-$1$}
		\htext(1.5 -0.5){$1$}
		\htext(2.5 -0.5){-$1$}
		\htext(3.5 -0.5){$1$}
		\esegment

		\move(36 0)
		\bsegment
		\move(0 0)\lvec(2 0)
		\move(0 0)\lvec(0 2)
		\move(1 0)\lvec(1 2)
		\move(0 2)\lvec(2 2)
		\move(2 0)\lvec(2 2)
		\move(0 2)\lvec(1 2)
		\move(0 0)\lvec(1 2)
		\move(1 0)\lvec(2 2)
		\move(0 2)\lvec(0 4)
		\move(0 4)\lvec(2 4)
		\move(1 6)\lvec(2 6)
		\move(2 2)\lvec(2 6)
		\move(1 2)\lvec(1 6)
		\move(1 4)\lvec(2 6)
		\htext(0.5 3){$1$}
		\htext(1.5 3){$1$}
		\htext(1.7 4.5){$2$}
		\htext(1.35 5.5){$2$}
		\htext(0.7 0.5){$0$}
		\htext(0.35 1.5){$0$}
		\htext(1.7 0.5){$0$}
		\htext(1.35 1.5){$0$}
		\htext(0.5 -0.5){-$1$}
		\htext(1.5 -0.5){$1$}
		
		\esegment
		
		\move(42 0)
		\bsegment
		\move(0 0)\lvec(2 0)
		\move(0 0)\lvec(0 2)
		\move(1 0)\lvec(1 2)
		\move(0 2)\lvec(2 2)
		\move(2 0)\lvec(2 2)
		\move(0 2)\lvec(1 2)
		\move(0 0)\lvec(1 2)
		\move(1 0)\lvec(2 2)
		\move(1 2)\lvec(1 4)
		\move(2 2)\lvec(2 4)
		\move(1 4)\lvec(2 4)
		\move(2 0)\lvec(4 0)
		\move(4 0)\lvec(4 8)
		\move(3 0)\lvec(3 2)
		\move(2 2)\lvec(4 2)
		\move(3 2)\lvec(3 8)
		\move(3 4)\lvec(4 4)
		\move(3 6)\lvec(4 6)
		\move(3 8)\lvec(4 8)
		\move(2 0)\lvec(3 2)
		\move(3 0)\lvec(4 2)
		\move(3 4)\lvec(4 6)
		\move(2 4)\lvec(2 6)
		\move(1 4)\lvec(1 6)
		\move(1 6)\lvec(2 6)
		\move(2 6)\lvec(3 6)
		\move(2 4)\lvec(3 4)
		\move(0 2)\lvec(0 4)
		\move(0 4)\lvec(1 4)
		\move(2 4)\lvec(3 6)
		\move(1 4)\lvec(2 6)
		
		\htext(0.5 3){$1$}
		\htext(2.5 3){$1$}
		
		\htext(3.7 4.5){$2$}
		\htext(3.35 5.5){$2$}
		\htext(2.7 4.5){$2$}
		\htext(2.35 5.5){$2$}
		\htext(1.7 4.5){$2$}
		\htext(0.35 1.5){$0$}
		\htext(2.7 0.5){$0$}
		\htext(2.35 1.5){$0$}
		\htext(3.7 0.5){$0$}
		\htext(3.35 1.5){$0$}
		
		\htext(1.5 3){$1$}
		\htext(3.5 3){$1$}
		\htext(3.5 7){$1$}
		
		\htext(0.7 0.5){$0$}
		\htext(1.7 0.5){$0$}
		\htext(1.35 1.5){$0$}
		\htext(0.5 -0.5){-$1$}
		\htext(1.5 -0.5){$1$}
		\htext(2.5 -0.5){-$1$}
		\htext(3.5 -0.5){$1$}
		
		\esegment

		\esegment

		
		\move(0 -24)
		\bsegment
		\move(-9 0)
		\bsegment
		\move(0 0)\lvec(2 0)
		\move(0 0)\lvec(0 2)
		\move(1 0)\lvec(1 2)
		\move(0 2)\lvec(2 2)
		\move(2 0)\lvec(2 2)
		\move(0 2)\lvec(1 2)
		\move(0 0)\lvec(1 2)
		\move(1 0)\lvec(2 2)
		
		\move(2 2)\lvec(2 4)
		\move(2 4)\lvec(3 4)
		\move(2 0)\lvec(4 0)
		\move(4 0)\lvec(4 8)
		\move(3 0)\lvec(3 2)
		\move(2 2)\lvec(4 2)
		\move(3 2)\lvec(3 8)
		\move(3 4)\lvec(4 4)
		\move(3 6)\lvec(4 6)
		\move(3 8)\lvec(4 8)
		\move(2 0)\lvec(3 2)
		\move(3 0)\lvec(4 2)
		\move(3 4)\lvec(4 6)
		\htext(3.7 4.5){$2$}
		\htext(3.35 5.5){$2$}

		\htext(2.7 0.5){$0$}
		\htext(2.35 1.5){$0$}
		\htext(3.7 0.5){$0$}
		\htext(3.35 1.5){$0$}
		
		\htext(2.5 3){$1$}
		
		\htext(3.5 3){$1$}
		\htext(3.5 7){$1$}
		
		\htext(0.7 0.5){$0$}
		\htext(1.7 0.5){$0$}
		\htext(1.35 1.5){$0$}
		\htext(0.5 -0.5){$1$}
		\htext(1.5 -0.5){-$2$}
		\htext(2.5 -0.5){$2$}
		\htext(3.5 -0.5){-$1$}
		\esegment
		\move(-4 0)
		\bsegment
		\move(0 0)\lvec(2 0)
		\move(0 0)\lvec(0 2)
		\move(1 0)\lvec(1 2)
		\move(0 2)\lvec(2 2)
		\move(2 0)\lvec(2 2)
		\move(0 2)\lvec(1 2)
		\move(0 0)\lvec(1 2)
		\move(1 0)\lvec(2 2)
		\move(0 2)\lvec(0 4)
		\move(1 2)\lvec(1 8)
		\move(2 2)\lvec(2 8)
		\move(0 4)\lvec(2 4)
		\move(1 6)\lvec(2 6)
		\move(1 8)\lvec(2 8)
		\move(1 4)\lvec(2 6)
		
		\htext(1.5 7){$1$}
		\htext(0.5 3){$1$}
		\htext(1.5 3){$1$}
		\htext(0.7 0.5){$0$}
		\htext(0.35 1.5){$0$}
		\htext(1.7 0.5){$0$}
		\htext(1.35 1.5){$0$}
		\htext(0.5 -0.5){$1$}
		\htext(1.5 -0.5){-$1$}
		\htext(1.7 4.5){$2$}
		\htext(1.35 5.5){$2$}
		\esegment

		\move(-1 0)
		\bsegment
		\move(0 0)\lvec(2 0)
		\move(0 0)\lvec(0 2)
		\move(1 0)\lvec(1 2)
		\move(0 2)\lvec(2 2)
		\move(2 0)\lvec(2 2)
		\move(0 2)\lvec(1 2)
		\move(0 0)\lvec(1 2)
		\move(1 0)\lvec(2 2)
		\htext(0.7 0.5){$0$}
		\htext(1.5 3){$1$}
		\htext(1.7 0.5){$0$}
		\htext(1.35 1.5){$0$}
		\htext(0.35 1.5){$0$}
		\htext(0.5 -0.5){-$3$}
		\htext(1.5 -0.5){$3$}
		\move(2 2)\lvec(2 4)
		\move(1 2)\lvec(1 4)
		\move(1 4)\lvec(2 4)
		\esegment

		\move(2 0)
		\bsegment
		\move(0 0)\lvec(2 0)
		\move(0 0)\lvec(0 2)
		\move(1 0)\lvec(1 2)
		\move(0 2)\lvec(2 2)
		\move(2 0)\lvec(2 2)
		\move(0 2)\lvec(1 2)
		\move(0 0)\lvec(1 2)
		\move(1 0)\lvec(2 2)
		
		\move(2 2)\lvec(2 4)
		\move(2 4)\lvec(3 4)
		\move(2 0)\lvec(4 0)
		\move(4 0)\lvec(4 8)
		\move(3 0)\lvec(3 2)
		\move(2 2)\lvec(4 2)
		\move(3 2)\lvec(3 8)
		\move(3 4)\lvec(4 4)
		\move(3 6)\lvec(4 6)
		\move(3 8)\lvec(4 8)
		\move(2 0)\lvec(3 2)
		\move(3 0)\lvec(4 2)
		\move(3 4)\lvec(4 6)

		\move(2 4)\lvec(2 6)
		\move(2 6)\lvec(3 6)
		\move(2 4)\lvec(3 6)
		
		\htext(3.7 4.5){$2$}
		\htext(3.35 5.5){$2$}

		\htext(2.7 0.5){$0$}
		\htext(2.35 1.5){$0$}
		\htext(3.7 0.5){$0$}
		\htext(3.35 1.5){$0$}
		
		\htext(2.5 3){$1$}
		\htext(3.5 3){$1$}
		\htext(3.5 7){$1$}
		\htext(2.7 4.5){$2$}
		\htext(0.7 0.5){$0$}
		\htext(1.7 0.5){$0$}
		\htext(1.35 1.5){$0$}
		\htext(0.5 -0.5){$1$}
		\htext(1.5 -0.5){-$1$}
		\htext(2.5 -0.5){$1$}
		\htext(3.5 -0.5){-$1$}
		\esegment
		\move(8 0)
		\bsegment
		\move(0 0)\lvec(2 0)
		\move(0 0)\lvec(0 2)
		\move(1 0)\lvec(1 2)
		\move(0 2)\lvec(2 2)
		\move(2 0)\lvec(2 2)
		\move(0 2)\lvec(1 2)
		\move(0 0)\lvec(1 2)
		\move(1 0)\lvec(2 2)
		\htext(0.7 0.5){$0$}
		\htext(1.5 3){$1$}
		\htext(1.7 0.5){$0$}
		\htext(1.35 1.5){$0$}
		\htext(0.35 1.5){$0$}
		\htext(0.5 -0.5){$1$}
		\htext(-0.5 -0.5){-$2$}
		\htext(1.5 -0.5){$1$}
		\move(2 2)\lvec(2 4)
		
		\move(1 4)\lvec(2 4)
		\move(2 4)\lvec(2 6)
		\move(1 2)\lvec(1 6)
		\move(1 6)\lvec(2 6)
		\move(1 4)\lvec(2 6)
		\htext(1.7 4.5){$2$}
		\move(-1 0)\lvec(0 0)
		\move(-1 0)\lvec(-1 2)
		\move(-1 2)\lvec(0 2)
		\move(0 2)\lvec(0 4)
		\move(0 4)\lvec(1 4)
		
		\move(-1 0)\lvec(0 2)
		\htext(-0.3 0.5){$0$}
		\htext(-0.65 1.5){$0$}
		\htext(0.5 3){$1$}
		
		\esegment

		\move(12 0)
		\bsegment
		\move(0 0)\lvec(2 0)
		\move(0 0)\lvec(0 2)
		\move(1 0)\lvec(1 2)
		\move(0 2)\lvec(2 2)
		\move(2 0)\lvec(2 2)
		\move(0 2)\lvec(1 2)
		\move(0 0)\lvec(1 2)
		\move(1 0)\lvec(2 2)
		\move(1 2)\lvec(1 4)
		\move(2 2)\lvec(2 4)
		\move(1 4)\lvec(2 4)
		\move(-1 0)\lvec(0 0)
		\move(-1 0)\lvec(-1 2)
		\move(-1 2)\lvec(0 2)
		\move(-1 0)\lvec(0 2)
		\htext(-0.3 0.5){$0$}
		\htext(0.7 0.5){$0$}
		\htext(1.5 3){$1$}
		\htext(1.7 0.5){$0$}
		\htext(1.35 1.5){$0$}
		\htext(0.35 1.5){$0$}
		\htext(-0.5 -0.5){-$1$}
		\htext(0.5 -0.5){-$1$}
		\htext(1.5 -0.5){$2$}
		
		\esegment
		\move(15 0)
		\bsegment
		\move(0 0)\lvec(2 0)
		\move(0 0)\lvec(0 2)
		\move(1 0)\lvec(1 2)
		\move(0 2)\lvec(2 2)
		\move(2 0)\lvec(2 2)
		\move(0 2)\lvec(1 2)
		\move(0 0)\lvec(1 2)
		\move(1 0)\lvec(2 2)
		\htext(0.7 0.5){$0$}
		
		\htext(1.7 0.5){$0$}
		\htext(1.35 1.5){$0$}
		\htext(0.5 -0.5){-$1$}
		\htext(1.5 -0.5){$1$}
		\move(1 2)\lvec(1 6)
		\move(2 2)\lvec(2 6)
		\move(1 4)\lvec(2 4)
		\move(1 6)\lvec(2 6)
		\htext(1.5 3){$1$}
		\move(1 4)\lvec(2 6)
		\htext(1.7 4.5){$2$}
		
		\esegment
		
		\move(20 0)
		\bsegment
		\move(0 0)\lvec(2 0)
		\move(0 0)\lvec(0 2)
		\move(1 0)\lvec(1 2)
		\move(0 2)\lvec(2 2)
		\move(2 0)\lvec(2 2)
		\move(0 2)\lvec(1 2)
		\move(0 0)\lvec(1 2)
		\move(1 0)\lvec(2 2)
		\htext(0.7 0.5){$0$}
		\htext(1.5 3){$1$}
		\htext(1.7 0.5){$0$}
		\htext(1.35 1.5){$0$}
		\htext(0.35 1.5){$0$}
		\htext(0.5 -0.5){-$1$}
		\htext(-0.5 -0.5){$2$}
		\htext(-1.5 -0.5){-$2$}
		\htext(1.5 -0.5){$1$}
		\move(2 2)\lvec(2 4)
		
		\move(1 4)\lvec(2 4)
		\move(2 4)\lvec(2 6)
		\move(1 2)\lvec(1 6)
		\move(1 6)\lvec(2 6)
		\move(1 4)\lvec(2 6)
		\htext(1.7 4.5){$2$}
		\move(-1 0)\lvec(0 0)
		\move(-1 0)\lvec(-1 2)
		\move(-1 2)\lvec(0 2)
		\move(0 2)\lvec(0 4)
		\move(0 4)\lvec(1 4)
		
		\move(-2 0)\lvec(-2 2)
		\move(-2 0)\lvec(-1 0)
		\move(-2 2)\lvec(-1 2)
		\move(-1 0)\lvec(0 2)
		\move(-2 0)\lvec(-1 2)
		\htext(-1.3 0.5){$0$}
		\htext(-0.3 0.5){$0$}
		\htext(-0.65 1.5){$0$}
		\htext(0.5 3){$1$}
		
		\esegment
		
		\move(23 0)
		\bsegment
		\move(0 0)\lvec(2 0)
		\move(0 0)\lvec(0 2)
		\move(1 0)\lvec(1 2)
		\move(0 2)\lvec(2 2)
		\move(2 0)\lvec(2 2)
		\move(0 2)\lvec(1 2)
		\move(0 0)\lvec(1 2)
		\move(1 0)\lvec(2 2)
		\htext(0.7 0.5){$0$}
		
		\htext(1.7 0.5){$0$}
		\htext(1.35 1.5){$0$}
		\htext(0.5 -0.5){-$3$}
		\htext(1.5 -0.5){$3$}
		\esegment

		\move(28 0)
		\bsegment
		\move(0 0)\lvec(2 0)
		\move(0 0)\lvec(0 2)
		\move(1 0)\lvec(1 2)
		\move(0 2)\lvec(2 2)
		\move(2 0)\lvec(2 2)
		\move(0 2)\lvec(1 2)
		\move(0 0)\lvec(1 2)
		\move(1 0)\lvec(2 2)
		\htext(0.7 0.5){$0$}
		\htext(1.5 3){$1$}
		\htext(1.7 0.5){$0$}
		\htext(1.35 1.5){$0$}
		\htext(1.35 5.5){$2$}
		\htext(0.35 1.5){$0$}
		\htext(0.5 -0.5){-$1$}
		\htext(-0.5 -0.5){$1$}
		\htext(-1.5 -0.5){-$1$}
		\htext(1.5 -0.5){$1$}
		\move(2 2)\lvec(2 4)
		
		\move(1 4)\lvec(2 4)
		\move(2 4)\lvec(2 6)
		\move(1 2)\lvec(1 6)
		\move(1 6)\lvec(2 6)
		\move(1 4)\lvec(2 6)
		\htext(1.7 4.5){$2$}
		\move(-1 0)\lvec(0 0)
		\move(-1 0)\lvec(-1 2)
		\move(-1 2)\lvec(0 2)
		\move(0 2)\lvec(0 4)
		\move(0 4)\lvec(1 4)
		
		\move(-2 0)\lvec(-2 2)
		\move(-2 0)\lvec(-1 0)
		\move(-2 2)\lvec(-1 2)
		\move(-1 0)\lvec(0 2)
		\move(-2 0)\lvec(-1 2)
		\htext(-1.3 0.5){$0$}
		\htext(-0.3 0.5){$0$}
		\htext(-0.65 1.5){$0$}
		\htext(0.5 3){$1$}
		\esegment
		\move(31 0)
		\bsegment
		\move(0 0)\lvec(2 0)
		\move(0 0)\lvec(0 2)
		\move(1 0)\lvec(1 2)
		\move(0 2)\lvec(2 2)
		\move(2 0)\lvec(2 2)
		\move(0 2)\lvec(1 2)
		\move(0 0)\lvec(1 2)
		\move(1 0)\lvec(2 2)
		\move(0 2)\lvec(0 4)
		\move(0 4)\lvec(2 4)
		\move(1 6)\lvec(2 6)
		\move(2 2)\lvec(2 6)
		\move(1 2)\lvec(1 6)
		\move(1 4)\lvec(2 6)
		\move(1 6)\lvec(1 8)
		\move(2 6)\lvec(2 8)
		\move(1 8)\lvec(2 8)
		\htext(1.5 7){$1$}
		\htext(0.5 3){$1$}
		\htext(1.5 3){$1$}
		\htext(1.7 4.5){$2$}
		\htext(1.35 5.5){$2$}
		\htext(0.7 0.5){$0$}
		\htext(0.35 1.5){$0$}
		\htext(1.7 0.5){$0$}
		\htext(1.35 1.5){$0$}
		\htext(0.5 -0.5){-$1$}
		\htext(1.5 -0.5){$1$}
		\esegment
		
		\move(34 0)
		\bsegment
		\move(0 0)\lvec(2 0)
		\move(0 0)\lvec(0 2)
		\move(1 0)\lvec(1 2)
		\move(0 2)\lvec(2 2)
		\move(2 0)\lvec(2 2)
		\move(0 2)\lvec(1 2)
		\move(0 0)\lvec(1 2)
		\move(1 0)\lvec(2 2)
		\move(0 2)\lvec(0 4)
		\move(0 4)\lvec(2 4)
		\move(1 6)\lvec(2 6)
		\move(2 2)\lvec(2 6)
		\move(1 2)\lvec(1 6)
		\move(1 4)\lvec(2 6)
		\move(2 0)\lvec(3 0)
		\move(2 2)\lvec(3 2)
		\move(2 4)\lvec(3 4)
		\move(2 6)\lvec(3 6)
		\move(3 0)\lvec(3 6)
		\move(2 0)\lvec(3 2)
		\move(2 4)\lvec(3 6)
		
		\htext(2.7 0.5){$0$}
		\htext(2.35 1.5){$0$}
		\htext(2.7 4.5){$2$}
		\htext(2.35 5.5){$2$}
		
		\htext(0.5 3){$1$}
		\htext(1.5 3){$1$}
		\htext(2.5 3){$1$}
		\htext(1.7 4.5){$2$}
		
		\htext(0.7 0.5){$0$}
		\htext(0.35 1.5){$0$}
		\htext(1.7 0.5){$0$}
		\htext(1.35 1.5){$0$}
		\htext(0.5 -0.5){-$2$}
		\htext(1.5 -0.5){$1$}
		\htext(2.5 -0.5){$1$}
		\esegment
		
		\move(38 0)
		\bsegment
		\move(0 0)\lvec(2 0)
		\move(0 0)\lvec(0 2)
		\move(1 0)\lvec(1 2)
		\move(0 2)\lvec(2 2)
		\move(2 0)\lvec(2 2)
		\move(0 2)\lvec(1 2)
		\move(0 0)\lvec(1 2)
		\move(1 0)\lvec(2 2)
		\move(1 2)\lvec(1 4)
		\move(2 2)\lvec(2 4)
		\move(1 4)\lvec(2 4)
		\move(2 0)\lvec(4 0)
		\move(4 0)\lvec(4 8)
		\move(3 0)\lvec(3 2)
		\move(2 2)\lvec(4 2)
		\move(3 2)\lvec(3 8)
		\move(3 4)\lvec(4 4)
		\move(3 6)\lvec(4 6)
		\move(3 8)\lvec(4 8)
		\move(2 0)\lvec(3 2)
		\move(3 0)\lvec(4 2)
		\move(3 4)\lvec(4 6)
		\move(2 4)\lvec(2 6)
		\move(1 4)\lvec(1 6)
		\move(1 6)\lvec(2 6)
		\move(2 6)\lvec(3 6)
		\move(2 4)\lvec(3 4)
		\move(2 4)\lvec(3 6)
		\move(1 4)\lvec(2 6)

		\htext(2.5 3){$1$}
		\htext(3.7 4.5){$2$}
		\htext(3.35 5.5){$2$}
		\htext(2.7 4.5){$2$}
		\htext(2.35 5.5){$2$}
		\htext(1.7 4.5){$2$}
		\htext(0.35 1.5){$0$}
		\htext(2.7 0.5){$0$}
		\htext(2.35 1.5){$0$}
		\htext(3.7 0.5){$0$}
		\htext(3.35 1.5){$0$}
		\htext(1.5 3){$1$}
		\htext(3.5 3){$1$}
		\htext(3.5 7){$1$}
		\htext(0.7 0.5){$0$}
		\htext(1.7 0.5){$0$}
		\htext(1.35 1.5){$0$}
		\htext(0.5 -0.5){-$1$}
		\htext(1.5 -0.5){$1$}
		\htext(2.5 -0.5){-$1$}
		\htext(3.5 -0.5){$1$}
		
		\esegment
		\move(44 0)
		\bsegment
		\move(0 0)\lvec(2 0)
		\move(0 0)\lvec(0 2)
		\move(1 0)\lvec(1 2)
		\move(0 2)\lvec(2 2)
		\move(2 0)\lvec(2 2)
		\move(0 2)\lvec(1 2)
		\move(0 0)\lvec(1 2)
		\move(1 0)\lvec(2 2)
		\move(0 2)\lvec(0 4)
		\move(0 4)\lvec(2 4)
		\move(1 6)\lvec(2 6)
		\move(2 2)\lvec(2 6)
		\move(1 2)\lvec(1 6)
		\move(1 4)\lvec(2 6)
		\move(2 0)\lvec(3 0)
		\move(2 2)\lvec(3 2)
		\move(2 4)\lvec(3 4)
		\move(2 6)\lvec(3 6)
		\move(3 0)\lvec(3 6)
		\move(2 0)\lvec(3 2)
		\move(2 4)\lvec(3 6)
		\move(-1 0)\lvec(0 0)
		\move(-1 0)\lvec(-1 2)
		\move(-1 2)\lvec(0 2)
		\move(-1 0)\lvec(0 2)
		\htext(-0.3 0.5){$0$}
		
		\htext(2.7 0.5){$0$}
		\htext(2.35 1.5){$0$}
		\htext(2.7 4.5){$2$}
		\htext(2.35 5.5){$2$}
		
		\htext(0.5 3){$1$}
		\htext(1.5 3){$1$}
		\htext(2.5 3){$1$}
		\htext(1.7 4.5){$2$}
		
		\htext(0.7 0.5){$0$}
		\htext(0.35 1.5){$0$}
		\htext(1.7 0.5){$0$}
		\htext(1.35 1.5){$0$}
		\htext(-0.5 -0.5){$1$}
		\htext(0.5 -0.5){-$1$}
		\htext(1.5 -0.5){$1$}
		\htext(2.5 -0.5){-$1$}
		
		\esegment
		\move(50 0)
		\bsegment
		\move(0 0)\lvec(2 0)
		\move(0 0)\lvec(0 2)
		\move(1 0)\lvec(1 2)
		\move(0 2)\lvec(2 2)
		\move(2 0)\lvec(2 2)
		\move(0 2)\lvec(1 2)
		\move(0 0)\lvec(1 2)
		\move(1 0)\lvec(2 2)
		\move(1 2)\lvec(1 4)
		\move(2 2)\lvec(2 4)
		\move(1 4)\lvec(2 4)
		\move(2 0)\lvec(4 0)
		\move(4 0)\lvec(4 8)
		\move(3 0)\lvec(3 2)
		\move(2 2)\lvec(4 2)
		\move(3 2)\lvec(3 8)
		\move(3 4)\lvec(4 4)
		\move(3 6)\lvec(4 6)
		\move(3 8)\lvec(4 8)
		\move(2 0)\lvec(3 2)
		\move(3 0)\lvec(4 2)
		\move(3 4)\lvec(4 6)
		\move(2 4)\lvec(2 6)
		\move(1 4)\lvec(1 6)
		\move(1 6)\lvec(2 6)
		\move(2 6)\lvec(3 6)
		\move(2 4)\lvec(3 4)
		\move(0 2)\lvec(0 4)
		\move(0 4)\lvec(1 4)
		\move(2 4)\lvec(3 6)
		\move(1 4)\lvec(2 6)
		\move(-2 0)\lvec(0 0)
		\move(0 4)\lvec(0 6)
		\move(-1 0)\lvec(-1 6)
		\move(-2 0)\lvec(-2 4)
		\move(-2 4)\lvec(0 4)
		\move(-1 6)\lvec(0 6)
		\move(-2 2)\lvec(0 2)
		\move(-2 0)\lvec(-1 2)
		\move(-1 0)\lvec(0 2)
		\move(-1 4)\lvec(0 6)
		\htext(-0.3 4.5){$2$}
		\htext(0.5 3){$1$}
		\htext(2.5 3){$1$}
		
		\htext(3.7 4.5){$2$}
		\htext(3.35 5.5){$2$}
		\htext(2.7 4.5){$2$}
		\htext(2.35 5.5){$2$}
		\htext(1.7 4.5){$2$}
		\htext(0.35 1.5){$0$}
		\htext(2.7 0.5){$0$}
		\htext(2.35 1.5){$0$}
		\htext(3.7 0.5){$0$}
		\htext(3.35 1.5){$0$}
		
		\htext(-1.5 3){$1$}
		\htext(-0.5 3){$1$}
		\htext(1.5 3){$1$}
		\htext(3.5 3){$1$}
		\htext(3.5 7){$1$}
		
		\htext(0.7 0.5){$0$}
		\htext(1.7 0.5){$0$}
		\htext(1.35 1.5){$0$}
		\htext(-1.5 -0.5){-$1$}
		\htext(-0.5 -0.5){$1$}
		\htext(0.5 -0.5){-$1$}
		\htext(1.5 -0.5){$1$}
		\htext(2.5 -0.5){-$1$}
		\htext(3.5 -0.5){$1$}
		\htext(-0.3 0.5){$0$}
		\htext(-0.65 1.5){$0$}
		\htext(-1.3 0.5){$0$}
		\htext(-1.65 1.5){$0$}

		\esegment

		\esegment

		
		\move(0 -36)
		\bsegment
		
		\move(9 0)
		\bsegment
		\move(0 0)\lvec(2 0)
		\move(0 0)\lvec(0 2)
		\move(1 0)\lvec(1 2)
		\move(0 2)\lvec(2 2)
		\move(2 0)\lvec(2 2)
		\move(0 2)\lvec(1 2)
		\move(0 0)\lvec(1 2)
		\move(1 0)\lvec(2 2)
		\htext(0.7 0.5){$0$}
		\htext(1.5 3){$1$}
		
		\htext(1.7 0.5){$0$}
		\htext(1.35 1.5){$0$}
		\htext(0.35 1.5){$0$}
		\htext(0.5 -0.5){-$1$}
		\htext(1.5 -0.5){$1$}
		\move(2 2)\lvec(2 4)
		\move(1 2)\lvec(1 4)
		\move(1 4)\lvec(1 6)
		\move(2 4)\lvec(2 6)
		\move(1 4)\lvec(2 4)
		\move(1 6)\lvec(2 6)
		\move(1 4)\lvec(2 6)
		\htext(1.7 4.5){$2$}
		\htext(1.35 5.5){$2$}

		\esegment
		
		\move(16 0)
		\bsegment
		\move(0 0)\lvec(2 0)
		\move(0 0)\lvec(0 2)
		\move(1 0)\lvec(1 2)
		\move(0 2)\lvec(2 2)
		\move(2 0)\lvec(2 2)
		\move(0 2)\lvec(1 2)
		\move(0 0)\lvec(1 2)
		\move(1 0)\lvec(2 2)
		\move(1 2)\lvec(1 4)
		\move(2 2)\lvec(2 4)
		\move(1 4)\lvec(2 4)
		\move(-1 0)\lvec(0 0)
		\move(-1 0)\lvec(-1 2)
		\move(-1 2)\lvec(0 2)
		\move(-1 0)\lvec(0 2)
		\htext(-0.3 0.5){$0$}
		\htext(0.7 0.5){$0$}
		\htext(1.5 3){$1$}
		\htext(1.7 0.5){$0$}
		\htext(1.35 1.5){$0$}
		
		\htext(-0.65 1.5){$0$}
		\htext(-1.5 -0.5){-$1$}
		\htext(-0.5 -0.5){$1$}
		\htext(0.5 -0.5){-$1$}
		\htext(1.5 -0.5){$1$}
		\move(-2 0)\lvec(-1 0)
		\move(-2 0)\lvec(-2 2)
		\move(-2 2)\lvec(-1 2)
		\move(-2 0)\lvec(-1 2)
		\htext(-1.3 0.5){$0$}
		
		\esegment

		\move(21 0)
		\bsegment
		\move(0 0)\lvec(2 0)
		\move(0 0)\lvec(0 2)
		\move(1 0)\lvec(1 2)
		\move(0 2)\lvec(2 2)
		\move(2 0)\lvec(2 2)
		\move(0 2)\lvec(1 2)
		\move(0 0)\lvec(1 2)
		\move(1 0)\lvec(2 2)
		\move(3 2)\lvec(3 4)
		\move(2 2)\lvec(2 4)
		\move(2 4)\lvec(3 4)
		\htext(0.7 0.5){$0$}
		\htext(2.5 3){$1$}
		
		\htext(1.7 0.5){$0$}
		\htext(1.35 1.5){$0$}
		\htext(0.5 -0.5){-$2$}
		\htext(1.5 -0.5){$1$}
		\htext(2.5 -0.5){$1$}
		\move(2 0)\lvec(3 0)
		\move(3 0)\lvec(3 2)
		\move(2 2)\lvec(3 2)
		\move(2 0)\lvec(3 2)
		\htext(2.7 0.5){$0$}
		\htext(2.35 1.5){$0$}
		\esegment

		\move(29 0)
		\bsegment
		\move(0 0)\lvec(2 0)
		\move(0 0)\lvec(0 2)
		\move(1 0)\lvec(1 2)
		\move(0 2)\lvec(2 2)
		\move(2 0)\lvec(2 2)
		\move(0 2)\lvec(1 2)
		\move(0 0)\lvec(1 2)
		\move(1 0)\lvec(2 2)
		\move(1 2)\lvec(1 4)
		\move(2 2)\lvec(2 4)
		\move(1 4)\lvec(2 4)
		\move(2 0)\lvec(4 0)
		\move(4 0)\lvec(4 8)
		\move(3 0)\lvec(3 2)
		\move(2 2)\lvec(4 2)
		\move(3 2)\lvec(3 8)
		\move(3 4)\lvec(4 4)
		\move(3 6)\lvec(4 6)
		\move(3 8)\lvec(4 8)
		\move(2 0)\lvec(3 2)
		\move(3 0)\lvec(4 2)
		\move(3 4)\lvec(4 6)
		\move(2 4)\lvec(2 6)
		\move(2 4)\lvec(3 6)
		\move(2 4)\lvec(3 4)
		\move(2 6)\lvec(3 6)
		\move(1 6)\lvec(2 6)
		\move(1 4)\lvec(1 6)
		
		\move(1 4)\lvec(2 6)
		\move(0 2)\lvec(0 4)
		
		\move(0 4)\lvec(1 4)
		
		\move(-2 0)\lvec(0 0)
		\move(-2 2)\lvec(0 2)
		\move(-2 0)\lvec(-2 2)
		\move(-1 0)\lvec(-1 2)
		
		\move(-1 0)\lvec(0 2)
		\move(-2 0)\lvec(-1 2)

		\htext(0.5 3){$1$}
		\htext(1.5 3){$1$}
		\htext(3.7 4.5){$2$}
		\htext(3.35 5.5){$2$}
		\htext(2.7 4.5){$2$}
		\htext(2.35 5.5){$2$}
		\htext(1.7 4.5){$2$}
		\htext(2.5 3){$1$}
		
		\htext(2.7 0.5){$0$}
		\htext(2.35 1.5){$0$}
		\htext(3.7 0.5){$0$}
		\htext(3.35 1.5){$0$}
		
		\htext(3.5 3){$1$}
		\htext(3.5 7){$1$}
		\htext(0.7 0.5){$0$}
		\htext(-0.3 0.5){$0$}
		\htext(-1.3 0.5){$0$}
		\htext(1.7 0.5){$0$}
		\htext(1.35 1.5){$0$}
		\htext(0.35 1.5){$0$}
		\htext(-0.65 1.5){$0$}
		\htext(-1.5 -0.5){-$1$}
		\htext(-0.5 -0.5){$1$}
		\htext(0.5 -0.5){-$1$}
		\htext(1.5 -0.5){$1$}
		\htext(2.5 -0.5){-$1$}
		\htext(3.5 -0.5){$1$}
		
		\esegment

		\esegment
		\move(19 9)\avec(11 4)\htext(14.5 7){$1$}
		\move(21 9)\avec(21 4)\htext(20.6 7){$0$}
		\move(23 9)\avec(31 4)\htext(25.2 7){$2$}
		\move(8 -1)\avec(1 -6)\htext(4 -3){$0$}
		\move(9 -1)\avec(6 -7)\htext(6.4 -5){$1$}
		\move(10 -1)\avec(10.7 -5.4)\htext(9.8 -3.7){$2$}
		
		\move(20 -1)\avec(18 -7)\htext(18.7 -3){$1$}
		\move(21 -1)\avec(22.8 -9)\htext(21.4 -5){$0$}
		\move(22 -1)\avec(28 -8)\htext(24.3 -4.7){$2$}
		
		\move(32 -1)\avec(31 -5.5)\htext(31.1 -3){$0$}
		\move(33 -1)\avec(36 -6)\htext(33.9 -3.5){$2$}
		\move(34 -1)\avec(42 -6.5)\htext(37.3 -4){$1$}
		
		\move(-3 -13)\avec(-5.3 -15.5)\htext(-4.4 -13.8){$1$}
		\move(-2 -13)\avec(-2.4 -15.8)\htext(-2.5 -13.8){$0$}
		\move(-1 -13)\avec(3.5 -17.5)\htext(1 -15.8){$2$}
		
		\move(4 -13)\avec(-5 -15.7)\htext(2.2 -14.2){$0$}
		\move(4.8 -13)\avec(0 -19.7)\htext(4 -15){$1$}
		\move(5.5 -13)\avec(8 -19.3)\htext(6.8 -15){$2$}
		
		\move(10 -13)\avec(6.5 -18)\htext(9.2 -15){$0$}
		\move(10.5 -13)\avec(9.5 -17.5)\htext(10.3 -16.3){$1$}
		\move(11 -13)\avec(10.2 -29.5)\htext(11.1 -20){$2$}
		
		\move(16.5 -13)\avec(13.3 -19)\htext(14.7 -15.5){$1$}
		\move(17 -13)\avec(16.5 -17.5)\htext(16.3 -15.5){$2$}
		\move(17.6 -13)\avec(17.5 -31)\htext(18.1 -19){$0$}
		
		\move(22.5 -13)\avec(21 -17.5)\htext(22.2 -15.5){$2$}
		\move(23 -13)\avec(22 -31)\htext(23 -21){$1$}
		\move(23.5 -13)\avec(24 -21.5)\htext(24.3 -18){$0$}
		
		\move(29 -13)\avec(22 -17.5)\htext(26.1 -15.5){$0$}
		\move(30 -13)\avec(31 -29)\htext(30.8 -19){$1$}
		\move(29.5 -13)\avec(28.3 -19)\htext(28.3 -16){$2$}
		
		\move(36 -13)\avec(29.3 -17.5)\htext(33.8 -15.2){$0$}
		\move(37 -13)\avec(33.4 -17.5)\htext(35.8 -15.2){$1$}
		\move(37.5 -13)\avec(36 -17.5)\htext(37.3 -15.2){$2$}
		
		\move(43 -13)\avec(41.5 -15.5)\htext(42.7 -14.5){$1$}
		\move(44 -13)\avec(45 -17.5)\htext(44.8 -14.5){$0$}
		\move(45 -13)\avec(50 -17.5)\htext(47.4 -14.5){$2$}
		
		\vtext(-7 -26){\fontsize{8}{8}\selectfont{$\cdots$}}
		\vtext(-3 -26){\fontsize{8}{8}\selectfont{$\cdots$}}
		\vtext(0 -26){\fontsize{8}{8}\selectfont{$\cdots$}}
		\vtext(4 -26){\fontsize{8}{8}\selectfont{$\cdots$}}
		\vtext(8.5 -26){\fontsize{8}{8}\selectfont{$\cdots$}}
		\vtext(12.5 -26){\fontsize{8}{8}\selectfont{$\cdots$}}
		\vtext(16 -26){\fontsize{8}{8}\selectfont{$\cdots$}}
		\vtext(20 -26){\fontsize{8}{8}\selectfont{$\cdots$}}
		\vtext(24 -26){\fontsize{8}{8}\selectfont{$\cdots$}}
		\vtext(28 -26){\fontsize{8}{8}\selectfont{$\cdots$}}
		\vtext(32 -26){\fontsize{8}{8}\selectfont{$\cdots$}}
		\vtext(35.5 -26){\fontsize{8}{8}\selectfont{$\cdots$}}
		\vtext(40 -26){\fontsize{8}{8}\selectfont{$\cdots$}}
		\vtext(45 -26){\fontsize{8}{8}\selectfont{$\cdots$}}
		\vtext(51 -26){\fontsize{8}{8}\selectfont{$\cdots$}}

		\vtext(10 -38){\fontsize{8}{8}\selectfont{$\cdots$}}
		\vtext(16 -38){\fontsize{8}{8}\selectfont{$\cdots$}}
		\vtext(22.5 -38){\fontsize{8}{8}\selectfont{$\cdots$}}
		\vtext(30 -38){\fontsize{8}{8}\selectfont{$\cdots$}}
		
	\end{texdraw}
\end{center}
\end{example}

\end{document}